%% file: article.tex
\documentclass[a4paper,11pt,leqno,english]{book}
\usepackage[english]{babel}
\usepackage{lmodern} 
\usepackage[T1]{fontenc}
\usepackage{multicol}
\usepackage{amsmath}
\usepackage{amssymb}
\usepackage{latexsym}
\usepackage{macros}
\usepackage{yhmath}
\usepackage{imakeidx}
 \usepackage[]{hyperref}
\hypersetup{
    colorlinks=false,       
    linkcolor=red,         
    citecolor=blue,      
}
\usepackage{todonotes}

\makeindex[title=Index of notation,columns=1]
\title{Almost global existence of solutions for  
capillarity-gravity water waves equations \\ with periodic spatial boundary conditions} 

\author{Massimiliano Berti\\SISSA, Via Bonomea 265\\ 34136, Trieste, Italy\thanks{partially supported by   PRIN 2012 ``Variational and perturbative aspects of nonlinear differential problems''.}\\\&\\
Jean-Marc Delort\thanks{partially supported by the ANR project 13-BS01-0010-02 ``Analyse asymptotique des équations aux dérivées
  partielles d'évolution''.}\\ 
Universit{\'e} Paris 13,\\ 
Sorbonne Paris Cit{\'e}, LAGA, CNRS (UMR 7539),\\ 
99, Avenue J.-B. Cl{\'e}ment,\\ 
F-93430 Villetaneuse}

\date{} 

\begin{document} 
\frontmatter

\maketitle

\chapter*{Abstract}

The goal\blfootnote{\makebox[-.6cm]{} Keywords: Capillarity-gravity water waves equations, Long-time existence, Paradifferential calculus, Normal
  forms. MSC 76B15, 76B45,  35Q35, 35S50, 37J40, 70K45.
  } 
  of this monograph is to prove that 
  any solution of the Cauchy problem for the capillarity-gravity water
waves equations, in one space dimension, with  periodic, even in space,  initial data of small size $ \epsilon $, 
is almost globally defined in time on Sobolev spaces, i.e.\ it exists on a time interval of length of magnitude $\epsilon^{-N}$ for any $N$, as soon as
the initial data are smooth enough, and the gravity-capillarity parameters are taken outside an exceptional subset of zero
measure. In contrast to the many results known for these equations 
on the real line, with decaying Cauchy data, one cannot
make use of dispersive properties of the linear flow. 
Instead, our method is based on a normal forms procedure, in order to
eliminate those contributions to the Sobolev energy that are of lower degree of homogeneity in the solution. 

Since the water waves equations are a quasi-linear system, usual normal forms approaches
would face the well  known problem of  losses of derivatives in the unbounded transformations.
In this monograph, to overcome such a difficulty, after a paralinearization of the capillarity-gravity water waves equations, necessary to obtain energy estimates, and thus local existence
of the solutions,  
we first perform 
several paradifferential reductions of the equations to obtain a 
diagonal system  with constant coefficients symbols, up to smoothing remainders. 
Then we may start with a normal form procedure where the 
small divisors are compensated by  the previous paradifferential regularization.
 The reversible structure of the water waves equations,  and the fact that we look for solutions
even in $  x $, guarantees a
key cancellation which prevents the growth of the Sobolev norms of the solutions.

\tableofcontents
\mainmatter
\input{Intro}

\input{chap1max}

\input{chap2max}
\input{chap3}

\input{chap4max}
\input{chap5}

\input{chap6max}

\input{chap7}

\backmatter

\addcontentsline{toc}{chapter}{Bibliography}

\bibliographystyle{abbrv}
\bibliography{biblio}

\printindex

\end{document}

%% file: Intro.tex
\setcounter{chapter}{-1}
\chapter{Introduction}\label{cha:intro}

\section{Main Theorem}\label{sec:01}

The capillarity-gravity water waves equations describe the motion of the interface between an incompressible irrotational fluid in a gravity field
and air, in the presence of surface tension. In the case of the one dimensional problem with finite depth, corresponding to a two-dimensional fluid, the velocity of the fluid is given by the gradient of an harmonic potential
$\Phi $,  called  velocity potential, defined on the 
time dependent  domain 
$$
\Omega_t = \big\{(x,y)\in \R^2; - h <y<\eta(t,x) \big\} \, .
$$ 
As soon as the profile $ \eta (t, x) $ is known, $\Phi$ is determined by
the knowledge of its restriction to $y = \eta(t,x)$, say $\psi(t,x) = \Phi(t,x,\eta(t,x))$, and by the Neumann boundary condition at the bottom $\partial_y\Phi\vert_{y=-h}=0$. The resulting system on $(\eta,\psi)$ is the Craig-Sulem-Zakharov formulation of
the capillarity-gravity water waves equations 
\begin{equation}
  \label{eq:1}
  \begin{split}
    \partial_t \eta &= G(\eta)\psi\\
\partial_t\psi &= -g\eta + \kappa H(\eta) -\frac{1}{2}(\partial_x \psi)^2 + \frac{1}{2}\frac{(\eta'\partial_x\psi + G(\eta)\psi)^2}{1+\eta'^2}
  \end{split}
\end{equation}
where $g>0$ is the acceleration of gravity, $\kappa>0$ the surface 
tension, $ \eta' = \partial_x \eta $, 
$$
H(\eta) = \partial_x[(\partial_x \eta)(1+(\partial_x \eta)^2)^{-1/2}]
$$ 
is the mean curvature of the  wave profile, and $G(\eta)$ is the 
Dirichlet-Neumann operator, defined in terms of the potential $\Phi$
by 
$$ 
G(\eta)\psi  = (\partial_y\Phi -\partial_x\eta\partial_x\Phi)(t,x,\eta(t,x))
$$
(see \cite{CS} and \cite{Zakh} for the derivation of \eqref{1} when the surface tension vanishes, and the book of
Lannes~\cite{L} for a presentation of different models of water waves). 
Local  and global existence of solutions for these equations has
been the object of intensive studies during the last years. 

 In the case $\kappa=0$, the problem of local existence of solutions with smooth and \emph{small} Cauchy data with Sobolev
 regularity, defined on $\R$, has been solved by
Nalimov~\cite{Na}, for the infinite depth problem, and by Yosihara~\cite{Yo}, for finite depth (see also Craig~\cite{Craig}). For large Cauchy data, local
existence in infinite depth has been proved by S. Wu~\cite{Wu1} (see \cite{Wu2} for the case of 
two space   dimensions, i.e. a 3D fluid). The
similar question in finite depth (for a variable bottom in any dimension) has been solved by Lannes~\cite{La1}. The case of local existence for the free surface
incompressible Euler equation has been settled by Lindblad~\cite{Lin}.

Concerning the case of positive $\kappa$, local
existence of solutions with data in Sobolev spaces is due to Beyer and Gunther~\cite{BG} and to Coutand and
Shkroller~\cite{CoS} in the case of solutions of the incompressible free boundary Euler equation. Ifrim and Tataru~\cite{IT3} studied recently local existence when the fluid has
constant vorticity.  The case of finite depth has been settled by Ming and Zhang~\cite{MZ} and
arbitrary bottoms have been considered by Alazard, Burq and Zuily~\cite{ABZ1} for  rough initial data.
Finally, the problem of local existence with
Cauchy data that are periodic in space, instead of lying in a Sobolev space on $\R^d$, has been established by
Ambrose~\cite{Amb} and Ambrose-Masmoudi~\cite{AMS} for $\kappa\geq 0$ in the case of infinite depth, and by
Schweizer~\cite{Schw} for finite depth (even with a non zero vorticity). The case of non-localized Cauchy data lying in uniformly local spaces has been treated by Alazard, Burq and Zuily, in
the case of arbitrary rough bottoms~\cite{ABZ3}. 

Regarding to long time existence of solutions of the water waves equations, most results have been obtained when Cauchy data are small, smooth and decaying at
infinity, which  allows to exploit the dispersive properties of the flow of the linear part of the equations. The first contribution
has been the one of Sijue Wu~\cite{Wu3}, who showed that in one space dimension,
 i.e.\ for a two dimensional fluid of infinite
depth, solutions of the water waves equations with $\kappa=0$ 
exist over a time interval of exponential length $e^{c/\epsilon}$ when the size $\epsilon$ of the initial data goes to zero. In
two space dimensions, i.e.\ for three dimensional fluids, global existence with small decaying data has been obtained
independently by Germain, Masmoudi and Shatah~\cite{GMS1} and by Wu~\cite{Wu4}. Global existence for small data in one space
dimension has been proved independently by Ionescu and Pusateri~\cite{IP1}, Alazard and Delort~\cite{AD1} and by Ifrim and
Tataru~\cite{IT1}, for infinite depth fluids.

For the capillarity-gravity water waves equations, i.e. when $\kappa>0$, global existence is known in two space dimensions (three dimensional
fluids)  in infinite depth by Deng, Ionescu, Pausader and Pusateri~\cite{DIPP}. When the
surface tension is positive, but the gravity $g$ vanishes, global solutions in infinite depth fluids have been proved to
exist by Germain, Masmoudi and Shatah~\cite{GMS2} in dimension 2 and by Ionescu and Pusateri~\cite{IP2} in dimension
1. 

Finally, long time existence results have been obtained independently by  Ifrim and Tataru~\cite{IT2} and
Ionescu and Pusateri~\cite{IP2} 
for small data, that do not necessarily decay at
infinity.  These authors proved   that, in the case  $\kappa>0$, $g=0$ and 
for infinite depth fluids,  data of size $\epsilon$ in some Sobolev space (periodic or on the line), give rise to solutions defined on interval times of length at least $c/\epsilon^2$ (instead of the 
usual lifetime in $c/\epsilon$ that holds true in general for a nonlinear equation with quadratic non-linearity). In
\cite{IT3}, Ifrim and Tataru obtain a similar result when $\kappa = 0$, $g>0$ and constant
vorticity. The case of a zero vorticity  had
been previously treated by Hunter, Ifrim and Tataru in \cite{HIT}. 
It is implicitly contained in the energy estimates of Wu~\cite{Wu1}.
Very recently, Harrop-Griffiths, Ifrim and Tataru~\cite{IT4} obtained a $ c \epsilon^{-2}$ lower bound for the time of existence in the
case $\kappa = 0$, $g>0$, for irrotational incompressible fluids with finite depth.

\medskip

The goal of this monograph is to get, for one dimensional space periodic Cauchy data of size $\epsilon$, solutions of 
the water waves equations \eqref{1} with $g>0, \kappa>0$ defined over a time
 interval of length $c_N\epsilon^{-N}$, for arbitrary $N$. 
We shall be able to achieve this almost global existence result 
if the initial datum $ (\eta_0, \psi_0 ) $ is smooth enough, 
even in $  x $ and $ \eta_0 $ has zero average.

Solutions $ (\eta(t, x), \psi (t, x))  $ of \eqref{1} that are even in $ x $ are called ``standing waves''. 
This property is  preserved during the evolution of \eqref{1}. 
In this case also 
the velocity potential $ \Phi(t, x,y) $ is even and $ 2 \pi $-periodic in $ x $ and so the 
$ x$-component  of the velocity field $ (\Phi_x, \Phi_y) $ vanishes  at $ x = k \pi $, $ \forall k \in \Z $. 
Hence   there is no flux of fluid through the lines $ x = k \pi $, $ k \in \Z $, and 
a solution of the system \eqref{1} physically describes 
the motion of a liquid confined  between two walls. 

The assumption that
$\eta_0$ has zero average  is preserved during the evolutions because the  ``mass'' $ \int_\Tu \eta (x) d x $ is a prime integral
of \eqref{1}. The component $\eta(t,\cdot)$  of the solution 
 will thus lie in a Sobolev space $H_0^{s+\frac{1}{4}}(\Tu)$ of periodic functions with zero mean. 
 
Concerning  $ \psi $, notice that  the right hand side of \eqref{1} is well defined
when $\psi$ is in a space of functions modulo constants (which is natural, as only the  gradient of 
the velocity potential $\Phi $ has a physical
meaning). Thus, projecting the second equation on functions modulo constants, we see that we may look for $\psi$ in an
homogeneous Sobolev space $\Hds{s-\frac{1}{4}}(\Tu)$ of periodic functions modulo constants. 

The main result proved in  this monograph is the  following:
\begin{theoremi}
 \label{1} {\bf (Almost global existence of periodic capillarity-gravity waves)}
 There is a zero measure subset $\Ncal$ in $]0,+\infty[^2$ such that, for any 
 $(g,\kappa)$ in $]0,+\infty[^2-\Ncal$, for any
$N$ in $\N$, there is $s_0>0$ and, for any $s\geq s_0$, there are
$\epsilon_0>0, c>0, C>0$ such that, for any $\epsilon\in
]0,\epsilon_0[$, any  even function 
$(\eta_0,\psi_0)$ in $H_0^{s+\frac{1}{4}}(\Tu,\R)\times \Hds{s-\frac{1}{4}}(\Tu,\R)$ with 
$$
\norm{\eta_0}_{H^{s+\frac{1}{4}}_0}+ \norm{\psi_0}_{\Hds{s-\frac{1}{4}}}<\epsilon \, , 
$$
system \eqref{1} has a unique classical solution $(\eta,\psi)$ defined on $]-T_\epsilon,T_\epsilon[\times \Tu$ with
$T_\epsilon\geq c\epsilon^{-N}$, belonging to the space
\[C^0\bigl(]-T_\epsilon,T_\epsilon[,H_0^{s+\frac{1}{4}}(\Tu,\R)\times \Hds{s-\frac{1}{4}}(\Tu,\R)\bigr)\]
satisfying the initial condition $\eta\vert_{t=0} = \eta_0, \psi\vert_{t=0} = \psi_0$. Moreover, this solution is even in space
 and it stays at any time in the ball of center 0 and radius $C\epsilon$ of $H_0^{s+\frac{1}{4}}(\Tu,\R)\times \Hds{s-\frac{1}{4}}(\Tu,\R)$.
\end{theoremi}
The above theorem establishes almost global existence of solutions to the capillarity-gravity water waves equations  on a compact manifold,
the circle, when the parameters $(g,\kappa)$ avoid a zero measure subset. It is  of different nature than the results we
mentioned above, for decaying Cauchy data on the line, as there is
no  longer 
any dispersive property which makes 
decay the solutions, 
providing long time existence. 

Notice also that the assumption that the parameters
$(g,\kappa)$ avoid a subset of zero measure is essential for the normal forms reduction that will play a seminal role in
the proof. Without any restriction on the parameters, it is known that resonances may occur at the level of the quadratic part
of the non-linearity (the so called ``Wilton ripples'') or for higher order terms in the Taylor expansion of the
non-linearity. We refer  for example to Craig and 
Sulem~\cite{CS1} that discuss the possible dynamical instabilities that might be
generated by this phenomenon.

We do not know if the almost global solutions of the
Cauchy problem proved in Theorem \ref{1} 
are global in time or not. 
Let us mention nevertheless that one may construct classes of global solutions, namely time  periodic and also quasi-periodic ones.

In the case of zero surface tension, existence of small amplitude time periodic 
standing wave solutions  has been proved first by Plotnikov-Toland  in \cite{PlTo} for a fluid in finite depth, 
and  by Iooss, Plotnikov, Toland in \cite{IPT},  see also \cite{IP-SW2}, \cite{IP-SW1}, 
in the case of infinite depth. For the problem with surface tension 
existence of time periodic   standing wave solutions 
has been  proved by Alazard-Baldi \cite{AB}. 
Recently Berti-Montalto \cite{BM}  have  extended this result proving 
the existence  of also time quasi-periodic  capillarity-gravity 
standing wave solutions,  
as well as  their linear stability. 
All the above results prove the existence of solutions when the parameters $ (g, \kappa, h) $
satisfy suitable non-resonance conditions. 
Actually  the problem of the existence of periodic/quasi-periodic solutions 
is faced 
with small divisors difficulties
which are particularly hard due to the quasi-linear nature  of the water waves system \eqref{1}. 

It is well known that the existence of quasi-periodic motions is possible just for 
systems with some algebraic structure
which excludes ``secular motions'' (using the language of Celestial mechanics) 
and growth of the Sobolev norms. 
The most common ones are the Hamiltonian and the reversible structure.
The water-waves system  \eqref{1} exhibits  both of them.

In the proof of Theorem \ref{1}, it 
is  the reversible structure of \eqref{1}, and the fact that we restrict to solutions even in $ x $, 
that ensure
that  terms that could generate a growth of the Sobolev norms of  the standing waves 
actually vanish. Let us describe the meaning of that notion of reversibility. We say that 
the capillarity-gravity water waves system \eqref{1} is reversible with respect to the involution $ S = \sm{1}{0}{0}{-1} $
since 
it  may be written as a dynamical system 
$\vect{\dot{\eta}}{\dot{\psi}} = F\vect{\eta}{\psi}$, where the vector field 
$ F $ satisfies 
\begin{equation}\label{eq:1a}
S F\vect{\eta}{\psi} = -
F\bigl(S\vect{\eta}{\psi}\bigr) \, . 
\end{equation} 
Equivalently, denoting by $ \Phi^t $ the flow of \eqref{1},  the reversibility  property \eqref{1a} is equivalent to
$$
S \circ \Phi^{-t} = \Phi^t \circ S \, .
$$
We shall discuss in the next sections some intuitive reasons 
why the presence of this symmetry may produce stability results. 

\smallskip

To conclude  we also mention some existence results of 
other global solutions, as the traveling wave periodic solutions (standing waves are not traveling because they are even in space).  In dimension  $ 2 $ small amplitude traveling gravity water waves were proved by Levi-Civita \cite{LC}. 
In dimension $ 3 $ the existence of traveling capillarity-gravity periodic water waves 
has been proved by Craig-Nicholls 
\cite{CN} (it is not a small divisor problem) 
and by Iooss-Plotinikov \cite{IP-Mem-2009}-\cite{IoP2} in the case of zero surface tension
 (in such a case it is a small divisor problem).

\smallskip

In the following sections, we shall describe the main steps in the proof of  Theorem \ref{1} performed  in
Chapter~\ref{cha:3} and Chapter~\ref{cha:4}. The other chapters 
are devoted to the construction of the tools needed in those
two chapters, and to the reformulation of the  water waves system \eqref{1} as a paradifferential equation. For simplicity we
take the depth $ h = 1 $ and the gravity $ g = 1 $. 

\section{Introduction to the proof}\label{sec:01a}

In this section we shall  present on a very simplified  model 
some ideas of the normal forms 
method for semilinear PDEs,  and 
we shall point out the deep modifications that this procedure requires for
proving  the long time existence Theorem \ref{1} for 
the {\it quasi-linear} water waves system \eqref{1}. 
In the following sections of this introduction, we shall explain in detail 
our normal form approach, including several technical aspects of the proof.

\smallskip

Consider a semilinear equation of the form 
\begin{equation}\label{eq:2pa}
\begin{split}(D_t - m_\kappa(D_x))u &= P(u,\bar{u})\\
u\vert_{t=0} &= \epsilon u_0,\end{split}
\end{equation}
where $D_t = \frac{1}{i}\frac{\partial}{\partial t}$ and 
$ m_\kappa(D_x) = \Fcal^{-1}\circ \mk(\xi)\circ \Fcal$ 
is an even Fourier multiplier with real valued symbol that depends 
on some parameter $\kappa\in ]0,+\infty[$, 
acting on complex valued functions $ u (x) $ of $L^2(\Tu)$. In the present case 
\be\label{eq:mxid}
\mk(\xi) = (\xi\tanh\xi)^{\frac{1}{2}}(1+\kappa\xi^2)^{1/2}
\ee
comes from the linearization of the water waves equation at the flat surface and zero velocity potential. 
 Notice that since $ \mk(\xi) $ is even in $ \xi $, 
 the operator $ m_\kappa(D_x) $ leaves invariant the subspace of functions even in $  x $. 
Assume that the nonlinearity $ P $ is a polynomial,  homogeneous of degree $ p \geq 2 $, 
depending on $ (u, \bar u) $.  
The first serious difficulty of \eqref{1}, compared to  \eqref{2pa},
 is, of course, the quasi-linear nature of the nonlinear vector field  in \eqref{1}. 
 We shall discuss later how to take it into account. 

The initial datum  $ u_0 $ in \eqref{2pa} is in the Sobolev space $H^s(\Tu,\C)$, for some large $ s $. 
Writing the Sobolev energy inequality associated to \eqref{2pa}, one gets an a priori estimate 
for the solutions like 
\begin{equation}\label{eq:2paa}
\norm{u(t,\cdot)}_{H^s}^2 \leq \norm{u(0,\cdot)}_{H^s}^2 + C\Abs{\int_0^t
\norm{u(\tau,\cdot)}^{p+1}_{H^s}\,d\tau} \, .
\end{equation}
Since the initial datum $ \norm{u(0,\cdot)}_{H^s} \sim \epsilon $, 
one deduces by a bootstrap argument the a priori bound  $ \norm{u(t,\cdot)}_{H^s} = O(\epsilon)$ as
long as $ |t| $ is smaller than $c\epsilon^{-p+1}$, for some $ c $ small.
As a consequence, using that  \eqref{2pa} is locally well posed in 
Sobolev spaces,  the solution exists up to times of magnitude $ c\epsilon^{-p+1} $. 
Nevertheless, one may get a longer existence time combining such energy bounds to Shatah's 
style normal form methods, 
in the framework of periodic functions of space (instead of functions defined on $\R$). 
In the Hamiltonian setting this procedure is often called Birkhoff normal form, that 
for semilinear equations  is by now classical, see for example \cite{BaG}, \cite{DSz}, \cite{BDGS}.
The strategy is to look for some polynomial
correction $Q(u,\bar{u})$ to $u$, homogeneous of degree $p$, 
such that
\begin{multline}\label{eq:Shatah}
(D_t-\mk(D_x))[u+Q(u,\bar{u})] = (\textrm{terms of order $q$}>p) \\
+ (\textrm{terms of order $p$ that do not contribute to the
  energy}).
\end{multline}
Then
energy inequalities for this new equation imply that  
solutions with initial datum  of size $\epsilon$
exist up to times of magnitude $\epsilon^{-q}\gg\epsilon^{-p}$. If one may repeat the process as many times as desired, one
will get finally an existence time of order $\epsilon^{-N}$ for any $N$. 
The fact that the terms of order $ p $ that are left in \eqref{Shatah} 
do not contribute to the   energy inequality 
follows from a suitable algebraic properties of the 
nonlinearity $ P $. In many instances, the
 Hamiltonian nature of the vector field $ P $ provides that property. Here, we shall instead exploit another
 classical structure, namely 
\emph{reversibility}, that for semilinear PDEs  has been used in \cite{FHZ, FG}.
For the model equation \eqref{2pa}, considering  
the involution $ S : u \mapsto \bar u $ acting on the space of functions even in $  x $, 
this amounts to require that the nonlinearity satisfies $ P (u, \bar u ) = \overline{P ( \bar u, u )} $. 
To construct $Q$ from $P$, let us limit ourselves to the case when for instance 
$ P(u,\bar{u}) = a_{p, \ell}  u^\ell\bar{u}^{p-\ell} $ for a scalar coefficient $ a_{p, \ell} $. 
Notice that  it 
satisfies the above reversibility condition if and only if 
$ a_{p, \ell}  $ is {\it real}. We
look for $Q(u,\bar{u})$ as
\[
Q(u,\bar{u}) = M(\underbrace{u,\dots,u}_{\ell},\underbrace{\bar{u},\dots,\bar{u}}_{p-\ell})
\]
for some $p$-linear map $M$ to be determined in such a way that
\[
(D_t-\mk(D_x))M(u,\dots,u,\bar{u},\dots,\bar{u}) = - a_{p, \ell} u^\ell\bar{u}^{p-\ell} + \textrm{ higher order terms} \, .
\]
Distributing the time derivative on each argument of $M$ and replacing $D_tu, D_t\bar{u}$ from 
its expression coming from
\eqref{2pa}, we obtain
\begin{multline}\label{eq:2b}
    (D_t-\mk(D_x))M(u,\dots,u,\bar{u},\dots,\bar{u}) \\= 
    \sum_{j=1}^\ell M( u,\dots,\mk(D_x) u,\dots,u,\bar{u},\dots,\bar{u}) \\-
    \sum_{j=\ell+1}^p M(u,\dots,u,\bar{u},\dots,\mk(D_x) \bar{u},\dots,\bar{u})  \\- \mk(D_x) M(u,\dots,\bar{u}) + \textrm{
      higher order terms}
  \end{multline}
where  in the sums  $ \mk(D_x) $ acts on the $ j$-th component and where we used that $\mk(\xi)$ is even.
  Denote by $\Pin{}$,  $n$ in $\N $,  the spectral projector associated to the $n$th mode of $-\Delta$ on the circle. Then
  as $\mk(\xi)$ is even, $\mk(D_x)\Pin{} = \mk(n)\Pin{}$, 
so that if we replace in the right hand side of \eqref{2b} 
\[
\begin{split}
&M(u,\dots,\mk(D_x)
u,\dots,u,\bar{u},\dots,\bar{u})\\
&M(u,\dots,u,\bar{u},\dots,\mk(D_x) \bar{u},\dots,\bar{u})\\
&\mk(D_x) M(u,\dots,\bar{u})\end{split}
\]
 by respectively
\[
\begin{split}
&\Pin{p+1}M(\Pin{1}u_1,\dots,\mk(n_j) \Pin{j}u_j,\dots,\Pin{\ell}u_\ell,\Pin{\ell+1}u_{\ell+1},\dots,\Pin{p}u_p)\\
&\Pin{p+1}M(\Pin{1}u_1,\dots,\Pin{\ell}u_\ell,\Pin{\ell+1}u_{\ell+1},\dots,\mk(n_j)
\Pin{j}u_j,\dots,\Pin{p}u_p)\\
&\mk(n_{p+1})\Pin{p+1}M(\Pin{1}u_1,\dots,\Pin{\ell}u_\ell,\Pin{\ell+1}u_{\ell+1},\dots,\Pin{p}u_p)\end{split}
\]
we get for the terms homogeneous of degree $p$ in the right hand side of \eqref{2b} the expression
\begin{equation}
  \label{eq:2c}
  \Dcal_{\kappa,\ell}(n_1,\dots,n_{p+1}) \Pin{p+1}M(\Pin{1}u_1,\dots,\Pin{\ell}u_\ell,\Pin{\ell+1}u_{\ell+1},\dots,\Pin{p}u_p),
\end{equation}
with the ``small divisors''
\[
\Dcal_{\kappa,\ell}(n_1,\dots,n_{p+1}) =  
\sum_{j=1}^\ell\mk(n_j) - \sum_{j=\ell+1}^{p+1}\mk(n_j) \, .
\]
In order to obtain \eqref{Shatah}, one would have thus to choose $M$ so that
\begin{equation}\label{eq:2d}
\Dcal_{\kappa,\ell}(n_1,\dots,n_{p+1})\Pin{p+1}M(\Pin{1}u_1,\dots,\Pin{p}u_p) =
-\Pin{p+1}\bigl( a_{p, \ell}\prod_{j=1}^p\Pin{j}u_j\bigr)
\end{equation}
for any $n_1,\dots,n_{p+1}$ such that $\sum_{j=1}^{p+1} \epsilon_j n_j = 0$ for some choice of the signs 
$\epsilon_j \in \{-1, + 1 \} $. 
Suppose that 
the following non-resonance condition holds:
there is $ N_0$ such that 
\be\label{eq:SmDi}
\abs{\Dcal_{\kappa,\ell}(n_1,\dots,n_{p+1})}\geq c(\textrm{third largest among }  n_1,\dots,n_{p+1})^{-N_0}
\ee
for any $ n_1,\dots,n_{p+1} $ except in the  case when 
\begin{equation}\label{eq:cas-resonant}
p \textrm{ is odd}, \ \ell = \frac{p+1}{2}, \textrm{ and }
\{n_1,\dots,n_{\ell}\} = \{n_{\ell+1},\dots,n_{p+1}\},
\end{equation}
for
which $\Dcal_{\kappa,\ell}(n_1,\dots,n_{p+1})$ vanishes for any $\kappa$. 
In several instances,
 it turns out that,  when the symbol $ \xi \mapsto  \mk(\xi) $ is analytic/sub-analytic 
and satisfies suitable ``non degeneracy'' conditions, then 
\eqref{SmDi}  can be verified for any $ \kappa $ taken outside a subset of zero measure.
This is the case, for example,
for the symbol $ m_\kappa (\xi) $ given in \eqref{mxid}, see Proposition~\ref{6} below. 
Notice also that, for $ m_\kappa (\xi) $ in \eqref{mxid}, 
 the divisor $\Dcal_{\kappa,\ell}(n_1,\dots,n_{p+1})$
vanishes as well when all frequencies $ n_k $ are zero, but we ignore this point in this introduction, 
as we shall dispose of it by the fact that in the water waves system \eqref{1} the function $ \eta $
has zero average and $ \psi $ is defined modulo constants. 

Except in case \eqref{cas-resonant}, one may define $ M $ by division of \eqref{2d}. 
The small divisor condition \eqref{SmDi} and the fact that $ P $ is semilinear 
imply that $ M $ is a multilinear map that is {\it bounded} on $H^s$ for any large enough $ s $,  and 
thus the associated transformation
$ u \mapsto u + Q(u, \bar u)$ is   bounded  and invertible 
in a neighborhood of the origin of $ H^s $. 
 In the ``resonant''  case \eqref{cas-resonant}, the corresponding term homogeneous of degree 
$ p $, that is of the form 
\[
-\Pin{\ell}\Bigl[a_{2\ell-1,\ell}\prod_{j=1}^{\ell-1}\abs{\Pin{j}u}^2 \Pin{\ell}u \Bigr],
\] 
cannot be eliminated in the right hand side of
\eqref{2b}, but it does not contribute to the energy because the coefficient  $ a_{2\ell-1, \ell} $ is real.  
In dynamical systems language the  ``actions''
$ \| \Pi_n u \|_{L^2}^2 $ are prime integrals of the resulting ``resonant'' system.
The above procedure can be iterated at higher orders. 
The fact that the terms of order $ p $ that are left at each step
do not contribute to the growth of the Sobolev norms 
is a consequence of the preservation of the 
reversible structure and the fact that we restrict to a subspace of functions even in $ x $.

In trying to implement the above procedure for the complete water waves system \eqref{1} 
many  considerable problems arise. 

\smallskip
\begin{enumerate}
\item 
The first difficulty, that is already present in the local theory of water waves type equations, is the fact that 
\eqref{1} is a quasi-linear system for which a direct application of energy inequalities makes apparently lose derivatives. 
This problem is now well understood (see references in section~\ref{sec:02} below)
and, following the now classical approach of
Alazard-Métivier \cite{AM}, we settle it 
writing \eqref{1} as a paradifferential system in terms of the so called ``good unknown''. 
However we cannot directly use the results in  \cite{AM,ABZ1, AD2} 
since we have to ensure that the new paradifferential system admits a polynomial
expansion as the solution tends to zero. 
We explain this procedure in detail in section \ref{sec:02}. 
An important fact is that the  ``good unknown'' change of variable preserves the reversible structure. 

 \item 
Before performing a normal form analysis as described above, we  reduce the 
terms of the water waves para-differential system 
that are of positive order to constant coefficients.
This step is essential, otherwise the 
transformations  performed above to decrease the size of the nonlinearity, like 
$ u + Q(u, \bar u ) $, would be unbounded.
We perform such a reduction along the lines of 
the recent works of Alazard-Baldi \cite{AB} and
Berti-Montalto \cite{BM}
concerning respectively  periodic and quasi-periodic solutions of \eqref{1},
ending up with a system as
\[
\Bigl(D_t-m_\kappa(D_x)(1+\underline{\zeta}(U;t))\sm{1}{0}{0}{-1} - H(U;t,D_x)\Bigr)V = R(U; t)V 
\]
where $H(U;t,\xi)$ is a diagonal matrix of Fourier multipliers of order one 
(independent of $x$), with imaginary part 
of order $ 0 $, 
$\underline{\zeta}(U;t)$ is a real valued function of $t$, independent of $x$, and $R(U; t)$ is a smoothing operator
(we neglect here another smoothing remainder $ R_2(U; t)U $, see \eqref{9}).
This reduction to constant coefficients of the unbounded terms 
is possible because the dispersion relation $ \xi \mapsto m_\kappa (\xi ) $ is superlinear. 
We explain this procedure in detail in section \ref{sec:03}.

\item 
An energy inequality for the above system  implies an 
estimate like \eqref{2paa} (with $p=2$) proving existence of the solution up to times $ c \epsilon^{-1} $
(this time corresponds essentially to local existence theory). 
In order to prove almost global existence 
we have thus to
eliminate first those contributions to $ \Im H $ and $R$ that have a low degree of homogeneity in $U$, by a normal form method
similar to the one described above.
As in the model case \eqref{2pa}, not all such terms may be eliminated, and one has to check that
those which remain do not make grow the energy. This is a consequence 
of reversibility of the system, and of
the fact that our initial data are even real valued function.
We explain this in more detail in section~\ref{sec:04} of this introduction.
\end{enumerate}

\section[Paradifferential formulation]{Paradifferential formulation and good unknown}\label{sec:02}

In order to start the proof, it will be useful
to rewrite the water waves equations \eqref{1} as a paradifferential system. The classes of paradifferential operators we
shall need will be introduced and studied in detail in Chapters~\ref{cha:2} and \ref{cha:5}. 
In this introduction, let us just consider  symbols given by
functions $(x,\xi)\to
a(x,\xi)$,  with limited smoothness in $ x $,  
satisfying for some real $m$ estimates 
\begin{equation}\label{eq:2}
\abs{\partial_\xi^\beta a(x,\xi)}\leq C_\beta \absj{\xi}^{m-\beta},\quad \forall \beta \in \N \, , 
\end{equation}
where $\absj{\xi} = \sqrt{1+\xi^2}$
(the further information we shall need is to track the dependence of these symbols 
with respect to the unknown dynamical variables). 
One defines for $u$ a test function or a tempered distribution, the action of the Bony-Weyl quantization of $a$ on $u$ as
\[\opbw(a)u = \frac{1}{2\pi}\int e^{i(x-y)\xi}a_\chi\Bigl(\frac{x+y}{2},\xi\Bigr)u(y)\,dyd\xi,\]
where $a_\chi$ is the cut-off symbol whose $x$-Fourier transform $\hat{a}_{\chi}(\eta,\xi)$ is given by
$\hat{a}_{\chi}(\eta,\xi) = \chi(\eta/\absj{\xi})\hat{a}(\eta,\xi)$, $\chi$ being a $C^\infty_0(\R)$ function equal to one
close to zero, with small enough support. If moreover the symbol $a$ is $2\pi$ periodic in $x$, then $\opbw(a)$ acts 
from $ {\dot H}^s(\Tu)$ to $ {\dot H}^{s-m}(\Tu)$ for any $s$, and if \eqref{2} is satisfied also by all $x$-derivatives of $a$, then the 
difference $\opbw(a) - a(x,D)$  sends $H^{s}$ to $H^{s'}$ for any    $s, s'$. Using
the paralinearization formula of Bony, that asserts that
\[R(u,v) = uv -\opbw(u)v -\opbw(v)u\]
has a degree of smoothness equal to the sum of the degrees of smoothness of $u$ and $v$ (minus some universal constant), one
may write,  following essentially Alazard and Métivier \cite{AM} the capillarity-gravity wave equations as a paradifferential system
\[\Bigl(D_t - \opbw(A(\eta,\psi;t,x,\xi))\Bigr)\vect{\eta}{\psi} = R(\eta,\psi)\vect{\eta}{\psi}\]
where $ D_t = \frac{1}{i} \partial_t $, $A$ is a matrix of symbols (that depend on $t, x$ through the functions $\eta, \psi$) and $R$ is a smoothing
operator. One cannot obtain immediately energy estimates for the above paradifferential system because the eigenvalues of the
matrix $A$ have unbounded imaginary part when the frequency $\xi$ goes to infinity. This apparent loss of derivatives is not
the indication of a genuine instability of the system, but comes from the fact that one did not write the problem using the
right unknowns. This difficulty has been overcome in several different ways: Sijue Wu~\cite{Wu1} uses a Lagrangian formulation
of the water waves system; Alazard-Métivier~\cite{AM}, Alazard-Burq-Zuily~\cite{ABZ1} and Lannes~\cite{La1} use the ``good
unknown'' of Alinhac~\cite{Al1}, together with paradifferential calculus for the first two groups of authors and Nash-Moser
methods for the last one; Hunter, Ifrim and Tataru~\cite{HIT}, as well as the last two 
authors in their subsequent works \cite{IT2, IT1}, rely on a blending of the preceding ideas, 
 reformulating  the problem in convenient complex coordinates. 
 
 We adopt here the point of view of the good unknown of
Alinhac, rewriting in Chapter~\ref{cha:6} the water waves equations \eqref{1} as a paradifferential system in $(\eta,\omega)$ where
$\omega$, the good unknown, is defined by 
\begin{equation}\label{eq:2aa}
\omega = \psi-\opbw(B)\eta \, , 
\end{equation}
the function $B$ being given by
\begin{equation}\label{eq:2a}
B = B(\eta,\psi) = \frac{G(\eta)\psi+\eta'\partial_x\psi}{1+\eta'{}^2} = \partial_z\tilde{\Phi}\vert_{z=0},
\end{equation}
where $z$ is a new vertical coordinate in which the free surface of the fluid is given by $z=0$, and $\tilde{\Phi}(x,z)$ is 
the harmonic velocity potential $\Phi(x,y)$ expressed in the new coordinate system $(x,z)$. An important feature of the new system
satisfied by the new unknown is that its right hand side still satisfies the reversibility condition \eqref{1a} as \eqref{1}, 
while the Hamiltonian  structure is lost. The fact
that the reversibility property is preserved through the different reductions made in the proof of Theorem 
\ref{1} will play an essential role
in section~\ref{sec:04} (and section \ref{sec:43}), when dealing with normal forms. We shall discuss this issue more in detail at the end of that section.

The paradifferential water waves system satisfied by the new unknown $(\eta,\omega)$ may be rewritten conveniently 
in  a new complex variable (see section \ref{sec:63}), as presented in detail in section~\ref{sec:31}. Introduce the Fourier multiplier operator of order $-\frac{1}{4}$ given by
\[\lk = \Bigl(\frac{D\tanh D}{1+\kappa D^2}\Bigr)^{1/4}.
\]
If $\eta$
is in $H^{s+\frac{1}{4}}_0(\Tu,\R)$ and $\psi$ is in $\Hds{s-\frac{1}{4}}(\Tu,\R)$, then $\omega = \psi-\opbw(B)\eta$ belongs
 to the same space $\Hds{s-\frac{1}{4}}(\Tu,\R)$ as well, and we introduce
\begin{equation}\label{eq:3}
u = \lk\omega+i\lk^{-1}\eta.
\end{equation}
 We obtain an element of $\{u\in H^s(\Tu;\C); \int_\Tu\Im u\,dx = 0\}/\R$ that we may identify 
  to
 the space $\Hds{s}(\Tu,\C)$ of complex valued $H^s$ functions modulo complex constants. Actually, we are interested only on
 the subspace of even functions, that is endowed with the norm  $\sum_1^{+\infty}\norm{\Pin{}u}_{L^2}^2$, where
 $\Pin{}$ is the spectral projector associated to the $n$-th mode, acting on even periodic functions, i.e.\  
$$
\Pin{}u= \absj{u,\varphi_n} \varphi_n \, , \quad  \absj{u,\varphi_n}  = \frac{1}{\sqrt{\pi}}\int_\Tu u(x)\cos(nx)\,dx \, , \quad
  \varphi_n = \frac{\cos(nx)}{\sqrt{\pi}} \, .
  $$ 
  The reduction of 
  the water waves equations \eqref{1} to a complex system in the good unknown 
 will lead in section \ref{sec:63}  to the following proposition (see Proposition \ref{631}).
 \begin{propositioni}
   \label{2}  {\bf (Water waves equations in complex coordinates)}
Let $s\gg K \gg\rho\gg 1$ be integers. Let 
$(\eta,\psi)$ be a continuous function of time $t$, defined on some interval
$ I \subset \R $, symmetric with respect to $ t = 0 $, with values in the space 
$H_0^{s+\frac{1}{4}}(\Tu,\R)\times \Hds{s-\frac{1}{4}}(\Tu,\R)$, even in $x$, that solves system \eqref{1}. 
Define $\omega$ from $(\eta,\psi)$ by \eqref{2aa}. 
Assume moreover that for $k\leq K$, the $ \partial_t^k $
derivative of $(\eta,\omega)$ belongs to $H_0^{s+\frac{1}{4}-\frac{3}{2}k}(\Tu,\R)\times
\Hds{s-\frac{1}{4}-\frac{3}{2}k}(\Tu,\R)$, and is small enough in that space. Then the function $U= \vect{u}{\bar{u}}$, with
$u$ given by \eqref{3}, solves a paradifferential system of the form
\begin{equation}\label{eq:4}
D_tU = \opbw(A(U;t,x,\xi))U + R(U; t)U
\end{equation}
where 
 the $2\times 2$ matrix of symbols $A(U;t,x,\xi)$ has the form
\begin{multline*}
  A(U;t,x,\xi) = \Bigl(\mk(\xi)(1+\zeta(U;t,x))+ \lambda_{1/2}(U;t,x,\xi)\Bigr)\sm{1}{0}{0}{-1}\\
+ \Bigl(\mk(\xi)\zeta(U;t,x)+ \lambda_{-1/2}(U;t,x,\xi)\Bigr)\sm{0}{-1}{1}{0}\\
+ \lambda_{1}(U;t,x,\xi)\sm{1}{0}{0}{1} + \lambda_{0}(U;t,x,\xi)\sm{0}{1}{1}{0}
\end{multline*}
and
\begin{equation}\label{eq:dispersion}
\mk(\xi) = (\xi\tanh \xi)^{1/2}(1+\kappa\xi^2)^{1/2}
\end{equation}
is a constant coefficients symbol of order $3/2$, 
$$
\zeta(U;t,x) = [(1+\eta'{}^2)^{-3/2}-1] /2  
$$
is a real valued function of $(t,x)$, 
$ \lambda_j(U;t,x,\xi)$ is a symbol of order $j$, for $j = 1, 1/2, 0, -1/2$, with $\Im\lambda_j$ of order $j-1$ when $j = 1$
or $1/2$, and 
$R(U; t)$ is a $2\times 2$ matrix of smoothing operators that gain $\rho$ derivatives.
 \end{propositioni}
\textbf{Remarks}: 
$\bullet$ The symbol $ \mk $ in \eqref{dispersion}
describes the dispersion relation of the   linearized system \eqref{1} at $ \eta = 0 $, $ \psi = 0 $. 
Notice that $ \xi \mapsto \mk (\xi) $ is even  but,  in the subspace of functions even in $ x $, 
the linear frequencies
of oscillations $ \mk (n) $, $ n \in \N $,  are simple.  

\noindent$\bullet$ One may check that the eigenvalues of the matrix $A(U;t,x,\xi)$ are symbols whose imaginary part 
is of order zero (actually of order $ - 1 /2 $). 
 If $A$ where a diagonal matrix, this would imply that 
  $  \opbw(A(U;t,x,\xi)) $
is self-adjoint, modulo a bounded operator,  allowing to derive 
energy estimates for the solutions of \eqref{4} for small
times. Actually in Proposition~\ref{3} below we shall diagonalize the principal symbol of the system \eqref{4}
obtaining in this way energy estimates. 
The property that the imaginary parts of the eigenvalues have order zero is a consequence of the fact that we passed to the good unknown. Had we not done this preliminary
reduction, we would have found eigenvalues  with imaginary part of order $1/2$, 
provoking the instability we mentioned above.

\noindent$\bullet$ The matrix $A$ satisfies three   algebraic properties that will turn out to be essential for the proof of the theorem, namely:
\begin{itemize}
  \item[-] The {\it Reality} condition $\overline{A(U;t,x,-\xi)} = -SA(U;t,x,\xi)S$, 
  where $S$ is the matrix $S = -\sm{0}{1}{1}{0}$,
    which is the translation, in the present complex formulation, of the  involution map $S$ 
    introduced in section~\ref{sec:01}  in the  definition \eqref{1a} of reversibility. This property of $A$ is equivalent, at the operator level,  to
\[\overline{\opbw(A(U;t,x,\xi))V} = -S\opbw(A(U;t,x,\xi))S\overline{V}\]
and reflects the fact that in  system \eqref{4}, the second equation is obtained from the first one 
by   complex conjugation (i.e. that \eqref{1} is a real system).
\item[-] The {\it Parity preserving} condition $A(U;t,-x,-\xi) = A(U;t,x,\xi)$, that implies that the operator $\opbw(A)$
  preserves the space of even functions of $ x $.
\item[-] 
The {\it Reversibility} condition $A(U;-t,x,\xi) = -SA(U_S; t,x,\xi)S$, where $U_S(t) = SU(-t)$.  At the level of operators,
this condition reads
  $$
  \opbw(A(U;-t, \cdot)) = -S\opbw(A(U_S; t, \cdot))S \, . 
  $$ 
  We shall see below, in   Lemma~\ref{hom-nonhom},  that, for the 
  homogeneous components $ A_p $ of the 
  symbol $ A $ introduced in \eqref{4a}, the above condition  amounts to the autonomous reversibility property
  $$
  \opbw(A_p (SU;\cdot)) = -S\opbw(A_p (U;\cdot))S,
  $$ 
  so that the non-linearity $F(U) = \opbw(A_p (U;\cdot))U$ satisfies the reversibility condition $ S F(U) = -F(SU)$.  
  \end{itemize}

The smoothing operator $R(U; t)$ in the right hand side of \eqref{4} satisfies similar properties.

From a dynamical point of view we  can heuristically understand  
why these algebraic properties play  a key role. 
The action of the involution $ S $ on a vector $ U $ of  complex functions even in $ x  $, of the form 
$$ 
U = \bigl[\begin{smallmatrix}u\\\bar{u}\end{smallmatrix}\bigr] \, , \quad u(x) = \sum_{n \geq 1} u_n \cos ( n x ) \, , 
$$
reads  $ u_n \mapsto  - \bar u_n $, for any $ n \geq 1 $.  Introducing  action-angle variables by  the relation 
$  u_n := i \sqrt{ I_n  }  e^{i \theta_n } $, this involution reads   $(\theta_n , I_n ) \mapsto  (- \theta_n, I_n) $. 
An autonomous vector field written in action-angle variables 
$$
\dot \theta = g(\theta, I ) \, , \quad  \dot I = f(\theta, I ) 
$$
is reversible if $ f(\theta, I)$ is odd in $ \theta $ and  $ g (\theta, I) $ is  even in $ \theta $. 
Now, since the angles $ \theta $ are expected to rotate faster than the actions $ I $, 
at the first order 
we could expect the evolution of the actions to be approximated by the 
$\theta$-averaged equation $ \dot I (t) = 0 $, in accordance with the naive  ``averaging principle''.
 Thus reversibility appears as a natural algebraic  property, 
independent of the Hamiltonian nature of a system, 
which may prevent a systematic drift of the action variables, i.e. growth of Sobolev norms in the PDE language.  
The concept of reversibility was introduced 
in KAM theory by  Moser \cite{Mos1}, see also Arnold \cite{Ar} and \cite{BHS} for further developments,  
and then  
it has also been used to prove normal form  stability 
results, see for example the exponential estimates in \cite{Gio}, \cite{GP}  
near an elliptic equilibrium. 
Concerning PDEs we refer, for KAM results, to Zhang, Gao, Yuan \cite{GYZ} 
 for reversible derivative Schr\"odinger equations
and Berti, Biasco, Procesi \cite{BBP}
 for  reversible derivative wave equations, 
and to Faou-Gr{\'e}bert \cite{FG} and Fang-Han-Zhang \cite{FHZ} for  normal form results 
for semi-linear reversible PDEs.   

\smallskip

In view of the normal form procedure that will  conclude the proof of Theorem \ref{1}, it is not sufficient to define our symbols of
paradifferential operators using just estimates \eqref{2}. 
It will be  important to know 
that these symbols are polynomial in $ U $ up to the order $ N -1 $, plus
a non-homogeneous symbol which vanishes as $ O(\|U\|^N) $ as $ U \to 0 $. 
More precisely, the classes of symbols we shall use are given by finite sums
\begin{equation}\label{eq:4a}
A(U;t,x,\xi)= \sum_{p=0}^{N-1} A_p(U,\dots,U;x,\xi) + A_N(U;t,x,\xi)
\end{equation}
where $A_0$ is a constant coefficients symbol of order $m$,  independent of $U$, and 
 $ A_p (U,\dots,U;x,\xi) $, $ p=1,\dots,N - 1 $, resp. $ A_N (U; t,x,\xi) $, 
 are symbols of order $ m $  that depend on $ U $ as  monomials of degree $ p $, resp. 
 in a non-homogeneous way vanishing at order $ O( \| U \|^N ) $ as $ U \to 0 $.
 More    precisely 
$ A_p(U_1,\dots,U_p; x,\xi)$ are  symmetric $p$-linear functions of $(U_1,\dots,U_p)$ satisfying
bounds of the form
\[
\abs{\partial_x^\alpha\partial_\xi^\beta A_p(\Pin{1}U_1,\dots,\Pin{p}U_p; x,\xi)}\leq
C\abs{n}^{\mu+\alpha}\absj{\xi}^{m-\beta}\prod_1^p\norm{\Pin{j}U_j}_{L^2}
\]
where as above $\Pin{j}$ is the spectral projector associated to the $n_j$-th mode, and $\mu$ is a fixed integer. The
meaning of this inequality is that each time we make act one $x$-derivative on the symbol, we lose one power of $n$, i.e.\
one derivative acting on  $U_1,\dots,U_p$. We allow a fixed extra loss of $\mu$ derivatives. 
Assuming that 
$ U_1,\dots,U_p $ are in some Sobolev space $\Hds{\sigma}$,  we see that estimates of the form \eqref{2} 
are satisfied by about $\partial_x^{\sigma- \mu}$
space derivatives   of the symbol $ A_p (U, \ldots, U; x, \xi )$. 
Finally 
the non-homogeneous symbol $A_N (U; t,x,\xi) $ 
satisfies  similar estimates with a constant in the right hand side 
which vanishes at order $N$ when $U$ goes to zero,
see Definition \ref{212}. 
 In the whole monograph, for the
homogeneous symbols 
the dependence on time $ t $ will enter  only through the function 
$ U = U( t ) $, while  the non-homogeneous symbols
may depend explicitly on time $ t $. 

Similar decompositions, in multilinear contributions plus a remainder vanishing
at order $N$ when $U$ goes to zero, have to be assumed as well on the smoothing operators.  Consequently, one has to establish
a symbolic calculus for symbols of paradifferential operators that admit such a decomposition, with remainders of the same
type. This is what is done in Chapter~\ref{cha:2}. Moreover, one has to check that the paradifferential symbols in \eqref{4}
do belong to such classes. This obliges us to revisit the paralinearization of the water wave equations made in
\cite{AM,
  ABZ1, AD2}, in order to verify such a property. This is the object of Chapters~\ref{cha:5} and \ref{cha:6}. We show that
the paralinearization of the Dirichlet-Neumann operator gives rise to a symbol having the wanted 
decomposition in multilinear contributions plus a symbol vanishing at large order when $U$ goes to zero. We need also to get
similar information on the smoothing remainders. Because of that, we do not make use of a variational method to study the 
Dirichlet-Neumann boundary value problem as in \cite{AM, ABZ1, ABZ2}, but we construct a paradifferential parametrix à la Boutet de
Monvel~\cite{BdM1, BdM2, BdM3}, introducing classes of para-Poisson operators whose symbols have a decomposition in
multilinear terms. Next, we apply these results to the construction of the good unknown and the paralinearization of the water waves 
system.  In particular, in Proposition \ref{615} and section \ref{sec:72}  we provide the 
paralinearization formula of the  Dirichlet-Neumann operator, and in Proposition \ref{621} 
the paralinearization of the equation for $ \pa_t \omega $. 

Once the water waves system has been written under the paradifferential form \eqref{4}, one may start the reasoning that will ultimately prove the existence of
its solutions over a time interval of length of order $\epsilon^{-N}$ if the initial data are of size $\epsilon$ 
and smooth enough. 

 As already mentioned in section \ref{sec:01a}, our approach consists in two main steps. The first 
is a  reduction of   \eqref{4} to a  nonlinear system with 
paradifferential operators  with constant coefficients symbols (in $ x $), up to smoothing remainders.
This paradifferential  reduction is presented in section \ref{sec:03}.
The second step 
 is  a  normal form procedure that decreases the size in $ U $ of the non-linear terms, 
and it is presented in section \ref{sec:04}.
  We  underline that 
all the transformations used to reduce to constant coefficients the water waves system \eqref{4} up to smoothing 
remainders  are defined by paradifferential operators, and they are bounded  maps acting on Sobolev spaces $ H^s $. 
Also all the transformations in the normal form procedure are bounded maps on 
Sobolev spaces $ H^s $, for $ s $ large enough.
Actually the small divisors  will be compensated either by the paradifferential regularization of the symbols defining such 
transformations, 
either  by  the smoothing character of the 
remainders, as we shall explain in detail in section \ref{sec:04}.

Notice that it is essential 
for us to reduce the water waves system \eqref{4} to constant coefficients \emph{before} starting a normal forms 
method which reduces the size in $ U$ of the nonlinear terms, as otherwise the quasi-linear
character of the equations would generate losses of derivatives at each transformation. 
Such unbounded changes of variables  would just provide a \emph{formal} normal form, similarly to  the works 
of Craig-Worfolk~\cite{CW} and Dyachenko and
Zakharov~\cite{DZ},  for  the pure gravity equations. We mention that, in presence of capillarity,  
Craig-Sulem~\cite{CS1} have recently proved the boundedness of the third order
Birkhoff normal form transformation. Such a construction of a bounded normal form at order three 
is also related to the result in 
Hunter, Ifrim and Tataru which proves 
an $\epsilon^{-2}$ lifespan for the solution of the pure gravity water waves equations with non
localized data of size $\epsilon$ in \cite{HIT} (see also the recent result of the last two authors in the case of a constant
non zero vorticity \cite{IT3}, and \cite{IT2} for capillarity water waves equations). 

In the present work, we overcome the issue of the boundedness of the normal form using that, after 
reducing the water waves system to constant coefficients in $ x $, up to sufficiently regularizing operators, one falls into a normal form 
framework similar to the one applicable for semi-linear PDEs. 

\section{Reduction to constant coefficients}\label{sec:03}

In order to prove Theorem \ref{1} we are going to constuct
for any integer $ N $, a modified energy 
$E_s(U(t,\cdot))\sim \norm{U(t,\cdot)}_{\Hds{s}}^2 $, equivalent to 
the  square of the $ H^s$-Sobolev norm, 
that satisfies, along 
any small amplitude solutions of \eqref{4},  the bound
\begin{equation}\label{eq:5}
\frac{d}{dt}E_s(U(t,\cdot)) = O(\norm{U(t,\cdot)}_{\Hds{s}}^{N+2}),\ \ {\rm as} \ \  U\to 0 \, . 
\end{equation}
This implies the energy inequality
\[
E_s(U(t,\cdot)) \leq E_s(U(0,\cdot)) + C \Big| \int_0^t \norm{U(\tau,\cdot)}_{\Hds{s}}^{N+2}\,d\tau \Big|
\]
so that, if $E_s(U(0,\cdot))\sim \epsilon$, one may prove by a 
bootstrap argument that, if $\epsilon$ is small enough, and $t$ satisfies
$\abs{t}\leq c\epsilon^{-N}$ for some small enough $c$, then $E_s(U(t,\cdot)) = O(\epsilon^2)$, 
thus $ \norm{U(t,\cdot)}_{\Hds{s}} = O(\epsilon )$. This a priori estimate,
combined with the fact that local existence with smooth Cauchy data holds true according to~\cite{Schw}, implies that the
solution may be extended up to times of magnitude $c \epsilon^{-N}$ 
 (actually local existence 
could be deduced by system \eqref{9} below, but, to avoid further technicalities, 
we directly rely on the results in \cite{Schw}).

The construction of a modified energy $ E_s $ which satisfies  \eqref{5} will rely, as already said, 
on two main conceptually different 
procedures. 
First 
we shall perform a
series of  non-linear para-differential changes of variables, 
similar to the transformations used in Alazard-Baldi \cite{AB} and Berti-Montalto
\cite{BM} to reduce the linearized equations (which arise in a Nash-Moser iteration to prove the 
existence of periodic and quasi-periodic solutions)  to constant coefficients.
Then we shall develop a  normal form method parallel 
to those used by Bambusi, Delort, Grébert and
Szeftel~\cite{BDGS}, \cite{BaG}, \cite{DSz}, and Faou-Gr{\'e}bert \cite{FG} 
for semi-linear PDEs. 
The modified energy $ E_s $ is explicitly constructed in \eqref{4415} (with $ q = N - 1 $).  

We describe below the reduction steps, and will explain the normal form method in the next section.
\medskip

\textbf{Step 1: Diagonalization of the system}
\medskip

We prove in section~\ref{sec:32}, using symbolic calculus, that one may replace the matrix of symbols $A$ in the right hand
side of \eqref{4} by a \emph{diagonal} matrix, up to a modification of the smoothing operator $R(U; t)$. More precisely, 
we get (Proposition \ref{322}):
\begin{propositioni}  {\bf (Diagonalization of the matrix symbol $ A $)} 
  \label{3} 
  There exist $2\times 2$ matrices of symbols 
  of order 0, $P(U; t, \cdot)$, $Q(U; t, \cdot)$, 
  such that $\opbw(P )\circ\opbw(Q) - \mathrm{Id}$ is a 
  $\rho$-smoothing operator (for a large given $\rho$) so that, 
  if $W = \opbw(Q)U$, then $W$ solves the system
  \begin{equation}
    \label{eq:6}
    \bigl(D_t -\opbw(A^{(1)}(U; t, \cdot))\bigr)W = R'(U; t)W + R''(U; t)U
  \end{equation}
where $R'(U; t), R''(U; t )$ are $\rho$-smoothing operators and $A^{(1)}$ is a \emph{diagonal} matrix of symbols
\begin{multline}
  \label{eq:7}
  A^{(1)}(U;t,x,\xi) =  \Bigl(\mk(\xi)(1+\zeta^{(1)}(U;t,x)) + \lambda_{1/2}^{(1)}(U;t,x,\xi)\Bigr)\sm{1}{0}{0}{-1} \\
+ \lambda_{1}^{(1)}(U;t,x,\xi) \sm{1}{0}{0}{1}
\end{multline}
where $\zeta^{(1)}$ is a real valued function of $(t,x)$, $\lambda_j^{(1)}$ is a symbol of order $j$, whose imaginary part is
of order $j-1$. Moreover, $A^{(1)}$ has a homogeneous development of the form \eqref{4a}, and it satisfies the reality,
parity preserving and reversibility conditions.
\end{propositioni}

The proof of the above proposition is made through conjugation by paradifferential operators, in order to decrease
successively the order of the non diagonal terms in the matrix of symbols in \eqref{4}. We may easily explain the idea of that
diagonalization at principal order on the form taken by equation \eqref{1} on the good unknown. In Chapter~\ref{cha:6}, we
shall introduce some real valued function $V$, depending on $\eta, \psi$ such that, if we set $c(t,x) =
(1+\eta'{}^2)^{-3/2}$, the couple $(\eta,\omega)$ satisfies an equation of the form 
$\Pcal \vect{\eta}{\omega} = R$,  where
  \begin{equation}
    \label{eq:7a}
   \Pcal =  \partial_t + \begin{bmatrix}V\partial_x&-D \tanh D \\ 1+\kappa c(t,x)D ^2&V\partial_x\end{bmatrix}
  \end{equation}
and $R$ is  given by the action of operators of non positive order on $\vect{\eta}{\omega}$, and is at least quadratic in
$(\eta,\omega)$. Actually, the first equation in \eqref{7a} comes from the expression \eqref{6143} of $G(\eta)\psi$ and the
second one follows from \eqref{624}. If we conjugate the operator $\Pcal$ by $\sm{\lk}{0}{0}{-\lk}$, we obtain 
the operator
\[ 
\partial_t + \begin{bmatrix}V\partial_x&-\mk(D)\\ c(t,x)\mk(D)&V\partial_x\end{bmatrix}
\]
up to remainders whose action on the unknown gives terms of the same form as $R$. 
Conjugating again this system by
the matrix valued multiplication operator $\sm{1}{0}{0}{c(t,x)^{1/2}}$, we obtain
\[\partial_t + \begin{bmatrix}V\partial_x&-c(t,x)^{1/2}\mk(D)\\  c(t,x)^{1/2}\mk(D)&V\partial_x\end{bmatrix}\]
acting on some new unknown $\vect{\tilde{\eta}}{\tilde{\omega}}$. If we set $u = \tilde{\omega}+i\tilde{\eta}$, we get
the complex equation \[\partial_tu = -ic(t,x)^{1/2}\mk(D)u + V\partial_x u\]
modulo again remainders as above. This is, at principal order, the diagonalized equation we are seeking for.
Actually  in order to obtain, after having performed the above conjugation, 
another paradifferential equation we have to use above
the paraproduct instead of the multiplication operator.

\medskip

\textbf{Step 2: Reduction to constant coefficients at principal order}
\medskip

The goal of the next steps is to reduce to constant coefficients 
the matrix symbol $A^{(1)}$ in \eqref{7}. We shall exploit in an essential way  
that the dispersion relation $\mk(\xi)$ given by \eqref{dispersion} is \emph{superlinear}. We start making this reduction
for the principal part of the  matrix symbol  \eqref{7}  given according to   \eqref{dispersion} by
the product of $\sm{1}{0}{0}{-1}$ and of $(1+\zeta^{(1)}(U;t,x))\sqrt{\kappa}\abs{\xi}^{3/2}$. We would like to eliminate the
$x$-dependence in that symbol. If we make a time dependent change of variables 
$ y = \Phi_U (t, x) $, i.e. 
$ x = \Phi^{-1}_U(t, y)$, the above principal symbol becomes 
$$
(1+\zeta^{(1)}(U;t, \Phi_U(t, x)))\sqrt{\kappa} 
\big| \big( \pa_y \Phi_U^{-1}(t,y) \big)_{| y = \Phi_U(t, x)} \big|^{3/2} \abs{\xi}^{3/2}.
$$
To get a constant coefficients symbol, we cannot just choose the diffeomorphism  
$ \Phi_U $ so that $(1+\zeta^{(1)}(U;t,y))\abs{\pa_y \Phi_U^{-1}(t,y)}^{3/2}
= 1$, namely to take  $ \pa_y \Phi_U^{-1}(t,y)  = (1+\zeta^{(1)}(U;t,y))^{-2/3}$, since $ \Phi_U^{-1} $, and thus 
$ \Phi_U $,  must be  a diffeomorphisms  of $ \Tu $, i.e. we need 
$\Phi_U^{-1} (t,y) - y $ to be periodic. Instead, we choose 
 $\Phi_U^{-1} (t, y)  = y + \gamma (t,y) $ where $ \gamma (t,y) $ is 
the unique periodic function of $ y $, with zero mean, solving the equation
\[
\pa_y \gamma (y) = (1+\zu(U;t))^{2/3}(1+\zeta^{(1)}(U;t,y))^{-2/3} -1
\]
where $\zu(U;t)$ is defined in order to make zero the space average on $\Tu$ of the right handside. The difficulty that one
encounters is due to the fact that this change of variables $\Phi_U$ depends on $U$, so has only limited smoothness. Because
of that, instead of defining a new unknown
function $V$ by the composition $V = W\circ \Phi_U$, we use a 
paracomposition operator in the sense of
Alinhac~\cite{Al}. We set $V= \Phi_U^\star W$, where the paracomposition operator 
$\Phi_U^\star $ is defined  and studied in
section~\ref{sec:Para} (we provide an alternative definition using flows). In that way, when we compute the equation satisfied by the unknown $V$, we still get a paradifferential
equation. More precisely, we prove (Proposition \ref{411}):
\begin{propositioni}
  \label{4} {\bf (Reduction of the highest order)}
If we define $\Phi_U$ as above and set $V = \Phi_U^\star W$, then $V$ solves the system 
\begin{multline}
  \label{eq:8}
  \Bigl(D_t -\opbw\Bigl(\bigl[(1+\zu(U;t))\mk(\xi) + \tilde{\lambda}(U;t,x,\xi)\bigr]\sm{1}{0}{0}{-1} +
  \mu(U;t,x,\xi) {\mathcal I}_2\Bigl)\Bigr)V\\
= R'(U; t)V + R''(U; t)U
\end{multline}
where $\tilde{\lambda}$ is of order $1/2$ with $ \Im\tilde{\lambda} $ of order $-1/2$, $\mu$ is of order one, with
$\Im\mu$ of order zero, $R'$, $R''$ are smoothing operators. Moreover, the reality, parity preserving and reversibility
conditions are still satisfied by the matrix symbol in the left hand side of \eqref{8} and the smoothing operators in its right hand side.
\end{propositioni}
The next step of the proof will be to eliminate the non constant coefficients parts
of $\tilde{\lambda}$ and $\mu$ up to remainders of very negative order.
\medskip

\textbf{Step 3: Reduction to constant coefficients at arbitrary order}
\medskip

Denote by $F(U)$ a diagonal $2\times 2$ matrix of symbols of order $1/2$ which is self-adjoint, up to contributions of order
zero. We define in section~\ref{sec:42}, for any $\theta$ between $-1$ and 1 the operator $\of{\theta} =
\exp[i\theta\opbw(F(U))]$, which is the flow generated by the linear system \eqref{423}. 
Symbolic calculus shows that the diagonal matrix  of symbols  of order $3/2$ given by
\[D_0(U;\cdot) = (1+\zu(U;t))\mk(\xi)\sm{1}{0}{0}{-1}\]
transforms as 
\begin{multline*}
  \of{-1}\opbw\bigl(D_0(U;\cdot)\bigr)\of{1} = \opbw\bigl(D_0(U;\cdot)\bigr)\\
+ (1+\zu(U;t))\opbw\bigl(\absp{F(U,\cdot),\mk(\xi)}\bigr)\sm{1}{0}{0}{-1} + \textrm{ lower order terms}.
\end{multline*}
If one conjugates system \eqref{8} by $\of{1}$, one  obtains a similar system 
with a new symbol given by
\begin{multline*}
  \bigl[(1+\zu(U;t))\mk(\xi) + \tilde{\lambda}(U;t,x,\xi) + (1+\zu(U;t))\absp{F,\mk}\bigr]\sm{1}{0}{0}{-1}\\ +
  \mu(U;t,x,\xi) {\mathcal I}_2 + \textrm{ lower order terms}.
\end{multline*}
We decompose $\tilde{\lambda} = \tilde{\lambda}^{\mathrm{D}} + \tilde{\lambda}^{\mathrm{ND}}$, where
$\tilde{\lambda}^{\mathrm{D}}$ is the $x$-average of  $\tilde{\lambda}$ and $\tilde{\lambda}^{\mathrm{ND}}$ has zero
average. We write in the same way $\tilde{\mu} = \tilde{\mu}^{\mathrm{D}} + \tilde{\mu}^{\mathrm{ND}}$. Solving 
the  equation 
\begin{equation}\label{eq:8a}
\Bigl((1+\zu(U;t))\absp{F,\mk} + \tilde{\lambda}^{\mathrm{ND}}\Bigr)\sm{1}{0}{0}{-1} + \tilde{\mu}^{\mathrm{ND}}
{\mathcal I}_2 = 0
\end{equation}
at principal order, we reduce ourselves to the case when the contributions of order 1 and 1/2 in \eqref{8} have constant coefficients. Repeating the reasoning up to some very negative order, we get (Proposition \ref{421}):
\begin{propositioni}
  \label{5} {\bf (Reduction  to  constant coefficients of $\eqref{8}$)}
There is a diagonal matrix $F(U)$ of symbols of order $1/2$  such that if we set $\tilde{V} = \of{-1}V$, then $\tilde{V}$
solves the system
\begin{equation}
  \label{eq:9}
\begin{split}
   \Bigl(D_t -\opbw\bigl((1+\zu(U;t))\mk(\xi)\sm{1}{0}{0}{-1} - H(U;t,\xi)\bigr)\Bigr)\tilde{V}\\ = R_1(U;t)\tilde{V} + R_2(U;t)U
\end{split}\end{equation}
where $R_1, R_2$ are $\rho$-smoothing operators and $H$ is a diagonal matrix of \emph{constant coefficients} 
symbols  (in $ x $) of order
one, such that $\Im H$ is of order zero. Moreover, the reality, parity preserving and reversibility conditions are satisfied.
\end{propositioni}

Actually the constant coefficient symbols of  $ H $ are of order $ 1 / 2 $, see
the  remark  after Proposition \ref{421}. 
This information is not necessary for the subsequent normal form arguments but it is in agreement 
with the asymptotic expansion of the Floquet exponents of the periodic and quasi-periodic solutions found in  
\cite{AB}, \cite{BM}.

In  equation \eqref{8a} it is essential that the 
symbol $\mk$ defined in \eqref{dispersion} is of order strictly larger than 1, i.e.\ that 
the  capillarity-gravity linear dispersion
relation is superlinear. This is what allows to
construct a symbol $F(U)$ that is of order strictly smaller than the one of $\tilde{\mu}$ in the right hand side, and to start the
induction that  eliminates all variable coefficient symbols up to some order as negative as we want. 
This is also what implies that the contributions coming from the
conjugation of $D_t$ by $\of{1}$ enter into the remainders, as they are of order strictly smaller than 
the main part. If 
the linear part of the operator were just of order one, 
like for instance for Klein-Gordon equations, it would be no longer possible to
reduce in this way the system to constant coefficients before 
performing the normal forms introduced in next section. We refer to \cite{D1, D2} and
references therein for long time existence results for quasi-linear Hamiltonian Klein-Gordon equations.
\medskip

In  \eqref{9}, the operators 
$ \opbw(1+\zu(U;t))\mk(\xi)\sm{1}{0}{0}{-1}$ and $\opbw(\Re H)$
are self adjoint. Consequently, the $L^2$ energy inequality  
associated to that system reads, for example in the case $ t \geq 0 $,  
\begin{multline*}
\norm{\tilde{V}(t,\cdot)}_{L^2}^2 \leq \norm{\tilde{V}(0,\cdot)}_{L^2}^2 + 
\int_0^t\norm{\opbw(\Im
  H(\tau,\xi))\tilde{V}(\tau,\cdot)}_{L^2} \norm{\tilde{V}(\tau,\cdot)}_{L^2} \,d\tau   \\
  + \int_0^t\norm{R_1(U;t)\tilde{V}(\tau,\cdot) + R_2(U;t)U(\tau,\cdot)}_{L^2} \norm{\tilde{V}(\tau,\cdot)}_{L^2}\,d\tau  \, .
\end{multline*} 
Moreover, since 
the coefficients in left hand side of \eqref{9} are constant in $ x $, we may commute as
 many space derivatives as we 
wants with the equation, and deduce
from  the above $ L^2 $-estimate a similar inequality for Sobolev norms.
As a consequence,  if the symbol $ \Im H $, which is of order zero,  were vanishing  
as  $ O(\| U \|^N) $ when $U$ goes to zero, as well as the
smoothing terms $ R_1, R_2 $, we would deduce from that the estimate
\[
\norm{\tilde{V}(t,\cdot)}_{\Hds{s}}^2 \leq 
\norm{\tilde{V}(0,\cdot)}_{\Hds{s}}^2 + 
C \int_0^t\norm{U(\tau,\cdot)}_{\Hds{\sigma}}^N(\norm{\tilde{V}(\tau,\cdot)}_{\Hds{s}}^2
+  \norm{U(\tau,\cdot)}_{\Hds{s}}^2)\,d\tau  .
\]
Using that at $t=0$, $U$ is of size $\epsilon$ in $\Hds{s}$, and that, as long as $U$ remains small, $U$ and $\tilde{V}$ are
of the same magnitude, one would deduce from that, by  bootstrap, an a priori estimate $\norm{U(t,\cdot)}_{\Hds{s}} =
O(\epsilon)$ over a time interval of length $c\epsilon^{-N}$ for some small $c$. Together with local existence theory, this
would imply that the solution may be extended up to such an interval of time. In conclusion,
to prove our main theorem, we
have to show that we may modify equation \eqref{9} in order to ensure that  $\Im H(U;\cdot)$ 
and $R_1(U;t)$, $R_2(U;t)$ will 
vanish as $ O( U^N )$ as $ U $ goes to zero. We shall achieve this goal through a normal forms procedure.

\section{Normal forms}\label{sec:04}

Let $B(U;t,\xi)$ be a diagonal matrix of constant coefficients symbols of order zero to be determined. 
We conjugate system \eqref{9} 
by the operator $\exp\bigl(\opbw(B(U;t,\xi))\bigr)$. Setting $\tilde{V}^1 = \exp\bigl(\opbw(B(U;t,\xi))\bigr) \tilde{V}$, we get
\begin{multline}
  \label{eq:10}
   \Bigl(D_t -\opbw\bigl[D_tB(U;t,\xi)+ (1+\zu(U;t))\mk(\xi) \sm{1}{0}{0}{-1}  - H(U;t,\xi)\bigr]\Bigr)\tilde{V}^1\\ = \textrm{ smoothing terms}.
\end{multline}
We want to choose $B$ in order to eliminate the contributions to $\Im H$ which are homogeneous of degree $p <N$. Thus we decompose 
\[
H(U;t,\xi) = \sum_{p=1}^{N-1}H_p(U,\dots,U;\xi) + H_N(U;t,\xi)
\]
with $H_p(U_1,\dots,U_p; \xi)$ a symmetric $p$-linear map in $(U_1,\dots,U_p)$ with values in diagonal matrices of constant
coefficient symbols, and $H_N$ a symbol vanishing at least at order $N$ when $U$ goes to zero. 
We look for $B$ as  
\[
B = \sum_{p=1}^{N-1}B_p(U,\dots,U; \xi)
\]
with $B_p$ a $p$-linear map. Rewriting equation \eqref{4} as \[D_t U = \mk(D)\Kcal + \textrm{ terms vanishing at least at
  order 1 in } U,
  \]
 (such terms are unbounded operators of order 3/2)
where $\Kcal = \sm{1}{0}{0}{-1}$, we may write
\begin{multline*}
  D_t(B_p(U,\dots,U)) = \sum_{j=1}^p B_p(U,\dots,U,D_tU,U,\dots,U)\\
= \sum_{j=1}^p B_p(U,\dots,U,\mk(D)\Kcal U,U,\dots,U) + \textrm{ terms of higher order in } U.
\end{multline*}
To eliminate the lower order terms contributions in $\Im H$, we are thus reduced to find $B_p$ so that
\begin{equation}
\label{eq:11}
\sum_{j=1}^p B_p(U,\dots,U,\mk(D)\Kcal U,U,\dots,U) = i\Im \tilde{H}_p(U,\dots,U)
\end{equation}
where $\tilde{H}_p$ is computed  recursively from $H_p$ and from $B_{p'}$, $p'<p$. Replacing $U = \vect{u}{\bar{u}}$ by $\Pi_{n_j}^+U =
\vect{\Pin{j}u}{0}$ in  the first $\ell$ components and by $\Pi_{n_j}^-U =
\vect{0}{\Pin{j}\bar{u}}$ in the last $p-\ell$ ones, and using that 
$$
\mk(D)\Kcal\Pin{j}^\pm = \pm \mk(n_j)\Pin{j}^\pm \, , 
$$ 
we
see that it is enough to determine $B_p$ so that
\begin{multline}
\label{eq:12}
\Bigl(\sum_{j=1}^\ell \mk(n_j) - \sum_{j=\ell+1}^p
\mk(n_j)\Bigr)\\\times B_p(\Pin{1}^+U_1,\dots,\Pin{\ell}^+U_\ell,\Pin{\ell+1}^-U_{\ell+1},\dots,\Pin{p}^-U_p; \xi)\\ = 
i\Im \tilde{H}_p(\Pin{1}^+U_1,\dots,\Pin{\ell}^+U_\ell,\Pin{\ell+1}^-U_{\ell+1},\dots,\Pin{p}^-U_p; \xi).
\end{multline}
Thus we just have to be able to divide 
the right hand side of \eqref{12} 
by the ``small divisor'' 
\[
\Dcal_\ell(n_1,\dots,n_p) = \sum_{j=1}^\ell \mk(n_j) - \sum_{j=\ell+1}^p
\mk(n_j) 
\]
when the parameter $\kappa$ is taken outside a convenient subset of zero measure.
Clearly, this will not be possible if $p$ is even, $\ell = p/2$ and one has a two by two cancellation between terms of the
first and second sum so that $ \Dcal_\ell(n_1,\dots,n_p) = 0 $. The possibility to solve nevertheless equation \eqref{12} follows from the following proposition:
\begin{propositioni}
  \label{6}
(i) There is a zero measure subset $\Ncal$ of $]0,+\infty[$ such that, if $\kappa$ is taken outside $\Ncal$, there are
$N_0\in \N$, and $c>0$ such that for any $n_1,\dots,n_p$ in $\N^*$,
\begin{equation}
  \label{eq:13}
  \abs{\Dcal_\ell(n_1,\dots,n_p)}\geq c(n_1+\cdots+n_p)^{-N_0}
\end{equation}
except if $p$ is even, $\ell = p/2$ and $\{n_1,\dots,n_\ell\} = \{n_{\ell+1},\dots,n_{p}\}$.

(ii) In the latter case 
\begin{equation}
\label{eq:14}
\Im\tilde{H}_p(\Pin{1}^+U_1,\dots,\Pin{\ell}^+U_\ell,\Pin{\ell+1}^-U_{\ell+1},\dots,\Pin{p}^-U_p;\xi)\equiv 0.
\end{equation}
\end{propositioni}
Proposition~\ref{6} allows consequently to solve equation \eqref{12} in any case, and thus to eliminate the contributions to $\Im H$ of lower
degree of homogeneity. By the estimate \eqref{13} the small divisor 
that appears dividing the right hand side of \eqref{12} by $\Dcal_\ell$ 
produces a loss of  $ O(N_0 )$ derivatives on the smoothness of $B_p$ as a function of $x$, but the associated 
paradifferential operator $ \opbw(B_p (U;t,\xi)) $
 remains bounded on any $H^{s}$ if $U$ is in $ H^\sigma $  with a large $ \sigma $ (depending on $ N_0$). Indeed
 the $H^{s}$-boundedness of the paradifferential operator $ \opbw(B_p (U;t,\xi)) $ 
depends only on a finite number (\emph{independent of $s$}) of
derivatives of  $U$. 

One may proceed in a similar way to eliminate the smoothing contributions, which are
homogeneous of low order in $ U $, in the right hand side of
\eqref{10}. More precisely,   in section \ref{sec:44} we shall construct  iteratively,
 by an analogous  normal 
form method, quasi-invariant modified energy Sobolev norms 
for the system \eqref{10}, see \eqref{4415}, 
constructing ultimately a modified energy $ E_s $ satisfying
\eqref{5}. 
We require at each iterative step a small divisor estimate 
as
\be\label{eq:SmDi1}
\Abs{ 
\sum_{j=0}^\ell\mk(n_j) - \sum_{j=\ell+1}^{p+1}\mk(n_j)}\geq c\max(n_0,\dots,n_{p+1})^{-N_0}
\ee
for any integer $ n_0,\dots,n_{p+1} $ except in the  case when 
\begin{equation*}
p \textrm{ is even}, \  \ \ell = \frac{p}{2}, \  \textrm{ and } \ 
\{n_0,\dots,n_{\ell}\} = \{n_{\ell+1},\dots,n_{p+1}\},
\end{equation*}
and we verify in Lemma \ref{442}-($ii$) a cancellation similar to \eqref{14}. 
Notice that the small divisors estimate \eqref{SmDi1} is very weak because the right hand side contains the maximum of all 
integers $ n_0, \ldots, n_{p+1} $ and not just the third largest $ \max_3 (n_0, \ldots, n_{p+1}) $, as in \eqref{SmDi}. 
The latter condition is essential 
 for the normal form approach developed in the papers \cite{BaG},  \cite{BDGS}, \cite{FG}, concerning semilinear PDEs.
It  could be imposed as well in the present context. The use of the weak non resonance condition \eqref{SmDi1}
produces losses of $ O(N_0) $ derivatives at each normal form transformation which, however, 
in the present approach, are compensated 
by the smoothing character of the remainders, when the Sobolev regularity of $ H^s $ is large enough. 

In conclusion the preceding Proposition \ref{6}, which is proved in sections \ref{sec:43} and \ref{sec:71}, is the last remaining step in the proof of the main Theorem \ref{1}.

\medskip

\begin{proof1}{Idea of the proof of  Proposition \ref{6}}
  The small divisor estimate \eqref{13} of Part (i) 
  follows from more general results of
  Delort-Szeftel~\cite{DSz} concerning measure estimates of sublevels of subanalytic functions. 
Part (ii) is a consequence of the fact that the matrices of symbols $\tilde{H}_p$ obtained by
  the successive reductions described above still satisfy the reality, parity preserving and reversibility
  properties. Combined together, these three properties may be shown to imply the vanishing of 
  the left hand side of \eqref{14}, see Lemma \ref{432}. 
\end{proof1}

As explained above, it is essential 
that the cancellation property \eqref{14} holds in the case when the
left hand side of \eqref{12} vanishes identically. 
This   cancellation follows ultimately from the reversible nature of the capillarity-gravity water waves 
equations  \eqref{1}, and it is
essential for us that this property, 
that holds for  system  \eqref{1}, be preserved by all the
different reductions we make. In particular, a key point is that the definition  $\omega =
\psi-\opbw(B)\eta$ of the new unknown shows, according to \eqref{2a}, that $\omega$ is an 
odd function (actually a linear
function) of $\psi$. Because of that, the water waves system satisfied by $(\eta, \omega)$ will still satisfy the
reversibility property \eqref{1a}. 
Actually we prove in Corollary \ref{622} that the new water waves  system in the $ (\eta, \omega) $ variables 
satisfies a time dependent reversibility property. 

 As already mentioned elsewhere, reversibility is not the only algebraic information that might be used to ensure the solvability of an equation
of the form \eqref{12}, and  so proving that 
the terms that could generate a growth in modified Sobolev norms,
actually vanish. An alternative property,  
that has been used extensively in other works, is  
the Hamiltonian character of the system. In this case, one controls a modified Sobolev energy of the solution, defined through composition of the usual
Sobolev energy by canonical transformations. In that way, one may pursue all the reductions staying in the Hamiltonian
framework. Equation \eqref{12} is then replaced by a similar equation at the level of  Hamiltonians, 
often called ``homological equation'', that may be solved in all cases,
except the one corresponding to (ii) of the preceding proposition. But the contributions to the Hamiltonian generated by the
indices satisfying   $\{n_1,\dots,n_\ell\} = \{n_{\ell+1},\dots,n_{p}\}$ when $p$ is even and $\ell= p/2$
depend actually  only on action coordinates, and thus cannot generate growth of the energy. We refer to the works of
Bambusi-Grébert~\cite{BaG} and Bambusi, Delort, Grébert and Szeftel~\cite{BDGS} were such ideas are developed for
semi-linear equations, and to \cite{D1, D2} for the application of a similar method to quasi-linear Klein-Gordon type
equations. 

We have not followed this Hamiltonian approach for the capillarity-gravity water wave equations
because the passage from the initial unknown to the good one
does not seem to preserve the Hamiltonian character of the system. 
We exploit instead the reversible structure, whose preservation through the change of the good unknown is trivial.

Let us mention that an index of notation is provided at the end of the volume.

\bigskip

{Acknowledgement: we thank Walter Craig and Fabio Pusateri for useful comments which led to an improvement 
of the manuscript.}

%% file: chap1max.tex
\chapter{Main result}\label{cha:1}
\section{The periodic capillarity-gravity equations}\label{sec:11}

  Let $h>0$ be a constant, $\Tu = \R \slash (2 \pi \Z) $ the circle,
  and consider an incompressible and irrotational perfect fluid, occupying at time
  $t$ a two dimensional domain, periodic in the horizontal variable, given by 
  \begin{equation}
    \label{eq:111}
    \Omega_t = \big\{ (x,y)\in \Tu\times\R \, ;  \ -h<y<\eta(t,x) \big\},
  \end{equation}
where $\eta:\R\times\Tu\to \R$ is a smooth enough function such that $\norml{\eta}{\infty}<h$. The velocity field in
$\Omega_t$ is then the gradient of some harmonic function $\Phi$, called the velocity potential. The evolution of the motion of the fluid is determined by
the knowledge of the two functions
\[
(t,x)\to \eta(t,x), \quad (t,x;y)\to \Phi(t,x,y)
\]
defined for $t$ in some interval, $x$ in $\Tu$, $(x,y)$ in $\Omega_t$, subject to the following boundary conditions (see for
instance the book of Lannes~\cite{L})
\begin{equation}
  \label{eq:112}
  \begin{split}
    \partial_t \eta &= \partial_y\Phi - (\partial_x\eta)(\partial_x\Phi) \textrm{ on } y=\eta(t,x)\\
\partial_t\Phi &= -g\eta + \kappa H(\eta) -\frac{1}{2}\bigl[(\partial_x\Phi)^2+(\partial_y\Phi)^2\bigr] \textrm{ on }
y=\eta(t,x)\\
\partial_y \Phi &= 0 \textrm{ on } y=-h,
  \end{split}
\end{equation}
where $g>0$ is the acceleration of gravity, $\kappa>0$ the surface tension, and 
$$
H(\eta) = \partial_x[(\partial_x
\eta)(1+(\partial_x \eta)^2)^{-1/2}] 
$$ 
is the mean curvature. 
From now on, we shall write $\eta'$ for $\partial_x\eta$. We define
\[\psi(t,x) = \Phi(t,x,\eta(t,x))\]
the restriction of the velocity potential $\Phi$ to the free interface and, following Zakharov~\cite{Zakh} and Craig-Sulem~\cite{CS}, we express
\eqref{112} as a system in the variables $(\eta,\psi)$.

Define the Dirichlet-Neumann operator \index{G@$G(\eta)$ (Dirichlet-Neumann operator)} $G(\eta)$ of the upper boundary of the set $\Omega_t$ as
\begin{equation}
  \label{eq:112a}
  G(\eta)\psi = \sqrt{1+\eta'{}^2}(\partial_n\Phi)\vert_{y=\eta(t,x)} = (\partial_y\Phi -\eta'\partial_x\Phi)(t,x,\eta(t,x))
\end{equation}
where $\partial_n$ is the outward unit normal at the free interface $y= \eta(t,x)$. According to Craig-Sulem~\cite{CS},
$(\eta,\psi)$ satisfies the capillarity-gravity water waves system
\begin{equation}
  \label{eq:113}
  \begin{split}
    \partial_t \eta &= G(\eta)\psi\\
\partial_t\psi &= -g\eta + \kappa H(\eta) -\frac{1}{2}(\partial_x\psi)^2 + \frac{1}{2}\frac{(\eta'\partial_x\psi + G(\eta)\psi)^2}{1+\eta'^2}.
  \end{split}
\end{equation}
Our goal is to prove that, for fixed $h$, and almost all values of the  parameters $(g,\kappa)\in ]0,+\infty[^2$, for any fixed
integer $N$, Cauchy data that are smooth enough, of small enough size $ \epsilon > 0 $, even in $ x $, 
generate a solution defined on a time interval of length at least $c\epsilon^{-N}$.

\smallskip

Notice that the right hand side of \eqref{113} remains invariant if we replace $(\eta(x),\psi(x))$ by
$(\eta(-x),\psi(-x))$, so that  initial data which are even in $  x $ give rise to a solution that remains even in $  x $ at any future  time $ t $.

\section{Statement of the main theorem}\label{sec:12}

Let us introduce some notation. For $n$ in $\N^*$, we denote by $\Pin{}$ the orthogonal projector from $L^2(\Tu,\C)$ or
$L^2(\Tu,\C^2)$ to the subspace spanned by $\{e^{inx},e^{-inx}\}$.  For any $s$ in $\R$, we define 
\begin{equation}
  \label{eq:121}
  \begin{split}
    \Hsz{s}(\Tu,\R) &  \stackrel{\textrm{def}}{=}   \Big\{ u\in H^s(\Tu,\R) \, ; \, \int_\Tu u(x)\,dx =0 \Big\}\\
\Hds{s}(\Tu,\R) &  \stackrel{\textrm{def}}{=} H^s (\Tu,\R)/\R \, .
  \end{split}
\end{equation}
Both $ \Hsz{s} $ and the homogeneous Sobolev space $ \Hds{s} $
will be endowed with the same norm
\begin{equation}
  \label{eq:122}
  \norm{u}_{\Hsz{s}} = \norm{u}_{\Hds{s}} = \Bigl(\sum_{n=1}^{+\infty}\norml{\Pin{}u}{2}^2\Bigr)^{1/2}.
\end{equation}
Furthermore, we denote by $\Hsze{s} = \Hsze{s}(\Tu,\R)$,  $ \Hdse{s} = \Hdse{s} (\Tu,\R) $, 
the subspaces of $ \Hsz{s}(\Tu,\R), \Hds{s}(\Tu,\R)$ formed by the functions even in $ x $. Our main result is the following:
\begin{theorem}
  \label{121} {\bf (Almost global existence of periodic capillarity-gravity waves)}
There is a zero measure subset $\Ncal$ in $]0,+\infty[^2$ such that, for any $(g,\kappa)$ in $]0,+\infty[^2-\Ncal$, for any
$N$ in $\N$, there is $s_0>0$ and for any $s\geq s_0$, there are $\epsilon_0>0, c>0, C>0$ such that, for any $\epsilon\in
]0,\epsilon_0[$, any function  
$ (\eta_0, \psi_0) \in \Hsze{s+\frac{1}{4}}(\Tu,\R) \times \Hdse{s-\frac{1}{4}}(\Tu,\R) $ satisfying 
$$ 
\norm{\eta_0}_{H^{s+\frac{1}{4}}_0} + \norm{\psi_0}_{ {\dot H}^{s - \frac{1}{4}}} <\epsilon \, , 
$$
system \eqref{113} has a unique classical solution $(\eta,\psi)$ defined on $]-T_\epsilon,T_\epsilon[\times \Tu$ with
$T_\epsilon\geq c\epsilon^{-N}$, belonging to the space
\[C^0(]-T_\epsilon,T_\epsilon[,\Hsze{s+\frac{1}{4}}(\Tu,\R)\times \Hdse{s-\frac{1}{4}}(\Tu,\R))\]
satisfying the initial condition $\eta\vert_{t=0} = \eta_0, \psi\vert_{t=0} = \psi_0 $. 
Moreover this solution 
stays  at any time in the ball of center 0 and radius $C\epsilon$ of $\Hsze{s+\frac{1}{4}}\times \Hdse{s-\frac{1}{4}}$.
\end{theorem}
\textbf{Remarks}: $\bullet$ The above theorem provides an ``almost global'' solution of system \eqref{113} with small smooth
periodic  
even initial data, when the parameters $(g, \kappa )$ stay outside a subset of zero measure. Notice that we
assume that $g>0$ and $\kappa>0$, which will be essential for our proof.

$\bullet$ We assume that at $t=0$, $\eta$ has zero average. It is well known that the quantity $\int_{\Tu}\eta(t,x)\;dx$ is
conserved during the evolution, so that $\eta(t,\cdot)$ has zero average at any time.

$\bullet$ The function $\psi$ is taken in a Sobolev space defined modulo constants, i.e. the homogeneous Sobolev spaces in \eqref{121}. This is related to the fact that only
$\partial_x\psi$ has a physical meaning.

$\bullet$ An essential property 
 is that the water waves system~\eqref{113} is reversible in the following sense. Let \index{S@$S$
  (Reversibility operator in real form)}
$ S $ be the linear involution 
i.e. $ S^2 = {\rm Id} $, 
\begin{equation}\label{eq:121a}
S : \R^2 \to \R^2 \, , \quad {\rm with \ matrix} \quad \sm{1}{0}{0}{-1} \, .
\end{equation}
Then, 
denoting by $ F = F(\eta, \psi) = 
 \vect{F_1(\eta, \psi)}{F_2(\eta, \psi)} $ the right hand side   in \eqref{113}, we have 
\be\label{eq:Frev}
F\circ S = -S\circ F \, , 
\ee
namely that $ F_1 (\eta, - \psi) = - F_1 (\eta, \psi )$ is odd in $ \psi $ and $ F_2 (\eta, - \psi) =  F_2 (\eta, \psi )$ is even in 
$ \psi $. 

Denoting by $ \Phi^t $ the flow of \eqref{113},  the reversibility  property \eqref{Frev} is equivalent to
$$
S \circ \Phi^{-t} = \Phi^t \circ S \, .
$$
In other words, if  
$ \vect{\eta (t; \eta_0, \psi_0)}{\psi (t; \eta_0, \psi_0)} $
is the solution of \eqref{113} with initial condition 
$ \eta (0; \eta_0, \psi_0) = \eta_0 $, $  \psi (0; \eta_0, \psi_0) = \psi_0  $, then
$$
S \vect{\eta (-t; \eta_0, \psi_0)}{\psi (-t; \eta_0, \psi_0)} =  \vect{\eta (-t; \eta_0, \psi_0)}{-\psi (-t; \eta_0, \psi_0)} 
$$
is solution as well of \eqref{113}, and, at $ t = 0 $, it takes the value 
$ \vect{\eta_0}{-\psi_0} $. Therefore, by uniqueness of the solution,  
$$
\eta (-t; \eta_0, \psi_0)  =  \eta (t; \eta_0, - \psi_0) \, , \quad 
-\psi (-t; \eta_0, \psi_0) =  \psi (t; \eta_0, -\psi_0) \, , 
$$
which may be written as
\begin{equation}
  \label{eq:125}
  \vect{\eta}{\psi}\Bigl(-t;\vect{\eta_0}{\psi_0}\Bigr) = S \vect{\eta}{\psi}\Bigl(t;S\vect{\eta_0}{\psi_0}\Bigr).
\end{equation}
The reversibility of the water waves system 
plays an essential role in the proof of our theorem. Together with evenness,
this condition will ensure that some quantities, that could generate a growth in modified 
Sobolev norms, actually vanish, see Lemmas \ref{432} and \ref{442}.

\medskip

The first step in the proof of Theorem \ref{121} will be to rewrite the water waves system 
\eqref{113} in terms of the ``good unknown'' of Alinhac. This will be essential in order to 
obtain energy estimates. To do
so, we  introduce tools of paradifferential calculus in the next chapter.

\smallskip

Let us fix some notation that will be used in the rest of the monograph. Since the quotient map induces an isometry from
$\Hsz{s}$ to $\Hds{s}$, we shall consider $(\eta,\psi)$ as an element of $\Hdse{s+\frac{1}{4}}\times \Hdse{s-\frac{1}{4}}$,
identifying thus $\eta$ to its inverse image through this isometry.

To simplify notation, we shall assume that the depth $h$ is equal to 1. Moreover, we may reduce to the case $g=1$: actually,
considering instead of $(\eta,\psi)$ the couple of functions $(\eta(\sqrt{g}t,x),\sqrt{g}\psi(\sqrt{g}t,x))$, we reduce
\eqref{113} to the same system, with $(g,\kappa)$ replaced by $(1,\kappa/g)$.

Let us mention that in the whole monograph, we shall denote indifferently $\norml{\cdot}{2}$ or $\norm{\cdot}_{\Hds{0}}$. It
should be understood that in any case the zero frequency of functions at hand is discarded.

%% file: chap2max.tex
\chapter{Paradifferential calculus}\label{cha:2}

\section{Classes of symbols}\label{sec:21}

We define in this chapter several classes of symbols of paradifferential operators 
that we shall use in the 
whole monograph. Our classes are essentially
standard ones, except that our symbols 
are depending on some functions  (that in the application will be the solution $ U $ of the equation itself), 
and we control the semi-norms of the symbols from
these functions. 

To fix ideas, let us mention that in the whole monograph we shall deal with parameters satisfying
\begin{equation}
  \label{eq:211}
  s \gg \sigma\gg K \gg \rho \gg N
\end{equation}
where $ N $ is the exponent in the lower bound of the existence time of Theorem~\ref{121} 
and $\rho$ will be the smoothing
degree at which we shall stop the symbolic calculus of paradifferential operators. 
As explained in the introduction $ N $ coincides with the number  of normal form steps 
performed in sections
\ref{sec:43} and \ref{sec:44}.  
In order to implement these Birkhoff normal form 
procedures we have to choose the parameter $ \kappa $ such that the
small divisor estimate \eqref{713} of Proposition \ref{711} holds. The effect of the small divisors 
is to produce a loss of derivatives proportional to $ N_0 $ and $ N $, that may 
be compensated by taking  the regularizing index $ \rho \gg N $ large enough. This requires also 
to develop the paradifferential 
calculus  
in Sobolev spaces $ \Hds{s} $ with a regularity $ s $ large enough,  $ s \geq \sigma $
with  $ \sigma \gg \rho $.  Moreover we look for classical solutions 
 $ U(t) $ of the capillarity-gravity water waves equations which are $ K $-times differentiable in time, 
 with derivatives  $ \pa_t^k U $ in $\Hds{\sigma-\frac{3}{2}k}$, $ k = 0, \ldots, K $, 
 and thus we need  $ \sigma \gg K $. 
  Finally, the reason why $ K \gg \rho $ is that, in order to obtain system \eqref{426}, which has constant coefficients
  up to $ \rho $-smoothing remainders, see also \eqref{10}, 
  we shall  perform a  large number of 
  conjugations, large with respect to $ \rho $, of the water waves system \eqref{113}, 
  to obtain para-differential   operators which are enough regularizing.
Each conjugation consumes one time derivative, so that we shall need $ K \gg \rho $.

We also notice that our solutions will remain in a
small ball of the space of $C^k$ functions of time with values in $\Hds{\sigma-\frac{3}{2}k}$, and  their time derivative of
order $k$ will be continuous in $t$ with values in $\Hds{s-\frac{3}{2}k}$, for $k\leq K$.

\smallskip

We define below several classes of symbols and operators. We shall stick to the following convention: for $p$  in $\N$,
we shall denote by $\tilde{A}_p$ (with $A$ replaced by several letters, affected eventually with some indices) classes of
symbols, i.e. functions of $(x,\xi)$,  (or operators), 
which are symmetric $p$-linear maps of some argument $\Ucal = (U_1,\dots,U_p)$
belonging to a functional space, where each $ U_j = U_j (t) $ are functions of time.  
The notation $A_N[r]$ with $N$ in $\N^*$, $r>0$ will be used for 
symbols (or operators) depending in a non-homogeneous way
 on a function $U$ staying in the ball of center zero, 
radius $r$ in a convenient space, and vanishing at order $N$ when $U$ goes
to zero in that space. Finally, $\Sigma A_p[r,N]$ will denote sums of functions homogeneous in $ U $ 
of degree $q$, $p\leq q\leq N-1$,
defined as restrictions at $U_1=\cdots=U_q = U$ of a $ q $-linear map  of $\tilde{A}_q$, and of an element of $A_N[r]$. In the whole 
monograph, for the
homogeneous symbols (or operators) the dependence on time $ t $ will enter  only through the function 
$ U = U( t ) $, while  the non-homogeneous symbols
(or operators)  in $A_N[r]$ may depend explicitly on time $ t $. 

\medskip

Before giving the first instance where these conventions will be used, we introduce some notation. 
If $K$ is in $\N$, $I$ is an interval of $\R$, symmetric with respect to the origin $ t = 0 $, 
 $\sigma$ is in $\R$, 
we denote by \index{Ca@$\CKH{\sigma}{\C^2}$ (Space of functions)} $\CKH{\sigma}{\C^2}$
the space of continuous functions $U $ of $t\in I$ with
values in $\Hds{\sigma}(\Tu,\C^2)$, which are 
differentiable in $t$ with values in $\Hds{\sigma-\frac{3}{2}k} 
(\Tu,\C^2) $, for any $ 0 \leq k\leq K$, 
with continuous derivatives 
$\partial_t^k U \in \Hds{\sigma-\frac{3}{2}k} (\Tu,\C^2) $. 
The space  $\CKH{\sigma}{\C^2}$ is endowed with the norm
\begin{equation}
  \label{eq:212}
  \sup_{t\in I}\nnorm{U(t,\cdot)}_{K,\sigma} \quad \textrm{ where } 
  \quad
 \nnorm{U(t,\cdot)}_{K,\sigma} = \sum_{k=0}^K\norm{\partial^k_t U(t,\cdot)}_{\Hds{\sigma-\frac{3}{2}k}}.
\end{equation}
We denote by \index{Cb@$\CKHR{\sigma}{\C^2}$ (Space of functions)} $\CKHR{\sigma}{\C^2}$ 
the subspace of $\CKH{\sigma}{\C^2}$ made of the functions of $ t $ 
with values in $\Bigl\{ U\in
\Hds{\sigma}(\Tu;\C^2); U = \bigl[\begin{smallmatrix}u\\\bar{u}\end{smallmatrix}\bigr]\Bigr\}$.

If  $ \Ucal = (U_1,\dots,U_p)$ with $U_j$ in $\Hds{\infty}(\Tu,\C^2)$,
respectively if $U$ is in $\CKH{\sigma}{\C^2}$, we set
\begin{equation}
  \label{eq:213}
\begin{split}
     \index{Gla@$\Gcal^\sigma_{0,p}(\Ucal)$ (Product of norms)} \Gcal^\sigma_{0,p}(\Ucal) =
 \prod_{j=1}^p\norm{U_j}_{\Hds{\sigma}} \, , \\
\index{Gl@$\Gcal^\sigma_{K,p}(U)$ (Norm at power $p$)} \Gcal^\sigma_{K,p}(U,t) =
 \nnorm{U(t,\cdot)}^p_{K,\sigma}\, .
\end{split}
\end{equation}
When $p=0$, we set $\Gcal^\sigma_{K,0} = 1 $ by convention. If $\Ucal = (U_1,\dots,U_p)$ is a $p$-tuple of functions, $n =
(n_1,\dots,n_p)$ an element of $(\N^*)^p$, we set 
\begin{equation}
  \label{eq:213a}
  \Pin{}\Ucal = (\Pin{1}U_1,\dots,\Pin{p}U_p)
\end{equation}
where $\Pin{j}$ is the spectral projector defined before \eqref{121}. Finally, let us mention that in chapters~5 and~6 we
shall use similar notations where $U_j$ will be in $\Hds{\sigma}(\Tu,\R)$ or $\Hds{\sigma}(\Tu,\R^2)$.
\begin{definition}{\bf ($p$-homogeneous symbol)}  
  \label{211}
Let $m$ in $\R$, $p$ in $\N^*$. We denote by \index{Ga@$\Gt{m}{p}$ (Space of homogeneous symbols)} $\Gt{m}{p}$ the space of
symmetric $p$-linear maps from $\Hds{\infty}(\Tu,\C^2)^p$ to the space of $C^\infty$ functions 
in  $(x,\xi)\in\Tu\times\R$
\[\Ucal \to ((x,\xi)\to a(\Ucal; x,\xi ))\]
satisfying the following: There is $\mu\geq0$ and for  any $\alpha, \beta\in \N$, there is $C>0$ and for
any $\Ucal$ in $\Hds{\infty}(\Tu,\C^2)^p$, any 
$n = (n_1,\dots,n_p)$ in $(\N^*)^p $,
\begin{equation}
  \label{eq:214}
  \abs{\partial^\alpha_x\partial^\beta_\xi a(\Pin{}\Ucal;x,\xi)} \leq
  C\abs{n}^{\mu+\alpha}\absj{\xi}^{m-\beta}\Gcalsm{0}{0,p}{\Pin{}\Ucal} \, . 
\end{equation}
Moreover, we assume that if, for some $(n_0,\dots,n_p)$ in $ 
\N\times(\N^*)^p $, 
\begin{equation}
  \label{eq:215}
  \Pin{0}a(\Pin{1}U_1,\dots,\Pin{p}U_p;\cdot)\not\equiv 0,
\end{equation}
then there is a choice of signs $\epsilon_0, \ldots, \epsilon_p \in \{-1,1\}$ such that $\sum_0^p\epsilon_jn_j = 0$.

When $p=0$, we denote by $\Gt{m}{0}$ the space of constant coefficient symbols $\xi\to a(\xi)$, that satisfy inequality
\eqref{214} for $\alpha =0$, with in the right hand side the $\abs{n}$ factor replaced by one, and with the convention
$\Gcalsm{0}{0,0}{\Pin{}\Ucal} =1$. 
\end{definition}
\textbf{Remarks}: 
$\bullet$
     In the sequel, we shall consider functions $\Ucal = (U_1,\dots,U_p)$ where $U_j$ depends also on
time $t$, so that the corresponding symbols in the above definition are functions of $(t,x,\xi)$ that we  denote also by 
$a(\Ucal;t,x,\xi)$.

$\bullet$
 If the dependence of $a$ on $(U_1,\dots,U_p)$ is polynomial with constant coefficients in $x$, then, as $\Pin{j}U_j$ is a linear
combination of $e^{in_j x}$,  $e^{-in_j x}$, condition \eqref{215} holds automatically. 

$\bullet$
 If $a \in {\tilde \Gamma}^m_p$ and $b \in {\tilde \Gamma}^{m'}_q $ then 
$a b \in {\tilde \Gamma}^{m+m'}_{p+q} $.  
If $a \in {\tilde \Gamma}^m_p $ then  $ \partial_x a  \in {\tilde \Gamma}^{m}_{p} $ and 
$ \partial_\xi a  \in {\tilde \Gamma}^{m-1}_{p} $.

$\bullet$
Let $ K \in \N $ and $ \sigma_{0} \geq \frac32 K + \mu + \frac12 $. 
For $ \sigma \geq \sigma_0 $ let consider a function 
$$ 
U \in C_*^K  (I,\Hds{\sigma}(\Tu,\C^2)) \, .
$$ 
Take $ a $ in $ {\tilde \Gamma}^m_p $, $ p \in \N^* $.  
We claim that, for all $ 0 \leq k \leq K $, $ \alpha \leq \sigma - 
\sigma_0 $, $ \beta \in \N $, the function
\begin{equation}\label{eq:hom-nonhom}
a( \underbrace{U; \ldots, U}_{p}; t,x,\xi) = a( U (t, \cdot) , \ldots, U (t, \cdot) ; 
x,\xi) 
\end{equation}
satisfies the estimate  
\begin{equation}
  \label{eq:217}
  \abs{\partial^k_t\partial^\alpha_x\partial^\beta_\xi a(U,\dots,U;t,x,\xi)} 
  \leq C\absj{\xi}^{m-\beta}\nnorm{U(t,\cdot)}_{k,\sigma_0+\alpha} 
  \nnorm{U(t,\cdot)}_{k,\sigma_0}^{p-1}.
\end{equation}
Indeed, 
if we make act $k$ time derivatives on $ a $, we have, 
by the  $ p $-linearity of $ a $ in each argument $ U $,  
\begin{multline*}
\partial_{t}^{k} \partial_{x}^{\alpha} \partial_{\xi}^{\beta} 
a(U(t, \cdot ), \ldots, U(t, \cdot ); x,\xi)  = \\
\sum_{n_1, \ldots, n_p \in \N^*}\sum_{k_{1}+ \ldots +
k_{p} = k } C_{k_1, \ldots, k_p } 
\partial_{x}^{\alpha} \partial_{\xi}^{\beta} a( \partial_t^{k_1} 
\Pi_{n_1} U (t), \ldots, 
\partial_t^{k_p} \Pi_{n_p} U (t ) ; x,\xi)
\end{multline*}
for suitable binomial coefficients $ C_{k_1, \ldots, k_p }$. By 
\eqref{214} and recalling the definition of the norm \eqref{212}, 
each term in the right hand side of the above sum, with
for instance $n_1\geq n_2,\dots,n_p$, 
is bounded by 
\begin{multline*}
    C \langle \xi \rangle^{m - \beta } \abs{n}^{\mu+\alpha}\prod_1^p\norm{\partial_t^{k_j}
   \Pin{j}U(t,\cdot)}_{L^2} \leq  \\
    C\langle \xi \rangle^{m - \beta } n_1^{\mu - \sigma_{0}  + \frac32 k_1 } 
    c^1_{n_1} \| \partial_t^{k_1} U(t) \|_{ {\dot H}^{\sigma_0+\alpha 
    - \frac32 k_1 }} 
    \prod_{2}^p n_j^{-\sigma_0 + \frac32 k_j } c^j_{n_j} 
\| \partial_t^{k_j} U(t) \|_{ {\dot H}^{\sigma_0 - \frac32 k_j}}  \\
   \leq 
    C\langle \xi \rangle^{m - \beta } 
    n_1^{\mu-\sigma_0+\frac{3}{2}k}\prod_2^p n_j^{-\sigma_0} 
    c^1_{n_1}\dots
c^p_{n_p} \nnorm{U(t,\cdot)}_{k,\sigma_0+\alpha} \nnorm{U(t,\cdot)}_{k,\sigma_0}^{p-1}
\end{multline*}
where $(c^j_{n_j})_{n_j}$ 
stands for some $\ell^2$ sequence. Since  $\sigma_0-\frac{3}{2}K>\mu+\frac{1}{2}$, 
we obtain \eqref{217}.  In conclusion, 
the function $ a(U, \ldots, U; t, x, \xi)$ is a symbol of a 
pseudo-differential operator with limited smoothness in $x$, with a 
control of its semi-norms in terms of $ U $,   vanishing in $ U $ with degree $ p $. 
Note that the bound in 
\eqref{217} is tame, namely it is
linear in the highest norm  $ \nnorm{U(t,\cdot)}_{k,\sigma_0 +\alpha} $.

\medskip

We define next symbols $ a(U;t,x,\xi) $  for which the $U$-dependence is not homogeneous, using estimates 
of type \eqref{217}, where moreover
we allow also $K'$ extra time derivatives in the right hand side.
For $\sigma\in \R$, $k\in \N$, $r>0$, $I$ an interval of $ \R $, we 
define
\begin{equation}
  \label{eq:216}
  \index{Ba@$B^k_\sigma(I,r)$ (Ball in the space of functions of $(t,x)$)}  
B^k_\sigma(I,r) = \Big\{ U\in C^k_*(I,\Hds{\sigma}(\Tu,\C^2)); \ \sup_{t\in I}\nnorm{U(t,\cdot)}_{k,\sigma}<r \Big\}.
\end{equation}

\begin{definition}{\bf (Non-homogeneous symbol)}
  \label{212}
Let $m$ be in $\R$, $p$ in $\N$, $K'\leq K$ in $\N$, $r>0$. One denotes by $\Gr{m}{K,K',p}$ \index{Gb@$\Gr{m}{K,K',p}$ (Space of non-homogeneous symbols)}
the space of functions $(U,t,x,\xi) \to a(U;t,x,\xi)$, defined for $U$ in $B_{\sigma_0}^K(I,r)$, for some large enough
$\sigma_0$, $t$ in $I$, $x$ in $\Tu$, $\xi$ in $\R$, with complex values, such that for any $0\leq k\leq K-K'$, any $\sigma\geq
\sigma_0$, there are $C>0$, $r(\sigma)\in ]0,r[$ and for any $U$ in  $B_{\sigma_0}^K(I,r(\sigma))\cap\CkH{\sigma}{\C^2}$, any
$\alpha, \beta$ in $\N$, with $\alpha\leq \sigma-\sigma_0$
\begin{equation}
  \label{eq:218}
  \abs{\partial_t^k\partial^\alpha_x\partial^\beta_\xi a(U;t,x,\xi)} \leq C\absj{\xi}^{m-\beta}
  \Gcals{\sigma_0}{k+K',p-1}{U}\Gcals{\sigma}{k+K',1}{U},
\end{equation}
(where, if $p=0$, the right hand side has to be replaced by 
$C\absj{\xi}^{m-\beta}  $).
\end{definition}
Notice that  the above non-homogeneous symbols  $ a(U;t,x,\xi) $ 
depend explicitly on the time variable $ t $ (not only through the function
$ U(t)$ as for the homogeneous symbols). 
We introduce also a definition for the subclass of non-homogeneous symbols that depend
on time only through $U$.

\begin{definition}{\bf (Autonomous non-homogeneous symbol)}
  \label{212bis}
We denote by $\Gra{m}{K,0,p}$ \index{Gba@$\Gra{m}{K,0,p}$ (Space of autonomous
  symbols)} the subspace of  $\Gr{m}{K,0,p}$ made of the non-homogeneous symbols $(U,x,\xi)\to a(U;x,\xi)$ 
 that satisfy estimates \eqref{218} with $K'=0$,
the time dependence being only through $U = U( t) $.
\end{definition}

\smallskip

\noindent
\textbf{Remarks}: $\bullet$ By \eqref{217}, a symbol $ a ({\cal U}; \cdot ) $ of $\Gt{m}{p}$ 
defines,  by restriction to the diagonal, the symbol $ a(U, \ldots, U; \cdot ) $ of $\Gra{m}{K,0,p}$ for
any $r>0$. 

$\bullet$ If $a$ is in $\Gr{m}{K,K',p}$ with $K'\leq K-1$, then $\partial_ta$ is in $\Gr{m}{K,K'+1,p}$.

$\bullet$ If $ a  \in \Gamma_{K,K',p}^m [r]  $ then  
$ \partial_\xi a   \in \Gamma_{K,K',p}^{m-1} [r]  $ and 
$ \partial_x a  \in \Gamma_{K,K',p}^m [r]  $, by the fact that we may increase the value of $  \sigma_0 $ 
in Definition \ref{212}.

$\bullet$ If $ a \in \Gamma_{K,K',p}^m [r] $ and $ b \in \Gamma_{K,K',q}^{m'} [r] $ then 
$ a b \in \Gamma_{K,K',q+p}^{m+m'} [r]  $. 
Let us prove this remark. 
Notice that if $\alpha = \alpha_1+\alpha_2$, we
have 
\begin{equation}
  \label{eq:218a}
  \nnorm{U}_{k_1,\alpha_1+\sigma_0}\nnorm{V}_{k_2,\alpha_2+\sigma_0} \leq
  C\bigl(\nnorm{U}_{k_1,\alpha+\sigma_0}\nnorm{V}_{k_2,\sigma_0} + \nnorm{U}_{k_1,\sigma_0}\nnorm{V}_{k_2,\alpha+\sigma_0}\bigr).
\end{equation}
Actually, recalling the definition \eqref{212} 
of the norm $\nnorm{\cdot}_{K,\sigma}$, we have
\begin{multline*}
  \nnorm{U}_{k_1,\alpha_1+\sigma_0}\nnorm{V}_{k_2,\alpha_2+\sigma_0}  = \
  \sum_{k_1'=0}^{k_1}\sum_{k'_2=0}^{k_2}
  \norm{\partial_t^{k'_1}U}_{{\dot H}^{\alpha_1+\sigma_0 -\frac{3}{2}k'_1}}
  \norm{\partial_t^{k'_2}V}_{ {\dot H}^{ \alpha_2+\sigma_0 
  -\frac{3}{2}k'_2} } \\
  \leq \sum_{k_1'=0}^{k_1}\sum_{k'_2=0}^{k_2} 
  \norm{\partial_t^{k'_1}U}_{ {\dot H}^{\alpha+\sigma_0 -\frac{3}{2}k'_1} }
\norm{\partial_t^{k'_2}V}_{ {\dot H}^{ \sigma_0 -\frac{3}{2}k'_2 } } +\\
\norm{\partial_t^{k'_1}U}_{ {\dot H}^{ \sigma_0 -\frac{3}{2}k'_1 } }
\norm{\partial_t^{k'_2}V}_{ {\dot H}^{\alpha+\sigma_0 -\frac{3}{2}k'_2 } }
\end{multline*}
by interpolation.  
By the last inequality, 
\eqref{218a} follows.
Taking now $a\in \Gr{m}{K,K',p}$ and $b\in \Gr{m'}{K,K',q}$, write for $0\leq k\leq K-K'$, $\alpha, \beta$ in $\N$, with
$\alpha\leq \sigma-\sigma_0$
\begin{multline*}
  \partial_t^k\partial_x^\alpha\partial_\xi^\beta\bigl(a(U;t,x,\xi)b(U;t,x,\xi)\bigr) \\
= \sum_{\substack{k_1+k_2=k,\alpha_1+\alpha_2=\alpha,\\\beta_1+\beta_2=\beta}}C_{k_1,\alpha_1,\beta_1}^{k_2,\alpha_2,\beta_2}
(\partial_t^{k_1}\partial_x^{\alpha_1}\partial_\xi^{\beta_1}a(U;t,x,\xi))(\partial_t^{k_2}\partial_x^{\alpha_2}\partial_\xi^{\beta_2}a(U;t,x,\xi))
\end{multline*}
for suitable binomial coefficients $C_{k_1,\alpha_1,\beta_1}^{k_2,\alpha_2,\beta_2}$. Using \eqref{218}, we get
\begin{multline*}
  \Abs{\partial_t^k\partial_x^\alpha\partial_\xi^\beta\bigl(a(U;t,x,\xi)b(U;t,x,\xi)\bigr)} \\
\leq C\sum_{\substack{k_1+k_2=k,\alpha_1+\alpha_2=\alpha,\\\beta_1+\beta_2=\beta}}\absj{\xi}^{m-\beta_1}\Gcals{\sigma_0}{k_1+K',p-1}{U}\Gcals{\alpha_1+\sigma_0}{k_1+K',1}{U}\\
\times\absj{\xi}^{m'-\beta_2}\Gcals{\sigma_0}{k_2+K',q-1}{U}\Gcals{\alpha_2+\sigma_0}{k_2+K',1}{U}\\
\leq
C\sum_{\alpha_1+\alpha_2=\alpha}\absj{\xi}^{m+m'-\beta}\nnorm{U}_{k+K',\sigma_0}^{p+q-2}\nnorm{U}_{k+K',\alpha_1+\sigma_0}\nnorm{U}_{k+K',\alpha_2+\sigma_0}\\
\leq C\absj{\xi}^{m+m'-\beta}\nnorm{U}_{k+K',\sigma_0}^{p+q-1}\nnorm{U}_{k+K',\alpha+\sigma_0}\\
\leq C\absj{\xi}^{m+m'-\beta}\Gcals{\sigma_0}{k+K',p+q-1}{U}\Gcals{\sigma}{k+K',1}{U}
\end{multline*}
where we used \eqref{218a} and took $\sigma\geq\alpha+\sigma_0$. Thus $ab$ satisfies the estimate \eqref{218} of 
elements of
$\Gr{m+m'}{K,K',p+q}$.
\medskip

Finally, we define symbols that 
may be expressed as the sum of homogeneous elements of
$\Gt{m}{q}$, $ q = p-1, \ldots, N - 1$,  plus a non-homogeneous symbol vanishing in $ U $ at order $ N $. 
\begin{definition}{\bf (Symbols)}
  \label{213}
Let $m$ be in $\R$, $p$ in $\N$, $K'\leq K$ in $\N$, $r>0$, $N$ in $\N$ with $p\leq N$.  One denotes by $\sG{m}{K,K',p}{N}$ \index{Gc@$\sG{m}{K,K',p}{N}$ (Space of symbols)}
the space of functions $(U,t,x,\xi) \to a(U;t,x,\xi)$ such that there are homogeneous symbols $a_q$ in $\Gt{m}{q}$ 
for $q = p,\dots,N-1$, and 
a non-homogeneous symbol $a_N$ in $\Gr{m}{K,K',N}$,  so that
\begin{equation}
  \label{eq:219}
  a(U;t,x,\xi) = \sum_{q=p}^{N-1} a_q(U,\dots,U; x,\xi) + a_N(U;t,x,\xi) \, .
\end{equation}
We set $\sG{-\infty}{K,K',p}{N} = \bigcap_m \sG{m}{K,K',p}{N}$.

We define the subclasses of autonomous symbols 
$\sGa{m}{K,0,p}{N}$ \index{Gca@$\sGa{m}{K,0,p}{N}$ (Space of autonomous symbols)} by formula
\eqref{219}, where $a_N$ is taken in the class  $\Gra{m}{K,0,N}$ of Definition~\ref{212bis}. We set finally $\sGa{-\infty}{K,0,p}{N} = \bigcap_m \sGa{m}{K,0,p}{N}$.
\end{definition}
For $ p = N $ we mean that the symbol $ a = a_N $ 
is purely non-homogeneous. 

\smallskip

\noindent
\textbf{Remarks}: $ \bullet $
We have the following simple inclusions
\begin{align*}
& m_1 \leq m_2 \, , \qquad \sG{m_1}{K,K',p}{N} \subseteq \sG{m_2}{K,K',p}{N} \, , \\
& p_1 \leq p_2 \, , \ \ \qquad \sG{m}{K,K',p_2}{N} \subseteq \sG{m}{K,K',p_1}{N} \, , \\
& K_1' \leq K_2' \, , \qquad \sG{m}{K,K_1',p}{N} \subseteq \sG{m}{K,K_2',p}{N} \, ,
\end{align*}
as well as the similar ones for autonomous symbols.

$ \bullet $
If $ a \in \sG{m}{K,K',p}{N}$ and  $ b \in \sG{m'}{K,K',p'}{N}$ then $ ab \in \sG{m+m'}{K,K',p+p'}{N} $. 
A similar property holds for autonomous symbols.

\section{Quantization of symbols}

If $(x,\xi)\to a(x,\xi)$ is a smooth symbol, one defines, for every 
$ \tau \in [0,1] $,  a quantization of the symbol
$a$  as the operator acting on test
functions $u$ (or on tempered distributions) by
\begin{equation}\label{eq:quantiz-t}
  \index{Oa@$\mathrm{Op}_\tau$ ($\tau$-quantization)} 
  \mathrm{Op}_\tau (a)u = 
  \frac{1}{2\pi}\int_{\R\times\R}e^{i(x-y)\xi} a(\tau x+(1 - \tau ) y,\xi)u(y)\,dyd\xi \, .
\end{equation}
In the case $ \tau =\frac{1}{2}$, one gets the
 Weyl quantization
\begin{equation}
  \label{eq:2110}
  \index{Ob@$\opw$ (Weyl quantization)} \opw(a)u 
  = \frac{1}{2\pi}\int_{\R\times\R}e^{i(x-y)\xi} a\Bigl(\frac{x+y}{2},\xi\Bigr)u(y)\,dyd\xi 
\end{equation}
and for $ \tau = 1 $ the usual quantization
$$
 \mathrm{Op} (a) u =   \frac{1}{2\pi}\int_{\R\times\R}e^{i(x-y)\xi} a(x,\xi)u(y)\,dyd\xi \, .
$$
These formulas are meaningful, in particular,  for $u$ in $C^\infty(\Tu)$ (identifying $u$ to a $2\pi$-periodic function, and thus
to a tempered distribution). If we decompose $u$ in Fourier series as 
$$
u(x) = \sum_{n\in  \Z}\hat{u}(n)\frac{e^{inx}}{\sqrt{2\pi}} \, , 
$$ 
 we may compute the oscillatory integral \eqref{quantiz-t}, and we 
 find out that  the action of $\mathrm{Op}_\tau (a)$  on $u$ is given by the periodic function
\[
\mathrm{Op}_\tau (a)u= \frac{1}{\sqrt{2\pi}}\sum_{k\in \Z}\Bigl(\sum_{n\in
  \Z}\hat{a}(k-n,(1-\tau)k+\tau n)\hat{u}(n)\Bigr)\frac{e^{ikx}}{\sqrt{2\pi}} 
  \]
where $\hat{a}(k,\xi)$ stands for the Fourier coefficients of the periodic function $x\to a(x,\xi)$. 
In particular, the Weyl quantization is given by
\begin{equation}
\label{eq:2110a}
\opw(a)u = \frac{1}{\sqrt{2\pi}}\sum_{k\in \Z}\Bigl(\sum_{n\in
  \Z}\hat{a}\Bigl(k-n,\frac{k+n}{2}\Bigr)\hat{u}(n)\Bigr)\frac{e^{ikx}}{\sqrt{2\pi}} ,
\end{equation}
and the usual quantization, corresponding to $ \tau =1$, by
\be\label{eq:usual-Q}
\begin{split}
\mathrm{Op}(a)u =  
\frac{1}{\sqrt{2\pi}}\sum_{k\in \Z}\Bigl(\sum_{n\in  \Z}\hat{a}(k-n,n)\hat{u}(n)\Bigr)\frac{e^{ikx}}{\sqrt{2\pi}} 
\\ 
= \sum_{n\in\Z}a(x,n)\hat{u}(n)\frac{e^{inx}}{\sqrt{2\pi}} \, .
\end{split}
\ee
If the symbol $ a = a(\xi ) $ does not depend of $ x $,  
then, for any $ \tau \in [0, 1 ] $, the associated operator
 $\mathrm{Op}_\tau (a) =  \mathrm{Op} (a)  = a(D) $  is the usual Fourier  multiplier.

The above formulas allow to transform the symbols between different quantization,  
in particular  we have 
\be\label{eq:clwe} 
\mathrm{Op}(a) = \opw( b ) \, \quad {\rm where} \quad \hat{b}(n,\xi) = \hat{a}\bigl(n,\xi-\frac{n}{2}\bigr) \, , \ \forall n \in \Z \, . 
\ee
We shall need in Chapter~\ref{cha:6}  an asymptotic  expansion of the Weyl symbol $ b $ in terms of the symbol $ a $. Since 
the  distributional kernel of 
the operator ${\rm Op}^W (b) $ defined in \eqref{2110} is the oscillatory integral
$$
K(x,y) =  \frac{1}{2 \pi} \int_{\R} e^{i (x-y)\xi } b\big( \frac{x + y}{2}, \xi \big)  d \xi  
$$
we may recover the Weyl symbol $ b $ 
 from the kernel by the inverse formula 
\be\label{eq:symbl-from-Kernel}
b(x, \xi) = \int_\R K \big( x + \frac{t}{2}, x - \frac{t}{2} \big) e^{ - i t \xi} d t \, . 
\ee
Now, identifying an operator $ {\rm Op} (a) $ with its distributional 
kernel,  
we derive by  \eqref{symbl-from-Kernel} that 
\be\label{eq:def-e}
{\rm Op} (a) = 
 \frac{1}{2 \pi} \int_{\R} e^{i (x-y)\xi } a(x, \xi)  d \xi  = {\rm Op}^W (b) 
\ee
where $ b $ is the symbol 
$$
 b(x,\xi) = \frac{1}{2\pi}\int_{\R^2} e^{-iz\zeta}a\bigl(x+\frac{z}{2},\xi-\zeta\bigr)\,dzd\zeta \, . 
$$
\begin{lemma}
  \label{248}
 Let $a$ be in $\Gt{m}{p}$,  $\Ucal = (U_1,\dots,U_p)$ in  $\Hds{\infty}(\Tu,\C^2)^p$,  and consider the 
 function
\begin{equation}
  \label{eq:2432}
  b(\Ucal;x,\xi) = \frac{1}{2\pi}\int_{\R^2} e^{-iz\zeta}a\bigl(\Ucal;x+\frac{z}{2},\xi-\zeta\bigr)\,dzd\zeta
\end{equation}
(interpreted as an oscillatory integral). 
Then for any integer $ A \geq 1 $  we have the expansion
\begin{equation}
  \label{eq:2433}
   b(\Ucal;x,\xi) = \sum_{\alpha=0}^{A-1}\frac{(-1)^\alpha}{2^\alpha
    \alpha!}\bigl(\partial_x^\alpha D_\xi^\alpha a\bigr)(\Ucal;x,\xi) +  \tilde{b}(\Ucal;x,\xi) 
\end{equation}
where $ D_\xi = \frac{1}{i} \pa_\xi $ and $\tilde{b}$ satisfies  the bounds: for any 
$ \alpha, \beta \in \N $, 
$ n \in (\N^*)^p $, 
\begin{equation}
  \label{eq:2434}
  \abs{\partial_x^\alpha\partial_\xi^\beta \tilde{b}(\Pin{}\Ucal;x,\xi)} \leq 
  Q\biggl(\frac{\abs{n}}{\absj{\xi}}\biggr) \abs{n}^{\mu+\alpha+\beta} \absj{\xi}^{m-\beta}\Gcalsm{0}{0,p}{\Pin{}\Ucal}
\end{equation}
where $ Q $ is a polynomial of  
valuation larger or equal to $ A $ and degree at most $ 2 A - 1 $, and $\mu$ depends only on $m$.
\end{lemma}

\begin{proof}
Let $\chi$ be in $C_0^\infty(\R)$, equal to one close to zero, with small enough support. 
We split the integral in  \eqref{2432} as the sum of
\be\label{eq:int1}
\frac{1}{2\pi}\int_{\R^2} \chi (\zeta\absj{\xi}^{-1}) e^{-iz\zeta} a\bigl(\Ucal;x+\frac{z}{2},\xi-\zeta\bigr)\,dzd\zeta
\ee
and
\be\label{eq:int2}
\frac{1}{2\pi}\int_{\R^2} (1-\chi)(\zeta\absj{\xi}^{-1}) e^{-iz\zeta}a\bigl(\Ucal;x+\frac{z}{2},\xi-\zeta\bigr)\,dzd\zeta \, . 
\ee
We first consider the integral \eqref{int2} with $\Ucal$ replaced by $\Pin{}\Ucal$.
For any $ N_0 \in \N $, we perform $ N_0 $
$\partial_z$-integrations by parts in \eqref{int2} and two $\partial_\zeta$ integrations by
  parts. 
Thus we gain a decaying factor $\absj{z}^{-2}\absj{\zeta}^{-N_0}$ and, 
using  \eqref{214},   
the integral \eqref{int2} is bounded  by 
\[
\begin{split}
C\int_{\R}  \absj{\xi-\zeta}^m\absj{\zeta}^{-N_0}\1_{\{\abs{\zeta}\geq \delta \absj{\xi}\}}\,d\zeta  \ 
\abs{n}^{\mu + N_0}\Gcalsm{0}{0,p}{\Pin{}\Ucal} \\
 \leq C'  \absj{\xi}^{M(m)-N_0}  \abs{n}^{\mu  + N_0}  \Gcalsm{0}{0,p}{\Pin{}\Ucal} 
\end{split}
 \]
for any $ N_0 > m + 1 $, where we set $M(m) = m+1$ if $m>-1$, $M(-1) = \epsilon $ for a small 
$ \epsilon > 0 $ and $M(m) = 0$ if
$m<-1$.   
If we make act $\partial_x^\alpha\partial_\xi^\beta$ derivatives on 
\eqref{int2}, we get in the same way a bound in
\[
\absj{\xi}^{M(m-\beta) -N_0}\abs{n}^{\mu + N_0 + \alpha} \Gcalsm{0}{0,p}{\Pin{}\Ucal} \, .
\]
Taking  $N_0 = N'_0 + \beta + L$, for some $ L = L(m) $ depending only on $m$,  
we get a bound in 
\[
\absj{\xi}^{m+1-N'_0-\beta}\abs{n}^{\mu'+N'_0+\alpha+\beta} \Gcalsm{0}{0,p}{\Pin{}\Ucal}
\]
where $\mu' = \mu + L $. 
Consequently \eqref{int2} satisfies \eqref{2434} for any $A$ (up to a change of
notation for $\mu$).

Consider next \eqref{int1}.
We Taylor expand up to order $ 2 A - 1 $ 
the symbol $ a\bigl(\Ucal;x+\frac{z}{2},\xi-\zeta\bigr)  $ at the point 
$ (x, \xi ) $, 
writing the integrand in \eqref{int1} as the product of $ e^{-iz\zeta}$ and 
\begin{multline}
  \label{eq:2436}
    \sum_{ \alpha+\beta \leq 2A-1}\frac{1}{\alpha!\beta!}\bigl(\partial_x^\alpha\partial_\xi^\beta
    a\bigr)(\Ucal;x,\xi)\Bigl(\frac{z}{2}\Bigr)^\alpha (-\zeta)^\beta \chi\Big(\frac{\zeta}{\absj{\xi}}\Bigr)\\
+\sum_{\alpha+\beta=2A}\frac{2A}{\alpha!\beta!} \int_0^1(1-\lambda)^{2A}\bigl(\partial_x^\alpha\partial_\xi^\beta
    a\bigr)(\Ucal; x+\frac{\lambda z}{2},\xi-\lambda \zeta)\\\times\Bigl(\frac{z}{2}\Bigr)^\alpha
    (-\zeta)^\beta \chi\Big(\frac{\zeta}{\absj{\xi}}\Bigr)\,d\lambda.
\end{multline}
Since the oscillatory integral 
$$
\frac{1}{2 \pi} \int_{\R^2} e^{- i z \zeta } (i z)^\alpha \zeta^\beta \chi ( \zeta  \langle \xi \rangle^{-1} ) d \zeta d z  = 
\alpha ! \delta_{\alpha \beta} 
$$ 
where $ \delta_{\alpha \beta } = 0 $ for $ \alpha \neq \beta  $ and 
$  \delta_{\alpha \beta } = 1 $ for $ \alpha = \beta $, 
when we plug the sum in \eqref{2436} for $\alpha+\beta\leq 2A-1$ inside the integral in \eqref{int1}, 
we get the sum 
$$
\sum_{\alpha=0}^{A-1}\frac{(-1)^\alpha}{2^\alpha   \alpha!}\bigl(\partial_x^\alpha D_\xi^\alpha a\bigr)(\Ucal;x,\xi) 
$$
in the right hand
side of \eqref{2433}. 

We are left with showing that the last term of \eqref{2436}, where we replace $\Ucal$ by $\Pin{}\Ucal$,
induces
in the integral \eqref{int1} contributions satisfying estimates \eqref{2434}. 
For each $ \lambda \in [0,1] $, the integrals we have to study are of the form
\begin{equation}
  \label{eq:compl1}
  \int_{\R^2}e^{-iz\zeta}c(\Pin{}\Ucal;x,\xi,z,\zeta)z^\alpha\zeta^\beta\,dzd\zeta
\end{equation}
where $c$ is supported for $ \abs{\zeta} < \delta \absj{\xi} $, 
and, by \eqref{214}, for any $ \alpha' , \beta' \in \N $, it satisfies estimates as
\begin{equation}
  \label{eq:compl2}
  \abs{\partial_z^{\alpha'}\partial_\zeta^{\beta'}c(\Pin{}\Ucal;x,\xi,z,\zeta)}\leq 
  C\abs{n}^{\mu+\alpha+\alpha'}\absj{\xi}^{m-\beta-\beta'} \Gcalsm{0}{0,p}{\Pin{}\Ucal} 
\end{equation}
with $\alpha+\beta = 2A$. Let us consider first the case $\alpha\leq \beta$. We write
$z^\alpha e^{-iz\zeta} = (-D_\zeta)^\alpha e^{-iz\zeta}$ and perform integrations by parts in $\zeta$ to write 
\eqref{compl1} as
\begin{multline*}
  \int_{\R^2} e^{-iz\zeta}D_\zeta^\alpha\bigl[\zeta^\beta c(\Pin{}\Ucal;x,\xi,z,\zeta)\bigr]\,dzd\zeta\\
= \sum_{\substack{\alpha_1 + \alpha_2
    =\alpha}}C_{\alpha_1,\alpha_2, \beta}\int_{\R^2}
    e^{-iz\zeta}\zeta^{\beta-\alpha_1}D_\zeta^{\alpha_2}c(\Pin{}\Ucal;x,\xi,z,\zeta)\,dzd\zeta \, . 
\end{multline*}
In each integral above, we write
$\zeta^{\beta-\alpha_1}  e^{-iz\zeta} = (-D_z)^{\beta-\alpha_1} e^{-iz\zeta}$ and perform integrations by parts in $z$. We end
up with integrals of the form
\[
\int_{\R^2}e^{-iz\zeta}D_z^{\beta-\alpha_1}D_\zeta^{\alpha_2}c(\Pin{}\Ucal;x,\xi,z,\zeta)\,dzd\zeta \, .
\]
This is of the form
\begin{equation}
  \label{eq:compl3}
  \int_{\R^2}e^{-iz\zeta}\tilde{c}(\Pin{}\Ucal;x,\xi,z,\zeta)\,dzd\zeta
\end{equation}
where $ \tilde{c} $ is supported for $ \abs{\zeta}<\delta \absj{\xi} $
and by \eqref{compl2}, 
and since  $\alpha+\beta = 2A$ and $\alpha_1+\alpha_2 = \alpha $, it satisfies, for any $ \beta_1 \in \N $, 
\be\label{eq:tildec-est}
\begin{split}
\abs{ \partial_\zeta^{\beta_1} \tilde{c}(\Pin{}\Ucal;x,\xi,z,\zeta)}\leq
  C\abs{n}^{\mu+\alpha+\beta-\alpha_1}\absj{\xi}^{m-\beta-\alpha_2- \beta_1} \Gcalsm{0}{0,p}{\Pin{}\Ucal} \\
\leq C \biggl(\frac{\abs{n}}{\absj{\xi}}\biggr)^{2A-\alpha_1} \abs{n}^\mu \absj{\xi}^{m - \beta_1} \Gcalsm{0}{0,p}{\Pin{}\Ucal} \, .  
\end{split}
\ee
Notice that, since $\alpha\leq\beta$, the exponent $2A-\alpha_1 $ is between $A$ and $2A$, for any $ 0 \leq 
\alpha_1 \leq \alpha $. We define next
the differential operator
$$
L = \bigl(1+z^2\absj{\xi}^2\bigr)^{-1}\bigl(1-z\absj{\xi}^2D_\zeta\bigr)
$$ 
so that
$L(e^{-iz\zeta}) = e^{-iz\zeta}$. We write $ e^{-iz\zeta} = L^2 ( e^{-iz\zeta})  $ in \eqref{compl3} 
and we integrate by parts twice. 
Since $ L^t = \big(1+z^2\absj{\xi}^2\big)^{-1}\big(1 + z\absj{\xi}^2D_\zeta\big) $,  by \eqref{tildec-est},  
we gain a factor $(1+\abs{z}\absj{\xi})^{-2}$, and 
taking the cut-off $\chi$ in \eqref{2436}  into account, we bound \eqref{compl3} with  
\[
\begin{split}
C \int_{\R^2}
(1+\abs{z}\absj{\xi})^{-2}\1_{\{ \abs{\zeta}<\delta \absj{\xi} \} }\,dzd\zeta \, 
\biggl(\frac{\abs{n}}{\absj{\xi}}\biggr)^{2A-\alpha_1} \abs{n}^\mu\absj{\xi}^m
\Gcalsm{0}{0,p}{\Pin{}\Ucal} \\
\leq C'  \, 
\biggl(\frac{\abs{n}}{\absj{\xi}}\biggr)^\gamma \abs{n}^\mu\absj{\xi}^m
\Gcalsm{0}{0,p}{\Pin{}\Ucal}
\end{split}
\]
for some $ A \leq \gamma\leq 2A$. We have proved that 
\eqref{compl1} satisfies  a bound of the form \eqref{2434} (with derivation indices  $\alpha=\beta =
0$). The case $\beta\leq\alpha$ is treated in the same way, as well as similar estimates when derivatives act on the
symbol. 
\end{proof}

If $a$ is a symbol in $\Gt{m}{p}$, 
it follows from \eqref{def-e} and 
\eqref{2433},  that 
\begin{equation}
  \label{eq:2437}
  \frac{1}{2\pi}\int_\R e^{i(x-y)\xi}a(\Ucal;x,\xi)\,d\xi = \sum_{\alpha=0}^{A-1}\opw(b_\alpha) + \opw(\tilde{b})
\end{equation}
with
\begin{equation}
  \label{eq:2438}
  b_\alpha(\Ucal;x,\xi) = \frac{(-1)^\alpha}{\alpha! 2 ^\alpha}\big( \partial_x^\alpha D_\xi^\alpha a \big)(\Ucal;x,\xi).
\end{equation}
Finally notice  that
\be\label{eq:avv}
\overline{\mathrm{Op}(a) [u]} =  \mathrm{Op}({\bar a}^\vee) [ \bar u] \qquad {\rm where} \qquad 
a^\vee (x, \xi) \stackrel{\textrm{def}}{=} a (x, - \xi) \, . 
\ee

We now define the {\it para-differential} quantization. Given $ p \in \N^* $, 
we fix some smooth functions, even with respect to  each of their arguments,
\[
\begin{split}
  \chi_p :\  &\R^p\times\R \to \R, \quad p\in \N \, , \\
\chi :\  &\R\times\R \to \R, 
\end{split}\]
satisfying,  for some small $\delta>0 $,  the support conditions
\begin{equation}
  \label{eq:2111}
  \begin{split}
    \supp\chi_p &\subset \{(\xi',\xi)\in \R^p\times\R;\abs{\xi'}\leq \delta\absj{\xi}\}, \chi_p\equiv 1 \textrm{ for }
    \abs{\xi'}\leq \frac{\delta}{2}\absj{\xi}, \\
\supp\chi &\subset \{(\xi',\xi)\in \R\times\R;\abs{\xi'}\leq \delta\absj{\xi}\}, \chi\equiv 1 \textrm{ for }
    \abs{\xi'}\leq \frac{\delta}{2}\absj{\xi}.
  \end{split}
\end{equation}
For $ p = 0 $ we set $\chi_0\equiv 1$.  We assume moreover that for any derivation indices  
$\alpha, \beta $,
\begin{equation} \label{eq:2111bis}\begin{split}
\abs{\partial_\xi^\alpha\partial_{\xi'}^\beta \chi_p(\xi',\xi)}&\leq
  C_{\alpha,\beta}\absj{\xi}^{-\alpha- |\beta|}, \ \forall \alpha\in \N, \forall \beta\in \N^p\\
\abs{\partial_\xi^\alpha\partial_{\xi'}^\beta \chi(\xi',\xi)}&\leq
  C_{\alpha,\beta}\absj{\xi}^{-\alpha-\beta},\ \forall \alpha, \beta\in \N \, .
\end{split}\end{equation}
An example of a function satisfying the last condition above, 
that will be used extensively below, is 
$\chi(\xi',\xi) = 
{\tilde \chi } (\xi'/\absj{\xi})$, where $ \tilde \chi  $ is
a $C^\infty_0(\R)$ function, with small enough support, equal to one on a neighborhood of zero.

\begin{definition} {\bf (Bony-Weyl quantization)} \label{Def:BW}
If $ a $ is a symbol in $\Gt{m}{p}$, respectively in $\Gr{m}{K,K',p}$, we set, using notation \eqref{213a}, 
\begin{equation}
  \label{eq:2112}
  \begin{split}
    a_{\chi_p}(\Ucal; x,\xi) &= \sum_{n\in (\N^*)^p}\chi_p(n,\xi)a(\Pin{}\Ucal; x,\xi)\\
a_{\chi}(U;t,x,\xi) &= \Fcal_{\xi'}^{-1}[\chi(\xi',\xi)\hat{a}(U;t,\xi',\xi)] = \chi (D, \xi) a(U;t,x,\xi)
  \end{split}
\end{equation}
where, in the last equality, $\hat{a}$ stands for the Fourier transform relatively to the $x$ variable 
and  $  \Fcal_{\xi'}^{-1} $ denotes the inverse Fourier transform relatively to the $ \xi ' $ variable. We get 
symbols $ a_{\chi_{p} }, a_{\chi} $ that
are smooth in $x$. If $ a $ is a symbol  in $\Gt{m}{p}$, respectively in $\Gr{m}{K,K',p}$,
we define  its Bony-Weyl quantization as the operators
\begin{equation}
  \label{eq:2113}\begin{split}
  \index{Ob@$\opbw$ (Bony-Weyl quantization)}\opbw(a(\Ucal; \cdot)) &= \opw(a_{\chi_p}(\Ucal; \cdot))\\
\opbw(a(U;t,\cdot)) &= \opw(a_{\chi}(U;t,\cdot)).
\end{split}\end{equation}
Finally, if $a$ is a symbol in $\sG{m}{K,K',p}{N}$, that we decompose as in \eqref{219}, we define
its Bony-Weyl quantization as the operator
\begin{equation}
  \label{eq:2114}
  \opbw(a(U;t,\cdot)) = \sum_{q=p}^{N-1} \opbw(a_q(U,\dots,U; \cdot)) + \opbw(a_N(U;t,\cdot)) \, .
\end{equation}
For symbols belonging to the autonomous subclass $\sGa{m}{K,0,p}{N}$, we shall not write the time dependence in
equality \eqref{2114}.
\end{definition}

\noindent
\textbf{Remarks}: 
$\bullet$
When $p=0$, elements of $\Gt{m}{0}$ are  
constant coefficient symbols $ \xi \to a(\xi) $ (see Definition \ref{211}),  and 
their Bony-Weyl quantization $  \opbw(a) = a(D) $ 
is nothing but the Fourier multiplier associated to $ a(\xi) $.

$\bullet$ If $ a $ is a symbol in $\Gt{m}{p}$ then $ a_{\chi_p} $ is in $\Gt{m}{p} $ as well. 

$\bullet$  Since the symbol $ a(U; t, x, \xi) $ is periodic in $ x $, the regularized 
function $ a_\chi $ in the second line in
\eqref{2112} can be written, by \eqref{usual-Q}, as
\be\label{eq:regul-per}
a_\chi (U; t, x, \xi)  = \chi (D, \xi) [a(x, \xi)] 
= \sum_{n\in\Z} \chi (n,\xi) \hat{a}(U; t, n, \xi) \frac{e^{inx}}{\sqrt{2\pi}} 
\ee
where  $ \hat{a}(U; t, n, \xi) $ denotes the Fourier coefficient of $ a $ with respect to $ x $. 
We can also write $ a_\chi (U; t, x, \xi) = K (\cdot, \xi) * a( U; t, \cdot, \xi ) $ where 
the kernel $ K(x, \xi ) =  \Fcal_{\xi'}^{-1} ( \chi ( \xi' , \xi ) ) $ satisfies, 
by \eqref{2111}-\eqref{2111bis},  
$$
\| \pa_\xi^\gamma K( \cdot , \xi )\|_{L^1 (\R)} \leq C_\gamma \langle \xi \rangle^{-\gamma} \, , \quad \forall \gamma \in \N \, .
$$
It follows that $ a_\chi $ satisfies estimates of the form \eqref{218} as the symbol $ a $. 

$\bullet$
We have  $ (a_{\chi_p} )^\vee =  a_{\chi_p} $, $ (a_{\chi})^\vee =  a_{\chi} $ and  so, 
 by \eqref{avv}, 
\be\label{eq:opbavv}
\overline{\opbw( a ) [u]} =  \opbw({\bar a}^\vee) [ \bar u] \, .
\ee

$\bullet$  If the symbol $ a(U; t, x ) $ does not depend on $ \xi $ then 
$ \opbw( a ) $ is called the para-product operator for the function $ a(U; t, x ) $. 

\smallskip

The above definitions are independent  
 of the cut-off functions $\chi, \chi_p$ satisfying  \eqref{2111}, 
 \eqref{2111bis},   up to smoothing operators that we
define now.

For a family $(n_1,\dots,n_{p+1})$ of $(\N^*)^{p+1}$ we denote by 
$$
\maxdn{1}{p+1} 
$$ 
the second largest among the integers $n_1,\dots,n_{p+1}$. 
\begin{definition}  \label{214}
(i) {\bf ($p$-homogeneous smoothing operator)} Let $p$ be
in $\N^*$,  $\rho$ in $\R_+$. 
We denote  by \index{Ra@$\Rt{-\rho}{p}$ (Homogeneous smoothing operators)} $\Rt{-\rho}{p}$  the space of
$(p+1)$-linear maps from the space $\Hds{\infty}(\Tu,\C^2)^p\times \Hds{\infty}(\Tu,\C)$ to the space $\Hds{\infty}(\Tu,\C)$, symmetric in $(U_1,\dots,U_p)$, of
the form \[(U_1,\dots,U_{p+1})\to
R(U_1,\dots,U_p)U_{p+1}\] that satisfy the following:

There is some $\mu\geq 0$, some $C>0$ and,  for any $\Ucal = (U_1,\dots,U_p)$ in $\Hds{\infty}(\Tu,\C^2)^p$, any $U_{p+1}$ in
$\Hds{\infty}(\Tu,\C)$, any  $n=(n_1,\dots,n_p)$ in $(\N^*)^p$, any $n_0, n_{p+1}$ in $\N^*$,
\begin{multline}
  \label{eq:2115}
  \norm{\Pin{0}R(\Pin{}\Ucal)\Pin{p+1}U_{p+1}}_{L^2} \\ 
  \leq C\frac{\maxdn{1}{p+1}^{\rho+\mu}}{\maxn{1}{p+1}^{\rho}}\Gcalsm{0}{0,p+1}{\Pin{}\Ucal,\Pin{p+1}U_{p+1}} \, .
\end{multline}
Moreover,  if
\begin{equation}
  \label{eq:2116}
  \Pin{0}R(\Pin{1}U_1,\dots,\Pin{p}U_p)\Pin{p+1}U_{p+1} \not\equiv 0
\end{equation}
then there is a choice of signs $\epsilon_0,\dots,\epsilon_{p+1} \in \{-1,1\}$ 
such that $\sum_0^{p+1}\epsilon_j n_j = 0$.

(ii) {\bf (Non-homogeneous smoothing operator)}
Let $N$ be in $\N^*$, $K'\leq K$ in $\N$, $ r >  0 $. We denote by \index{Rb@$\Rr{-\rho}{K,K',N}$ (Non-homogeneous smoothing operators)}~$\Rr{-\rho}{K,K',N}$ the space
of maps $(V,t,U)\to R(V;t)U$  defined on  $\Br{K}{}\times I\times\CKH{\sigma}{\C}$ for some $\sigma>0$, that are linear in $U$, and
such that for any $s$ with $s>\sigma$, there are $C>0$, $r(s)\in ]0,r[$ and for any $V$ in
$\Brs{K}{}{s}\cap\CKH{s}{\C^2}$, any $U$ in $\CKH{s}{\C}$, any $0\leq k\leq K-K'$, 
any $t\in I$, one has the estimate
\begin{equation}
  \label{eq:2117}\begin{split}
  \norm{\partial_t^k(R(V;t)U)(t,\cdot)}_{\Hds{s+\rho-\frac{3}{2}k}} \leq C\sum_{k'+k''=k}\Bigl(\Gcals{\sigma}{k'+K',N}{V}
    \Gcals{s}{k'',1}{U} \\ +  \Gcals{\sigma}{k'+K',N-1}{V} \Gcals{\sigma}{k'',1}{U} \Gcals{s}{k'+K',1}{V}\Bigr).
\end{split}\end{equation}

(iii) {\bf (Smoothing operator)}
Let $p, N$ be in $\N^*$, with $ p \leq N$, $K'\leq K$ in $\N$, $\rho$ in $\R_+$. 
We denote by \index{Rc@$\sR{-\rho}{K,K',p}{N}$
  (Smoothing operators)}~$\sR{-\rho}{K,K',p}{N}$ the space of maps $(V,t,U)\to R(V;t)U$ that may be written as
\begin{equation}
  \label{eq:2118}
  R(V;t)U = \sum_{q=p}^{N-1} R_q(V,\dots,V)U + R_N(V;t)U
\end{equation}
for some $R_q$ in $\Rt{-\rho}{q}$, $q = p,\dots,N-1$ and $R_N$ in $\Rr{-\rho}{K,K',N}$
(for $ p = N $ we mean that $ R = R_N $ is a purely non-homogeneous smoothing operator). 

(iv) {\bf (Autonomous smoothing operators)}
We define, with the notations of (ii) and (iii) above, the class of autonomous non-homogeneous smoothing operators
\index{Rba@$\Rra{-\rho}{K,0,N}$ (Autonomous non-homogeneous smoothing operators)}~$\Rra{-\rho}{K,0,N}$ as the subspace of
$\Rr{-\rho}{K,0,N}$ made of those maps $(U,V)\to R(U)V$ satisfying estimates \eqref{2117} with $K'=0$, the time dependence
being only through $U= U(t)$. In the same way, we denote by
 \index{Rca@$\sRa{-\rho}{K,0,p}{N}$
  (Autonomous smoothing operators)}~$\sRa{-\rho}{K,0,p}{N}$ the space of maps $(U,V)\to R(U)V$ of the form \eqref{2118} with $K'=0$,
where the last term belongs to $\Rra{-\rho}{K,0,N}$.
\end{definition}

We have the following simple inclusions
\begin{align*}
& \rho_1 \leq \rho_2 \, , \ \  \qquad \sR{-\rho_2}{K,K',p}{N} \subseteq \sR{-\rho_1}{K,K',p}{N} \, , \\
& p_1 \leq p_2 \, , \ \ \qquad \sR{-\rho}{K,K',p_2}{N}\subseteq \sR{-\rho}{K,K',p_1}{N} \, , \\
& K_1' \leq K_2' \, , \qquad \sR{-\rho}{K,K_1',p}{N} \subseteq \sR{-\rho}{K,K_2',p}{N} \, .
\end{align*}

\noindent
\textbf{Remarks}: $\bullet$ Notice 
that if $R$ is in $\Rt{-\rho}{p}$, $ p \geq N $, 
then  
$(V,U) \to R(V,\dots,V)U$ is in $\Rra{-\rho}{K,0,N}$. 
Actually, for any $0 \leq k\leq K $, 
by the multi-linearity of $ R $ in each argument and  \eqref{2115},  we have 
\begin{align*}
& \norm{\partial_t^k\Pin{0}R(V,\dots,V)U}_{L^2} \leq C \!\!\!\!\!\!\!\! \!\!\!
\sum_{n_1,\dots,n_{p+1} \atop 
{k_1+ \ldots + k_{p+1} = k} } \!\!\!\!\!\!\!\!\!\!
\norm{R( \partial_t^{k_1} \Pi_{n_1}V, \ldots,  \partial_t^{k_p} \Pi_{n_p}V) \partial_t^{k_{p+1}} \Pi_{n_{p+1}} U}_{L^2} \\
& \leq C 
\sum_{n_1,\dots,n_{p+1} \atop 
{k_1+ \ldots + k_{p+1} = k} } 
\frac{\maxdn{1}{p+1}^{\rho+\mu}}{\maxn{1}{p+1}^{\rho}} 
\prod_1^p \| \partial_t^{k_j} \Pi_{n_j} V \|_{L^2} \| \partial_t^{k_{p+1}} \Pi_{n_{p+1}} U \|_{L^2} \, . 
\end{align*}
Let us bound each term in the right hand side of the above sum with for 
instance $n_1\geq n_2\geq\cdots\geq n_{p+1}$.  Recalling the definition of the norms \eqref{212}-\eqref{213},  
we get, for $  s\geq\sigma>\frac{3}{2}k $,
\begin{equation}
  \label{eq:2119}\begin{split}
\norm{\partial_t^k\Pin{0}R(V,\dots,V)U}_{L^2} \leq  C  \sum_{n_1,\dots,n_{p+1}}
  \frac{n_2^{\rho+\mu}}{n_1^{\rho-\frac{3}{2}k}}n_1^{-s}n_2^{-\sigma}\prod_3^{p+1}n_j^{-\sigma}\prod_1^{p+1}c^j_{n_j}\\\times
 \sum_{k'+k''=k}\Gcals{s}{k',1}{V} \Gcals{\sigma}{k',p-1}{V}\Gcals{\sigma}{k'',1}{U} 
\end{split}\end{equation}
 where $(c^j_{n_j})_{n_j} $ stands for some $\ell^2$ sequences. 
The assumption \eqref{2116} implies that 
$ n_0 \leq (p+1) n_1$. Then,  if $\sigma \geq   
\max\{ \rho +  \mu + \frac{1}{2}, \frac{3K}{2}\} $, Young inequality for convolution of sequences implies that 
the product of the left hand side of \eqref{2119} by $n_0^{s+\rho-\frac{3}{2}k}$ gives a
$\ell^2$-sequence, and   
\begin{align*}
 \norm{\partial_t^k R(V, \ldots, V)U}_{\Hds{s+\rho-\frac{3}{2}k}} & \leq 
 C\sum_{k'+k''=k}\Gcals{\sigma}{k',p-1}{V} \Gcals{\sigma}{k'',1}{U} 
 \Gcals{s}{k',1}{V} \\
 & \leq 
 C\sum_{k'+k''=k}\Gcals{\sigma}{k',N-1}{V} \Gcals{\sigma}{k'',1}{U} 
 \Gcals{s}{k',1}{V} 
\end{align*}
since $ p \geq N $ and  $ \sup_{t \in I}\nnorm{V(t, \cdot )}_{K,\sigma} < r $ is small.
This is the second term in the right hand side of  \eqref{2117} with $ K' = 0 $. 
The first term appears when for instance $n_{p+1}\geq\cdots\geq n_1$. 

$ \bullet $ The composition of smoothing operators $ R_1 $ and  $ R_2 $ belonging, respectively, to the spaces $ \sR{-\rho}{K,K',p_1}{N}$ and 
$\sR{-\rho}{K,K',p_2}{N}$, is a smoothing operator $ R_1 R_2 $  in $ \sR{-\rho}{K,K',p_1+p_2}{N} $.  
 A similar statement holds for the autonomous subclasses.

\medskip

We study now the boundedness properties on Sobolev spaces 
of the paradifferential operators \eqref{2113}.

\begin{proposition} 
    {\bf (Action of a paradifferential operator)}
  \label{215}
(i) Let $ m \in \R $, $ p \in \N $. There is $ \sigma > 0 $ such that for any symbol $ a $ in $\Gt{m}{p}$, the map
\begin{equation}
  \label{eq:2120}
  (U_1,\dots,U_{p+1})\to \opbw(a (U_1,\dots,U_p;\cdot))U_{p+1}
\end{equation}
extends, for   any $s$ in $\R$, as a continuous $(p+1)$-linear map 
$$
\Hds{\sigma}(\Tu,\C^2)^p \times\Hds{s}(\Tu,\C) \to \Hds{s-m}(\Tu,\C) \, . 
$$ 
 Moreover,  there is  a constant $ C $, depending only on $s$ and on \eqref{214} with $ \alpha = \beta = 0 $, 
such that 
\begin{equation}
  \label{eq:2121}
  \norm{\opbw(a(\Ucal;\cdot))U_{p+1}}_{\Hds{s-m}}\leq C\Gcalsm{\sigma}{0,p}{\Ucal}\norm{U_{p+1}}_{\Hds{s}} 
\end{equation}
where $\Ucal = (U_1,\dots,U_p)$. Finally, if  for some $(n_0,\dots,n_{p+1}) \in (\N^*)^{p+2}$,
\be\label{eq:pro-x-in}
\Pin{0}  \opbw(a( \Pi_n \Ucal;\cdot)) \Pin{p+1}U_{p+1} \not\equiv 0 
\ee
 then there is a choice of the signs $\epsilon_j\in \{-1,1\}$, $j=0,\dots,p+1$, such that 
$\sum_0^{p+1} \epsilon_j n_j = 0$, and the indices satisfy 
\begin{equation}
  \label{eq:2122}
n_0 \sim n_{p+1}  \, , \quad  n_j \leq C \delta n_0, \quad  n_j \leq C \delta n_{p+1}, \quad   j=1,\dots,p \, .
\end{equation}

(ii) Let $r>0$, $m\in \R$, $p\in \N$, $K'\leq K\in \N$, $a$ in $\Gr{m}{K,K', p}$. There is $\sigma>0$ and for any $U$ in $\Br{K}{}$,
the operator $\opbw(a(U;t,\cdot))$ extends, for any $s$ in
$\R$, 
 as a bounded linear operator from the space 
 $$
 C_*^{K-K'}(I,\Hds{s}(\Tu,\C)) \to C_*^{K-K'}(I,\Hds{s-m}(\Tu,\C)) \, .
 $$
 Moreover,  there is  a constant $ C $, depending only on $s$, $r$ and on \eqref{218} with $ 0 \leq \alpha  \leq 2, \beta = 0$, 
such that, for any $ t $ in $ I $, any $ 0\leq k\leq K-K'$,  
\begin{equation}
  \label{eq:2123}
  \norm{\opbw(\partial_t^ka(U;t,\cdot))}_{\Lcal(\Hds{s},\Hds{s-m})}\leq C\Gcals{\sigma_0}{k+K',p}{U} \, ,
\end{equation}
so that 
$$
\nnorm{ \opbw(a(U;t,\cdot)) V(t) }_{s-m} \leq C \Gcals{\sigma}{K,p}{U} \nnorm{ V(t) }_{s} \, . 
$$
\end{proposition}

\noindent
\textbf{Remarks}: $\bullet$  The relation $ n_0 \sim n_{p+1}$ in \eqref{2122} shows that the action of $ \opbw(a ) $
does not enlarge much the support of the Fourier transform of functions. 

$\bullet$ Notice that the paradifferential operator $\opbw(a)$ acts on  homogeneous spaces of functions modulo constants.

\begin{proof}
(i)   Let us prove  \eqref{2121}. Fixed $ {\cal U } = (U_1,\dots,U_p) $ 
we denote by 
\be\label{eq:func-b}
 b(x,\xi) = a_{\chi_p}({\cal U};  x,\xi) = \sum_{n\in (\N^*)^p}\chi_p(n,\xi)a(\Pin{}\Ucal;x,\xi)
 \ee
the regularized symbol defined in 
\eqref{2112} (without indicating the time dependence on $ t$). 
 Then $b(x, \xi) $ is a periodic function of $x$ whose Fourier coefficients $\hat{b}(\ell,\xi) =
\frac{1}{\sqrt{2\pi}}\int_\Tu e^{-i\ell x}b(x,\xi)\,dx$ ($\ell\in \Z$) satisfy, 
according to \eqref{214}  for $ \alpha = 0 $, $ \beta = 0 $, the bound 
\begin{multline} 
\abs{\hat{b}(\ell,\xi)}  
\leq  C \sum_{n=(n_1,\dots,n_p) \in (\N^*)^{p}}\abs{n}^\mu \Gcalsm{0}{0,p}{ \Pi_n \Ucal}\absj{\xi}^{m}   \\
 \leq C \sum_{n \in (\N^*)^{p}}\abs{n}^\mu \prod_1^p n_j^{-\sigma}\Gcalsm{\sigma}{0,p}{  \Ucal}\absj{\xi}^{m}. 
\label{eq:sec-pas}
\end{multline}
Moreover, as $\hat{b}(\ell,\xi) = \widehat{\Pi_{\abs{\ell}}b}(\ell,\xi)$, it follows from \eqref{215}  that
in the sum \eqref{sec-pas}, the indices satisfy,
 for some choice of the signs  $\epsilon_j \in \{ -1, 1 \} $, 
\be\label{eq:ell-ep}
\sum_{j=1}^p\epsilon_jn_j = \pm\ell \, .
\ee
If we take $\sigma-\mu>1$, we deduce from \eqref{sec-pas} and \eqref{ell-ep} the estimate
\begin{equation}
  \label{eq:2123a}
  \abs{\hat{b}(\ell,\xi)}\leq c_\ell\absj{\xi}^{m}\Gcalsm{\sigma}{0,p}{\Ucal}
\end{equation}
for some $\ell^1$-sequence $(c_\ell)_{\ell}$. Notice also that there is $\delta' = O(\delta)$  such that 
\be\label{eq:support}
\hat{b}(\ell, \xi) \neq 0 \quad \Longrightarrow \quad \abs{\ell}\leq\delta'\absj{\xi} \, . 
\ee
Indeed, by \eqref{2111}, \eqref{2112},  in the sum \eqref{sec-pas}, in addition 
to the restriction $\sum_1^p\epsilon_jn_j = \pm\ell $, 
the indices satisfy also $ | n | \leq \delta \langle \xi \rangle $. 
Thus $\abs{\ell}\leq C | n | \leq C \delta \absj{\xi}$ holds on the support of $\hat{b}(\ell, \xi)$. 

By \eqref{2113} and \eqref{func-b} we have $ \opbw (a ({\cal U}; \cdot ))   = \opw(b)  $.
Denoting by $\hat{u}(n)$
with $ n $ in $\Z$ the Fourier coefficients of $u\in C^\infty(\Tu)$, it follows from 
\eqref{2110a} that 
\begin{equation}
  \label{eq:2124}
  \opw(b)u =  \frac{1}{\sqrt{2\pi}} \sum_{k \in \Z}
  \Big(\sum_{n' \in \Z} \hat{b}\Bigl(k-n',\frac{k+n'}{2}\Bigr)\hat{u}(n') \Big)  \frac{e^{i k x}}{\sqrt{2\pi}} 
\end{equation}
and, for the property on the support of $\hat{b}(\ell, \xi )$ proved in \eqref{support}, 
 the sum  \eqref{2124} 
is restricted  to the indices 
\be\label{eq:para-weyl}
\abs{k-n'}\leq \delta'\Absj{\frac{k+n'}{2}} \, . 
\ee
This implies, in particular,  
 that,  if $u$ is constant, so is $\opw(b)u$, i.e. $\opw(b)$ acts on  homogeneous spaces of functions modulo
constants. Moreover, according to \eqref{2123a}, 
\eqref{2124}, and the fact that, by \eqref{para-weyl}, the indices
 in the sum \eqref{2124} satisfy $ k \sim n' $, 
the $\Hds{s-m}(\Tu)$ norm of $\opw(b)u$ is bounded from
above by the product of $ C \Gcalsm{\sigma}{0,p}{\Ucal}$ and of
\[ 
\Norm{\absj{k}^{s-m}\sum_{n' \in \Z} c_{k-n'} \absj{n'}^m \abs{\hat{u}(n')} }_{\ell_k^2} \leq C
\Norm{\sum_{n' \in \Z}\abs{\hat{u}(n')}\absj{n'}^s c_{k-n'}}_{\ell_k^2} \, . 
\] 
If  $ u $ is in $\dot{H}^s$ then the sequence
$ ( \abs{\hat{u}(n')}\absj{n'}^s )_{n' \in \Z} $ is in $ \ell^2 $, and, since $ (c_\ell)_{\ell \in \Z} $ is 
in $ \ell^1  $, the Young inequality for sequences 
implies \eqref{2121}. 

By \eqref{2124} we  have that, for any $ n_0, n_{p+1} \in \N^* $,  
\begin{multline}
\Pin{0}\opbw(a(\Pin{}\Ucal;\cdot))\Pin{p+1}U_{p+1} =  \\
\frac{1}{\sqrt{2\pi}} \sum_{ \abs{k} = n_0 }
  \Big(\sum_{\abs{n'} = n_{p+1}} 
  \hat{b}\Bigl(k-n',\frac{k+n'}{2}\Bigr)\hat{U}_{p+1}(n') \Big)  \frac{e^{i k x}}{\sqrt{2\pi}} \label{eq:restriz}
\end{multline}
with the restriction $ \pm(k-n') = \sum_{1}^p \epsilon_j n_j $ on the indices
due to \eqref{ell-ep}. Thus if \eqref{pro-x-in} holds then $ \sum_0^{p+1} \epsilon_j n_j = 0 $.  
Let us finally prove \eqref{2122}. 
The   sum in \eqref{restriz} is restricted 
to the indices  satisfying \eqref{para-weyl}, so that $ k \sim n' $, and therefore $ n_0 = |k | \sim | n' | = n_{p+1} $.
Moreover $ | n | \leq \delta \Absj{\frac{k+n'}{2}} $ 
 and the inequalities \eqref{2122} follow by
\be \label{eq:rel-ind} 
|n |  \leq C \delta \abs{k}  \leq C \delta n_0 
\ee
since  $ |k| = n_0 $, and $ n_0 \sim n_{p+1} $. 
This concludes the proof.

(ii) The proof is similar.  
Setting $ b (U; t, x, \xi ) = \pa_t^k a_\chi (U; t,x, \xi ) $ for any $ 0 \leq k \leq K - K' $, we use that
\eqref{218} with $ |\alpha| \leq 2  $, $ \beta = 0 $, implies the bound
$$
  \abs{\hat{b}(\ell,\xi)}\leq \frac{C}{\langle \ell \rangle^2}\absj{\xi}^{m}\Gcalsm{\sigma}{k+K',p}{\Ucal}
$$
which replaces \eqref{2123a}. Then  \eqref{2123} is proved in the same  way. 
\end{proof}
\textbf{Remarks}: $\bullet$ 
In the above proof, we did not use any $\xi$ derivatives of the 
 symbol $ a $. The statement of Proposition \ref{215}-(i) applies 
 when $ a \in \Gt{m}{p} $ satisfies \eqref{214} with  just $ \alpha = \beta = 0 $, and  item (ii) 
applies when  $ a \in \Gr{m}{K,K', p} $ satisfies \eqref{218} for just $ |\alpha| \leq 2  $, $ \beta = 0 $. 

Notice also that we did not use  that  
$ \chi_p (n, \xi ) = 1$ for $ |n| \langle \xi \rangle^{-1} \leq \delta /2  $,
so that the above estimates hold with any cut-off function $\chi_p $, $ \chi $ satisfying just the 
support conditions introduced in \eqref{2111}. 

$\bullet$ The proof of the preceding proposition shows  that the definition \eqref{2113} of $\opbw(a)$
where $a$ is a symbol of $\Gt{m}{p}$ depends on the choice of 
 $\chi_p $ 
only modulo a smoothing operator in
$\Rt{-\rho}{p}$ for any $\rho$.  Actually, let $ \chi_p^{(1)}$, $ \chi_p^{(2)}$ 
be cut-off functions satisfying  
\eqref{2111} with two different small $ \delta_2 > \delta_1 > 0 $.  
We want to prove that, calling the difference $ \chi_p = \chi_p^{(2)} - \chi_p^{(1) }$, the operator
\be\label{eq:restoR}
R ( {\cal U}) = \opw \big( \sum_{n \in (\N^*)^p} \chi_p (n, \xi)   a(\Pin{}\Ucal; \cdot) \big) 
\ee
satisfies  \eqref{2115}. 
Applying  \eqref{2121} with $ s = m $ we derive
\be\label{eq:R-inter}
\| \Pin{0} R ( \Pin{}\Ucal ) \Pin{p+1}U_{p+1} \|_{L^2} \leq C 
n_1^\sigma \ldots n_p^\sigma
 n_{p+1}^m \Gcalsm{0}{0,p+1}{\Pin{}\Ucal,\Pin{p+1}U_{p+1}} \, . 
\ee
If $ \Pi_{n_0} R ( \Pin{}\Ucal ) \Pin{p+1}U_{p+1} \neq 0 $, then, by \eqref{restriz} 
and since  the indices in \eqref{restoR} satisfy 
$ \delta_1 \langle \xi \rangle  \leq |n |  \leq \delta_2 \langle \xi \rangle $, 
we deduce (arguing as for \eqref{rel-ind}) that 
$$
C_1 \delta n_{p+1} \leq |n | \leq C_2 \delta n_{p+1} \, . 
$$ 
Consequently $\maxdn{1}{p+1}\sim \maxn{1}{p+1}$ and 
by \eqref{R-inter} the estimate \eqref{2115} holds for any $\rho$ and $\mu  = p \sigma + m $. 
The property \eqref{2116}  follows as in the above proof of Proposition \ref{215}. 

A similar statement holds for $\opbw(a)$ when $a$ is in $\Gr{m}{K,K',p}$: when we compare two definitions coming from different
choices of $\chi$ in \eqref{2112}, the difference between both operators is $\rho$-regularizing i.e.\ \eqref{2117} is
satisfied when $\sigma$ is large enough relatively to $\rho$.

A homogeneous  symbol $ a $ of $\Gt{m}{p}$ induces 
the non-homogeneous symbol $  a(U, \ldots, U; x, \xi ) $ in $ \Gra{m}{K,0,p} $, and, 
according to \eqref{2112}, we can consider the regularized symbols 
$ a_{\chi_p} (U, \ldots, U; x, \xi) $ and $ a_{\chi} (U, \ldots, U; x, \xi) $. 
It follows that  
$$
\opw{(a_{\chi_p} (U, \ldots, U; x, \xi))} - \opw{(a_{\chi} (U, \ldots, U; x, \xi))}
$$
is the restriction at $U_1=\dots =U_p = U$ of a smoothing operator in $\Rt{-\rho}{p} $. 

 $\bullet$  
 Proposition \ref{215} implies that,  if  $a(U; t, \cdot)$ is  in $\sGM{-\rho}{K,K',p}{N}$, for some 
 $\rho \geq 0 $,  then $\opbw(a(U; t, \cdot))$ defines a smoothing operator in
the space 
$\sRM{-\rho}{K,K',p}{N}$.  Let us prove for example that, if $ a ({\cal U}; \cdot ) $ 
is in $ {\tilde \Gamma}_p^{-\rho} $, then 
$ \opbw(a( {\cal U}; \cdot)) $ is in $ \Rt{-\rho}{p} $. 
By \eqref{2121} with $ s = - \rho $ we get
$$
\| \Pin{0}\opbw(a(\Pin{}\Ucal;\cdot))\Pin{p+1}U_{p+1} \|_{L^2} \leq C 
\frac{n_1^\sigma \ldots n_p^\sigma}{n_{p+1}^\rho} \Gcalsm{0}{0,p+1}{\Pin{}\Ucal,\Pin{p+1}U_{p+1}}
$$
and therefore, in order to prove \eqref{2115}, it is sufficient to show that
\be\label{eq:upb}
\frac{n_1^\sigma \ldots n_p^\sigma}{n_{p+1}^\rho} \leq C 
\frac{{ \max_2 (n_1, \ldots, n_{p+1}) }^{\rho+\mu}}{ \max (n_1, \ldots, n_{p+1} )^\rho }  
\ee
for some $ \mu \geq   0 $.  
If $  n_{p+1} = \max (n_1, \ldots, n_{p+1})  $ then \eqref{upb} is immediate. 
Otherwise $ n_{p+1} \leq \max_2 ( n_1, \ldots, n_{p+1})  $ and
the second inequality in \eqref{2122} implies that for any $ j = 1, \ldots, p $,  $  n_j  \leq C \delta n_{p+1} 
\leq C \delta \max_2 ( n_1, \ldots, n_{p+1}) $ and 
$ \max(n_1, \ldots, n_{p+1}) \sim n_{p+1} $. Also in this case \eqref{upb} follows. 
\medskip

In some instances, we shall use that operators of the form $\opbw(a(V,\cdot))$ or $R(V)$ satisfy some boundedness properties,
without having to keep track of the number of lost derivatives in a very precise way. We introduce a notation for such
classes.
\begin{definition}
  \label{216}
Let $p$ be in $\N$, $m$ in $\R_+$, $K'\leq K$  in $\N$,  $N\in \N^*$,  $r>0$.

(i) We denote by \index{Ma@$\Mt{m}{p}$ (Space of homogeneous maps)} $\Mt{m}{p}$ the space of $(p+1)$-linear maps 
from the space  $\Hds{\infty}(\Tu,\C^2)^p\times \Hds{\infty}(\Tu,\C)$ to the 
space  $\Hds{\infty}(\Tu,\C)$, symmetric in $(U_1,\dots,U_p)$, of the form 
$$
(U_1,\dots,U_{p+1})\to M(U_1,\dots,U_p)U_{p+1} 
$$ 
such that there is $C>0$ and, for any $\Ucal
= (U_1,\dots,U_p)$ in $\Hds{\infty}(\Tu,\C^2)^p$, any $U_{p+1}$ in $\Hds{\infty}(\Tu,\C)$, any $n_0, n_{p+1}$ in $\N^*$,
$n=(n_1,\dots,n_{p})$ in $(\N^*)^{p}$,
\begin{multline}
  \label{eq:2125}
  \norm{\Pin{0}M(\Pin{}\Ucal)\Pin{p+1}U_{p+1}}_{L^2} \\\leq C(n_0+\cdots+n_{p+1})^m\Gcalsm{0}{0,p+1}{\Pin{}\Ucal,\Pin{p+1}U_{p+1}}.
\end{multline}
Moreover, we assume that if 
\begin{equation}
  \label{eq:2126}
  \Pin{0}M(\Pin{1}U_1,\dots,\Pin{p}U_p) \Pin{p+1}U_{p+1}\not\equiv 0
\end{equation}
then there is a choice of signs $\epsilon_0,\dots,\epsilon_{p+1} 
\in \{-1,1\}$ such that $\sum_{0}^{p+1} \epsilon_jn_j = 0$. (When $p=0$, the
above conditions just mean that $M$ is a linear map from $\Hds{\infty}(\Tu,\C)$ to itself, satisfying 
estimate \eqref{2125} and \eqref{2126}).

(ii) We denote by \index{Mb@$\Mr{m}{K,K',N}$ (Space of non-homogeneous maps)} $\Mr{m}{K,K',N}$ the space of maps 
$(V,t,U)\to M(V;t)U$, defined on
$\Br{K}{}\times I\times \CKH{\sigma}{\C}$ for some $\sigma>0$, that are linear in $U$, and such that for any $s$ with
$s>\sigma$, there are $C>0$, $r(s)\in ]0,r[$ and for any $V$ in
$B^K_{\sigma}(I,r(s))\cap\CKH{s}{\C^2}$, any $U$ in $\CKH{s}{\C}$,  any
$ 0 \leq k\leq K-K'$, any $t \in I$, one has the estimate  
\begin{multline}
  \label{eq:2127}
  \norm{\partial_t^k(M(V;t)U)(\cdot)}_{\Hds{s-\frac{3}{2}k-m}}\leq  C\sum_{k'+k''=k}\Bigl(\Gcals{\sigma}{k'+K',N}{V}
    \Gcals{s}{k'',1}{U} \\+ \Gcals{\sigma}{k'+K',N-1}{V} \Gcals{\sigma}{k'',1}{U} \Gcals{s}{k'+K',1}{V}\Bigr).
\end{multline}
We denote by  \index{Mba@$\Mr{m}{K,0,N}$ (Space of autonomous maps)}
$\Mra{m}{K,0,N}$ the subspace of $\Mr{m}{K,0,N}$ made of maps $(V,U)\to M(V)U$ that satisfy estimates \eqref{2127} with
$K'=0$, the time dependence being only through $V = V(t)$.

(iii) Assume $p\leq N $. We denote by \index{Mc@$\sM{m}{K,K',p}{N}$ (Space of maps)} $\sM{m}{K,K',p}{N}$ the space of maps $(V,t,U)\to
M(V;t)U$ that may be written as
\begin{equation}
  \label{eq:2128}
  M(V;t)U = \sum_{q=p}^{N-1} M_q(V,\dots,V)U + M_N(V;t)U
\end{equation}
for some $M_q$ in $\Mt{m}{q}$, $q = p,\dots,N-1$ and $M_N$ in $\Mr{m}{K,K',N}$
(for $ p = N $ we mean that $ M = M_N $ is purely non-homogeneous). 

We set \index{Mca@$\sMa{m}{K,0,p}{N}$ (Space of autonomous maps)} $\sMa{m}{K,0,p}{N}$ for the subspace of $\sM{m}{K,0,p}{N}$
made of sums of the form \eqref{2128} with $K'=0$ and $M_N$ taken in $\Mra{m}{K,0,N}$.
\smallskip

Finally we denote  

$ \bullet  $ \index{Md@$\Mt{}{p}$ (Space of homogeneous maps)}  $\Mt{}{p} = \bigcup_m \Mt{m}{p}$ 

$ \bullet  $ \index{Me@$\Mr{}{K,K',p}$ (Space of non-homogeneous maps)}~$\Mr{}{K,K',p} = \bigcup_m  \Mr{m}{K,K',p}$

$ \bullet  $
 \index{Mf@$\sM{}{K,K',p}{N}$
  (Space of maps)}~$\sM{}{K,K',p}{N} = \bigcup_m \sM{m}{K,K',p}{N}$. 
  
We shall use similar
notations for autonomous maps.
\end{definition}
\textbf{Remarks}: $\bullet$ If $M$ is in $\Mt{m}{N}$ then the 
map $(V,U)\to M(V,\dots,V)U$ is in $\Mra{m}{K,0,N}$.
Indeed conditions \eqref{2125}, \eqref{2126} imply 
that \eqref{2127}  is satisfied with $K'=0 $ 
for $s>\sigma\gg 1$. In particular, an element of $\sM{m}{K,K',p}{N}$ sends a
couple 
$ (V,U) \in \Brs{K}{}{s}\cap C_*^{K-K'}(I,\Hds{s}(\Tu,\C)) $ 
to an element $ M(V;t)U  $ in $  C_*^{K-K'}(I,\Hds{s-m}(\Tu,\C))$ for
$s>\sigma\gg1$.

$\bullet$ By Definition~\ref{214}, any smoothing operator $R$  in $\sR{-\rho}{K,K',p}{N}$ defines an element of
$\sM{m}{K,K',p}{N}$ for some $m\geq 0$. 

$ \bullet $
If $a$ is in $\sG{m}{K,K',p}{N}$, then the map $(V,U)\to \opbw(a(V; t, \cdot))U$ defined by
\eqref{2114} is in $\sM{m}{K,K',p}{N}$ according to Proposition~\ref{215}.

$\bullet$ Let $a$ be a homogeneous symbol in $\Gt{m}{p}$ and let $M$ be an operator 
in $\sM{}{K,K',q}{N-p}$ with $ p + q \leq N $. Consider the function 
$ U\to a(U,\dots,U,M(U;t)U;t,x,\xi) $. 
Decomposing $M$ as in \eqref{2128}, we deduce by \eqref{214}, \eqref{215}, \eqref{2125},
\eqref{2126} that
\[(U_1,\dots,U_{p+q})\to a(U_1,\dots,U_{p-1},M_q(U_p,\dots,U_{p+q-1})U_{p+q};x,\xi)\]
defines a homogeneous symbol in $\Gt{m}{p+q}$, and by \eqref{214}, \eqref{215}, \eqref{2127} that 
$U\to a(U,\dots,U,M_{N-p}(t;U)U;\cdot)$ is in $\Gr{m}{K,K',N}$.

$\bullet$ If $M$ is in  $\sM{}{K,K'_1,p}{N}$ and $\tilde{M}$ in 
$\sM{}{K,K'_2,1}{N-p}$, then 
the map $ (V,t,U) \to M(V+\tilde{M}(V;t)V;t)[U] $ is in $\sM{}{K,K',p}{N}$ with $K'=K'_1+K'_2$. 

$ \bullet $ If $M$ is in $\sM{}{K,K',p}{N}$ and $\tilde{M}$
in $\sM{}{K,K',q}{N}$  then 
the map $ (V,t,U) \to \pi(M(V;t)V)\pi(\tilde{M}(V;t)U)$ is in $\sM{}{K,K',p+q+1}{N}$ 
where $\pi$ denotes the canonical identification of $\Hds{\sigma}$ and $\Hsz{\sigma}$, so that the
product makes sense.

$\bullet$ Finally, if $M$ is in $\sM{m}{K,K',p}{N}$ and $M'$ in $\sM{m'}{K,K',q}{N}$, 
then 
$ M(V;t) \circ M'(V;t) $, namely the map 
$$ 
(V,t,U) \to  M(V;t) [M'(V;t) U] \, , 
$$ 
is in $\sM{m+m'}{K,K',p+q}{N}$.
\smallskip

We end this section with a lemma that will be useful below, and that follows from the above
remarks.
\begin{lemma}
  \label{217}
Let $C(U;t, \cdot)$ be a matrix of symbols in $\sGM{m}{K,K',p}{N}$ for some 
$m $ in $ \R $, $ K' \leq  K-1 $, $ 0 \leq p \leq  N $, and assume that $U$ is a solution of an equation
\begin{equation}
  \label{eq:2129}
  D_tU = \tilde{M}(U;t)U 
\end{equation}
for some $\tilde{M}$ in the class $\sMM{}{K,1,0}{N}$. Then $D_tC(U;t,\cdot)$ belongs to $\sGM{m}{K,K'+1,p}{N}$.
\end{lemma}
\begin{proof}
  Decompose according to \eqref{219}
\[
C(U;t,x,\xi) = \sum_{q=p}^{N-1}C_q(U,\dots,U; x,\xi) + C_N(U;t,x,\xi) \, .
\]
Consider first the non-homogeneous symbol $ C_N $. 
By the second remark after Definition \ref{212}
the symbol $D_tC_N$ belongs to $\GrM{m}{K,K'+1,N}$. Consider next the homogeneous contributions
\[
D_tC_q(U,\dots,U; x,\xi) = \sum_{\ell=1}^q C_q(\underbrace{U,\dots,D_tU}_{\ell},\dots,U; x,\xi) \, .
\]
Replacing $D_tU$ by $\tilde{M}(U;t) U$ given by \eqref{2129} and using the fourth remark after the statement of Definition~\ref{216}, we obtain the
wanted conclusion.
\end{proof}

\section{Symbolic calculus}\label{sec:22}

We study now the composition properties of the preceding operators. Define the differential operator 
\begin{equation}
  \label{eq:221}
  \index{s@$\sigma(D_x,D_\xi,D_y,D_\eta)$} \sigma(D_x,D_\xi,D_y,D_\eta) = D_\xi D_y-D_xD_\eta 
\end{equation}
where 
$  D_x = \frac{1}{i} \partial_x $ and  $ D_\xi $, 
$ D_y  $, $ D_\xi  $ are similarly defined. 
\begin{definition}
  \label{221} {\bf (Asymptotic expansion of composition symbol)}
Let $K'\leq K,\rho, p, q$ be in $\N$, $m, m'$ be in $\R$, $r>0$. 

(i) Let $a$ be an homogeneous symbol  in $\Gt{m}{p}$ and $b$ in $\Gt{m'}{q}$. 
For 
$$
\Ucal' = (U_1,\dots,U_p) \, , \quad 
\Ucal'' = (U_{p+1},\dots,U_{p+q})
$$
with $U_j$ in $\Hds{\infty}(\Tu,\C^2)$, 
$1\leq j\leq p+q$, set $\Ucal = (\Ucal',\Ucal'') $. We define  the formal
series
\begin{multline}
  \label{eq:222}
  \index{a@$(a\#b)_\infty$ (Composition of symbols)} 
  (a\#b)_\infty(\Ucal;x,\xi) \\=
  \sum_{\ell=0}^{+\infty}\frac{1}{\ell!}
  \Bigl(\frac{i}{2}\sigma(D_x,D_\xi,D_y,D_\eta)\Bigr)^\ell\Bigl[a(\Ucal';x,\xi) b(\Ucal'';y,\eta)\Bigr]\vert_{\substack{x=y,\\\xi=\eta}}
\end{multline}
and the sum
$$
(a\#b)_\rho =  \sum_{0 \leq \ell < \rho} \frac{1}{\ell!}
  \Bigl(\frac{i}{2}\sigma(D_x,D_\xi,D_y,D_\eta)\Bigr)^\ell\Bigl[a(\Ucal';x,\xi) b(\Ucal'';y,\eta)\Bigr]\vert_{\substack{x=y,\\\xi=\eta}}
$$ 
modulo symbols in $ \Gt{m+m'- \rho}{p+q}$. 

(ii)  
Let $a$ be a non-homogeneous symbol in $\Gr{m}{K,K',p}$ and $b$ in $\Gr{m'}{K,K',q}$.  For $U$ in $\Br{K}{}$ 
we define when $ \rho < \sigma - \sigma_0 $
\begin{multline}
  \label{eq:223}
   \index{a@$(a\#b)_\infty$ (Composition of symbols)} (a\#b)_\rho(U;t,x,\xi) \\=
  \sum_{0 \leq \ell  < \rho}\frac{1}{\ell!}\Bigl(\frac{i}{2}\sigma(D_x,D_\xi,D_y,D_\eta)\Bigr)^\ell\Bigl[a(U;t,x,\xi) b(U;t,y,\eta)\Bigr]\vert_{\substack{x=y,\\\xi=\eta}}
\end{multline}
modulo symbols in $\Gr{m +m' -\rho}{K,K',p+q} $. 

\end{definition}
\textbf{Remarks}: 
$ \bullet $ We shall use several times below that the first 
terms 
of the asymptotic  expansion \eqref{222} are 
\[a(\Ucal';\cdot) b(\Ucal'';\cdot) + \frac{1}{2i}\absp{a(\Ucal';\cdot),b(\Ucal'';\cdot)} +\cdots\]
with the usual definition of the Poisson bracket $\absp{a,b} = \partial_\xi a\partial_x b - \partial_x a\partial_\xi b$.
\medskip

$ \bullet $
It follows from \eqref{222} and the third remark after Definition~\ref{211} that in (i) of the definition, 
the symbol $(a\#b)_\rho$ is
in $\Gt{m+m'}{p+q}$ modulo $\Gt{m+m'-\rho}{p+q}$ (with an exponent
$\mu$ in \eqref{214} for $(a\#b)_\rho$  large enough in function of $\rho$). This just reflects the fact that the larger
$\rho$ (i.e. the more precise the symbolic calculus), the smoother the functions $U_j$ in the coefficients must be.

In the same way, using the remarks following Definition~\ref{212bis},  if  $a$ is in $\Gr{m}{K,K',p}$ and $b$ is in
$\Gr{m'}{K,K',q}$, then the symbol $(a\#b)_\rho$  belongs to $\Gr{m+m'}{K,K',p+q}$ modulo $\Gr{m+m'-\rho}{K,K',p+q}$
(taking $ \sigma $ in Definition~\ref{212} large enough with respect to  $ \rho $).  

$ \bullet $ 
If $ a $ is in  
$ \sG{m}{K,K',p}{N}$, $ b $ in $ \sG{m'}{K,K',p'}{N}$ and $ c $ in $ \sG{m''}{K,K',p''}{N}$
then 
$$ 
(a \# b \# c)_\rho \in \sG{m+m'+m''}{K,K',p+p'+p''}{N}
$$
modulo a symbol in $ \sG{m+m'+m''- \rho}{K,K',p+p'+p''}{N} $. 

$ \bullet $ In the sequel we shall use the observation that if $ a $ is in  
$ \sG{m}{K,K',p}{N}$, $ b $ is in $ \sG{m'}{K,K',p'}{N}$  and $c$ is in $ \sG{m''}{K,K',p''}{N}$ 
then
\begin{equation}\label{eq:trilinear}
(a \# b \# c)_\rho + (c \# b \# a)_\rho- 2a bc  
\end{equation}
is a symbol in $ \sG{m+m'+m''- 2}{K,K',p+p'+p''}{N} $. 

\medskip

The main goal of this section is to extend to our classes the well known result asserting that $(a\#b)_\rho$ is the symbol of
the composition, up to smoothing operators.
\begin{proposition} {\bf (Composition of Bony-Weyl operators)}
  \label{222}
(i) With the notations of (i) of the preceding Definition \ref{221},
\begin{equation}
  \label{eq:224}
  \opbw(a(\Ucal';\cdot))\circ \opbw(b(\Ucal'';\cdot)) - \opbw((a\#b)_\rho (\Ucal;\cdot))
\end{equation}
belongs to $\Rt{-\rho+m+m'}{p+q}$.

(ii)  With the notations of (ii) of the preceding Definition \ref{221},
\begin{equation}
  \label{eq:225}
  \opbw(a(U;t,\cdot))\circ \opbw(b(U;t,\cdot)) - \opbw((a\#b)_\rho (U;t,\cdot))
\end{equation}
belongs to $ \Rr{-\rho+m+m'}{K,K',p+q} $. 

If $ a $ and $ b $ are symbols  in the autonomous classes of
Definition~\ref{212bis}, then $(a\#b)_\rho$ is in $\Gra{m+m'}{K,0,p+q}$,  modulo a symbol in  
$\Gra{m +m' -\rho}{K,0,p+q} $,  and \eqref{225} is an 
autonomous smoothing operator.
\end{proposition}

\noindent
\textbf{Remark}:  According to Definition \ref{221},  the symbol $ (a\#b)_\rho (\Ucal;\cdot) $ in case (i), 
resp. $ (a\#b)_\rho (U;t,\cdot) $  in case (ii), 
is defined modulo a symbol in $  \Gt{m+m'- \rho}{p+q} $, resp. $\Gr{m +m' -\rho}{K,K',p+q} $. 
By the last remark following Proposition \ref{215}, the 
paradifferential  operator $ \opbw((a\#b)_\rho (\Ucal;\cdot)) $, resp. $ \opbw( (a\#b)_\rho (U;t,\cdot) ) $,
is thus defined modulo a smoothing remainder in $\Rt{-\rho+m+m'}{p+q}$, resp. $ \Rr{-\rho+m+m'}{K,K',p+q} $.

\medskip

To prove Proposition \ref{222}, we need lemmas  \ref{223} and \ref{224}.
Lemma \ref{224} proves that the composition of $ \opbw (a) \circ \opbw(b) $ can be written as
the $ \opw (c) $ of a suitable symbol $ c = a\#b $ and  Lemma \ref{223}  provides its asymptotic expansion.

In order to treat at the same time conditions (i) and (ii), let 
us introduce the following notation. Let $a(x,\xi), b(x,\xi)$ be two tempered distributions in $x$, smoothly depending on $\xi$. 
Assume that their $x$-Fourier transforms $\hat{a}(\eta,\xi)$, $\hat{b}(\eta,\xi)$ are supported for $\abs{\eta}\leq
\delta\absj{\xi}$ for a small enough $\delta>0$. Then the integral
\begin{equation}
  \label{eq:226}
  \index{ab@$a\#b$ (Composition of symbols)} a\#b(x,\xi) 
  \stackrel{\mathrm{def}}{=}  \frac{1}{(2\pi)^2}\int_{\R^2} e^{ix(\xie+\etae)} 
  \hat{a}\bigl(\etae,\xi+\frac{\xie}{2}\bigr)
    \hat{b}\bigl(\xie,\xi-\frac{\etae}{2}\bigr)\,d\xie d\etae
\end{equation}
is well defined as a distribution in $(\xie,\etae)$, compactly supported in 
$ \abs{\xie}+\abs{\etae}\leq\delta'\absj{\xi} $ 
for a small $\delta'>0$, smoothly depending on $\xi$ and 
acting on the smooth function $(\xie,\etae)\to
e^{ix(\xie+\etae)}$. 
Assume also that for some $\rho\in \N$, any $0\leq\alpha\leq\rho$, any $\beta\in \N$, we have for some
constant $M_{\alpha,\beta}(\cdot)$
\begin{equation}
  \label{eq:227}
  \begin{split}
    \abs{\partial_x^\alpha\partial_\xi^\beta a(x,\xi)} &\leq M_{\alpha, \beta}(a)\absj{\xi}^{m-\beta}\\
\abs{\partial_x^\alpha\partial_\xi^\beta b(x,\xi)} &\leq M_{\alpha, \beta}(b)\absj{\xi}^{m'-\beta}.
  \end{split}
\end{equation}
Then $a\#b$ is also given by the oscillatory integral 
\begin{multline}
  \label{eq:228}
  a\#b(x,\xi) =\\ \frac{1}{\pi^2}\int_{\R^4} e^{-2i\sigma(\xe,\xie,\ye,\etae)}a(x+\xe,\xi+\xie)b(x+\ye,\xi+\etae)\,d\xe d\ye d\xie d\etae
\end{multline}
where $ \sigma(\xe,\xie,\ye,\etae) = \xi^* y^*  - x^* \eta^* $
(which has a meaning if we remember that, by the support assumptions on $\hat{a}$, $\hat{b}$, we have that
$\abs{\xie}+\abs{\etae}\ll \abs{\xi}$, so that, to give a meaning to \eqref{228}, it is enough to perform $\partial_\xie$ and
$\partial_\etae$ integrations by parts to gain $\xe$ and $\ye$ decay). 
We also have that 
$$
  a\#b(x,\xi) = e^{ \frac{i}{2} \sigma(D_x,D_\xi,D_y,D_\eta)}  [a(x,\xi) b(y,\eta)]_{\vert x=y, \xi = \eta} \, . 
$$
Let us prove:
\begin{lemma}  \label{223}
Assume that $ a(x,\xi), b(x,\xi) $ satisfy \eqref{227} for some $ \rho $ in $ \N^* $ and their 
$x$-Fourier transforms $\hat{a}(\eta,\xi)$, $\hat{b}(\eta,\xi)$ are supported for $\abs{\eta}\leq
\delta\absj{\xi}$ for a small enough $\delta>0$. 
Then for any $\ell$ in
$\N$, $\ell\leq\rho$,
\begin{equation}
  \label{eq:229}
  \abs{\partial_x^\ell[(a\#b - (a\#b)_{\rho-\ell})](x,\xi)}\leq C_{\rho,\ell}\absj{\xi}^{m+m'-(\rho-\ell)},
\end{equation}
where, for some universal constants $K_{\rho,\ell}$, 
\begin{equation}
  \label{eq:2210}
  C_{\rho,\ell} = K_{\rho,\ell}\sum_{\alpha'+\alpha'' =\rho}\sum_{\beta'+\beta'' = \rho-\ell}M_{\alpha',\beta'}(a) M_{\alpha'',\beta''}(b) \, . 
\end{equation}
\end{lemma}
\begin{proof}
Notice first that $\partial_x^\ell(a\#b)$ (resp.\ $\partial_x^\ell(a\#b)_{\rho-\ell}$) is a linear
 combination of $(\partial_x^{\ell'}a)\#(\partial_x^{\ell''}b)$ (resp.\
 $\bigl((\partial_x^{\ell'}a)\#(\partial_x^{\ell''}b)\bigr)_{\rho-\ell}$ with $\ell'+\ell''=\ell$. Consequently, it suffices to
 prove \eqref{229} when $\ell = 0$. Define for 
  $ \tau $ in $[0,1]$, $\Xe = (\xe,\xie)$, $\Ye = (\ye,\etae)$ in $\R^2$, 
 the function
\begin{multline*}c(\tau,x,\xi) \\= \frac{1}{\pi^2}\int_{\R^4} [a(x',\xi')b(y',\eta')]\vert_{\substack{x' = x + \sqrt{\tau}\xe, y' = x +
    \sqrt{\tau}\ye\\ \xi' = \xi + \sqrt{\tau}\xie, \eta' = \xi + \sqrt{\tau}\etae}} e^{-2i\sigma(\Xe,\Ye)}\,d\Xe d\Ye \, .
    \end{multline*}
Notice that $ c(1, x, \xi ) = a\#b(x,\xi) $  is the function defined in \eqref{228}. 
By integrations by parts, one sees that 
the $k$-th $\tau$-derivative of $ c(\tau,x,\xi) $  is
\begin{multline*}
 \pa_\tau^k c (\tau,x,\xi) =  
  \frac{1}{\pi^2}\int_{\R^4} \Bigl(\frac{i}{2}\sigma(D_{x'},D_{\xi'},D_{y'},D_{\eta'})\bigr)^k  [a(x',\xi')b(y',\eta')]\\((x,\xi) +
  \sqrt{\tau}\Xe, (x,\xi)+\sqrt{\tau}\Ye) e^{-2i\sigma(\Xe,\Ye)}\,d\Xe d\Ye,
  \end{multline*}
so that Taylor formula gives, denoting the derivative $ \pa_\tau^k c = c^{(k)}  $, 
\[
a\#b(x,\xi) - (a\#b)_\rho(x,\xi) = \frac{1}{(\rho-1)!}\int_0^1 c^{(\rho)}(\tau,x,\xi)(1-\tau)^{\rho-1}\,d\tau \, . 
\] 
We may write
\begin{multline}
  \label{eq:2211}
  c^{(\rho)}(\tau,x,\xi) \\=
  \frac{1}{\pi^2}\int_{\R^4} e^{-2i\sigma(\Xe,\Ye)} e((x,\xi)+\sqrt{\tau}\Xe,(x,\xi)+\sqrt{\tau}\Ye)\,d\Xe d\Ye
\end{multline}   
where the symbol 
\begin{equation}\label{eq:2211a}
e (x',\xi',y',\eta') =  \Bigl(\frac{i}{2}\sigma(D_{x'},D_{\xi'},D_{y'},D_{\eta'})\bigr)^\rho  [a(x',\xi')b(y',\eta')] 
\end{equation}
satisfies, recalling \eqref{227},  the estimates  
\begin{multline}\label{eq:symb-e}
\abs{\partial_{\xi'}^{\gamma'}\partial_{\eta'}^{\gamma''}e(x',\xi',y',\eta')} \\\leq
K_{\rho,\gamma}\sum_{\substack{\alpha'+\alpha'' = \rho\\\beta'+\beta''=\rho}}M_{\alpha',\beta'+\gamma'}(a)
M_{\alpha'',\beta''+\gamma''}(b)\absj{\xi'}^{m-\beta'-\gamma'}\absj{\eta'}^{m'-\beta''-\gamma''}.
\end{multline}
The symbol $e$ in \eqref{2211a} may be written as a linear combination of expressions $a_1(x',\xi')b_1(y',\eta')$, where
$a_1, b_1$ may be expressed from derivatives of $a, b$. Then, in the same way as \eqref{228} may 
be written under  the form \eqref{226}, we 
obtain a representation of $c^{(\rho)}$ in
terms of a combination of integrals of the form
\be\label{eq:int-pas}
\int_{\R^2} e^{ix(\xi^*+\eta^*)}\hat{a}_1\bigl(\etae,\xi +
\tau\frac{\xie}{2}\bigr)\hat{b}_1\bigl(\xie,\xi-\tau\frac{\etae}{2}\bigr)\,d\xi^* d\eta^* , 
\ee
 where, because of the support properties of $\hat{a}, \hat{b}$, 
the integrand is supported 
on $\abs{\xi^*} + \abs{\eta^*}\leq C \delta\absj{\xi} $. Therefore we may 
  insert in the integral \eqref{int-pas} a cut-off $ \chi((\xi^*,\eta^*)/\absj{\xi})$, 
  where $ \chi $ is a compactly supported $ C^\infty (\R) $ function, equal to one close to zero. 
Now, expressing the Fourier transforms $\hat{a}_1, \hat{b}_1$ from $a_1, b_1$, we 
  get an integral of the form \eqref{2211}, with moreover the  cut-off $ \chi((\xi^*,\eta^*)/\absj{\xi})$.
 This shows that in \eqref{2211} we may insert  $\chi ((\xi^*,\eta^*)/\absj{\xi} )$ under the integral.
Performing in \eqref{2211} two integrations by parts in
$\partial_{\xie}$ and in $\partial_\etae$, we get an integrating factor
$(1+\absj{\xi}\abs{\ye})^{-2}(1+\absj{\xi}\abs{\xe})^{-2}$. Finally, using \eqref{symb-e}, we bound \eqref{2211} by
\[C_{\rho}\absj{\xi}^{m+m'-\rho}\int_{\substack{\abs{\xie}\ll\absj{\xi}\\\abs{\etae}\ll
    \absj{\xi}}}(1+\absj{\xi}\abs{\ye})^{-2}(1+\absj{\xi}\abs{\xe})^{-2}\,d\Xe d\Ye,\]
which gives \eqref{229} with $ C_{\rho,0} $ as in \eqref{2210} 
(with other constants $ M_{\alpha, \beta}$). 
\end{proof}
Before proceeding, we make the following remarks:

Let $a(x, \xi) $ satisfy \eqref{227} with $\alpha = 0$, $\beta$ in $\N$, with $x$-Fourier transform $\hat{a}(\hat{x},\xi)$
supported for $\abs{\hat{x}}\leq\delta\absj{\xi}$ for some small $\delta>0$. Then  $\hat{a}(\hat{x},\xi)$ may be
considered as an $\hat{x}$-compactly supported distribution, for each fixed $\xi$, which is smooth in $\xi$, and the
Definition \eqref{2110} of $\opw$ shows that, if $u$ is in  $\Scal(\R)$
\be\label{eq:Opwa}
\widehat{\opw(a)u}(\xi) = \frac{1}{2\pi}\int_{\R}\hat{a}\bigl(\eta,\xi-\frac{\eta}{2}\bigr)\hat{u}(\xi-\eta)\,d\eta
\ee
where the integral has to be interpreted as an $\eta$-compactly supported distribution depending smoothly on $\xi$ acting on
$\eta\to \hat{u}(\xi-\eta)$. Since $\abs{\eta}\ll\absj{\xi}$ on the support of \eqref{Opwa},
it follows that $\widehat{\opw(a)u}$ is in
$\Scal(\R)$, and setting 
\be\label{eq:aR}
a_R(x,\xi) = a(x,\xi)\theta(\xi/R)
\ee 
for some cut-off function $\theta\in C^\infty_0(\R)$ equal to one close to zero,  
$$
\opw(a_R)u \to \opw(a)u \quad {\rm in} \ \  \Scal(\R) \, , 
$$
when $R$ goes to $+\infty $.
Moreover, formula \eqref{226} shows that the $x$-Fourier transform of $a\#b$  is
\be\label{eq:intsh}
\widehat{a\#b}(\eta,\xi) = \frac{1}{2\pi}\int_{\R} \hat{a}\bigl(\xie,\xi + \frac{\eta}{2} - \frac{\xie}{2}\bigr)
\hat{b}\bigl(\eta-\xie,\xi-\frac{\xie}{2}\bigr)\,d\xie
\ee
and, by  the support conditions
$ |\xie | \leq \delta \langle \xi + \frac{\eta}{2} - \frac{\xie}{2} \rangle  $, 
$ |\eta-\xie | \leq \delta  \langle  \xi-\frac{\xie}{2} \rangle $, 
the  integrand in \eqref{intsh} is supported, for $ \delta $ small enough, on 
\be\label{eq:support1}
\abs{\xie}+\abs{\eta}\ll\abs{\xi} \, . 
\ee 
In addition, 
recalling \eqref{Opwa}, \eqref{aR}, \eqref{intsh}, we have that,
 for any $u$   in  $\Scal(\R)$, 
\[
\begin{split}
  \widehat{\opw(a_R\#b_R)u}(\xi) = \frac{1}{(2\pi)^2}\int_{\R^2} \widehat{a_R}\bigl(\xie,\xi - \frac{\xie}{2}\bigr)
\widehat{b_R}\bigl(\eta-\xie,\xi - \frac{\eta}{2}-\frac{\xie}{2}\bigr)\\\times\hat{u}(\xi-\eta)\,d\xie d\eta
\end{split}\]
where, by the conditions 
$ |\xie | \leq \delta \langle \xi - \frac{\xie}{2} \rangle  $, $ |\eta-\xie | \leq \delta \langle \xi - \frac{\eta}{2}-\frac{\xie}{2} \rangle  $,
 we have, on the support of the integrand,  
$$
\aabs{\xi-\frac{\xi^*}{2}}\sim \abs{\xi-\eta}\sim \aabs{\xi-\frac{\eta}{2}-\frac{\xie}{2}} \, .
$$ 
As a consequence, as $R$ goes to infinity, 
$$ 
 \widehat{\opw(a_R\#b_R)u}(\xi) \ \to  \ \widehat{\opw(a\#b)u}(\xi) \quad  {\rm in} \ \  \Scal'(\R) \, . 
$$
We shall exploit these remarks in the proof of the following lemma.

\begin{lemma}  \label{224}
Let $a$, $b$ be as in (i) (resp. (ii)) of Definition~\ref{221}. Then
\begin{equation}
  \label{eq:2212}
  \opbw(a)\circ\opbw(b) = \opw(c)
\end{equation}
where
\begin{equation}
  \label{eq:2213}
  c(\Ucal;x,\xi) = a_{\chi_p}(\Ucal';x,\xi)\# b_{\chi_q}(\Ucal'';x,\xi),
\end{equation}
with $a_{\chi_p}, b_{\chi_q}$ defined by the first formula \eqref{2112},

(resp. where
\begin{equation}
  \label{eq:2214}
  c(U;t,x,\xi) = a_{\chi}(U;t,x,\xi)\# b_{\chi}(U;t,x,\xi),
\end{equation}
with $a_{\chi}, b_{\chi}$ defined by the second formula~\eqref{2112}).
\end{lemma}

\begin{proof}
 Let us treat the case (i). By \eqref{2113}, we have
\[
\opbw(a)\circ\opbw(b) =\opw(a_{\chi_p})\circ \opw(b_{\chi_q}) \, .
\]
Since $ a $ is an homogeneous symbol  in $ \Gt{m}{p} $ and $ b $ in $ \Gt{m'}{q} $
the symbols $ a_{\chi_p}, b_{\chi_q} $   satisfy, by \eqref{214} and \eqref{2111}, \eqref{2111bis},  
the following estimates: 
there is $ \sigma $ such that for any $  \alpha,  \beta \in \N $, we have 
\begin{equation*}
\begin{split}
\abs{\partial^\alpha_x\partial^\beta_\xi a_{\chi_p} (\Ucal';x,\xi)} \leq
  C_{\alpha, \beta}  \absj{\xi}^{m-\beta + \alpha}\Gcalsm{\sigma}{0,p}{\Ucal'} \, , \\
  \abs{\partial^\alpha_x\partial^\beta_\xi b_{\chi_q} (\Ucal'';x,\xi)} \leq
 C_{\alpha, \beta} \absj{\xi}^{m'-\beta+ \alpha}\Gcalsm{\sigma}{0,q}{\Ucal''}  \, . 
 \end{split}
\end{equation*}
Moreover, by \eqref{215}, the support of their space Fourier transforms 
\be\label{eq:suppF}
{\rm supp} \big( \widehat{a_{\chi_p} ({\cal U}'; \cdot )}(\eta,\xi) \big) \, , 
\ 
{\rm supp} \big(\widehat{b_{\chi_q} ({\cal U}''; \cdot )}(\eta,\xi) \big) \, 
\subseteq \, \{  | \eta | \leq C \delta \langle \xi \rangle  \}
\ee
where $\delta$ is the small constant in \eqref{2111}. 
As a consequence we have that $  a_{\chi_p}(\Ucal';x,\xi)\# b_{\chi_q}(\Ucal'';x,\xi) $ is 
well defined as the oscillatory integral \eqref{228}. 

We are thus left with showing that
\begin{equation} \label{eq:2215}
 \opw(a_{\chi_p} ( {\cal U}'; \cdot ) )\circ\opw(b_{\chi_q } ({\cal U}'' ; \cdot )) = \opw((a_{\chi_p} \# b_{\chi_q})({\cal U}; \cdot)).
\end{equation}
If we replace $ a_{\chi_p}, b_{\chi_q} $ by $a_{\chi_p,R}, b_{\chi_q,R}$ 
defined as in \eqref{aR}, we obtain symbols that satisfy 
the assumptions of Theorem~7.9 of \cite{DSj}. Indeed, by \eqref{214}, for any $ \alpha, \beta $ in $ \N $, we have the estimates
$$
|\partial^\alpha_x \partial^\beta_\xi a_{\chi_p, R} (\Ucal';x,\xi)| \leq C_{\alpha, \beta,R} \langle \xi \rangle^m \, , 
\quad
|\partial^\alpha_x \partial^\beta_\xi a_{\chi_q, R} (\Ucal'';x,\xi)| \leq C_{\alpha, \beta,R} \langle \xi \rangle^{m'} \, ,
$$
and the symbol $ a_{\chi_p, R} $, resp.  $ b_{\chi_q, R} $, 
can be regarded  as a function in $ S_\delta (m_1) $, resp. $ S_\delta (m_2) $, for any $ 0 \leq \delta \leq \frac12 $, 
with order functions $ m_1 (x, \xi) = \langle \xi \rangle^m $, resp. $ m_2 (x, \xi) = \langle \xi \rangle^{m'} $.
Theorem~7.9 of \cite{DSj} with $ h = 1 $ implies that 
formula \eqref{2215} holds for $a_{\chi_p,R}, b_{\chi_q,R} $. 
By the remarks preceding the
statement of the lemma, we may pass to the limit when $R$ goes to infinity 
proving \eqref{2215} and therefore \eqref{2212}.
The proof of case (ii) is similar.
\end{proof}

\begin{proof1}{Proof of Proposition~\ref{222}}
We prove statement (i). By lemma~\ref{224}, we may write \eqref{224} as
\[
\opw\bigl(a_{\chi_p}(\Ucal';\cdot)\# b_{\chi_q}(\Ucal'';\cdot)\bigr) - \opbw((a\# b)_\rho(\Ucal;\cdot)) =
\opw(r(\Ucal;\cdot))
\]
where $ r = r_1 + r_2 $ with
\begin{equation}  \label{eq:2216}
  \begin{split}
    r_1(\Ucal;\cdot) &= a_{\chi_p}(\Ucal';\cdot)\# b_{\chi_q}(\Ucal'';\cdot) - (a_{\chi_p}\# b_{\chi_q})_\rho(\Ucal;\cdot)\\
 r_2(\Ucal;\cdot) &= (a_{\chi_p}\# b_{\chi_q})_\rho(\Ucal;\cdot) - (a\#b)_{\rho, \chi_{p+q}}(\Ucal;\cdot).
  \end{split}
\end{equation}
By \eqref{229} with $ \ell = 0 $, \eqref{2210} and \eqref{214} (recall \eqref{227}), we have 
\be\label{eq:bound-r1}
\abs{r_1(\Pin{}\Ucal;x,\xi)} \leq C \abs{n}^{\mu+\rho}\absj{\xi}^{m+m'-\rho}\Gcalsm{0}{0,p+q}{\Pin{}\Ucal} \, .
\ee
 Notice that the symbol $ r_1 $ satisfies, as a consequence of its definition \eqref{2216}, 
 \eqref{suppF}, \eqref{intsh}, \eqref{support1}, and the Definition \ref{221} of
 $ (a_{\chi_p}\# b_{\chi_q})_\rho $,  spectral localization properties that 
 ensure that $ \opw (r_1) $ coincides with $ \opbw ( r_1) $, 
 if $\chi_p, \chi_q$ in the definition of $r_1$ have been taken with small enough support. 
 More
 precisely, the space Fourier transform of $ r_1 $ is supported for $ | \eta | \ll \langle \xi \rangle $,
 if $ \delta $ in \eqref{2111} is small enough.
 Moreover by the first remark following the proof of Proposition~\ref{215} the estimate \eqref{bound-r1} is enough to 
 prove that $ \opw(r_1(\Pin{}\Ucal;\cdot) $ satisfies 
 \eqref{2121} (with $ m $ replaced by $ m + m' - \rho $ and  $ \sigma \sim \rho  $). Therefore 
applying \eqref{2121}  with $ s = m +m' - \rho $, 
we have, up to changing the definition of $\mu$,   for any $n=(n_1,\dots,n_{p+q})$ 
in $(\N^*)^{p+q}$, any $n_{p+q+1} $, the bound
\begin{multline}
  \label{eq:2217}
    \norm{\opw(r_1(\Pin{}\Ucal;\cdot))\Pin{p+q+1}U_{p+q+1}}_{L^2}\\
\leq C\abs{n}^{\mu+\rho} n_{p+q+1}^{m+m'-\rho}\Gcalsm{0}{0,p+q+1}{\Pin{}\Ucal,\Pin{p+q+1}U_{p+q+1}} \, .
\end{multline} 
Moreover, since
$ \opw(a_{\chi_p}(\Pi_{n'}\Ucal';\cdot)) $ and $\opw(b_{\chi_q}(\Pi_{n''}\Ucal'';\cdot)) $ satisfy
 \eqref{pro-x-in} (since $ a_{\chi_p}, b_{\chi_q}$ satisfy \eqref{215}), 
 the composed operator
$$ 
\opw(a_{\chi_p}(\Pi_{n'}\Ucal';\cdot)) \circ 
\opw(b_{\chi_q}(\Pi_{n''}\Ucal'';\cdot)) 
$$
 satisfies  the same property as well, and, by \eqref{2122} in Proposition \ref{215}, if  
\[ 
\opw(a_{\chi_p}(\Pi_{n'}\Ucal';\cdot))\circ \opw(b_{\chi_q}(\Pi_{n''}\Ucal'';\cdot)\Pin{p+q+1}U_{p+q+1} \neq 0,
\]
then $ n = (n',n'')$ satisfies $\abs{n}\ll n_{p+q+1}$ 
(if $\delta$ in \eqref{2111} is small enough). A similar statement holds
for
\[
\opw\bigl((a_{\chi_p}\# b_{\chi_q})_\rho(\Pin{}\Ucal;\cdot)\bigr)\Pin{p+q+1}U
\]
so that when \eqref{2217} does not vanish identically, we may bound its right hand side as
\be\label{eq:r1es}
C\frac{\maxdn{1}{p+q+1}^{\mu+\rho}}{\maxn{1}{p+q+1}^{\rho-m-m'}}\Gcalsm{0}{0,p+q+1}{\Pin{}\Ucal,\Pin{p+q+1}U_{p+q+1}} \, .
\ee
Indeed, if $ \maxn{1}{p+q} < n_{p+q+1} $ then \eqref{r1es} directly follows by \eqref{2217}. On the other hand, if 
$ \maxn{1}{p+q} \geq n_{p+q+1} $, then, by  $ \abs{n}\ll n_{p+q+1} $, we have that 
$ \max_{2}( n_1, \ldots, n_{p+q+1}) \sim  \max ( n_1, \ldots, n_{p+q+1}) $
and \eqref{r1es} follows as well. 
We have proved a bound of the form \eqref{2115}, showing that 
$\opw(r_1)$ is a smoothing operator of $\Rt{-\rho+m+m'}{p+q}$. 

We have
to prove a similar result for $ \opw(r_2(\Ucal;\cdot)) $ where $ r_2(\Ucal;\cdot) $ is defined 
in \eqref{2216}. Recalling Definition \ref{221}, we notice that 
$ r_2(\Pin{}\Ucal;\cdot)$  is a combination of symbols
of the form, for $ 0 \leq \ell < \rho $
\[\begin{split}
  \Bigl(\frac{i}{2}\sigma(D_x,D_\xi,D_y,D_\eta)\Bigr)^\ell
  \bigl[a(\Pi_{n'}\Ucal';x,\xi)b(\Pi_{n''}\Ucal'';&y,\eta)\\&\times\chi_p(n',\xi)\chi_q(n'',\eta)\bigr]\vert_{\substack{x=y\\\xi=\eta}}\\
-\chi_{p+q}(n',n'',\xi) \Bigl(\frac{i}{2}\sigma(D_x,D_\xi,D_y,D_\eta)\Bigr)^\ell &\\
\bigl[a(\Pi_{n'}\Ucal';x,\xi)&b(\Pi_{n''}\Ucal'';y,\eta) \bigr]\vert_{\substack{x=y\\\xi=\eta}}.
\end{split}
\]
Because of the properties \eqref{2111}  of 
$ \chi_p, \chi_q, \chi_{p+q} $, the above symbol is supported in the region where  
$$
\delta_1  \langle \xi \rangle \leq |(n', n'') | \leq \delta_2  \langle \xi \rangle,  
\quad {\rm thus} \quad \maxn{1}{p+q}\sim \langle \xi \rangle\, , 
$$
and, since $a$ is  in $\Gt{m}{p}$ and $b$ in $\Gt{m'}{q}$,  it is bounded in modulus by 
$$
C | n |^\sigma  \absj{\xi}^{m+m' - \rho} \Gcalsm{0}{0,p+q}{\Pin{}\Ucal}  
$$
for some $ \sigma \sim \rho $.  Arguing as for $ r_1 $ we deduce that $\opw(r_2( \Pi_n \Ucal;\cdot)) $
satisfies bounds of the form \eqref{2115} (with $ \rho $ replaced by $ \rho - m - m' $ and $ p + 1 $ replaced 
by $ p + q + 1$)  and therefore also 
$\opw(r_2)$ is a smoothing operator  of $\Rt{-\rho+m+m'}{p+q}$. 
This concludes the proof of (i) of the proposition.

We now prove statement (ii). By \eqref{2214} and  the second line of \eqref{2113} we may write \eqref{225} as
\[
\opw(r_1(U; t, \cdot) + r_2(U; t, \cdot) )
\]
where 
\begin{equation}  \label{eq:2216bis}
\begin{split}
 r_1(U;t, \cdot) &= a_{\chi}(U; t, \cdot)\# b_{\chi}(U;t,\cdot) - (a_{\chi}\# b_{\chi})_\rho(U; t, \cdot)\\
 r_2(U; t, \cdot) &= (a_{\chi}\# b_{\chi})_\rho(U; t, \cdot) - (a\#b)_{\rho, \chi}(U; t, \cdot).
  \end{split}
\end{equation}
The symbol $ a $ is in $ \Gr{m}{K,K',p} $ and $b$ in $ \Gr{m'}{K,K',q} $.
Taking $ \sigma $ large enough, depending on $ \rho $, we have that  \eqref{218}  holds up to the index
$ 0 \leq \alpha \leq \rho $ (and similarly for $ b $).
We notice that \eqref{229}, \eqref{2210} and \eqref{218} imply that the symbol $r_1$ 
satisfies bounds
\[
\abs{\partial_t^k \partial_x^\ell r_1(U; t,  x,\xi)} \leq C\absj{\xi}^{m+m'-(\rho-\ell)}\Gcals{\sigma}{k+K',p+q}{U}
\]
for $0\leq\ell\leq 2$.  By the first remark following the proof of Proposition~\ref{215} the estimate \eqref{2123} holds, i.e. 
$$
\norm{\opbw(\partial_t^k r_1(U;t,\cdot))}_{\Lcal(\Hds{s},\Hds{ s-m-m'+\rho - 2})}\leq C\Gcals{\sigma_0}{k+K',p+q}{U} \, , 
$$ 
and, for any $ 0 \leq k \leq K - K' $, 
\begin{multline*}
 \norm{\partial^k_t (\opw(r_1(U; t, \cdot))V)}_{\Hds{s-m-m'+\rho-2-\frac{3}{2}k}}\\
\leq C\sum_{k'+k''=k}\Gcals{\sigma}{k'+K',p+q}{U}\Gcals{s}{k'',1}{V}
\end{multline*}
which implies the estimate \eqref{2117} of a smoothing operator 
of $ \Rr{-\rho +m + m' }{K,K',p+q} $,  by renaming $ \rho $ as $ \rho + 2 $.
(We used again that Weyl and Bony-Weyl quantizations coincide for the symbol at hand, since
the space Fourier transform of $ r_1 (U; t,  x,\xi) $ is supported for $ | \eta | \ll | \xi |$). The study of $\opbw(r_2)$ is similar, as $r_2$ satisfies the
same bounds as $r_1$ above.

The statement concerning autonomous classes follows directly from the proof.
\end{proof1}
\textbf{Remark}:  The proof shows that the remainders in \eqref{224}, \eqref{225} have actually better estimates than the
general bounds  \eqref{2115}, \eqref{2117} satisfied by smoothing operators. In \eqref{2216} the symbols $r_1, r_2$ are
supported so that in \eqref{2217}, $\abs{n}\ll n_{p+q+1}$. Consequently, for any $\rho$, there is $\sigma>0$, and for any $\Ucal = (U,\dots,U)$ with $\norm{U}_{\Hds{\sigma}}$
smaller than some $r$, the operator \eqref{224} is bounded from $\Hds{s}$ to $\Hds{s+\rho-(m+m')}$ for \emph{any} $s$. The same
assertion holds for \eqref{225}.

\section{Composition theorems}\label{sec:23}

We  state in this section the composition results that will be used systematically in the rest of the monograph.
Let us consider symbols 
$$ 
a = \sum_{q=p}^N a_q \in \sG{m}{K,K',p}{N} \, , \quad
b = \sum_{q'=p'}^N b_{q'} \in \sG{m'}{K,K',p'}{N}
$$
decomposed as in  \eqref{219}. 
We define the symbol of $\sG{m+m'}{K,K',p+p'}{N} $,  
\begin{equation}
  \label{eq:231}
  \index{ac@$(a\#b)_{\rho,N}$ (Composition of symbols)} 
  (a\#b)_{\rho,N} = \sum_{q'' = p+p'}^{N-1} c_{q''}(U,\dots,U;x,\xi) + c_N(U;t,x,\xi) \, , 
\end{equation}
where, for $q''\leq N-1$, one sets
\[
c_{q''}(\Ucal;x,\xi) = \sum_{q+q'=q''}(a_q\# b_{q'})_\rho(\Ucal;x,\xi)
\]
and
\[
c_N(U;t,x,\xi) = \sum_{q+q'\geq N}(a_q\# b_{q'})_\rho(U;t,x,\xi) \, ,
\]
the factors $a_q , b_{q'}$, $q, q'\leq N-1$, 
in the expression of $c_N$ being considered  as symbols of $\Gr{m}{K,0,q}$,
$\Gr{m'}{K,0,q'}$, according to the first remark following Definition~\ref{212bis}. 

The following proposition is a direct consequence of Proposition~\ref{222}.
\begin{proposition} {\bf (Composition)} 
  \label{231}
  Let   $m, m'$ be in 
  $\R$, $K'\leq K$ in $\N$, $p, p', N$ in $\N$ with $p+p'\leq N$, $\rho$ in $\N$, $r>0$.
 If $a $ is a symbol in $ \sG{m}{K,K',p}{N}$ 
and  $ b $ is  a symbol in $ \sG{m'}{K,K',p'}{N}$,  then  
\begin{equation}
  \label{eq:232}
  \opbw(a(U;t, \cdot))\circ \opbw(b(U; t, \cdot)) = \opbw((a\#b)_{\rho,N}(U;t, \cdot)) + R(U;t)
\end{equation}
where $R$ is a smoothing operator in $\sR{-\rho+m+m'}{K,K',p+p'}{N}$. A similar statement holds for autonomous classes.
\end{proposition}
We may compose as well smoothing operators and paradifferential ones. 
The outcome is another smoothing operator. 
\begin{proposition}
  \label{232}
Let $a$ be a symbol in $\sG{m}{K,K',p}{N}$ with $m\geq 0$ and $R$ a smoothing operator in $\sR{-\rho}{K,K',p'}{N}$. Then the
composition operators 
$$
\opbw(a(U; t, \cdot))\circ R(U;t) \, , \quad R(U;t) \circ \opbw(a(U; t, \cdot)) 
$$ 
are in $\sR{-\rho+m}{K,K',p+p'}{N}$.  A similar statement holds for autonomous classes.
\end{proposition}

\begin{proof}
We prove the proposition for $ \opbw(a(U; t, \cdot))\circ R(U;t) $, the other is similar. 
  Decomposing $a =\sum_{q=p}^N a_q$ as  in  \eqref{219}  and $R = \sum_{q'=p'}^N R_{q'}$ according to \eqref{2118}, we have to show, on   the one hand, that
  \begin{equation}
    \label{eq:233}
    \opbw(a_q(U_1,\dots,U_q;\cdot))\circ R_{q'}(U_{q+1},\dots,U_{q+q'})
  \end{equation}
is an homogeneous 
smoothing operator in $\Rt{-\rho+m}{q''}$ if $q'' = q+q' \leq N-1$ and, on the other hand, that
\begin{equation}
  \label{eq:234}
 \begin{split}
&  \opbw(a_q(U,\dots,U;\cdot))\circ R_{q'}(U,\dots,U),\ q+q'\geq N, q, q'\leq N-1\\
&  \opbw(a_N(U;t, \cdot))\circ R_{q'}(U,\dots,U),\ 0\leq q'\leq N-1\\
&  \opbw(a_q(U,\dots,U;\cdot))\circ R_{N}(U;t),\ 0\leq q\leq N-1\\
&  \opbw(a_N(U; t, \cdot))\circ R_{N}(U;t)
 \end{split}
\end{equation}
are non-homogeneous smoothing operators in $\Rr{-\rho+m}{K,K',N}$. 

Let us first study \eqref{233}. Replacing $U_j$ by $\Pin{j}U_j$ we have to consider expressions of the form 
\begin{multline} \label{eq:235}
    \sum_{n'_0}\Pin{0}\opbw\bigl(a_q(\Pin{1}U_1,\dots,\Pin{q}U_{q};\cdot)\bigr) \Pi_{n'_0} \\\circ \Pi_{n'_0}
    R_{q'}\bigl(\Pin{{q+1}}U_{q+1},\dots,\Pin{q+q'}U_{{q+q'}}\bigr) \Pin{q+q'+1}U_{{q+q'+1}} \, .
\end{multline}
By \eqref{pro-x-in} for $ a_q $ and \eqref{2116} for
$ R_{q'} $, the indices in the sum \eqref{235} satisfy, for some choice of the signs $ \epsilon_j $,  $ j = 0, \ldots, q + q' + 1 $, 
\be\label{eq:Rq'-segni}
\sum_{j=0}^q \epsilon_j n_j = n_0'   \, , \qquad  n_0' = \sum_{j=q+1}^{q+q'+1} \epsilon_j n_j \, . 
\ee
As a consequence, \eqref{233} satisfies the corresponding  condition  \eqref{2116}.  
We are left with checking that \eqref{233} satisfies  estimates of the form \eqref{2115}.
Combining the bound  \eqref{2121} (with $ s = m $) for $a_q$, with 
\eqref{2115} for $R_{q'} $ we deduce, using also the second restriction to the indices 
in \eqref{Rq'-segni}, 
that the $ L^2 $ norm of \eqref{235} is bounded by 
\be\label{eq:primo-bound}
C n_1^\sigma \ldots n_q^\sigma
| n_0' |^m  \frac{\maxdn{q+1}{q+q'+1}^{\rho+\mu}}{\maxn{q+1}{q+q'+1}^{\rho}}
\Gcalsm{0}{0,q+q'+1}{\Pin{}\Ucal} 
\ee
where  $ n = (n_1, \ldots, n_{q+q'+1}) $ and $ {\cal U} = (U_1, \ldots, U_{q+q'+1}   ) $. 
By the first  identity in \eqref{Rq'-segni} we bound 
$ n'_0 \leq n_0 + n_{1}+\cdots +n_{q+1} $. 
By the property \eqref{2122} of  
$\opbw (a_q) $, the indices in the sum \eqref{235} satisfy 
$n_1,\dots,n_q\ll n_0\sim n'_0 $. 
Moreover, the second  identity in \eqref{Rq'-segni}
implies that $n'_0\leq n_{q+1}+\cdots +n_{q+q'+1}$, and one deduces that
\begin{equation}\label{eq:rel-ind-diff}
\begin{split}
 \maxn{q+1}{q+q'+1} \sim \maxn{1}{q+q'+1} \\
 \maxn{1}{q}\leq C\maxdn{1}{q+q'+1} \, .
\end{split}
\end{equation}
Inserting these inequalities in \eqref{primo-bound} we get 
that \eqref{235} satisfies an estimate of the form \eqref{2115}, 
with $\rho$ replaced by $\rho-m$, a new value of $ \mu $, and $ p $ replaced by $ q + q' $. 
This proves that \eqref{235} is a smoothing operator of $\Rt{-\rho+m}{q+q'}$. 

One has next to check that each operator in \eqref{234}
belongs to $\Rr{-\rho+m}{K,K',N}$. This follows from the combination of estimates \eqref{2117} and \eqref{2123}, writing for
instance 
$\partial_t^k[\opbw(a_N(U; t, \cdot))\circ R_N(U;t) V ]$ as a combination of expressions
$\opbw(\partial_t^{k'}a_N(U; t, \cdot)) [ \partial_t^{k''} (R_N(U;t)V) ] $ with $k'+k''=k$.  This concludes the proof.
\end{proof}

Let us study next ``inner compositions'' where we substitute to one of the coefficients of a symbol, or of a smoothing
operator,  one of the maps of Definition~\ref{216}.

\begin{proposition} \label{233}
Let $ K'_1, K'_2 $ be integers with $K'_1+K'_2 = K' \leq K $, and $ N, p, p' \in \N $ with 
$ p + p'  \leq  N $.

(i) Let $a$ be in $\Gt{m}{p}$ and $M$ be in $\sM{}{K,K',p'}{N-p}$. Then
\begin{equation}\label{eq:236}
 U\to a(U,\dots,U,M(U;t)U;t,x,\xi)
\end{equation}
is in $\sG{m}{K,K',p+p'}{N}$.  If the symbol $  a $ is independent of $ \xi $
(i.e. $ a $ is in $ \Ft{p}  $ according to Definition \ref{241}), so is the symbol in \eqref{236}
(thus it is a function in $ \sF{K,K',p}{N} $). Moreover 
if $ a $ is a symbol in $ \Gr{m}{K,K',N} $ then the symbol in  \eqref{236} is in $ \Gr{m}{K,K',N} $. 

(ii) Let $R$ be in $\Rt{-\rho}{p}$ and let $M$ be in $\Mt{m'}{p'}$ for some $ m' \geq 0 $. 
Then the composed operator
\begin{equation} \label{eq:237}
R(U_1,\dots,U_p)M(U_{p+1},\dots,U_{p+p'})
\end{equation}
is in $\Rt{-\rho+m'}{p+p'}$.

(iii) Let $R$ be in $\Rr{-\rho}{K,K',p}$ and let $M$ be in $\Mr{m'}{K,K',p'}$ for some $ m' \geq 0 $. Then the composition 
$R(U;t)M(U;t)$ is in $\Rr{-\rho+m'}{K,K',p+p'}$.

(iv) Let $R(V,W;t)$ be a smoothing operator of $\sR{-\rho}{K,K'_1,p}{N}$, depending linearly on $W $,  i.e. 
\[
R(V,W;t)  = \sum_{q=p}^{N-1} R_q(V,\dots,V,W) + R_N(V,W;t)
\]
where $R_q$ belongs to $\Rt{-\rho}{q}$ and $R_N$ satisfies, instead of \eqref{2117}, for any  $ 0 \leq k \leq K -  K_1' $, 
\begin{multline}
  \label{eq:237a}
    \norm{\partial_t^kR_N(V,W;t)U(t,\cdot)}_{\Hds{s+\rho-\frac{3}{2}k}}\\ \leq
     C\sum_{k'+k''=k}\Bigl(\Gcals{\sigma}{k'+K'_1,N-1}{V}\Gcals{\sigma}{k'+K'_1,1}{W}\Gcals{s}{k'',1}{U}\\
+ \Gcals{\sigma}{k'+K'_1,N-1}{V}\Gcals{s}{k'+K'_1,1}{W}\Gcals{\sigma}{k'',1}{U}\\
+ \Gcals{\sigma}{k'+K'_1,N-2}{V}\Gcals{\sigma}{k'+K'_1,1}{W} \Gcals{s}{k'+K'_1,1}{V}\Gcals{\sigma}{k'',1}{U}\Bigr).
\end{multline}
Let $M$ be in $ \sM{m'}{K,K_2',p'}{N} $. 
 Then $R(V,M(V;t)W;t)$ is in $\sR{-\rho+m'}{K,K',p+p'}{N}$, and it is linear in $W$.

All statements of the proposition have their counterpart for autonomous operators.
\end{proposition}
\begin{proof}
  (i) is nothing but the fourth remark following Definition~\ref{216}.

(ii) We may write \eqref{237}
in which we  replace $U_j$ by $\Pin{j}U_j$, acting on $\Pin{p+p'+1}U_{p+p'+1}$, and composed at
the left by $\Pin{0}$ as
\begin{multline}
  \label{eq:238}
\sum_{n'_{p+1}}\Pin{0} R(\Pin{1}U_1,\dots,\Pin{p}U_p)\\
\cdot\Pi_{n'_{p+1}}M(\Pin{p+1}U_{p+1},\dots,\Pin{p+p'}U_{p+p'})\Pin{p+p'+1}U_{p+p'+1} \, .
\end{multline}
By \eqref{2126}, the summation is limited to those $n'_{p+1}$ satisfying $\epsilon'_{p+1}
n'_{p+1}+\sum_{j = p+1}^{p+p'+1}\epsilon_j n_j=0 $ for some choice of the signs 
$\epsilon'_{p+1}, \epsilon_j $. 
Moreover, by \eqref{2115}, 
\eqref{2125} and $ n'_{p+1}  \leq C \maxn{p+1}{p+p'+1} $, 
the $L^2$ norm of \eqref{238}  is bounded from above by the sum in $n'_{p+1}$  
of
\begin{multline}
  \label{eq:239}
  C\frac{\max_2(n_1,\dots,n_p,n'_{p+1})^{\rho+\mu}}{\max(n_1,\dots,n_p,n'_{p+1})^{\rho}}
  \maxn{p+1}{p+p'+1}^{m'} \\ 
  \times 
  \Gcalsm{0}{0,p+p'+1}{\Pi_{n} \Ucal}
\end{multline} 
where  $n = (n_1,\dots,n_{p+p'+1})$ and  $ \Pi_{n} \Ucal = (\Pin{1} U_1,\ldots, \Pin{p+p'+1} U_{p+p'+1}) $. 
Assume first that one of the indices among $n_1,\dots,n_{p+p'+1}$ is much larger than all
the other ones, say $n_j$. If $p+1\leq j\leq p+p'+1$, then $n'_{p+1} \sim n_j$, so that \eqref{239} 
is smaller than
\[
C\frac{\maxdn{1}{p+p'+1}^{\rho+\mu}}{\maxn{1}{p+p'+1}^{\rho-m'}}\Gcalsm{0}{0,p+p'+1}
{\Pi_{n} \Ucal} \, .
\]
If, on the other hand, $1\leq j \leq p$, we have $n_j\gg n'_{p+1}$ and a similar bound follows. 
Finally, if the second
largest among $n_1\dots,n_{p+p'+1}$ is of the same magnitude 
as the largest one, i.e. $ \max_2 ( n_1\dots,n_{p+p'+1})  \sim  \max( n_1\dots,n_{p+p'+1})$, a bound of the form \eqref{2115} follows immediately (with $ \mu $ replaced by $ \mu + m' $).
Finally,  by \eqref{2116}   and \eqref{2126}, if \eqref{238} is non zero then  
$ \sum_{\ell=0}^{p+p'+1} \epsilon_\ell n_\ell = 0 $ for some choice of 
signs $ \epsilon_\ell \in \{ - 1, 1 \} $.
 Consequently, \eqref{237} belongs to $\Rt{-\rho+m'}{p+p'}$.

(iii) To study $R(U;t)\circ M(U;t)$, we use that  for $ 0 \leq k\leq K-K'$, we have by \eqref{2117},
\begin{multline}
\label{eq:239a}
  \norm{\partial_t^k (R(U;t)\circ M(U;t)V)}_{\Hds{s+\rho-m'-\frac{3}{2}k}} \\\leq C\sum_{k'+k''=k}\Bigl(\Gcals{\sigma}{k'+K',p}{U}
  \Gcals{s-m'}{k'',1}{M(U;t)V}\\
+ \Gcals{\sigma}{k'+K',p-1}{U}
  \Gcals{s-m'}{k'+K',1}{U}\Gcals{\sigma}{k'',1}{M(U;t)V}\Bigr) \, .
\end{multline}
Since, recalling \eqref{213}, \eqref{212}, 
\[
\Gcals{s-m'}{k'',1}{M(U;t)V} = \sum_{k'''\leq k''}\norm{\partial_t^{k'''}M(U;t)V}_{\Hds{s-m'-\frac{3}{2}k'''}},
\]
we may use \eqref{2127} to bound the first term in the right hand side of \eqref{239a} by
\begin{multline} \label{eq:stima-both}
C\sum_{k'+k''=k}\Bigl(\Gcals{\sigma}{k'+K',p+p'}{U}
  \Gcals{s}{k'',1}{V}\\
+ \Gcals{\sigma}{k'+K',p+p'-1}{U}\Gcals{\sigma}{k'',1}{V}
  \Gcals{s}{k'+K',1}{U}\Bigr).\end{multline}
The last term in \eqref{239a} satisfies a similar estimate. Therefore, 
for a new value of $\sigma$, \eqref{239a} is  bounded by \eqref{stima-both},
proving that $R(U;t)M(U;t)$ is in $\Rr{-\rho+m'}{K,K',p+p'} $.

(iv) We have to consider on the one hand composition of multilinear terms of the form, 
for $ q \in \{ p, \ldots, N - 1 \} $, $ q' \in \{ p', \ldots, N - 1 \} $, 
\be\label{eq:RqMq'}
R_q(U_1,\dots,U_{q-1},M_{q'}(U_{q+1},\dots,U_{q+q'})U_{q+q'+1})U_q \, . 
\ee
By \eqref{2115},  \eqref{2125}
we get that the $ L^2 $ norm of 
\begin{multline}
\sum_{n_q'} \Pi_{n_0} R_q \big(  \Pi_{n_1} U_1,\dots, \Pi_{n_{q-1}} U_{q-1}, \\
\Pi_{n_{q'}} M_{q'}( \Pi_{n_{q+1}}  U_{q+1},\dots, \Pi_{n_{q+q'}} U_{q+q'}) \Pi_{n_{q+q'+1}}  U_{q+q'+1} \big) \Pi_{n_q}  U_q
\end{multline}
is bounded by 
\begin{multline*}
  \label{eq:239bis}
  C\frac{\max_2(n_1,\dots,n_{q-1}, n_q', n_q)^{\rho+\mu}}{\max(n_1,\dots,n_{q-1}, n_q', n_q)^{\rho}}
  \max( n_q', n_{q+1}, \ldots, n_{q+q'+1} )^{m'} \\ 
  \times 
  \Gcalsm{0}{0,q+q'+1}{\Pi_{n} \Ucal} \, .
\end{multline*} 
Then  following the same arguments after  \eqref{239}, using \eqref{2126} and   \eqref{2116}, 
we prove that \eqref{RqMq'}  is an element of $\Rt{-\rho+m'}{p+q}$. 

On the other hand, we have to study the composition of  
\begin{equation*} 
\begin{split}
& R_{q}(V,M_{N}(V;t)W;t)U \, , \, \quad q = p, \ldots, N- 1 \, , \\
& R_{N}(V,M_{q'}(V;t)W;t)U \, , \quad q' = p', \ldots, N- 1 \, , 
\end{split}
\end{equation*}
and 
$$ 
R_{N}(V,M_{N}(V;t)W;t)U  \, .
$$ 
 If one combines \eqref{237a}
and \eqref{2127} as in the proof of (iii) above, one gets for the composition an estimate of the form \eqref{237a}, with $\rho$ replaced by $\rho-m'$, for a
new value of $\sigma$, and paying attention to the fact that the losses of time derivatives $K'_1$ for $R_q$ and $K'_2$ for
$M_{q'}$ cumulate. This shows that we obtain an element of $\Rr{-\rho+m'}{K,K',N}$. \end{proof}

Let us state a consequence of the preceding proposition.
\begin{proposition}  \label{234}
Let $a$ be a symbol in $\Gt{m}{p}$ and $M$  in $\sM{}{K,K',p'}{N-p}$, $ p + p' \leq N $, 
as in $(i)$ of  Proposition \ref{233}. Then
\begin{equation}
  \label{eq:2310}
  \opbw(a(U,\dots,U,W;\cdot))\vert_{W=M(U;t)U} = \opbw(c(U;t,\cdot)) + R(U;t)
\end{equation}
where  $ c $ is  the  symbol  in $  \sG{m}{K,K',p+p'}{N} $ given by \eqref{236},   
$$
c(U;t,\cdot) = a(U,\dots,U,M(U;t)U;t,x,\xi) \, , 
$$
and  $R(U;t)$  is a smoothing
operator in 
$\sR{-\rho}{K,K',p+p'}{N}$ for any $\rho$. If $M$ is autonomous, so is $R$ in the right hand side of \eqref{2310}.
\end{proposition}
\begin{proof}
  We decompose $M(U;t) = \sum_{q' = p'}^{N-p-1}M_{q'}(U,\dots,U) + M_{N-p}(U;t)$ with 
  $M_{q'}$ in $\Mt{}{q'}$ 
  and  $M_{N-p}$ in $\Mr{}{K,K',N-p}$. 
We consider first  the homogeneous terms. For $ q' = p',\dots,N-p-1 $ we define 
\[
c_{q'}(U_1,\dots,U_{p+q'};\cdot) = a(U_1,\dots,U_{p-1},M_{q'}(U_p,\dots,U_{p+q'-1})U_{p+q'};\cdot)
\]
that, by the fourth remark following Definition~\ref{216}, is an homogeneous symbol in $\Gt{m}{p+q'} $.
We want to prove that the difference  
$$
\opbw{(a(U_1, \ldots, U_{p-1}, W ))}\vert_{M_{q'}(U_p,\dots,U_{p+q'-1})U_{p+q'}} -
\opbw{ (c_{q'}(U_1, \ldots, U_{p+q'} )) } 
$$
is a smoothing operator in  $\Rt{-\rho}{p+q'}$. 
Let 
\begin{multline*}
  R_{q'}(U_1,\dots,U_{p+q'}) = \opw[a_{\chi_p}(U_1,\dots,U_{p-1},M_{q'}(U_p,\dots,U_{p+q'-1})U_{p+q'};\cdot)\\-
c_{q',\chi_{p+q'}}(U_1,\dots,U_{p+q'};\cdot)] \, .
\end{multline*}
By the definition \eqref{2112} of $a_{\chi_p}$
and $c_{q',\chi_{p+q'}} $ the symbol in the above expression computed at
$\Pin{1}U_1$,\dots,$\Pin{p+q'}U_{p+q'}$ may be written as
\begin{multline}\label{eq:2311}
\sum_{n'_p}a(\Pin{1}U_1$,\dots,$\Pin{p-1}U_{p-1},\\\Pi_{n'_p}M_{q'}(\Pin{p}U_p,\dots,\Pin{p+q'-1}U_{p+q'-1})\Pin{p+q'}U_{p+q'};x,\xi)\\
\times[\chi_p(n_1,\dots,n_{p-1},n'_p,\xi) - \chi_{p+q'}(n_1,\dots,n_{p+q'},\xi)] \, .
\end{multline}
The two cut-offs above are equal to one for $n_1+\cdots+n_{p+q'}\leq\delta'\absj{\xi}$ if $\delta'$ is small enough.
Indeed, for $ \chi_{p+q'}$ this directly follows by \eqref{2111}. 
Moreover, on the support of the $n'_p$ summation,  \eqref{2126} implies that 
$  n'_p = \sum_{j=p}^{p+q'} \epsilon_j n_j $ for some signs 
$\epsilon_p,\dots, \epsilon_{p+q'}$. 
As a consequence
$$ 
n_1 + \ldots n_{p-1}  + n'_p  \leq n_1+\cdots+n_{p+q'}\leq\delta'\absj{\xi}
$$
and,  if  $\delta'$ is small enough, 
$\chi_p(n_1,\dots,n_{p-1},n'_p,\xi) = 1$, by  \eqref{2111}. 
Consequently, on \eqref{2311}, we
have $\max(n_1,\dots,n_{p+q'})\sim \absj{\xi}$. According to the remarks after the proof of Proposition~\ref{215}, this
implies that $R_{q'}$ is an element of $\Rt{-\rho}{p+q'}$ for any $\rho$.

Concerning the non-homogeneous contribution, we define 
\[
c_{N-p}(U;t,\cdot) = a(U,\dots,U,M_{N-p}(U;t)U;\cdot)
\]
that, by  the fourth remark following Definition~\ref{216}, is a symbol in $ \Gr{m}{K,K',N} $. 

Recalling \eqref{2113}, the associated paradifferential operator 
$\opbw(c_{N-p}) $ is equal to  $ \opw(c_{N-p,\chi})$ where, by \eqref{2112}, 
\begin{multline}
  \label{eq:2312}
  c_{N-p,\chi}(U;t,x,\xi) = \chi(D,\xi)[c_{N-p}(U;t,x,\xi)]\\
= \sum_{n_1,\dots,n_p}\chi(D,\xi)[a(\Pin{1}U,\dots,\Pin{p-1}U,\Pin{p}M_{N-p}(U;t)U;t,x,\xi)] \, .
\end{multline}
If we restrict the sum in \eqref{2312} to 
$ n_1+\cdots +n_p \geq c\absj{\xi}$ for some $c$, it follows from \eqref{2121}, \eqref{2127}
that the associated operator satisfies \eqref{2117} for any $\rho$, as soon as $\sigma$ is large enough relatively to
$\rho$, and therefore it is a smoothing operator in $\Rr{-\rho}{K,K',N} $.  
We consider next the sum in \eqref{2312} restricted to 
the indices satisfying $n_1+\cdots+n_p < c \absj{\xi}$. 
By \eqref{215}, \eqref{regul-per} and \eqref{2111} we deduce that,  for $ c $ small enough, 
\begin{multline*}
 \sum_{n_1+\cdots+n_p < c \absj{\xi}}\chi(D,\xi)[a(\Pin{1}U,\dots,\Pin{p-1}U,\Pin{p}M_{N-p}(U;t)U;t,x,\xi) ] \\
=  \sum_{n_1+\cdots+n_p < c \absj{\xi}}  a(\Pin{1}U,\dots,\Pin{p-1}U,\Pin{p}M_{N-p}(U;t)U;t,x,\xi)  \, . 
\end{multline*}
Consequently, by \eqref{2111}, 
the difference between the above symbol and 
\begin{multline}\label{eq:sym-inner}
 a_{\chi_p}(U,\dots,U,M(U;t)U;t,x,\xi) = \sum_{n_1,\dots,n_p}\chi_p(n_1,\dots,n_p,\xi)\\
\times a(\Pin{1}U,\dots,\Pin{p-1}U,\Pin{p}M_{N-p}(U;t)U;t,x,\xi)
\end{multline}
is supported on indices satisfying  $n_1+\cdots+n_p \geq \delta' \abs{\xi} $ for some $ \delta' > 0 $.  
In conclusion, the difference between  \eqref{2312} and \eqref{sym-inner} 
is a smoothing operator of $\Rr{-\rho}{K,K',N}$.
\end{proof}
To end this subsection, we restate the Bony paralinearization formula 
of the composition operator, that will be used 
in Chapter \ref{cha:6}. 
\begin{lemma}
  \label{235} {\bf (Bony Paralinearization of the composition operator)}
Let $f$ be a smooth $\C$-valued function defined on a neighborhood of zero in $\C^2$, vanishing at zero at order $q\in
\N^{*}$. There is $ r >  0 $, a $1\times 2$ matrix of symbols $ a(U;\cdot) $
in $\Gra{0}{K,0,q-1}$ and a $1\times 2$ matrix of smoothing operators  $ R(U) $ 
of $\Rra{-\rho}{K,0,q'}$, $q' = \max(q-1,1)$, for any $\rho$,  such that
\begin{equation}
  \label{eq:2313}
  f(U) = \opbw(a(U;\cdot))U + R(U)U.
\end{equation}
Moreover, $a(U;\cdot)$ is  given by the  symbol $Df(U)$ which is independent of $\xi $.
\end{lemma}

\noindent
\textbf{Remark}: $ \opbw(a(U;\cdot))  $ is the para-product 
for the function $ a(U;x) = Df(U (x)) $, see Definition \ref{241}. 

\begin{proof}
Notice first that if $f$ is linear, \eqref{2313} is trivial, with $R=0$, so there is no restriction 
in assuming that $ D f(0) = 0 $, i.e. $ q \geq 2 $, since this amounts to add a linear term to $ f $. By the paralinearization formula of Bony, we know that $f(U) = T_{Df(U)}U + R(U)U$, where $R(U)$ satisfies \eqref{2117} and
  where
\[
T_{Df(U)}U = \frac{1}{2\pi}\int e^{i(x-y)\xi} c(U;x,\xi)U(y)\,dyd\xi
\]
with $ c(U;x,\xi) = \chi(\absj{\xi}^{-1}D)[Df(U)] $ for some $\chi\in C_0^\infty(\R)$ with small enough support, equal to one
close to zero. 
Defining  the $x$-periodic function $b(U;x,\xi)$ through its Fourier coefficients
\[\hat{b}(U;n,\xi) = \hat{c}\bigl(U;n,\xi-\frac{n}{2}\bigr), \  n\in \Z \, , \]
we deduce by \eqref{clwe} 
that  $T_{Df(U)}U = \opw(b(U;\cdot))U$. We may write
\begin{multline*}
  \hat{b}(U;n,\xi) = \chi\Bigl(n\Absj{\xi-\frac{n}{2}}^{-1}\Bigr)\widehat{Df(U)}(n)\\
= \chi(n\absj{\xi}^{-1})\widehat{Df(U)}(n) + \Bigl(\chi\Bigl(n\Absj{\xi-\frac{n}{2}}^{-1}\Bigr) -  \chi(n\absj{\xi}^{-1})\Bigr)\widehat{Df(U)}(n).
\end{multline*}
If the support of $\chi$ is small enough, we thus obtain
\[
T_{Df(U)}U = \opw(\chi(\absj{\xi}^{-1}D)[Df(U)])U + R_1(U)U
\]
where the first term is $\opbw(Df(U))U$, since the cut off $\chi(\absj{\xi}^{-1}D)$ satisfies the assumptions made after
\eqref{2111},  and the second one contributes to $R(U)U$ in \eqref{2313}. Notice that according to the 
second remark after the
proof of Proposition~\ref{215}, the definition of $\opbw(Df(U))U$ depends on the choice of $\chi$ only modulo a contribution
to $R_1(U)U$.
\end{proof}

\section{Paracomposition}\label{sec:Para}

We define in this section a paracomposition operator associated to a diffeomorphism 
 $  x \to x+ \beta (x) $ of $ \Tu $. 
The first property  of a paracomposition operator  is to  
be a bounded map between any Sobolev spaces, requiring for $ \beta (x)  $ just  limited smoothness.
The second property 
is to conjugate a para-differential operator into another para-differential one, up to smoothing remainders. 
Paracomposition operators have been introduced by Alinhac in \cite{Al}. 
We propose below an alternative  definition of paracomposition 
operator using flows.  

We first define the subspace of the functions formed 
by the symbols of $ \sG{0}{K,K',p}{N} $ introduced in Definition \ref{213}, which are
independent of $ \xi $.  
\begin{definition}
  \label{241} {\bf (Functions)}
Let $K'\leq K$ be in $\N$, $p$ in $\N$, $r>0$, $N$ in $\N$ with $p\leq N$. We denote by 

$ \bullet $ 
\index{Fa@$\Ft{p}$ (Space of homogeneous functions)} $\Ft{p}$
(resp. \index{Fb@$\Fr{K,K',p}$ (Space of non-homogeneous functions)} $\Fr{K,K',p}$, resp. \index{Fc@$\sF{K,K',p}{N}$ (Space of functions)}
$\sF{K,K',p}{N}$) the subspace of $\Gt{0}{p}$ (resp. $\Gr{0}{K,K',p}$, resp. $\sG{0}{K,K',p}{N}$) made of those symbols that are
independent of $\xi$;

$ \bullet $ 
 \index{Fd@$\sFR{K,K',p}{N}$ (Space of real functions)} $\sFR{K,K',p}{N}$ the restriction of elements
of $\sF{K,K',p}{N}$ to the subspace $\CKHR{\sigma}{\C^2}$ of
$\CKH{\sigma}{\C^2}$ defined after \eqref{212} that are moreover \emph{real valued}; 

$ \bullet $ $ \sFa{K,0,p}{N} = \sF{K,0,p}{N} \cap \sGa{0}{K,0,p}{N} $
the subspace of functions  \index{Fd@$\sFa{K,0,p}{N}$ (Space of autonomous functions)}
whose  dependence on $  t $ enters only through  $  U = U( t ) $. 
 \end{definition}

We now consider a diffeomorphism  $ x \to x+ \beta (U;t,x)$ of $ \Tu $ 
with a function  $ \beta (U; t, \cdot ) $ in the previous class.

\begin{lemma} {\bf (Inverse diffeomorphism)}
Let $ 0 \leq K' \leq K $ be in $ \N $ and $  \beta  ( x ) \stackrel{\mathrm{def}}{=} \beta (U; t,  x ) $ be a real valued function 
\be\label{eq:def:beta}
\beta (U; t, \cdot )  \in \Sigma {\cal F}_{K,K',1}[r,N] 
\ee
for $ U $ in the space $\CKHR{\sigma}{\C^2}$. 
If $\sigma$ is large enough, and $ U $ stays in the ball of center zero and radius $r$ in  $\CKHR{\sigma}{\C^2}$, the
map
\begin{equation}  \label{eq:2424}
  \Phi_U : x\to x+ \beta (U;t,x)
\end{equation}
is, for $r$ small enough, a diffeomorphism from $\Tu$ to itself, and 
its  inverse diffeomorphism may be written as
 \begin{equation} \label{eq:2425}
 \Phi_U^{-1} : y\to y + \gamma (U;t,y)
\end{equation}
for some $ \gamma  $ in $\sFR{K,K',1}{N}$.
\end{lemma}

\begin{proof}
By
Definition~\ref{241} and \eqref{218}, functions of $\Fr{K,K',1}$ satisfy bounds of the form 
\begin{equation}
  \label{eq:2438a}
  \abs{\partial_t^k\partial_x^\alpha \beta (U;t,x)}\leq C\Gcals{\sigma}{k+K',1}{U} \, ,
\  0\leq k\leq K-K' , \ \alpha\leq\sigma-\sigma_0 \, , 
 \end{equation}
if $\sigma\geq\sigma_0\gg 1$ and $U$ is in $\Brs{K}{0}{\sigma}\cap  C_*^{k+K'}(I,\Hds{\sigma}(\Tu,\C^2))$. 

Look first for a function $ \gamma(U;t,x)$, continuous in $(t,x) \in  I \times \Tu $, solving the fixed point problem
\begin{equation}
  \label{eq:fixedpoint}
F_\beta (\gamma) = \gamma \, , \quad  {\rm where}
\quad F_\beta (\gamma) \stackrel{\textrm{def}}{=} - \beta \circ(\mathrm{Id}+\gamma) \, . 
\end{equation}
By \eqref{2438a} applied with $\alpha=1, k=0, \sigma=\sigma_0+1$, we get 
$$
\| F_\beta(\gamma)-F_\beta(\gamma') \|_{L^\infty} \leq C \sup_{t \in I} \Gcals{\sigma_0+1}{K',1}{U}
\| \gamma-\gamma' \|_{L^\infty}
$$ 
where $ \| \gamma \|_{L^\infty} = \sup_{(t, x) \in I \times \Tu } |\gamma (t,x)| $,  and 
$$
\| F_\beta(\gamma) \|_{L^\infty } \leq  \| \beta \|_{L^\infty}  +  
C \sup_{t \in I} \Gcals{\sigma_0+1}{K',1}{U}
\| \gamma \|_{L^\infty} \, .
$$
 If $ C \sup_{t \in I} \Gcals{\sigma_0+1}{K',1}{U} \leq 1 / 2 $, the contraction mapping theorem 
 implies the existence of a unique continuous 
function $\gamma (U; t, x) $ solving \eqref{fixedpoint} and satisfying 
$ \| \gamma \|_{L^\infty} = 2 \| \beta \|_{L^\infty} =  O( \sup_{t \in I} \Gcals{\sigma_0+1}{K',1}{U}  ) $.
The solution $\gamma$ is actually more regular. Indeed,  consider the function
$ G : I \times \Tu \times \R \to \R  $
defined by
$$
G (t,x,y) \stackrel{\textrm{def}}{=} y + \beta (U; t, x + y ) \, , 
$$
so that $ G (t, x, \gamma (U; t, x)) = 0 $. 
Condition \eqref{2438a} implies that the function $ \beta $ is $ C^1 $ in $ (t,x)$, thus $ G $ is $ C^1 $ in $(t,x,y)$. 
Moreover
$$
(\pa_y G)(t,x,y) = 1 + (\pa_x \beta) (U; t, x + y )  \geq 1 -  C {\cal G}_{K',1}^{\sigma_0+1} (U, t) \geq \frac12 \, . 
$$ 
As a consequence, the implicit function theorem 
implies that the function $ \gamma (U; t, x )  $ is $ C^1 $  in $ (t,x)$  and 
\begin{equation*}
\begin{split} 
\pa_t \gamma (U; t, x) & = - \frac{ (\partial_t \beta) (U; t, x + \gamma (U; t, x))}{1 + 
(\partial_x \beta) (U; t, x + \gamma (U; t, x)) } \\
\pa_x \gamma (U; t, x) & = - \frac{(\partial_x \beta) (U; t, x + \gamma (U; t, x))}{1 + (\partial_x \beta) (U; t, x + \gamma (U; t, x)) } \, . 
\end{split}
\end{equation*}
By induction we deduce that there exist all the derivatives $ \pa_t^k \pa_x^\alpha \gamma (U; t, x) $
for any $ 0 \leq k \leq K - K' $ and  
$ \alpha $ such that $ k + \alpha \leq \sigma - \sigma_0 $. 
Moreover $\partial_t^k\partial_x^\alpha\gamma$ may be written, when $k+\alpha\geq 1$, as a combination of
functions of the form
\begin{multline}\label{eq:expder}
Q\bigl(\partial_t\beta
  (U; t,x+\gamma(U; t,x)), \partial_x\beta(U; t,x+\gamma(U; t,x))\bigr)\\
  \times\prod_{i=1}^m(\partial_t^{p_i+\ell'_i}\partial_x^{q_i+\ell''_i}\beta)(U;t,x+\gamma(U;t,x)),
\end{multline}
 where $Q$ is some rational map of its arguments, 
  $m$ is a positive integer, and $ p_i, q_i, \ell'_i, \ell''_i $
are non-negative integers satisfying  
\begin{equation}
  \label{eq:condindi}
  \ell'_i+\ell''_i = 1 \, ,\  \  \sum_{i=1}^m(p_i+q_i)\leq k+\alpha-1 \, , \  \   p_i+\ell'_i\leq k \,  .
\end{equation}
We shall use the  interpolation inequality
\be\label{eq:int-pro}
\norm{(\partial_x^{q_1}u_1)\cdots (\partial_x^{q_m}u_m)}_{L^\infty} \leq C 
\sum_{i=1}^m \prod_{j \neq   i} \norm{u_j}_{L^\infty} \norm{u_i}_{W^{Q, \infty}} 
 \ee
where $ Q = q_1 + \ldots + q_m $, 
which stems by Gagliardo-Nirenberg inequality
\begin{multline*}
\norm{(\partial_x^{q_1}u_1)\cdots (\partial_x^{q_m}u_m)}_{L^\infty} \leq C 
\| u_1 \|_{L^\infty}^{1- \frac{q_1}{Q}} \| u_1 \|_{W^{Q,\infty}}^{\frac{q_1}{Q}} \ldots 
\| u_m \|_{L^\infty}^{1- \frac{q_m}{Q}} \| u_m \|_{W^{Q,\infty}}^{\frac{q_m}{Q}} \\
\leq C 
\| u_1 \|_{L^\infty}^{\sum_{j \neq 1}\frac{q_j}{Q}} \| u_1 \|_{W^{Q,\infty}}^{\frac{q_1}{Q}} \ldots 
\| u_m \|_{L^\infty}^{\sum_{j \neq m}\frac{q_j}{Q}} \| u_m \|_{W^{Q,\infty}}^{\frac{q_m}{Q}} \\
\leq C \prod_{i=1}^m   \big(\prod_{j\neq  i} \norm{u_j}_{L^\infty} \norm{u_i}_{W^{Q, \infty}} \big)^{\frac{q_i}{Q}}
\leq C \sum_{i=1}^m \prod_{j\neq
    i} \norm{u_j}_{L^\infty} \norm{u_i}_{W^{Q, \infty}} \, . 
\end{multline*}
Let us estimate the modulus of \eqref{expder}. By \eqref{2438a} applied with $(k,\alpha) = (1,0)$ or $(0,1)$, 
we have
\begin{multline*}\abs{\partial_t\beta   (U;t,x+\gamma(U;t,x))} + \abs{\partial_x\beta(U;t,x+\gamma(U;t,x))} \\\leq
  C\Gcals{\sigma_0+1}{1+K',1}{U} \leq C r \, .
  \end{multline*}
Then, \eqref{int-pro}, \eqref{condindi} and \eqref{2438a} imply that the modulus of \eqref{expder} is bounded  by
\begin{multline*}
  C\sum_{i=1}^m\prod_{j\neq
    i}\norm{\partial_t^{p_j+\ell'_j}\partial_x^{\ell''_j}\beta}_{L^\infty}\norm{\partial_t^{p_i+\ell'_i}\partial_x^{\ell''_i}\beta}_{W^{Q,\infty}}\\
\leq C \sum_{i=1}^m\prod_{j\neq i}\Gcals{\sigma_0+1}{k+K',1}{U} \Gcals{\sigma_0+Q+ 1 
}{k+K',1}{U} 
\end{multline*}
and, since  $ Q\leq k-1+\alpha\leq K-K'+\alpha-1 $  by \eqref{condindi},  
  we get a bound  in 
  $$ 
  C \Gcals{\sigma_0+1}{k+K',m-1}{U} \Gcals{\sigma}{k+K',1}{U} \leq C'  \Gcals{\sigma}{k+K',1}{U}
  $$ 
  if  $\sigma\geq \alpha+\sigma'_0$ where $ \sigma'_0 =
  \sigma_0+K-K' $. 
In conclusion, we have proved  that $\gamma$ satisfies bounds of the form \eqref{2438a} with $ \sigma_0 $ replaced by
$ \sigma'_0 $,  i.e. $ \gamma $ belongs to $\Fr{K,K',1}$. 

Assume now that $ \beta $ is in $\sFR{K,K',1}{N}$. Then
\[ 
\beta (U;t,x) = \sum_{p'=1}^{N-1} \beta_{p'}(U,\dots,U;x) + \beta_N(U;t,x)
\]
with $ \beta_{p'}$ in $\Ft{p'}$ and $ \beta_N$ in $\Fr{K,K',N}$. 
In particular, $ \beta $ is in $\Fr{K,K',1}$, so that,  as proved above, there exists a solution $  \gamma(U;t,x) $ in
$ \Fr{K,K',1} $ of  the fixed
point problem
\begin{multline*}
  \gamma(U;t,x) = - \beta (U;t,x+\gamma(U;t,x))\\
= -\sum_{p'=1}^{N-1} \beta_{p'}(U,\dots,U;x+\gamma(U;t,x)) + \beta_N(U;t,x+\gamma(U;t,x)) \, .
\end{multline*}
By Taylor expansion of each $ \beta_{p'}(U,\dots,U;x+\gamma(U;t,x)) $ at the point $ x $,  
and substituting iteratively the same expression of $ \gamma $, 
we conclude,  
using the remarks following Definitions \ref{211} and \ref{212bis}, that $\gamma$ is in $\sFR{K,K',1}{N}$. 
Finally the function $ \beta_N(U;t,x+\gamma(U;t,x)) $ is in $ \Fr{K,K',N} $, as it can be verified using the formula
for the derivatives of composite functions. 
This concludes the proof of the lemma. 
\end{proof}

We want to associate to the composition operator
$$
u(x) \mapsto u( x +  \beta (x ))
$$
a paracomposition operator. 
The idea is the following. We consider the  family of associated composition operators 
\be \label{eq:omotopy}
u(x) \mapsto u( x + \theta \beta (x )) \, , \quad \theta \in \R  \, .
\ee
The group of transformations \eqref{omotopy} is  the flow of the 
transport equation 
\be\label{eq:transportPDE}
\begin{cases}
\pa_\theta u = b(U; \theta, t, x) \pa_x u \cr  
u(0) = u(x) 
\end{cases} 
\ee
where the function 
\be\label{eq:def:function-b}
b(U; \theta, t, x ) = \frac{\beta (U; t, x)}{ 1 + \theta  \beta_x (U; t, x)} \in \Sigma {\cal F}_{K,K',1}[r,N] \, . 
\ee
Indeed, differentiating \eqref{omotopy} we get 
$ \pa_\theta (u( x +  \theta \beta (x) ) ) = \beta (x) u_x ( x +  \theta \beta (x) ) $. 
Moreover 
$ \pa_x (u( x +  \theta \beta (x) ) ) = (1 + \theta  \beta_x (x)) \, u_x ( x +  \theta \beta (x) ) $
and we deduce that 
$$
\pa_\theta (u( x +  \theta \beta (x) ) ) = \frac{\beta (x)}{1+ \theta \beta_x (x) }  \pa_x (u( x +  \theta \beta (x) ) ) \, ,
$$
namely $ u( x +  \theta \beta (x) ) $ solves \eqref{transportPDE}.

The transport equation \eqref{transportPDE} can be written as  $ \pa_\theta u =  i {\rm Op} ( B(U; \theta, \cdot ) )  u  $
with the real valued symbol 
\be\label{eq:def:B}
B = B(U; \theta, \cdot )\stackrel{\textrm{def}}{=}  B(U; \theta, t, x, \xi ) = b(U; \theta, t,  x ) \xi 
\, . 
\ee
Notice that $ B(U; \theta, \cdot ) $ is in $  \Sigma \Gamma^1_{K,K',1}[r,N] $. 
For simplicity of notation we shall not denote in the sequel of this section the explicit 
$ t $ dependence of the symbols. 

As a good candidate of a para-composition operator associated to \eqref{omotopy}, 
we consider the flow $ \Omega_{B(U)} (\theta) $ of the linear para-differential equation
\be\label{eq:flow-homogeneous}
\begin{cases}
\frac{d}{d \theta} \Omega_{B(U)} (\theta) =
 i  {\rm Op}^{BW}  \big( B(U; \theta ) \big) \Omega_{B(U)} (\theta) \\
 \Omega_{B(U)}(0) = {\rm Id} \, . 
 \end{cases}
\ee
The operator $ \Omega_{B(U)} (1) $ has the first property 
required by a para-composition operator, namely to be a bounded map between any Sobolev spaces.
Indeed one can prove the following:  

\begin{lemma}\label{lem:flow}
Let $ 0 \leq K' \leq K $ be in $ \N $ and a function $ \beta (U; t, \cdot ) $  in the space $ \Sigma {\cal F}_{K,K',1}[r,N]  $ 
which is  real valued for $ U $ in the space $\CKHR{\sigma}{\C^2}$. 
Then there is $ \sigma \in \R_+ $ and, for any $ U \in C^K_{*\R} (I, {\dot H}^\sigma (\Tu, \R)) $, 
system \eqref{flow-homogeneous}  has a unique solution  $ \Omega_{B(U)} (\theta) $ defined
 for all $ |\theta | \leq 1 $. Moreover the 
flow operator   $ \Omega_{B(U)} (\theta) $
is bounded on $ {\dot H}^s $ for any $ s $, and there is 
$ r >  0 $, and, for any $ s \in \R $,  a constant $ C_s > 0 $ such that, for any $ U \in B_\sigma^K(I,r) $,   
for any $ W \in {\dot H}^s $ we have  
\be\label{eq:425-first}
C_s^{-1} \| W \|_{{\dot H}^s}  \leq \| \Omega_{B(U)} (\theta) W \|_{{\dot H}^s}  \leq C_s \| W \|_{{\dot H}^s} \, .
\ee
Moreover, for any $ 0 \leq k \leq K - K' $, we have
\be\label{eq:425-lemma}
\| \partial_t^k \big[ \Omega_{B(U)} (\theta) W \big] \|_{{\dot H}^{s-\frac{3}{2}k}}  \leq
C {\cal G}^s_{k,1} (W,t)  \, . 
\ee
\end{lemma}

\begin{proof}
The existence of 
the flow  operator $ \Omega_{B(U)}(\theta)  $ 
follows from the fact that the symbol $B(U; \theta)$ defined in \eqref{def:B} is  real. 
Actually, set for $\lambda>1$, 
$$ 
B_\lambda(U; \theta,  t,x,\xi) = B(U; \theta, t,x,\xi)\chi(\xi/\lambda)
$$ 
with $\chi$ in $C^\infty_0(\R)$
equal to one close to zero. Define for any given $W$ in $\Hds{s}$ for some $s$, the function $W_\lambda(\theta)$ as the
solution of the Banach space ODE
\begin{equation}
  \label{eq:423a}
\begin{split}
  \frac{d}{d\theta}W_\lambda(\theta) &= i\opbw(B_\lambda(U; \theta))W_\lambda(\theta)\\
W_\lambda\vert_{\theta=0} &= W.
\end{split}
\end{equation}
Let us show that we have a  bound, uniform in $\theta\in [-1,1]$ and $\lambda>1$, as 
\begin{equation}
  \label{eq:423b}
  C_s^{-1}\norm{W}_{\Hds{s}}\leq \norm{W_\lambda(\theta)}\leq C_s \norm{W}_{\Hds{s}} \, . 
\end{equation}
Set  $\Lambda_s(\xi) = (1+\xi^2)^{\frac{s}{2}}$, so that 
$\norm{\opbw(\Lambda_s(\xi))W}_{\Hds{0}} = \norm{ \Lambda_s(D) W}_{\Hds{0}} $ 
is   equivalent
to the Sobolev norm $\norm{W}_{\Hds{s}}$. Notice that, 
since $B(U; \theta)$ is a symbol  of order one, 
Propositions~\ref{222}, ~\ref{215} and the remark following the proof of Proposition~\ref{222}, imply that
\[[\opbw(\Lambda_s),\opbw(B_\lambda(U; \theta))]\opbw(\Lambda_s^{-1})\] is bounded on $\Hds{0}$ for any $s$, with operator norm
$O(\norm{U}_{\Hds{\sigma}})$ for some large enough $\sigma$, uniformly for $\lambda>1$. As \eqref{423a} implies
\begin{multline*}
  \frac{d}{d\theta}\Lambda_s(D)W_\lambda(\theta) = i\opbw(B_\lambda(U; \theta))[\Lambda_s(D)W_\lambda(\theta)]\\
+O_{L^2}(\norm{\Lambda_s(D)W_\lambda(\theta)}_{L^2}\norm{U}_{\Hds{\sigma}}),
\end{multline*}
we deduce  \eqref{423b} from the $L^2$ energy inequality
\[\norm{\Lambda_s(D)W_\lambda(\theta)}_{L^2} \leq \norm{\Lambda_s(D)W}_{L^2} +C\norm{U}_{\Hds{\sigma}}\Abs{\int_0^\theta
\norm{\Lambda_s(D)W_\lambda(\theta')}_{L^2}\,d\theta'}.\]
The estimate is uniform in $\lambda$ since $\opbw(B_\lambda) - \opbw(B_\lambda)^*$ is uniformly bounded on $L^2$, as a
consequence of the fact that the symbol $B_\lambda$ is a real.

We thus get a family of functions $(W_\lambda)_\lambda$ uniformy bounded in the space $C^0([-1,1],\Hds{s})$ that is equicontinuous in
$C^0([-1,1],\Hds{s'})$ for any $s'<s$. By Ascoli theorem, we may extract a subsequence that converges in the latter space to
a limit $W$ that solves \eqref{423a} with $ B_\lambda$ replaced by $ B $ and provides the definition of 
$ \Omega_{B(U)} ({\theta}) W =
W(\theta)$. Since 
\[\frac{\partial}{\partial\theta}\norm{ \Omega_{B(U)}(\theta) W}_{\Hds{s}}^2 
= 2\Re\absj{i\opbw(B(U; \theta))  \Omega_{B(U)}(\theta) W,  \Omega_{B(U)}(\theta) W}_{\Hds{s}}\]
the fact that $\opbw(B(U; \theta)) $ is a self-adjoint operator of order one, and symbolic
calculus, imply that $\norm{  \Omega_{B(U)}(\theta) W}_{\Hds{s}}$ 
converges to $\norm{  \Omega_{B(U)}(\theta') W}_{\Hds{s}}$, when $\theta$ goes to
    $\theta'$, from which one deduces easily that $  \Omega_{B(U)}(\theta) W$ 
    is continuous on $[-1,1]$ with values in $\Hds{s}$. 
Let us notice that \eqref{423b} applied to $W$ instead of $W_\lambda$ 
shows  that  \eqref{425-first} holds true.

By a similar reasoning the norms
$\norm{\sum_{k'=0}^k\partial_t^{k'}\cdot}_{\Hds{s-\frac{3}{2}k'}}$ satisfy an analogous inequality. 
More precisely, if we write
\[\frac{d}{d\theta}\partial_t  \Omega_{B(U)}(\theta)  = i\opbw(B(U; \theta))\partial_t
 \Omega_{B(U)}(\theta)  + i\opbw(\partial_t B(U; \theta))  \Omega_{B(U)}(\theta) \]
we deduce by Duhamel   principle that
$$
\partial_t  \Omega_{B(U)}(\theta) =  \Omega_{B(U)}(\theta) 
\int_0^\theta  \Omega_{B(U)}(\tau)^{-1}   i\opbw(\partial_t B(U; \theta))  \Omega_{B(U)}(\tau) \, d \tau  \, .
$$
Applying the estimate \eqref{425-first} for $  \Omega_{B(U)}(\theta) $ and for $  \Omega_{B(U)}(\theta)^{-1} $, and using 
Proposition \ref{215} (in particular \eqref{2123}), we get that 
$$
  \norm{(\partial_t  \Omega_{B(U)}(\theta) )W}_{\Hds{s-1}}\leq C\norm{W}_{\Hds{s}}\Gcals{\sigma}{1+K',1}{U} \, .
$$
Since $ 1 \leq \frac{3}{2} $ the estimate \eqref{425-lemma} for $ k = 1 $ follows.
Iterating the above reasoning we get, for any $ 1 \leq k \leq K - K' $, the estimate
$$ 
  \norm{(\partial_t^k  \Omega_{B(U)}(\theta) )W}_{\Hds{s-k}}\leq C\norm{W}_{\Hds{s}}\Gcals{\sigma}{k+K',1}{U} 
$$
and, since $ U \in B_\sigma^K(I,r) $, we deduce \eqref{425-lemma}. 
\end{proof}

We   prove below that, conjugating  a para-differential operator 
\be\label{eq:def:a-para}
{\rm Op}^{BW} ( a (U; \cdot )) \, ,  \qquad 
a (U; \cdot ) \in \Sigma \Gamma^m_{K,K',q}[r,N] \, ,
\ee
with the flow $ \Omega_B (\theta ) $,  
we still get a para-differential operator.
 The conjugated operator
\be\label{eq:true-conjugation}
A(U)(\theta) := \Omega_{B(U)} (\theta) {\rm Op}^{BW} (a (U; \cdot )) \big(\Omega_{B(U)} (\theta)\big)^{-1} 
\ee
solves the Heisenberg equation
\be\label{eq:Heis}
\begin{cases}
\frac{d}{d \theta} A(U)( \theta)  =
 i \big[ {\rm Op}^{BW}  \big( B(U; \theta ) \big) , A(U)( \theta)\big] \\
A(U)( 0 ) = {\rm Op}^{BW} (a (U; \cdot ))  \, . 
 \end{cases}
\ee
We now solve \eqref{Heis} in decreasing orders showing that it admits an approximate solution of the form
$$
A(U)( \theta) = {\rm Op}^{BW} ( a_0 (\theta) + a_1 (\theta)  + \cdots ) 
$$
 with
$$
a_0 (\theta )  \in \Sigma \Gamma^m_{K,K',q}[r,N] \, , \quad a_1 (\theta)  \in \Sigma \Gamma^{m-2}_{K,K',q+1}[r,N] \, , \ \ldots 
$$
and $ a_0 (\theta) $ is given in \eqref{sol1-tr}. 
In particular we will prove that
\be\label{eq:sol-a-0-t1}
a_0(U; 1 ,t, x, \xi ) = a \Big(  U;t, x +  \beta (U; t,x),  
\xi \big( 1+  \partial_y \gamma (U;t,  y) \big)_{|y = x +  \beta(U; t,x) }   \Big) 
\ee
which is the usual formula in Alinhac 
para-composition theorem, see Theorem \ref{thm:para}. 
Notice that the symbol $ a_1 $ 
has  order $  m - 2 $, twice smaller than $ a_0 $. This is an outcome of the Weyl quantization. 

\medskip

Before starting the proof, 
we first  check the
stability of our classes of symbols under composition by a change of variables.
\begin{lemma}
  \label{246}
Let $K'\leq K$ in $\N$, $m$ in $\R$, $p$ in $\N$, $N$ in $\N$ with $p\leq N$, $r>0$ small enough. Consider 
a symbol $ a $ in
$\sG{m}{K,K',p}{N}$ and functions $ b, c $ in $\sFR{K,K',1}{N}$. Then
\begin{equation}
  \label{eq:2429}
  a\bigl(V;t,x+b(V;t, x),\xi( 1 + c(V;t, x))\bigr)
\end{equation}
is in $\sG{m}{K,K',p}{N}$. In particular, if the symbol $a$ is independent of $\xi$, i.e.\ $ a $ is a function 
in $\sF{K,K',p}{N}$, we obtain that $a(V;t,x+b(V;t, x))$ is in $\sF{K,K',p}{N}$.
\end{lemma}
\begin{proof}
  We decompose
\[
\begin{split}
  a &= \sum_{q=p}^{N-1} a_q(V,\dots,V;\cdot) + a_N(V; t, \cdot)\\
b &= \sum_{q'=1}^{N-1} b_{q'}(V,\dots,V;\cdot) + b_N(V; t, \cdot)\\
c &= \sum_{q''=1}^{N-1} c_{q''}(V,\dots,V;\cdot) + c_N(V; t, \cdot)
\end{split}\]
with $a_q$ (resp.\  $b_{q'}$, resp.\  $c_{q''}$) in $\Gt{m}{q}$ (resp.\  $\Ft{q'}$, resp.\ $\Ft{q''}$) and 
$a_N$ (resp.\ $b_{N}$, $c_{N}$) in $\Gr{m}{K,K',N}$ (resp.\  $\Fr{K,K',N}$). By a Taylor expansion at order 
$N$
in $(x,\xi)$, we may write \eqref{2429} as a linear combination of terms of the form 
\be\label{eq:tay-ex}
(\partial_x^\alpha\partial_\xi^\beta a)(V; t, x,\xi) b (V; t, x)^\alpha \xi^\beta c (V; t, x)^\beta
\ee
with  $ \alpha, \beta $ in $ \N $, $ 0 \leq \alpha+\beta\leq N-1 $, 
plus an integral remainder
\begin{multline}\label{eq:tay-re}
  \int_0^1(1-\lambda)^{N-1}(\partial_x^\alpha\partial_\xi^\beta a)(V;t,x+\lambda b(V;t, x),\xi(1+\lambda c(V;t, x))\,d\lambda\\
\times b(V;t,x)^\alpha(\xi c(V;t,x))^\beta
\end{multline}
with $\alpha+\beta = N $. 
Substituting  in \eqref{tay-ex} the homogeneous and non-homogeneous 
symbols in the expansion of $ a, b, c $, we obtain the following contributions:  

$\bullet$ For any $ \alpha, \beta $ in $ \N $, $ 0 \leq \alpha+\beta\leq N-1 $,
\be\label{eq:symb-tay}
  (\partial_x^\alpha\partial_\xi^\beta a_q)(V ;x,\xi) b_{q'}(V; x)^\alpha \xi^\beta c_{q''}(V; x)^\beta
\ee
with indices $ p \leq q \leq N - 1 $, $ 1 \leq q' , q'' \leq N - 1 $. 
By the remarks following Definition \ref{211}, \eqref{symb-tay}
is the restriction to
$V_1=\cdots=V_{\tilde{q}} = V$ of an homogeneous symbol in $\Gt{m}{\tilde{q}}$ with 
$ \tilde{q} = q + q' \alpha + q'' \beta \geq q $. 
If $ \tilde{q} \leq N - 1 $ 
this gives an homogeneous term  in the expansion of \eqref{2429}.
If $ \tilde q\geq N$,  it gives a non-homogeneous symbol in $\Gr{m}{K,0,N} \subset \Gr{m}{K,K',N}$.

$\bullet$ Symbols 
$  (\partial_x^\alpha\partial_\xi^\beta a_q) b_{q'}^\alpha \xi^\beta c_{q''}^\beta $
where $\alpha+\beta\leq N-1$,  but with at least one index among $q, q', q''$ equal to $N$. 
By  the remarks following Definition \ref{212bis},  
$  (\partial_x^\alpha\partial_\xi^\beta a_q)  \xi^\beta  $ is in $ \Gamma_{K,K',q}^m [r] $
and $ b_{q'}^\alpha c_{q''}^\beta $ is in $ \Gamma_{K,K',\alpha q' + \beta q''}^0 [r] $, so that their 
product gives a contribution in $\Gr{m}{K,K',N}$.

$\bullet $ 
Also the integral remainder \eqref{tay-re}  
gives  a contribution to $\Gr{m}{K,K',N}$. Indeed 
the symbol $ \xi^\beta b^\alpha c^\beta $
 is in $ \Gamma_{K,K',N}^{\beta} [r]  $ and it is sufficient to prove that
$ (\partial_x^\alpha\partial_\xi^\beta a)(V;t,x+\lambda b(V;t, x),\xi(1+\lambda c(V;t, x)) $ is in 
$ \Gamma^{m - \beta}_{K,K',0}  [r] $, uniformly in $ \lambda \in [0,1]$. This can be verified by the 
formula for the derivative of the composition of functions.

This concludes the proof. 
\end{proof}

We use the following asymptotic expansion formula for the commutator,
 which is derived by the composition results of para-differential operators for the Weyl  
quantization.

\begin{lemma} \label{lem:comm}
{\bf (Commutator  expansion)}
Let  $ a $ be in $  \Sigma \Gamma^m_{K,K',q} [r,N] $ and $ B  $ in $ \Sigma  \Gamma^1_{K,K',1}[r,N] $. Then  
\be
\begin{split}\label{eq:commutator-formula}
\big[{\rm Op}^{BW} ( i B(U;   \cdot ) ),  {\rm Op}^{BW} ( a )\big] & =  
{\rm Op}^{BW}  (\{ B(U; \cdot ),  a \}) \\ 
& +   {\rm Op}^{BW} \big( r_{-3, \rho}(B,a) \big)   + R 
\end{split}
\ee
with  symbols 
\be\label{eq:stima-simbols}
 \{ B(U; \cdot ),  a \}  \in \Sigma \Gamma^{m}_{K,K',q+1}[r,N] \, , \ 
r_{-3, \rho}(B,a) \in \Sigma \Gamma^{m-2}_{K,K',q+1}[r,N] \, , 
\ee
and  a smoothing remainder  
\be\label{eq:remainder1}
R \in \Sigma {\cal R}^{-\rho + m  + 1}_{K,K',q  +1 } [r, N] \, .
\ee
Note that the symbol $ r_{-3, \rho}(B,a) $ has order $ m -2 $. If $B$ is a bounded family of symbols depending on some
parameter $\theta \in [-1,1]$, then the above conclusion holds uniformly in $\theta$.
\end{lemma}

\begin{proof}
It follows by Proposition \ref{231} and formulas \eqref{222} and \eqref{223}. Note that  
the terms of even (resp. odd) rank in the asymptotic expansion of the symbol   
$ a \# B $ in the Weyl quantization 
are symmetric (resp. antisymmetric) in $ (a,B) $. Consequently the terms of even rank vanish in the symbol
of the commutator.
\end{proof}

By \eqref{commutator-formula},  the equation \eqref{Heis} is solved, at the highest order, by 
$$
A_0 (U)( \theta) = {\rm Op}^{BW} ( a_0(U; \theta, \cdot  ) )
$$
where  $  a_0 (U; \theta, \cdot ) = a_0 (\theta) $ is the solution of the transport equation
\be\label{eq:def:a0}
\begin{cases}
\frac{d}{d \theta} a_0 ( \theta)  =   \{ B(U; \cdot ),  a_0 (\theta) \} \\
a_0 (U)( 0 ) =  a (U; \cdot )  \, . 
 \end{cases}
\ee
This transport equation may be solved by the methods of characteristics. 

\begin{lemma}\label{lem:trasport}
The solution of the transport equation \eqref{def:a0} is 
\be\label{eq:sol1-tr}
a_0(U; \theta , x, \xi ) = a (U; \phi^{\theta,0} (x,\xi) ) 
\ee
where 
\be\label{eq:sol-cara}
\phi^{\theta, 0} (x, \xi) = \Big( x + \theta \beta (U; x),  
\xi \big( 1+   \partial_y \gamma (U; \theta, y) \big)_{|y = x + \theta \beta(U; x) }   \Big) 
\ee
are the solutions of the characteristic  system
\be\label{eq:HS-car}
\begin{cases}
\frac{d}{ds} x (s) = - b(s,  x(s) ) \cr
\frac{d}{ds} \xi (s) = b_x (s, x(s) ) \xi (s) \, .
\end{cases} 
\ee
In particular for $ \theta =1  $ we get \eqref{sol-a-0-t1}. 
Moreover  $ a_0 (\theta, \cdot) \in \Sigma \Gamma^m_{K,K',q}[r,N]$, 
with estimates uniform in $ | \theta | \leq 1 $. 
\end{lemma}

\begin{proof}
The function   $ s \mapsto a_0 ( s,  x(s), \xi(s))  $ is constant 
along the solutions of the Hamiltonian system generated by $ b(s,x) \xi $, namely \eqref{HS-car}.
Hence, denoting by $ \phi^{\theta_0,\theta} (x,\xi) $ the solution  of 
\eqref{HS-car}  with initial condition $ \phi^{\theta_0,\theta_0} (x,\xi) = (x,\xi)  $, the solution of the transport equation 
\eqref{def:a0} is 
$$
a_0(\theta, x, \xi ) = a (\phi^{\theta,0} (x,\xi) ) 
$$
proving \eqref{sol1-tr}. 
The solutions of the Hamiltonian system \eqref{HS-car} are directly given by the path of diffeomorphisms 
$$
y = x + \theta \beta (x) \quad \Leftrightarrow \quad x = y + \gamma (\theta, y ) \, , \quad \theta \in [-1,1] \, . 
$$  
Indeed, recalling the definition of  $ b(\theta, x )$  in \eqref{def:function-b}, $ (x(\theta), \xi(\theta)) $ defined by 
$$
x(\theta) :=  
y + \gamma (\theta, y )  \, , \quad
\xi(\theta)  :=  \xi_0 (1+  \partial_x \beta ) ( x(\theta) ) =  \frac{\xi_0}{1+  \partial_y \gamma (\theta, y) } 
$$
is the solution of \eqref{HS-car}  with initial condition
$$
x(0)  = y  \, , \quad \quad \xi(0) = \xi_0 \, . 
$$
Taking the inverse flow one obtains \eqref{sol-cara}. 

Finally, 
since  $ a  $ is in $  \Sigma \Gamma^m_{K,K',q}[r,N] $ and $ \beta $  in $ \Sigma {\cal F}_{K,K',1}[r,N] $, 
lemma \ref{246}   implies that 
the symbol $ a_0 (\theta, \cdot) $ defined in \eqref{sol1-tr} is 
in $ \Sigma \Gamma^m_{K,K',q}[r,N]$, with estimates uniform in  $ | \theta | \leq 1 $.  
\end{proof}

Let us now quantify how approximately
$ A^{(0)} (\theta) = {\rm Op}^{BW} ( a_0 (\theta) ) $ solves \eqref{Heis}.
By \eqref{def:a0} and \eqref{commutator-formula} we have 
\be\label{eq:approx-0}
\begin{split}
\pa_\theta  A^{(0)} (\theta) & =  {\rm Op}^{BW} (   \partial_\theta a_0 (\theta)  ) = 
 {\rm Op}^{BW} ( \{ B(U; \cdot ),  a_0 (\theta) \} )  \\
 & =
 \big[{\rm Op}^{BW} ( i B(U;  \theta, \cdot ),  {\rm Op}^{BW} ( a_0 (\theta) )\big] - 
 {\rm Op}^{BW} \big( r_{-3, \rho}(B,a_0 (\theta) ) \big)  - R   \\
& =  \big[{\rm Op}^{BW} ( i B(U;  \theta, \cdot ),    A^{(0)} (\theta) \big] - 
 {\rm Op}^{BW} \big( r_{-3, \rho}(B,a_0(\theta) ) \big)  - R (\theta) \, . 
 \end{split}
\ee
Since $ a_0 (\theta, \cdot) \in \Sigma \Gamma^m_{K,K',q}[r,N]$,  for any $ \theta \in [-1,1] $, we get,
by \eqref{stima-simbols} and \eqref{remainder1},  
\be\label{eq:def:primo-R}
r_{-3, \rho}(B,a_0 (\theta) ) \in  \Sigma \Gamma^{m-2}_{K,K',q+1}[r,N] \, , \quad 
R(\theta) \in \Sigma {\cal R}^{-\rho + m  + 1}_{K,K',q  +1 } [r, N] \, .
\ee
Thus $ A^{(0)} (\theta) $ is an approximate solution of \eqref{Heis} up to a symbol of order $ m - 2 $. 
We define next the further approximation 
\be\label{eq:new:A1}
A^{(1)} ( \theta) \stackrel{\textrm{def}}{=}  {\rm Op}^{BW} ( a_0 + a_1  ) = A^{(0)} ( \theta) +  {\rm Op}^{BW} ( a_1 (\theta) ) 
\ee
of the equation \eqref{Heis}. We choose 
 the symbol  $ a_1 (U; \theta, \cdot ) $ to solve the non-homogeneous transport equation 
\be\label{eq:better-approQ1}
\begin{cases}
  \partial_\theta a_1 (\theta)   -
 \{ B(U;  \theta, \cdot ),   a_1(\theta)  \} = r_{-3, \rho}(B, a_0(\theta))    \\
 a_1 (0) = 0 \, .
 \end{cases}
\ee
By \eqref{approx-0} we have  
\begin{align*}
\partial_\theta A^{(1)} ( \theta) &   = \pa_\theta A^{(0)} ( \theta)   +   {\rm Op}^{BW}  ( \partial_\theta a_1 (\theta)  ) \\
 & =  \big[{\rm Op}^{BW} ( i B(U;  \theta, \cdot ),    A^{(0)} (\theta) \big] - 
 {\rm Op}^{BW} \big( r_{-3, \rho}(B,a_0 (\theta)) \big) \\
 & \quad  - R (\theta)
   +   {\rm Op}^{BW}  ( \partial_\theta a_1 (\theta) ) \, . 
 \end{align*}
 Hence, recalling \eqref{new:A1} and 
  using \eqref{commutator-formula}, we get 
 \begin{align}
& \partial_\theta A^{(1)} ( \theta)  \nonumber  \\
& =  i \big[{\rm Op}^{BW} ( B(U;  \theta ),  A^{(1)} ( \theta)  \big] 
  +  {\rm Op}^{BW} \big(\partial_\theta a_1 -  r_{-3, \rho}(B,a_0) \big) ) \nonumber \\
 & \quad  -
  i \big[{\rm Op}^{BW} ( B(U;  \theta ),  {\rm Op}^{BW}  ( a_1  )  \big]  - R (\theta)  \nonumber  \\
&  = i \big[{\rm Op}^{BW} ( B(U;  \theta ),  A^{(1)} ( \theta)  \big] 
  + {\rm Op}^{BW} \big(    \partial_\theta a_1 - r_{-3, \rho}(B,a_0)  \nonumber \\
  & \quad  -
 \{ B(U;  \theta, \cdot ),   a_1  \}  \big)  + {\rm Op}^{BW}  \big( r_{-3, \rho} (B, a_1) \big)  + R'   \nonumber 
\end{align}
where the symbol $ r_{-3, \rho} (B, a_1) $ and the operator $ R'  $ are lower order terms that we estimate below. 
Since $ a_1 $ solves \eqref{better-approQ1} we get 
\be\label{eq:tolgo1}
\begin{cases}
\partial_\theta A^{(1)} ( \theta) =  i \big[{\rm Op}^{BW} ( B(U;  \theta, \cdot ),  A^{(1)} ( \theta)  \big]  
+ {\rm Op}^{BW}  \big( r_{-3, \rho} (B, a_1) \big)  + R' (\theta)  \\
A^{(1)} ( 0 ) =  {\rm Op}^{BW}  (a) \, . 
\end{cases}
\ee
The non-homogeneous transport 
equation \eqref{better-approQ1} can be solved by the method of characteristics.

\begin{lemma} \label{lem:non-homo}
The solution of \eqref{better-approQ1} may be written as 
\be\label{eq:formula:a1}
a_1 (\theta, x, \xi) = \int_0^\theta r_{-3, \rho}(B(s, \cdot), a_0 (s, \cdot) ) ( \phi^{\theta, s}  (x,\xi) ) \, ds   
\ee
where 
$ \phi^{\theta, s}  (x,\xi) $  is the solution of \eqref{HS-car}
that satisfies the initial condition 
$ \phi^{\theta, \theta}  (x,\xi) = (x,\xi)  $.
In particular 
\be\label{eq:stima:a1-s}
a_1 (1, x, \xi) = \int_0^1 r_{-3, \rho}(B(s, \cdot) ,a_0 (s, \cdot ) ) ( \phi^{1, s}  (x,\xi) ) \, ds   \, . 
\ee
Moreover the symbol $ a_1 (\theta, x, \xi) $ is in  $ \Sigma \Gamma_{K,K',q+1}^{m-2} [r,N] $, with 
estimates uniform in $ | \theta | \leq 1 $. 
\end{lemma}

\begin{proof}
Let $ a_1 (\theta, x, \xi ) $ be a solution of \eqref{better-approQ1} and 
$ \phi^{0,s} (x_0, \xi_0 ) = (x(s), \xi(s)) $ be the solution of 
the characteristic system \eqref{HS-car} with initial condition $ \phi^{0, 0} (x_0, \xi_0 )  = (x_0, \xi_0 )$.
Thus the derivative of the scalar function $ s \mapsto $ $ a_1 (s, x(s), \xi(s)) $ satisfies 
\begin{align*}
\frac{d}{ds} a_1 (s, x(s), \xi(s)) & 
=  \partial_s a_1 (s, x(s), \xi(s) )   -  \{ B(U;  s, \cdot ),   a_1(s, \cdot ) \} (x(s), \xi(s)) \\ 
& =  r_{-3, \rho}(B (s, \cdot ) ,a_0 (s,  \cdot )  )(x(s), \xi(s) ) \, .
\end{align*}
Since $  a_1 (0, x_0, \xi_0 ) = 0 $ we deduce that 
\begin{align}
a_1 (\theta, \phi^{0, \theta} (x_0, \xi_0))   & = a_1 (\theta, x(\theta), \xi(\theta)) \\
&  =  \int_0^\theta 
r_{-3, \rho}(B (s, \cdot ) ,a_0 (s,  \cdot )  )(x(s), \xi(s) ) \, d s\nonumber \\
& =  \int_0^\theta 
r_{-3, \rho}(B (s, \cdot ) ,a_0 (s,  \cdot )  )( \phi^{0,s} (x_0, \xi_0 ) ) ) \, d s  \, . \label{eq:prima-int}
\end{align}
Setting 
$$ 
(x, \xi) = \phi^{0,\theta} (x_0, \xi_0) \quad \Longleftrightarrow \quad 
(x_0, \xi_0) = \phi^{\theta,0} (x, \xi)  
$$
we deduce by \eqref{prima-int}, and using the 
composition rule of the flow $ \phi^{0,s} \phi^{\theta,0} = \phi^{\theta,s}  $, that 
\begin{align}
a_1 (\theta, x, \xi)   =  \int_0^\theta 
r_{-3, \rho}(B (s, \cdot ) ,a_0 (s,  \cdot )  )( \phi^{\theta,s} (x, \xi)   ) ) \, d s  \label{eq:a1-inte}
\end{align}
which is \eqref{formula:a1}. 
Now \eqref{def:primo-R}, 
 lemma \ref{246} and lemma~\ref{lem:trasport} 
imply that the symbol  $ r_{-3,\rho} (B, a_0)(\phi^{\theta,s} (x, \xi)  ) $ is in $ \Sigma \Gamma_{K,K',q+1}^{m-2}[r,N]  $,
with estimates uniform in $ | \theta |, |s|  \leq 1 $. Thus integrating \eqref{a1-inte} we finally  deduce that 
$ a_1 (\theta)  \in \Sigma \Gamma_{K,K',q+1}^{m-2} [r,N] $, with estimates uniform in $ | \theta | \leq 1 $. 
\end{proof}

Let us now quantify how approximately 
$ A^{(1)} (\theta)  $ defined in \eqref{new:A1}-\eqref{better-approQ1} solves \eqref{Heis}, i.e.
estimate the lower order terms in \eqref{tolgo1}.   
By lemma \ref{lem:trasport}  the symbol 
$ a_1 (\theta, \cdot) $ is in $  \Sigma \Gamma^m_{K,K',q}[r,N]$, with estimates uniform  in $ | \theta | \leq 1 $,  
 and by \eqref{stima-simbols} (with $ m - 2 $ instead of $ m $, and $ q + 1 $ instead  of $ q $),  we get 
$$
r_{-3, \rho} (B, a_1) \in \Sigma \Gamma^{m-4}_{K,K',q+2}[r,N] \,  ,
$$
uniformly in $ | \theta | \leq 1 $. 
On the other hand the remainder $ R'  $  in \eqref{tolgo1} is the sum of the remainder  $ R $ in 
$ \Sigma {\cal R}^{-\rho + m  + 1}_{K,K',q  +1 } [r, N] $ defined in \eqref{def:primo-R} plus a 
smaller contribution in $  \Sigma {\cal R}^{m-1-\rho}_{K,K',q+ 2} $, thus 
$ R'  $ is in $ \Sigma {\cal R}^{-\rho + m  + 1}_{K,K',q  +1 } [r, N] $. 

Repeating $ \ell $ times ($ \ell \sim \rho / 2 $)  the above procedure, 
until the new paracomposition term  may be incorporated into the remainder $ R $, we find  a
solution 
\be\label{eq:def:Aell}
A^{(\ell)} ( \theta) = {\rm Op}^{BW} ( a_0 (\theta) + a_1 (\theta) + \ldots + a_{\ell} (\theta) )
\ee
of the equation 
\be\label{eq:tolgo-ell}
\begin{cases}
\partial_\theta A^{(\ell)} ( \theta) =  i \big[{\rm Op}^{BW} ( B(U;  \theta, \cdot ),  A^{(\ell)} ( \theta)  \big]  
+ R (\theta)  \\
A^{(\ell)} ( 0 ) =  {\rm Op}^{BW}  (a) 
\end{cases}
\ee
where $  R (\theta) $ is in $ \Sigma {\cal R}^{-\rho + m  + 1}_{K,K',q  +1 } [r, N] $, uniformly in $ \theta \in [-1,1]$.

Finally, we estimate the difference between the approximate solution  $ A^{(\ell)} ( \theta) $
and the true conjugated operator  $ A ( \theta) $ in  \eqref{true-conjugation}. 
We write 
\be \label{eq:differenza}
 A^{(\ell)} ( \theta) - A ( \theta)  =  
\big( A^{(\ell)} ( \theta)  \Omega_B (\theta) - \Omega_B (\theta) {\rm Op}^{BW}  (a)  \big)  (\Omega_B ( \theta ))^{-1} \, . 
\ee
The operator 
$$ 
V(\theta) \stackrel{\textrm{def}}{=} A^{(\ell)} ( \theta)  \Omega_B (\theta) - \Omega_B (\theta) {\rm Op}^{BW}  (a)  
$$ 
solves  the non-homogeneous equation 
$$
\begin{cases}
\partial_\theta V (\theta ) = i {\rm Op}^{BW}( B )  V(\theta)  + R (\theta) \Omega_B (\theta) \\ 
V ( 0) = 0 \, . 
\end{cases}
$$
Thus, by Duhamel principle and \eqref{flow-homogeneous} we get 
$$ 
V (\theta ) :=  \Omega_B ( \theta)  \int_0^\theta  (\Omega_B ( \tau ))^{-1} R (\tau) \Omega_B (\tau)\, d \tau 
$$ 
and, substituting in \eqref{differenza} we get 
\be\label{eq:differenza-AA}
A^{(\ell)} ( \theta) - A ( \theta)  =  
\Omega_B ( \theta)  
\Big( \int_0^\theta  (\Omega_B ( \tau ))^{-1} R (\tau) \Omega_B (\tau)\, d \tau \Big) (\Omega_B ( \theta ))^{-1} \, . 
\ee
We finally deduce the main result for the 
paracomposition operator $ \Omega_B (1) $ associated to the diffeomorphism $ \Phi_U $ in \eqref{2424}.

\begin{theorem} {\bf (Paracomposition)} \label{thm:para}
Let $ q $ be in $ \N $, $ K' \leq K $, 
$ N \in \N $ with $ q \leq N -1 $, $ r >  0 $.  
Let  $ \beta (U; t, \cdot ) $ be a function  in $ \Sigma {\cal F}_{K,K',1}[r,N] $
which is  real valued for $ U $ in the space $\CKHR{\sigma}{\C^2}$ and  consider the diffeomorphism
$ \Phi_U : x \mapsto x + \beta (U; t, x ) $.  Let 
$ a (U; \cdot ) $ be a symbol in $ \Sigma \Gamma^m_{K,K',q}[r,N] $.  If $ \sigma $ is large enough, 
and $ U $ stays in the ball of center zero and 
radius $ r $ small enough in $ C_{* \R}^K (I, {\dot H}^\sigma (\Tu, \C^2 ) ) $,  
then the flow $ \Omega_{B(U)} (\theta) $ defined in \eqref{flow-homogeneous} is well defined 
for $ | \theta | \leq 1 $,
and, for any $ \rho $ large enough, there is a symbol  $ a_\Phi \in \Sigma \Gamma^m_{K,K',q}[r,N] $ such that 
the conjugated operator
\be\label{eq:Alinach}
\Omega_{B(U)} (1) {\rm Op}^{BW} (a (U; \cdot )) \big(\Omega_{B(U)} (1)\big)^{-1}  
= {\rm Op}^{BW} ( a_\Phi (U; \cdot)  ) + R(U;t)
\ee
where $ R(U;t) $ is  in $  \Sigma {\cal R}^{-\rho + m}_{K,K',q  +1 } [r, N] $. 
Moreover we have an expansion 
\be\label{eq:a-phi-svi}
a_\Phi (U; t,x,\xi) = a^0_{\Phi} (U; t,x,\xi) + a^1_{\Phi} (U; t,x,\xi) 
\ee
where the principal symbol $ a^0_\Phi $  
is given by, denoting   the inverse diffeomorphism of $ \Phi_U$ as in \eqref{2425}, 
\be\label{eq:a-0-forma}
a^0_\Phi (U; t, x, \xi ) = a \Big( U;   \Phi_U(t,x),  
\xi \partial_y \big(  \Phi_U^{-1}(t,y) \big)_{|y = \Phi_U(t,x) }   \Big)  \, , 
\ee
 $ a^0_\Phi $  is in $ \Sigma \Gamma^m_{K,K',q}[r,N]  $
and the  symbol  
$ a^1_\Phi $ is in $  \Sigma \Gamma^{m-2}_{K,K',q+1}[r,N]  $. Finally $ a_\Phi \equiv 1 $ if 
$ a \equiv 1 $.

The operator \index{f@$\Phi^\star$ (Paracomposition operator)} 
\begin{equation}\label{eq:def-paracomp}
\Phi^\star_U  \stackrel{\textrm{def}}{=} \Omega_{B(U)} (1)
\end{equation}
is by definition the paracomposition operator by the
  diffeomorphism $\Phi_U$. There are multilinear operators $M_p$ in $\Mt{}{p}$ for $p=1,\dots,N-1$ and an operator $M_N$ in $\Mr{}{K,K',N}$ 
  such that
\begin{equation}
  \label{eq:243}
  \Phi_U^\star W = W +\sum_{p=1}^{N-1} M_p(U,\dots,U)W + M_N(U;t)W \, .
\end{equation}
\end{theorem}

\begin{proof}
Recalling \eqref{def:Aell} we set 
$$ 
a^1_\Phi  \stackrel{\textrm{def}}{=} a_1 (1) + \ldots + a_\ell (1) \in \Sigma \Gamma^{m-2}_{K, K', q+1}[r,N] \, .
$$ 
Moreover 
 we know by \eqref{2424} and
by \eqref{sol-a-0-t1} that $a_\Phi^0$ is given by \eqref{a-0-forma}. 
Then we estimate  the difference  \eqref{differenza-AA} taken at $\theta=1$. 
The remainder $ R (\tau) $ is in the class $ \Sigma {\cal R}^{-\rho + m  + 1}_{K,K',q  +1 } [r, N] $. In order to
prove that the operator  in \eqref{differenza-AA} is a smoothing remainder,
decomposed in homogeneous operators plus  an $ O( U^N) $ term,  we have to
Taylor expand the flow 
\begin{multline} \label{eq:int-rem}
   \Omega_{B(U)} (\theta) = \textrm{Id} +
   \sum_{\ell=1}^{N-q-1}\int\1_{\theta_\ell<\cdots<\theta_1<\theta}\prod_{j=1}^\ell\bigl[i\opbw(B(U;\theta_j,\cdot))\bigr]\,d\theta_1\dots
   d\theta_\ell\\
+ \int\1_{\theta_{N-q}<\cdots<\theta_1<\theta}
\prod_{j=1}^{N-q}\bigl[i\opbw(B(U;\theta_j,\cdot))\bigr]\Omega_{B(U)}(\theta_{N-q})\,d\theta_1\dots
   d\theta_{N-q} \, .
\end{multline}
Consider the integral term $ I(U) $ in \eqref{int-rem}. We may write
\begin{multline}
\pa_t^k \Big( \prod_{j=1}^{N-q}\bigl[\opbw(B(U;\theta_j,\cdot))\bigr]\Omega_{B(U)}(\theta_{N-q}) W \Big) = \\
\sum_{k_1 + \ldots + k_{N- q+1} = k} C_{k_1, \ldots, k_{N-q+1}}
\Big( \prod_{j=1}^{N-q}\bigl[\opbw( \pa_t^{k_j} B) \bigr] 
\pa_t^{k_{N- q+1}} [\Omega_{B(U)}   W] \Big) 
\end{multline}
for suitable coefficients $ C_{k_1, \ldots, k_{N-q+1}} $. 
Since $ B(U; \theta) $ is
 in $ \sG{1}{K, K',1}{N}  $, it 
 follows from Proposition \ref{215} 
 and \eqref{425-lemma} that, for any $k\leq K- K'$,
\[
\norm{\partial_t^k (I(U)W)}_{\Hds{s-\frac{3}{2}k- (N-q)}} \leq
C\sum_{k'+k''=k}\Gcals{\sigma}{k'+K',N-q}{U}\Gcals{s}{k'',1}{W} \, .
\]
We replace in  \eqref{differenza-AA} $\Omega_{B(U)}(\theta)$, $\Omega_{B(U)}(\tau)$ by the right hand side of
\eqref{int-rem}, eventually taken at some order smaller than $N-q$, and
do the same for the analogous expansions of $(\Omega_{B(U)}(\theta))^{-1}$, $(\Omega_{B(U)}(\tau))^{-1}$. 
The terms containing at least one integral remainder provide 
elements of $ {\cal R}^{ -\rho + m + 1 + N}_{K, K', N}[r] $. 
By the results of section \ref{sec:23} the polynomial
contributions to the expansion may be expressed from compositions of the smoothing term in \eqref{differenza-AA} with
paradifferential operators. 
We deduce that the difference 
 $ A^{(\ell)} ( \theta) - A ( \theta) $ in \eqref{differenza-AA} is in 
 $  \Sigma {\cal R}^{-\rho + m  + 1+ N}_{K,K',q  +1 } [r, N] $, and 
 renaming $ \rho - N - 1 $ by 
 $ \rho  $, we  obtain that it is in $  \Sigma {\cal R}^{-\rho + m }_{K,K',q  +1 } [r, N] $.

Finally, recalling \eqref{def-paracomp}, the formula 
\eqref{243} follows from \eqref{int-rem} at $\theta = 1$ acting on $W$. Actually, the right hand side of this last
quantity may be written as \eqref{243} by Proposition~\ref{231}, the remarks following Definition~\ref{216} and
lemma~\ref{lem:flow}.
\end{proof}

\textbf{Remark}: 
The above proof shows that, if $ R ( U; t ) $ 
is a smoothing operator in the class $\sR{-\rho}{K,K',1}{N}$, then  $ \Omega_{B(U)}(1)\circ R(U; t ) \circ (\Omega_{B(U)} (1) )^{-1}$  belongs to  $\sR{-\rho+N}{K,K',1}{N}$.

\medskip

Theorem~\ref{thm:para} shows that the time $ 1 $ flow $ \Omega_{B(U)} (1) $ generated by \eqref{flow-homogeneous}
defines a para-composition operator associated to the composition operator \eqref{omotopy},  
providing 
the usual formula \eqref{Alinach}, \eqref{a-0-forma} in Alinhac  theorem. 
An advantage of the Weyl quantization is that  the  symbol  $ a^1_\Phi $  in \eqref{a-phi-svi} has two orders than 
less the principal symbol. 

\medskip

The  conjugation operator 
$ \Omega_{B(U)} (\theta ) \circ \pa_t  \circ \Omega_{B(U)} (\theta )^{- 1}  $
can be analyzed in a similar way.  

\begin{proposition}\label{cong-Dt}
Use the assumptions and notations of  Theorem \ref{thm:para}. 
Let $U$ be a solution of equation
\eqref{2129} of lemma~\ref{217}. Then
 for any $ \rho $ large enough, we have 
\begin{align}
\Omega_{B(U)} (\theta ) \circ \pa_t  \circ \Omega_{B(U)} (\theta )^{- 1}  & 
= \pa_t + 
 \Omega_{B(U)} (\theta ) \circ \big(\pa_t \Omega_{B(U)} (\theta )^{- 1}\big)  \label{eq:def:Psi-theta} \\ 
& =  \pa_t + {\rm Op}^{BW} (e(U; \cdot)) + R(U;t) \nonumber 
\end{align}
where $ e $ is a symbol in $ \Sigma \Gamma^1_{K, K'+1,1}[r,N]  $ with 
$ {\rm Re} (e) $ in $  \Sigma \Gamma^{-1}_{K, K'+1,1}[r,N]  $ and $ R(U;t) $ is a smoothing operator
in $ \Sigma {\cal R}^{-\rho}_{K, K'+1,1}[r,N] $.
\end{proposition}

\begin{proof}
Recall that the symbol $ B := B(U; \theta ) $ depends on  the time $ t $ through the function $ U $ and
that the flow $  \Omega_B (\theta ) $ solves \eqref{flow-homogeneous}.  
Setting  $ \Psi (\theta) =  \Omega_B (\theta ) \circ \pa_t  \circ \Omega_B (\theta )^{- 1}  $  we have
\begin{align*}
\pa_\theta \Psi (\theta) & = \pa_\theta \big( \Omega_B (\theta ) \circ \pa_t  \circ \Omega_B (\theta )^{- 1}  \big) \\
& =   \big( \pa_\theta \Omega_B (\theta )\big)  \circ \pa_t  \circ \Omega_B (\theta )^{- 1} + 
\Omega_B (\theta )  \circ \pa_t  \circ  \big(   \pa_\theta \Omega_B (\theta )^{- 1} \big) \\
& = i {\rm Op}^{BW} (B(\theta) ) \Omega_B (\theta )   \circ \pa_t  \circ \Omega_B (\theta )^{- 1} \\ 
& \quad - 
\Omega_B (\theta )  \circ \pa_t  \circ  
\Omega_B (\theta )^{-1} \big(   \pa_\theta \Omega_B (\theta ) \big) \Omega_B (\theta )^{-1} \\
& = i {\rm Op}^{BW} (B(\theta) ) \Psi (\theta)  - \Psi (\theta)  i {\rm Op}^{BW} (B(\theta) ) 
\end{align*}
and therefore $\Psi (\theta) $ solves the Heisenberg equation 
$$
\begin{cases}
\frac{d}{d \theta} \Psi (\theta)   =
 i \big[ {\rm Op}^{BW}  \big( B(U; \theta ) \big) , \Psi (\theta) \big] \\
\Psi  ( 0 ) = \pa_t   \, . 
 \end{cases}
$$
Comparing with \eqref{def:Psi-theta} we write
$$
\Psi (\theta) = \pa_t + Q(\theta)  \, , \quad 
Q(\theta) =  \Omega_B (\theta ) \circ \big(\pa_t \Omega_B (\theta )^{- 1}\big) \, ,
$$
and we find out that  $ Q (\theta) $ solves 
\be\label{eq:for-Q}
\begin{cases}
\frac{d}{d \theta} Q (\theta)   =
 i \big[ {\rm Op}^{BW}  \big( B(U; \theta ) \big) , Q (\theta) \big] - i  {\rm Op}^{BW}  \big( \pa_t B(U; \theta ) \big) \\
Q  ( 0 ) = 0   \, . 
 \end{cases}
\ee
This forced equation can be analyzed as above and we find that its solution is a paradifferential operator 
$$
Q(\theta) = {\rm Op}^{BW} ( q_0 + q_1 + \ldots ) \, , 
$$ 
up to a regularizing operator, where the principal symbol $ q_0 (\theta, t, \cdot ) $ solves 
\be\label{eq:def:q0}
\begin{cases}
\frac{d}{d \theta} q_0 (\theta)   =  \{   B(U; \theta ) , q_0 (\theta) 	\} - i  \pa_t  B (U; \theta )  \\
q_0 (0 ) = 0   \, . 
 \end{cases}
\ee
By  lemma  \ref{lem:non-homo} the solution of this  non-homogeneous transport equation is 
\be\label{eq:intq0}
 q_0 (\theta, x, \xi ) =  - i \int_0^\theta \pa_t  B ( U; \theta,  \phi^{\theta, s}  (x,\xi) ) \, ds \, .
\ee
Recalling the definition of $ B $ in \eqref{def:B}, and lemma \ref{246}, 
we get  that the function   $ B ( U; \theta,  \phi^{\theta, s}  (x,\xi) ) $ is in $ \Sigma \Gamma^1_{K,K',1}[r,N] $, 
with estimates uniform in $ | \theta |, | s | \leq 1 $. By lemma~\ref{217},  we have also  that 
$ \pa_t B ( U; \theta,  \phi^{\theta, s}  (x,\xi) ) $ is in $ \Sigma \Gamma^1_{K,K'+1,1}[r,N] $ 
and by \eqref{intq0}
$$ 
q_0 (\theta, \cdot )  \in \Sigma \Gamma^1_{K,K'+1,1}[r,N] \, ,  
$$ 
with estimates uniform in $ | \theta | \leq 1 $. Let us quantify how approximately
$ Q^{(0)} (\theta) = {\rm Op}^{BW} ( q_0 (\theta) ) $ solves \eqref{for-Q}.
By \eqref{def:q0} and \eqref{commutator-formula} we have 
\begin{align}
\pa_\theta  Q^{(0)} (\theta) & =  
 {\rm Op}^{BW} ( \{ B(U; \cdot ),  q_0 (\theta) \} ) - i   {\rm Op}^{BW} ( B_t (U; \theta )) \nonumber \\
 & =
 \big[{\rm Op}^{BW} ( i B(U;  \theta, \cdot ),  {\rm Op}^{BW} ( q_0 (\theta) )\big] - 
 {\rm Op}^{BW} \big( r_{-3, \rho}(B,q_0 (\theta) ) \big) \nonumber  \\
 & \quad - R(\theta)  - i   {\rm Op}^{BW} ( B_t (U; \theta )) \nonumber \\
& =  \big[{\rm Op}^{BW} ( i B(U;  \theta, \cdot ),    Q^{(0)} (\theta) \big]  - i   {\rm Op}^{BW} ( \pa_t B (U; \theta ))
 \nonumber \\
 & \quad  -  {\rm Op}^{BW} \big( r_{-3, \rho}(B, q_0(\theta) ) \big)  - R (\theta)  \nonumber 
 \end{align}
where, by \eqref{stima-simbols}, \eqref{remainder1} (with $ m = 1 $ and $ q = 1 $),  for any $ \theta \in [-1,1] $, 
$$
r_{-3, \rho}(B,q_0 (\theta) ) \in  \Sigma \Gamma^{-1}_{K,K'+1,2}[r,N] \, , \quad 
R(\theta) \in \Sigma {\cal R}^{-\rho + 2}_{K,K'+1,2 } [r, N] \, .
$$
Notice that, thanks to the Weyl quantization, the 
symbol $ r_{-3, \rho}(B,q_0 (\theta) ) $ has order $ -1 $.  
Repeating the same argument as for Theorem \ref{thm:para}
and renaming $ \rho - N $  by $ \rho $, we deduce  the Proposition.
\end{proof}
     

%% file: chap3.tex
\chapter[Complex equations and diagonalization]{Complex formulation  of the equation and diagonalization of 
the matrix symbol}\label{cha:3}

In chapter~\ref{cha:6} below, we shall obtain an equivalent form of the water waves equations \eqref{113} 
in terms of a complex unknown, suitable to prove energy estimates. In this chapter,
after introducing some algebraic properties of the water waves equations,  
we 
present the general  form of such paradifferential  system, see \eqref{3112}-\eqref{3115},  and we state the main 
Theorem \ref{311} concerning almost global existence of its solutions with a small initial datum.
The proof of this result is then provided in the last section of Chapter~\ref{cha:3} and in Chapter~\ref{cha:4}. 

\section{Reality, parity and reversibility properties}

We denote by \index{S@$S$ (Reversibility operator in complex form)} $S$ the linear 
 involution, i.e. $ S^2 = {\rm Id} $, 
\be\label{eq:def-S}
S : \C^2 \to \C^2 \, , \quad {\rm with \ matrix \ }  - \bigl[\begin{smallmatrix}0&1\\1&0\end{smallmatrix}\bigr] \, .
\ee
This  is  the translation in the present complex framework of the
map introduced in~\eqref{121a}, and denoted by the same letter. 
We shall consider the action of the group $\Z/2\Z$, identified to $\{\mathrm{Id},S\}$,   
on the balls of functions $B^k_\sigma(I,r)$ defined in \eqref{216},
\begin{equation}
  \label{eq:action}\begin{split}
  (g&,U) \longrightarrow U_g\\
\Z/2\Z&\times B^k_\sigma(I,r)  \to B^k_\sigma(I,r)
\end{split}\end{equation}
defined by
\begin{equation}
  \label{eq:310}
   \index{U@$U_S$ (Action of reversibility operator $S$ on $U$)}U_S(t) \stackrel{\textrm{def}}{=} SU(-t) \, , 
   \quad \forall t  \in I \, . 
\end{equation}
Notice that
the spaces $\CKHR{\sigma}{\C^2}$ defined after \eqref{212} are left invariant under this action. In our
application to the proof of our main theorem, the function $(t,x)\to U(t,x)$ will be a solution to the system \eqref{3112}
below that translates system \eqref{113} on the unknown $U$. 
Since any solution of  \eqref{113}  satisfies \eqref{125}, the
solutions  of \eqref{3112} will satisfy the corresponding property, namely the solution of that system
with initial datum $SU(0)$ coincides with
$SU(-t)$ at any time. In other words, if $ U(t)$ is a solution of  
\eqref{3112}, the function 
$U_S (t) $ defined by \eqref{310} is nothing but  the solution of the same system \eqref{3112}
with initial datum $SU(0)$. 

We  first define some algebraic properties
for a matrix of symbols that will play a crucial role in the long time existence proof. 

\begin{definition}\label{def:algebra}
Let $ A $ be an element of the class $\sGM{m}{K,K',p}{N}$ 
i.e.\  a $2\times2$ matrix
whose entries are symbols in $\sG{m}{K,K',p}{N}$. 
Let  $U$ satisfy $SU = -\bar{U}$. 
We set \index{A@$A(U;t,x,\xi)^\vee$} 
$$
A(U;t,x,\xi)^\vee = A(U;t,x,-\xi) \, . 
$$
The matrix $ A $ satisfies 
  the \index{Cc@Reality conditions} 	{\bf reality} condition if 
  \begin{equation}
    \label{eq:311}
    \bar{A}(U;t,x,\xi)^\vee = -SA(U;t,x,\xi)S \, , 
  \end{equation}
 the {\bf anti-reality} condition if 
\begin{equation}
    \label{eq:312}
    \bar{A}(U;t,x,\xi)^\vee = SA(U;t,x,\xi)S \, ,
  \end{equation}
the \index{Cd@Parity preserving condition} {\bf parity preserving} condition if 
\begin{equation}
  \label{eq:313}
  A(U;t,-x,-\xi) = A(U;t,x,\xi) \, . 
\end{equation}
Finally, 
we say that $A$ satisfies the \index{Ce@Reversibility conditions} {\bf reversibility} condition 
if
\begin{equation}
  \label{eq:314}
  A(U;-t,x,\xi)S = -SA(U_S;t,x,\xi)
\end{equation}
and the  {\bf anti-reversibility} condition if 
\begin{equation}
  \label{eq:315}
  A(U;-t,x,\xi)S = SA(U_S;t,x,\xi).
\end{equation}
\end{definition}
Below, we shall use these properties for elements $U$ of  $\CKHR{\sigma}{\C^2}$ for 
some large enough $\sigma$ and $K$. 
\medskip

\noindent
\textbf{Remarks}:
$\bullet$  The product $ A (U; t, x, \xi) B (U; t, x, \xi) $ of matrices of symbols $ A, B $ satisfying \eqref{312}, resp. \eqref{313}, resp. \eqref{315}, 
satisfy \eqref{312}, resp. \eqref{313}, resp. \eqref{315}, as well. 
In addition, if $ A $ satisfies \eqref{311}, 
resp. \eqref{314},  and $ B $ satisfies \eqref{312},  resp. \eqref{315}, then $ A B $ satisfies \eqref{311}, resp. \eqref{314}. 

$\bullet$ If a matrix  $ A (U; t, x, \xi) $ satisfy one among the properties \eqref{311}-\eqref{315} and it is invertible, then $ A (U; t, x, \xi)^{-1} $ satisfies the same property. 

\medskip

The reversibility and antireversibility conditions \eqref{314}, \eqref{315} 
may be expressed more explicitly for the homogeneous components of the symbols,
for which the time dependence  enters through $ U(t) $, see Definition \ref{211}.

\begin{definition}
Let $ A_q $ be a matrix of homogeneous symbols in $\GtM{m}{q}$.
Let $ U_j $, $ j = 1, \ldots, q $ satisfy $ SU_j = - \bar{U}_j $.  
We say that $ A_q $ satisfies 
the reversibility condition if 
\begin{equation}
  \label{eq:314p}
  A_q(SU_1,\dots,SU_q;x,\xi)S = -SA_q(U_1,\dots,U_q;x,\xi) \, , 
\end{equation}
and the anti-reversibility condition if 
\begin{equation}
  \label{eq:315p}
  A_q(SU_1,\dots,SU_q;x,\xi)S = SA_q(U_1,\dots,U_q;x,\xi) \, .
\end{equation}
\end{definition}

We have the following lemma. 

\begin{lemma}
  \label{hom-nonhom}
Let $A$ be in $\sGM{m}{K,K',p}{N}$. Decompose  $A$ according to \eqref{219} as
\be\label{eq:A-dev}
A(U;t,x,\xi) = \sum_{q=p}^{N-1}A_q(U,\dots,U; x, \xi) + A_N(U;t,x,\xi)
\ee
with $A_q$ in $\GtM{m}{q}$ for $q=p,\dots,N-1$ and $A_N$ in $\GrM{m}{K,K',N}$.

If $A_q$, $q=p,\dots,N-1$ satisfies \eqref{314p} (resp.\ \eqref{315p}) and if $A_N$ satisfies \eqref{314} (resp.\
\eqref{315}), then $A$ satisfies \eqref{314} (resp.\ \eqref{315}).

Conversely, if $A$ satisfies \eqref{314} (resp.\ \eqref{315}), we may find symbols $A'_q$, $q=p,\dots,N-1$, satisfying
\eqref{314p} (resp.\ \eqref{315p}) such that, for any $U = U(t) $ we have 
\begin{equation}
  \label{eq:hom-nonhom1}
A(U;t,x,\xi) = \sum_{q=p}^{N-1}A'_q(U,\dots,U; x,\xi) + A_N(U;t,x,\xi).
\end{equation}
\end{lemma}
\begin{proof}
  Let us write the proof for conditions \eqref{314} and \eqref{314p}. If $A_q$ satisfies \eqref{314p}, then according to
  \eqref{hom-nonhom} and to the definition \eqref{310} of $U_S$, we have 
  \[\begin{split}
    A_q(U,\dots,U;-t,x,\xi)S &= A_q(U(-t),\dots,U(-t);x,\xi)S\\
&= A_q(SU_S(t),\dots,SU_S(t);x,\xi)S\\
&=-SA_q(U_S(t),\dots,U_S(t);x,\xi)\\
&= -SA_q(U_S,\dots,U_S;t,x,\xi)
  \end{split}\]
  so that $ A_q(U,\dots,U; t,x,\xi) $ satisfies \eqref{314}. 
Since also  $ A_N $ satisfies \eqref{314}  then $A$ defined in \eqref{A-dev} satisfies \eqref{314}.

Conversely, assume that the matrix of symbols $ A $ in \eqref{A-dev} satisfies \eqref{314}. 
Thus each $ A_q (U, \ldots, U; t, x, \xi) $ for $ q = p, \ldots, N -1 $, satisfies \eqref{314} as well. 
Define from $A_q$ a new multilinear symbol 
\[A'_q(U_1,\dots,U_q;x,\xi) = \frac{1}{2}\bigl[A_q(U_1,\dots,U_q;x,\xi) - S A_q(SU_1,\dots,SU_q;x,\xi)S\bigr].\]
Then, by construction, $A'_q$ satisfies \eqref{314p}. In particular, if we 
replace $ U_1, \ldots, U_q $ by $U$  we get by \eqref{310}
\begin{multline}
A'_q(U,\dots,U;t,x,\xi) = \frac{1}{2}\bigl[A_q(U,\dots,U;t,x,\xi) - S A_q(U_S,\dots,U_S;-t,x,\xi)S\bigr] \\
= A_q(U,\dots,U;t,x,\xi)
\end{multline}
by the condition \eqref{314}. Hence we deduce \eqref{hom-nonhom1}. 
\end{proof}
\textbf{Remark:} From now on, each time we shall consider symbols $ A $ of the form 
\eqref{A-dev} satisfying the reversibility  \eqref{314} or antireversibility \eqref{315}
condition, we shall assume that the homogeneous contribution in the decomposition of $A$ satisfy \eqref{314p} or
\eqref{315p}.
\smallskip

Let us introduce analogous conditions at the level of operators. 

\begin{definition}\label{Def:RPR}
Let $ M(U; t ) $ be a linear operator, depending on $  U $ satisfying $SU = -\bar{U}$. We say that 
$ M $ satisfies the  {\bf reality}  condition if 
\begin{equation}
  \label{eq:316}
  \overline{M(U; t)V} = -SM(U; t )S\overline{V} \, , 
\end{equation}
the {\bf anti-reality} condition  if 
\begin{equation}
  \label{eq:317}
  \overline{M(U; t )V} = SM(U; t)S\overline{V} \, .
\end{equation}
Define the map $\tau$, acting on functions of $x$,  by 
\be\label{eq:def-tau} 
(\tau V)(x) \stackrel{\textrm{def}}{=}   V(-x) \, . 
\ee 
We say that $ M $ is {\bf parity preserving} 
if 
\begin{equation}
  \label{eq:318}
  M(U; t)\circ\tau = \tau\circ M(U; t ).
\end{equation}
Finally, we say that $ M $ satisfies the {\bf reversibility}  condition  if
\begin{equation}
  \label{eq:319}
  M(U; - t )S = -SM(U_S; t)
\end{equation}
and the {\bf anti-reversibility} condition if 
\begin{equation}
  \label{eq:3110}
  M(U; -t )S = SM(U_S; t ) \, .
\end{equation} 
\end{definition}

\noindent
\textbf{Remarks}:
$\bullet$
The operator $M(U; t) = \opbw(A(U; t, \cdot))$ satisfies  \eqref{316} (resp.\ \eqref{317}, resp.\ \eqref{318}, resp.\ \eqref{319}, resp.\ 
\eqref{3110}) if the matrix symbol $ A(U; t, \cdot) $ satisfies the 
corresponding property \eqref{311} (resp.\ \eqref{312}, resp.\ \eqref{313}, resp.\ \eqref{314}, resp.\ 
\eqref{315}).   Notice that since the functions $\chi, \chi_p$ used in \eqref{2112} are even with respect to each of their
argument, as soon as \eqref{311}, \eqref{312}, \eqref{313} hold, the same is true for the Fourier truncated symbols
\eqref{2112}.

$\bullet$
An operator $ M (U; t ) $ satisfying the anti-reality condition  \eqref{317} 
maps the subspace of $2$-vectors  \{$SV = -\overline{V}$\} (i.e.\ $V$ of the form
$\bigl[\begin{smallmatrix}v\\\bar{v}\end{smallmatrix}\bigr]$) into itself. 

$\bullet$ The operator $ M(U; t ) $ satisfies the reality property  \eqref{316} if and only if
$ i M(U; t ) $ satisfies the anti-reality property  \eqref{317}, and viceversa.

$ \bullet $
If $ M(U; t ) $ satisfies \eqref{319},  respectively \eqref{3110}, we  say that $F(U; t) = M(U; t)U$ is reversible, respectively antireversible.

$\bullet$ If   $ M(U; t) $ satisfies one among the properties \eqref{316}-\eqref{3110} and it is invertible, then $ M (U; t)^{-1} $ satisfies the same property. 

\smallskip

Arguing as in the proof of lemma~\ref{hom-nonhom} we deduce the following lemma. 

\begin{lemma}\label{non-hombis}
If $M$ is decomposed as in \eqref{2118} or \eqref{2128} as a sum
\[M(U;t) = \sum_{q=p}^{N-1}M_q(U,\dots,U) + M_N(U;t)\]
in terms of homogeneous operators $M_q$, $q=p,\dots,N-1$, and if $M$ satisfies the reversibility  condition \eqref{319}, 
respectively anti-reversibility \eqref{3110},  we may assume that  
$M_q$, $ q = p, \ldots, N - 1 $,  satisfy the reversibility property 
\begin{equation}
  \label{eq:319p}
M_q(SU_1,\dots,SU_q)S = -SM(U_1,\dots,U_q) \, ,
\end{equation}
respectively the anti-reversibility property
\begin{equation}
  \label{eq:3110p}
M_q(SU_1,\dots,SU_q)S = SM(U_1,\dots,U_q).
\end{equation}
\end{lemma}

\smallskip

These algebraic properties behave    as follows under composition:  
\begin{lemma}\label{lem:Compos}
 Composition of an operator satisfying the reality property \eqref{316} (resp.\  the reversibility property \eqref{319}) with one or several operators satisfying the antireality property \eqref{317} (resp. the anti-reversibility property 
\eqref{3110}) still satisfies the reality property \eqref{316} (resp. reversibility property \eqref{319}).
Composition of operators satisfying the parity preserving property \eqref{318} satisfy as well the
parity preserving property \eqref{318}.  Composition of operators satisfying the
anti-reality property \eqref{317} satisfy the anti-reality property \eqref{317}  as well. 
\end{lemma}

\section[Complex formulation]{Complex formulation of the capillarity-gravity water waves equations}\label{sec:31}

We fix from now on large integers $\rho, N$. The time (resp.\ space) smoothness of the solution we shall consider will be
measured by an integer $K$ (resp. a real number $s$), and we shall assume
\begin{equation}
  \label{eq:3111}
  s\gg \sigma\gg K\gg \rho\gg N
\end{equation}
where $N$ is the exponent of $1/\epsilon$ in the lower bound of the time of existence of the solutions ($T_\epsilon\geq
c\epsilon^{-N}$) that we want to prove (see the statement of Theorem~\ref{121}).

We shall prove in Proposition~\ref{631} that the water waves system \eqref{113} is equivalent, for small solutions,
even in $ x $,
to the following system
\begin{equation}
  \label{eq:3112}
  \begin{split}
    (D_t - \opbw(A(U;t,x,\xi)))U &= R(U; t)U\\
U\vert_{t=0} &= \epsilon U_0
  \end{split}
\end{equation}
where $A(U; t, \cdot)$ is a two by two matrix of symbols, that may be decomposed in the basis of $\Mcal_2(\C)$ given by
\begin{equation}
  \label{eq:3113}
  \Ical_2 = \bigl[\begin{smallmatrix}1&0\\0&1\end{smallmatrix}\bigr], \quad \Kcal =
  \bigl[\begin{smallmatrix}1&0\\0&-1\end{smallmatrix}\bigr], \quad \Jcal =
  \bigl[\begin{smallmatrix}0&-1\\1&0\end{smallmatrix}\bigr], \quad 
  \Lcal = \bigl[\begin{smallmatrix}0&1\\1&0\end{smallmatrix}\bigr]
\end{equation}
as follows:
\begin{multline}
  \label{eq:3114}
A(U;t,x,\xi) = \bigl(\mk(\xi)(1+\zeta(U;t,x)) + \ld(U;t,x,\xi)\bigr)\Kcal \\
+ \bigl(\mk(\xi)\zeta(U;t,x) + \ldm(U;t,x,\xi)\bigr)\Jcal\\
+ \lambda_1(U;t,x,\xi)\Ical_2 + \lambda_0(U;t,x,\xi)\Lcal
\end{multline}
where 
\begin{equation}
  \label{eq:3115}
  \mk(\xi) = (\xi\tanh\xi)^{\frac{1}{2}}(1+\kappa\xi^2)^{\frac{1}{2}}(1-\chi(\xi)),
\end{equation}
$\chi(\xi)$ being an even $C^\infty_0(\R)$ function equal to one close to 
zero, supported for $\abs{\xi}<\frac{1}{2}$;

 where $\lambda_j$, $j= 1, \frac{1}{2}, 0, -\frac{1}{2}$ are symbols of $\sG{j}{K,1,1}{N}$ for some small enough $r>0$,
satisfying $\Im\lambda_j \in \sG{j-1}{K,1,1}{N}$ if $j = 1$ or $\frac{1}{2}$, and such that $A(U;t,x,\xi)$ satisfies 
the reality, parity preserving and reversibility properties 
\eqref{311}, \eqref{313} and \eqref{314}. 
Moreover $R(U; t)$ in the right hand side of \eqref{3112} is a smoothing 
operator of $\sRM{-\rho}{K,1,1}{N}$ that satisfies the reality, parity preserving 
and reversibility properties \eqref{316}, \eqref{318}, \eqref{319}. Finally  
the function $\zeta$ defined in \eqref{635}  
is an element of $\sFR{K,0,1}{N}$
(actually even of the subspace $\sFR{K,0,2}{N}$).

\smallskip

The above paradifferential formulation of the water waves system is suitable to prove energy estimates. 
In some  instances we 
shall  need to write  the water waves system \eqref{3112} 
just as
\begin{equation}
  \label{eq:4312}
  D_t U =  \mk(D)\Kcal U + M(U; t)U
\end{equation}
where the operator $ M(U; t ) $ is the element of $\sMM{}{K,1,1}{N}$
defined by the sum  of $\opbw\bigl(A(U; t, \cdot) -\mk(\xi)\Kcal\bigr)$ and 
the smoothing operator $ R(U; t) $ in \eqref{3112}.  
Thus $ M(U; t )U $ collects all the terms of the water waves system \eqref{3112}
which are  at least quadratic in $ U $ (but the operator $ M(U; t) $ loses derivatives). 
The fact that $ M(U; t ) $ is in $\sMM{}{K,1,1}{N}$ follows by the remarks after Definition~\ref{216}.
Moreover, since the matrix of symbols $A$ satisfies \eqref{311}, \eqref{312}, 
\eqref{314} and since $R(U; t )$ in \eqref{3112} satisfies \eqref{316}, \eqref{318}, \eqref{319}, the operator $M(U; t )$
satisfies as well the reality, parity preserving, and  reversibility  properties \eqref{316}, \eqref{318}, \eqref{319}.

\smallskip

We look for solutions of \eqref{3112} with   Cauchy data  $U_0$ in $\Hdse{s}(\Tu,\C^2)$ satisfying 
$$
 SU_0 = -\bar{U}_0 \, , \ \  i.e.\   
U_0 =\vect{u_0}{\bar{u}_0} 	\, . 
$$ 
Notice that $SU_0=
-\bar{U}_0$ implies that for any $t$, 
$ \bar{U}(t) = -SU(t)$  since the reality condition \eqref{316} implies that these two functions satisfy the same equation
with the same Cauchy data. 
This amounts to look for real valued solutions of the water waves system \eqref{113}.

\begin{theorem}
  \label{311} {\bf (Almost global existence for \eqref{3112})}
There is a zero measure subset $\Ncal$ of $]0,+\infty[$ and for any $N$ in $\N$, any $\kappa$ in $]0,+\infty[-\Ncal$, there
are $K$ in $\N$, $s_0$ in $\N$ such that, for any $s\geq s_0$, there are positive constants $\epsilon_0$, $c$, $C$ such that
for any $U_0$ in the unit ball of $\Hdse{s}(\Tu,\C^2)$, with $SU_0 = - { \bar U}_0$, for any $\epsilon$ in
$]0,\epsilon_0[$, system \eqref{3112} has a unique classical solution 
$$U \in 
\cap_{k=0}^K
C^k(]-T_\epsilon,T_\epsilon[,\Hdse{s-\frac{3}{2}k}(\Tu,\C^2)) \, , \quad {\it with}  
\quad T_\epsilon\geq c\epsilon^{-N} \, . 
$$ Moreover, the solution
satisfies
\begin{equation}
  \label{eq:3116}
  \sup_{]-T_\epsilon,T_\epsilon[}\norm{\partial_t^kU(t,\cdot)}_{\Hds{s-\frac{3}{2}k}(\Tu,\C^2)} \leq C\epsilon,\ 0\leq k\leq K.
\end{equation}
\end{theorem}
Theorem  \ref{311} and Proposition \ref{631} then imply  Theorem \ref{121}. 

\smallskip

The first part of the proof of Theorem \ref{311} will be to reduce \eqref{3112} to an equivalent system, where the non diagonal
part of the matrix symbol \eqref{3114}, 
namely $\bigl(\mk\zeta + \ldm\bigr)\Jcal +\lambda_0\Lcal$ will be replaced by a symbol of very
negative order, whose associated operator may be incorporated to the smoothing term $R(U; t )$.

\section{Diagonalization of the system}\label{sec:32}

The goal of this subsection is to transform the matrix symbol $A$ in the left hand side of \eqref{3112} 
into a diagonal one, up to  a smoothing term $ R(U; t) $.
 We
introduce first the class of matrices of symbols we shall use.
\begin{definition}
  \label{321}
Let $m$ be in $\R_-$, $N$ in $\N$, $K'\leq K$ in $\N$, $r>0$. We denote by \index{E@$\sE{m}{K,K',1}{N}$ (Space of matrices of symbols)}
$\sE{m}{K,K',1}{N}$ the set of matrices of symbols of the form
\begin{equation}
  \label{eq:321}
  P(U; \cdot) =   P(U; t, \cdot)  = (1+\alpha)\Ical_2 +\beta\Lcal+\theta\Jcal +\gamma\Kcal
\end{equation}
where
\begin{equation}
  \label{eq:322}
  \alpha, \beta \in \sG{m}{K,K',1}{N}, \quad \gamma, \theta \in \sG{\min\bigl(m-\frac{1}{2},-\frac{3}{2}\bigr)}{K,K',1}{N}
\end{equation}
are such that $P$ satisfies the anti-reality, parity preserving, anti-reversibility properties  
 \eqref{312}, \eqref{313}, \eqref{315}, and 
\begin{equation}
  \label{eq:323}
  \Im\alpha, \Im\beta \in \sG{\min\bigl(m,-\frac{3}{2}\bigr)}{K,K',1}{N}.
\end{equation}
\end{definition}
Our goal is to prove:
\begin{proposition}
  \label{322} 
  {\bf (Diagonalization of the matrix symbol  \eqref{3114})} 
Assume that $K\geq\rho+1$, $ \rho \in \N $, and that $r>0$ is small enough. There exist

$ \bullet $ 
matrices of symbols  $P(U; t, \cdot), Q(U; t, \cdot)$ in $\sE{0}{K,\rho,1}{N} $,  such that 
the operator $\opbw(P)\circ\opbw(Q) - \mathrm{Id}$ is in $\sRM{-\rho}{K,\rho,1}{N}$ and satisfies \eqref{317},
\eqref{318} and \eqref{3110},

$ \bullet $ 
symbols $\lu{1}_j$ in $\sG{j}{K,\rho+1,1}{N}$, $j= 1, \frac{1}{2}$ satisfying
\begin{equation}
  \label{eq:324}
  \Im\lu{1}_j \in \sG{j-1}{K,\rho+1,1}{N} 
\end{equation}
and a function $\zeta^{(1)}(U;t,x)$ in $\sFR{K,\rho+1,1}{N}$,  so that the diagonal 
matrix
\begin{multline}
  \label{eq:325}
  A^{(1)}(U;t,x,\xi) = \bigl(\mk(\xi)(1+\zeta^{(1)}(U;t,x)) \\+ \ldu{1}(U;t,x,\xi)\bigr)\Kcal + \lu{1}_1(U;t,x,\xi)\Ical_2
\end{multline}
satisfies the reality  \eqref{311}, parity preserving \eqref{313}, reversibility \eqref{314} properties, 

$ \bullet $ 
smoothing operators $R'(U;t)$, respectively $R''(U;t) $, belonging to the space $\sRM{-\rho+\frac{3}{2}}{K,\rho+1,1}{N}$, 
resp. belonging to
$\RrM{-\rho+\frac{3}{2}}{K,\rho+1,N}$, satisfying the reality \eqref{316}, 
parity preserving \eqref{318} and reversibility \eqref{319} properties,

such that 
the function $W = \opbw(Q)U$ solves the system
\begin{equation}
  \label{eq:326}
  \bigl(D_t - \opbw(A^{(1)}(U;t, \cdot)\bigr)W = R'(U;t)W + R''(U;t)U
\end{equation}
on the time interval over which a solution $ U $ of system \eqref{3112}
 is defined. 
\end{proposition}
We shall prove three lemmas. 
\begin{lemma}
  \label{323} {\bf (Parametrix)}
Take a matrix of symbols  
$P$ in $\sE{m}{K,K',1}{N}$ and  let $\rho$ be  in $ \N $. 
Then, if $r$
is small enough, there is 
a matrix of symbols $Q (U; \cdot) $ in $\sE{m}{K,K',1}{N}$ such that
\begin{equation}
  \label{eq:327}
\begin{split}
  \bigl(P(U;\cdot)\# Q(U;\cdot)\bigr)_{\rho,N} - \Ical_2 \in \sE{-\rho}{K,K',1}{N}\\
P(U;\cdot)^{-1} - Q(U;\cdot) \in \sGM{\min\bigl(m-\frac{3}{2},-2\bigr)}{K,K',1}{N}.
\end{split}
\end{equation}
\end{lemma}
\begin{proof}
As  
  $\alpha, \beta, \theta, \gamma$ in \eqref{321} go to zero when $ U $ is in $\Br{K}{}$ defined by \eqref{216} and $r$ goes
  to zero, we see that, for $ r $ small enough  (and $\sigma$ fixed large enough) the matrix of symbols (recall \eqref{3113})
  
 $$
 P(U; \cdot ) =   \begin{bmatrix} 1 + \alpha + \gamma & \beta - \theta \\ 
 \beta + \theta & 1 + \alpha - \gamma 
 \end{bmatrix} 
 $$ 
 is invertible, with inverse
\[P(U;\cdot)^{-1} = (1+\alpha')\Ical_2 +\beta'\Lcal +\theta'\Jcal +\gamma'\Kcal\]
where, setting $ \Delta = (1+\alpha)^2-\beta^2-\gamma^2+\theta^2$,
\begin{equation}
  \label{eq:328}
  \alpha' = (1+\alpha)\Delta^{-1} -1, \ \beta' = -\beta\Delta^{-1}, \ \gamma' = -\gamma\Delta^{-1}, \
   \theta' = -\theta\Delta^{-1}.
\end{equation}
 We claim that the matrix $P^{-1}$ belongs to $\sE{m}{K,K',1}{N}$. 
Indeed,  since the order of the symbols $ \alpha $, $ \beta $, $ \gamma $, $ \delta $ is 
$m\leq 0$, a Taylor expansion of $\Delta^{-1}$ when $(\alpha, \beta, \gamma, \theta)$ goes to zero implies
that
\[\alpha', \beta' \in \sG{m}{K,K',1}{N}, \quad  \gamma', \theta' \in \sG{\min\bigl(m-\frac{1}{2},-\frac{3}{2}\bigr)}{K,K',1}{N} \]
which is condition \eqref{322}.
Moreover, since $ \alpha, \beta, \gamma, \delta $ satisfy  \eqref{322} and \eqref{323} 
the symbols $\alpha', \beta'$ satisfy also \eqref{323}. 

Finally, as $S^{-1} = S$, the fact that $P $ satisfies the  anti-reality, parity preserving and  anti-reversibility conditions 
\eqref{312},  \eqref{313},  \eqref{315}, implies that the inverse matrix $P^{-1}$ satisfies the same properties. Consequently,
$P^{-1}$ belongs to $\sE{m}{K,K',1}{N}$. 

Let us compute $(P\#P^{-1})_{\rho,N}$ using the composition formulas \eqref{231},
\eqref{222}, \eqref{223}.  The first 
term  of the asymptotic expansions  \eqref{222}, \eqref{223},  corresponding to $\ell = 0$, 
 is the identity $\Ical_2$. 
The next term, corresponding to $\ell=1$, is 
\begin{multline}
  \label{eq:329}
\frac{1}{2i}\bigl[(\partial_\xi P)(\partial_x P^{-1}) - (\partial_x P)(\partial_\xi P^{-1})\bigr] \\
= \frac{1}{2i}\bigl[-(\partial_\xi P)P^{-1}(\partial_x P) P^{-1} + (\partial_x P) P^{-1} (\partial_\xi P) P^{-1}\bigr].
\end{multline}
Write then the matrix of symbols $ P $ in  \eqref{321} as 
$$ 
P = P_0+\theta\Jcal+\gamma\Kcal \, , \quad P_0 = (1+\alpha) {\cal I}_2 + \beta {\cal L} \, , 
$$ 
so that, by \eqref{322}, $ P - P_0 $ is in   
$\sGM{\min\bigl(m-\frac{1}{2},-\frac{3}{2}\bigr)}{K,K',1}{N}$ as well as 
$ P^{-1} - P_0^{-1} = - P^{-1} ( P - P_0 ) P_0^{-1}$. 
If in \eqref{329}, we replace $P$
by $P_0$, and $P^{-1}$ by $P_0^{-1}$, the error we generate is then in the space of symbols 
$\sGM{\min\bigl(m-\frac{3}{2},-\frac{5}{2}\bigr)}{K,K',1}{N}$. On the other hand, the right hand side of \eqref{329} with $P$
replaced by $P_0$ vanishes, since $P_0$ and $P_0^{-1}$ and their derivatives are linear combinations of $\Ical_2$ and
$\Lcal$, so that commute between themselves. Finally, terms corresponding to $\ell\geq 2$ in \eqref{222}, \eqref{223} applied
to compute $P\# P^{-1}$ involve at least two $\partial_\xi$ derivatives, so belong to $\sGM{m-2}{K,K',1}{N}$. Consequently, as
$ m \leq 0 $, the matrix of symbols 
\be\label{eq:Resto}
R = \Ical_2 - (P\# P^{-1})_{\rho,N}  \in  \sGM{\min\bigl(m-\frac{3}{2},-2\bigr)}{K,K',1}{N} \, . 
\ee
 This result differs from the first condition in \eqref{327} by the more modest smoothing properties of
 $ R $.    In order to get \eqref{327} we have to construct a matrix of symbols $ Q $, close to $ P^{-1} $ in the sense
 specified in \eqref{327}, such that $ P \# Q $ differs from the  identity $ {\cal I}_2 $ up to a smoother symbol.  
 We define
\begin{equation}
  \label{eq:3210}
  Q = P^{-1} + (P^{-1}\# R)_{\rho,N} + (P^{-1}\# R\#R)_{\rho,N} + \cdots + (P^{-1}\# \underbrace{R\#\cdots\#R}_{\ell})_{\rho,N}
\end{equation}
for some $\ell$ to be chosen. In view of \eqref{Resto},
the term involving $\ell'$ factors $R$ ($0\leq\ell'\leq\ell)$ in the sum 
\eqref{3210} belongs  to
$\sGM{\ell'\min\bigl(m-\frac{3}{2},-2\bigr)}{K,K',1}{N}$. 
By \eqref{3210} and \eqref{Resto} the matrix of 
symbols   $ Q-P^{-1} $  is in 
$$
\sGM{\min\bigl(m-\frac{3}{2},-2\bigr)}{K,K',1}{N}
$$ 
and therefore the second condition \eqref{327} holds. Moreover $Q$
satisfies \eqref{322}, \eqref{323} as this holds for $P^{-1}$.

Finally, it follows from the fact that $P$, $P^{-1}$ satisfy \eqref{312}, \eqref{313}, 
 \eqref{315}, from the definition of $ R $,  and the remark after
 Definition \ref{def:algebra},
 that the same holds true for $ R $. 
Moreover, again by composition, conditions \eqref{312}, \eqref{313},
\eqref{315} are satisfied also by $Q$.  Thus $Q$ is in $\sE{m}{K,K',1}{N}$.
To conclude the proof we just have to prove that the first property \eqref{327} holds, 
if $\ell$ has been taken large enough in \eqref{3210}. By 
\eqref{3210} and \eqref{Resto} we get 
\[(P\# Q)_{\rho,N} = \sum_{\ell'=0}^\ell\bigl((\Ical_2-R) \# \underbrace{R\#\cdots\#R}_{\ell'}\bigr)_{\rho,N} = \Ical_2 -
(\underbrace{R\#\cdots\#R}_{\ell+1})_{\rho,N}\]
which is equal to $\Ical_2$ modulo symbols of order $ (\ell + 1) \min (m-\frac{3}{2},-2\bigr) $
as negative as we want if $\ell$ is taken large enough. This concludes
the proof.
\end{proof}

\begin{lemma}
  \label{324} {\bf (Step of diagonalization)}
Let  $ m $ be in $ ]-\infty,\frac{3}{2}]$, $K, K'$ in $\N$ with $K-1\geq K'\geq 0$, $N$ in $\N^*$, $\rho$ in $\N$, $r>0$. Consider symbols
\begin{equation}
  \label{eq:3211}
\begin{split}
  \lambda'\in \sG{\frac{3}{2}}{K,K',1}{N}, \quad b'\in \sG{m}{K,K',1}{N}\\ 
  c' \in \sG{\min\bigl(m-\frac{1}{2},0\bigr)}{K,K',1}{N}, \quad d' \in \sG{1}{K,K',1}{N}
\end{split}\end{equation}
satisfying
\begin{equation}
\label{eq:3212}
 \Im\lambda' \, , \Im b'\in \sG{-\frac{1}{2}}{K,K',1}{N}, \quad \Im d' \in \sG{0}{K,K',1}{N}.
\end{equation}
Define 
\begin{equation}\label{eq:matrix-A'}
A'(U;t,x,\xi) = (\mk(\xi)+\lambda')\Kcal + b'\Jcal + c'\Lcal + d'\Ical_2
\end{equation}
and assume that $A'$ satisfies the reality, parity preserving and reversibility properties 
\eqref{311}, \eqref{313} and \eqref{314}.
Then there are\\
$\bullet$ A matrix of symbols $Q$ in $\sE{m-\frac{3}{2}}{K,K',1}{N}$,\\
$\bullet $ A matrix  $P$ in $\sE{m-\frac{3}{2}}{K,K',1}{N}$ such that $(P\#Q)_{\rho, N}-\Ical_2$  is in
$\sE{-\rho}{K,K',1}{N}$,  \\
$\bullet$ Symbols
\begin{equation}
  \label{eq:3213}
  \begin{split}
    \lambda'' \in \sG{\frac{3}{2}}{K,K'+1,1}{N}, \quad b'' \in \sG{m-1}{K,K'+1,1}{N},\\
c'' \in \sG{m-\frac{3}{2}}{K,K'+1,1}{N}, \quad d'' \in \sG{1}{K,K'+1,1}{N},
  \end{split}
\end{equation}
with
\begin{equation}\begin{split}
  \label{eq:3214}
  \Im \lambda'' \, , \ \Im b'' \in \sG{-\frac{1}{2}}{K,K'+1,1}{N},  \quad \Im d'' \in \sG{0}{K,K'+1,1}{N} \, , 
\end{split}\end{equation}
such that 
$$
A'' = (\mk(\xi)+\lambda'')\Kcal + b''\Jcal + c''\Lcal + d''\Ical_2
$$ 
satisfies the reality, parity preserving and reversibility properties
 \eqref{311}, \eqref{313} and
\eqref{314},\\
$\bullet$ A smoothing operator $R(U;t)$ in $\sRM{-\rho+\frac{3}{2}}{K,K'+1,1}{N}$, satisfying conditions \eqref{316}, \eqref{318} and
\eqref{319},

such that
\begin{equation}
  \label{eq:3215}
  \opbw(Q)[D_t-\opbw(A')] = [D_t-\opbw(A'')]\opbw(Q) + R(U;t).
\end{equation}

Finally, in the case $ m = \frac{3}{2} $, if there are  $\zeta'_1, \zeta'_2$ in $\sFR{K,K',1}{N}$ such that
\begin{equation}
  \label{eq:3216}
  \begin{split}
    \lambda'(U;t,x,\xi) - \mk(\xi)\zeta'_1(U;t,x) \in \sG{\frac{1}{2}}{K,K',1}{N}\\
b'(U;t,x,\xi) - \mk(\xi)\zeta'_2(U;t,x) \in \sG{\frac{1}{2}}{K,K',1}{N},
  \end{split}
\end{equation}
then, for some function $\zeta''$ of $\sFR{K,K'+1,1}{1}$, one has
\begin{equation}
  \label{eq:3217}
  \lambda''(U;t,x,\xi) - \mk(\xi)\zeta''(U;t,x)\in \sG{\frac{1}{2}}{K,K'+1,1}{N}.
\end{equation}
\end{lemma}
Note that the off-diagonal symbols $ b'' ,  c'' $ of $ A'' $ in \eqref{3213} are of lower order 
(actually $ 1$-smoother) with respect to those of $ A' $, namely 
$ b' , c ' $ in \eqref{3211}. On the other hand the order of the diagonal symbols is unchanged. 
This step ``consumes $ 1 $-time derivative'', i.e. note that the second
index $ K' + 1 $ in \eqref{3213} is  larger than
$ K' $ in \eqref{3211}.

\begin{proof}
  The eigenvalues of the matrix $A'(U;t,x,\xi)$ in \eqref{matrix-A'} are given by
\[d'(U;t,x,\xi) \pm\Bigl((\mk+\lambda'(U;t,x,\xi))^2 + c'(U;t,x,\xi)^2 -b'(U;t,x,\xi)^2\Bigr)^{\frac{1}{2}}\]
when $\xi$ stays outside some fixed neighborhood of zero, and when $U$ belongs to $\Br{K}{}$, for some small enough $r$ so
that the argument under the square root 
stays in the domain where the principal determination of that function is well defined. 
Take a cut-off function  $\chi_1$ in $C^\infty(\R)$, even, real valued, 
$0\leq \chi_1\leq 1$, with $\chi_1$  equal to zero on a
convenient neighborhood of zero and equal to one outside a larger  neighborhood of zero. 
Define
\begin{multline}
  \label{eq:3218}
\omega_\pm(U;t,x,\xi) = d'(U;t,x,\xi) \\
\pm (\mk(\xi)+\lambda'(U;t,x,\xi))\biggl[1+\frac{c'{}^2-b'{}^2}{(\mk(\xi)+\lambda')^2}\chi_1(\xi)\biggr]^{\frac{1}{2}}
\end{multline}
which coincide with the eigenvalues of $ A' $ outside a neighborhood of zero. 
By \eqref{3211}, \eqref{3212},  the symbols
 $ \omega_\pm \mp \mk(\xi) $ are in
$\sG{\frac{3}{2}}{K,K',1}{N}$ and  the imaginary part of the term under the square
root in \eqref{3218} is of order at most $-2$. Together with the assumption \eqref{3212} on $\lambda'$, this shows that
$\Im(\omega_\pm-d')$ is in $\sG{-\frac{1}{2}}{K,K',1}{N}$. 
It follows that the diagonal matrix of symbols 
$$
\bigl[\begin{smallmatrix}\omega_+&0\\0&\omega_-\end{smallmatrix}\bigr]
$$ 
may be written as a contribution to
$(\mk(\xi)+\lambda'')\Kcal + d''\Ical_2$ in $A''$, with $\lambda'', d''$ satisfying \eqref{3213}, \eqref{3214}. Moreover, if
$m=\frac{3}{2}$ and $\lambda', b'$ are of the form \eqref{3216}, we get that $\lambda''$ is of the form \eqref{3217}.

The fact that by assumption $A'(U;t,x,\xi)$ satisfies 
\eqref{311} may be written 
\begin{equation}
  \label{eq:3219}
  \begin{split}
    \bar{m}_\kappa^\vee(\xi) = \mk(\xi),\  \bar{\lambda}'{}^\vee(U;\cdot) = \lambda'(U;\cdot),\  \bar{b}'{}^\vee(U;\cdot) =
    b'(U;\cdot)\\
\bar{c}'{}^\vee(U;\cdot) =    -c'(U;\cdot),\  \bar{d}'{}^\vee(U;\cdot) =    -d'(U;\cdot).
  \end{split}
\end{equation}
In the same way, the reversibility condition \eqref{314} translates as
\begin{equation}
  \label{eq:3220}
  \begin{split}
     {\lambda}'(U;-t,\cdot) = \lambda'(U_S;t,\cdot),\quad  {b}'(U;-t,\cdot) =
    b'(U_S;t,\cdot)\\
{c}'(U;-t,\cdot) =    -c'(U_S;t,\cdot),\quad  {d}'(U;-t,\cdot) =    -d'(U_S;t,\cdot) \, .
  \end{split}
\end{equation}
As a consequence  the symbol
\begin{equation}
  \label{eq:3221}
  D(U;t,x,\xi) = \big(\mk(\xi)+\lambda'(U;t,x,\xi)\bigr)\biggl[1+\Bigl(1+\frac{c'{}^2-b'{}^2}{(\mk+\lambda')^2}\chi_1(\xi)\Bigr)^{\frac{1}{2}}\biggr]
\end{equation}
satisfies
\be\label{eq:matD}
\bar{D}^\vee(U;\cdot) = D(U;\cdot),\ D(U;-t,\cdot) = D(U_S;t,\cdot) \, .
\ee
A matrix which diagonalizes $A'$ on the domain where $\chi_1(\xi) = 1$ is
\[P(U;t,x,\xi) =
\frac{\chi_1(\xi)}{\Bigl(1-\frac{b'{}^2-c'{}^2}{D^2}\Bigr)^{\frac{1}{2}}} \begin{bmatrix}1&\frac{b'-c'}{D}\\\frac{b'+c'}{D}&1\end{bmatrix}.\]
We claim that $P$ belongs to  $\sE{m-\frac{3}{2}}{K,K',1}{N}$. 
Indeed we may write $ P $ as
\begin{equation}
  \label{eq:3222}
  P(U;t,x,\xi) = (1+\alpha)\Ical_2 +\beta\Lcal +\theta \Jcal
\end{equation}
with symbols $\alpha, \beta$ in $\sG{m-\frac{3}{2}}{K,K',1}{N}$ and  $\theta$ in
$\sG{\min\bigl(m-2,-\frac{3}{2}\bigr)}{K,K',1}{N}$. 
Thus $\alpha, \beta, \theta$ satisfy \eqref{322} with $m$ replaced by $m-\frac{3}{2}$.
Moreover, it follows from \eqref{3219}, \eqref{3220}, \eqref{matD} that
\[
\begin{split}
  \bar{\alpha}^\vee(U;\cdot) = \alpha(U;\cdot),\quad   \alpha(U;-t,\cdot) = \alpha(U_S;t,\cdot)\\
\bar{\beta}^\vee(U;\cdot) = \beta(U;\cdot),\quad  \beta(U;-t,\cdot) = \beta(U_S;t,\cdot)\\
\bar{\theta}^\vee(U;\cdot) = -\theta(U;\cdot),\quad  \theta(U;-t,\cdot) = -\theta(U_S;t,\cdot) 
\end{split}
\]
and therefore $P$ satisfies \eqref{312} and \eqref{315}. It obviously satisfies as well \eqref{313}. 
 We have seen that $\Im D =
\Im(\omega_+-d') +\Im\lambda'$ is in $\sG{-\frac{1}{2}}{K,K',1}{N}$. It follows that $\Im\alpha, \Im\beta$ are of order
$\min\bigl(m-\frac{3}{2},-\frac{3}{2}\bigr)$, so that \eqref{323} is satisfied with $m$ replaced by
$m-\frac{3}{2}$. Consequently $P$  is in  $\sE{m-\frac{3}{2}}{K,K',1}{N}$.
Also the matrix  
\begin{equation} \label{eq:P-1}
P^{-1} = (1+\alpha)\Ical_2 -\beta\Lcal -\theta\Jcal  
\end{equation}
(when $\chi_1\equiv 1$)  is in   $\sE{m-\frac{3}{2}}{K,K',1}{N}$.
By construction
\begin{equation}
  \label{eq:3223}
  P^{-1}A'P = \begin{bmatrix}\omega_+&0\\0&\omega_-\end{bmatrix} = (\mk(\xi) +\tilde{\lambda}'(U;t,x,\xi))\Kcal + d'\Ical_2
\end{equation}
on the domain where $\chi_1\equiv 1$, for a symbol $\tilde{\lambda}'$ in $\sG{\frac{3}{2}}{K,K',1}{N}$ with $\Im
\tilde{\lambda}'$ in  $\sG{-\frac{1}{2}}{K,K',1}{N}$. Moreover, the matrix \eqref{3223} satisfies \eqref{311}, \eqref{313} and
\eqref{314}, and \eqref{3217} is satisfied by $\tilde{\lambda}'$ if we assume \eqref{3216}. 

We apply next to $P$ lemma~\ref{323}, with $m$ replaced by $m-\frac{3}{2}$. We obtain $Q$ in $\sE{m-\frac{3}{2}}{K,K',1}{N}$
satisfying \eqref{327} (with $m$ replaced by $m-\frac{3}{2}$). We write
\begin{equation}
  \label{eq:3224}
  (Q\# A'\# P)_{\rho,N} = ((Q-P^{-1})\# A'\# P)_{\rho,N} + (P^{-1}\# A'\# P)_{\rho,N}.
\end{equation}
 Since $ Q-P^{-1}  $ is of order $ \min(m- 3, - 2 )$, $ A' $ is of order  $ 3/2 $ and $ P $ of order $ m - \frac32 \leq 0 $, 
 the first term in the right hand side has coefficients in $\sG{\min\bigl(m-\frac{3}{2},-\frac{1}{2}\bigr)}{K,K',1}{N}$ and  
 we may write the matrix 
\begin{equation}\label{eq:RM1}
((Q-P^{-1})\# A'\# P)_{\rho,N} = \lambda''\Kcal + b''\Jcal + c''\Lcal + d''\Ical_2
\end{equation}
with $\lambda'', b'', c'', d''$ satisfying conditions \eqref{3213}.
Next we write the last term of \eqref{3224} as 
\begin{equation}
  \label{eq:3225}
   (P^{-1}\# A'\# P)_{\rho,N} = P^{-1}A'P + R \, ,
\end{equation}
where $R$ is obtained from the terms indexed 
by  $ 1 \leq \ell < \rho $  in the composition formulas \eqref{222}, \eqref{223}. 
By 
\eqref{3222}, \eqref{matrix-A'} and  \eqref{P-1}  the matrix $R$ is given by the terms of rank 
 $ 1 \leq  \ell < \rho $
in the expansion
\begin{multline}
  \label{eq:3226}
\bigl(((1+\alpha)\Ical_2 -\beta\Lcal-\theta\Jcal)\#((\mk+\lambda')\Kcal + b'\Jcal + c'\Lcal + d'\Ical_2)\\\#
((1+\alpha)\Ical_2 +\beta\Lcal+\theta\Jcal)\bigr)_\rho \, .
\end{multline}
To finish the proof of  lemma \ref{324}, we need the following result.

\begin{lemma}
  \label{325}
The matrix $R$ in \eqref{3225} may be written as
\begin{equation}
  \label{eq:3227}
  R = \lambda''\Kcal + b''\Jcal + c''\Lcal + d''\Ical_2
\end{equation}
where $\lambda'', b'', c'', d''$ satisfy conditions \eqref{3213}, \eqref{3214} and the matrix $R$ satisfies \eqref{311}, \eqref{313}, \eqref{314}. Moreover, $\lambda''$ satisfies \eqref{3217}.
\end{lemma}
\begin{proof}
  Consider first those contributions to \eqref{3226} coming from one of the $\theta\Jcal$ factors. Since $\theta$ is of order
  $\min\bigl(m-2,-\frac{3}{2}\bigr)$, and the other factors are at most of order $\frac{3}{2}$ (resp. $0$) if they come from
  the middle (resp.\ one of the extreme) terms, the overall order would be $\min\bigl(m -\frac{1}{2},0\bigr)$. But since we
  are interested only on those contributions coming from the terms of the asymptotic expansion except the first one, we gain
  at least one extra order, i.e.\ we obtain a contribution to \eqref{3227} of order $\min\bigl(m -\frac{3}{2},-1\bigr)$, so
  that conditions \eqref{3213}, \eqref{3214} are largely satisfied. We are thus reduced to the study of elements of rank at
  least one in the asymptotic expansion \eqref{3226} with $\theta$ removed. Notice also that $c'$ being of order $\min\bigl(m
  -\frac{1}{2},0\bigr)$, the contribution to those elements of the asymptotic expansion \eqref{3226} satisfy as well
  \eqref{3213}, \eqref{3214}. We may as well discard $c'$, and are reduced to 
  examine terms of rank at least one in the
  following asymptotic expansions:
  \begin{equation} 
    \label{eq:3228}
    \bigl(((1+\alpha)\Ical_2 -\beta\Lcal)\# ((\mk+\lambda')\Kcal)\# ((1+\alpha)\Ical_2+\beta\Lcal)\bigr)_\rho
  \end{equation}
\begin{equation}
    \label{eq:3229}
    \bigl(((1+\alpha)\Ical_2 -\beta\Lcal)\# (b'\Jcal)\# ((1+\alpha)\Ical_2+\beta\Lcal)\bigr)_\rho
  \end{equation}
\begin{equation}
    \label{eq:3230}
    \bigl(((1+\alpha)\Ical_2 -\beta\Lcal)\# (d'\Ical_2)\# ((1+\alpha)\Ical_2+\beta\Lcal)\bigr)_\rho \, .
  \end{equation}
Consider first \eqref{3228}, that may be written using the relations $\Lcal\Kcal = -\Kcal\Lcal = \Jcal$, $\Lcal\Kcal\Lcal =
-\Kcal$ as
\begin{multline*}
  \bigl[(1+\alpha)\#(\mk+\lambda')\#(1+\alpha) + \beta\#(\mk+\lambda')\#\beta\bigr]_\rho\Kcal\\
-\bigl[(1+\alpha)\#(\mk+\lambda')\#\beta + \beta\#(\mk+\lambda')\#(1+\alpha)\bigr]_\rho \Jcal.
\end{multline*}
The expansions \eqref{222}, \eqref{223} (see in particular \eqref{trilinear}) 
show that 
the terms of rank one in   the coefficients of $\Kcal$
and $\Jcal$ above vanish. Consequently,
 the contribution to $R$ come only from terms of rank at least two in the
expansion. As $\mk+\lambda'$ is of order $\frac{3}{2}$ and $\alpha, \beta$ of order $m -\frac{3}{2}$, we get expressions of
order $m-2\leq-\frac{1}{2}$, so that conditions \eqref{3213}, \eqref{3214} are satisfied.

In the same way, using $\Lcal\Jcal = -\Jcal\Lcal = \Kcal$, $\Lcal\Jcal\Lcal =
-\Jcal$, we write \eqref{3229} as
\[
  \bigl[(1+\alpha)\#b'\#(1+\alpha) + \beta\#b'\#\beta\bigr]_\rho\Jcal
-\bigl[\beta\# b'\#(1+\alpha) + (1+\alpha)\#b'\#\beta\bigr]_\rho \Kcal
\]
and conclude that the corresponding contribution to $R$ are of order $m-2$. 
Finally, using that $ \Lcal^2 = \Ical_2 $, write \eqref{3230} as
\begin{multline*}
  \bigl[(1+\alpha)\#d'\#(1+\alpha) - \beta\#d'\#\beta\bigr]_\rho \Ical_2\\
+\bigl[(1+\alpha)\#d'\#\beta - \beta\# d'\#(1+\alpha)\bigr]_\rho \Lcal.
\end{multline*}
Again, by \eqref{trilinear} applied with $a=c$, we see that the 
coefficient of $\Ical_2$ is of order $-1$, since $d'$ is of order $1$, so that the conditions imposed
to $d''$ in \eqref{3213}, \eqref{3214} are largely satisfied. In the coefficients of $\Lcal$, we just use that $\alpha,
\beta$ are of order $m-\frac{3}{2}\leq 0$, and $d'$ of order $1$, to conclude that terms of rank one of the expansion are at
most of order $m-\frac{3}{2}$, which is the condition imposed to $c''$ in \eqref{3213}. 

To finish the proof of the lemma, we  have to check that conditions \eqref{311}, \eqref{313}, \eqref{314} hold true for
the matrix $R$. This follows from the fact that these conditions are satisfied by $A'$ and that $P$, $P^{-1}$ satisfy \eqref{312},
\eqref{313}, \eqref{315}. Notice also that \eqref{3217} is satisfied by $\lambda''$ since the contributions to this symbol
are all of non positive order.
\end{proof}
\vspace{2ex}\noindent\textsl{End of the proof of lemma~\ref{324}:} Recall that we have constructed matrices $P, Q$ in
$\sE{m-\frac{3}{2}}{K,K',1}{N}$ such that, according to \eqref{3224}, \eqref{3225}, \eqref{RM1}, \eqref{3223} and \eqref{3227},
\begin{equation}
  \label{eq:3231}
  (Q\#A'\#P)_{\rho,N} = A'' \stackrel{\textrm{def}}{=} (\mk(\xi)+\lambda'')\Kcal+b''\Jcal+c''\Lcal+d''\Ical_2
\end{equation}
for new values of the symbols $\lambda'', b'', c'', d''$ satisfying \eqref{3213}, \eqref{3214}, the matrix $A''$ satisfying
as well \eqref{312}, \eqref{313} and \eqref{314}. We have checked also that \eqref{3217} holds under the assumptions
\eqref{3216}. We are left with showing that \eqref{3215} holds, up to a modification of the matrix $A''$. Compute from
\eqref{3231}
\begin{multline}
  \label{eq:3232}
(D_t -\opbw(A''))\circ \opbw(Q) \\
= \opbw(Q)D_t -\opbw\bigl((Q\# A'\# P)_{\rho,N}\bigr)\circ\opbw(Q) + \opbw(D_tQ)\\
= \opbw(Q)\bigl[D_t - \opbw(A')\bigr] -\opbw\bigl((Q\# A'\# (P\#Q -\Ical_2))_{\rho,N}\bigr)\\
+\opbw(D_tQ) + R(U;t)
\end{multline}
where $ R(U;t) $, the smoothing remainder coming from \eqref{224}, \eqref{225}, may be taken in
$\sRM{-\rho+\frac{3}{2}}{K,K',1}{N}$. Moreover, $R(U;t)$ will satisfy conditions \eqref{316}, \eqref{318}, \eqref{319}, as it is
obtained from the composition of one operator satisfying these conditions and of operators obeying conditions \eqref{317},
\eqref{318} and \eqref{3110}. 
Because of \eqref{327}, the fact that $ A' $ is of order $ 3/2 $, and $ Q $ is of order $ \leq 0 $, 
the term 
\be\label{eq:newr2}
\opbw\bigl((Q\# A'\# (P\#Q -\Ical_2))_{\rho,N}\bigr) \in \sRM{-\rho+\frac{3}{2}}{K,K',1}{N} 
\ee
may be incorporated as well to $R(U;t)$.  To conclude the proof of \eqref{3215}, it remains to write
\begin{equation}
  \label{eq:3233}
 \opbw(D_tQ) =  \opbw(\tilde{A}'')\circ \opbw(Q) + R(U;t)
\end{equation}
for another smoothing operator $R(U;t)$ in $\sR{-\rho+\frac{3}{2}}{K,K'+1,1}{N}$, and 
for a matrix of symbols $\tilde{A}''$ satisfying the same
conditions as $A''$.
In conclusion  \eqref{3232}, \eqref{newr2}, \eqref{3233} will imply  
\eqref{3215}  with $A''$ replaced by $A'' + \tilde{A}''$  in its right
hand side.
In order to prove \eqref{3233} we remark first that, since $Q$ is in 
$\sE{m-\frac{3}{2}}{K,K',1}{N}$, then, recalling Definition \ref{321},
 $Q $ is in $ \sGM{m-\frac{3}{2}}{K,K',1}{N} $, 
and so, by lemma~\ref{217}, 
\begin{equation}\label{eq:DtQ}
D_tQ \in \sGM{m-\frac{3}{2}}{K,K'+1,1}{N} \, . 
\end{equation} 
Moreover, if $Q$ satisfies \eqref{312} (resp.\ \eqref{313}, resp.\ \eqref{315}), then $D_tQ$ satisfies \eqref{311}
(resp.\ \eqref{313}, resp.\ \eqref{314}). 
If we apply \eqref{327}, we may write
\[\opbw(D_tQ) = \opbw\bigl((D_tQ\#P)_{\rho,N}\bigr)\circ\opbw(Q)\]
modulo an operator $R(U;t)$ in $\sRM{-\rho+\frac{3}{2}}{K,K'+1,1}{N}$, satisfying \eqref{316}, \eqref{318}, \eqref{319}.
Thus
\eqref{3233} holds with $ \tilde{A}'' = (D_tQ\#P)_{\rho,N} $ 
 and we are left with checking that $ \tilde{A}''  $ satisfies the same properties as $ A'' $.
By \eqref{DtQ} and since $ P $ is of order zero, 
conditions \eqref{3213} hold. We have to check
  \eqref{3214}. Notice that, as $P, Q$ are in $\sE{m-\frac{3}{2}}{K,K',1}{N}$, conditions \eqref{322}, \eqref{323} imply that we
  may write \[P = (1+\alpha)\Ical_2 +\beta\Lcal + \textrm{ order }(-\frac{3}{2}),\ Q = (1+\alpha')\Ical_2 + \beta'\Lcal +
  \textrm{ order }(-\frac{3}{2}).\]
 It follows that in $(D_tQ\#P)_{\rho,N}$ the only contributions of order strictly larger 
  than $-\frac{3}{2}$ are multiples of $\Ical_2$ and $\Lcal$, and are of non positive order. This shows that \eqref{3214}
  holds and concludes the proof, since \eqref{311}, \eqref{313}, \eqref{314} hold, as they do for $D_tQ$.
\end{proof}

\begin{proof1}{Proof of Proposition~\ref{322}} We shall apply iteratively $ \rho $-times lemma  \ref{324},
obtaining  off-diagonal para-differential operators which are regularizing enough 
to be incorporated to the smoothing term in \eqref{326}. 
Consider a solution $U$ of the system \eqref{3112}.
Apply lemma~\ref{324} to the matrix $A'=A$ of equation \eqref{3112} and $m=\frac{3}{2}$, $K'=1$. We get a matrix
$A_1=A''$ whose coefficients satisfy \eqref{3213}, \eqref{3214} with $m= \frac{3}{2}$, $K'=1$, and \eqref{311}, \eqref{313},
\eqref{314}, and a matrix $Q_1$ in $\sE{0}{K,1,1}{N}$ such that by \eqref{3215}
\[(D_t -\opbw(A_1))\circ\opbw(Q_1) = \opbw(Q_1)[D_t-\opbw(A)] + R_1(U;t)\]
 with $R_1$ in $\sGM{-\rho+\frac{3}{2}}{K,2,1}{N}$, satisfying \eqref{316}, \eqref{318} and \eqref{319}. We iterate the
construction, getting matrices 
$$
Q_\ell \in \sE{1-\ell}{K,\ell,1}{N} \, , \quad 
A_\ell = (\mk(\xi)+\lambda_\ell)\Kcal + b_\ell
\Jcal + c_\ell\Lcal + d_\ell\Ical_2 
$$ 
where the coefficients $\lambda_\ell, b_\ell, c_\ell, d_\ell$ satisfy \eqref{3213},
\eqref{3214} with $m$ replaced by $\frac{3}{2}-(\ell-1)$, and $K'$ replaced by $\ell$, $A_\ell$ satisfying also \eqref{311},
\eqref{313} and \eqref{314}. We get
\begin{multline}
  \label{eq:3234}
\bigl(D_t-\opbw(A_\ell)\bigr)\opbw(Q_\ell)\circ\cdots\circ\opbw(Q_1)\\
= \opbw(Q_\ell) \circ\cdots\circ\opbw(Q_1) \bigl(D_t-\opbw(A)\bigr) + R_\ell(U;t)
\end{multline}
for some new $R_\ell(U;t)$ in $\sRM{-\rho+\frac{3}{2}}{K,\ell+1,1}{N}$. Since 
$$
b_\ell\Jcal \in \sGM{\frac{3}{2}-\ell}{K,\ell+1,1}{N} \, , \quad c_\ell\Lcal \in \sGM{\frac{3}{2}-\ell-\frac{1}{2}}{K,\ell+1,1}{N} \, , 
$$
we see, by the last remark after Proposition \ref{215}, that for $\ell=\rho$ we
obtain contributions that may be as well incorporated to the smoothing term in \eqref{326}. We thus define $A^{(1)}$ in
\eqref{325} taking only the remaining components of $A_\rho$, i.e.\ 
\[A^{(1)}(U;t,x,\xi) = (\mk(\xi)+\lambda_\rho)\Kcal + d_\rho\Ical_2\]
with symbols $\lambda_\rho$ that may be written by \eqref{3217} as $\mk\zeta^{(1)} + \ldu{1}$ with $\ldu{1}$ in
$\sG{\frac{1}{2}}{K,\rho+1,1}{N}$ satisfying \eqref{324}, and with $d_\rho$ that brings $\lambda_1^{(1)}$ in \eqref{325}. We
define
\begin{equation}
  \label{eq:3235}
  W = \opbw(Q_\rho) \circ\cdots\circ\opbw(Q_1)U.
\end{equation}
If we make act \eqref{3234} with $\ell=\rho$ on $U$, we deduce from \eqref{3112}
\begin{equation}
  \label{eq:3236}
  (D_t-\opbw(A^{(1)}(U; t, \cdot))W = R(U;t)U
\end{equation}
for some new $R(U;t)$ in $\sRM{-\rho+\frac{3}{2}}{K,\rho+1,1}{N}$. To finish the proof of the proposition, we still have 
to check that one may write the right hand side of the equation \eqref{3236}
as in \eqref{326} and to construct the matrix $P$. 
We define 
$$
P = (P_1\#\cdots\#P_\rho)_{\rho,N}
$$ 
where $P_\ell$ is the matrix associated to $Q_\ell$ by lemma~\ref{323}.  By \eqref{327}, the operator 
$$
\opbw((P_\ell\#Q_\ell)_{\rho,N})-\mathrm{Id} = R_\ell(U;t) \in \sRM{-\rho}{K,\ell,1}{N}
$$ 
and satisfies \eqref{317},
\eqref{318}, \eqref{3110}.  It follows  that 
$$
\opbw((P\#Q)_{\rho,N})-\mathrm{Id} = \tilde{R}(U;t) \in 
\sRM{-\rho}{K,\rho,1}{N}
$$ and satisfies as well \eqref{317},
\eqref{318}, \eqref{3110}.
By \eqref{3235}, we may write
\[U = \opbw(P)W - \tilde{R}(U;t)U.\]
Inserting this expression of $ U $ in the right hand side of \eqref{3236} repeatedly, we see, by Proposition \ref{232} and the last remark following Definition \ref{214}, 
that we may write  $ R(U; t) U $ as $ R'(U;t)W+R''(U;t)U$
with $R', R''$ satisfying the properties of the statement. This concludes the proof.
\end{proof1}

%% file: chap4max.tex
\chapter[Reductions and proof of main theorem]{Reduction to a constant coefficients operator and proof of the main theorem}\label{cha:4}

In this chapter we shall reduce the operator $\opbw(A^{(1)}(U; t, \cdot))$ in \eqref{326}, given in terms 
of the diagonal matrix
$A^{(1)}$ defined in \eqref{325}, to a constant coefficients operator, up to smoothing operators. 
We shall do that first for terms of higher order, and then for the lower
order contributions.

\section[Reduction of highest order]{Reduction to constant coefficients of the hig\-hest order part}\label{sec:41}

We apply Proposition~\ref{322} with $\rho$ replaced by
 $ \rho' = \rho + m $, for some integer $ m $ to be chosen in function of
$ N $ (in Proposition \ref{411} we require $ m \geq N + \frac32 $). 
Recalling the form \eqref{325} of  the diagonal matrix $A^{(1)}$, 
 we may rewrite equation \eqref{326} as the system \small
\begin{multline}
  \label{eq:415}
D_tW -\\ \opbw\begin{bmatrix}\mk(\xi)(1+\zeta^{(1)}(U;\cdot))+ \ldu{1} + \lu{1}_1 & \!\!\! \!\! 0\\0
  & \!\!\!\! \!\! -\mk(\xi)(1+\zeta^{(1)}(U;\cdot))- \ldu{1} + \lu{1}_1\end{bmatrix}W\\
=R'(U;t)W + R''(U;t)U
\end{multline}\normalsize
with $R'(U;t)$ (resp. $R''(U;t)$) in the class $\sRM{-\rho'+\frac{3}{2}}{K,\rho'+1,1}{N}$ (resp.\ in the class $\RrM{-\rho'+\frac{3}{2}}{K,\rho'+1,N}$),
satisfying the reality, parity preserving, reversibility conditions \eqref{316}, \eqref{318} and \eqref{319}. 
Moreover $ A^{(1)} $ 
satisfies the reality, parity preserving, reversibility conditions \eqref{311}, \eqref{313}, \eqref{314}, 
with coefficients verifying \eqref{324}. We  
rewrite  explicitly all these conditions as
 \begin{equation}
 \begin{aligned} \label{eq:411}
&   \lambda_1^{(1)} \in \sG{1}{K,\rho'+1,1}{N} \, , \qquad \  \ldu{1} \in \sG{\frac{1}{2}}{K,\rho'+1,1}{N}, \\
&   \Im \lambda_1^{(1)} \in \sG{0}{K,\rho'+1,1}{N} \, , \quad \Im\ldu{1} \in \sG{-\frac{1}{2}}{K,\rho'+1,1}{N}, \\
& \zeta^{(1)} \in \sFR{K,\rho'+1,1}{N}  \, , \qquad  \ \Im\zeta^{(1)}= 0 \, , 
\end{aligned}
\end{equation}
and, recalling that $\mk (\xi) $ is real valued, even, and  
  $ S {\cal K} = -  {\cal K} S $, $ S^2 = I $, 
\begin{equation}
  \label{eq:412}
  \begin{split}
  \bar{\lambda}^{(1)}_{1}(U;t,x,-\xi) &= -\lu{1}_1(U;t,x,\xi) \\
    \bar{\lambda}^{(1)}_{1/2}(U;t,x,-\xi) &= \ldu{1}(U;t,x,\xi) 
  \end{split}
\end{equation}
\begin{equation}
  \label{eq:413}
  \begin{split}
  \lambda^{(1)}_{1}(U;t,-x,-\xi) &= \lu{1}_1(U;t,x,\xi) \\
    \lambda^{(1)}_{1/2}(U;t,-x,-\xi) &= \ldu{1}(U;t,x,\xi) \\
\zeta^{(1)}(U;t,-x) &= \zeta^{(1)}(U;t,x)
  \end{split}
\end{equation}
\begin{equation}
  \label{eq:414}
  \begin{split}
  \lambda^{(1)}_{1}(U;-t,x,\xi) &= -\lu{1}_1(U_S;t,x,\xi) \\
  \lambda^{(1)}_{1/2}(U;-t,x,\xi) &= \ldu{1}(U_S;t,x,\xi) \\
  \zeta^{(1)}(U;-t,x) &= \zeta^{(1)}(U_S;t,x) \, . 
  \end{split}
\end{equation}
In Proposition \ref{411} we  shall 
conjugate system \eqref{415} under the paracomposition operator
$ \Phi_U^\star = \Omega_{B(U)} {\cal I}_2  $ defined in section \ref{sec:Para} induced by 
a diffeomorphism of $ \Tu $,  
\begin{equation}
  \label{eq:418-new}
  \Phi_U(t,x) = x + \beta (U;t,x) \, , 
\end{equation}
for a small periodic function $ \beta (U;t,x) $ to be chosen, in such a way that in \eqref{4112} the highest order 
coefficient  in front of $ m_\kappa (\xi) $ is constant in the space variable.  
It turns out that, denoting the inverse diffeomorphism of $ \Phi_U $ by 
\begin{equation}
  \label{eq:418}
  \Phi^{-1}_U(t,y) = y + \gamma(U;t,y) 
\end{equation}
we have to choose $ \gamma (U;t,y) $ as the unique primitive with zero average of the zero mean periodic function
\begin{equation}
  \label{eq:417}
  (1+\zu(U;t))^{\frac{2}{3}}[1+\zeta^{(1)}(U;t,y)]^{-\frac{2}{3}}-1 
\end{equation}
where
\begin{equation}
  \label{eq:416}
  \zu(U;t) = \biggl[\frac{1}{2\pi}\int_\Tu (1+\zeta^{(1)}(U;t,y))^{-\frac{2}{3}}\,dy\biggr]^{-\frac{3}{2}} -1 \, . 
\end{equation}
Notice that, as $\zeta^{(1)}(U;\cdot)$ vanishes at $U=0$, the above functions are well defined
for $ U $ small enough. 
Moreover
$\zu(U;t) $ is in $\sFR{K,\rho'+1,1}{N}$ and, by the third condition in \eqref{414}, 
  satisfies $ \zu(U;-t) = \zu(U_S;t) $.

\begin{proposition}
  \label{411}{\bf (Reduction of the highest order)}
Let $ \rho ' = \rho + N + 2 $.   There are:\\
$\bullet$ A function $\zu(U;t)$ in $\sFR{K,\rho'+1,1}{N}$, independent of $x$, satisfying
\begin{equation}
  \label{eq:419}
  \zu(U;-t) = \zu(U_S;t),
\end{equation}
$\bullet$ A symbol $\tilde{\lambda}(U; t, \cdot)$ in $\sG{\frac{1}{2}}{K,\rho'+1,1}{N}$ 
such that $\Im \tilde{\lambda}(U;t, \cdot ) $ belongs to $\sG{-\frac{1}{2}}{K,\rho'+1,1}{N}$, satisfying
\begin{equation}
  \label{eq:4110}
  \begin{split}
    \overline{\tilde{\lambda}}^\vee(U;t,x,\xi) &= \tilde{\lambda}(U;t,x,\xi)\\
\tilde{\lambda}(U;t,-x,-\xi) &= \tilde{\lambda}(U;t,x,\xi)\\
\tilde{\lambda}(U;-t,x,\xi) &= \tilde{\lambda}(U_S;t,x,\xi),
  \end{split}
\end{equation}
$\bullet$ A symbol $\mu(U; t, \cdot)$ in $\sG{1}{K,\rho'+2,1}{N}$ such that $\Im\mu(U; t, \cdot)$ belongs to the class $\sG{0}{K,\rho'+2,1}{N}$ and
\begin{equation}
  \label{eq:4111}
  \begin{split}
    \overline{\mu}^\vee(U;t,x,\xi) &= -\mu(U;t,x,\xi)\\
\mu(U;t,-x,-\xi) &= \mu(U;t,x,\xi)\\
\mu(U;-t,x,\xi) &= -\mu(U_S;t,x,\xi),
  \end{split}
\end{equation}
$\bullet$ Operators $R'(U;t)$ (resp.\ $R''(U;t)$) belonging to  $\sRM{-\rho+\frac{3}{2}}{K,\rho'+2,1}{N}$ (resp.\ $\RrM{-\rho+\frac{3}{2}}{K,\rho'+2,N}$)
satisfying the reality, parity preserving and reversibility conditions \eqref{316}, \eqref{318} and \eqref{319},

 such that, if we set $V= \Omega_B (1) W = \Phi_U^\star W$ with $W$ a solution to \eqref{415}, then
$V$ satisfies the system
\begin{multline} \label{eq:4112}
\bigl(D_t - \opbw\bigl[\bigl((1+\zu(U;t))\mk(\xi) + \tilde{\lambda}(U;t, \cdot)\bigr)\Kcal + \mu(U; t, \cdot)\Ical_2\bigr]\bigr)V\\
= R'(U;t)V+ R''(U;t)U.
\end{multline}
\end{proposition}

\begin{proof}
Applying Proposition \ref{cong-Dt} with $K'=\rho'+1$,  
under the change of variable $ V = \Omega_{B} (1) W $, system \eqref{415}
transforms into 
\begin{multline}
  \label{eq:4113}
D_tV = \\
 \Omega_{B} (1)  \opbw\bigl[\bigl(\mk(\xi)(1+\zeta^{(1)}(U;\cdot))+ \ldu{1}(U;\cdot)\bigr)\Kcal +
\lu{1}_1(U;\cdot)\Ical_2\bigr]\Omega_{B}^{-1} (1)  V \\
+ \big( {\rm Op}^{BW}(e_1 (U;\cdot)) \Ical_2 \big) V  
 + R(U;t)V  + \Omega_B (1) R'(U;t) \Omega_B^{-1} (1) V  + \Omega_B (1) R''(U;t)U 
\end{multline} 
where $e_1 $ is in $\sG{1}{K,\rho'+2,1}{N}$ with $\Im e_1 $ in $\sG{-1}{K,\rho'+2,1}{N}$ 
and $R(U;t) $ is in the space $\sR{-\rho}{K,\rho'+2,1}{N}$. 
Thus the operator
\begin{equation}
  \label{eq:4115}
   {\rm Op}^{BW}(e_1(U;\cdot)) \Ical_2  
\end{equation}
may be incorporated to $ {\rm Op}^{BW}(\mu) {\mathcal I}_2 $ in \eqref{4112}. Consider next the term
\begin{equation}
\label{eq:4118a}
\Omega_B (1) R'(U;t) \Omega_B^{-1} (1) V  + \Omega_B (1) R''(U;t)U.   
\end{equation}
Applying the remark following the proof of Theorem~\ref{thm:para},  
since $R'(U;t)$ is in
$\sRM{-\rho'+\frac{3}{2}}{K,\rho'+1,1}{N}$, we get that  
$$
\Omega_B (1) R'(U;t) \Omega_B^{-1} (1) \in 
\sRM{-\rho'+\frac{3}{2}+N}{K,\rho'+1,1}{N}
$$
and, since  $\rho' \geq \rho+N+ \frac32 $,  
we get an operator of $\sRM{-\rho}{K,\rho'+1,1}{N}$. 
The last term in  \eqref{4118a} belongs as well to $\RrM{-\rho}{K,\rho'+1,N}$.
We finally analyze  the  contribution
\begin{equation}
  \label{eq:4114}
  \Omega_{B} (1) \opbw\bigl[\bigl(\mk(1+\zeta^{(1)})+ \ldu{1}\bigr)\Kcal + \lu{1}_1\Ical_2\bigr]
   \Omega_{B}^{-1} (1) 
\end{equation}
that, by Theorem~\ref{thm:para} 
we may write  as
\begin{equation}
  \label{eq:4119}
  \opbw\Bigl[\bigl(\mk(1+\zeta^{(1)})+\ldu{1}\bigr)_\Phi\Kcal + (\lu{1}_1)_\Phi\Ical_2\Bigr] + R_2(U;t)
\end{equation}
where $R_2$ belongs to $\sRM{-\rho+\frac{3}{2}}{K,\rho'+1,1}{N}$. 
According to \eqref{a-phi-svi}-\eqref{a-0-forma} the symbol 
$\bigl(\mk(1+\zeta^{(1)})+\ldu{1}\bigr)_\Phi$ may be expanded as
\begin{equation}
  \label{eq:4120}
  \mk\Bigl(\xi\partial_y( \Phi_U^{-1}(t,y))\vert_{y=\Phi_U(t,x)}\Bigr)(1+\zeta^{(1)}(U;t,\Phi_U(t,x))) +
  (\ldu{1})^0_\Phi 
\end{equation}
modulo a symbol in $ \Sigma \Gamma^{- \frac12}_{K, \rho' + 1, 1} [r, N] $. 
But, according to   \eqref{418}, the definition of $ \gamma $ as a primitive of \eqref{417}, 
\[
\partial_y(\Phi_U)^{-1}(t,y) = 1 +\partial_y \gamma(U;t,y) = (1+\zu(U;t))^{\frac{2}{3}}(1+\zeta^{(1)}(U;t,y))^{-\frac{2}{3}} \, . 
\]  
Since by \eqref{3115}, $\mk(\xi) =
\sqrt{\kappa}\abs{\xi}^{\frac{3}{2}}(1+O(\xi^{-2}))$ when $\abs{\xi}$ goes to $+\infty$, one checks that
\[\mk(\xi) - \mk\bigl(\xi(1+\zu)^{\frac{2}{3}}(1+\zeta^{(1)})^{-\frac{2}{3}}\bigr)(1+\zu)^{-1}(1+\zeta^{(1)})\]
belongs to $\sG{-\frac{1}{2}}{K,\rho'+1,1}{N}$. 
It follows that 
$$ 
\bigl(\mk(1+\zeta^{(1)})+\ldu{1}\bigr)_\Phi = 
\mk(\xi)(1+\zu(U;t)) +   (\ldu{1})^0_\Phi  
$$ 
modulo a symbol of $\sG{-\frac{1}{2}}{K,\rho'+1,1}{N}$. 
Since by \eqref{411} the imaginary term $\Im\ldu{1}$ is of order  $-\frac{1}{2}$, 
we see that $(\ldu{1})_\Phi^0 $ is a symbol of order $\frac{1}{2}$ whose
imaginary part is of order $-\frac{1}{2}$. 
In the same way $ (\lu{1}_1)_\Phi $ is equal to $ (\lu{1}_1)_\Phi^0  $ modulo 
a symbol in $\sG{-1}{K,\rho'+1,1}{N} $.  Hence \eqref{411} implies that  
$ (\lu{1}_1)_\Phi $ 
is a symbol of $\sG{1}{K,\rho'+1,1}{N}$ whose imaginary part belongs
to $\sG{0}{K,\rho'+1,1}{N}$. 
We have thus written \eqref{4119} as
\begin{equation}
  \label{eq:4121}
  \opbw\bigl((\mk(\xi)(1+\zu(U;t))+ \tilde{\lambda}_{\frac{1}{2}}(U;\cdot))\Kcal + \tilde{\lambda}_1(U;\cdot)\Ical_2\bigr) + 
  R_2(U;t)
\end{equation}
for symbols   $ \tilde{\lambda}_j \in \sG{j}{K,\rho'+1,1}{N} $, $  j = \frac{1}{2}, 1 $,  satisfying 
$$
\Im  \tilde{\lambda}_{\frac{1}{2}}  \in \sG{-\frac{1}{2}}{K,\rho'+1,1}{N} \, , \quad 
\Im  \tilde{\lambda}_1 \in \sG{0}{K,\rho'+1,1}{N} \, , 
$$
and a smoothing operator $R_2(U;t)$ in $\sRM{-\rho+\frac{3}{2}}{K,\rho'+1,1}{N}$. Thus also 
\eqref{4114} gives contributions of  the form \eqref{4112}. 

It remains to check that the algebraic properties of  reality, parity and reversibility are 
preserved. 

Since $  \zeta^{(1)}(U;\cdot) $ is even in $ x $, by  the third condition \eqref{413}, the 
real valued function 
$ \gamma(U; \cdot )$ defined  as  a 
primitive of \eqref{417} is an odd  element of $\sFR{K,\rho'+1,1}{N}$. 
Moreover $ \gamma (U;-t,y) = \gamma (U_S;t,y)$ 
by  the third identity in \eqref{414}, \eqref{416}, and its definition as a primitive of \eqref{417}. 
Consequently also the function $ \beta (U; t, x ) $ defined in \eqref{418-new} 
by the inverse diffeomorphism   is odd, real valued and satisfies 
\be\label{eq:beta-inv}
\beta (U;-t,x) = \beta (U_S;t,x) \, . 
\ee
It follows that the matrix symbol 
$  B(U; \theta, t, x, \xi) {\mathcal I}_2 $,  with 
\be\label{eq:B-inv}
\quad  B(U; \theta, t, x, \xi) = 
 b ( U;  t, \theta, x ) \xi = \frac{\beta (U;t,x) }{1 + \theta \partial_x \beta(U;t,x)   } \xi \, ,
\ee
satisfies the reality, parity preserving and anti-reversibility properties \eqref{311}, \eqref{313}, \eqref{315}. 
Hence, as the flow $ \Omega_{B(U)} (1)  $  defined in \eqref{flow-homogeneous} 
is generated by the vector field $ {\rm Op}^{BW} (i B(U; \cdot )) $, we deduce that 
\be\label{eq:arepaare}
\Omega_{B(U)} (1) {\cal I}_2 \, , \  \Omega_{B(U)} (1)^{-1} {\cal I}_2 
\quad {\rm satisfy} \quad  \eqref{317}, \eqref{318}, \eqref{3110} \, . 
\ee
These elementary properties will be checked in the proof of lemma~\ref{420a} below. 
As a consequence, since  $R'(U;t)$, $R''(U;t)$ satisfy the  conditions \eqref{316}, \eqref{318} and \eqref{319}, 
  lemma \ref{lem:Compos} implies that 
$ \Omega_B (1) R'(U;t) \Omega_B^{-1} (1)$ as well as  $\Omega_B (1) R''(U;t)  $ in \eqref{4118a}
satisfy also
the reality, parity, reversibility properties  \eqref{316}, \eqref{318} and \eqref{319}. 

Similarly, since the
matrix  $ A^{(1) }$ in \eqref{325} satisfies \eqref{311}, \eqref{313} and \eqref{314}, 
the associated paradifferential operator, and thus the operator in
\eqref{4114} obtained by its conjugation with the flow $  \Omega_{B} (1)  $,  
satisfies the reality \eqref{316}, parity preserving \eqref{318} and 
reversibility \eqref{319} properties. 
Thus the same operator in \eqref{4119}
and \eqref{4121} satisfies as well these properties.
We now write \eqref{4121} as the sum of a paradifferential operator and a smoothing remainder 
that satisfy, each of them separately, the properties of reality, parity preserving  and 
reversibility. For  the reversibility property, we  write \eqref{4121} as the sum of the 
paradifferential
operator 
\begin{multline*}
  \frac{1}{2}\Bigl[\opbw\Bigl((\mk(1+\zu(U;t))+\tilde{\lambda}_{\frac{1}{2}}(U;t,\cdot))\Kcal +
  \tilde{\lambda}_1(U;t,\cdot)\Ical_2\Bigr)\\
-S \opbw\Bigl((\mk(1+\zu(U_S;-t))+\tilde{\lambda}_{\frac{1}{2}}(U_S;-t,\cdot))\Kcal +
  \tilde{\lambda}_1(U_S;-t,\cdot)\Ical_2\Bigr)S\Bigr] \\
  = 
\opbw\Bigl( \big( \mk \big(1+ \frac{\zu(U;t) + \zu(U_S;-t) }{2} \big)+
\frac{1}{2}  \big( \tilde{\lambda}_{\frac12}(U;t,\cdot) + 
\tilde{\lambda}_{\frac12}(U_S;-t,\cdot) \big) \big) \Kcal  \\
+  
 \frac12 (\tilde{\lambda}_1(U;t,\cdot) -  \tilde{\lambda}_1(U_S;-t,\cdot)) \Ical_2 \Bigr)
\end{multline*}
whose matrix  symbol satisfies \eqref{314}
and the  smoothing operator $ \frac{1}{2}\bigl[R_2(U;t) -SR_2(U_S;-t)S\bigr]$
that satisfies 
\eqref{319}. 
Then we decompose this operator arguing  in a similar way to
ensure  also  the reality  and the parity conditions. This shows that the
contribution of \eqref{4114} to \eqref{4113} may be written as in \eqref{4112}, with symbols satisfying 
conditions \eqref{419}, \eqref{4110}, \eqref{4111}, that are a translation of \eqref{311}, \eqref{313}, \eqref{314} for the corresponding
matrices, and with a smoothing term  $R'(U;t) $ satisfying \eqref{316}, \eqref{318}, \eqref{319}.

Consider now \eqref{4115} and the remainder $ R(U;t) $ in \eqref{4113}. 
By Proposition~\ref{cong-Dt} we have written 
\be\label{eq:Op-temp} 
 \Omega_{B(U)} (1) D_t \Omega_{B(U)}^{-1} (1) \Ical_2  =  {\rm Op}^{BW}(e_1(U;\cdot)) \Ical_2   + R(U;t)  
\ee
as the sum of a paradifferential and a smoothing operator. 
We first show that  
$$ 
\Omega_{B(U)} (1) D_t \Omega_{B(U)}^{-1} (1) \Ical_2  =  - D_t \Omega_{B(U)} (1) \Omega_{B(U)}^{-1}  (1) \Ical_2 
$$ 
satisfies 
 the reality, parity preserving and reversibility properties  \eqref{316}, \eqref{318} and \eqref{319}. 
 By \eqref{arepaare} and  lemma \ref{lem:Compos} it is sufficient to prove that 
$  D_t \Omega_{B(U)} (1) {\mathcal I}_2  $  satisfies  \eqref{316}, \eqref{318}, \eqref{319}. 
Recalling  the definition of $ \Omega_{B(U)} (1) $ in \eqref{flow-homogeneous} 
this is implied  by 
the fact that 
$ D_t {\rm Op}^{BW}( i B(U;\cdot)) \Ical_2 $ satisfies  \eqref{316}, \eqref{318}, \eqref{319}, namely that 
$ \partial_t  B(U; \theta, t, x, \xi) {\mathcal I}_2  $ satisfies  \eqref{311}, \eqref{313}, \eqref{314}. 
 But this follows from the fact, seen after \eqref{B-inv}, that 
$ B(U; \theta, t, x, \xi) {\mathcal I}_2 $, 
satisfies  \eqref{311}, \eqref{313}, \eqref{315} and the definition of these properties.
Finally we can repeat the previous argument for \eqref{4121} to ensure a decomposition as in \eqref{Op-temp} 
with 
$  {\rm Op}^{BW}(e_1(U;\cdot)) \Ical_2   $ and  $ R(U; t) $
that satisfy separately 
 the reality, parity preserving and reversibility properties \eqref{316}, \eqref{318} and \eqref{319}. 
\end{proof}

\section{Reduction to constant coefficient symbols}\label{sec:42}

In Proposition~\ref{411} we have obtained a diagonal system of the form 
\begin{equation}
  \label{eq:421}
  \bigl(D_t-\opbw(D(U;t, \cdot )) \bigr)V = R'(U;t)V + R''(U;t)U
\end{equation}
where $V = \Phi_U^\star W$, the smoothing operators 
$R'(U;t), R''(U;t)$ are, respectively, in  $\sRM{-\rho+\frac{3}{2}}{K,\rho'+2,1}{N}$ 
and $\RrM{-\rho+\frac{3}{2}}{K,\rho'+2,N}$, and 
where $D(U; t, \cdot)$ is a  diagonal matrix of symbols  of the form
\begin{equation}
  \label{eq:422}
  D(U;t,x,\xi) = (1+\zu(U;t))\mk(\xi)\Kcal + \tilde{\lambda}(U;t,x,\xi)\Kcal + \mu(U;t,x,\xi)\Ical_2
\end{equation}
where $\zu$ is a function of $\sFR{K,\rho'+1,1}{N}$, independent of $x$, the  symbol $\tilde{\lambda}$ is  in  
$\sG{\frac{1}{2}}{K,\rho'+1,1}{N}$ with $\Im  \tilde{\lambda} $ in $\sG{-\frac{1}{2}}{K,\rho'+1,1}{N}$, 
the symbol $\mu$ is in  $\sG{1}{K,\rho'+2,1}{N}$ with
$\Im\mu$ in $\sG{0}{K,\rho'+2,1}{N}$. Moreover the matrix of symbols $D$ satisfies the 
reality, parity preserving and reversibility properties  \eqref{311}, \eqref{313}, \eqref{314}, 
and the operators $R'$, $R''$ satisfy the 
reality, parity  preserving and reversibility properties \eqref{316}, \eqref{318}, \eqref{319}.

\smallskip

Our goal is to perform a new conjugation of the system 
\eqref{421} in order to replace in \eqref{422}, $\tilde{\lambda}$ and $\mu$ by
  \emph{constant coefficients} symbols, up to remainders of very negative order. 
We shall conjugate  \eqref{421} by the flow generated by the linear system
\begin{equation}
  \label{eq:423}
  \begin{split}
    \frac{d}{d\theta}\Omega_{F(U)}(\theta) &= i\opbw(F(U)) \Omega_{F(U)}(\theta)\\
\Omega_{F(U)}(0) &= \mathrm{Id}
  \end{split}
\end{equation}
where the matrix symbol $ F(U) = F(U; t)$ is self-adjoint (if $ F(U) $ has positive order).  
We first provide some properties of such auxiliary flow, proved as in 
lemma \ref{lem:flow}.  

\begin{lemma}\label{420} {\bf(Auxiliary flow)}
Let $0\leq K'\leq K$ be  integers, $ m \leq 1 $,  and let $F$ be a symbol in $\sGM{m}{K,K',1}{N} $,  
modulo a symbol of order zero,  valued in the space of
self-adjoint matrices (when $U$ belongs to $\CKHR{\sigma}{\C^2}$). 

Then there is $\sigma$ in $\R_+$,
and for any $U$ in  $\CKHR{\sigma}{\C^2}$, the system \eqref{423}
has a  unique solution $\of{\theta}$ defined for $\theta\in [-1,1]$. The  linear  operator $\of{\theta}$ is bounded on $\Hds{s}$ for any $s$,
and there is 
$r>0$ and for any $s\in \R$ a constant $C_s>0$ such that, for any $U$ in 
$\Br{K}{}$, any $W$ in $\Hds{s}$, any $\theta$ in $[-1,1]$
\begin{equation}
  \label{eq:425}
  C_s^{-1}\norm{W}_{\Hds{s}} \leq \norm{\oft W}_{\Hds{s}} \leq C_s \norm{W}_{\Hds{s}}.
\end{equation}
 Moreover, for any $k\leq K-K'$, we have
\begin{equation}
  \label{eq:425b}
  \norm{\partial_t^k[\of{\theta}W]}_{\Hds{s-\frac{3}{2}k}}\leq C \Gcals{s}{k,1}{W}  \, . 
\end{equation}
Finally $ \of{- \theta} = \big(\of{\theta} \big)^{-1} $. 
\end{lemma}
{\bf Remark:} 
We   shall apply below this lemma for $ m \leq 1/ 2 $. Clearly  if $ m \leq 0 $ 
the statement follows by just an 
ODE argument (and the matrix $ F(U) $ does  not need to be self-adjoint). 
The property  $ \of{- \theta} = \of{\theta}^{-1} $ holds  because $ F(U)$ is autonomous in the variable $ \theta $.

\smallskip

The flow $\oft$ inherits the following algebraic properties of the operator $\opbw(F(U))$,
where $ S $ is the involution defined in \eqref{def-S} and 
 $\tau$ the map in \eqref{def-tau}. 
 Since  $ F(U) = F(U; t) $ depends on the time
variable $ t $, we write below also 
$ \Omega_{F(U;t,\cdot)}(\theta) $  to make
appear explicitly this time dependence.  

\begin{lemma}\label{420a}
If  $\opbw(F(U))$ satisfies the reality \eqref{316}, resp. parity preserving \eqref{318}, anti-reversibility \eqref{3110},  
property, then the flow operator $\oft$ satisfies, for any $\theta $, 
the anti-reality \eqref{317}, resp. parity preserving \eqref{318}, anti-reversibility \eqref{3110}, 
property i.e. 
\begin{equation}\begin{split}
  \label{eq:424}
  \overline{\oft W} &= S\oft S\overline{W}\\ \oft\circ\tau &= \tau\circ\oft\\ 
\Omega_{F(U;-t,\cdot)}(\theta)  &=S \Omega_{F(U_S;t,\cdot)}(\theta) S \, .
\end{split}\end{equation}
\end{lemma}

\begin{proof}
By assumption $ \opbw(F(U)) $ satisfies respectively
\[
\begin{split}
\overline{\opbw(F(U))W} & = -S\opbw (F(U))S\overline{W} \\
\opbw(F(U))\circ\tau & = \tau\circ\opbw(F(U)) \\
\opbw(F(U;-t,\cdot)) & = S\opbw(F(U_S;t,\cdot))S \, .
\end{split}
\]
Then $\oftt$ defined respectively by
\[
\begin{split}
  \oftt W &= \overline{\oft W} -S\oft S\overline{W}\\
\oftt &= \oft\circ \tau - \tau\circ\oft\\
\tilde{\Omega}_{F(U;t,\cdot)}(\theta) &= \Omega_{F(U;-t,\cdot)}(\theta) -S \Omega_{F(U_S;t,\cdot)}(\theta) S
\end{split}\]
satisfies respectively 
\[
\begin{split}
 \frac{d}{d\theta}\oftt & = iS\opbw(F(U))S\oftt \\
 \frac{d}{d\theta}\oftt  & = i\opbw(F(U))\oftt \\
\frac{d}{d\theta} \tilde{\Omega}_{F(U;t,\cdot)}(\theta) & = iS\opbw(F(U_S; t))S \tilde{\Omega}_{F(U;t,\cdot)}(\theta)
\end{split}\]
with zero initial condition. 
Consequently $ \oftt = 0 $ for all $ \theta $ and  \eqref{424} follows. 
\end{proof}
Let us write down the 
analogous properties 
satisfied by the paracomposition operator  $ \Phi_U^\star = \Omega_{B(U)} (1) {\cal I}_2 $ defined in
Theorem~\ref{thm:para}.  Since the change of variables $\Phi_U$ depends on the time
variable $t$, we shall write below $ \Phi_U(t;\cdot)^\star = \Omega_{B(U;t,\cdot)} (1) {\cal I}_2 $ instead of  $
\Phi_U^\star$  when we want to make
appear explicitly this time dependence. 

\begin{lemma}\label{lem:flowp}
We have 
\begin{equation}\begin{split}
  \label{eq:424a}
 &  \overline{\Phi_U(t;\cdot)^\star W} = \Phi_U(t;\cdot)^\star\overline{W} =
  \Phi_{U_S}(-t;\cdot)^\star\overline{W} \\
    & \Phi_U(t;\cdot)^\star \circ\tau = \tau\circ \Phi_U(t;\cdot)^\star \\
  & S \Phi_U(t;\cdot)^\star = \Phi_U(t;\cdot)^\star S = \Phi_{U_S}(-t;\cdot)^\star S \, .
\end{split}\end{equation}
\end{lemma}
\begin{proof}
Recall that $  \Omega_{B(U; t,\cdot)}( \theta ) $ is the scalar flow defined in \eqref{flow-homogeneous}.
Thus  $  \Omega_{B(U;t, \cdot)}( \theta ) {\cal I}_2 $, in particular $ \Phi_U(t;\cdot)^\star $, 
commutes with $ S = - \bigl[\begin{smallmatrix}0&1\\1&0\end{smallmatrix}\bigr] $.
The operator $ \opbw{(B(U; t,\cdot ))} {\cal I}_2 $ 
satisfies
\eqref{316},  \eqref{318}, \eqref{3110}, because
$ B(U;t, \cdot ) {\cal I}_2 $ satisfies the 
reality, parity preserving and antireversibility properties \eqref{311}, \eqref{313}, \eqref{315}
as proved after \eqref{B-inv}.  
Thus  the first identity  \eqref{424} (with $F$ replaced by $B$) and the fact that 
$\Omega_{B(U;t,\cdot)} (\theta)  {\cal I}_2 $  commutes with $S$, imply 
$  \overline{\Phi_U(t;\cdot)^\star W} = \Phi_U(t;\cdot)^\star\overline{W} $.
The second identity in \eqref{424} directly proves $ \Phi_U(t;\cdot)^\star \circ\tau = \tau\circ \Phi_U(t;\cdot)^\star $.  
Finally $ \Phi_U(t;\cdot)^\star = \Phi_{U_S}(-t;\cdot)^\star $ follows by the third identity in \eqref{424} 
and the fact that $ \Phi_U(t;\cdot)^\star $  commutes with $ S $,
or, more directly,  by 
$B(U;-t,\cdot) = B(U_S;t,\cdot) $ according to \eqref{beta-inv} and \eqref{B-inv}.
 All the identities in \eqref{424a} are proved. 
\end{proof}

The main result of this section is the following one:

\begin{proposition}\label{421} {\bf (Reduction  to  constant coefficients of \eqref{421})}
For any integers $N$, $\rho$ in $\N$, we denote by $\rho'$ the 
number introduced at the beginning of section~\ref{sec:41} (and in Proposition \ref{411}).  Set 
\be\label{eq:cond}\begin{split}
\ku &= 
\textrm{integer part of }
\rho'+2\rho+\frac{3}{2}N+2 \, ,\\ \quad L &= \max(2\rho-1,N-1) \, , 
\end{split}\ee
and take $K\geq\ku$. Let  $U \in \CKHR{\sigma}{\C^2}$ be in 
$\Br{K}{}$ for $ r $ small enough provided by lemma \ref{420}. There are\\
$\bullet$ A family of  diagonal 
matrices of symbols $F_\ell(U)$ in $\sGM{\frac{1}{2}-\frac{\ell}{2}}{K,\ku,1}{N}$, $\ell = 0,\dots,L$, 
with $\Im F_0=0$, satisfying the reality \eqref{311}, parity preserving \eqref{313} and 
anti-reversibility \eqref{315} properties, \\
$\bullet$ Smoothing operators $R_1(U;t)$ in the space $\sRM{-\rho+\frac{3}{2}}{K,\ku,1}{N}$, $R_2(U;t)$ in 
 $\RrM{-\rho+\frac{3}{2}}{K,\ku,N}$, 
satisfying the reality \eqref{316}, parity preserving \eqref{318}, reversibility \eqref{319} properties,  \\
$\bullet$ A diagonal matrix $H(U;t,\xi)$ of symbols  in $\sG{1}{K,\ku,1}{N}$, 
independent  of 
$ x $,  with $\Im H$ in $\sGM{0}{K,\ku,1}{N}$, 
satisfying the reality \eqref{311}, parity preserving \eqref{313} and reversibility \eqref{314} properties, 

such that, if $V$ solves the system \eqref{421}, then $\tilde{V} = \of{-1}V$ where 
$ F = \sum_0^L F_\ell $, solves the system
\begin{multline}
  \label{eq:426}
\bigl(D_t - \opbw\bigl((1+\zu(U;t))\mk(\xi)\Kcal - H(U;t,\xi)\bigr)\bigr)\tilde{V}\\
= R_1(U;t)\tilde{V} + R_2(U;t)U.
\end{multline}
\end{proposition}

\medskip

\noindent
\textbf{Remark}:  Actually the  constant coefficient symbols 
of the diagonal matrix $H(U;t,\xi)$  are  in $\sG{1/2}{K,\ku,1}{N} $ and not just in $\sG{1}{K,\ku,1}{N} $. 
Indeed the potential  order one contribution to $H(U;t,\xi)$ will be given by the $ x $-average of the matrix 
symbol   $ \mu (U; t,x, \xi) {\cal I}_2 $ in \eqref{4112}.  
Since the imaginary part of $\mu$ is of order zero, we have to
cope just with the average in $x$ of the real part. But by the first (resp. second) condition \eqref{4111}, the average of
$\Re \mu$ is an odd (resp. even) function of $\xi$, so has to vanish. 
This remark is not necessary for the subsequent arguments but it provides
the expected asymptotic expansion of the Floquet exponents in the periodic and quasi-periodic case 
\cite{AB}, \cite{BM}.

\medskip

The proof of Proposition \ref{421} is based on an iterative algorithm which replaces the variable coefficient 
diagonal symbols $ D(U; t, x, \xi ) $ in \eqref{422} into constant coefficient symbols up to terms 
of very negative order  whose associated paradifferential operator
 may be incorporated to the smoothing remainders $ R_1 (U; t) $. 
At each step of the iteration we get constant coefficient symbols up to $ 1/2 $ smoother ones. Thus we shall 
perform  $ L \sim 2 \rho $ transformations, see \eqref{cond}. 

In order to prove \eqref{426} we 
study separately  in the next two lemmas the conjugation of $ \opbw (D(U;  \cdot )) $ and $ D_t $, respectively, 
under the flow map $ \of{ 1} $ generated by \eqref{423}. For  simplicity of 
notation in the sequel  we neglect to write the explicit $ t $ dependence of $ D(U; \cdot ) = D(U; t,   \cdot ) $. 
\begin{lemma}
  \label{422}
Let $r>0$, $N\in \N$, $\rho\in \N$ be given. Let $K' =\rho'+2$, $L=\max(2\rho-1,N-1)$ and assume $L\leq K-K'$. 
Consider a
family of diagonal matrices of symbols $F_\ell$ in $\sGM{\frac{1}{2}-\frac{\ell}{2}}{K,K'+\ell,1}{N}$, $\ell = 0,\dots,L $, 
with $\Im F_0 = 0 $, 
satisfying \eqref{311},
\eqref{313} and \eqref{315}. Let $D(U;  \cdot)$ be the diagonal matrix \eqref{422}.

Then there are diagonal matrices $G_\ell$  of symbols in 
$\sG{1-\frac{\ell}{2}}{K,K'+\ell,1}{N}$, $\ell = 1,\dots,L$,  with 
$\Im G_1 \equiv 0$,  satisfying \eqref{311}, \eqref{313}
and \eqref{314}, $G_\ell$ depending only on $F_{\ell'}$, $\ell'<\ell$, such that, setting 
$ F = \sum_{0}^LF_\ell$, we have 
\begin{multline}
  \label{eq:427}
\of{-1}\opbw(D(U;  \cdot))\of{1}\\
=\opbw(D(U;  \cdot)) + \sum_{\ell=1}^L\opbw\bigl[G_\ell(U;  \cdot) + (1+\zu(U;t))\absp{F_\ell(U;\cdot),\mk(\xi)\Kcal}\bigr]\\
+ \opbw((1+\zu(U;t))\absp{F_0(U;\cdot),\mk(\xi)})\Kcal + R(U;t)
\end{multline} 
where $R(U;t)$ is an operator in $\sRM{-\rho+\frac{3}{2}}{K,K'+L,1}{N}$, satisfying the properties 
 \eqref{316}, \eqref{318} and \eqref{319}.
\end{lemma}
\begin{proof}
  Since
  \begin{multline*}
    \frac{d}{d\theta}\Bigl( \of{-\theta}\opbw(D(U;  \cdot))\oft\Bigr) \\
= -i\of{-\theta}\bigl[\opbw(F(U)),\opbw(D(U;  \cdot ))\bigr]\oft
  \end{multline*}
we may write, iterating this identity, and applying Taylor formula
\begin{multline}
  \label{eq:428}
\of{-1}\opbw(D(U;  \cdot))\of{1}\\
= \opbw(D(U;  \cdot)) + \sum_{q=1}^L\frac{(-i)^q}{q!}\adf^q\opbw(D(U;  \cdot))\\
+\frac{(-i)^{L+1}}{L!}\int_0^1\of{-\theta}\adf^{L+1}\opbw(D(U;  \cdot))\of{\theta}\\\times(1-\theta)^L\,d\theta
\end{multline}
with the  notation  
\index{am@$\adf$ (Commutator)} 
$$
\adf B = [\opbw(F(U)),B] \, . 
$$ 
Notice that since $F$ is a diagonal matrix in
$\sGM{\frac{1}{2}}{K,K'+L,1}{N}$, each commutator $ [\opbw(F(U)), \cdot \, ]  $ gains 
$ 1/ 2$ unit on the order of the operators and one order of vanishing as $U$
goes to zero. It follows that \eqref{428} is an expansion in  operators with decreasing orders
and increasing degree of homogeneity. More precisely,  
according to the composition result of Proposition~\ref{231}, 
formulas \eqref{222}, \eqref{223}, and Proposition \ref{232}
(see also lemma \ref{lem:comm}), where we replace the smoothing index $\rho$ by some $\tilde \rho $ to be chosen below, 
we may write
$$
\adf^q\opbw(D(U;  \cdot )) = \opbw(C_q (U;  \cdot))+ R_q (U;t)
$$
where $ C_q (U; \cdot ) $ is a diagonal matrix of symbols  in $ \sG{\frac{3}{2}-\frac{q}{2}}{K,K'+L,q}{N}$ 
and $ R_q $ is a diagonal  smoothing operator in 
$$ 
\sRM{- \tilde \rho +\frac{3}{2} + \frac{q}{2}}{K,K'+L,q}{N} 
$$ 
for $\tilde \rho$ as large as we want. If the level $ L $ at which we
stop the Taylor expansion \eqref{428} is large enough, namely $L\geq 2\rho-1$ and $L\geq N-1$, then,
the last remark following Proposition \ref{215}
implies that  the operator 
$$ 
\opbw{(C_{L+1} (U;  \cdot ))} \in  \RrM{-\rho+\frac{3}{2}}{K,K'+L,N}   \, . 
$$
The same is true for the
$R_q$'s for $q=1,\dots,L+1$ if we take $\tilde \rho \geq\rho +\frac{L+1}{2}$. 
Therefore 
$$
\adf^{L+1}\opbw(D(U;  \cdot ))  \in \RrM{-\rho+\frac{3}{2}}{K,K'+L,N} \, .
$$
Using also  \eqref{425b},  if follows that the integral term in \eqref{428} 
is a smoothing operator in $\RrM{-\rho+\frac{3}{2}}{K,K'+L,N}$, so contributes to the last operator $ R $ in \eqref{427}.

Consider now the general term in the sum \eqref{428}. This is a combination of operators of the
form
\begin{equation}
  \label{eq:429}
  i^q\Bigl[\opbw(F_{\ell_1}),\Bigl[\opbw(F_{\ell_2}),\Bigl[\cdots\Bigl[\opbw(F_{\ell_q}),\opbw(D)\Bigr]\cdots\Bigr]\Bigr]\Bigr]
\end{equation}
with $ 1 \leq q \leq L $, $ 0 \leq \ell_1, \ldots, \ell_q \leq L $.  Set $ l = (\ell_1, \ldots, \ell_q ) $. 
Again by Proposition \ref{231}, formulas \eqref{222}, \eqref{223} and Proposition \ref{232}, 
each  operator in \eqref{429} may be written as $ \opbw(G_{q, l }(U;  \cdot)) + R_{q, l }(U;t) $, 
where $ G_{q, l } $ is a diagonal matrix of symbols in
$$
\sG{\frac{3}{2}-\frac{1}{2}\bigl(\ell_1+\cdots+\ell_q\bigr)-\frac{q}{2}}{K,K'+ \max(\ell_j),q}{N}
$$
and $ R_{q, l } $ is a diagonal matrix in $\sRM{-\rho+\frac{3}{2}}{K,K'+ \max(\ell_j),q}{N}$
(using again that $ \tilde \rho \geq\rho +\frac{L+1}{2} $). 
Notice that all the symbols $ G_{q, l}  $ such that 
$$
\ell_1+\cdots+\ell_q +q =\ell+1
$$ have order $ 1- \frac{\ell}{2} $, and that, if $\ell_1+\cdots+\ell_q +q
=\ell+1$, then $\ell_1,\dots,\ell_q$ are strictly smaller than $\ell$, except if $q=1$, $\ell_1=\ell$. 
We define the symbol  $G_\ell$ in the right hand side of \eqref{427}, 
 for $ \ell = 1, \ldots, L $,  as the sum of the $ G_{q, l} $
 for which $\ell_1+\cdots+\ell_q +q
=\ell+1$ and   $\ell_1,\dots,\ell_q<\ell$. The
remaining symbol of order $1- \frac{\ell}{2} $ is 
\begin{equation}
  \label{eq:4210}
  i[\opbw(F_\ell),\opbw(D)] \, .
\end{equation}
Replace in \eqref{4210} $D$ by its expression \eqref{422}. By 
 Proposition \ref{231} and formulas \eqref{222}, \eqref{223}, the commutator 
$ i[\opbw(F_\ell), \opbw(\tilde{\lambda}\Kcal +\mu\Ical_2)] $
has order $\frac{1}{2}-\frac{\ell}{2} = 1 - \frac{\ell+1}{2} $, up to a smoothing remainder, 
and so it may be incorporated to $G_{\ell+1}$. 
The remaining operator  
\be\label{eq:comfe}
[i\opbw(F_\ell),\opbw\bigl((1+\zu(U;t))\mk(\xi)\bigr)\Kcal] 
\ee
may be written as the sum of 
\[ \opbw((1+\zu(U;t))\absp{F_\ell,\mk(\xi)}\Kcal) \, , \] 
 a contribution to $\opbw(G_{\ell+2})$, 
and a smoothing
operator $R(U; t)$ as above. The operator in \eqref{comfe}
(of order $ 1 - \frac{\ell}{2} $)
is the term written in the right hand side of \eqref{427}, where we have distinguished the cases $ \ell = 0 $ and 
$ \ell =1, \ldots, L  $.

Finally, the other operators $ \opbw{(G_{q, l})} $ of order $ 1- \frac{\ell}{2} $
with $ \ell \geq L + 1 $, 
may be incorporated as above into  the remainder $ R $ in  \eqref{427}.  

Let us verify that $\Im G_1 =0$. Remark that the symbol $G_1$ is the sum of the $ G_{q, l}$
such that  $\ell_1+\cdots+\ell_q + q = 2$ and  $ \ell_1, \ldots , \ell_q < 1 $. 
Hence $ q =  2 $ and $\ell_1 = \ell_2 = 0 $ (the case $q=1, \ell_1=1$ is not allowed).  
So we just have to consider
\[-[\opbw(F_0),[\opbw(F_0),\opbw(D)]].\]
Replacing $ D $ by its expression \eqref{422}, we  get
 the operator  of nonpositive order $ -[\opbw(F_0),[\opbw(F_0),\opbw(\tilde{\lambda}\Kcal +\mu\Ical_2)]] $, that contribute to
$G_\ell$, $\ell\geq 2 $,  and the operator
\[-[\opbw(F_0),[\opbw(F_0),\opbw\bigl((1+\zu(U;t))\mk(\xi)\Kcal\bigr)]].\]
The principal part of this operator has symbol $\absp{F_0,\absp{F_0,\mk}}(1+\zu)\Kcal$ which is a diagonal matrix with real
entries since $\mk$, $F_0$, $\zu$ are real valued. This shows that $\Im G_1=0$.

We still have to check that $G_\ell$ satisfies \eqref{311}, \eqref{313}, \eqref{314} and that $R(U;t)$ obeys \eqref{316},
\eqref{318}, \eqref{319}. It suffices to see that these last three properties are verified by the left hand side of
\eqref{427}. As $F_\ell$ satisfies \eqref{311}, \eqref{313}, \eqref{315} for any $\ell$, $\opbw(F_\ell )$ satisfies
\eqref{316}, \eqref{318}, \eqref{3110}. Hence lemma \ref{420a} implies that 
$\of{\pm\theta}$ obeys \eqref{317}, \eqref{318}, \eqref{3110}, and, 
since $\opbw(D)$ satisfies \eqref{316}, \eqref{318}, \eqref{319}, lemma \ref{lem:Compos}
implies that the left hand side of \eqref{427} satisfies the same properties.  
\end{proof}

\begin{lemma}
  \label{423}
Let $F_\ell$, $ \ell = 0, \ldots, L $, be a family of diagonal matrices of symbols 
as in lemma~\ref{422}. Then there are diagonal matrices of symbols $G_\ell$, $\ell=1,\dots,L$ as in the statement of the
preceding lemma such that
\begin{equation}
  \label{eq:4211}
  \of{-1}D_t\of{1} = D_t +\sum_{\ell=1}^L\opbw(G_\ell(U;  \cdot )) + R(U;t)
\end{equation}
where $R(U;t)$ is in $\sRM{-\rho+\frac{3}{2}}{K,K'+L+1,1}{N}$ and satisfies 
the reality \eqref{316}, parity preserving \eqref{318} and reversibility \eqref{319} properties.
\end{lemma}
\begin{proof}
We have 
  \begin{multline*}
    \frac{d}{d\theta}\Bigl( \of{-\theta} \circ D_t  \circ \oft\Bigr) \\
= -i\of{-\theta}\bigl[\opbw(F(U)), D_t \bigr]\oft \, . 
  \end{multline*}
Since  the commutator
$ [\opbw(F(U)), D_t \bigr] = - \opbw(D_t F(U))  $, we derive,
iterating the above identity and  applying  Taylor formula, the expansion
  \begin{multline}
    \label{eq:4212}
\of{-1}D_t\of{1} = D_t +\sum_{q=1}^L\frac{(-i)^q}{q!}\adf^q D_t\\
+\frac{(-i)^{L+1}}{L!}\int_0^1\of{-\theta}\bigl(\adf^{L+1} D_t\bigr)\oft (1-\theta)^L\,d\theta\\
= D_t -\sum_{q=1}^L\frac{(-i)^q}{q!}\adf^{q-1}\opbw(D_tF(U))\\
-\frac{(-i)^{L+1}}{L!}\int_0^1\of{-\theta}\adf^{L}\opbw(D_tF(U))\oft\\\times (1-\theta)^L\,d\theta 
  \end{multline}
with the convention  $  \adf^0 = {\rm Id} $.   

Since $F_\ell$ is a diagonal matrix of symbols in $\sGM{\frac{1}{2}-\frac{\ell}{2}}{K,K'+\ell,1}{N}$, and 
since $U$ solves
equation \eqref{4312}, it follows from lemma~\ref{217} that 
$D_tF_\ell$ is a diagonal matrix with entries in $\sG{\frac{1}{2}-\frac{\ell}{2}}{K,K'+\ell+1,1}{N}$. 
The general term of the sum in the right hand side
of \eqref{4212} is the combination of operators of the form 
$$
  i^{q+1}\Bigl[\opbw(F_{\ell_1}),\Bigl[\opbw(F_{\ell_2}),\Bigl[\cdots\Bigl[\opbw(F_{\ell_q}),\opbw(D_tF_{\ell_{q+1}})\Bigr]\cdots\Bigr]\Bigr]\Bigr]
  $$
with $ 0 \leq q\leq L-1$  and $ 0 \leq \ell_1, \ldots, \ell_{q+1} \leq L $.
Set $ l = (\ell_1, \ldots, \ell_{q+1} ) $. 
By Proposition \ref{231}, formulas \eqref{222}, \eqref{223} and Proposition \ref{232} (applied with the smoothing index
$\rho$ replaced by some $\tilde \rho$ to be chosen), each of these operators has the 
form  $\opbw(G_{q, l}(U;  \cdot))+R_{q,l}(U;t)$, where $G_{q,l}$ 
is a diagonal matrix of symbols belonging to
$$
\sG{\frac{1}{2}-\frac{1}{2}(\ell_1+\cdots+\ell_{q+1})-\frac{q}{2}}{K,K'+\max(\ell_j)+1,q+1}{N}
$$ 
and $R(U;t)$ is in 
$\sRM{-\tilde \rho+\frac{q+1}{2}-\frac{1}{2}(\ell_1+\cdots+\ell_{q+1})}{K,K'+\max(\ell_j)+1,q+1}{N}$. 
Choosing $\tilde \rho-\rho$ large
enough (depending on $ L $), we get that $R(U;t)$ is in 
$$
\sRM{-\rho +\frac{3}{2}}{K,K'+\max(\ell_j)+1,q+1}{N} \, . 
$$ 
For $ \ell = 1, \ldots , L $, we  denote by $G_\ell $,   the 
sum of the symbols $ G_{q, l} $ such that  
$\ell_1+\dots+\ell_{q+1}+q+1 = \ell $. Thus 
the symbol $G_\ell $ is in $\sGM{1-\frac{\ell}{2}}{K,K'+\ell,1}{N}$ and, 
since all the $ \ell_j $ are strictly smaller than $\ell$, 
$ G_\ell $ depends only on  $F_{\ell'}$, $\ell'<\ell $.

The other operators $ \opbw{(G_{q, l})} $ of order $ 1- \frac{\ell}{2} $
with $ \ell \geq L + 1 $, 
may be incorporated, by the last remark following Proposition \ref{215},  into  the remainder $ R $ in  \eqref{4211}.  
Finally, arguing in the same way, 
as we  fixed $L$ large enough ($L\geq2\rho-1$, $ L \geq N - 1 $), 
the integral term in
\eqref{4212} gives a contribution belonging to 
$ \RrM{-\rho+\frac{3}{2}}{K,K'+L+1,N} $.

Let us check that
\[\begin{split}i^q\adf^{q-1}\opbw(D_tF(U))\\ i^{L+1}\of{-\theta}\adf^L(\opbw(D_tF))\oft\end{split}\]
satisfy \eqref{316}, \eqref{318}, \eqref{319}.

Notice first that 
if $F$ satisfies \eqref{315}, it follows immediately that $D_tF$ satisfies 
\eqref{314}. In the same way, we see that if $F$ satisfies \eqref{311}, $D_tF$ satisfies \eqref{312}. Consequently,
$\opbw(D_tF)$ satisfies \eqref{317}, \eqref{318} and \eqref{319}. Moreover, by assumption, $\opbw(iF)$ satisfies \eqref{317},
\eqref{318}, \eqref{3110}. Remark that if an operator $M$ satisfies \eqref{317} (resp.\ \eqref{318}, resp.\ \eqref{319}) and
an operator $M'$ satisfies \eqref{317} (resp. \eqref{318}, resp.\ \eqref{3110}), then $[M,M']$ satisfies \eqref{317} (resp.\
\eqref{318}, resp.\ \eqref{319}), so that the general term of the sum in the right hand side of \eqref{4212}, or the term
below the integral, satisfies \eqref{316} -- taking into account the extra power of $i$ that is present -- (resp.\
\eqref{318}, \eqref{319}). We are thus left with showing that these properties are preserved through conjugation by
$\of{\theta}$ in the integral term. This follows from the fact that $\of{\theta}$ satisfies \eqref{317}, \eqref{318} and
\eqref{3110}, as it has already be seen.

Finally, let us check that $\Im G_1 = 0 $. Actually, $G_1$ is the sum 
of the symbols $ G_{q, l } $ 
such that $ \ell_1+\cdots+\ell_{q+1} + q = 0 $, so that $q=0$, $\ell_1=0 $, 
i.e.\ 
$ G_1 =  i D_t F_0 = \partial_tF_0 $. Since, by assumption $ F_0 $ is real, so is $ G_1 $. 
This concludes the proof.
\end{proof}
\begin{proof1}{Proof of Proposition~\ref{421}}
We set as in lemma~\ref{422}, $K'=\rho'+2$, $L = \max(2\rho-1,N-1)$ and 
fix $K\geq K'+L$. 

For some $F = \sum_0^L F_\ell$  to be determined, satisfying the assumptions of lemma~\ref{422}, we set
$\tilde{V} = \of{-1}V$. Conjugating \eqref{421} with the flow $ \of{-1} $, we get
\begin{multline}
  \label{eq:4215}
\of{-1}\bigl(D_t-\opbw(D(U;  \cdot))\bigr)\of{1}\tilde{V}\\
=\of{-1}R'(U; t)\of{1}\tilde{V} + \of{-1}R''(U;t)U.
\end{multline}
According to lemmas~\ref{422} and~\ref{423}, and  recalling the definition   of $ D(U;  \cdot ) $ in \eqref{422}, we 
may write \eqref{4215} as
\begin{multline}
  \label{eq:4216}
\bigl[D_t - \opbw((1+\zu(U;t))\mk(\xi))\Kcal \\- \sum_{\ell=0}^L\opbw\bigl(G_\ell +
(1+\zu(U; t))\absp{F_\ell,\mk(\xi)\Kcal}\bigr)\Bigr]\tilde{V}\\
= \of{-1}R'(U;t)\of{1}\tilde{V} + \of{-1}R''(U;t)U + R(U;t)\tilde{V}
\end{multline}
where:\\
$\bullet$ $R'(U;t)$ is in $\sRM{-\rho+\frac{3}{2}}{K,K',1}{N}$, $R''(U;t)$ is in $\RrM{-\rho+\frac{3}{2}}{K,K',N}$, $R(U;t)$ is in $\sRM{-\rho+\frac{3}{2}}{K,K'+L+1,1}{N}$,\\
$\bullet$ $G_0 = \Re\mu(U;t,x,\xi)\Ical_2$ in \eqref{422}, so that $G_0$ is in
$\sGM{1}{K,K',1}{N}$, and satisfies $\Im G_0 = 0$,\\
$\bullet$ $G_1$ is the sum of the diagonal matrices of symbols denoted by that letter in \eqref{427}, \eqref{4211} plus  
$ \Re \tilde{\lambda}(U;t,x,\xi)  \Kcal $  coming from  \eqref{422}.
lemmas~\ref{422} and~\ref{423}, and  the properties
of $\tilde{\lambda}$ in Proposition \ref{411}, imply that $G_1$ is a diagonal matrix
of symbols in  $\sGM{\frac12}{K,K'+1,1}{N}$  with  $\Im  G_1(U;t,x,\xi) =0$. \\
$\bullet$ $G_2$ is made from the contributions denoted by this letter in \eqref{427}, \eqref{4211}, plus the symbol
$\Im\mu(U;t,x,\xi)\Ical_2$ coming from \eqref{422}. This is a diagonal matrix of symbols in $\sGM{0}{K,K'+2,1}{N}$.\\
$\bullet$ $G_3$ is made of the similar contributions coming 
from \eqref{427}, \eqref{4211} plus the symbol 
$i \Im \tilde{\lambda}(U;t,x,\xi) \Kcal $ that we discarded from $G_1$. This is a diagonal matrix of symbols in
$\sGM{-\frac{1}{2}}{K,K'+3,1}{N}$.\\
$\bullet$ Finally, $G_\ell$, $\ell\geq 4$ come from the corresponding terms in \eqref{427}, \eqref{4211}.

All these symbols satisfy the reality, parity-preserving and reversibility properties 
\eqref{311}, \eqref{313}, \eqref{314}, as follows from lemmas~\ref{422} and~\ref{423} and the
properties of the matrix  $D$ in \eqref{422}.

We construct now $F_\ell$ so that $G_\ell+ (1+\zu(U;t))\absp{F_\ell,\mk(\xi)\Kcal}$ has constant coefficients, which are
moreover real valued when $\ell = 0, 1$, i.e.\ when the order of the symbol is positive. The diagonal matrix $G_\ell$ may be
written as
\[ G_\ell = g'_\ell(U;t,x,\xi)\Ical_2 + g''_\ell(U;t,x,\xi)\Kcal\]
with $g'_\ell, g''_\ell$ in $\sG{1-\frac{\ell}{2}}{K,K'+\ell,1}{N}$. Moreover, since $ G_0 $ and $ G_ 1 $ are real,   
$$
\Im g'_0, \Im g''_0 = 0 \, , \quad 
\Im g'_1, \Im g''_1 = 0 \, . 
$$ 
We may also, up to a modification of the smoothing operators, assume that $g'_\ell, g''_\ell$
vanish for $\abs{\xi}\leq \frac{1}{2}$. We decompose 
$$ 
G_\ell = G_{\ell}^D + G_{\ell}^{ND}  \quad {\rm where} 
\quad  G_\ell^{\mathrm{D}} = \gpd{\ell}\Ical_2+\gsd{\ell}\Kcal
$$ 
and 
\[
\begin{split}
 \gpd{\ell}(U;t,\xi) = \frac{1}{2\pi}\int_\Tu g'_\ell(U;t,x,\xi)dx \ ,  \  &  \ \gsd{\ell}(U;t,\xi) = \frac{1}{2\pi}\int_\Tu
  g''_\ell(U;t,x,\xi)dx \\
 \gpnd{\ell} = g'_\ell - \gpd{\ell} \, ,  \ & \ \gsnd{\ell} = g''_\ell - \gsd{\ell} \, . 
\end{split}\]
To eliminate the variable coefficients part in the left hand side of \eqref{4216}, we need to find
matrices of symbols 
$$
F_\ell = f'_\ell (U;t,x,\xi) \Kcal + f''_\ell (U;t,x,\xi) \Ical_2
$$ 
such that
\[G_\ell^{\mathrm{ND}} + (1+\zu(U;t))\absp{F_\ell,\mk(\xi)\Kcal} = 0 \, , \]
i.e. expanding the Poisson bracket 
\begin{equation}
  \label{eq:4217}
  \begin{split}
    \gpnd{\ell}(U;t,x,\xi) - 
  (1+\zu(U;t))m'_\kappa(\xi)\frac{\partial f'_\ell}{\partial x}(U;t,x,\xi) = 0\\
\gsnd{\ell}(U;t,x,\xi) - 
  (1+\zu(U;t))m'_\kappa(\xi)\frac{\partial f''_\ell}{\partial x}(U;t,x,\xi) = 0.
\end{split}
\end{equation}
As $\gpnd{\ell}, \gsnd{\ell}$ have zero $x$-average, and  $1+\zu(U;t)\geq \frac{1}{2}$ if $U$ is small
enough, we may find a unique
pair $(f'_\ell, f''_\ell)$ of functions with zero $x$-average, solving these equations. Notice that $m'_\kappa$ vanishes
close to zero, but we may always truncate $ \gpnd{\ell}, \gsnd{\ell}$ outside a neighborhood of $\xi=0$, as smoothing symbols
contribute to the right hand side of \eqref{426}. Since
$g'_0, g''_0$ are real valued, we get that $f'_0, f''_0$, and thus $ F_0 $, are real valued. As $m'_\kappa$ is elliptic of order
$1/2$, we get that $f'_\ell, f''_\ell $, and thus $ F_\ell $, belong to $\sG{\frac{1}{2}-\frac{\ell}{2}}{K,K'+\ell,1}{N}$. 
We have obtained symbols $F_\ell$ such that the only remaining terms in the sum in the left hand side of
\eqref{4216} is
\[-\Bigl(\sum_{\ell=0}^L \opbw\bigl(\gpd{\ell}\Ical_2+\gsd{\ell}\Kcal\bigr)\Bigr)\tilde{V}.\]
This contributes to $\opbw(H)\tilde{V}$ in  \eqref{426}, since $\gpd{\ell}, \gsd{\ell}$ have constant coefficients by construction, and since
their imaginary part vanishes if they are of positive order, i.e.\ when $\ell = 0, 1$.
Actually,  since 
$ G_0 = \Re \mu (U; t,x, \xi) {\cal I}_2 $ and $ \mu (U; t,x, \xi) $ satisfies 
\eqref{4111} we deduce that its $ x $-average vanishes at the order $ 1 $.
This proves the remark stated after Proposition \ref{421}.

We are left with proving that \eqref{311}, \eqref{313}, \eqref{315} hold for $F_\ell$, and that the smoothing terms may be
written as in \eqref{426}. Notice that since $G_\ell$ satisfies \eqref{311}, we have
\begin{equation}
  \label{eq:4218}
  \overline{(G_\ell^{\mathrm{ND}})}^\vee = -S (G_\ell^{\mathrm{ND}})S \, .
\end{equation}
Since \eqref{4217} may be written as 
$$
G_\ell^{\mathrm{ND}}(U; t, \cdot) = (1+\zu(U;t))m'_\kappa(\xi)\Kcal \frac{\partial
  F_\ell}{\partial x} \, , 
  $$ 
  it follows from \eqref{4218}, the fact that $m'_\kappa(\xi)$ is odd and the relation $S\Kcal S =
-\Kcal$ that \eqref{311} holds for $F_\ell$. In the same way, condition \eqref{313} for $G_\ell^{\mathrm{ND}}$, together
with oddness of $m'_\kappa(\xi)$ implies that \[\bigl(\frac{\partial F_\ell}{\partial x}\bigr)(U;t,-x,-\xi) =
-\bigl(\frac{\partial F_\ell}{\partial x}\bigr)(U;t,x,\xi)\] from which  the property \eqref{313} for $F_\ell$ follows 
by integration. Finally, since
$G_\ell^{\mathrm{ND}}$ satisfies \eqref{314}, it follows from \eqref{419}, $ \Kcal^2 = I $, $ \Kcal S \Kcal = -S $
that $\frac{\partial  F_\ell}{\partial x}(U;-t,\cdot) = S \frac{\partial F_\ell}{\partial x}(U_S;t,\cdot)S$, so that $F_\ell$ satisfies \eqref{315}.

Consider now the smoothing terms in the right hand side of \eqref{4216}. Since 
$R'(U;t), R''(U;t), R(U;t)$ satisfy 
the  reality, parity preserving and reversibility properties 
\eqref{316}, \eqref{318}, \eqref{319}, and since the operator $\oft$ satisfies 
 \eqref{317},  \eqref{318} and  \eqref{3110}, we conclude that the
operators in the right hand side of \eqref{4216} satisfy as well \eqref{316}, \eqref{318} and \eqref{319}. 

Moreover, 
since $R(U;t)$ is in 
$\sRM{-\rho+\frac{3}{2}}{K,K'+L+1,1}{N}$ and $\ku\geq K'+L+1$, as follows by 
 the definition of $\ku$ in \eqref{cond}
 (taking into account also the translation on $\rho$ that we shall perform below),
we conclude that  $R(U;t) $ is in  $\sRM{-\rho+\frac{3}{2}}{K, \ku ,1}{N}$ and so it contributes to $R_1(U;t) $ in \eqref{426}.  
 
 In addition $ R''(U;t) $ is in $ \RrM{-\rho+\frac{3}{2}}{K,K',N} $, $ K' \leq \ku $, and, by \eqref{425b}
 and  recalling the Definition \ref{214}-($ii$), 
we deduce that  $ \of{-1}R''(U;t) $ is in $\RrM{-\rho+\frac{3}{2}}{K,\ku,N} $ 
as it satisfies \eqref{2117}. 
Thus the smoothing operator $ \of{-1}R''(U;t) $ contributes  to  $ R_2(U;t) $ in \eqref{426}. 

Let us finally show that the operator $\of{-1} R'(U;t)\of{1} $ where  $R'(U;t)$ is in 
$\sRM{-\rho+\frac{3}{2}}{K,K',1}{N}$, 
contributes to $ R_1 (U;t) $ 
in the right hand side of
\eqref{426}. By \eqref{423} and applying Taylor formula
\begin{multline}
  \label{eq:4219}
  \of{1} = \mathrm{Id} + \sum_{\ell=1}^{N-1}\frac{1}{\ell!}(i\opbw(F(U)))^\ell \\+ \frac{1}{(N-1)!}\int_0^1 \opbw(iF(U))^N\oft(1-\theta)^{N-1}\,d\theta \, .
\end{multline}
Consider the integral term $I(U)$ in \eqref{4219}. Since $ F(U) $ is a matrix of symbols 
 in $ \sGM{\frac{1}{2}}{K, K'+ L,1}{N}  $, it 
 follows from Proposition \ref{215}, in particular \eqref{2123}, and
\eqref{425b} that for any $k\leq K- (K'+L)$,
\[\norm{\partial_t^k (I(U)W)}_{\Hds{s-\frac{3}{2}k-\frac{N}{2}}} \leq
C\sum_{k'+k''=k}\Gcals{\sigma}{k'+K'+L,N}{U}\Gcals{s}{k'',1}{W} \, .\]
Consequently, the terms obtained replacing in $\of{-1}R'(U;t)\of{1} $, $\of{1}$ by $I(U)$ will 
provide an  operator $ R'''(U;t)$ belonging to the space  $\RrM{-\rho+\frac{3}{2}+\frac{N}{2}}{K,K'+L,N}$. 
If we replace  $\rho$ by $\rho+\frac{N}{2}$, 
and change accordingly $K'=\rho'+2$ into $K'$ given by the 
integer part of 
 $\rho'+2+\frac{N}{2}$, (so that,
recalling  definition \eqref{cond},  $K'\leq \ku $), 
we get a smoothing operator of $\RrM{-\rho+\frac{3}{2}}{K,\ku,N}$ that
contributes to $R_1 (U;t) $ in 
 \eqref{426}.  By the composition results of section~\ref{sec:23}, in particular 
  Proposition \ref{231}, we may write, 
  for each $ \ell = 1, \ldots, N - 1 $,  
\begin{equation}
  \label{eq:4219a}
  (i\opbw(F(U)))^\ell  =   \opbw(M_\ell(U; t, \cdot)) + \tilde{R}_\ell(U;t)
\end{equation}
where $M_\ell$ is a matrix of symbols in the space  $\sGM{\ell/2}{K,K'+L,\ell}{N}$ 
and $\tilde{R}_\ell(U;t)$ is a matrix of smoothing operators in 
 $\sRM{-\rho + \frac{\ell}{2}}{K,K'+L,\ell}{N}$. 
Then we  compose each operator in \eqref{4219a} at the left with $\of{-1}R'(U;t) $, 
 where we expand also the flow  $\of{-1} $ as in \eqref{4219} at the  order  $ N - \ell  -1 $, instead of $ N $. 
By the previous arguments 
we get  again contributions to the smoothing
term in the right hand side of \eqref{426} performing as above a translation in $\rho$. This concludes the proof.
\end{proof1}

\section{Normal forms}\label{sec:43}

In Proposition \ref{421} we have obtained 
the system \eqref{426} which is diagonal,  up to smoothing terms, 
and the symbol $ (1+\zu(U;t))\mk(\xi)\Kcal - H (U; t, \xi )  $ 
has constant coefficients. The associated operator
commutes thus to derivatives, so that getting a Sobolev energy inequality is equivalent to getting an $L^2$ (or
$\Hds{0}$)-energy inequality. If the symbol $ H $ were real valued, 
the associated operator would be self-adjoint on $\Hds{0}$, so that,
forgetting for a while the smoothing operators in the  the right hand side of \eqref{426}, we would get preservation of the $\Hds{0}$ norm (and the
$\Hds{s}$ norm) of $\tilde{V}$. It turns out that the imaginary part of $H$ is not zero, but only given by an operator of
order zero. In this section we shall perform a normal form construction to replace $\Im H$ by a symbol, still of order zero, but vanishing like
$\norm{U}_{\Hds{\sigma}}^N$ when $U$ goes to zero. Consequently, the 
remaining non self-adjoint part of the equation will not
affect energy estimates up to a time of order $\epsilon^{-N}$, where $\epsilon$ is the size of the Cauchy data. 
In the next section \ref{sec:44} we shall
perform another 
normal form procedure to replace the smoothing operator $R_1 (U; t) $ in the right hand side of \eqref{426}, 
 by an operator vanishing at order $N$ at $U=0$, like $R_2(U; t)$, modulo again remainders that do not contribute to the
energy inequality  (actually we shall only construct modified Sobolev energies).

\medskip

\textbf{Remark}: From now on, our symbols will be always computed at $U$ belonging to $\CKHR{\sigma}{\C^2}$, with $\sigma, K$ large enough. In
particular, the argument $U$ is always of the form $U = \bigl[\begin{smallmatrix}u\\\bar{u}\end{smallmatrix}\bigr]$. When we
consider a matrix of symbols $A$ in $\sGM{m}{K,K',1}{N}$ for some $K'\leq K$, we denote 
\[
\Im (A(U; t, \cdot)) = \frac{1}{2i}[A(U; t, \cdot) - \overline{A(U; t, \cdot)} ] \, . 
\]
Recalling Definition~\ref{213}, we may decompose $A = \sum_{q=1}^{N-1}A_q + A_N$ where $A_q$ is in $\GtM{m}{q}$ and $A_N$ in
$\GrM{m}{K,K',N}$. We shall denote by \index{i@$\Im A_q$ (Imaginary part of a matrix of symbols) } $\Im A_q$ the matrix of $q$-linear forms
\begin{equation}
  \label{eq:430}
  \frac{1}{2i}\bigl[ A_q(U_1,\dots,U_q;\cdot) - \overline{A_q(-S\bar{U}_1,\dots,-S\bar{U}_q;\cdot)} \bigr]
\end{equation}
so that, when restricted to $U_j = U$ for any $j$, with $U$ in the above space, we do get $\Im (A_q(U,\dots,U;\cdot))$, since we
have the relation $\bar{U} = -SU$.

\begin{proposition}  \label{431}
{\bf (Reduction of $ \Im H $)} 
Assume that the parameter $ \kappa $ is outside the subset of zero measure of Proposition \ref{711}, so that estimate
\eqref{713} holds.
Then  there are

$ \bullet $
 a family $(B_q(\Ucal;\xi))_{q=1,\dots,N-1}$ of diagonal matrices of homogeneous symbols in $\Gt{0}{q}$, 
  with constant coefficients
in $x$, 
whose restrictions to $U_1 =\dots=U_q = U$ with $\bar{U} = -SU$ 
satisfy conditions \eqref{312}, \eqref{313},
\eqref{315} (Notice that by lemma \ref{hom-nonhom} the last anti-reversibility condition may be expressed equivalently through condition \eqref{315p}),

$ \bullet $
 a family  $(H_q^1(\Ucal;\xi))_{q=1,\dots,N} $  of diagonal matrices of symbols with constant coefficients in $ x $,
\be\label{eq:formH1q}
\begin{split} 
& H_q^1(\Ucal;\xi) \in \GtM{1}{q} \, , \ q=1,\dots,N-1 \, , \\ 
& H_N^1(U;t,\xi) \in \GrM{1}{K,\ku,N} \, , 
\end{split}
\ee
 such that the homogeneous symbols $H^1_q $, $ q = 1,\dots,N-1 $, are real valued,  
 and $\Im H^1_N$ is in $\GrM{0}{K,\ku,N}$, 
these matrices verifying 
the reality, parity preserving and reversibility properties  \eqref{311}, \eqref{313}, 
  \eqref{314}, when restricted to $U_1 =\dots=U_q = U$ as above,

such that, if we set
  \begin{equation}
    \label{eq:431}
    \begin{split}
      B(U;t,\xi) &= \sum_{q=1}^{N-1}B_q(U,\dots,U;\xi),\ \tilde{V}_1 = 
      \exp( \opbw (B(U;t,\xi))) \tilde{V}\\
H^1(U;t,\xi) &= \sum_{q=1}^{N-1}H^1_q(U \dots,U;\xi) + H^1_N(U;t,\xi),
    \end{split}
  \end{equation}
then $\tilde{V}^1$ solves the equation 
\begin{multline}
  \label{eq:432}
\bigl(D_t - \opbw\bigl((1+\zu(U;t))\mk(\xi)\Kcal + H^1(U;t,\xi)\bigr)\bigr)\tilde{V}^1\\
= R_1(U;t)\tilde{V}^1 + R_2(U;t)U
\end{multline}
where the smoothing operator $R_1(U;t)$ is in $\sRM{-\rho+\frac{3}{2}}{K,\ku,1}{N}$ and $R_2(U;t)$ is in 
$\RrM{-\rho+\frac{3}{2}}{K,\ku,N}$, $R_1(U;t)$, $ R_2 (U;t) $ satisfying 
the reality, parity preserving and reversibility conditions \eqref{316}, \eqref{318} and \eqref{319}. Moreover, we may
  write
  \begin{equation}
    \label{eq:433}
    \tilde{V}^1 = \tilde{V} + M(U;t)\tilde{V}
  \end{equation}
for some operator $M(U;t)$ in $\sMM{}{K,\ku,1}{N}$ satisfying conditions \eqref{317}, \eqref{318} and \eqref{3110}. 
Finally, 
for a large enough $\sigma$ and any $s$, one has the 
bound
\begin{equation}
  \label{eq:434}
  \norm{\tilde{V}^1-\tilde{V}}_{\Hds{s}}\leq C_s \norm{U}_{\Hds{\sigma}}\norm{\tilde{V}}_{\Hds{s}}.
\end{equation}
\end{proposition}
Before starting the proof of the proposition, we need to exhibit some structural
properties of the symbol $H$ in the left hand
side of \eqref{426}.
\medskip

For $p$ in $\N^*$, $ p $  even, we define the set 
\begin{equation}
  \label{eq:435}
  \Ccal_{p,\frac{p}{2}} = \bigl\{(n_1,\dots,n_p)\in (\N^*)^p; \{n_1,\dots,n_{\frac{p}{2}}\} = \{n_{\frac{p}{2}+1},\dots,n_p\}\bigr\}
\end{equation}
collecting those integer vectors $ (n_1,\dots,n_p) $  of $(\N^*)^p$ such that there is a bijection from the subset of the 
first  $p/2$ components of
$(n_1,\dots,n_p)$ onto the subset of the last $p/2$ ones. If $p$ is odd, or $p$ is even and $\ell\neq \frac{p}{2}$, define
$\Ccal_{p,\ell}$ to be the empty set.

For any $n$  in $\N^*$, define 
\begin{equation}
  \label{eq:436}
  \Pin{}^+ = \bigl[\begin{smallmatrix}1&0\\0&0\end{smallmatrix}\bigr]\Pin{},\qquad \Pin{}^- = \bigl[\begin{smallmatrix}0&0\\0&1\end{smallmatrix}\bigr]\Pin{} \, , 
\end{equation}
the composition of the spectral projectors $\Pin{}$ with projection from $\C^2$ to $\C\times\{0\}$ (resp.\ to
$\{0\}\times\C$). 
For  $ U $ 
satisfying $\bar{U} = -SU $, namely $ U $ of the form $ U = \vect{u}{\bar{u}}$,
the projectors $ \Pin{}^\pm $ can be written as follows.  Denote by $(\varphi_n)_{n\in \N^*}$
a real valued Hilbert basis of the space of
  even $L^2$ functions with zero mean, with $\varphi_n$ in the range of $\Pin{}$, 
  i.e. $\varphi_n(x) =   \frac{1}{\sqrt\pi}\cos(nx) $. 
  Then if we set $\hat{u}(n) = \frac{1}{\sqrt{\pi}}\int_\Tu u(x)\cos(nx)\,dx$, we have 
  \begin{equation}
    \label{eq:4311}
    \Pin{}U =\Bigl[\begin{smallmatrix}\hat{u}(n)\\\overline{\hat{u}(n)}\end{smallmatrix}\Bigr]\varphi_n,
    \quad \Pin{}^+U =
    \hat{u}(n)e_+\varphi_n \, ,\quad \Pin{}^-U =   \overline{\hat{u}(n)}e_-\varphi_n \, ,
  \end{equation}
where $e_+ = [\begin{smallmatrix}1\\0\end{smallmatrix}]$, $e_- =
[\begin{smallmatrix}0\\1\end{smallmatrix}]$.

\begin{lemma}
  \label{432}
Let $H_p(\Ucal;\xi)$ 
be a matrix of \emph{constant} coefficients homogeneous  symbols 
in  $\GtM{m}{p}$ for some $ m $ in $ \R $, $ p $ in $ \N^* $,
satisfying the 	reality, parity preserving and reversibility properties
 \eqref{311},  \eqref{313} and  \eqref{314p}. Then for any function $ U $ even in $ x $, satisfying 
 $ S U = - \bar{U} $, for any even $p $ in $ \N^* $
and  $ (n_1,\dots,n_\ell)$  in $(\N^*)^\ell$,  $\ell = \frac{p}{2}$,  we have 
\be\label{eq:ImHp}
\Im  H_p(  \Pin{1}^+U,\dots,\Pin{\ell}^+U, \Pin{1}^-U,\dots,\Pin{\ell}^-U ; \xi) \equiv 0 \, .
\ee
As a consequence 
\begin{multline}
  \label{eq:437}
  \Im H_p(U,\dots,U;\xi) \\=
  \sum_{\ell=0}^p\sum_{\substack{(n_1,\dots,n_p)\\\not\in\Ccal_{p,\ell}}}\bigl(\begin{smallmatrix}p\\\ell\end{smallmatrix}\bigr)
  \Im  H_p( \Pin{1}^+U,\dots,\Pin{\ell}^+U,\Pin{\ell+1}^-U,\dots,\Pin{p}^-U;\xi) \, .
\end{multline}
\end{lemma}
\begin{proof}
  Notice first that  \eqref{437} is trivial when one omits the restriction in the summation on $n_1,\dots,n_p$: this
  is just what one obtains using that $\Im H_p$ is a symmetric function of its arguments and  writing $\Pin{j} =
  \Pin{j}^++\Pin{j}^-$ for any $j = 1, \ldots, p $. Therefore  \eqref{437} follows by \eqref{ImHp}.

Recalling \eqref{4311} we have to check by $\C$-linearity of $\Im H_p$ that
\[\Im H_p(\varphi_{n_1}e_+,\dots,\varphi_{n_\ell}e_+,\varphi_{n_1}e_-,\dots,\varphi_{n_\ell}e_-;\xi)\equiv 0.\]
By \eqref{311}, we have
\[(\Im H_p)(U,\dots,U;-\xi) = S(\Im H_p)(U,\dots,U;\xi)S\]
and by \eqref{314p}
\[(\Im H_p)(SU,\dots,SU;\xi) = -S(\Im H_p)(U,\dots,U;\xi)S\]
for any $U$ satisfying $SU = -\bar{U}$. Consequently, as $p$ is even
\[(\Im H_p)(U,\dots,U;-\xi) = - (\Im H_p)(\bar{U},\dots,\bar{U};\xi) \, .\]
Decompose $U = \sum \Pin{j}U$, with $\Pin{j}U$ given by \eqref{4311}.   The above equality may be written, using symmetry and
  $\C$-linearity, as
  \begin{multline*}
    \sum_{n_1,\dots,n_p}\sum_{\ell=0}^p\bin{p}{\ell}\hat{u}(n_1)\cdots\hat{u}(n_\ell)\overline{\hat{u}(n_{\ell+1})}\cdots\overline{\hat{u}(n_{p})}\\
   \times (\Im H_p)(\varphi_{n_1}e_+,\dots,\varphi_{n_\ell}e_+,\varphi_{n_{\ell+1}}e_-,\dots,\varphi_{n_{p}}e_-;-\xi) =\\
-\sum_{n_1,\dots,n_p}\sum_{\ell=0}^p\bin{p}{\ell}\overline{\hat{u}(n_{\ell+1})}\cdots\overline{\hat{u}(n_{p})}\hat{u}(n_1)\cdots\hat{u}(n_\ell)\\
   \times (\Im H_p)(\varphi_{n_{\ell+1}}e_+,\dots,\varphi_{n_{p}}e_+,\varphi_{n_1}e_-,\dots,\varphi_{n_\ell}e_-;\xi) \, .
  \end{multline*}
Identifying the coefficients of
$\hat{u}(n_1)\cdots\hat{u}(n_\ell)\overline{\hat{u}(n_{\ell+1})}\cdots\overline{\hat{u}(n_{p})}$ on each side, we get
\begin{multline*}
  (\Im H_p)(\varphi_{n_1}e_+,\dots,\varphi_{n_\ell}e_+,\varphi_{n_{\ell+1}}e_-,\dots,\varphi_{n_{p}}e_-;-\xi)\\
= -(\Im H_p)(\varphi_{n_{\ell+1}}e_+,\dots,\varphi_{n_{p}}e_+,\varphi_{n_1}e_-,\dots,\varphi_{n_\ell}e_-;\xi) \, .
\end{multline*}
In particular, if $\ell = \frac{p}{2}$ and $n_{\ell+1} = n_1,\dots, n_p = n_\ell$, we obtain that \[\xi \to (\Im
H_p)(\varphi_{n_1}e_+,\dots,\varphi_{n_\ell}e_+,\varphi_{n_{\ell+1}}e_-,\dots,\varphi_{n_{p}}e_-;\xi)\] 
is an odd
function. Making the same reasoning starting from \eqref{313}
\[(\Im H_p)(U,\dots,U;-\xi) = \Im H_p(U,\dots,U;\xi)\]
we conclude that  
$(\Im H_p)(\varphi_{n_1}e_+,\dots,\varphi_{n_\ell}e_+,\varphi_{n_{\ell+1}}e_-,\dots,\varphi_{n_{p}}e_-;\xi)$ is also an even
function, so that it vanishes identically.
 This proves the lemma. 
\end{proof}
We shall need a second lemma to prove Proposition \ref{431}.
\begin{lemma}
  \label{433}
Let $B(\Ucal;\xi)$ be a constant coefficients matrix with entries in $\Gt{0}{p}$, satisfying
\begin{equation}
  \label{eq:439}
  \overline{B(U,\dots,U;\xi)}^\vee = \pm S B(U,\dots,U;\xi)S
\end{equation}
for any function  $U(x)$, even in $x$ and satisfying $SU = -\bar{U}$. Then for any such $U$, any indices $n_1,\dots,n_p$ in
$\N^*$, any $\ell = 0,\dots, p$
\begin{multline}
  \label{eq:4310}
\overline{B(\Pin{1}^+U,\dots,\Pin{\ell}^+U,\Pin{\ell+1}^-U,\dots,\Pin{p}^-U;\xi)}^\vee\\
=\pm S B(\Pin{1}^-U,\dots,\Pin{\ell}^-U,\Pin{\ell+1}^+U,\dots,\Pin{p}^+U;\xi)S \, .
\end{multline}
In the same way, if we assume
\[
B(SU,\dots,SU;\xi) = -SB(U,\dots,U;\xi)S \, , 
\] 
then
\begin{multline}\label{eq:4310aa}
B(\Pin{1}^+SU,\dots,\Pin{\ell}^+SU,\Pin{\ell+1}^-SU,\dots,\Pin{p}^-SU;\xi)\\
= -S B(\Pin{1}^-U,\dots,\Pin{\ell}^-U,\Pin{\ell+1}^+U,\dots,\Pin{p}^+U;\xi)S \, .
\end{multline}
\end{lemma}
\begin{proof}
  We may write using \eqref{4311}
\begin{multline*}
  B(U,\dots,U;\xi)  \\=\sum_{\ell=0}^p
  \bigl(\begin{smallmatrix}p\\\ell\end{smallmatrix}\bigr)\sum_{n_1,\dots,n_p}B(\Pin{1}^+U,\dots,\Pin{\ell}^+U,\Pin{\ell+1}^-U,\dots,\Pin{p}^-U;\xi)
  \\
= \sum_{\ell=0}^p
\bigl(\begin{smallmatrix}p\\\ell\end{smallmatrix}\bigr)\sum_{n_1,\dots,n_p}\hat{u}(n_1)\cdots\hat{u}(n_\ell)\overline{\hat{u}(n_{\ell+1})}\cdots\overline{\hat{u}(n_{p})}
\\
\times B(e_+\varphi_{n_1},\dots,e_+\varphi_{n_\ell},e_-\varphi_{n_{\ell+1}},\dots,e_-\varphi_{n_{p}};\xi) \, .
\end{multline*}
Plug this expression inside \eqref{439}, and identify on both sides the coefficients of $\overline{\hat{u}(n_1)}\dots
\overline{\hat{u}(n_{\ell})}\hat{u}(n_{\ell+1})\dots \hat{u}(n_{p})$. We get
\begin{multline*}
  \overline{B(e_+\varphi_{n_1},\dots,e_+\varphi_{n_\ell},e_-\varphi_{n_{\ell+1}},\dots,e_-\varphi_{n_{p}};\xi)}^\vee\\
= \pm S B(e_-\varphi_{n_1},\dots,e_-\varphi_{n_\ell},e_+\varphi_{n_{\ell+1}},\dots,e_+\varphi_{n_{p}};\xi) S
\end{multline*}
using that $B$ is $p$-linear symmetric in its first argument. This gives \eqref{4310}.

To prove the last statement of the lemma, we write using again \eqref{4311}
\begin{multline*}
  B(SU,\dots,SU;\xi)  \\=\sum_{\ell=0}^p
  \bigl(\begin{smallmatrix}p\\\ell\end{smallmatrix}\bigr)\sum_{n_1,\dots,n_p}B(S\Pin{1}^+U,\dots,S\Pin{\ell}^+U,S\Pin{\ell+1}^-U,\dots,S\Pin{p}^-U;\xi)
  \\
= (-1)^p\sum_{\ell=0}^p
\bigl(\begin{smallmatrix}p\\\ell\end{smallmatrix}\bigr)\sum_{n_1,\dots,n_p}\hat{u}(n_1)\cdots\hat{u}(n_\ell)\overline{\hat{u}(n_{\ell+1})}\cdots\overline{\hat{u}(n_{p})}
\\
\times B(e_-\varphi_{n_1},\dots,e_-\varphi_{n_\ell},e_+\varphi_{n_{\ell+1}},\dots,e_+\varphi_{n_{p}};\xi).
\end{multline*}
By identification with the expansion of $SB(U,\dots,U; \xi)S$ we get
\begin{multline*}
  B(e_-\varphi_{n_1},\dots,e_-\varphi_{n_\ell},e_+\varphi_{n_{\ell+1}},\dots,e_+\varphi_{n_{p}};\xi)\\
= -(-1)^p S B(e_+\varphi_{n_1},\dots,e_+\varphi_{n_\ell},e_-\varphi_{n_{\ell+1}},\dots,e_-\varphi_{n_{p}};\xi) S
\end{multline*}
which implies \eqref{4310aa}.
\end{proof}
\begin{proof1}{Proof of Proposition~\ref{431}}
We conjugate the equation \eqref{426} 
by the  operator $\exp(\opbw(B(U;t,\xi)))$
where  $ B(U;t,\xi) $ is a diagonal matrix of constant coefficient symbols as in \eqref{431}, 
with $ B_q ({\cal U}; \xi ) $ in $ \GtM{0}{q} $  
to be chosen. 
The left hand side operator in \eqref{426} transforms into 
\begin{multline}
  \label{eq:4313}
\exp(\opbw(B(U;t,\xi))) \bigl[D_t -\opbw\bigl((1+\zu(U;t))\mk(\xi)\Kcal - H(U;t,\xi)\bigr)\bigr]\\
\times \exp(\opbw(-B(U;t,\xi)))\\
= D_t -\opbw\bigl[D_tB(U;t,\xi)+ (1+\zu(U;t))\mk(\xi)\Kcal  - H(U;t,\xi)\bigr].
\end{multline}
We now choose   
\be\label{eq:defB1n}
B (U;t,\xi) = \sum_{q=1}^{N-1} B_q(U,\dots,U;\xi)
\ee
so that the symbol
\begin{equation} \label{eq:4314}
  D_tB(U;t,\xi) - H(U;t,\xi) = H^1(U;t,\xi)
\end{equation}
has the form  $ H^1= \sum_{q=1}^{N-1}H^1_q + H^1_N $ 
 in \eqref{431}, with  
homogeneous symbols $H^1_q $ 
which  are {\it real} valued,  \eqref{formH1q} holds, 
 and $\Im H^1_N$ is in $\GrM{0}{K,\ku,N}$. 
First we notice that, differentiating each $ B_q $ in \eqref{defB1n} and inserting 
the expression  \eqref{4312} of $ D_t U $,  we get 
\begin{multline}
  \label{eq:4315}
D_t(B_q(U,\dots,U;\xi)) = \sum_{q'=1}^q B_q( \underbrace{U,\dots, U}_{q'-1}, D_tU,\dots,U;\xi)\\
= \sum_{q'=1}^q B_q(\underbrace{U,\dots, U}_{q'-1}, \mk(D) \Kcal U,\dots,U;\xi)\\
+ \sum_{q'=1}^q B_q(\underbrace{U,\dots, U}_{q'-1},M(U; t)U,\dots,U;\xi) \, .
\end{multline}
Since we look for $B_q(U,\dots,U;\xi)$ satisfying \eqref{312} (resp.\
\eqref{313}, resp.\ \eqref{315}), these properties imply that  $\partial_tB_q(U,\dots,U;\xi)$ satisfies \eqref{312} (resp.\
\eqref{313}, resp.\ \eqref{314}). Consequently, the left hand
side of \eqref{4315} satisfies \eqref{311} (resp.\ \eqref{313}, resp.\ \eqref{314}). Decomposing the right hand side in
homogeneous contributions, we see that each of them satisfies 
conditions \eqref{311}, \eqref{313} and \eqref{314p}, by lemma~\ref{hom-nonhom}.

According to Proposition~\ref{421}, the matrix $ H $ in the left hand side
of \eqref{4314} 
may be written as
\be\label{eq:defH1n}
H(U;t,\xi)  = \sum_{q=1}^{N-1}H_q(U,\dots,U;\xi) + H_N(U;t,\xi)
\ee
where $ H_q $ is a diagonal matrix of symbols of $\Gt{1}{q}$ 
with  imaginary part in $\Gt{0}{q}$,  and where $H_N$ is a
diagonal matrix with entries in $\Gr{1}{K,\ku,N}$ and imaginary part  in $\Gr{0}{K,\ku,N}$. Moreover, 
these matrices of symbols
satisfy the reality, parity preserving and reversibility properties \eqref{311}, \eqref{313} and \eqref{314} (or \eqref{314p}
for the homogeneous components by lemma \ref{hom-nonhom}). 

By  \eqref{defB1n},  \eqref{4315}, \eqref{defH1n} we may write the left hand side of 
\eqref{4314} as
\begin{multline}
  \label{eq:4316}
\sum_{q=1}^{N-1} \Bigl[\sum_{q'=1}^q B_q(\underbrace{U,\dots, U}_{q'-1},\mk(D)\Kcal U,\dots,U;\xi) - H_q(U,\dots,U;\xi)\\
+\sum_{q'=1}^q B_q(\underbrace{U,\dots, U}_{q'-1},M(U; t)U,\dots,U;\xi)\Bigr] - H_N \, . 
\end{multline}
Since each $ B_q  $ is in $ \Gt{0}{q}  $ and $ M(U; t) $ is in $\sMM{}{K,1,1}{N}$, 
 Proposition~\ref{233}-(i) implies  that 
each  symbol   $ B_q(U,\dots,M(U; t)U,\dots,U;\xi)$ is in $\sGM{0}{K,\ku,q+1}{N}$. 
We may group these symbols together with $H_q $ and $  H_N$, according to their degree of homogeneity,  and rewrite \eqref{4316} as
\begin{multline}
  \label{eq:4317}
\sum_{q=1}^{N-1} \Bigl[\sum_{q'=1}^q B_q(\underbrace{U,\dots, U}_{q'-1},\mk(D)\Kcal U,\dots,U;\xi) - 
\tilde{H}_q(U,\dots,U;\xi)\Bigr]\\
 - \tilde{H}_N(U;\xi)
\end{multline}
where $\tilde{H}_q$ is a diagonal matrix of symbols in $\Gt{1}{q}$ with imaginary part in $\Gt{0}{q}$, and where
$\tilde{H}_N$ is a diagonal matrix with entries in $\Gr{1}{K,\ku,N}$ and imaginary part in $\Gr{0}{K,\ku,N}$.  Note that
$\tilde{H}_q$ depends only on $B_{q'}$, $q'<q$.  Moreover $\tilde{H}_q$ and $\tilde{H}_N $ satisfy, like $ H_q$ and $ H_N $,
the reality, parity preserving and reversibility properties \eqref{311}, \eqref{313}, \eqref{314} 
(or \eqref{314p} for the homogeneous component $\tilde{H}_q$), as the last sum in \eqref{4315} satisfies these properties as well.

To obtain \eqref{4314}, we have to choose $B_q$ in order to compensate the imaginary part of $\tilde{H}_q$ in \eqref{4317},
so that we shall be left with only the real part of $\tilde{H}_q$, $q=1,\dots,N-1$, that  provides the contribution to
$H^1_q$ in \eqref{4314}, \eqref{431}. In other words, we want to find $B_q$ so  that
$$
\sum_{q'=1}^q B_q(U_1,\dots ,\mk(D)\Kcal U_{q'},\dots,U_q;\xi) = i\Im \tilde{H}_q(U_1,\dots,U_q;\xi) \, .
$$
We now determine
\[
B_q(\Pin{1}^+U_1,\dots,\Pin{\ell}^+U_\ell,\Pin{\ell+1}^-U_{\ell+1},\dots,\Pin{q}^-U_q;\xi) 
\]
for any  $ \ell =1, \ldots, q $, and  $ n_1,\dots,n_q \in \N^* $.  
We decompose 
$  U_j = \sum_{n_j \in \N^*} \Pi_{n_j}^+ U_j  + \Pi_{n_j}^- U_j  $, $ j = 1, \ldots, q $,  and
since 
$$
\mk(D)\Kcal \Pin{j}^\pm U_j = \pm\mk(n_j)\Pin{j}^\pm  U_j \, , 
$$ 
we write
this equation
\begin{multline*}
  \Bigl[\sum_{q' = 1}^\ell \mk(n_{q'}) - \sum_{q' =\ell+  1}^q \mk(n_{q'})\Bigr]\\\times  B_q(\Pin{1}^+U_1,\dots,
  \Pin{\ell}^+U_\ell,\Pin{\ell+1}^-U_{\ell+1},\dots,\Pin{q}^-U_q;\xi)\\
= i\Im \tilde{H}_q(\Pin{1}^+U_1,\dots  \Pin{\ell}^+U_\ell,\Pin{\ell+1}^-U_{\ell+1},\dots,\Pin{q}^-U_q;\xi) \, .
\end{multline*}
As $ \tilde{H}_q $ satisfies the assumptions of lemma~\ref{432}, the right hand side vanishes if $q$ is even, $\ell =
\frac{q}{2}$ and $\{n_1,\dots,n_\ell\} = \{n_{\ell+1},\dots,n_q\}$. In all the other cases, by Proposition \ref{711},  
when the
  parameter $\kappa$ is fixed outside a subset of zero measure, the quantity
\begin{equation}\label{eq:4317a}
\Dcal_\ell(n_1,\dots,n_q) = \sum_{q'=1}^\ell \mk(n_{q'}) - \sum_{q'=\ell+1}^q\mk(n_{q'})
\end{equation}
does not vanish, and its absolute value is actually bounded from below by $c\abs{n}^{-N_0}$, 
$n =    (n_1,\dots,n_q)$, for some integer $N_0 $.  
    We may thus define
    \begin{multline}
      \label{eq:4318}
B_q(U,\dots,U;\xi) = \sum_{\ell=0}^q
\bigl(\begin{smallmatrix}q\\\ell\end{smallmatrix}\bigr)\sum_{(n_1,\dots,n_q)\not\in\Ccal_{q,\ell}}\Dcal_\ell
(n_1,\dots,n_q)^{-1}\\
\times i\Im\tilde{H}_q(\Pin{1}^+U,\dots,\Pin{\ell}^+U,\Pin{\ell+1}^-U,\dots,\Pin{q}^-U;\xi) \, .
    \end{multline}
The lower bound  $\abs{\Dcal_\ell (n)} \geq c\abs{n}^{-N_0} $ 
implies that the constant coefficients symbol $ B_q $ satisfies estimates of the form 
\eqref{214} with $ m = 0 $   as $ \Im \tilde{H}_q$
does, changing the value of $\mu$ depending on $ N_0$ (actually  
we have to check \eqref{214} just for $ \alpha = 0 $ since $ B_q $ is  constant  in $ x $).
By the last remark after Definition \ref{211}
we obtain  symbols $ B_q $, $ q =1, \ldots , N - 1 $,  which act on $ {\dot H}^\sigma $ 
taking $ \sigma $ large enough with respect to $ \mu $, i.e. 
large  with respect to the number of steps $ N $  and the loss of derivatives $ N_0 $
produced by the small divisors $ \Dcal_\ell (n) $. 

Therefore $B_q$ is a diagonal 
matrix of homogeneous symbols in $\GtM{0}{q}$. It satisfies \eqref{313} as $\tilde{H}_q$
does. Let us check that \eqref{315} holds. Since $\tilde{H}_q$ satisfies 
\eqref{314p}, we deduce from the last statement in
lemma~\ref{433} that \eqref{4310aa} holds with $B$ replaced by $i\Im\tilde{H}_q$. We deduce that
\begin{multline*}
  B_q(SU,\dots,SU;\xi) = -\sum_{\ell=0}^q
\bigl(\begin{smallmatrix}q\\\ell\end{smallmatrix}\bigr)\sum_{(n_1,\dots,n_q)\not\in\Ccal_{q,\ell}}\Dcal_\ell
(n_1,\dots,n_q)^{-1}\\
\times iS\Im\tilde{H}_q(\Pin{1}^-U,\dots,\Pin{\ell}^-U,\Pin{\ell+1}^+U,\dots,\Pin{q}^+U;\xi)S \, . 
\end{multline*}
Since $\tilde{H}_q$ is symmetric in its first $q$ arguments and \[\Dcal_\ell(n_1,\dots,n_\ell,n_{\ell+1},\dots,n_q) = -
\Dcal_\ell(n_{\ell+1},\dots,n_q,n_1,\dots,n_\ell),\] we obtain that 
\[B_q(SU,\dots,SU;\xi) = SB_q(U,\dots,U;\xi)S.\] 
This is condition \eqref{315p} (applied to the arguments $U_1=\cdots = U_q = U$), which, if we consider $B_q(U,\dots,U;\xi)$ as a non-homogeneous symbol, implies the
anti-reversibility condition \eqref{315}.

Let us check that $B_q$ satisfies as well \eqref{312}. The matrix $i\Im\tilde{H}_q$ satisfies \eqref{311} by assumption i.e.\
condition \eqref{439} with the minus sign. It follows from lemma~\ref{433} that
\begin{multline*}
  \overline{i\Im \tilde{H}_q(\Pin{1}^+U,\dots,\Pin{\ell}^+U,\Pin{\ell+1}^-U,\dots,\Pin{q}^-U;\xi)}^\vee\\
= -Si\Im\tilde{H}_q(\Pin{1}^-U,\dots,\Pin{\ell}^-U,\Pin{\ell+1}^+U,\dots,\Pin{q}^+U;\xi)S,
\end{multline*}
so that, by \eqref{4318}
\begin{multline*}
  \overline{B_q(U,\dots,U;\xi)}^\vee = -i\sum_{\ell=0}^q \bigl(\begin{smallmatrix}q\\\ell\end{smallmatrix}\bigr)\sum_{(n_1,\dots,n_q)\not\in\Ccal_{q,\ell}}\Dcal_\ell
(n_1,\dots,n_\ell,n_{\ell+1},\dots,n_q)^{-1}\\
\times S\Im\tilde{H}_q(\Pin{1}^-U,\dots,\Pin{\ell}^-U,\Pin{\ell+1}^+U,\dots,\Pin{q}^+U;\xi)S\\
= i \sum_{\ell=0}^q
\bigl(\begin{smallmatrix}q\\\ell\end{smallmatrix}\bigr)\sum_{(n_1,\dots,n_q)\not\in\Ccal_{q,\ell}}\Dcal_\ell(n_{\ell+1},\dots,n_q,n_1,\dots,n_\ell)^{-1}\\
\times S\Im\tilde{H}_q(\Pin{\ell+1}^+U,\dots,\Pin{q}^+U,\Pin{1}^-U,\dots,\Pin{\ell}^-U;\xi)S\\
=S B_q(U,\dots,U;\xi)S \, .
\end{multline*}
In other words, $B_q$ satisfies \eqref{312}.

We are left with checking that \eqref{432} holds. 
Under the change of variables 
$ \tilde{V}^1 = \exp\bigl(\opbw(B(U;t;\xi))\bigr)\tilde{V} $, 
taking into account \eqref{4313} and \eqref{4314},
the system \eqref{426} transforms into
\begin{multline}
  \label{eq:4319}
\Bigl(D_t -\opbw\bigl[(1+\zu(U;t))\mk(\xi)\Kcal + H_1(U;t,\cdot)\bigr]\Bigr)\tilde{V}_1\\
= \exp\bigl(\opbw(B(U;t,\xi))\bigr)\bigl[R_1(U;t)\exp\bigl(-\opbw(B(U;t,\xi))\bigr)\tilde{V}_1\bigr]\\
+ \exp\bigl(\opbw(B(U;t,\xi))\bigr)[R_2(U;t)U].
\end{multline}
Let us expand
\begin{equation}
  \label{eq:4320}
  \exp\bigl(\opbw(B(U;t,\xi))\bigr) = \sum_{q=0}^{N-1}\frac{(\opbw(B(U;t,\xi)))^q}{q!} + S_N(U; t).
\end{equation}
As $S_N$ is the remainder of an absolutely convergent series of bounded operators on $\Hds{s}$ for any $s$, as well as its
$\partial_t^k$-derivatives, $k\leq K-K'$, we see that the replacement of one of the exponentials in the right hand side of
\eqref{4319} by $S_N(U;t)$ brings operators satisfying the bounds \eqref{2117} defining $\Rr{-\rho+\frac{3}{2}}{K,\ku,N}$. The
corresponding terms may be incorporated to $R_1(U ;t)\tilde{V}^1 + R_2(U;t)U$ in \eqref{432}. Consider next the sum in
\eqref{4320}. By the definition \eqref{431} of $B$, and the symbolic calculus properties of Proposition~\ref{231} and 
Proposition~\ref{232} (which are quite trivial here since we deal with constant coefficient symbols), we may write that sum as
\[
\mathrm{Id} + \sum_{q=1}^{N-1}\Bigl(\opbw(\tilde{B}_q(U,\dots,U;\cdot))+R'_q(U; t)
\Bigr)\]
modulo terms of the same form as $S_N(U; t)$ in \eqref{4320}. 
Replacing the exponentials in \eqref{4319} by such a sum brings
again smoothing expressions contributing to $R_1(U; t)\tilde{V}^1 + R_2(U; t)U$ in the right hand side of \eqref{432}, by
Propositions~\ref{232} and~\ref{233}.

Moreover, since $B$ satisfies \eqref{312}, \eqref{313} and \eqref{315}, the operators $S_N $, $\opbw(\tilde{B}_q)$,
$R'_q $ satisfy \eqref{317}, \eqref{318} and \eqref{3110}. Since in \eqref{4319}, $R_1 $ and $R_2 $ satisfy
\eqref{316}, \eqref{318} and \eqref{319}, it follows by composition that the similar operators in \eqref{432} satisfy these
properties as well. This concludes the proof of \eqref{432}.
Finally, the definition \eqref{431} of $\tilde{V}^1$ and \eqref{4320} show that $\tilde{V}^1$ may be written as \eqref{433},
and that \eqref{434} holds, as a consequence of the boundedness of paradifferential operators on Sobolev spaces.
\end{proof1}

\section{Proof of Theorem~\ref{311}}\label{sec:44}

The proof of Theorem~\ref{311} will rely on a normal form construction of modified energy Sobolev norms, which are
quasi-invariant. The small divisors, 
which produce losses of derivatives,
 are compensated by  the smoothing character of the  operators $ R_1 $, $ R_2 $ in the right hand side 
  of \eqref{432}. 
We define first the classes of multilinear forms 
that will be used to construct a modified
energy for \eqref{432}.
\begin{definition}
  \label{441} {\bf (Multilinear energy forms)}
Let $\rho, s$ be in $\R_+$, $p$ in $\N$. One denotes by \index{L@$\Lt{-\rho}{p,\pm}$ (Space of multilinear energy forms)}
$\Lt{-\rho}{p,\pm}$ the space of symmetric $(p+2)$-linear forms 
$$
(U_0,\dots,U_{p+1}) \to L (U_0,\dots,U_{p+1})
$$
defined on
$\Hds{\infty}(\Tu,\C^2)$ and satisfying for some $\mu$ in $\R_+$ and all $n_0,\dots,n_{p+1}$ in $(\N^*)^{p+2}$, all
$(U_0,\dots,U_{p+1})$ in $\Hds{\infty}(\Tu,\C^2)^{p+2}$,
\begin{multline}
  \label{eq:441}
\abs{L(\Pin{0}U_0,\dots,\Pin{p+1}U_{p+1})} \leq C\max(n_0,\dots,n_{p+1})^{2s-\rho}\\
\times\max\nolimits_3(n_0,\dots,n_{p+1})^{\rho+\mu}\prod_{0}^{p+1}\norm{\Pin{\ell}U_\ell}_{L^2}
\end{multline}
where $\max_3(n_0,\dots,n_{p+1})$ is the third largest among those integers, and such that
\begin{multline}
  \label{eq:442}
L(\Pin{0}U_0,\dots,\Pin{p+1}U_{p+1}) \not\equiv 0\\
\Rightarrow \sum_{0}^{p+1}\epsilon_\ell n_\ell = 0 \textrm{ for some choice of the signs } \epsilon_\ell \in \{-1,1\}
\end{multline}
and, for any $U_0,\dots,U_{p+1}$ satisfying $SU_j = -\bar{U}_j$,
\begin{equation}
  \label{eq:443}
  L(SU_0,\dots,SU_{p+1}) = \pm L(U_0,\dots,U_{p+1}).
\end{equation}
\end{definition}
\textbf{Remark}: To check \eqref{443}, it is enough to prove that
\begin{equation}
  \label{eq:443a}
  L(SU,\dots,SU) = \pm L(U,\dots,U)
\end{equation}
for any $U$ satisfying $SU=-\bar{U}$, as follows from the $\R$-linearity and symmetry of $L$.

\smallskip

We define below a multilinear form  $ L $ of $ \Lt{-\rho}{p,\pm}$ that will be used in the proof of lemma \ref{443}.

\smallskip

\noindent
\textbf{Example}: Consider $R$ an element of $\Rt{-\rho}{p}$  and define $L(U_0,\dots,U_{p+1})$ to be the symmetrization of
\begin{equation}
  \label{eq:444}
  (U_0,\dots,U_{p+1}) \to \int_\Tu(\abs{D}^sSU_0)(\abs{D}^s R(U_1,\dots,U_p)U_{p+1})\,dx.
\end{equation}
Let us show that if $s\gg \rho$, we get an element of $\Lt{-\rho}{p,+}$ (resp.\ $\Lt{-\rho}{p,-}$) if 
$R(U_1,\dots,U_p)$ satisfies \eqref{3110p} (resp.\ \eqref{319p}).
First, we have the estimate
\begin{multline}\label{eq:boundL}
  \abs{L(\Pin{0}U_0,\dots,\Pin{p+1}U_{p+1})}\\
\leq Cn_0^{2s}\norm{\Pin{0}U_0}_{L^2} \norm{\Pin{0}R(\Pin{1}U_1,\dots,\Pin{p}U_p)\Pin{p+1}U_{p+1}}_{L^2}.
\end{multline}
By condition \eqref{2116}, we see that \eqref{442} holds. Moreover \eqref{2115}  implies 
 that the right hand side in \eqref{boundL} is bounded by
\begin{equation}
  \label{eq:445}
  Cn_0^{2s}\max\nolimits_2(n_1,\dots,n_{p+1})^{\rho+\mu}\max(n_1,\dots,n_{p+1})^{-\rho}\prod_{0}^{p+1}\norm{\Pin{\ell}U_\ell}_{L^2}.
\end{equation}
We may assume $n_1\geq n_2\geq\cdots\geq n_{p+1}$ and because of \eqref{442}, $\abs{n_0-n_1}\leq Cn_2$. If $n_0\geq
\frac{1}{2}n_1$, we have that $n_0\sim n_1$ and $\max_2(n_1,\dots,n_{p+1})\sim \max_3(n_0,\dots,n_{p+1})$ so that
\eqref{445} implies \eqref{441}. If $n_0\leq \frac{1}{2} n_1$, then $n_1 \sim n_2$, so that $n_0\leq C\max_3(n_0,\dots,n_{p+1})$. Then we estimate \eqref{445}, and thus \eqref{boundL}, by
\[C\max(n_0,\dots,n_{p+1})^{2s+\mu} \Bigl(\frac{\max_3(n_0,\dots,n_{p+1})}{\max(n_0,\dots,n_{p+1})}\Bigr)^{2s}
  \prod_{0}^{p+1}\norm{\Pin{\ell}U_\ell}_{L^2}\]
which, for $2s\geq\rho+\mu $,  is bounded by the right hand side in \eqref{441}.

Finally we check that the form defined in \eqref{444} satisfies \eqref{443a}. Since the matrix  $S$ is symmetric we have
\begin{multline}\label{eq:ex-sym}
  \int_\Tu (\abs{D}^sS(SU))(\abs{D}^s R(SU,\dots,SU)SU)\,dx\\
= \int_\Tu(\abs{D}^sSU)(\abs{D}^s SR(SU,\dots,SU)SU)\,dx.
\end{multline}
If $ R(U_1,\dots,U_p) $ satisfies \eqref{3110p} (resp.\
\eqref{319p}) then the right hand side in \eqref{ex-sym} is equal to \eqref{444} (resp.\ minus \eqref{444}) computed at $U_0=\dots=U_{p+1}=U$ . This concludes the proof.

\medskip

We shall need the following properties of the multilinear forms of the class $\Lt{-\rho}{p,\pm}$.
\begin{lemma}  \label{442}
(i) Let $L$ be in $\Lt{-\rho}{p,\pm}$. Then for any $m\geq 0$ such that $\rho>m+\frac{1}{2}$  and any  $s>\rho+\mu+m+\frac{1}{2}$, $L$ extends as a continuous
$(p+2)$-linear form on $\Hds{s}\times\cdots\times\Hds{s}\times\Hds{s-m}\times\Hds{s}\times\cdots\times\Hds{s}$.

(ii) Assume that $p=2\ell$ is even, and let $L$ be in $\Lt{-\rho}{p,-}$. If $U$ is an even function of $x$, satisfying $SU =
-\bar{U}$, then for any $n_0,\dots,n_\ell$ in $\N^*$,
\begin{equation}
  \label{eq:446}
  L(\Pin{0}^+U,\dots,\Pin{\ell}^+U,\Pin{0}^-U,\dots,\Pin{\ell}^-U) = 0 
\end{equation}
where $ \Pin{}^{\pm} $ are defined in \eqref{436}-\eqref{4311}. 

(iii) Assume that the parameter $\kappa$ is outside the subset of zero measure of Proposition~\ref{711}, so that 
estimate \eqref{713} holds, namely using notation \eqref{4317a}, for some $c>0$ and $N_0\in \N$,
\begin{equation*}
\abs{\Dcal_\ell(n_0,\dots,n_{p+1})}\geq c\max(n_0,\dots,n_{p+1})^{-N_0}
\end{equation*}
 for any $(n_0,\dots,n_{p+1})$ in $(\N^*)^{p+2}$ if $p$ is odd or $p$ is even and $\ell\neq\frac{p}{2}$, and for any
$(n_0,\dots,n_{p+1})$ in $(\N^*)^{p+2}$ such that  
\begin{equation*}
  \{n_0,\dots,n_\ell\} \neq \{n_{\ell+1},\dots,n_{p+1}\}
\end{equation*}
when $p$ is even and $\ell = \frac{p}{2}$. Then, 
for any $L$ in $\Lt{-\rho}{p,-}$, there is $\tilde{L}$ in $\Lt{-\rho+N_0}{p,+}$ such that 
\begin{equation}
  \label{eq:447}
  \sum_{j=0}^{p+1} \tilde{L}(U,\dots , \mk(D)\Kcal U ,\dots,U) = iL(U,\dots,U)
\end{equation}
(where $\mk(D)\Kcal$ acts on the argument corresponding to $U_j$ in the above sum).

(iv) Let $L$ be a multilinear form in $\Lt{-\rho}{p,\pm}$ and let $M$ be an operator  in $\sMM{m}{K,K',q}{N}$ (see Definition~\ref{216}) satisfying
conditions \eqref{317} and \eqref{319} (resp.\ and  \eqref{3110}). Then 
\be\label{eq:U-comp-M}
U\to L(U,\dots,U,M(U; t)U,U,\dots,U)
\ee
may be written as the sum 
$ \sum_{q'=0}^{N-p-q-1}L_{q'}(U,\dots,U)$ where  $ L_{q'} $ are suitable 
multilinear forms $L_{q'}$ in $\Lt{-\rho+m}{p+q+q',\mp}$ (resp.\
$\Lt{-\rho+m}{p+q+q',\pm}$), plus a term that, at any time $t$, is
\be\label{eq:resN}
O\bigl(\norm{U(t,\cdot)}_{\Hds{s}}^{p+2}\Gcals{\sigma}{K',N-p}{U}+ 
 \Gcals{\sigma}{K',N}{U} \Gcals{s}{K',1}{U} \norm{U(t,\cdot)}_{\Hds{s}}^{p+1}\bigr)
\ee
 if $s>\sigma\gg \rho$, and if $\Gcals{\sigma}{K',1}{U}$ is bounded. 

\end{lemma}
\begin{proof}
(i) Fix an index $ \ell_0 \in \{0, \ldots, p + 1 \} $ and assume 
 that  $U_\ell$ is in $\Hds{s}$, $\ell= 0,\dots,p+1$, $ \ell \neq \ell_0 $, and $U_{\ell_0} \in \Hds{s-m} $.
The left hand side of \eqref{441} evaluated at $n_0\geq n_1\geq\cdots\geq
  n_{p+1}$ is bounded  by
\[Cn_0^{2s-\rho+m}n_2^{\rho+\mu}\prod_{0}^{p+1} n_\ell^{-s}c^\ell_{n_\ell}\bigl(\prod_{\substack{\ell\neq\ell_0\\0\leq\ell\leq
    p+1}}\norm{U_\ell}_\Hds{s}\bigr)\norm{U_{\ell_0}}_{\Hds{s-m}}\]
where 
$(c^\ell_{n_\ell})_{n_\ell}$ is a sequence in $\ell^2$ and, by \eqref{442},  $ n_0 \sim n_1 $. 
Summing over all the indices $ n_0, \ldots, n_{p+1} $ satisfying \eqref{442} one gets
the conclusion.

(ii) Write condition \eqref{443} for $ L $ (with the minus  sign because
$ L $ is in $ \Lt{-\rho}{p,-} $) with 
the vectors valued functions
\[ 
U_j = \vect{\hat{u}(n_j)}{\overline{\hat{u}(n_j)}}\varphi_{n_j} = \hat{u}(n_j) \varphi_{n_j}e_+ + \overline{\hat{u}(n_j)}
\varphi_{n_j}e_- \, , \quad j = 0, \ldots, p + 1 \, , 
\]
having used the notation introduced in \eqref{4311}. Then, since $ S e_+ = - e_- $ and $ S e_- = - e_+ $,  we get
\begin{multline*}
  L(\varphi_{n_0}\hat{u}(n_0)e_- + \varphi_{n_0}\overline{\hat{u}(n_0)}e_+,\dots,\varphi_{n_{p+1}}\hat{u}(n_{p+1})e_- +
  \varphi_{n_{p+1}}\overline{\hat{u}(n_{p+1})}e_+)\\
= -  (-1)^{p+2}L(\varphi_{n_0}\hat{u}(n_0)e_+ + \varphi_{n_0}\overline{\hat{u}(n_0)}e_-,\dots,\varphi_{n_{p+1}}\hat{u}(n_{p+1})e_+ +
  \varphi_{n_{p+1}}\overline{\hat{u}(n_{p+1})}e_-).
\end{multline*}
Using the $\C$-linearity of $L$ and identifying the coefficients of
\[\hat{u}(n_{0})\cdots\hat{u}(n_{\ell})\overline{\hat{u}(n_{\ell+1})}\cdots\overline{\hat{u}(n_{p+1})}\] on each side, we get
\begin{multline}
  \label{eq:447aa}
  L(\varphi_{n_{0}}e_-,\dots,\varphi_{n_{\ell}}e_-,\varphi_{n_{\ell+1}}e_+,\dots,\varphi_{n_{p+1}}e_+)\\
= -  (-1)^{p+2}L(\varphi_{n_{0}}e_+,\dots,\varphi_{n_{\ell}}e_+,\varphi_{n_{\ell+1}}e_-,\dots,\varphi_{n_{p+1}}e_-).
\end{multline}
If $p$ is even, $\ell = \frac{p+2}{2}$ and $n_0 = n_{\ell+1},\dots,n_\ell = n_{p+1}$, we get by symmetry
\[L(\varphi_{n_{0}}e_+,\dots,\varphi_{n_{\ell}}e_+,\varphi_{n_{0}}e_-,\dots,\varphi_{n_{\ell}}e_-) = 0\] 
which, recalling \eqref{4311}, implies \eqref{446}.

(iii) Decompose
\[L(U,\dots,U) =
\sum_{\ell=-1}^{p+1}\bin{p+2}{\ell+1}\sum_{n_0,\dots,n_{p+1}}L(\Pin{0}^+U,\dots,\Pin{\ell}^+U,\Pin{\ell+1}^-U,\dots,\Pin{p+1}^-U).\]
According to (ii), if $p$ is even, and $\ell=p/2$, we may assume that in the corresponding sum $\{n_0,\dots,n_\ell\} \neq
\{n_{\ell+1},\dots,n_{p+1}\}$. 
In order to solve \eqref{447} we 
define $\tilde{L}$ as the symmetric $(p+2)$-linear form associated to the homogeneous
map
\begin{multline}\label{eq:447ab}
U\to
i\sum_{\ell=-1}^{p+1}\bin{p+2}{\ell+1}\sum_{n_0,\dots,n_{p+1}}\Dcal_\ell(n_0,\dots,n_{p+1})^{-1}\\
\times L(\Pin{0}^+U,\dots,\Pin{\ell}^+U,\Pin{\ell+1}^-U,\dots,\Pin{p+1}^-U)
\end{multline}
with
$\Dcal_\ell(n_0,\dots,n_{p+1})$  given by \eqref{4317a}.
By  \eqref{713} we have  an estimate \[\abs{\Dcal_\ell(n_0,\dots,n_{p+1})}^{-1} \leq C\max(n_0,\dots,n_{p+1})^{N_0}\] for all indices 
$n_0,\dots,n_{p+1}$, except those for which $p$ is even, $\ell = p/2$ and $\{n_0,\dots,n_\ell\}=
\{n_{\ell+1},\dots,n_{p+1}\}$, which are anyway excluded in the sum \eqref{447ab}
by the property \eqref{446}. 
 This shows 
 that $\tilde{L}$ satisfies \eqref{441} 
 with $\rho, \mu$ replaced by $(\rho-N_0,\mu +N_0)$. Moreover $\tilde{L}$ also satisfies \eqref{442}. 
 To prove that $\tilde{L}$ is in $\Lt{-\rho+N_0}{p,+}$, it
remains to show that \eqref{443a} holds for $\tilde{L}$ with the $+$ sign. Write the $n_0,\dots,n_{p+1}$ sum in \eqref{447ab}
expressing $\Pin{j}^+U = \hat{u}(n_j)\varphi_{n_j}e_+$,  $\Pin{j}^-U = \overline{\hat{u}(n_j)}\varphi_{n_j}e_-$. We get
\begin{multline}
  \label{eq:447ac}
\sum_{n_0,\dots,n_{p+1}}\Dcal_\ell(n_0,\dots,n_{p+1})^{-1}
\hat{u}(n_{0})\cdots\hat{u}(n_{\ell})\overline{\hat{u}(n_{\ell+1})}\cdots\overline{\hat{u}(n_{p+1})}\\
\times  L(\varphi_{n_{0}}e_+,\dots,\varphi_{n_{\ell}}e_+,\varphi_{n_{\ell+1}}e_-,\dots,\varphi_{n_{p+1}}e_-).
\end{multline}
The corresponding expression for $\tilde{L}(SU,\dots,SU)$ is (changing the indices in the summation)
\begin{multline*}
  (-1)^{p+2}\sum_{n_0,\dots,n_{p+1}}\Dcal_\ell(n_{\ell+1},\dots,n_{p+1},n_0,\dots,n_{\ell})^{-1}\\
\times\overline{\hat{u}(n_{\ell+1})}\cdots\overline{\hat{u}(n_{p+1})}\hat{u}(n_{0})\cdots\hat{u}(n_{\ell})\\
\times
L(\varphi_{n_{\ell+1}}e_+,\dots,\varphi_{n_{p+1}}e_+,\varphi_{n_{0}}e_-,\dots,\varphi_{n_{\ell}}e_-).
\end{multline*}
If we use \eqref{447aa} and the fact that
\[\Dcal_\ell(n_0,\dots,n_\ell,n_{\ell+1},\dots,n_{p+1}) = -\Dcal_\ell(n_{\ell+1},\dots,n_{p+1},n_0,\dots,n_\ell),\] 
we see that we recover expression \eqref{447ac} i.e.\ that \eqref{443a} holds. Finally, the fact that \eqref{447} holds trus
follows from the definition of $\tilde{L}$ and $\Dcal_\ell$.

(iv) 
By the first remark after Definition~\ref{216} and \eqref{2127} for $ k = 0 $,  
any  operator $ M (U;t) $ in the class $\sMM{m}{K,K',q}{N}$, 
$ 0 \leq q \leq N $,  is bounded from $\Hds{s}$ to
$\Hds{s-m}$ for any $s$ large enough,  and 
$$
\norm{M(U;t)U(t,\cdot)}_{\Hds{s-m}}\leq C\Gcals{\sigma}{K',q}{U}\norm{U(t,\cdot)}_{\Hds{s}} +
C\Gcals{\sigma}{K',N}{U}  \Gcals{s}{K',1}{U} \, . 
$$
Notice that the last term comes only from the non-homogeneous part of 
$ M (U;t) $. Combining this estimate with (i) of the lemma, 
  we get that, at any time $t$,
\begin{multline*}
|L(U,\dots ,M(U;t)U,\dots,U)(t,\cdot)| \leq C \norm{U(t,\cdot)}_{\Hds{s}}^{p+1} \norm{M(U;t)U(t,\cdot)}_{\Hds{s-m}}\\ \leq
C \Gcals{\sigma}{K',q}{U}\norm{U(t,\cdot)}_{\Hds{s}}^{p+2} + 
C \Gcals{\sigma}{K',N}{U} \Gcals{s}{K',1}{U} \norm{U(t,\cdot)}_{\Hds{s}}^{p+1}
\end{multline*}
so that all the terms with $ p + q \geq N $ contribute to \eqref{resN}. 
Thus, 
it is sufficient to check that  
if $ \tilde M $ is in $\MtM{m}{\tilde q}$, $ q \leq \tilde q < N - p  $,  then
\[
(U_0,\dots,U_{p+ \tilde q+1})\to L(U_0,\dots,U_{p}, {\tilde M}(U_{p+1},\dots,U_{p+ \tilde q})U_{p+\tilde q+1})
\]
defines after symmetrization a multilinear form $\tilde{L}$ satisfying the estimates of Definition~\ref{441} for elements of
$\Lt{-\rho+m}{p+ \tilde q,\pm}$. The multilinear form $ \tilde{L} $ is denoted by $  L_{q' } $, 
 $ q' = \tilde q - q $, $ 0 \leq q' \leq N - p - q -1 $, in the statement of (iv). 
By  \eqref{2125}-\eqref{2126} and since $ L $ satisfies \eqref{441}, 
we get 
\begin{multline*}
| L(\Pi_{n_0}U_0,\dots, \Pi_{n_p} U_p, \tilde M( \Pi_{n_{p+1}} U_{p+1},\dots, \Pi_{n_{p+ \tilde q}} U_{p+ \tilde q})
\Pi_{n_{p + \tilde q+1}} U_{p + \tilde q+1})| \\ 
\leq C \sum_{n'}
\max(n_0, \ldots, n_p, n')^{2s-\rho} \max\nolimits_3 (n_0, \ldots, n_p, n')^{\rho+\mu} \\
\times (n'+ n_{p+1} + \ldots + n_{p+ \tilde q+1})^m \prod_{0}^{p+ \tilde q+1}\norm{\Pin{\ell}U_\ell}_{L^2} 
\end{multline*}
where the sum is restricted to the indices 
$ n' = \sum_{p+1}^{p+ \tilde q+1} \epsilon_\ell n_\ell $ for some choice of the signs
$ \epsilon_\ell \in \{ -1 , 1 \} $. 
The right hand side of the previous formula is bounded by
$$
C \max(n_0, \ldots, n_{p+ \tilde q+1})^{2s- \rho + m } \max\nolimits_3 (n_0, \ldots, n_{p+ \tilde q+1})^{\rho+\mu} 
\prod_{0}^{p+ \tilde q+1}\norm{\Pin{\ell}U_\ell}_{L^2} 
$$
and therefore $\tilde{L}$  satisfies \eqref{441} 
with $ \rho  $ and $ p $ replaced by $ \rho - m $ and $ p + \tilde q $. 
Since $ L $ satisfies  \eqref{442} and $ M $ satisfies \eqref{2126} we 
derive  as well  that \eqref{U-comp-M} satisfies the corresponding condition \eqref{442}.
 We have still to check that condition \eqref{443} (or \eqref{443a}) holds with the $\mp$ (resp.\ $\pm$)
sign if $M$ satisfies \eqref{317} and \eqref{319} (resp.\ \eqref{317} and \eqref{3110}). Let us treat the first case, i.e.\
show that  if $\tilde{M}$ denotes
some multilinear component of $M$, then 
\begin{equation}
  \label{eq:447ad}
  L(SU,\dots,SU,\tilde{M}(SU,\dots,SU)SU) = \mp L(U,\dots,U,\tilde{M}(U,\dots,U)U)
\end{equation}
for any $U$ such that $SU= -\bar{U}$. Notice first that by \eqref{317} and $S \bar{U} = - U $, 
\[
\overline{\tilde{M}(U,\dots,U)U} = S\tilde{M}(U,\dots,U)S\bar{U} = -S(\tilde{M}(U,\dots,U)U),
\] 
so
that $\tilde{M}(U,\dots,U)U$ satisfies the same property as $U$. Moreover since $ M $ satisfies 
the reversibility condition \eqref{319}, 
lemma~\ref{non-hombis} implies that the homogeneous component $ \tilde{M} $ satisfies the 
reversibility condition \eqref{319p}, namely 
$\tilde{M}(SU,\dots,SU)SU = -S(\tilde{M}(U,\dots,U)U)$. In conclusion \eqref{447ad} follows
by \eqref{443} applied with $U_0=\cdots=U_p = U$,  $ U_{p+1}=\tilde{M}(U,\dots,U)U$. 
\end{proof}

Finally, let us state a Sobolev energy inequality that will be the starting point of 
the normal forms reduction giving Theorem~\ref{311}.

\begin{lemma}
  \label{443} { \bf (First energy inequality)}
Let $ \rho \in  \N$, $\rho\gg N $. 
There are indices $s\geq s_0\gg\sigma\gg K\gg\rho$ as in \eqref{211} and a family $(L_p)_{1\leq p\leq N-1}$ of 
multilinear forms 
belonging to $\Lt{-\rho}{p,-}$ for any $s\geq s_0 $, such that the following holds:

Let $ U $ be a solution of system  \eqref{3112}
defined for all times  $ ]-T,T[ $ 
satisfying the properties of Theorem~\ref{311}, in particular, $U(t) $ is an even function of $x$
satisfying $SU(t) = - \bar{U}(t)$,  for any $ t $. Let $W = \opbw(Q(U; t, \cdot))U$ defined in
Proposition~\ref{322}, $V = \Phi^\star_UW$ given by Proposition~\ref{411}, $\tilde{V} = \of{-1}V$ introduced in
Proposition~\ref{421} and finally \[\tilde{V}^1 = \exp(\opbw(B(U;t,\xi))) \tilde{V}\] 
 given by Proposition~\ref{431}. Then  for any
$s\geq s_0$,
\begin{multline}
  \label{eq:446a}
  \frac{d}{dt} \int_\Tu\abs{\abs{D}^s\tilde{V}^1(t,x)}^2\,dx = \sum_1^{N-1}L_p(U,\dots,U) 
  + O\bigl( \norm{U(t,\cdot)}_{\Hds{s}}^{N+2}\bigr)
\end{multline}
 as long as $ \norm{U(t,\cdot)}_{\Hds{s}} $
stays small enough. Moreover 
  \begin{equation}
    \label{eq:447a}
    C_s^{-1}\norm{\tilde{V}^1}_{\Hds{s}} \leq \norm{U}_{\Hds{s}}\leq  C_s\norm{\tilde{V}^1}_{\Hds{s}} 
  \end{equation}
 for some   $ C_s > 0 $.
\end{lemma}

\textbf{Remark}: 
Since  $ SU(t) = -\bar{U}(t) $ 
and each of the  
operators $ \opbw(Q) $, $ \Phi^\star_U $,  $ \of{-1} $, $ \exp(\opbw(B(U;t,\xi))) $, 
satisfies the   anti-reality condition  \eqref{317} 
(see Proposition~\ref{322}, Lemmas \ref{lem:flowp} and ~\ref{420a}, Proposition \ref{431}), 
then the function 
$$
\tilde{V}^1 = \exp(\opbw(B(U;t,\xi))) \circ  \of{-1} \circ \Phi^\star_U \circ \opbw(Q(U; t, \cdot))U 
$$
satisfies  $S\tilde{V}^1(t) = - \overline{\tilde{V}^1(t)}$ as well, 
by the second remark after Definition \ref{Def:RPR}.


\medskip
\begin{proof}
Notice first that by Proposition~\ref{632} 
and recalling \eqref{212}-\eqref{213}, we have the bound 
\be\label{eq:boundUds}
{\cal G}_{K,1}^s (U, t) = \nnorm{U(t,\cdot)}_{K,s}  \leq  C_s  \norm{U (t, \cdot)}_{\Hds{s}} \, . 
\ee
According to system \eqref{432}, the left hand side of \eqref{446a} is equal to 
  \begin{multline}
    \label{eq:448}
2\Re i\int_\Tu\overline{(\abs{D}^s\tilde{V}^1)}\Bigl[\opbw\bigl((1+\zu(U;t))\mk(\xi)\Kcal +
H^1(U;t,\xi)\bigr)\abs{D}^s\tilde{V}^1\Bigr]\,dx\\
+ 2\Re i\int_\Tu\overline{(\abs{D}^s\tilde{V}^1)}(\abs{D}^sR_1(U;t)\tilde{V}^1)\,dx\\
+ 2\Re i\int_\Tu\overline{(\abs{D}^s\tilde{V}^1)}(\abs{D}^sR_2(U;t)U)\,dx \, .
  \end{multline}
In the first integral above, the contribution of $\opbw\bigl((1+\zu(U;t))\mk(\xi)\Kcal\bigr)$  
is zero, as this operator is self-adjoint on
$L^2$. In the same way, since $H^1_q$, $1\leq q\leq N-1$ is real valued in the decomposition \eqref{431} of $H^1(U;t,\xi)$,
the corresponding contribution to \eqref{448} vanishes. 
The same is true for $\Re H^1_N(U;t,\xi)$, so that the first integral
in \eqref{448} is actually equal to the contribution coming from $i(\Im H^1_N)(U;t,\xi)$.  
Since by Proposition~\ref{431} $\Im H^1_N$ is in $\GrM{0}{K,\ku,N}$
we get by \eqref{2123} (with $ k = 0 $, $ s = 0 $ and $ m = 0 $) and  \eqref{boundUds}  
 that the first term in \eqref{448} is, 
for $s$ large enough and $ K \geq \ku $, bounded by 
 $$
O \bigl(\Gcals{\sigma}{\ku,N}{U}\norm{\tilde{V}^1(t,\cdot)}^2_{\Hds{s}}\bigr) =
O \bigl( \norm{U (t, \cdot)}_{\Hds{s}}^N  \norm{\tilde{V}^1(t,\cdot)}^2_{\Hds{s}}\bigr) \, . 
 $$
Consider the last integral in \eqref{448}. Since $R_2(U;t)$ is in
$\RrM{-\rho+\frac{3}{2}}{K,\ku,N}$, we get, by \eqref{2117} with $ k = 0 $, 
and  \eqref{boundUds}, 
a bound in
$$
O \bigl( \norm{U (t, \cdot)}_{\Hds{s}}^{N+1}  \norm{\tilde{V}^1(t,\cdot)}_{\Hds{s}}\bigr)  
$$
for $ s, \sigma $ satisfying \eqref{211}. 
In the second integral in \eqref{448}, we may decompose the smoothing operator 
$$ 
R_1(U; t) = \sum_{p=1}^{N-1}R_{1,p}(U,\dots,U) + R_{1,N}(U;t) 
$$ 
with $R_{1,p}$ in $\Rt{-\rho+\frac{3}{2}}{p}$ and 
$R_{1,N}(U;t)$ in $\Rr{-\rho+\frac{3}{2}}{K,\ku,N}$. By \eqref{2117} with $ k = 0 $, and  \eqref{boundUds},
the contribution coming from the last term is bounded by 
$$
O \bigl( \norm{\tilde{V}^1(t,\cdot)}_{\Hds{s}}^2 \Gcals{s}{\ku,N}{U}\bigr) = 
O \bigl( \norm{\tilde{V}^1(t,\cdot)}_{\Hds{s}}^2  \norm{U(t,\cdot)}_{\Hds{s}}^N \bigr) \, . 
$$
 On the other hand, as we have remarked after the statement of lemma \ref{443} that $\overline{\tilde{V}^1}
= -S\tilde{V}^1$, we may write the contribution of the term $ R_{1,p} $ 
to the second integral in \eqref{448} as \eqref{444}. 
 By lemma \ref{non-hombis}, since 
 since $ R_1 (U, t ) $ satisfies the reversibility property \eqref{319} 
 the homogeneous smoothing operators $ R_{1,p} $  satisfy \eqref{319p}. 
We
have seen in the example following Definition~\ref{441} that we get an expression
$\tilde{L}_p(\tilde{V}^1,U,\dots,U,\tilde{V}^1)$ with $\tilde{L}_p$ in $\Lt{-\rho}{p,-} $.  
Consequently, we have written the left hand side of \eqref{446a} as
\begin{multline} \label{eq:449}
\sum_{p=1}^{N-1} \tilde{L}_p(\tilde{V}^1,U,\dots,U,\tilde{V}^1) \\ 
+
O \bigl( \norm{\tilde{V}^1(t,\cdot)}_{\Hds{s}}^2  \norm{U(t,\cdot)}_{\Hds{s}}^N + 
\norm{U (t, \cdot)}_{\Hds{s}}^{N+1}  \norm{\tilde{V}^1(t,\cdot)}_{\Hds{s}} \bigr) \, . 
\end{multline}
To deduce \eqref{446a} from \eqref{449}, it is
sufficient to show that \eqref{447a} holds and that the above terms 
$\tilde{L}_p(\tilde{V}^1,U,\dots,U,\tilde{V}^1)$  may be expressed as in the right hand side of
\eqref{446a}. Recall that by \eqref{431}, Proposition~\ref{421}, Proposition~\ref{411}, 
Proposition~\ref{322},
\begin{equation}
  \label{eq:4410}
  \tilde{V}^1 = \exp(\opbw(B(U;t,\xi)))\of{-1}\Phi_U^\star \opbw(Q(U; t, \cdot))U \, .
\end{equation}
Let us show that $\tilde{V}^1 = U + M(U;t)U$ for some $M$
in $\sMM{}{K,\ku,1}{N}$ satisfying \eqref{317} and \eqref{3110}, and that $\norm{\tilde{V}^1}_{\Hds{s}}\sim \norm{U}_{\Hds{s}}$ for all large
enough $s$, when $ \norm{U (t, \cdot)}_{\Hds{s}} $  is small enough.

By Proposition~\ref{322}, and the Definition \ref{321} of $\sE{0}{K,\rho,1}{N}$, 
there is a matrix of symbols $Q_1(U; t, \cdot)$ in
$\sGM{0}{K,\rho,1}{N}$ satisfying \eqref{312} and \eqref{315} such that 
\be\label{eq:W-U}
W = U + \opbw(Q_1(U; t, \cdot))U \, .
\ee
By the third
remark following
Definition~\ref{216}, $\opbw(Q_1(U; t, \cdot))$ is an 
element of $\sMM{0}{K,\rho,1}{N}$ and it satisfies moreover \eqref{317} and
\eqref{3110}. By
Proposition~\ref{215}, a bound of the form $\norm{\opbw(Q_1(U; t, \cdot))U}_{\Hds{s}} =
O\bigl(\norm{U}_{\Hds{s}}^2 \bigr)$  holds if $s$ is large
enough, so that $\norm{W}_{\Hds{s}}\sim \norm{U}_{\Hds{s}}$, as long as $ \norm{U}_{\Hds{s}} $
 stays small enough.

Formula \eqref{243} allows to write $V=\Phi_U^\star W$  as
\begin{equation}
  \label{eq:4411}
  V = W +\sum_{p=1}^{N-1} M_p(U,\dots,U)W + M_N(U;t)W
\end{equation}
with $M_p$ in $\MtM{}{p}$ and $M_N$ in $\MrM{}{K,\rho,N}$. 
Moreover, it follows from the definition \eqref{def-paracomp} of 
$\Phi_U^\star$ and  \eqref{425-first} that
$\norm{V}_{\Hds{s}} \sim \norm{W}_{\Hds{s}}$. 
Inserting  in \eqref{4411} 
the expression  of $W$  obtained in \eqref{W-U}, and using the last remark after Definition~\ref{216}, we deduce that $ V $ may be written as 
\begin{equation} \label{eq:4412}
V = U + M(U; t)U =   U +\sum_{p=1}^{N-1}M_p(U,\dots,U)U + M_N(U;t)
\end{equation}
for some other $M_p$ in 
$\MtM{}{p}$ and $M_N$ in  $\MrM{}{K,\rho,N}$, i.e. 
for some operator $M(U;t) $  in the space  $\sMM{}{K,\rho,1}{N}$. Let us check that
$M(U;t)$ satisfies \eqref{317} and \eqref{3110}. Recalling that  
$ V = \Phi_U(t;\cdot)^\star W $ and that $ \Phi_U(t;\cdot)^\star $ satisfies  \eqref{424a}, 
we get  that $ M_p $ and $ M_N $ in
\eqref{4411}  satisfy \eqref{3110} and \eqref{317}. Since 
also  $\opbw(Q_1(U; t, \cdot))$  in \eqref{W-U} satisfies 
\eqref{3110}  and \eqref{317}
we get that  also the operators 
$ M_p $ and $ M_N $ in  \eqref{4412} satisfy \eqref{3110} and \eqref{317}.

Consider next $\tilde{V}=
\of{-1}V$. By \eqref{425}, we have 
$$
\norm{\tilde{V}}_{\Hds{s}} \sim \norm{V}_{\Hds{s}} \sim \norm{U}_{\Hds{s}}
$$
 if
$ \norm{U}_{\Hds{s}} $  is small
enough. Moreover, $\of{-1}V$ may be expressed as in  \eqref{4219}, \eqref{4219a}, so that, combining with \eqref{4412}, we get a
similar expression for $\tilde{V}$ in terms of $U$, with a new $M_p$ belonging to $\MtM{}{p}$, $M_N$ in
$\MrM{}{K,\ku,N}$. Since the $F_\ell$ in Proposition~\ref{421} satisfy \eqref{311} and \eqref{315}, $\opbw(iF(U))$ satisfies
\eqref{317} and \eqref{3110}, and
$\of{-1}$ does so as well. By composition, we deduce that the $M_p$'s satisfy also \eqref{317} and \eqref{3110}. Finally, we
consider $\tilde{V}^1$ given by the expression 
$\tilde{V}^1 =
\exp(\opbw(B(U;t,\xi)))\tilde{V}$. 
By \eqref{434}, \eqref{447a} holds. Moreover \eqref{433} and the above expression of $\tilde{V}$
in terms of $U$ allows us to obtain an expression for $\tilde{V}^1$ in terms of $U$ similar to \eqref{4412}. To conclude the
proof of \eqref{446a}, we plug these expressions inside \eqref{449}. By (iv) of lemma~\ref{442}, we express
\eqref{449} in terms of new multilinear forms $L_p(U,\dots,U)$ belonging to $\Lt{-\rho}{p,-}$ and of remainders 
satisfying $O(\norm{U(t;\cdot)}_{\Hds{s}}^{N+2})$. This concludes the proof.
\end{proof}
\begin{proof1}{Proof of Theorem~\ref{311}}
The proof is based on a bootstrap argument. 
  Let $\kappa$ be fixed in $]0,+\infty[-\Ncal$ where the set $ {\cal N} $ is 
  provided by Proposition~\ref{711}. Take $N$ an arbitrary positive
  integer. 
  We shall prove that there are constants $s_0\gg K\gg N$ and for any $s\geq s_0$, 
  there are constants 
  $$
  \epsilon_0>0 \, , \quad c > 0 \, , \quad 0 < A_0 < A_1 < \dots < A_K \, , 
  $$
 such
  that if a solution $ U(t ) $ of system \eqref{3112} exists on some time  interval $I = ]-T,T[$ and satisfies, together with 
  the properties of Theorem~\ref{311}, the bounds
  \begin{equation}
   \label{eq:4413}
   \sup_{]-T,T[} \norm{\partial_t^k U(t;\cdot)}_{\Hds{s-\frac{3}{2}k}} \leq A_k\epsilon,\quad k = 0,\dots,K,
  \end{equation}
then, for $\epsilon\in ]0,\epsilon_0[$ and $T<c\epsilon^{-N}$, the solution $ U(t) $ actually satisfies 
on the same time interval the improved bounds
\begin{equation}
  \label{eq:4414}
   \sup_{]-T,T[} \norm{\partial_t^k U(t;\cdot)}_{\Hds{s-\frac{3}{2}k}} \leq \frac{A_k}{2}\epsilon,\quad k = 0,\dots,K.
\end{equation}
This directly  implies, taking into account the classical results of local existence (see Schweizer~\cite{Schw})  that
 the solution $ U $ may be extended up to an interval of length $c\epsilon^{-N}$, and that
\eqref{3116} holds.

To prove that \eqref{4413} implies \eqref{4414}, let us show first that we may find recursively 
multilinear forms 
\begin{equation}\label{eq:expr}
\begin{split}
\tilde{L}_p \in 
\Lt{ - \rho + (N_0 +m) (p-1) + N_0}{p,+} \, , \quad 1\leq p\leq N -1 \, , \\
L_p^{(q)} \in \Lt{-\rho+(N_0 +m)q }{p,-} \, , \quad q+1\leq p\leq N -1 \, , 
\end{split}
\end{equation}
for 
$ s\gg K\gg \rho\gg N $, 
such that, for $q = 0,\dots,N-1$
\begin{multline}
  \label{eq:4415}
\frac{d}{dt}\Bigl[\int_\Tu \abs{\abs{D}^s\tilde{V}^1(t;\cdot)}^2\,dx +
\sum_{p=1}^q\tilde{L}_p(U(t;\cdot),\dots,U(t;\cdot))\Bigr]\\
= \sum_{p=q+1}^{N-1}L_p^{(q)}(U(t;\cdot),\dots,U(t;\cdot)) 
+ O\big( \norm{U(t;\cdot)}_{\Hds{s}}^{N+2} \big) \, . 
\end{multline}
The number $ m \geq 3/2 $ in \eqref{expr}
is the order of the operator $M(U; t)$ in $\sMM{m}{K,1,1}{N}$ defined in \eqref{4312}.  
Notice  that 
in \eqref{expr} at each step of the iteration there is a loss of derivatives proportional to 
$ N_0 $ due to the small divisors. It will be compensated  
by  taking  the regularizing index $ \rho $ large enough with respect to the number of steps $ N $ and $ N_0 $.

Notice 
that if $q=0$, then \eqref{4415}
follows from  \eqref{446a} with $ L_p^{(0)} = L_p $. 
Then we proceed by induction. Assume that \eqref{4415} holds at rank $q-1$ for some $q\geq 1$, and define $\tilde{L}_q$ to be the multilinear form of
$\Lt{-\rho + (N_0+m)(q-1)+N_0}{q,+}$ given by (iii) of lemma~\ref{442} applied to $L = L_q^{(q-1)}$. 
Using \eqref{4312} we have 
\begin{multline}
  \label{eq:4416}
\frac{d}{dt}\tilde{L}_q(U(t;\cdot),\dots,U(t;\cdot)) = i \sum_{j=0}^{q+1} 
\tilde{L}_q( \underbrace{U,\dots, U}_{j}, \mk(D)\Kcal U,\dots,U)\\
+  i \sum_{j=0}^{q+1} \tilde{L}_q(\underbrace{U,\dots, U}_{j},M(U;t) U,\dots,U)
\end{multline}
for some operator $M(U;t)$ in $\sMM{m}{K,1,1}{N}$ for some $ m \geq 3/2 $, 
satisfying \eqref{319}. Using (iv) of lemma~\ref{442} and \eqref{boundUds}, we 
write the last sum in \eqref{4416}
as a contribution to the right hand side of \eqref{4415} at rank $ q $, 
so that
adding \eqref{4415} at rank $q-1$ and \eqref{4416}, and using \eqref{447}, we get \eqref{4415} at rank $q$.

We integrate next \eqref{4415} at rank $q=N-1$ from 0 to $t$. We get (in the case $t\geq 0$)
\begin{multline*}
  \norm{\tilde{V}^1(t;\cdot)}^2_{\Hds{s}} + \sum_{p=1}^{N-1}\tilde{L}_p(U(t;\cdot),\dots,U(t;\cdot))\\
\leq  \norm{\tilde{V}^1(0;\cdot)}^2_{\Hds{s}} + \sum_{p=1}^{N-1}\tilde{L}_p(U(0;\cdot),\dots,U(0;\cdot)) 
+ C\int_0^t  \norm{U(\tau;\cdot)}_{\Hds{s}}^{N+2} \,d\tau
\end{multline*}
as long as \eqref{4413} holds. Taking into account \eqref{447a} and (i) of lemma~\ref{442}, that applies to  the multilinear 
forms  $ \tilde{L}_p $,   $ 1 \leq p \leq N -1 $,  in \eqref{expr},  if 
$\rho$ is large enough relatively to $N_0, N$ (i.e. $ \rho > (N_0 + m ) N + \frac12 $),  
we get, for some new constant $ C $,
\be
  \label{eq:4417}
  \norm{U(t;\cdot)}_{\Hds{s}}^2 \leq C \norm{U(0;\cdot)}_{\Hds{s}}^2 + C\norm{U(t;\cdot)}_{\Hds{s}}^3 
  + C\int_0^t
\norm{U(\tau;\cdot)}_{\Hds{s}}^{N+2}\,d\tau
\ee
as long as \eqref{4413} holds. Taking $A_0$ large enough relatively to $C$, we may assume that
\be\label{eq:U0}
\norm{U(0;\cdot)}_{\Hds{s}}^2\leq \frac{A_0^2}{8C}\epsilon^2 \, . 
\ee 
Moreover, by \eqref{4413}, the integrand in the last term in the right hand side of
\eqref{4417} is bounded at any time $\tau$ by $ (A_0\epsilon)^{N+2}$. If $\epsilon<\epsilon_0$ is small enough, 
 and 
$$
T<c\epsilon^{-N} \qquad 
 {\rm with} \qquad  
 cCA_0^N<\frac{1}{16} \, , 
 $$ 
 it follows by \eqref{4417}, \eqref{U0} and \eqref{4413}  that estimate \eqref{4414} with $k=0$ holds, for
  any $t\in ]-T,T[ $. Finally, by \eqref{boundUds}, we also deduce that 
 the derivatives  $\partial_t^kU(t;\cdot)$ satisfy \eqref{4414}, 
 if $\epsilon$ is small enough, $s\gg K$, and taking $A_{k+1}$ large enough
  relatively to $A_k$, $k=0,\dots,K-1$.
\end{proof1}

%% file: chap5.tex
\chapter[Dirichlet-Neumann  problem]{The Dirichlet-Neumann paradifferential problem}\label{cha:5}

In order to be able
to derive system \eqref{3112} from the water waves equations \eqref{113} in Chapter~\ref{cha:6}, we shall need a
paradifferential a representation of the Dirichlet-Neumann operator $G(\eta)$. This will be obtained in  section~\ref{sec:61}
following  the method of
   Alazard-Métivier~\cite{AM}, Alazard-Burq-Zuily~\cite{ABZ1, ABZ2} (see also Alazard-Delort~\cite{AD2}). 
A difference with the above references is  that
  we need to obtain such a representation in terms of the classes of paradifferential operators that have been introduced in
  Chapter~\ref{cha:2}. To do so, we shall study in the present chapter  the Dirichlet-Neumann boundary value problem in a strip, introducing paradifferential versions of the
  Poisson operators of Boutet de Monvel~\cite{BdM1,BdM2,BdM3}.

\section{Paradifferential and para-Poisson operators}\label{sec:51}

 We shall work on the strip
\[
\Bcal = \big\{ (z,x)\in \R \times \Tu \,  ; -1 \leq z \leq 0 \big\} \, .
\]
The functions we shall consider will depend on $ (z, x) $ in $ \Bcal $ and on the time parameter $t$ staying in some interval $[-T,T]$. 
We shall use the notations of Chapter~\ref{cha:2} for paradifferential and smoothing operators, except that here our functions and symbols will also
depend on the parameter $z\in [-1,0]$. The spaces measuring the smoothness of the functions at hand are the following.

\begin{definition}{\bf (Space of functions on a strip)} \label{spacesFsj}
For $j$ in $\N$ and $s$ in $\R$, we define the space
\begin{equation}
  \label{eq:511}\begin{split}
  \index{Fz@$F_j^s$ (Space of functions on a strip)} F_j^s = \big\{\Phi \in L^\infty([-1,0],\Hds{s}(\Tu;\C)) \ {\rm such \ that} \qquad \qquad  \  \ \\
 \qquad \partial^{j'}_z\Phi \in  
  L^\infty([-1,0],\Hds{s-j'}(\Tu;\C)),  \, \forall  j' = 1,\dots,j \big\}
\end{split}\end{equation}
together with the natural norm
\begin{equation}
  \label{eq:512}
  \norm{\Phi}_{F_j^s} = \sum_{j'=0}^j\norm{\partial_z^{j'}\Phi}_{ L^\infty([-1,0],\Hds{s-j'})} \, .
\end{equation}
We set $F^\infty_\infty \stackrel{\textrm{def}}{=} \bigcap_{s,j}F_j^s$.
\end{definition}

We extend first the Definitions~\ref{211} and~\ref{212bis} to symbols that depend now also on the variable $ z $. We shall use in
this section only classes of autonomous symbols, and not  
versions in a strip of the class of  non-homogeneous symbols of Definition~\ref{212}. In
particular, the time dependence of our (non-homogeneous) symbols will be only through their dependence in the arguments $\Phi,
\Psi, U$,...\  even if we do not repeat this condition in all statements. 

We denote by $\CKHb{\sigma}{j} =
\bigcap_0^KC^k(I,F^{\sigma-\frac{3}{2}k}_j)$. We generalize notations \eqref{212}, \eqref{213} to functions depending
on $z$ setting
\begin{equation}
  \label{eq:513}
  \begin{split}
    \nnorm{\Phi(t,\cdot)}_{K,\sigma,j} &= \sum_{k=0}^K\norm{\partial_t^k\Phi(t,\cdot)}_{F_j^{\sigma-\frac{3}{2}k}}\\
\index{Gm@$\Gcals{\sigma,j}{K,p}{\Phical}$ (Power of norm of a function on a strip)} \Gcals{\sigma,j}{K,p}{\Phical} &=
\nnorm{\Phi(t,\cdot)}^p_{K,\sigma,j}\\
\index{Gma@$\Gcalsm{\sigma,j}{0,p}{\Phical}$ (Product of norms of functions on a strip)} \Gcalsm{\sigma,j}{0,p}{\Phical} &=
\prod_{p'=1}^p \norm{\Phi_{p'}}_{F^\sigma_j}
  \end{split}
\end{equation}
where in the last formula $\Phical$ stands for a family $\Phical = (\Phi_1,\dots,\Phi_p)$ of functions of $(z, x)$ in
$[-1,0] \times \Tu $. When $p=0$, we set $\Gcal^{\sigma,j}_{K,0}=1$ by convention.

\begin{definition}
  \label{511} {\bf (Symbols in a strip)}
(i) Let $m$ be in $\R$, $p$ in $\N^* $. We denote by \index{Ge@$\Gtb{m}{p}$ (Space of homogeneous symbols in a strip)}
$\Gtb{m}{p}$ the space of symmetric $p$-linear maps from $(F^\infty_\infty)^p$ to the space of functions that are 
 $C^\infty$ in $(z, x, \xi ) \in [-1,0] \times \Tu \times \R $
\begin{equation}
  \label{eq:514}
  \Phical = (\Phi_1,\dots,\Phi_p) \to \bigl((z,x,\xi) \to a(\Phical;(z,x,\xi))\bigr)
\end{equation}
satisfying the following: There is $\mu\geq 0$ and  for any $j, \alpha, \beta$ in $\N$, there is $C>0$
and for any $\Phical$ in $(F^\infty_\infty)^p$, any $n = (n_1,\dots,n_p)$ in $(\N^*)^p$
\begin{equation}
  \label{eq:515}
\abs{\partial^j_z\partial^\alpha_x\partial^\beta_\xi a(\Pin{}\Phical;z,x,\xi)}
\leq C\abs{n}^{\mu+\alpha+j}\absj{\xi}^{m-\beta}\Gcalsm{0,j}{0,p}{\Phical}.
\end{equation}
Moreover, we assume that if $(n_0,\dots,n_p)$ is in $\N\times (\N^*)^p$ and if
\begin{equation}
  \label{eq:515a}
  \Pin{0}a(\Pin{1}\Phi_1,\dots,\Pin{p}\Phi_p;\cdot)\not\equiv 0,
\end{equation}
then there is a choice of signs $\epsilon_j\in\{-1,1\}$ such that  $\sum_0^p \epsilon_jn_j = 0$.

When $p=0$, we denote by $ \Gtb{m}{0}$ 
the space of constant coefficient symbols $(z, \xi) \to a(z, \xi)$, that satisfy inequality
\eqref{515} for $\alpha =0$, with in the right hand side the $\abs{n}$ factor replaced by one.

(ii) Let $ r > 0 $, $K$ in $\N$. We denote by \index{Gea@$\Grb{m}{K,0,j,p}$  (Space of non-homogeneous symbols in a
  strip)} $\Grb{m}{K,0,j,p}$ the space of maps $(\Phi;z,x,\xi) \to a(\Phi; t, z,x,\xi)$ defined for
\[\Phi\in \index{Bb@$\Brb{k}{0}{j}$ (Ball in the space of functions of $(t,z,x)$)}  \Brb{K}{0}{j}  \stackrel{\textrm{def}}{=}
  \big\{ U\in \CKHb{\sigma_0}{j} \, ; \ \sup_{t\in I}\nnorm{\Phi(t,\cdot)}_{K,\sigma,j} < r 
  \big\}
  \]
for some large enough $\sigma_0$,  $x\in \Tu$, $\xi\in \R$, $z\in [-1,0]$, with complex values, such that for any
$0\leq k\leq K$, any $0\leq j'\leq j$, any $\sigma\geq \sigma_0$, there are $C>0$, $r(\sigma)\in ]0,r[$, and for any $\Phi$
in $\Brsb{K}{0}{\sigma}{j}\cap\CKHb{\sigma}{j}$, any $\alpha, \beta \in \N $ with $\alpha\leq \sigma-\sigma_0$,  
any  $ t \in I $, 
\begin{equation}
  \label{eq:516}
\abs{\partial^k_t\partial^{j'}_z\partial^\alpha_x\partial^\beta_\xi a(\Phi;z,x,\xi)}
\leq C\absj{\xi}^{m-\beta}\Gcals{\sigma_0,j'}{k,p-1}{\Phi}\Gcals{\sigma,j'}{k,1}{\Phi} \, 
\end{equation}
(where, in the case $p=0$, the right hand side should be read as $C\absj{\xi}^{m-\beta} $).

(iii) We denote by \index{Geb@$\sGb{m}{K,0,j,p}{N}$ (Space of symbols in a strip)}  $\sGb{m}{K,0,j,p}{N}$ the space of functions \[(\Phi;z,x,\xi)\to a(\Phi;z,x,\xi)\] that may be written as
\begin{equation}
  \label{eq:517}
   a(\Phi;z,x,\xi) = \sum_{q=p}^{N-1} a_q(\underbrace{\Phi,\dots,\Phi}_{q};z,x,\xi) + a_N(\Phi;z,x,\xi)
\end{equation}
for some $a_q$ in $\Gtb{m}{q}$, $a_N$ in $\Grb{m}{K,0,j,N}$.
\end{definition}
\textbf{Remarks}: $\bullet$ We shall have to consider 
occasionally symbols of $\Grb{m}{K,0,j,p}$
depending on a couple of functions $(\Phi,\Phi')$ with linear dependence on $\Phi'$, instead of a single function $\Phi$
(for example in the Proof of 
Proposition \ref{614}). In
that case, if we replace in the left hand side of \eqref{516} $a(\Phi;z,x,\xi)$ by  $a(\Phi,\Phi';z,x,\xi)$, the right
hand side of this inequality should read 
\begin{multline}
  \label{eq:517a}
C\absj{\xi}^{m-\beta}\Bigl[\Gcals{\sigma_0,j'}{k,p-1}{\Phi}\Gcals{\sigma,j'}{k,1}{\Phi'} \\+ \Gcals{\sigma_0,j'}{k,p-2}{\Phi}\Gcals{\sigma_0,j'}{k,1}{\Phi'}\Gcals{\sigma,j'}{k,1}{\Phi}\Bigr] \, .
\end{multline}
In the same way, we shall consider symbols of $\sGb{m}{K,0,j,p}{N}$, depending on $\Phi, \Phi'$, linear in $\Phi'$, given by
\[ \sum_{q=p}^{N-1} a_q(\underbrace{\Phi,\dots,\Phi}_{q-1},\Phi';z,x,\xi) + a_N(\Phi,\Phi';z,x,\xi)\]
with $a_q$ in $\Gtb{m}{q}$ and $a_N$ in $\Grb{m}{K,0,j,N}$, linear in $\Phi'$.

$\bullet$ We shall also use the notation $a(\Phi,\psi;z,x,\xi)$ for symbols depending on $\Phi$ and linearly on some
function $\psi(t,x)$, independent of $z$. In this case \eqref{517a} has to be replaced by a similar estimate, but with
$\Gcals{\sigma,j'}{k,1}{\Phi'}$ and $\Gcals{\sigma_0,j'}{k,1}{\Phi'}$ replaced respectively by
$\Gcals{\sigma}{k,1}{\psi}$ and $\Gcals{\sigma_0}{k,1}{\psi}$, and a similar 
modification of the notation for the elements of $\Gtb{m}{q}$. 

$\bullet$ If $a$ is a symbol in $\sGb{m}{K,0,j,p}{N}$, we may define the associated 
paradifferential operator $\opbw(a(\Phi;\cdot))$ as in \eqref{2114}, $z$ playing the role of a
  parameter. According to Proposition \ref{215},
  if $\Phi$ is in $\Brb{K}{}{j}$ for
some $\sigma$ such that $\sigma-j$ is large enough (independently of $s$), then the paradifferential operator  
  $\opbw(a(\Phi;\cdot))$ is bounded from $F^s_j$ to $F^{s-m}_{j}$ for any $s$. 

$\bullet$ In Definition \ref{511}, the functions $\Phi_j$ (resp. $\Phi$) belong to $F^\infty_\infty$
(resp.\ are functions of time with values in a space $F^{\sigma_0}_j$), so are $\C$ valued. We shall use also the same notation for the
above spaces of symbols, when we consider function $\Phi_j, \Phi$ that are $\C^2$ valued instead of $\C$ valued.

\smallskip

Let us define as well smoothing operators acting on the spaces $F^s_j$.

\begin{definition}
  \label{512} 
  {\bf (Smoothing operator in a strip)}
Let $K$ in $\N$, $p, N$ in $\N^*$, $p\leq N$, $\rho\geq0$, $j$ in $\N$.

(i)  One defines \index{Rd@$\Rtb{-\rho}{p}$ (Space of homogeneous smoothing operators in a strip)} $\Rtb{-\rho}{p}$ as the space of
$(p+1)$-linear maps from the space $(F^\infty_\infty)^{p+1}$ to 
$F^\infty_\infty$, symmetric in $(\Phi_1,\dots,\Phi_p)$, 
of the form 
$$
(\Phi_1,\dots,\Phi_{p+1})\to R(\Phi_1,\dots,\Phi_p)\Phi_{p+1} \, , 
$$ 
satisfying condition \eqref{2116} and such that
there is some $\mu\geq 0$, and for any $j \in \N $, some $C>0$ such that, for any $\Phical = (\Phi_1,\dots,\Phi_p)$ in
$(F^\infty_\infty)^p$, any $\Phi_{p+1}$ in  $F^\infty_\infty$, any $n = (n_1,\dots,n_p)$ in $(\N^*)^p$, any $n_0,
n_{p+1}$ in  $\N^*$, 
\begin{multline}
  \label{eq:518}
\sup_{z\in [-1,0]}\norm{\partial_z^j \Pin{0}R(\Pin{}\Phical)\Pin{p+1}\Phi_{p+1}}_{\Hds{0}}\\
\leq C\frac{\maxdn{1}{p+1}^{\rho+\mu}}{\maxn{1}{p+1}^{\rho-j}}\Gcalsm{0,j}{0,p+1}{\Phical,\Phi_{p+1}}.
\end{multline}

(ii) One defines for $N$ in $\N$, $r>0$,  \index{Re@$\Rrb{-\rho}{K,0,j,N}$  (Space of non-homogeneous smoothing operators in a strip)}
$\Rrb{-\rho}{K,0,j,N}$ as the space of maps $(\Phi,\Psi)\to R(\Phi)\Psi$, that are 
defined on $\Brb{K}{}{j}\times\CKHb{\sigma}{j}$ for some $\sigma>0$, that are linear in $\Psi$ and such that, for any $s$
with $s>\sigma$, there are $C>0$ and $r(s)\in ]0,r[$, and for any $\Phi$ in $\Brsb{K}{}{s}{j}\cap\CKHb{s}{j}$, any $\Psi$ in
$\CKHb{s}{j}$, any $0\leq k\leq K$, any $0\leq j'\leq 
j$, one has the estimate
\begin{multline}
  \label{eq:519}
\norm{\partial_t^kR(\Phi)\Psi(t,\cdot)}_{F^{s+\rho-\frac{3}{2}k}_{j'}}\leq C\sum_{k'+k''=k}\bigl(\Gcals{\sigma,j'}{k',N}{\Phi}\Gcals{s,j'}{k'',1}{\Psi}\\
+ \Gcals{\sigma,j'}{k',N-1}{\Phi}\Gcals{\sigma,j'}{k',1}{\Psi}\Gcals{s,j'}{k'',1}{\Phi}\bigr) \, .
\end{multline}

(iii) One denotes by \index{Rf@$\sRb{-\rho}{K,0,j,p}{N}$ (Space of smoothing operators in a strip)}  
$\sRb{-\rho}{K,0,j,p}{N}$ the space of maps $(\Phi,\Psi)\to R(\Phi)\Psi$ that may be written as
\begin{equation}
  \label{eq:5110}
  R(\Phi)\Psi = \sum_{q=p}^{N-1} R_q(\Phi,\dots,\Phi)\Psi + R_N(\Phi)\Psi
\end{equation}
with $R_q$ in $\Rtb{-\rho}{q}$ and $R_N$ in $\Rrb{-\rho}{K,0,j,N} $.  	
\end{definition}

\textbf{Remarks}:  
$\bullet$ If $a$ is a symbol in the class $ \sGb{m}{K,0,j,p}{N}$ and  $b$ in $\sGb{m'}{K,0,j,p'}{N}$, then    
the composition proposition of symbolic calculus~\ref{231}
applies  and formula \eqref{232}  holds with a smoothing operator in the class 
$ \sRb{-\rho+m+m'}{K,0,j,p+p'}{N} $ of Definition~\ref{512}.

$\bullet$ Let $R$ be an homogeneous smoothing operator of $\Rt{-\rho}{p}$ as defined in (i) of
Definition~\ref{214}. We claim that $R$ defines also a smoothing operator  of  $\Rtb{-\rho}{p} $.
Indeed, since $R$ is multilinear, 
given functions $\Phi_j $, $ j = 1, \ldots, p + 1 $ in $F^\infty_\infty $, we may compute 
$$
\pa_z^{j} R(\Phi_1,\dots,\Phi_p)\Phi_{p+1}    = \! \!\! \! \!  \sum_{j_1+ \cdots + j_{p+1}= j} \!  \! \! C_{j_1 \ldots j_{p+1}}
R( \pa_z^{j_1}\Phi_1,\dots, \pa_z^{j_p} \Phi_p) \pa_z^{j_{p+1}} \Phi_{p+1}
$$ 
for suitable binomial coefficients  $ C_{j_1 \ldots j_{p+1}} $.
Using  \eqref{2115} and the smoothing estimates
$ \Vert \Pi_{n_\ell} \pa_z^{j_\ell} \Phi_\ell \Vert_{{\dot H}^{-j}}  \leq n_{\ell}^{-j_\ell} \Vert \Phi_\ell \Vert_{F^0_j} $, 
$ \ell = 1, \ldots, p + 1 $,  one deduces   
 \eqref{518}. 
 
$ \bullet $  If $ R $ is an homogeneous smoothing operator of $\Rt{-\rho}{p}$ 
then  $R (\Phi, \ldots, \Phi ) \Psi $  satisfies  \eqref{519} with $ \rho - \alpha $, for any 
$ \alpha > 1 / 2 $, instead of $  \rho $.

$\bullet$ As in the case of symbols, we shall use the same notation for classes of smoothing operators
when we allow the arguments $\Phi_1,\dots,\Phi_p$ (resp.\  $\Phi$) in (i) (resp.\ (ii)) of the above definition to be $\C^2$ valued instead of $\C$ valued.

$\bullet$ Occasionally, we shall have to consider elements of $\Rtb{-\rho}{p}$ acting on functions 
$(\Phi_1,\dots,\Phi_{p+1})$ of $ (F_\infty^\infty)^{p+1} $, 
where at least one of the arguments  is a function of $ F_\infty^\infty $ 
that does not depend on $ z $, i.e. it belongs to $\Hds{\infty}(\Tu;\C)$
(for example in the proof of Proposition \ref{614}). Of course, such a function  is also
in $F^\infty_\infty$, so that the estimates \eqref{518} remain meaningful. 
In the same way, we shall have to consider
smoothing operators of $\sRb{-\rho}{K,0,j,p}{N}$ where  the argument $\Psi  $ of \eqref{5110}
is a function that does not depend on
$ z $, so that \eqref{519} holds with $ \Gcals{\sigma,j'}{k',1}{\Psi} = \Gcals{\sigma}{k',1}{\Psi} $. 
For simplicity, 
even in such cases, we shall denote these classes of operators 
using the same notations. 
\medskip

\medskip

We shall need as well para-Poisson operators, sending functions defined on one of the boundaries of the strip $\Bcal$ to functions
defined on the whole $\Bcal$. We first define the Poisson symbols. 
\begin{definition}
  \label{513} 
 {\bf (Poisson symbols)} 
  Let $m$ be in $\R$, $p\leq N,  K$ in $\N$.

(i) {\bf (Homogeneous Poisson symbols)}
One denotes by \index{Pa@$\Pt{m,\pm}{p}$ (Poisson homogeneous symbols from the boundary to the interior)} $\Pt{m,+}{p}$ (resp.\
$\Pt{m,-}{p}$) the space of symmetric $p$-linear maps defined on  
$\Hds{\infty}(\Tu,\C^2)^p$
\[\Ucal = (U_1,\dots,U_p)\to \bigl((z,x,\xi)\to a(\Ucal;z,x,\xi)\bigr)\]
such that for any $\ell, j$ in $\N$, $z^\ell \partial_z^j a(\Ucal;z,x,\xi)$ (resp.\ $(1+z)^\ell\partial_z^j
a(\Ucal;z,x,\xi)$) belongs to the class of symbols $\Gt{m+j-\ell}{p}$ of Definition~\ref{211}, and satisfies bounds \eqref{214}
with $ m $ replaced by $ m + j - \ell $, and some $\mu$ independent of $\ell$, uniformly in $z \in [-1,0]$.

In the same way, \index{Pb@$\Pti{m}{p}$ (Poisson homogeneous symbols from the interior to the interior)} $\Pti{m}{p}$ denotes the space of symmetric $p$-linear maps 
\[
\Ucal = (U_1,\dots,U_p)\to \bigl((z,z',x,\xi)\to a(\Ucal;z,z',x,\xi)\bigr)
\]
with values in functions of  $z, z' \in [-1,0]$, $x\in \Tu$, $\xi\in \R$ that may be written as
\be\label{eq:a+a-}
a_+(\Ucal;z,z',x,\xi)\1_{z-z'>0} + a_-(\Ucal;z,z',x,\xi)\1_{z-z'<0}
\ee
where for any $\ell, j, j'$ in $\N$, $(z-z')^\ell\partial_z^j\partial_{z'}^{j'} a_\pm(\Ucal;z,z',\cdot)$ belongs to
$\Gt{m+j+j'-\ell}{p}$ and satisfies \eqref{214} with $m$ replaced by $m+j+j'-\ell$ and some $\mu$ independent of $\ell$,
uniformly in $z, z'$.

(ii) {\bf (Non-homogeneous Poisson symbols)}
Let $r>0$. One denotes by \index{Pc@$\Prr{m,\pm}{K,0,p}$ (Poisson non-homogeneous symbols from the boundary to the interior)}
$\Prr{m,+}{K,0,p}$ (resp.\ $\Prr{m,-}{K,0,p}$) the space of functions 
\[(U,z,x,\xi)\to a(U;z,x,\xi)\]
such that for any $ j, \ell $ in $ \N $, the symbol 
$z^\ell\partial_z^ja(U;z,\cdot)$ (resp.\ $(1+z)^\ell\partial_z^ja(U;z,\cdot)$) belongs to
$\Gra{m+j-\ell}{K,0,p}$ (see Definition \ref{212bis}), and satisfies 
\eqref{218} with $ m $ replaced by $ m + j - \ell $, $ K' = 0 $, and some $\sigma, \sigma_0$ independent of
$\ell$, uniformly in $z\in [-1,0]$. 

One defines \index{Pd@$\Prri{m}{K,0,p}$ (Poisson non-homogeneous symbols from the interior to the interior)} 
$\Prri{m}{K,0,p}$ as the space of functions
\begin{equation}
  \label{eq:5110a}
  a(U;z,z',x,\xi) = \sum_{+,-} a_\pm(U;z,z',x,\xi)\1_{\pm(z-z')>0},
\end{equation}
where for any integers $\ell, j, j'$, $(z-z')^\ell\partial_z^j\partial_{z'}^{j'} a_\pm(U;t,z,z',\cdot)$ belongs to the class
$\Gra{m+j+j'-\ell}{K,0,p}$ and satisfies \eqref{218} with $m$ replaced by $m+j+j'-\ell$, $ K' = 0 $,
and some $\sigma$ independent of $\ell$,
uniformly in $z, z'$.

(iii) {\bf (Poisson symbols)}
One defines  the Poisson symbols \index{Pe@$\sP{m,\pm}{K,0,p}{N}$ (Poisson symbols from the boundary to the interior)} 
$\sP{m,\pm}{K,0,p}{N}$
(resp.\  \index{Pf@$\sPi{m}{K,0,p}{N}$ (Poisson symbols from the interior to the interior)} $\sPi{m}{K,0,p}{N}$) 
as the sum of homogeneous Poisson symbols $ a_q $ in $\Pt{m,\pm}{q}$, $ q = p, \ldots, N - 1 $,  evaluated at $ (U, \ldots, U ) $, 
plus a non-homogeneous Poisson symbol $ a_N $ of $\Prr{m,\pm}{K,0,N}$, 
\be\label{eq:5110a-bis}
a = \sum_{q=p}^{N-1} a_q (U, \ldots, U; z,z', x, \xi) + a_N (U; z,z', x, \xi) \, . 
\ee
We denote $ \sP{-\infty,\pm}{K,0,p}{N} $ the intersection $ \bigcap_m\sP{m,\pm}{K,0,p}{N} $.
\end{definition}  

\textbf{Remarks}: 
$ \bullet $ 
If $ a $  is a Poisson symbol in $\Pt{m,\pm}{p}$, resp. $ \sP{m,\pm}{K,0,p}{N} $, then 
$ \pa_z a $ is in $\Pt{m+1,\pm}{p}$, resp. $ \sP{m+1,\pm}{K,0,p}{N} $. Similarly
if $ a $  is a Poisson symbol in $\Pti{m}{p}$, resp. $ \sPi{m}{K,0,p}{N} $, then $ \pa_z a $ and $ \pa_{z'} a $  are in 
$\Pti{m+1}{p}$, resp. $ \sPi{m+1}{K,0,p}{N} $. 

$ \bullet $ 
If $ a $  is a Poisson symbol in $ \sP{m,\pm}{K,0,p}{N} $, resp. $ \sPi{m}{K,0,p}{N} $, 
then  $ \pa_\xi a $ is in $ \sP{m-1,\pm}{K,0,p}{N} $, resp. $ \sPi{m-1}{K,0,p}{N} $ . 

$ \bullet $ If $ a $  is a Poisson symbol in $\Pt{m,+}{p}$, resp. $\Pt{m,-}{p}$, 
then  $ z a $ is in $\Pt{m-1,+}{p}$, resp. $ (1+z) a $ is in $ \Pt{m-1,-}{p} $. Similarly
if $ a $  is a Poisson symbol in $\Pti{m}{p}$, resp. $ \sPi{m}{K,0,p}{N} $, then $ (z - z')a $  is in 
$\Pti{m-1}{p}$, resp. $ \sPi{m-1}{K,0,p}{N} $. 

$ \bullet $  Let $ \theta \in ]0,1[$. If $ a $ is in $ \sPi{m}{K,0,p}{N} $, then
 $ | z - z' |^{\theta} a $ is in $ \sPi{m- \theta}{K,0,p}{N} $. Indeed, for any $ z, z' $ in $ [-1,0] $, for any 
 $ \alpha, \beta \in \N $,   
 $$
 | | z - z' |^{\theta} \pa_x^\alpha \pa_\xi^\beta a_\pm | = 
 | ( z - z' ) \pa_x^\alpha \pa_\xi^\beta a_\pm |^{\theta}  | \pa_x^\alpha \pa_\xi^\beta a_\pm |^{1- \theta} 
 $$
 satisfies the estimates of a symbol in $ \sGa{m-\theta}{K,0,p}{N} $, uniformly in $ z, z' $.

$ \bullet $ The following Poisson symbols 
\be\label{eq:def:CSymbols}
 \Ccal(z,\xi)  = \frac{\cosh((z+1)\xi)}{\cosh\xi} \in \Pt{0,+}{0} \,  , \ \  \Scal(z,\xi) = \frac{\sinh(z\xi)}{\xi\cosh\xi}  \in 
 \Pt{-1,-}{0} \, , 
\ee
and
\[\begin{split}
\index{Kz@$K_0(z,z',\xi)$ (Poisson kernel)}K_0(z,z',\xi) = (\cosh\xi)\bigl(\Ccal(z,\xi)\Scal(z',\xi)\1_{z-z'<0} + \Scal(z,\xi)\Ccal(z',\xi)\1_{z-z'>0}\bigr)\\ \in \Pti{-1}{0}
\end{split}\]
will arise in the next section. 

$ \bullet $ 
We associate to a Poisson symbol
$ a (U; z, x, \xi) $ in $\sP{m,\pm}{K,0,p}{N}$ the corresponding paradifferential operator 
sending functions defined on $ \Tu $ to functions of $ (z, x) \in \Bcal $ by applying 
\eqref{2114} for each $ z $. We call $ \opbw (a (U; z, \cdot )) $ the para-Poisson operator associated to $ a $. 
On the other hand, to a Poisson symbol  $ a $ in 
$\sPi{m}{K,0,p}{N}$ we associate the para-Poisson operator
$$
  V\to \int_{-1}^0\opbw( a(U;z,z',\cdot))V(z',\cdot)\,dz' \, , 
$$
which sends a function $ V $ defined on the strip  $ \Bcal $ into another function of $ \Bcal $. 

\medskip

According to the Definition~\ref{513} 
of Poisson symbols, the boundedness properties of a paradifferential operator given in  
Proposition~\ref{215} imply the following lemma.

\begin{lemma} {\bf (Action of a para-Poisson operator from the boundary to the interior)} \label{Poisson:boundary-interior}
Let $ a $ be a Poisson symbol in $\sP{m,+}{K,0,p}{N}$ (resp.\  $\sP{m,-}{K,0,p}{N}$).
If $ U $ is in $\Br{K}{0}$ for some large enough $\sigma_0$, 
then, for any $s$ in $\R$, $j$ in $\N$, 
any $ V $ belonging to
$\Hds{s}(\Tu,\C)$, the function  
$$
(z,t,x)\to z^\ell\partial_z^j\opbw(a(U; z, \cdot))V
$$ 
(resp.\  $(z,t,x)\to (1+z)^\ell\partial_z^j\opbw(a(U; z, \cdot))V$) is in 
$\CKH{s-m+\ell-j}{\C}$ uniformly in 
$z\in [-1,0] $, and for any $ 0 \leq k \leq K $
$$
\| z^\ell\partial_z^j\opbw( \pa_t^k a(U; z, \cdot)) V \|_{L^\infty([0,1], \Hds{s-m+\ell-j})} \leq C \| V \|_{\Hds{s}}
$$ 
uniformly in $ t \in I $, the constant $ C $ depending only on $ \nnorm{U(t,\cdot)}_{K,\sigma_0} $.  
In particular $ \| \opbw( \pa_t^k a(U; z, \cdot)) V \|_{F^{s-m}_j} \leq C \| V \|_{\Hds{s}} $.
\end{lemma}

To study the action of para-Poisson operators associated to symbols in the space $\sPi{m}{K,0,p}{N}$, we introduce another scale of spaces.
\begin{definition}\label{515} {\bf (Space of functions on a strip)}
We denote by $E^s_j$ the subspace of $F^s_j$ (introduced in Definition \ref{spacesFsj})
given by
\begin{equation}
  \label{eq:5111}
  \index{Eb@$E^s_j$ (Space of functions on a strip)} E^s_j = \big\{ \Phi\in F^s_j \, ; \  
  \partial_z^{j'}\Phi\in L^2([-1,0], \ \Hds{s-j'+\frac{1}{2}}(\Tu)); \, j' = 0,\dots,j \big\} 
\end{equation}
together with the natural norm
\begin{equation}
  \label{eq:512E}
  \norm{\Phi}_{E_j^s} = \norm{\Phi}_{F_j^s} + \sum_{j'=0}^j\norm{\partial_z^{j'}\Phi}_{ L^2([-1,0],\Hds{s  -j' + \frac12})} 
   \, .
\end{equation}
We denote by $ C^K_*(I,E^s_j) = \bigcap_0^K C^k(I,E^{s-\frac{3}{2}k}_j)$
and by  $\nnorm{\Phi}_{K,s,j}^E$ the norm defined by the first line in \eqref{513} with $F^{\sigma-\frac{3}{2}k}_j$ replaced by $E^{s-\frac{3}{2}k}_j$. 
We set $ E^\infty_\infty \stackrel{\textrm{def}}{=} \bigcap_{s,j} E_j^s$.
\end{definition}

\noindent
\textbf{Remark}: We have the continuous inclusions $ E^s_j \subset F^s_j \subset E^{s- \frac12}_j $. 
\smallskip

These spaces are characterized in terms of Littlewood-Paley decomposition, that we now 
recall. Consider  a Littlewood-Paley partition of unity, 
$$
1 = \psi(\xi) + \sum_{ \ell \geq 1} \varphi(2^{- \ell }\xi) \, ,
$$ 
where the $ C^\infty $ 
function $\varphi $, resp. $ \psi $, is supported for
$C^{-1}\leq \abs{\xi}\leq C$ for some constant $ C > 2 $, 
resp.  for $ \abs{\xi}\leq c $ for some $ c >  0 $, 
$\psi$ and $\varphi$ being even in $ \xi $. 
We denote 
\be\label{eq:Delta-ell}
\Delta_0 = \psi(D)  \, , \quad  \Delta_\ell = \varphi(2^{- \ell }D) \, , \   \ell \geq 1  \, . 
\ee
We have the following characterization of the spaces $ E_j^s $.

\begin{lemma}\label{CarSpacE}
{\bf (Littlewood-Paley characterization of $ E_j^s $)}
A function $ \Phi $ is in $ E_j^s $ if and only if, for any $ 0 \leq j' \leq j $, $ z \in [-1,0] $,
\begin{equation}
  \label{eq:5112}
  \begin{split}
    \norm{\Delta_\ell\partial_z^{j'}\Phi(z,\cdot)}_{L^2(\Tu)} &\leq c_\ell(z) 2^{-\ell(s-j')}\\
\norm{\Delta_\ell\partial_z^{j'}\Phi}_{L^2([-1,0],L^2(\Tu))} &\leq c_\ell 2^{-\ell(s+\frac{1}{2}-j')}
  \end{split}
\end{equation}
for a sequence $(c_\ell)_\ell$ in $\ell^2 (\N, \R) $ 
and a sequence $(c_\ell(z))_\ell$ in $\ell^2 (\N, \R) $ such that
$$
\| \Phi \|_{F^s_j}^2 \simeq \sup_{z \in [-1,0]}\sum_\ell\abs{c_\ell(z)}^2  \, , \quad 
\| \Phi \|_{E^s_j}^2 \simeq \| \Phi \|_{F^s_j}^2 + \sum_{\ell  \geq 0} c_\ell^2 \, . 
$$
We also have
\begin{equation} \label{eq:5112bis}
 \norm{\Delta_\ell\partial_z^{j'}\Phi(z,\cdot)}_{L^2(\Tu)} \leq c_\ell' (z) 2^{-\ell(s-j' + \frac12)}
\end{equation}
for a sequence $ (c_\ell' (z))_\ell $ such that  $ \sum_\ell \norm{c'_{\ell}(z)}_{L^2([-1,0], \R)}^2 =  \sum_{\ell  \geq 0}c_\ell^2 $.  
\end{lemma}

\begin{proof}
By the characterization of the Sobolev spaces with a 
Paley-Littlewood decomposition we have, for any $ 0 \leq j' \leq j $,
\begin{equation*}
\begin{split}
\norm{\partial_z^{j'}\Phi (z, \cdot )}_{\Hds{s  -j'}}^2 & \simeq \sum_\ell 
\Big(\norm{\Delta_\ell \partial_z^{j'}\Phi (z, \cdot )}_{L^2(\Tu)} 2^{\ell (s-j')}\Big)^2 \\
\norm{\partial_z^{j'}\Phi}_{ L^2([-1,0],\Hds{s  -j' + \frac12})}^2  & \simeq 
\sum_\ell \Big(\norm{\Delta_\ell \partial_z^{j'}\Phi }_{L^2([-1,0],L^2(\Tu))} 2^{\ell (s-j'+ \frac12)}\Big)^2
\end{split}
\end{equation*}
and \eqref{5112} follows with
\begin{equation*}
\begin{split}
c_\ell (z) & 
\stackrel{\textrm{def}}{=}  \max_{0 \leq j' \leq j} \norm{\Delta_\ell \partial_z^{j'}\Phi (z, \cdot )}_{L^2(\Tu)} 2^{\ell (s-j')} \, \\ 
c_\ell &  \stackrel{\textrm{def}}{=}  \max_{0 \leq j' \leq j} \norm{\Delta_\ell \partial_z^{j'}\Phi }_{L^2([-1,0],L^2(\Tu))} 2^{\ell (s-j'+ \frac12)}  \,.
\end{split}
\end{equation*}
Finally \eqref{5112bis} follows with $ c_\ell' (z) = c_\ell (z) 2^{\ell/2} $. 
\end{proof}

The para-Poisson operators associated to a symbol  of
$\sPi{m}{K,0,p}{N}$ gain 
one derivative in the $E_j^s$ scale (we shall use this property for instance in the proof  of lemma \ref{532}).
\begin{lemma}
  \label{516} 
  {\bf (Action of a para-Poisson operator from the interior to the interior)} 
Let $p,  K $ be in $\N$, $r>0$, $N\geq p$ and let $ a $ be a Poisson symbol  of 
$\sPi{m}{K,0,p}{N}$. There is $\sigma_0>0$ such that if $U$
is in $\Br{K}{0}$, for any $s$ in $\R$, any $j$ in $\N$, any $k\leq K$, the para-Poisson operator
\begin{equation}
  \label{eq:5113}
  V\to \int_{-1}^0\opbw(\partial_t^ka(U;z,z',\cdot))V(z',\cdot)\,dz'
\end{equation}
is bounded from $E^s_j$ to $E^{s+1-m}_j$ uniformly in $t\in I$, the bound depending only 
on $\nnorm{U(t,\cdot)}_{K,\sigma_0}$.
\end{lemma}

\begin{proof}
  We may assume $k=0$ and we do not write time dependence. By Definition~\ref{513}, the right hand side of \eqref{5113} may be written as
  \begin{equation}
    \label{eq:5114}
    \int_{-1}^z\opbw(a_+(U;z,z',\cdot))V(z')\,dz' + \int_z^0 \opbw(a_-(U;z,z',\cdot))V(z')\,dz'.
  \end{equation}
Denote by $\Delta_\ell = \varphi(2^{-\ell}D_x)$ the elements of some Littlewood-Paley
decomposition as in \eqref{Delta-ell}. 
It follows from the definition of $\opbw$ that the action of $\opbw(a_\pm)$ does not enlarge much the support
of the Fourier transform of functions, see \eqref{2122}, \eqref{support}-\eqref{2124} and the first remark after Proposition \ref{215}.
We may thus find a compactly  supported function $\tilde{\varphi}$ of $ C^\infty_0 (\R^*) $,
equal to one in a large enough compact subset of $ \R^* $,  so that, setting $\tilde{\Delta}_\ell = \tilde{\varphi}(2^{-\ell}D_x)$, it results 
$$
\Delta_\ell\opbw(a_\pm) = \Delta_\ell\opbw(a_\pm)\tilde{\Delta}_\ell \, . 
$$ 
By the boundedness properties of paradifferential operators of Proposition~\ref{215}, if $\sigma_0$ is large enough, we have, for any $N$, 
\begin{multline}
  \label{eq:5115}
\abs{z-z'}^N\norm{\Delta_\ell\opbw(a_\pm(U;z,z',\cdot))V(z',\cdot)}_{L^2(\Tu)}\\
\leq C2^{\ell(m-N)}\norm{\tilde{\Delta}_\ell V(z',\cdot)}_{L^2(\Tu)}
\end{multline}
the constant $ C $ depending only on $\nnorm{U(t,\cdot)}_{0,\sigma_0}$.
Since the function $ V (z', \cdot) $ is in $ E^s_0 $ we get by \eqref{5112bis} that
$$
\| \tilde{\Delta}_\ell V(z',\cdot) \|_{L^2(\Tu)} \leq c'_{\ell}(z') 2^{- \ell (s+ \frac12)}
$$
for some sequence $(c'_{\ell}(z'))_{\ell}$ that satisfies  
\be\label{eq:c1z} 
\sum_\ell \norm{c'_{\ell}( z' )}_{L^2(dz')}^2  < C \| V \|_{E^s_0}^2 \, . 
\ee
Therefore \eqref{5115}  is bounded by  $C2^{\ell (m-N-s-\frac{1}{2})}c'_{\ell}(z')$. 
Using \eqref{5115} with $N= 0 $ and $ N = 2 $,  according if $ | z - z' | 2^\ell < 1 $ or  
$ | z - z' | 2^\ell \geq 1 $,
we may thus bound 
the $ L^2(dx) $ norm of the action of $\Delta_\ell$ on \eqref{5114}, obtaining 
\be\label{eq:f1}
\Big\| \Delta_\ell \int_{-1}^0\opbw( a(U;z,z',\cdot))V(z',\cdot)\,dz' \Big\|_{L^2(\Tu)} 
\leq C 2^{\ell (m-s-\frac{1}{2})}\tilde{c}'_\ell(z)
\ee
with
\[
\tilde{c}'_\ell(z) = \int_{-1}^0(1+2^\ell\abs{z-z'})^{-2} c'_\ell(z')\,dz' \, . 
\] 
By the Cauchy-Schwartz inequality 
\begin{multline*}
|\tilde{c}'_\ell(z)|^2 
\leq \Bigl(\int_{-1}^0 \frac{dz'}{(1+2^\ell\vert z-z'\vert)^2}\Bigr)  \Bigl(\int_{-1}^0 \frac{\vert c_\ell' (z')\vert^2}{(1+2^\ell\vert
  z-z'\vert)^2}dz'\Bigr) \\
\leq C 2^{-\ell}\int_{-1}^0 \frac{\vert c_\ell' (z')\vert^2}{(1+2^\ell\vert
  z-z'\vert)^2}dz'
\end{multline*}
and 
we deduce that 
\begin{multline}\label{eq:f2}
\sup_{z\in [-1,0]}
\sum_\ell 2^\ell \abs{\tilde{c}'_\ell(z)}^2 \,  + \,   
\int_{-1}^0 \sum_\ell 2^{2\ell} \abs{\tilde{c}'_\ell(z)}^2\,dz \\ 
\leq C  \sum_\ell \norm{c'_{\ell}( z' )}_{L^2( d z' )}^2 \leq C'  \| V \|_{E^s_0}^2 
\end{multline}
where the last inequality is given by \eqref{c1z}.
By \eqref{f1} and \eqref{f2} we deduce that \eqref{5113} (for $ k = 0 $) satisfies estimates \eqref{5112} with 
$ j' = 0 $, the Sobolev index $ s $ replaced by $ s + 1 - m $, 
 and  sequences $ c_\ell (z) = \tilde{c}'_\ell(z ) 2^{\ell/2} $,  $ c_\ell^2 = 2^{2\ell} \int_{-1}^0 |\tilde{c}'_\ell(z)|^2 dz  $, 
and hence
$$
 \int_{-1}^0\opbw( a(U;z,z',\cdot))V(z',\cdot)\,dz'  \in E^{s+1-m}_0 \, ,
$$ 
and its  $E^{s+1-m}_0$-norm is bounded by $ \| V \|_{E^{s}_0} $. 

We have next to study the
$z$-derivatives of \eqref{5114}. Notice that the first derivative is given by \eqref{5114} with $a_\pm$ replaced by
$\partial_z a_\pm $, and this term may be treated as above, plus the contribution 
\be\label{eq:plusder}
\opbw(a_+(U;z,z,\cdot))V(z,\cdot) - \opbw(a_-(U;z,z,\cdot))V(z,\cdot) \, .
\ee
Since the symbols $a_\pm(U;z,z,\cdot)$ are in $\sGa{m}{K,0,p}{N}$, uniformly in $z$, and 
$ V $ is in $ L^\infty([-1,0],\Hds{s}) \cap L^2([-1,0],\Hds{s+\frac{1}{2}})$, 
Proposition~\ref{215} implies that \eqref{plusder}
belongs to $L^\infty([-1,0],\Hds{s-m}) \cap L^2([-1,0],\Hds{s+\frac{1}{2}-m})$ with a  norm bounded by $ C \| V \|_{E^s_0}$. 
Higher order $\partial_z$ derivatives are
treated in the same way.
\end{proof}

Arguing in a similar way to Lemma \ref{516}  we have also the following 
lemma concerning 
the action of a para-Poisson operator from the boundary to the interior with values in  the scale $E^{s}_j $.
\begin{lemma}
  \label{para-BI}  
{\bf (Action of a para-Poisson operator from the boundary to the interior)} 
Let $p,  K $ be in $\N$, $r>0$, $N\geq p$ and let $ a $ be a Poisson symbol  of  $\sP{m,\pm}{K,0,p}{N}$.
There is $\sigma_0>0$ such that if $U$
is in $\Br{K}{0}$, for any $s$ in $\R$, any $j$ in $\N$, any $k\leq K$,  the para-Poisson operator
$$
  V\to \opbw(\partial_t^ka(U; z, \cdot)) V 
$$
is bounded from $ \Hds{s}(\Tu,\C)$  to $ E^{s-m}_j$, uniformly in $t\in I$ and for any $j$, the bound depending only 
on $\nnorm{U(t,\cdot)}_{K,\sigma_0}$.
\end{lemma}

We shall define next the natural classes of smoothing operators 
that will give remainders in the symbolic calculus associated
to Poisson symbols of $\sP{m,\pm}{K,0,p}{N}$ and $\sPi{m}{K,0,p}{N} $. 
Let us introduce a variant of notation \eqref{513}, namely define
\begin{equation}
  \label{eq:5115a}
  \index{Gn@$\Gcalst{\sigma,j}{K,1}{U}$ (Norm on a strip)} \Gcalst{\sigma,j}{K,1}{U}  = \nnorm{U (t, \cdot)}^E_{K,\sigma,j} = \sum_{k=0}^K\norm{\partial^k_t U(t,\cdot)}_{E^{\sigma-\frac{3}{2}k}_j} \, .
\end{equation} 

\begin{definition}
  \label{514} {\bf (Smoothing operators from the boundary to  a strip and on a strip)}
Let $\rho$ be in $\R_+$, $p, K$ in $\N$, $p\leq N$.

(i)  We denote by \index{Rg@$\Rt{-\rho,\pm}{p}$ (Space of homogeneous smoothing operators from the boundary to  a strip)}
$\Rt{-\rho,+}{p}$ (resp. $\Rt{-\rho,-}{p}$) the space of $(p+1)$-linear operators defined on 
$\Hds{\infty}(\Tu,\C^2)^p\times\Hds{\infty}(\Tu,\C)$, depending on $ z \in [-1,0] $, 
\[
(U_1,\dots,U_{p+1})\to R(U_1,\dots,U_p;z)U_{p+1}
\]
symmetric in $(U_1,\dots,U_p)$, such that for any $j, \ell$ in $\N$ with 
$j-\ell\leq\rho$, $(z^\ell\partial_z^j R(U_1,\dots,U_p;z))_{z\in [-1,0]}$ 
(resp.\ $((1+z)^\ell\partial_z^j R(U_1,\dots,U_p;z))_{z\in [-1,0]}$) is a bounded family
of smoothing operators of $\Rt{-\rho+j-\ell}{p}$ defined in Definition~\ref{214}, in the sense that \eqref{2115} holds with $\rho$
replaced by $\rho+\ell-j$ and some $\mu$ independent of $\ell$, uniformly in $  z $.

We denote by  \index{Rh@$\Rti{-\rho}{p}$ (Space of homogeneous smoothing operators on a strip)}  
$\Rti{-\rho}{p}$
  the space of $(p+1)$-linear operators, defined 
  for $ (U_1, \ldots, U_p, V ) $ in $\Hds{\infty}(\Tu,\C^2)^p\times  E^\infty_\infty $,  
  of the form 
\be \label{eq:homo-RP}
V \to    \int_{-1}^0 R(U_1,\dots,U_p;z,z') V(z',\cdot)\,dz'
\ee 
symmetric in $(U_1,\dots,U_p) $, whose integral kernel 
depend on $z, z'\in [-1,0]$, and  may be written as
\[
R(U_1,\dots,U_p;z,z') = \sum_{+,-}R_\pm(U_1,\dots,U_p;z,z')\1_{\pm(z-z')>0}
\]
where, for any $j, j',\ell$ in $\N$ with $j+j'-\ell\leq \rho-1$, 
\begin{equation}
  \label{eq:5110b}
  \big( (z-z')^\ell\partial_z^j\partial_{z'}^{j'} R_\pm(U_1,\dots,U_p;z,z') \big)_{z,z'\in [-1,0]}
\end{equation}
is a bounded family of smoothing operators of , in the  sense 
that \eqref{2115} holds with $\rho$ 
replaced by $ -\rho+j+j'-\ell+1 $ and some $\mu$ independent of $\ell$, uniformly in $  z, z' $.

(ii) Let $r>0$. We denote by \index{Ri@$\Rr{-\rho,\pm}{K,0,N}$ (Space of non-homogeneous smoothing operators   from the 
boundary to a strip)}$\Rr{-\rho,\pm}{K,0,N}$ (resp.\  \index{Rj@$\Rri{-\rho}{K,0,N}$ (Space of smoothing operators on a strip)}    $\Rri{-\rho}{K,0,N}$)
the space of maps $(U,V)\to R(U)V$ that are defined on $\Br{K}{}\times\CKH{\sigma}{\C}$ (resp.\ on
$\Br{K}{}\times C_*^K(I,E^\sigma_j)$) for some $\sigma>0$ (resp.\ some $\sigma\geq j\geq 0$), with values in
$C_*^K(I,E^{\sigma+\rho}_j)$, that are linear in $V$, and such
that for any $s>\sigma$, there are $C>0$, $r(s)\in ]0,r[$ and for any $U\in \Brs{K}{}{s}\cap
\CKH{s}{\C^2}$, any $ V $ in $\CKH{s}{\C}$, any $0\leq k\leq K$, any $j$ in $\N$, any
$ t $ in $ I $, one has a bound
\begin{multline}
  \label{eq:5115b}
\norm{\partial_t^kR(U)V(t,\cdot)}_{E^{s+\rho-\frac{3}{2}k}_j} \leq
C\bigl(\sum_{k'+k''=k}\Gcals{\sigma}{k',N}{U}\Gcals{s}{k'',1}{V}\\
+ \Gcals{\sigma}{k',N-1}{U}\Gcals{\sigma}{k'',1}{V}\Gcals{s}{k',1}{U}\bigr)
\end{multline}
(resp.\ for any $0\leq k\leq K$, any $0\leq j'\leq j$, any $t\in I$, one has a bound for any $V$ in $C_*^K(I,E^s_j)$ 
\begin{multline}
  \label{eq:5115c}
\norm{\partial_t^kR(U)V(t,\cdot)}_{E^{s+\rho-\frac{3}{2}k}_{j'}} \leq
C\bigl(\sum_{k'+k''=k}\Gcals{\sigma}{k',N}{U}\Gcalst{s,j'}{k'',1}{V}\\
+ \Gcals{\sigma}{k',N-1}{U}\Gcalst{\sigma,j'}{k'',1}{V}\Gcals{s}{k',1}{U}\bigr)
\end{multline}
where we have used notation \eqref{5115a}). 

Moreover, the operator $ R(U) $ is autonomous in the sense that 
the time dependence enters only through $ U = U(t) $.

(iii) One denotes by \index{Rk@$\sR{-\rho,\pm}{K,0,p}{N}$ (Space of smoothing operators   from the boundary to a strip)}
$\sR{-\rho,\pm}{K,0,p}{N}$ (resp.\ \index{Rl@$\sRi{-\rho}{K,0,p}{N}$ (Space of smoothing operators on a strip)}
$\sRi{-\rho}{K,0,p}{N}$) the space of sums of operators
$$
V\to \sum_{q=p}^{N-1} R_q(U,\dots,U;z)V + R_N(U)V
$$ 
\big(resp.
\be\label{eq:Resto-int}
V\to \sum_{q=p}^{N-1} \int_{-1}^0R_q(U,\dots,U;z,z')V(z',\cdot)\,dz' + R_N(U)V \big)
\ee
with $ R_q $ in $ \Rt{-\rho,\pm}{q} $
(resp.\ $\Rti{-\rho}{q}$) $q = p,\dots,N-1$ and $R_N$ in $\Rr{-\rho,\pm}{K,0,N}$ (resp.\ $\Rri{-\rho}{K,0,N}$).
\end{definition} 
 
\textbf{Remarks}:
\noindent $\bullet$  If $ R $ is in $\sR{-\rho,\pm}{K,0,p}{N}$, then $ R_{\vert z = 0 } $ is in $\sRa{-\rho}{K,0,p}{N}$, according to 
 Definition \ref{214}.

 $\bullet$  If $ R $ is in $\sR{-\rho,\pm}{K,0,p}{N}$, then $ \pa_z R  $ is in $\sR{-\rho +1, \pm }{K,0,p}{N}$.

 $\bullet$ Consider a  homogeneous smoothing operator in $ \Rti{-\rho}{p} $ 
 with kernel $ R $, as in \eqref{homo-RP}. 
 Then the operator with kernel $ \pa_z R $ is in $  \Rti{-\rho+1}{p} $ and, 
for any $ \theta \in [0,1]$, the operator with kernel 
 $ | z - z' |^{\theta} R $ is in  $ \Rti{-\rho -  \theta }{p} $.

 $\bullet$ In the sequel we shall  identify a
 non-homogeneous $R(U)V$ in $\Rr{-\rho,\pm}{K,0,N}$ to the integral expression given in terms of its Schwartz kernel $R(U;z,z')$ i.e.\ write
 the action of this operator on $V$ as
 $$ 
R(U)V =  \int_{-1}^0  R_N (U; z, z') V(z', \cdot) d z'.
 $$
Notice that we characterize  
the homogeneous operators of $ \Rti{-\rho  }{p} $ by properties of the kernel
$ R(U_1, \ldots, U_p; z, z') $, and  
the non-homogeneous ones  in $\Rr{-\rho,\pm}{K,0,N}$ by their action 
between the spaces $ E^{s}_j $. 
 
 $\bullet$ If the homogeneous smoothing  operator $\tilde{R}$ is in $\Rt{-\rho,\pm}{p}$ (resp.\ in $\Rti{-\rho}{p}$) 
 then $ R(U)V = \tilde{R}(U,\dots,U;z)V$ (resp. 
\begin{equation}
  \label{eq:5115d}
R(U)V = \int_{-1}^0 \tilde{R}(U,\dots,U;z,z')V(z')\,dz' \, )
\end{equation}
defines a smoothing operator of 
$\Rr{-\rho,\pm}{K,0,p}$ (resp.\ $\Rri{-\rho}{K,0,p}$) for any $r>0$. Let us prove this claim 
in the case of interior operators. 
Let us show that the estimate \eqref{5115c} holds for the operator \eqref{5115d} when for instance $k=0,
j'=1$.  
We have to bound
$$
\| R(U) V  \|_{E^{s+\rho}_0} +   \| (\pa_z R(U))  V  \|_{E^{s+\rho-1}_0} +  \| R(U) [ \pa_z V ] \|_{E^{s+\rho-1}_0} 
$$
by the right hand side of  \eqref{5115c} with $ k=0, j'=1$ (and $ N $ replaced by $ p $). Below we prove in detail the estimate for 
$ \| (\pa_z R(U))  V  \|_{E^{s+\rho-1}_0} $.  
According to (i) of Definition~\ref{514} 
the operator  $\tilde{R}$ may be written  as
\[
\tilde{R}_+(U,\dots,U;z,z')\1_{z-z'>0} + \tilde{R}_-(U,\dots,U;z,z')\1_{z-z'<0} 
\]
so that 
\begin{multline}
  \label{eq:5115e}
\partial_z\tilde{R}(U,\dots,U;z,z') = \bigl(\tilde{R}_+(U,\dots,U;z,z) - \tilde{R}_-(U,\dots,U;z,z)\bigr)\delta(z-z')\\
+ \partial_z\tilde{R}_+(U,\dots,U;z,z')\1_{z-z'>0} +  \partial_z\tilde{R}_-(U,\dots,U;z,z')\1_{z-z'<0} \, . 
\end{multline}
Recalling  lemma \ref{CarSpacE},  we have to bound by the right hand side of 
\eqref{5115c} 
the product of $2^{\ell(s+\rho-1)}$ (resp.\ $2^{\ell(s+\rho-\frac{1}{2})}$) times the
$L^\infty_z(\ell^2_\ell L^2(\Tu))$ (resp.\ $L^2_z(\ell^2_\ell L^2(\Tu))$) norm of each of the expressions
\begin{equation}
  \label{eq:5115f}\begin{split}
  \Delta_\ell R^0(U,\dots,U;z)V(z,\cdot)\\
 \Delta_\ell \int_{-1}^0R^1(U,\dots,U;z,z')V(z',\cdot)\,dz' 
\end{split}\end{equation}
where $R^0 = (\tilde{R}_+-\tilde{R}_-)\vert_{z=z'}$ and $ R^1 $ denotes the operator in 
second line in \eqref{5115e}. 
According to Definition~\ref{514}  
the family of operators $ R^0  $  satisfies \eqref{2115} with $\rho$ replaced by
$\rho-1$, uniformly in $z$, and \eqref{2116}, and $(z-z')^qR^1$ satisfies \eqref{2115} with $\rho$ replaced by
$\rho+q-2$, uniformly in $(z,z')$ and \eqref{2116}. We estimate the $L^2 (\Tu) $ norm of the second expression in \eqref{5115f} by
\begin{equation}
  \label{eq:5115g}
  \sum_{n_0,\dots,n_{p+1}}\int_{-1}^0\norm{\Delta_\ell\Pin{0} R^1(\Pin{1}U,\dots,\Pin{p}U;z,z')\Pin{p+1}V(z',\cdot)}_{L^2}\,dz'.
\end{equation}
Since $ R^1 $ is symmetric in its $ p $-arguments 
we may limit the above sum to indices satisfying $n_1\leq\cdots\leq n_p$. Moreover, recalling 
\eqref{Delta-ell}, the index $ n_0 $ is of magnitude $ 2^\ell $, 
and by condition \eqref{2116} there is  $ c >  0 $ such that 
\be\label{eq:massimo}
\max ( n_p, n_{p+1}) \geq c 2^\ell \, . 
\ee
Consider first the case $n_{p+1}\geq n_p$. 
The fact that $ V $ is in $E_0^s$ implies
that
\begin{equation}  \label{eq:5115h}
  \norm{\Pin{p+1}V(z',\cdot)}_{L^2} \leq c_{n_{p+1}}(z') n_{p+1}^{-s} \Gcalstm{s,0}{0,1}{V}
\end{equation}
for some sequence $( c_{n_{p+1}}(z'))_{n_{p+1}}$ satisfying
\begin{equation}   \label{eq:5115i}
  \sup_{z'\in [-1,0]}\sum_{n_{p+1}}\abs{c_{n_{p+1}}(z')}^2 < +\infty \, , \ \ 
   \sum_{n_{p+1}} n_{p+1}\int_{-1}^0 \abs{c_{n_{p+1}}(z')}^2\,dz'<+\infty \, .
\end{equation}
Moreover $\norm{\Pin{j}U}_{L^2}\leq c_{n_j} n_j^{-\sigma}\Gcalsm{\sigma}{0,1}{U}$ 
for an $\ell^2$ sequence $ (c_{n_j})_{n_j}$. 
Consequently, applying \eqref{2115} to $R^1$ and to
$(z-z')^2R^1$ (with a gain of two units on $\rho$ in that case),
 and taking into account that $ 2^\ell \sim n_0 $ and $ n_0 \leq C n_{p+1} $, we get that 
\eqref{5115g} is bounded  by $ C \Gcalsm{\sigma}{0,p}{U}\Gcalstm{s,0}{0,1}{V}$ times
\begin{multline} \label{eq:5115j} 
 \sum_{n_0}\cdots\sum_{n_{p+1}}\int_{-1}^0 (1+2^\ell\abs{z-z'})^{-2} \prod_{j=1}^p n_j^{-\sigma} c_{n_j}
\,  n_{p+1}^{-s} c_{n_{p+1}}(z')\;dz'\\
\times\Bigl(\frac{n_p}{n_{p+1}}\Bigr)^{\rho-2} n_p^\mu \, \1_{\abs{n_0-n_{p+1}}\leq Cn_p, n_0\sim 2^\ell}
 \1_{n_1\leq\cdots\leq n_p\leq n_{p+1}}
\end{multline}
for some $ \mu $ independent of $ s, \sigma, \rho $, and noticing that the support condition 
$ \abs{n_0-n_{p+1}}\leq Cn_p $ comes from  \eqref{2116}. 
For  $ \sigma \geq \rho + \mu + 1 $, summing 
in $ n_1, \ldots , n_{p-1} $  and using \eqref{massimo}, 
we bound \eqref{5115j} by
\begin{multline} \label{eq:5115k}
C 2^{-\ell(s+\rho-2)}\sum_{n_0,n_p,n_{p+1}} n_p^{-3}\1_{\abs{n_0-n_{p+1}}\leq Cn_p, n_0\sim 2^\ell}\\
\times \int_{-1}^0(1+ 2^\ell\abs{z-z'})^{-2} c_{n_{p+1}}(z')\,dz' \, .
\end{multline}
The $L^2(dz)$ norm of \eqref{5115k} is bounded,
using the Cauchy-Scwhartz inequality and performing a change of variable in the integral,  by 
\begin{equation}\label{eq:5115l}
C 2^{-\ell(s+\rho-1)}\sum_{n_0,n_p,n_{p+1}} 
 n_p^{-3} \1_{\abs{n_0-n_{p+1}}\leq Cn_p, n_0\sim 2^\ell}\norm{c_{n_{p+1}}(z')}_{L^2(dz')} \, .
\end{equation}
We may write 
$\norm{c_{n_{p+1}}(z')}_{L^2(dz')} = 2^{-\ell/2}\delta_{n_{p+1}} $ where, 
by \eqref{5115i} and the fact that $ n_{p+1}\geq c2^\ell $ (see \eqref{massimo}) 
the sequence $(\delta_{n_{p+1}})_{n_{p+1}}$ is in $\ell^2 (\N) $.  
Then we write \eqref{5115l} as  $ C 2^{-\ell(s+\rho-\frac{1}{2})}  d_\ell$ where
$$
d_\ell =  \sum_{n_0,n_p,n_{p+1}} 
 n_p^{-3} \1_{\abs{n_0-n_{p+1}}\leq Cn_p, n_0\sim 2^\ell} \delta_{n_{p+1}} \, . 
$$
Since $(d_\ell)_\ell$ is an  $\ell^2$ sequence, we get a contribution to the
second expression \eqref{5115f} satisfying the second estimate \eqref{5112} of a function of $E_0^{s+\rho-1}$. The first
estimate is obtained taking the $L^\infty(dz)$ instead of $L^2(dz)$ norms of 
\eqref{5115k}, which leads to \eqref{5115l}
multiplied by $2^{\ell/2}$.

Consider now \eqref{5115g} 
in the case $n_{p+1}\leq n_p$. In this case, we use instead of
\eqref{5115h}
\[\norm{\Pin{p+1} V(z',\cdot)}_{L^2} \leq c_{n_{p+1}}(z') n_{p+1}^{-\sigma}\Gcalstm{\sigma,0}{0,1}{V},\]
the estimate $\norm{\Pin{p}U}_{L^2}\leq c_{n_p}n_p^{-s}\Gcalsm{s}{0,1}{U}$, that 
$n_p\geq c2^\ell$  (see \eqref{massimo}) and $\abs{n_0-n_{p}}\leq
C\max(n_{p+1},n_{p-1})$. We get then for \eqref{5115g} an estimate by the product 
of $\Gcalsm{\sigma}{0,p-1}{U}\Gcalstm{\sigma,0}{0,1}{V}\Gcalsm{s}{0,1}{U}$ and of an expression similar to \eqref{5115j}. By the
same computations as above, we obtain that the estimates \eqref{5112} of an element of $E^{s+\rho-1}_0$ are satisfied. The
second expression \eqref{5115f} is thus in $E^{s+\rho-1}_0$.

The first expression \eqref{5115f} is bounded by \eqref{5115j} where we remove the integral, take $z=z'$ and replace $\rho$
by $\rho+1$. Similar computations as above 
show that $ \| R^0(U,\dots,U;z)V(z,\cdot) \|_{E^{s+\rho-1}_0} $, 
as well as $ \| R(U)[\pa_z V] \|_{E^{s+\rho-1}_0} $, are bounded
by the right hand side of  \eqref{5115c} with $ k=0, j'=1$ and $ N = p $.

If one takes further $\partial_z$ derivatives of \eqref{5115e}, one gets similarly that $\int_{-1}^0\partial_z^j
\tilde{R}(U,\dots,U;z,z')V(z')\,dz'$ is in $E_0^{s+\rho-j}$ if $V$ is in $E_j^s$. This proves the remark for interior
operators. The corresponding statement for the classes $\Rr{-\rho,\pm}{K,0,p}$ is proved in a similar and easier way.

\smallskip

$\bullet$ 
By the third remark after Proposition~\ref{215},  if $a$ is in $\sG{-\rho}{K,0,p}{N}$,
then $\opbw(a)$ defines an element of $\sR{-\rho}{K,0,p}{N}$. As a consequence, 
if $ a $ is a Poisson symbol in $ \sPi{-\rho}{K,0,p}{N} $ as in  \eqref{5110a-bis}, then $ \opbw(a) $  
is the kernel of the smoothing operator in $ \sRi{-\rho}{K,0,p}{N}$, 
$$
V \to \int^0_{-1} \opbw (a (U; z, z') )V(z', \cdot ) d z' \, . 
$$
Identifying an operator with the  kernel, 
we shall also simply write
 that $ \opbw(a) $  is in $ \sRi{-\rho}{K,0,p}{N} $.
Similarly  if $ a $ is in $\sP{-\rho,\pm}{K,0,p}{N}$ then $\opbw(a)$ is a smoothing operator in 
$\sR{-\rho,\pm}{K,0,p}{N}$.

\medskip

Let us study now composition of  operators associated to the classes of Poisson symbols of Definition~\ref{513}.
\begin{proposition}
  \label{517} {\bf (Composition of para-Poisson operators)}
Let $m, m'$ be in $\R$, $p, p',  K, N$ in $\N$, with $N\geq p+p'$, $\rho\geq0$.

(i) Let $a(U;\cdot)$ be a symbol in $\sGa{m}{K,0,p}{N}$ and $e(U;z,x,\xi)$ be a Poisson symbol 
in $\sP{m',\pm}{K,0,p'}{N}$. Set
\be\label{eq:tilde-e} 
\tilde{e}(U;z,x,\xi) =   (a\#e)_{\rho,N}   \qquad (resp.\ (e\#a)_{\rho,N} ) \, ,  
\ee
defined in \eqref{231} where $ z \in [-1,0] $ is considered as a parameter. Then 
$\tilde{e}$ is a Poisson symbol  in $\sP{m+m',\pm}{K,0,p+p'}{N}$, and  
\[
\opbw(a(U;\cdot))\circ\opbw(e(U;z,\cdot)),\ \opbw(e(U;z,\cdot))\circ\opbw(a(U;\cdot))
\]
may be written as 
$$
\opbw(\tilde{e}(U;z,\cdot))+R(U;z)
$$ 
where  $R(U;\cdot)$ is a smoothing remainder in $\sR{-\rho+m+m',\pm}{K,0,p+p'}{N}$. 

(ii) Let $a$ be  a symbol of $\sGa{m}{K,0,p}{N}$  and let $c(U;z,z',\cdot)$ be a Poisson symbol in $\sPi{m'}{K,0,p'}{N}$. Set
$$
\tilde{c}(U;z,z',\cdot) = (a\#c)_{\rho,N} \qquad (resp.\ (c\#a)_{\rho,N}) \, , 
$$ 
where $(z,z')$ is considered as a parameter. Then $ \tilde{c} $ is a Poisson symbol in $ \sPi{m+m'}{K,0,p+p'}{N} $
and the operator $ R(U)$ defined by 
\begin{multline*}
  R(U)V = \int_{-1}^0\opbw(a(U;\cdot))\circ \opbw(c(U;z,z',\cdot))V(z')\,dz'\\
-\int_{-1}^0\opbw(\tilde{c}(U;z,z',\cdot))V(z')\,dz',
\end{multline*}
(resp. by
\begin{multline*}
  R(U)V = \int_{-1}^0\opbw(c(U;z,z',\cdot))\circ \opbw(a(U;\cdot)) V(z')\,dz'\\
-\int_{-1}^0\opbw(\tilde{c}(U;z,z',\cdot))V(z')\,dz' \big)
\end{multline*}
belongs to 
$\sRi{-\rho+m+m'-1}{K,0,p+p'}{N}$. Identifying $ R $ with 
its kernel we simply write that $ R =
\opbw(a(U;\cdot))\circ \opbw(c(U;z,z',\cdot)) - \opbw(\tilde{c}(U;z,z',\cdot))  $ is in $\sRi{-\rho+m+m'-1}{K,0,p+p'}{N}$.

(iii) Let $c_j(U;z,z',\cdot)$ be in $\sPi{m_j}{K,0,p_j}{N}$ for $j = 1, 2$, with $ m_j \in \R $, $p_j$ in $\N$ with 
$p_1+p_2\leq N$. Then 
\begin{equation}
  \label{eq:5116}
  \tilde{c}(U;z,z',\cdot) \stackrel{\textrm{def}}{=}  \int_{-1}^0\bigl(c_1(U;z,z'',\cdot)\# c_2(U;z'',z',\cdot)\bigr)_{\rho,N}\,dz'' 
\end{equation}
is a symbol  in $\sPi{m_1+m_2-1}{K,0,p_1+p_2}{N}$ and the operator $R(U)$ defined 
by
 \begin{multline}
   \label{eq:5117}
R(U)V = \int_{-1}^0\int_{-1}^0\opbw(c_1(U;z,z'',\cdot))\circ\opbw(c_2(U;z'',z',\cdot)) V(z')\,dz''dz'\\
-\int_{-1}^0\opbw(\tilde{c}(U;z,z',\cdot))V(z')\,dz'
 \end{multline}
is in $\sRi{-\rho+m_1+m_2-1}{K,0,p_1+p_2}{N}$.

(iv) Let $ c $ be in $\sPi{m}{K,0,p}{N}$ and $ e $ be in $\sP{m',\pm}{K,0,p'}{N}$. Define
\be\label{eq:tilde-c}
\tilde{c}(U;z,\cdot) = \int_{-1}^0\bigl(c(U;z,z',\cdot)\# e(U;z',\cdot)\bigr)_{\rho,N}\,dz' \, .
\ee
Then $\tilde{c}$ is in $ \sP{m+m'-1,\pm}{K,0,p+p'}{N} $ and the operator $R(U)$ defined by
\begin{multline*}
  R(U) = \int_{-1}^0\opbw(c(U;z,z',\cdot))\circ\opbw(e(U;z',\cdot))\,dz'\\
- \opbw(\tilde{c}(U;z,\cdot))
\end{multline*}
is in $\sR{m+m'-1-\rho,\pm}{K,0,p+p'}{N}$.
\end{proposition}

\begin{proof}
  (i) We have to consider on the one hand the case when $e$ and $a$ are replaced by multilinear symbols in
  $\Pt{m',\pm}{q'}$ and $\Gt{m}{q}$ respectively. The composition result of Proposition~\ref{222} together with
  Definitions~\ref{513} and~\ref{514} bring the conclusion. 

On the other hand, we have to check that, if $a$ is in $\Gra{m}{K,0,q}$ and $e$ is in $\Prr{m',\pm}{K,0,q'}$, 
then the  operator 
\be\label{eq:RUz}
R (U; z ) = \opbw(a(U;\cdot))\circ\opbw(e(U;z,\cdot)) - \opbw(\tilde{e}(U;z,\cdot))  \, , 
\ee
where $ \tilde{e} = (a\#e)_{\rho,N} $, 
belongs to $\Rr{m+m'-\rho,\pm}{K,0,q+q'}$. In the case of sign $+$, for instance, for any integer $ i $, 
the symbol $ z^i e $ is in $\Gra{m'-i}{K,0,q'}$ and  Proposition~\ref{222} implies that 
$$
z^i R (U; z ) = \opbw( a )\circ\opbw( z^i e ) - \opbw( (a\# z^i e)_{\rho,N})  
$$
is a bounded family in $z$ of smoothing operators of $\Rr{-\rho+m+m'-i}{K,0,q+q'}$. 
We now prove that $ R $  is also a 
smoothing operator in  $\Rr{-\rho + m+m',+}{K,0,q+q'}$ according to 
Definition \ref{514}-(ii),  
i.e.\ $ R $ satisfies bounds of the form \eqref{5115b}. We first 
notice that, since
the remainder $ R (U; z ) $  in \eqref{RUz} is the difference of paradifferential operators, then,  
by \eqref{2122},  it does not enlarge much the support of the Fourier transform of functions, 
namely there is a function $ \tilde{\varphi}$ in $C^\infty_0(\R^*)$, equal to one on a large enough 
compact ring, such that,
setting  $\tilde{\Delta}_{\ell} = \tilde{\varphi}(2^{-\ell}D)$, for $\ell>0$,  we have
\be\label{eq:Delta-diag}
\Delta_\ell   R(U;z)  = \Delta_\ell R(U;z)\tilde{\Delta}_\ell \, , 
\ee
and a similar property for the zero frequency.  Using that the operators 
$ z^i R $,  $ i = 0, 1 $,  are in $\Rr{-\rho+m+m'-i}{K,0,q+q'}$, 
it follows from \eqref{Delta-diag} and estimates \eqref{2117} 
that
\begin{multline}
\label{eq:5117a}
  \norm{\Delta_\ell R(U; z)V}_{\Hds{s+\rho-m-m'}}\\ \leq C(1+2^\ell\abs{z})^{-1} 
  \Gcals{\sigma}{0,q+q'}{U}\norm{\tilde{\Delta}_\ell
    V}_{L^2} 2^{\ell s}\\
+C(1+2^\ell\abs{z})^{-1} \Gcals{\sigma}{0,q+q'-1}{U} \Gcals{s}{0,1}{U}\norm{\tilde{\Delta}_\ell
    V}_{L^2} 2^{\ell \sigma} \, .
\end{multline}
We deduce by  
$$  
\norm{\Delta_\ell R(U; z)V}_{L^2} \leq 2^{- \ell (s+\rho-m-m')} \norm{\Delta_\ell R(U; z)V}_{\Hds{s+\rho-m-m'}} \, , 
$$
\eqref{5117a} and  lemma \ref{CarSpacE}, 
the estimate
\begin{multline*}
  \norm{R(U;z)V}_{L^\infty([-1,0], \Hds{s+\rho-m-m'})} + \norm{R(U;z)V}_{L^2([-1,0],\Hds{s+\frac{1}{2}+\rho-m-m'})}\\
\leq C\bigl[\Gcals{\sigma}{0,q+q'}{U}\Gcals{s}{0,1}{V} + 
\Gcals{\sigma}{0,q+q'-1}{U}\Gcals{\sigma}{0,1}{V}\Gcals{s}{0,1}{U}\bigr]
\end{multline*}
i.e.\ a bound of the form \eqref{5115b} when $j=0$, $k=0$ and 
$ s  $ is replaced by $ s + \rho - m - m' $ and 
$ N $  by $ q+q' $. One obtains similarly the estimates involving $\partial_z$ or
$\partial_t$ derivatives.

(ii) We decompose $a$ (resp.\ $c$) as a sum of homogeneous symbols $a_q$ of $\Gt{m}{q}$ (resp.\ $c_{q'}(z;\cdot)$ of $\Pti{m'}{q'}$) and
of a non-homogeneous symbol of $\Gra{m}{K,0,N}$ (resp.\  of $\Prri{m'}{K,0,N}$). 
By the composition Proposition~\ref{222} we get in the expression of $R(U)$ on the one hand multilinear contributions 
\be\label{eq:contr-RHO}
\sum_{\pm}\int_{-1}^0 R_{q'',\pm}(U,\dots,U;z,z',\cdot)\1_{\pm(z-z')>0}V(z',\cdot)\,dz' \, , 
\ee
where $(z-z')^\ell\partial_z^j\partial_{z'}^{j'}R_{\pm,q''}$ is in
$\Rt{-\rho + m+m'-\ell+j+j'}{q''}$ uniformly in $z, z'$. 
Thus, according to (i) of Definition~\ref{514},  we get that  \eqref{contr-RHO}
is a smoothing operator of $\Rti{-\rho + m+m'-1}{q''}$. 
In the same way, the contributions with at least one nonhomogeneous component have the form 
\be\label{eq:contr-RN}
\sum_{\pm}\int_{-1}^0 R_{N,\pm}(U;z,z',\cdot)\1_{\pm(z-z')>0}V(z',\cdot)\,dz'
\ee
where $(z-z')^\ell\partial_z^j\partial_{z'}^{j'}R_{\pm,N}$ is in $\Rr{-\rho + m+m'-\ell +j+j'}{K,0,N}$ uniformly in $z, z'$. 
In order to  prove that \eqref{contr-RN} defines an operator in $ \Rri{-\rho + m+m' - 1 }{K,0,N} $
we have to show that it 
satisfies bounds of the form \eqref{5115c} with $ \rho $ replaced by $ \rho -m -m' +1 $. 
Let us consider for instance the case $ k = 0 $, $ j' = 0 $.
Using \eqref{Delta-diag} and \eqref{2117} we get
\begin{multline*}
\int_{-1}^0\norm{\Delta_\ell R_{N,\pm}(U;z,z')V(z',\cdot)}_{\Hds{s+\rho+1-m-m'}}\1_{\pm(z-z')>0}\,dz' \\
\leq C2^\ell\int_{-1}^0\norm{\Delta_\ell R_{N,\pm}(U;z,z') \tilde{\Delta}_\ell V(z',\cdot)}_{\Hds{s+\rho-m-m'}}\1_{\pm(z-z')>0}\,dz' 
\\ \leq C \int_{-1}^0(1+2^\ell\abs{z-z'})^{-2} 
  \Gcals{\sigma}{0,N}{U}\norm{\tilde{\Delta}_\ell
    V(z',\cdot)}_{L^2}\,dz' 2^{\ell (s +1)}\\
+C \int_{-1}^0(1+2^\ell\abs{z-z'})^{-2}  \Gcals{\sigma}{0,N-1}{U} \Gcals{s}{0,1}{U}\norm{\tilde{\Delta}_\ell
    V(z',\cdot)}_{L^2}\,dz' 2^{\ell (\sigma+1)} \, .
\end{multline*}
The $dz'$ integration makes gain one power of $2^{-\ell}$, that shows, using  lemma \ref{CarSpacE}, 
that \eqref{contr-RN} satisfies  \eqref{5115c} with $\rho $ replaced by $\rho + 1 - m - m' $ (in the case $ k = 0 $, $ j' = 0 $).

(iii) We decompose  the Poisson symbols $c_j = \sum_{+,-} c_{j,\pm}\1_{\pm(z-z')>0}$ 
according to in Definition~\ref{513}. Then  \eqref{5116} may
be written as the sum of the following contributions:
\begin{multline}\label{eq:5117aa}
  \1_{z-z'>0}\biggl[\int_{z'}^z\bigl(c_{1,+}(U;z,z'',\cdot)\# c_{2,+}(U;z'',z',\cdot)\bigr)_{\rho,N}\,dz''\\
+ \int_{-1}^{z'}\bigl(c_{1,+}(U;z,z'',\cdot)\# c_{2,-}(U;z'',z',\cdot)\bigr)_{\rho,N}\,dz''\\
+ \int_{z}^{0}\bigl(c_{1,-}(U;z,z'',\cdot)\# c_{2,+}(U;z'',z',\cdot)\bigr)_{\rho,N}\,dz''\biggr]\\
+  \1_{z-z'<0}\biggl[\int_{-1}^z\bigl(c_{1,+}(U;z,z'',\cdot)\# c_{2,-}(U;z'',z',\cdot)\bigr)_{\rho,N}\,dz''\\
+ \int_{z'}^{0}\bigl(c_{1,-}(U;z,z'',\cdot)\# c_{2,+}(U;z'',z',\cdot)\bigr)_{\rho,N}\,dz''\\
+ \int_{z}^{z'}\bigl(c_{1,-}(U;z,z'',\cdot)\# c_{2,-}(U;z'',z',\cdot)\bigr)_{\rho,N}\,dz''\biggr].
\end{multline}
Recalling  Definition~\ref{513}, the symbols 
$$
(1+\abs{z-z''}\absj{\xi})^2 c_{1,\pm}(U;z,z'',x,\xi)
$$ 
(resp.\ $(1+\abs{z''-z'}\absj{\xi})^2
c_{2,\pm}(U;z'',z',x,\xi)$) are   uniformly bounded in the class $\sGa{m_1}{K,0,p_1}{N}$ (resp.\  $\sGa{m_2}{K,0,p_2}{N}$ ). 
Combining  Definition \ref{221} of $(a\# b)_{\rho,N}$ together with the fact that the $dz''$-integration brings a $O(\absj{\xi}^{-1})$, each of the above integrals, and 
thus $\tilde{c}$ defined by \eqref{5116}, is a Poisson symbol in 
$ \sPi{m_1+m_2-1}{K,0,p_1+p_2}{N} $. The assertion that the operator
defined in \eqref{5117} is in $\sRi{-\rho+m_1+m_2-1}{K,0,p_1+p_2}{N} $
is a consequence of
Proposition~\ref{222} and the definition of smoothing operators, as in the proof of (i) and (ii) above.

(iv) We write the symbol $\tilde{c}$ in \eqref{tilde-c} as
\[
  \int_{-1}^z\bigl(c_+(U;z,z',\cdot)\# e(U;z',\cdot)\bigr)_{\rho,N}\,dz'
+ \int_z^0 \bigl(c_-(U;z,z',\cdot)\# e(U;z',\cdot)\bigr)_{\rho,N}\,dz'.
\]
Using that $ (1+\abs{z-z'}\absj{\xi})^2 c_{\pm}(U;z,z',x,\xi) $, resp. $ e(U;z',x, \xi) $, 
are  symbols  uniformly bounded in the class $\sGa{m}{K,0,p}{N}$, resp. $\sGa{m'}{K,0,p'}{N} $,
we conclude, arguing as in (iii), that 
$ \tilde{c} $ is in $ \sP{m+m'-1,\pm}{K,0,p+p'}{N} $ 
(gaining one on the order because of the $dz'$
integration). The claim that $ R(U) $ is in $ \sR{m+m'-1-\rho,\pm}{K,0,p+p'}{N} $ follows by 
the composition result from Proposition~\ref{222} arguing as in the proof of (i) and (ii) above.
\end{proof}

We may as well compose paradifferential and para-Poisson operators with smoothing operators.
\begin{proposition}
  \label{518} 
(i) Let $a$ be a symbol in $\sGa{m}{K,0,p_1}{N}$ and $R$ be 
a smoothing operator in $\sR{-\rho,\pm}{K,0,p_2}{N}$ (resp.\ $\sRi{-\rho}{K,0,p_2}{N}$). Then 
$\opbw(a(U;\cdot))\circ R(U)$ is in $\sR{-\rho+m,\pm}{K,0,p_1+p_2}{N}$ (resp.\ $\sRi{-\rho+m}{K,0,p_1+p_2}{N}$).

(ii) Let $c_1$ be a Poisson symbol in $\sPi{m}{K,0,p_1}{N}$ and $R$ be in $\sRi{-\rho}{K,0,p_2}{N}$. Then
\begin{equation}
  \label{eq:5117b}
  V\to \int_{-1}^0\opbw(c_1(U;z,z'',\cdot))\circ R(U)V(z'', \cdot )\,dz''
\end{equation}
is in $\sRi{-\rho+m-1}{K,0,p_1+p_2}{N}$.

(iii) Let $c_1$ be in $\sPi{m}{K,0,p_1}{N}$ as above and $R$ be in $\sR{-\rho,\pm}{K,0,p_2}{N}$. Then 
\[\int_{-1}^0\opbw(c_1(U;z,z',\cdot))\circ R(U; z')\,dz'\]
is in $\sR{-\rho+m-1,\pm}{K,0,p_1+p_2}{N}$.
\end{proposition}
\begin{proof}
  In the case when $R$ is in $\sR{-\rho,\pm}{K,0,p_2}{N}$, (i) follows from the fact that the multilinear contributions to
  the composition are of the form \eqref{235}, so they form a 
  uniformly bounded family of elements of $\Rt{-\rho+m}{p}$ for some $p$ between $p_1+p_2$ and $N$. If we make act
  $z^\ell\partial_z^j$ or $(1+z)^\ell\partial_z^j$, we get a similar conclusion with $\rho$ replaced by $\rho+\ell-j$. This
  shows that the multilinear expressions belong to $\Rt{-\rho+m,\pm}{p}$, $p_1+p_2\leq p\leq N-1$. Concerning the
  contributions vanishing at order $N$ at $U=0$, one has just to use the boundedness of $\opbw(a(U; \cdot))$ from $E^s_j$ to
  $E^{s-m}_j$ for any $s$ and $j$, to deduce from the estimates \eqref{5115b} similar ones for the composition with
  $\opbw(a(U;\cdot))$. The case of $R$ in $\sRi{-\rho}{K,0,p_2}{N}$ is similar.

(ii) Let $ c_{1,p_1'} $ be an homogeneous
Poisson symbol in $ \Pti{m}{p_1'} $ for some $ p_1 \leq p_1' \leq N - 1 $ 
and $ R_{p'_2} $ be the kernel of an homogeneous smoothing operator  
in $ \Rti{-\rho}{p_2'} $ for  some $ p_2 \leq p_2' \leq N - 1 $.
The  contribution of these multilinear terms to the operator in \eqref{5117b} is
\begin{multline}\label{eq:5117c}
V \to  \int_{-1}^0  \Big\{ \int_{-1}^0 \opbw\bigl(c_{1,p_1'}(U_1,\dots,U_{p'_1};z,z'',\cdot)\bigr)\\
\circ  R_{p'_2}(U_{p'_1+1},\dots,U_{p'_1+p'_2};z'',z')  \,dz'' \Big\} V(z' , \cdot ) \, dz'  .
\end{multline}
According to \eqref{a+a-} and to Definition~\ref{514}, we may decompose
\[
c_{1,p'_1} = c_{1,p'_1,+}\1_{z>z''} + c_{1,p'_1,-}\1_{z<z''},\ R_{p'_2} = R_{p'_2,+}\1_{z''>z'} + R_{p'_2,-}\1_{z''<z'} \, .
\]
Proceeding as in \eqref{5117aa}, the kernel of the 
operator \eqref{5117c} may be written as the product of $\1_{\pm(z-z')>0}$ and
of integrals of the form
\begin{multline}\label{eq:5117d}
  \int_{I(z,z')} \opbw\bigl(c_{1,p_1',\pm}(U_1,\dots,U_{p'_1};z,z'',\cdot)\bigr)\\
\circ  R_{p'_2,\pm}(U_{p'_1+1},\dots,U_{p'_1+p'_2};z'',z')  \,dz''
\end{multline}
where $I(z,z')$ are intervals of integration like those in \eqref{5117aa}, and we have to show 
that if we make act
$(z-z')^\ell\partial_z^j\partial_{z'}^{j'}$ on \eqref{5117d}, we get  a bounded family 
of operators of $\Rt{-(\rho - m +1) + j+j'-\ell+1}{p'_1+p'_2} $
uniformly in $  z, z' $. Let us just treat the case $\ell = j = j' =0$. 
Recalling  Definition \ref{214} we have to prove 
 bounds as in \eqref{2115} for  
\begin{multline}\label{eq:Rint-int}
  \sum_{n'_0}\int_{I(z,z')}\Pin{0}\opbw\bigl(c_{1,p_1',\pm}(\Pin{1}U_1,\dots,\Pi_{n_{p'_1}}U_{p'_1};z,z'',\cdot)\bigr)\\
\circ   \Pi_{n'_0} R_{p'_2,\pm}(\Pi_{n_{p'_1+1}}U_{p'_1+1},\dots,\Pi_{n_{p'_1+p'_2}}U_{p'_1+p'_2};z'',z')  \\
\times \Pi_{n_{p'_1+p'_2+1}}V   \,dz'' 
\end{multline}
(similarly to \eqref{235}) where, by \eqref{pro-x-in}  (applied to $ c_{1, p_1',\pm} $) and \eqref{2116}, 
the indices in the sum \eqref{235} satisfy, for some choice of the signs 
$ \epsilon_j $ 
\begin{equation}\label{eq:5117e}
\sum_{j=0}^{p_1'} \epsilon_j n_j = n_0'   \, , \qquad  n_0' = \sum_{j= p_1'+1}^{p_1' + p_2'+1} \epsilon_j n_j \, . 
\end{equation}
As a consequence the operator in \eqref{Rint-int} satisfies, for any $ z, z' $,  the corresponding
condition \eqref{2116}.
Since  $ c_{1,p_1'} $ is in $ \Pti{m}{p_1'} $ (see Definition~\ref{513})
we apply \eqref{2121} (with $ s = m $) for the first  operator in \eqref{Rint-int}. 
Then, 
we apply  \eqref{2115} to $ R_{p'_2, \pm} $  
and to $ (z'-z'')^2 R_{p'_2, \pm} $ that, according to Definition~\ref{514}, is a
family of smoothing operators in $\Rt{-\rho-2+1}{p_2'} $, uniformly in $ z', z'' $.
Consequently we  estimate the $L^2$ norm of 
\eqref{Rint-int} by
\begin{multline*}
  C n_1^\sigma\dots n_{p'_1}^\sigma |n_0'|^m
  \frac{\max_2(n_{p'_1+1},\dots,n_{p'_1+p'_2+1})^{\rho+\mu}}{\max(n_{p'_1+1},\dots,n_{p'_1+p'_2+1})^{\rho-1}}
  \prod_{\ell=1}^{p'_1+p'_2}\norm{\Pin{\ell}U_\ell}_{L^2}  \\
\times \norm{\Pin{p'_1+p'_2+1}V}_{L^2} \int_{-1}^0\bigl(1+\max(n_{p'_1+1},\dots,n_{p'_1+p'_2+1})\abs{z'-z''}\bigr)^{-2}\,dz'' \, . 
\end{multline*}
\noindent Changing variable in the integral we get a factor $ \max(n_{p'_1+1},\dots,n_{p'_1+p'_2+1})^{-1} $, and   
$n_1,\dots,n_{p'_1}\leq C \max_2(n_1,\dots,n_{p'_1+p'_2+1})$ 
and $ \max (n_{p_1'+1}, \ldots, n_{p_1'+p_2'+1}) \sim \max (n_{1}, \ldots, n_{p_1'+p_2'+1}) $ 
(that we deduce as \eqref{rel-ind-diff} by $n_1,\dots,n_{p'_1}\ll n_0\sim n'_0$ and \eqref{5117e}),
imply  a bound in
\[
C \frac{\max_2(n_1,\dots,n_{p'_1+p'_2+1})^{\rho-m+\mu'}}{\max(n_1,\dots,n_{p'_1+p'_2+1})^{\rho-m}}
  \prod_{\ell=1}^{p'_1+p'_2}\norm{\Pin{\ell}U_\ell}_{L^2}\norm{\Pin{p'_1+p'_2+1}V}_{L^2}
  \]
for some new $\mu'$. In conclusion \eqref{Rint-int} is in  $\Rt{-(\rho - m +1)+1}{p'_1+p'_2} $.

For the terms of \eqref{5117b}
vanishing at order $N$, one uses lemma~\ref{516} together with estimates \eqref{5115c} defining 
$\Rri{-\rho}{K,0,p'_2}$.

(iii) is proved in a similar way as (ii).
\end{proof}
To conclude this section, we shall establish several composition results involving 
the smoothing operators 
introduced in Definition~\ref{512}, that will be useful in Chapter~\ref{cha:6}.

In the sequel, we shall consider sometimes functions of $(t,x)\in I\times\Tu$ as functions of $(z,t,x)\in [-1,0]\times
I\times\Tu$ independent of $z$. In particular, a function in $\CKH{s}{\C}$ may be considered 
as an element of $C_*^K(I,F^s_j)$, for any $j$, where the space $F_j^s$ is defined in \eqref{511}. 

\begin{lemma}
  \label{519} (i) Let $ m $ be in $\R$, $\rho$ positive, $ K, p, q, N, j$ in $\N$, $p+q\leq N$, $r>0$. 
  Let $a$ be a Poisson symbol in $\sPi{m}{K,0,p}{N}$ 
  and $ R $ a smoothing operator in the class $\sRb{-\rho}{K,0,j,q}{N}$ of Definition~\ref{512}. Consider
  \begin{equation}
    \label{eq:5118}
    \tilde{R}(U)V = \int_{-1}^0\opbw(a(U;z,z',\cdot))R(U;z')V\,dz'
  \end{equation}
defined for $U$ in $\Br{K}{}$  
and $V$ in $\CKH{\sigma}{\C}$, 
for some large enough $\sigma$ (and are
considered as functions of $ z'$, constant in the variable $z'$).
Then,   for any $ j' \leq j + 1 $, $\partial_z^{j'}\tilde{R}(U)$ is a bounded family in $z$ of smoothing operators of 
$\sRa{-\rho+m+j' }{K,0,p+q}{N}$
(here we adopt the abuse of notation introduced in the last remark after Definition \ref{512}).

(ii) Let $R_{\mathrm{int}}$ be a smoothing operator in $\sRi{-\rho'}{K,0,p}{N}$ for some $\rho'\geq0 $
(see \eqref{Resto-int}), 
and take $ R $ in $ \sRb{-\rho}{K,0,j,q}{N} $ with  $ \rho \geq \frac{3}{2}$. 
Define
\begin{equation} \label{eq:5119}
\tilde{R}(U)V = R_{\mathrm{int}}(U)[R(U)V] = \int_{-1}^0 R_{\mathrm{int}}(U; z, z') [R(U)V ](z', 	\cdot ) d z'
\end{equation}
for $U \in \Br{K}{}$  and $V$ in $\CKH{\sigma}{\C}$, the last expression in \eqref{5119} above being the one in terms of the
Schwartz kernel of the operator. 

Then, for any $ j' \leq j $, $ \partial_z^{j'} \tilde{R}(U) $ is 
a bounded family in $ z $ of  operators in 
$ \sRa{-\rho'+j'+\epsilon}{K,0,p+q}{N} $ for any $ \epsilon > 0 $.
\end{lemma}

\begin{proof}
  (i) We consider on the one hand multilinear contributions, and on the other hand those terms vanishing at least at order
  $ N $ at $ U = 0 $. Let us first consider the homogeneous terms.
  Let $ a $ be a symbol in $ \Pti{m}{p} $ for some $p$ and let $R$ be in $\Rtb{-\rho}{q}$ 
   for some $q$. Let us compute for instance
     \begin{multline}
    \label{eq:5120}
\partial_z\int_{-1}^0\opbw\bigl(a(\Pi_{n'} \Ucal';z,z',\cdot)\bigr) \circ R(\Pi_{n''} \Ucal'' ;z')\Pin{p+q+1}U_{p+q+1}\,dz'
  \end{multline}
where $ \Ucal' = (U_1,\dots,U_p)$, $n' = (n_1,\dots,n_p)$, $\Ucal'' = (U_{p+1},\dots,U_{p+q})$, $n'' =
(n_{p+1},\dots,n_{p+q})$  and $ U_j $ are functions 
only of $ x $. 
With the notation $ a = a_+ \1_{z-z'>0} + a_- \1_{z-z'<0} $ in \eqref{a+a-}, 
we may write \eqref{5120} as the sum of 
\begin{equation}
  \label{eq:5121}
  \opbw(a_+(\Pi_{n'}\Ucal';z,z ,\cdot)R(\Pi_{n''}\Ucal'';z)\Pin{p+q+1}U_{p+q+1}
\end{equation}
\begin{equation}
  \label{eq:5122}
  -  \opbw(a_-(\Pi_{n'}\Ucal';z,z ,\cdot)R(\Pi_{n''}\Ucal'';z)\Pin{p+q+1}U_{p+q+1}
\end{equation}
\begin{equation}
  \label{eq:5123}
  \int_{-1}^0  \opbw(a_1(\Pi_{n'}\Ucal';z,z',\cdot)R(\Pi_{n''}\Ucal'';z')\Pin{p+q+1}U_{p+q+1}\,dz'
\end{equation}
where  $ a_1 = (\partial_z a_+)\1_{z-z'>0} +  (\partial_z a_-)\1_{z-z'<0} $. We make act at the left of
\eqref{5121}, \eqref{5122}, \eqref{5123} another projector $\Pin{0}$ and bound the $L^2(\Tu)$ norm at fixed $z$. 
Let us make in detail the estimate for \eqref{5123}. We write 
\begin{equation} \label{eq:5123bis}
\sum_{n_0'}  \int_{-1}^0  \Pi_{n_0} \opbw(a_1(\Pi_{n'}\Ucal';z,z',\cdot)
\Pi_{n_0'} R(\Pi_{n''}\Ucal'';z')\Pin{p+q+1}U_{p+q+1}\,dz'
\end{equation}
  where (because of
conditions \eqref{pro-x-in} and \eqref{2116}) the index $ n'_0 $ is such that 
\be\label{eq:indices-r}
\sum_{j = 0}^p \epsilon_j n_j + \epsilon_0' n'_0 = 0  \ \, , \quad 
\epsilon''_0 n'_0 + \sum_{j= p+1}^{p+q+1}\epsilon_j n_j = 0 \, ,  
\ee
for some choice of the signs $\epsilon_j, \epsilon'_0, \epsilon_0'' \in \{-1, 1\}$.
By \eqref{indices-r}  we deduce that, 
if  \eqref{5123bis} is not zero, then  
$\sum_{j=0}^{p+q+1}\epsilon_jn_j = 0 $ for some choice of signs $\epsilon_j $ so that 
\eqref{5123} satisfies the corresponding condition \eqref{2116}. 

Since 
$ a $ is a Poisson symbol in $ \Pti{m}{p} $, the remarks following Definition \ref{513} 
imply that $ a_1 = \pa_z a $ is in $  \Pti{m+1}{p}  $ and $ | z- z' |^\theta a_1 $ is in 
$ \Pti{m+1-\theta}{p}  $,   for any $\theta\in ]0,1[$. Thus 
$ | z- z' |^\theta a_1 $ is in $\Gt{m+1- \theta}{p} $, uniformly in $z, z'$,  and 
Proposition \ref{215} implies that the $ L^2 $ norm of 
\eqref{5123bis} is bounded by
$$
  C\abs{n'}^\mu n_0^{m+1-\theta}\prod_1^p\norm{U_j}_{L^2} 
  \int_{-1}^0\norm{\Pi_{n'_0}R(\Pi_{n''}\Ucal'';z')\Pin{p+q+1}U_{p+q+1}}_{L^2}\frac{dz'}{\abs{z-z'}^\theta} \, . 
$$
Applying \eqref{518} to the term $ R $ that belongs to $\Rtb{-\rho}{q}$, and integrating in $ d z' $, we obtain 
a bound in
\[
C\abs{n'}^\mu n_0^{m+1-\theta}\frac{\maxdn{p+1}{p+q+1}^{\rho+\mu'}}{\maxn{p+1}{p+q+1}^\rho}\prod_1^{p+q+1}\norm{U_j}_{L^2} \, .
\]
By \eqref{2122}, the sum  \eqref{5123bis} is restricted to indices $  n_0' \sim n_0  $ and 
 \eqref{indices-r} implies that $ n_0' \leq \maxn{p+1}{p+q+1} $. 
 Moreover, arguing as for \eqref{rel-ind-diff}, 
\begin{equation*}
\begin{split}
\maxn{p+1}{p+q+1} \sim \maxn{1}{p+q+1}  \\
|n'| \leq C\maxdn{1}{p+q+1} \qquad \qquad 
\end{split}
\end{equation*}
and one gets an estimate of the form
\eqref{2115}  
with $\rho$ replaced by $\rho-m-1+\theta$, $ p + 1 $ replaced by $ p + q+ 1 $, and for a new value of $\mu$. 
This proves that \eqref{5123} is in $\Rt{-\rho+m+ \epsilon}{p+q}$ with $ \epsilon = 1-\theta $  in $ ]0,1[ $,
uniformly in $ z $,  and the same holds for the terms \eqref{5121}, \eqref{5122}. 
Higher order $ \partial_z $ derivatives may be treated in the same way obtaining that 
$ \pa_z^{j'} R $ is in $ \Rt{-\rho+m+ j' - 1 + \epsilon}{p+q} $.  

Consider now the case when $a$ is a non-homogeneous symbol in $\Prri{m}{K,0,p}$ and 
$R$ is a smoothing operator in $\Rrb{-\rho}{K,0,j,N}$. 
Writing  $ a = a_+ \1_{z-z'>0} + a_- \1_{z-z'<0} $ we have that 
\begin{equation}
\label{eq:5124}
\partial_z\int_{-1}^0\opbw(a(U;z,z',\cdot))R(U; z')V\,dz'
\end{equation}
is equal to 
\begin{multline}
  \label{eq:5125}
\opbw(a_+(U;z,z,\cdot))R(U; z)V - \opbw(a_-(U;z,z,\cdot))R(U;z)V\\
+ \int_{-1}^0\opbw(a_1(U;z,z',\cdot))R(U; z')V\,dz'  
\end{multline}
where $ a_1 = (\partial_z a_+)\1_{z-z'>0} +  (\partial_z a_-)\1_{z-z'<0} $. 
Combining estimates \eqref{519} for $ R $,  with the boundedness of $\opbw(a_\pm)$ (resp.\
  $\opbw(\abs{z-z'}^\theta a_1)$) from $\Hds{s+\rho}$ to the space $\Hds{s+\rho-m}$ (resp.\  from $\Hds{s+\rho}$ to
  $\Hds{s+\rho-m-1+\theta}$), that follows 
  by Proposition \ref{215}  and the fact that these symbols are in $ \sGa{m}{K,0,p}{N}  $ 
  (resp.\ $  \sGa{m+1-\theta}{K,0,p}{N}  $), we get for \eqref{5125} bounds of the form 
  \eqref{2117} with $\rho$ replaced by $\rho-(m+1)+\theta $, uniformly in $ z $. We have
    proved that $\partial_z \tilde{R}(U)$ is in $\Rra{-\rho+m+\epsilon}{K,0,N}$ 
    with $ \epsilon = 1-\theta $  in $ ]0,1[ $. Higher $\partial_z$  
    derivatives are treated in the same way. 
    
    Finally, we have  to consider the contributions $ {\rm Op}^{BW} (a_N) \circ R_q $ 
    where $  R_q $ is an  homogeneous smoothing term  in $\Rtb{-\rho}{q}$.  
Because of the  third remark after Definition \ref{512}, the 
associated operator $ R_q (U, \ldots, U )$ satisfies 
\eqref{519} with $ \rho - \alpha $ instead of $ \rho $ for $ \alpha > 1 / 2 $ and we obtain 
 that the operator \eqref{5124} is 
a remainder  in $\Rra{-\rho+m+\epsilon+\alpha}{K,0,N} 
\subset\Rra{-\rho+m+1}{K,0,N} $, taking $ \alpha + \epsilon \leq 1 $.  
 This concludes the proof of (i).

(ii) Let us study first the multilinear contributions. 
According to Definition \ref{514}-(i)  we write the 
homogeneous smoothing operator $ R_{\mathrm{int}}  $ as
$$
V \to \int_{-1}^{0} \big(R_{\mathrm{int},+} \1_{z-z'>0} +  R_{\mathrm{int}, -} \1_{z-z'<0}\big)V(z', \cdot) d z' 
$$
where 
$ R_{\mathrm{int},\pm} (U_1, \ldots, U_p; z, z' )$ 
are homogeneous smoothing operators of $\Rt{-\rho'+1}{p}$, uniformly in $ z, z' $. 
The analogous of \eqref{5120} with $\opbw(a)$ replaced by
$R_{\mathrm{int}}(\Pin{1}U_1,\dots,\Pin{p}U_p;z,z')$ 
is the sum (similarly to \eqref{5121}-\eqref{5123})  of 
$$
R_{\mathrm{int},+}(\Pi_{n'}\Ucal';z,z)R(\Pi_{n''}\Ucal'';z)\Pin{p+q+1}U_{p+q+1} 
$$
$$
-R_{\mathrm{int},-}(\Pi_{n'}\Ucal';z,z)R(\Pi_{n''}\Ucal'';z)\Pin{p+q+1}U_{p+q+1}
$$
\begin{equation}
\label{eq:5125a}
\int_{-1}^0 R_{\mathrm{int},1}(\Pi_{n'}\Ucal';z,z')R(\Pi_{n''}\Ucal'';z')\Pin{p+q+1}U_{p+q+1}\,dz'
\end{equation}
where $ R_{\mathrm{int},1} = (\partial_z R_{\mathrm{int},+})\1_{z-z'>0} 
+  (\partial_z R_{\mathrm{int}, -})\1_{z-z'<0} $.
Recalling the third remark following Definition \ref{514} 
we have that  $\abs{z-z'}^\theta R_{\mathrm{int},1}$ is in $ \Rti{-\rho'+1-\theta}{p} $, that is 
$ (\abs{z-z'}^\theta R_{\mathrm{int},1}(\Pi_{n'}\Ucal'; z, z'))_{z,z' \in [-1,0]} $ is a bounded family of
 smoothing operators of $\Rt{-\rho'+2-\theta}{p}$, uniformly in $ z, z' $. 
By \eqref{518} and \eqref{2115} we estimate the $ L^2 $ norm of
$$ 
\sum_{n_p'} \Pin{0}  \int_{-1}^0 R_{\mathrm{int},1}(\Pi_{n'}\Ucal';z,z') \Pi_{n'_p} 
R(\Pi_{n''}\Ucal'';z')\Pin{p+q+1}U_{p+q+1}\,dz' 
$$ 
by   the product of $  \prod_{j=1}^{p+q+1}\norm{U_j}_{L^2}$ times 
\be\label{eq:upper-b}
C\frac{\max_2(n_1,\dots,n_p,n'_p)^{\rho' - 2 + \theta +\mu}}{\max(n_1,\dots,n_p,n'_p)^{\rho'-2+\theta}}
\frac{\max_2(n_{p+1},\dots,n_{p+q+1})^{\rho+\mu}}{\max(n_{p+1},\dots,n_{p+q+1})^{\rho}} \, . 
\ee
By  \eqref{2116} we know that, for some choice of the signs $\epsilon'_j, \epsilon''_j $, 
\be\label{eq:auton2}
\epsilon'_0n_0+\sum_1^p\epsilon'_jn_j + \epsilon'_pn'_p = 0 \, ,  \quad  
\epsilon''_pn'_p+ \sum_{p+1}^{p+q+1}\epsilon''_jn_j = 0 \, .
\ee
Assume first that one among $ n_1,\dots, n_p, n_{p+1}, \dots , n_{p+q+1} $ is 
much larger than all the other ones, say $ n_j $. If $ 1 \leq j \leq p $ then, using \eqref{auton2} 
we derive that $ n_p' \leq n_p $ and \eqref{upper-b} is bounded by 
\be\label{eq:finalRR}
C \frac{\max_2(n_1,\dots,n_{p+q+1})^{\rho' - 2 + \theta +\mu}}{\max(n_1,\dots,n_{p+q+1})^{\rho'-2+\theta}}
\ee
for some new value of $ \mu $. 
If $ p + 1 \leq j \leq p + q + 1 $ then $ n_p' \sim \max (n_{1},\dots,n_{p+q+1}) $ 
and \eqref{finalRR} follows as well. 
When the largest two among $ n_1,\dots,n_{p+q+1} $
are of the same magnitude, i.e $ \max_2(n_1,\dots,n_{p+q+1})
\sim \max(n_1,\dots,n_{p+q+1}) $, the bound \eqref{finalRR} follows immediately.
This proves that \eqref{5125a} is a bounded family in 
$ z $ of smoothing  operators of $\Rt{-\rho'+1+\epsilon}{p}$ where $ \epsilon = 1 - \theta $. 

Consider next the contributions that are non homogeneous, i.e.\  for $0\leq j'\leq j$ the operator
\begin{equation}
  \label{eq:5125bis}
  V\to \partial_z^{j'}(R_{\mathrm{int},p}(U)\cdot R_q(U,\cdot)V)(z,\cdot)
\end{equation}
where $ R_{\mathrm{int},p} $ is in $\Rri{-\rho'}{K,0,p}$ and $R_q \in  \Rrb{-\rho}{K,0,j,q} $. 
We want to show that
 \eqref{5125bis} defines a family indexed by $z$ of  operators of $\Rr{-\rho'+j'}{K,0,p+q}$,
with bounds uniform in $z$. According to \eqref{2117}, 
 we have to show an estimate
  \begin{multline}
   \label{eq:5125ter}
\norm{\partial_t^k\partial_z^{j'} (R_{\mathrm{int},p}(U)\cdot R_q(U,\cdot)V)(z,\cdot)}_{\Hds{s+\rho'-\frac{3}{2}k-j'}}\\
\leq C\sum_{k'+k''=k}\Bigl[ \Gcals{\sigma}{k',p+q}{U} \Gcals{s}{k'',1}{V}  + 
\Gcals{\sigma}{k',p+q-1}{U} \Gcals{s}{k'',1}{U} \Gcals{\sigma}{k',1}{V}\Bigr].
 \end{multline} 
By \eqref{5115c} applied to $R_{\mathrm{int},p}$ with $\rho$ replaced by $\rho'$, we have a bound of the left hand side of
\eqref{5125ter} by 
\begin{multline}
  \label{eq:5125qua}
C\Bigl[\sum_{k'+k''=k}\Gcals{\sigma}{k',p}{U}\Gcalst{s,j'}{k'',1}{R_q(U,\cdot)V} \\+ 
\sum_{k'+k''=k}\Gcals{\sigma}{k',p-1}{U}\Gcalst{\sigma,j'}{k'',1}{R_q(U,\cdot)V}\Gcals{s}{k',1}{U}\Bigr].
\end{multline}
Recalling the notation \eqref{5115a}, the general term in the first sum in \eqref{5125qua} is bounded by 
\begin{multline*}
  \Gcals{\sigma}{k',p}{U}\Bigl(\sum_{k''_1=0}^{k''}\norm{\partial_t^{k''_1} R_q(U,\cdot)V}_{E_{j'}^{s-\frac{3}{2}k''_1}}\Bigr)\\
\leq \Gcals{\sigma}{k',p}{U}\Bigl(\sum_{k''_1=0}^{k''}\norm{\partial_t^{k''_1} R_q(U,\cdot)V}_{F_{j'}^{s+\frac{1}{2}-\frac{3}{2}k''_1}}\Bigr)
\end{multline*}
(see the remark after Definition~\ref{515}). 
The above expression is bounded according to \eqref{519} 
(that we apply with $ \rho $ replaced by 
$ \rho - \alpha $ with $ \alpha = 3/ 4 $, by the third remark following Definition \ref{512}), 
and using that $ \rho - \alpha > \frac12 $ since $ \rho \geq \frac32 $,   
by the right hand side of \eqref{5125ter}. 
The general term in the second sum in \eqref{5125qua} is bounded
in the same way (changing the definition of $\sigma$ in the final result).
\end{proof}

Finally, we shall state a result of the same type as lemma~\ref{519}, when one replaces in \eqref{5118}, \eqref{5119} the
smoothing map $R(U)$ by an operator $ M(U;z) $ that may lose derivatives, but which is supported for $z$ in a compact subset of $[-1,0[$. We
shall estimate then the quantities \eqref{5118}, \eqref{5119} only for $z$ in a neighborhood of zero.

We now consider operators $ M(U) = M(U; z) $ of the form 
\be\label{eq:MUVB}
M(U)V = \sum_{q'=q}^{N-1}M_{q'}(U,\dots,U)V + M_N(U)V \, , 
\ee 
where $(U_1,\dots,U_{q'+1})\to M_{q'}(U_1,\dots,U_{q'})U_{q'+1}$ is a family of $(q'+1)$-linear maps, defined on
$\Hds{\infty}(\Tu,\C^2)^{q'}\times \Hds{\infty}(\Tu,\C)$, satisfying for each fixed $z \in [-1, 0] $ conditions \eqref{2126}, and,
for some $m\geq 0$, and for any $n_0,\dots,n_{q'+1}$, the estimate 
\begin{multline}
  \label{eq:5126}
\sup_{-1\leq z\leq 0}\norm{\Pin{0}M_{q'}(\Pin{}\Ucal)\Pin{q'+1}U_{q'+1}}_{\Hds{0}}\\
\leq C(n_0+\cdots+n_{q'+1})^{m}\Gcalsm{0}{0,q'+1}{\Ucal,U_{q'+1}}
\end{multline}
with the notation $\Ucal = (U_1,\dots,U_{q'+1})$. Moreover, we suppose that for some $\sigma'_0\geq
\sigma_0$ large enough, we have for $0\leq k\leq K$ the estimate
\begin{equation}
  \label{eq:5127}
\sup_{-1\leq z\leq 0}\norm{\partial_t^k M_N(U)V}_{\Hds{\sigma_0}}
\leq  C\sum_{k'+k''=k}\Gcals{\sigma'_0}{k',N}{U}\Gcals{\sigma'_0}{k'',1}{V} \, .
\end{equation}
Finally, we assume that for some $\delta\in ]0,1[$,
\begin{equation}
  \label{eq:5128}
  M(U)V\equiv 0 \quad \textrm{ if } \ -2\delta\leq z\leq 0 \, .
\end{equation}
We may now state the following lemma which is a variant of  lemma~\ref{519}. 
\begin{lemma} \label{5110}
Let $ m $ be in $\R$, $\rho$ positive, $ K, p, q, N $  in $\N$,  $p+q\leq N$, $r>0$. 
 Let $a$ be a Poisson symbol in $\sPi{m}{K,0,p}{N}$ and
$R_{\mathrm{int}}$ be a smoothing operator in $\sRi{-\rho}{K,0,p}{N}$.
Let $ (M(U; z))_{z \in [-1,0]} $ be a family of operators as in \eqref{MUVB} satisfying \eqref{5126}-\eqref{5128}.  Define
\begin{equation}
  \label{eq:5129}
  \tilde{R}(U)V = \int_{-1}^0\opbw(a(U;z,z',\cdot))M(U;z')V\,dz'
\end{equation}
or
\begin{equation} \label{eq:5130}
\tilde{R}(U)V = \int_{-1}^0R_{\mathrm{int}}(U;z,z')M(U;z')V\, dz' \, .
\end{equation}
Then $\partial_{z}^{j}\tilde{R}(U)$ restricted to $z\in [-\delta,0]$ is, for any $j \in \N $, a bounded family in $z$ 
of  smoothing operators of
$\sRa{-\rho}{K,0,p+q}{N}$ for any $\rho$.
\end{lemma}
\begin{proof}
  Consider first \eqref{5129}. If $z'$ stays in the support of $M(U;z')$ and $-\delta\leq z\leq 0$, we may write 
\[a(U;z,z',\cdot) = (z-z')^\ell a(U;z,z',\cdot)\omega(z,z')\]
where $\omega$ is smooth for $z'\leq -2\delta, -\delta\leq z\leq 0$. According to Definition~\ref{513}, for those values of
$z, z'$, we may write $a= \sum_{p'=p}^{N-1} a_{p'} + a_N$ where $a_{p'}$ is a family in $(z,z')$ of 
homogeneous symbols in $ \Gt{m-\ell}{p'} $ satisfying
estimates \eqref{214} with $m$  replaced by $m-\ell$,  and some $\mu$ independent of $\ell$, and where $a_N$ satisfies
\eqref{218}, with $m, K'$ replaced by $ m - \ell, 0 $ and some $ \sigma, \sigma_0 $ independent of $ \ell $. 
By \eqref{2121}, and denoting
$\Ucal' = (U_1,\dots,U_{p'})$, $n' = (n_1,\dots,n_{p'})$, we have an estimate of the form
\begin{multline}\label{eq:st1}
  \norm{\Pin{0}\opbw(a_{p'}(\Pi_{n'}\Ucal';z,z',\cdot))\Pi_{n'_{p'+1}}V}_{\Hds{0}}\\
\leq C n_0^{m-\ell}\abs{n'}^\mu \Gcalsm{0}{0,p'}{\Pi_{n'}\Ucal'} \norm{\Pi_{n'_{p'+1}}V}_{\Hds{0}}
\end{multline}
for some $ \mu $ independent of $ \ell $. 
Moreover \eqref{215}
and \eqref{2122} imply that  
\be\label{eq:para-b}
\sum_0^{p'} \epsilon_j n_j + \epsilon'_{p'+1} n'_{p'+1} = 0 \, , \quad \abs{n'}\ll n_0\sim n'_{p'+1} 
\ee
for some choice of signs $ \epsilon_j $ and $ \epsilon'_{p'+1} $. 
In \eqref{st1} we replace $ V $ by the quantity  $ M_{q'}(\Pi_{n''}\Ucal'')\Pin{p'+q'+1}U_{p'+q'+1}$ where
 $\Ucal'' = (U_{p'+1},\dots,U_{p'+q'})$, $n'' = (n_{p'+1},\dots,n_{p'+q'})$, that, by \eqref{2126}, 
 is different from zero only if 
\be\label{eq:aut2}
n_{p'+1}'  = \sum_{p' + 1}^{p'+q'+1} \epsilon_j n_j 
\ee 
Thus, using \eqref{5126} for $M_{q'}$,
we  get for the resulting expression a bound in
\begin{multline}
  \label{eq:5131}
C  n_0^{m-\ell}(n_1+\cdots+n_{p'})^\mu ( n_{p'+1}+\cdots+n_{p'+q'+1})^{m} \\\times
\Gcals{0}{0,p'+q'+1}{\Pi_{n'}\Ucal',\Pi_{n''}\Ucal'',\Pi_{n_{p'+q'+1}}U_{p'+q'+1}}
\end{multline}
where $\mu$ and $m$ are independent of $\ell$. If among the indices $n_1,\dots,n_{p'+q'+1}$ the largest two ones are of the
same magnitude, an estimate of the form \eqref{2115} with some $\mu$,
with $ p $ replaced by $ p'+q' $, and an arbitrary $\rho$ holds trivially. On the other
hand, if the largest one among the indices $ n_1,\dots,n_{p'+q'+1} $ 
is much larger than the second largest, 
then, by \eqref{para-b} and \eqref{aut2} we deduce that 
$ \max(n_1,\dots,n_{p'+q'+1})  =  \max(n_{p'+1},\dots,n_{p'+q'+1}) $ 
and $ n_0 \leq C \max(n_1,\dots,n_{p'+q'+1})$. Then 
by \eqref{5131} and taking $ \ell  \geq \rho + m $ 
we deduce an estimate of the form \eqref{2115}. 
   This settles the case of multilinear contributions to
  \eqref{5129}.

If, on the other hand, we consider the contribution coming from the non-homogeneous 
symbol $a_N$ term in the expression of $a$, and from a $M(U)$
satisfying \eqref{5127}, we use \eqref{2123} with $m, K'$ replaced by $m-\ell,0$ and some $\sigma$ independent of $\ell$. We obtain
\begin{multline*}\norm{\partial_t^k\opbw(a(U;\cdot))M(U)V}_{\Hds{\sigma_0-m+\ell-\frac{3}{2}k}} \\\leq
C\sum_{k'+k''=k}\Gcalsm{\sigma_0}{k',N}{U}\Gcals{\sigma_0}{k'',1}{M(U)V}\end{multline*}
which, combined with \eqref{5127}, implies an estimate of the form \eqref{2117} if $\ell$ is taken large enough relatively to
$s$ and $\rho$. We obtain estimates for the $\partial_z^j$ derivatives in the same way, noticing that for $z$ in
$[-\delta,0]$, $z'$ in the support of $M(U;z')$, $\partial_z^j a(U;z,z',\cdot)$ is a bounded function of $z$ for any $j$. One
treats \eqref{5130} in the same way.
\end{proof}

\section{Parametrix of Dirichlet-Neumann problem}\label{sec:52}

We consider a real valued even function 
$\eta$ in $\CKH{\sigma}{\R}$ for some $K$ in $\N$ and  some large 
enough $\sigma >\frac{3}{2} K $. We denote by $\eta'$ the space derivative of $\eta$ and by $\eta'\otimes \eta'$ the quadratic map
associated to the bilinear map $(\eta_1,\eta_2)\to \eta'_1 (x) \eta'_2 (x) $. 
This defines an homogeneous symbol 
of $\Gt{0}{2}$. In the same
way the map $\eta'\to (\xi\to \eta' (x) \xi)$ defines an homogeneous symbol of $\Gt{1}{1}$. 

We want to construct a parametrix $E$ of
the paradifferential elliptic operator, acting on functions of $(z,x)\in [-1,0]\times\Tu$,  given by
\begin{equation}
  \label{eq:521}
  (1+\opbw(\eta'\otimes\eta'))\partial_z^2 -2i\opbw(\eta'\xi)\partial_z -\opbw(\xi^2)
\end{equation}
together with boundary conditions
\begin{equation}
  \label{eq:522}
  E\vert_{z=0}= \mathrm{Id},\ \  \partial_z E\vert_{z=-1}=\mathrm{Id}
\end{equation}
up to smoothing operators. 

Recall that we introduced after Definition~\ref{513} the Poisson symbols
\begin{equation}
  \label{eq:523}
  \begin{split}
    \Ccal(z,\xi) &= \frac{\cosh((z+1)\xi)}{\cosh\xi},\quad \Scal(z,\xi) = \frac{\sinh(z\xi)}{\xi\cosh\xi}\\
K_0(z,z',\xi) &= (\cosh\xi)\bigl(\Ccal(z,\xi)\Scal(z',\xi)\1_{z-z'<0} + \Scal(z,\xi)\Ccal(z',\xi)\1_{z-z'>0}\bigr) 
  \end{split}
\end{equation}
where  $\Ccal$ is in $\Pt{0,+}{0}$, $ \Scal $ is in $\Pt{-1,-}{0}$ and  $ K_0  $ is in $ \Pti{-1}{0} $.
Notice that the function $ K_0(z,z',\xi) = K_0(z',z,\xi) $ is symmetric in $ (z, z' )$ and
$$
\pa_z K_0 (z, z', \xi ) = 
(\cosh\xi)\bigl( \pa_z \Ccal(z,\xi)\Scal(z',\xi)\1_{z-z'<0} +  
\pa_z \Scal(z,\xi)\Ccal(z',\xi)\1_{z-z'>0}\bigr)
$$
is a Poisson symbol in $\Pt{0,-}{0}$. 

The
 definition of the parametrix $E$ is made precise in the following 
 Proposition.
 \begin{proposition}
   \label{521} {\bf (Parametrix of Dirichlet-Neumann problem)}
Let $N\in \N^*$, $\rho\in \N$ be given. There are $r>0$ and Poisson symbols $e_{+,1} (\eta; \cdot )$ 
in $\sP{0,+}{K,0,1}{N}$, $e_{-,1} (\eta; \cdot ) $ in
$\sP{-1,-}{K,0,1}{N}$, $K_1 (\eta; \cdot )$ in $\sPi{-1}{K,0,1}{N}$, which are even as functions of $(x,\xi)$
(for $ \eta (x) $ even), such that, if we set
\begin{equation}
  \label{eq:524}
  \begin{split}
    e_+(\eta;z,x,\xi) &= \Ccal(z,\xi) + e_{+,1}(\eta;z,x,\xi)\\
 e_-(\eta;z,x,\xi) &= \Scal(z,\xi) + e_{-,1}(\eta;z,x,\xi)\\
 K(\eta;z,z',x,\xi) &= K_0(z,z',\xi) + K_1(\eta;z,z',x,\xi)
  \end{split}
\end{equation}
and if one defines, for functions $g_\pm(x)$  and $f(z,x)$,  
\begin{multline}
  \label{eq:525}
  E(t,z)[g_+,g_-,f] \\ \stackrel{\textrm{def}}{=}    \opbw(e_+)g_+ + \opbw(e_-)g_- + \int_{-1}^0
  \opbw \big( K(\eta;z,z', \cdot ) \big) f(z',\cdot)\,dz',
\end{multline}
then the action of \eqref{521} on \eqref{525} may be written as
\begin{equation}
  \label{eq:526}
  f+R_+(\eta;z)g_+ + R_-(\eta;z)g_-+ \int_{-1}^0 R_{\mathrm{int}}(\eta;z,z')f(z',\cdot)\,dz'
\end{equation}
where $R_+(\eta';z)$ (resp.\ $R_-(\eta';z)$, resp.\ $R_{\mathrm{int}}(\eta;z,z')$) is in $\sR{-\rho+2,+}{K,0,1}{N}$ (resp.\
$\sR{-\rho+1,-}{K,0,1}{N}$, resp.\ $\sRi{-\rho+1}{K,0,1}{N}$). 
Moreover the following boundary conditions hold
\begin{equation}
  \label{eq:527}
  E(t,0)[g_+,g_-,f] = g_+ , \quad (\partial_z E)(t,-1)[g_+,g_-,f] = g_- \, .
\end{equation}
Finally
\be\label{eq:paze+}
\partial_z e_{+,1} (\eta; z, x, \xi )\vert_{z=0} = i\frac{\eta'}{1+\eta'^2}\xi -  (\tanh \xi) \xi \frac{\eta'^2}{1+\eta'^2}
\ee
modulo $\sGa{0}{K,0,1}{N}$.
 \end{proposition}
In order to prove  Proposition \ref{521}, we start solving the boundary value problem 
corresponding to \eqref{521}-\eqref{522} at the level of principal symbols.
Notice that the Poisson symbols $ \Ccal, \Scal $ and $  K_0 $ in \eqref{523} solve 
the following boundary value problems 
$$
  \bigl( \partial_z^2 -\xi^2\bigr)  \Ccal(z,\xi) = 0  \, , \quad 
  \Ccal(z,\xi) \vert_{z=0} = 1 \,  , \quad 
\partial_z   \Ccal(z,\xi) \vert_{z=-1}  = 0 \, , 
 $$
$$
  \bigl( \partial_z^2 -\xi^2\bigr)  \Scal(z,\xi) = 0  \, , \quad 
  \Scal(z,\xi) \vert_{z=0} = 0 \, ,  \quad 
\partial_z   \Scal(z,\xi) \vert_{z=-1}  = 1 \, , 
$$
and
\begin{equation}\label{eq:bpforkzero}
\begin{split}
\bigl( \partial_z^2 -\xi^2\bigr)  K_0 (z, z', \xi) = \delta (z- z' ) \qquad \ \\
K_0 (z, z', \xi) \vert_{z=0} = 0 \, , \,  
\partial_z   K_0 (z, z', \xi) \vert_{z=-1}  = 0 \, . 
\end{split}\end{equation}
In the next lemma we solve the ODE  \eqref{529} for a general small $ \eta$. 
For simplicity of notation we do not write explicitly the $t$ dependence which is irrelevant.
\begin{lemma}
  \label{522}
There are symbols $e_{+,1}^0(\eta;z,\cdot)$ in $\sP{0,+}{K,0,1}{N}$,  
$e_{-,1}^0(\eta;z,\cdot)$ in $\sP{-1,-}{K,0,1}{N}$,
$K_1^0(\eta;z,z',\cdot)$ in $\sPi{-1}{K,0,1}{N}$, which are even functions of $(x,\xi)$ if $\eta$ is an even function in $x$,
satisfying the boundary conditions 
\begin{equation}
  \label{eq:5211}
  \begin{split}
    e_{+,1}^0\vert_{z=0} &= \partial_ze_{+,1}^0\vert_{z=-1} = e_{-,1}^0\vert_{z=0} = \partial_ze_{-,1}^0\vert_{z=-1} = 0\\
K_1^0\vert_{z=0} &= 0, \  \partial_zK_1^0\vert_{z=-1} = 0,
  \end{split}
\end{equation}
such that, setting 
\begin{equation}
  \label{eq:528}
  \begin{split}
    e_+^0(\eta;z,\cdot) &= \Ccal(z,\xi) + e_{+,1}^0(\eta;z,\cdot)\\
 e_-^0(\eta;z,\cdot) &= \Scal(z,\xi) + e_{-,1}^0(\eta;z,\cdot)\\
K^0(\eta;z,z',\cdot) &= K_0(z,z',\cdot) + K_1^0(\eta;z,z',\cdot) \, , 
  \end{split}
\end{equation}
the unique solution of the ODE
\begin{equation}
  \label{eq:529}
  \begin{split}
    Pu \stackrel{\textrm{def}}{=} \bigl((1+\eta'^2)\partial_z^2 -2i\eta'\xi\partial_z -\xi^2\bigr)u &= f\\
u\vert_{z=0} &= g_+\\
\partial_zu\vert_{z=-1} &= g_-
  \end{split}
\end{equation}
is given by
\begin{multline}
  \label{eq:5210}
u(z,x,\xi) = e_+^0(\eta;z,x,\xi) g_+ (x)  + e_-^0(\eta;z,x,\xi) g_- (x) \\
+\int_{-1}^0 K^0(\eta;z,z',x,\xi)f(z', x)\,dz'.
\end{multline}
Moreover
\begin{equation}
  \label{eq:5211a}
  P e^0_{\pm} = 0 \, , \quad   PK^0 = \delta(z-z') \, .
\end{equation}
Finally $ \partial_ze_+^0\vert_{z=0} $ has the following expansion
\begin{equation}
  \label{eq:5212}
  \partial_ze_+^0 (\eta; z, x, \xi ) \vert_{z=0} = \xi\tanh\xi + i\frac{\eta'}{1+\eta'^2}\xi - 
  (\tanh \xi ) \xi  \frac{\eta'^2}{1+\eta'^2}
\end{equation}
modulo a symbol  of $\sGa{-\infty}{K,0,1}{N}$.
\end{lemma}
\begin{proof}
The solution $  u $ of the linear equation \eqref{529} may be written as the sum \eqref{5210} where
$ e_+^0 $, $ e_-^0 $ and $ K^0 $ solve respectively 
\be\label{eq:e+0}
P e_+^0 = 0 \, , \quad e_+^0 \vert_{z=0} = 1 \,  , \quad 
\partial_z e _+^0 \vert_{z=-1}  = 0 
\ee
\be\label{eq:e-0}
P e_-^0 = 0 \, , \quad e_-^0 \vert_{z=0} = 0 \,  , \quad 
\partial_z e _-^0 \vert_{z=-1}  = 1 
\ee
$$
P K^0 =  \delta (z- z') \, , \quad K^0 \vert_{z=0} = 0 \,  , \quad 
\partial_z K^0 \vert_{z=-1}  = 0  \, .
$$
Dividing $ P $ by  $ 1+\eta'^2 $, we see that 
$ e_+^0 $, $ e_-^0 $ are the solutions of the homogeneous linear equation 
\begin{equation}
  \label{eq:5214}
  \bigl(\partial_z^2 -2ia\xi\partial_z - (1+b)\xi^2\bigr)e^0_{\pm} (z,x,\xi) = 0 
\end{equation}
with the corresponding boundary conditions in \eqref{e+0}, \eqref{e-0}, 
where 
  \begin{equation}
    \label{eq:5213}
    a = a(\eta') = \frac{\eta'}{1+\eta'^2} \, , \quad   b = b(\eta') = -\frac{\eta'^2}{1+\eta'^2} \, .
  \end{equation}
  The solutions $w_\pm$  of the homogeneous equation \eqref{5214} (where we consider $ a, b $ as real constants) with boundary values
\begin{equation}
  \label{eq:5216}
  w_+\vert_{z=0} = 1 \, , \  \partial_z w_+\vert_{z=-1} = 0 \, , \ w_-\vert_{z=0} = 0 \, , \ \partial_z w_-\vert_{z=-1} = 1 \, , 
\end{equation}
are, setting $c = c(a,b) = \sqrt{1+b-a^2}-1$, 
\begin{equation}
\label{eq:5215}
\begin{split}
  w_+(z,\xi,a,b) &=
  e^{iza\xi}\frac{\cosh((z+1)\xi(1+c))}{\cosh(\xi(1+c))}\frac{1-\frac{ia}{1+c}\tanh((z+1)\xi(1+c))}{1-\frac{ia}{1+c}\tanh(\xi(1+c))}\\
w_-(z,\xi,a,b) &= e^{i(z+1)a\xi} \frac{\cosh(z\xi(1+c))}{\cosh(\xi(1+c))}\frac{\tanh(z\xi(1+c))}{\xi(1+c)}\\
&\makebox[4cm]{}\times\biggl(1-\frac{ia}{1+c}\tanh(\xi(1+c))\biggr)^{-1}.
\end{split}\end{equation}
For $a, b$ in a neighborhood of zero, $ z\in [-1,0]$, $ w_\pm $ are analytic functions of $a, b$ that satisfy
the estimates
\begin{equation}
  \label{eq:5217}
  \begin{split}
    \abs{z^\ell\partial_z^j\partial_a^\alpha\partial_b^\beta\partial_\xi^\gamma w_+(z,\xi,a,b)}\leq
    C\absj{\xi}^{j-\ell-\gamma}\\
\abs{(1+z)^\ell\partial_z^j\partial_a^\alpha\partial_b^\beta\partial_\xi^\gamma w_-(z,\xi,a,b)}\leq
    C\absj{\xi}^{-1+j-\ell-\gamma}.
  \end{split}
\end{equation}
Actually, to check  the first estimate \eqref{5217}, it suffices to remark that $w_+$ is of order $-\infty$ in $\xi$, as well
as its derivatives, uniformly for $z$ in $[-1,-\frac{1}{2}]$. On the other hand, for $z$ in $[-\frac{1}{2},0]$ and for
instance $\xi>0$, we may write
\begin{equation}
  \label{eq:5218}
  w_+(z,\xi,a,b)=e^{iza\xi+z\xi(1+c)}(1+G(z,\xi,a,b))
\end{equation}
where $G$ as well as its derivatives is uniformly of order $-\infty$, while the exponential factor satisfies the wanted
estimates as $1+c\geq \frac{1}{2}$ for $a, b$ small enough.

If we expand $w_+(z,\xi,a,b)$ by Taylor formula relatively to the variables $(a,b)$ at $(0,0)$, we get  $\Ccal(z,\xi)$ as the constant term,
polynomial contributions in $(a,b)$, whose coefficients satisfy estimates of the form \eqref{5217}, and 
remainders, vanishing at
order $N$ at $(0,0)$, and verifying as well \eqref{5217}. Expressing in each polynomial 
term $a = a(\eta') $ and $ b = b(\eta' ) $ by
\eqref{5213}, and developing the functions $a(\eta'), b(\eta')$ as polynomials in $\eta'$ plus a remainder vanishing at order $N$ at
$\eta'=0$, we have that 
\be\label{eq:defe+}
e_+^0 = w_+ (z,\xi,a (\eta'),b (\eta' )) = \Ccal(z,\xi) + e_{+,1}^0 \, , 
\ee
as in  \eqref{528}, where $ e_{+,1}^0 $ is a Poisson symbol in $ \sP{0,+}{K,0,1}{N} $.  We argue similarly 
for 
\be\label{eq:defe-} 
e_-^0 = w_- (z,\xi,a (\eta'),b (\eta' )) = \Scal(z,\xi) + e_{-,1}^0 \, . 
\ee
Let us prove the expansion \eqref{5212}. Differentiating \eqref{5215} at $ z = 0 $ we get 
$$
(\pa_z w_+ )_{|z=0} = i a \xi  +   \xi (1+c) \tanh (  \xi (1+c)) +
\frac{i a \xi \cosh^{-2} ( \xi (1+c))}{1- \frac{i a}{1+c} \tanh (\xi (1+c))}
$$
The last term is in 
$\sGa{-\infty}{K,0,1}{N}$. 
Since  $ \xi \tanh (  \xi (1+c)) - \xi \tanh (  \xi ) 
$ is in $  \sGa{-\infty}{K,0,1}{N} $ 
and  since, by  \eqref{5213}, 
$$
c = \sqrt{1+b-a^2} - 1  = - \frac{\eta'^2}{1 + \eta'^2}  \, , 
$$
we get 
$$
(\pa_z w_+ )_{|z=0} = \xi \tanh (  \xi ) + i a  \xi   - \frac{\eta'^2}{1 + \eta'^2} \xi \tanh (  \xi )   
$$
modulo a symbol  in  $ \sGa{-\infty}{K,0,1}{N}  $,  
proving the expansion \eqref{5212}. 

Let us obtain next the expression of  $K^0 (\eta; \cdot) $ in \eqref{528}. 
Consider the wronskian of the solutions $(w_+,w_-)$ 
\[
W(z,\xi,a,b) = \begin{vmatrix}w_+ & w_-\\\partial_zw_+&\partial_z w_-\end{vmatrix}.
\]
Since the wronskian $ W(z) $ solves the differential equation 
$ \partial_z W (z) = 2 i a \xi W(z)$ we have that $ W(z) = e^{2 i a \xi (z+1)} W(-1)$.
By 
 \eqref{5216},  $ W(-1) = w_+(-1) $, 
and, inserting the value of $ w_+(-1) $ obtained by \eqref{5215},  we get 
\begin{equation}
  \label{eq:5219}
  W(z,\xi,a,b) = e^{i(2z+1)a\xi}(\cosh(\xi(1+c)))^{-1}\Bigl(1-\frac{ia}{1+c}\tanh(\xi(1+c))\Bigr)^{-1}.
\end{equation}
Let us define
\begin{multline}
  \label{eq:5220}
\tilde{K}(z,z',\xi,a,b) \\= \Bigl(w_+(z,\xi,a,b)w_-(z',\xi,a,b)\1_{z-z'<0}
+w_-(z,\xi,a,b)w_+(z',\xi,a,b)\1_{z-z'>0}\Bigr)\\\times W(z',\xi,a,b)^{-1}.
\end{multline}
Since $ w_{\pm} $ are the solutions of the homogeneous linear equation \eqref{5214}, 
the action of the differential operator in \eqref{5214} on $\tilde{K}$ 
gives $\delta(z-z')$, and the boundary conditions \eqref{5216} imply $\tilde{K}\vert_{z=0} =
0, \partial_z \tilde{K}\vert_{z=-1} = 0$. Moreover, proceeding as in \eqref{5218}, one checks that the coefficients
$\tilde{K}_\pm$ of $\1_{\pm(z-z')>0}$ in \eqref{5220} satisfy estimates
\[\abs{(z-z')^\ell\partial_z^j\partial_{z'}^{j'}\partial_a^\alpha\partial_b^\beta\partial_\xi^\gamma
  \tilde{K}_\pm(z,z',\xi,a,b)}\leq C\absj{\xi}^{-1-\ell+j+j'-\gamma}.\]
Decomposing $\tilde{K}_\pm$ in Taylor series in the variables $(a,b)$ at $(0,0)$, and expressing 
$a, b$ from $\eta'$ by \eqref{5213}, we conclude that
\be\label{eq:defK0}
K^0(\eta;z,z',x,\xi) \stackrel{\textrm{def}}{=} \tilde{K}(z,z',\xi,a,b)(1+\eta'{}^2)^{-1}
\ee
 is a symbol of $\sPi{-1}{K,0,0}{N}$, which has the expansion \eqref{528}, and 
 the action of the operator
 $(1+\eta'{}^2)\partial_z^2 -2i\eta'\xi\partial_z -\xi^2$ on $ K^0 $ gives $\delta(z-z')$, proving the last identity in \eqref{5211a}. 

Finally the symbols $e^0_\pm$, $K^0$ are even in $ (x, \xi )$, 
when $\eta$ is even, as it follows by formulas \eqref{defe+}, \eqref{defe-}, \eqref{5215}, \eqref{5220}, noticing that, if
 $\eta$ is even, then $b(\eta'), c(\eta'), a(\eta')\xi$ are even functions of $(x,\xi)$.
  The proof of the lemma is concluded.
\end{proof}
\begin{proof1}{Proof of Proposition~\ref{521}}
We look for  $e_\pm (\eta; \cdot) $, $K (\eta; \cdot) $ in \eqref{524} as a sum of Poisson symbols in decreasing order (forgetting the time dependence in the notation) of the form
\begin{equation}
  \label{eq:5221}
  \begin{split}
    e_+(\eta;z,x,\xi) &= \Ccal(z,\xi) + \sum_{j=0}^{\rho-1} e_{+,1}^j(\eta;z,x,\xi)\\
e_-(\eta;z,x,\xi) &= \Scal(z,\xi) + \sum_{j=0}^{\rho-1} e_{-,1}^j(\eta;z,x,\xi)\\
K(\eta;z,z',x,\xi) &= K_0(z,z',\xi) + \sum_{j=0}^{\rho-1} K_1^j(\eta;z,z',x,\xi)
  \end{split}
\end{equation}
where 
$e_{+,1}^0 \in \sP{0,+}{K,0,1}{N}$, $e_{-,1}^0 \in \sP{-1,-}{K,0,1}{N} $, $K_1^0 \in \sPi{-1}{K,0,1}{N}$ are defined in lemma~\ref{522}, and, for $ j =1, \ldots, \rho - 1$,  

$ \bullet $ $e_{+,1}^j (\eta; \cdot) $ are in $\sP{-j,+}{K,0,1}{N}$, 

$ \bullet $ $e_{-,1}^j (\eta; \cdot) $ are  in $\sP{-1-j,-}{K,0,1}{N}$, 

$ \bullet $ $K_1^j (\eta; \cdot)$  are in  $\sPi{-1-j}{K,0,1}{N}$. 

We also require that these symbols satisfy  the boundary conditions
\begin{equation}
  \label{eq:5222}
  e_{\pm,1}^j\vert_{z=0} = 0 \, , \ \partial_ze_{\pm,1}^j\vert_{z=-1} = 0 \, , \ K_1^j\vert_{z=0} = 0 \, ,  \ \partial_zK_1^j\vert_{z=-1} = 0 \, . 
\end{equation}
As a consequence, if $E$ is defined by \eqref{525}, it follows by  \eqref{523}, \eqref{5211} and \eqref{5222} 
that the boundary conditions
\eqref{527} are satisfied. 

In order to prove the proposition we have thus to find symbols 
$e_{\pm,1}^j$, $K_1^j$, $ 1 \leq j \leq \rho - 1 $, as above so that the action of the operator
\eqref{521} on $E$ gives \eqref{526}. 

Using that $ \eta'\otimes\eta'  $ is  a symbol in  $\Gt{0}{2}$,  $  \eta'\xi $ is in $\Gt{1}{1}$, and
that the Poisson symbol $ \partial_z^2 e_+ $ is in $ \sP{2,+}{K,0,0}{N} $ and 
 $ \partial_z^2 e_- $ is in $ \sP{1,+}{K,0,0}{N} $, 
by the 
composition result of Proposition~\ref{517} we compute 
\begin{multline} \label{eq:5223}
\bigl((1+\opbw(\eta'\otimes\eta'))\partial_z^2-2i\opbw(\eta'\xi)\partial_z -\opbw(\xi^2)\bigr)\circ \opbw(e_\pm)\\
= \opbw(Pe_\pm) + \opbw\bigl(((\eta'\otimes\eta')\#\partial_z^2e_\pm)_{\rho,N} - \eta'{}^2\partial_z^2e_\pm\bigr)\\
-2i \opbw\bigl(((\eta'\xi)\#\partial_z e_\pm)_{\rho,N} - (\eta'\xi)\partial_ze_\pm\bigr)\\
-\opbw\bigl((\xi^2\# e_\pm)_{\rho,N} - \xi^2 e_\pm\bigr) 
+R_\pm(\eta)
\end{multline}
where  $ P $ is defined in \eqref{529} and where $R_+(\eta)$ (resp.\ $R_-(\eta)$) 
belongs to $\sR{-\rho+2,+}{K,0,1}{N}$ (resp.\
$\sR{-\rho+1,-}{K,0,1}{N}$). 
Recalling the definition  \eqref{528} of the symbols $ e_\pm^0  $, we have   
$ e_\pm = e_{\pm}^0 + \sum_{j=1}^{\rho-1} e_{\pm,1}^j $ 
and, since $Pe_{\pm}^0 = 0 $ by  \eqref{5211a}, we deduce that 
$Pe_\pm$ in \eqref{5223} is 
$$
Pe_\pm = \sum_{j=1}^{\rho-1}Pe_{\pm,1}^j \, . 
$$ 
The other symbols in the right hand side of \eqref{5223} may be written as
$$
\sum_{j=1}^{\rho-1}\Lcal_{j,\pm}(e_\pm^0,\dots,e_{\pm,1}^{j-1}) + \Lcal_{\rho,\pm}(e_\pm^0,\dots,e_{\pm,1}^{\rho-1})
$$ 
where $\Lcal_{j,+}$ (resp.\ $\Lcal_{j,-}$) belongs to $\sP{2-j,+}{K,0,1}{N}$
(resp.\ to $\sP{1-j,-}{K,0,1}{N}$), and depends only on $e_{\pm,1}^\ell$, with $\ell<j$, because of the
Definition~\ref{221} of $(\cdot\#\cdot)_\rho$ and of the fact that in the right hand side of \eqref{5223}, the first term of
the asymptotic expansion is removed from each expression. 
By the last remark following  Definition~\ref{514}, the paradifferential operators 
$ \opbw \big(\Lcal_{\rho,\pm}(e_\pm^0,\dots,e_{\pm,1}^{\rho-1})\big) $ may be incorporated to the smoothing remainder
$ R_\pm(\eta)  $, and  thus we may write the right hand side of \eqref{5223} as
\[
\sum_{j=1}^{\rho-1}\opbw\bigl(Pe_{\pm,1}^j - \Lcal_{j,\pm}(e^0_\pm,e_{\pm,1}^1,\dots,e_{\pm,1}^{j-1})\bigr) + R_\pm(\eta) \, .
\]
In order to prove the Proposition we need to find recursively Poisson symbols $ e_{\pm,1}^j $, $ j = 1, \dots, \rho - 1 $ 
(in the classes described above), solving
\begin{equation}\label{eq:5224}
\begin{split}
& Pe_{\pm,1}^j = \Lcal_{j,\pm}(e^0_\pm,e_{\pm,1}^1,\dots,e_{\pm,1}^{j-1})\\
& e_{\pm,1}^j\vert_{z=0} = 0\\
& \partial_ze_{\pm,1}^j\vert_{z=-1} = 0 \, .
\end{split}\end{equation}
By \eqref{5210}, the solution of \eqref{5224} is given by
\[ 
e_{\pm,1}^j = 
\int_{-1}^0K^0(\eta;z,z',x,\xi)\Lcal_{j,\pm}(e^0_\pm,e_{\pm,1}^1,\dots,e_{\pm,1}^{j-1})(\eta;z',x,\xi)\,dz'.
\]
By  Proposition~\ref{517}-(iv),  since 
$ K^0 $  is a Poisson symbol in $\sPi{-1}{K,0,0}{N}$ 
and  $\Lcal_{j,+}$ (resp.\ $\Lcal_{j,-}$) is in $\sP{2-j,+}{K,0,1}{N}$
(resp.\ to $\sP{1-j,-}{K,0,1}{N}$),
we  deduce that  $e_{+,1}^j $ is in $ \sP{-j,+}{K,0,1}{N}$ 
and $e_{-,1}^j $ is in $  \sP{-1-j,-}{K,0,1}{N}$.  

The formula \eqref{paze+} follows by \eqref{5212} and the fact that  
$ \sum_{j=1}^{\rho-1}e_{+,1}^j $ is a symbol  in $ \sP{-1,+}{K,0,1}{N}$.

Finally let us  construct iteratively in a similar way the 
Poisson symbols $ K^j_1  \in \sPi{-1-j}{K,0,1}{N} $, $ j = 1, \ldots, \rho -1 $,  
 in the last line of \eqref{5221}. 

First of all, recall that the Poisson symbol  $ K^0  $ 
in \eqref{528}  satisfies $ P K^0 = \delta(z-z')$ (see \eqref{5211a}) 
and that $ K^0 $ is in $\sPi{-1}{K,0,0}{N}$.
Moreover by \eqref{defK0} and \eqref{5220} we deduce that $ \partial_z K^0 (\eta; \cdot) $ is in
$\sPi{0}{K,0,0}{N}$ as the coefficient of $\delta(z-z')$ in the $\partial_z$ derivative vanishes. Finally,
by the equation  $ P K^0 = \delta(z-z')$ and recalling the expression of $ P $ in \eqref{529}, we deduce that 
$$
\partial_z^2K_0 - (1+\eta'{}^2)^{-1}\delta(z-z') \quad {\rm  is \ in } \ \sPi{1}{K,0,0}{N} \, . 
$$
Therefore applying the composition result of Proposition \ref{517}-(ii) we get
\be\label{eq:first-step}
\opbw ( P) \circ \opbw ( K^0 ) = \delta (z- z') + {\cal L}_1 + R 
\ee
where (recall the definition of $ P $  in \eqref{529}) 
\begin{equation}
\label{eq:5227}
\Lcal_1 \stackrel{\textrm{def}}{=} \bigl(((1+\eta'{}^2)\partial_z^2 -2i\eta'\xi\partial_z -\xi^2)\# K^0\bigr)_{\rho,N} - \delta(z-z')
\end{equation}
belongs to $\sPi{0}{K,0,1}{N}$ and 
the smoothing operator $ R $ is in $ \sRi{-\rho+1}{K,0,1}{N}$. 
Then we argue iteratively. Suppose that we have yet  defined   
Poisson symbols $\Lcal_1, \ldots, \Lcal_{j} $ with $ \Lcal_{j} \in \sPi{1-j}{K,0,1}{N} $,  and
$ K_1^1,\dots, K_1^{j-1}$ with $K_1^{j-1}(\eta; \cdot) \in \sPi{-j}{K,0,1}{N}$. 
Then  define, for $ j \geq 1 $,  the $ K_1^j $ as the solution of 
\begin{equation}
  \label{eq:5225}
  \begin{split}
    PK_1^j(\eta;z,z',x,\xi) &= -\Lcal_j(\eta;z,z',x,\xi)\\
K_1^j\vert_{z=0} &= 0\\
\partial_z K_1^j\vert_{z=-1} &= 0 
 \end{split}
\end{equation}
(where $P$ is the operator \eqref{529} acting relatively to the $z$ variable, $z'$ being a parameter)
and 
\begin{equation}
  \label{eq:5228}
  \Lcal_{j+1} \stackrel{\mathrm{def}}{=}  \bigl(((1+\eta'{}^2)\partial_z^2 -2i\eta'\xi\partial_z -\xi^2)\# K_1^{j}\bigr)_{\rho,N}  +\Lcal_{j} \, . 
\end{equation}
By \eqref{5210}, the solution $ K_1^j $ of \eqref{5225} is given by
\begin{equation}
  \label{eq:5226}
  K_1^j(\eta;z,z',x,\xi) = -\int_{-1}^0 K^0(\eta;z,z'',x,\xi) \Lcal_j(\eta;z'',z',x,\xi)\,dz'' 
\end{equation}
and, since 
$ K^0 $ is in $\sPi{-1}{K,0,0}{N}$ and $ \Lcal_j $ in $  \sPi{1-j}{K,0,1}{N} $, 
 Proposition~\ref{517}-(iii) implies that $ K_1^j $ is a Poisson symbol 
in $\sPi{-1-j}{K,0,1}{N}$. In the same way, since $\partial_z
K^0(\eta;\cdot)$ is in $\sPi{0}{K,0,0}{N}$, we deduce that 
$$
\partial_zK_1^j (\eta; \cdot) = - \int_{-1}^0 \partial_z K^0(\eta;z,z'',x,\xi) \Lcal_j(\eta;z'',z',x,\xi)\,dz''
$$ 
is a Poisson symbol in $\sPi{-j}{K,0,1}{N}$, and, finally, using the first
equation \eqref{5225} and recalling the definition of $P $  in \eqref{529}, 
we get that $\partial_z^2K_1^j (\eta; \cdot) $ is in $\sPi{1-j}{K,0,1}{N}$. 
We conclude that the symbol  $ \Lcal_{j+1} $ defined in \eqref{5228} is in $\sPi{-j}{K,0,1}{N} $, because,
by the first equation in \eqref{5225},  the first term in the expansion of  $ (P \# K_1^{j})_{\rho,N} $ cancels out. 

In conclusion, applying the composition  result of Proposition \ref{517}, we deduce that 
the Poisson symbol $ K (\eta; \cdot ) = K^0 + \sum_{j=1}^{\rho-1} K_1^j (\eta; \cdot) $ 
satisfies, by \eqref{first-step}, \eqref{5225}, \eqref{5228}, 
\begin{multline*}
  \opbw\Bigr((1+\eta'{}^2)\partial_z^2 -2i\eta'\xi\partial_z -\xi^2\Bigl) \circ \opbw\bigr(K(\eta;z,z',\cdot)\bigl)\\
= \delta(z-z') + \opbw(\Lcal_\rho) + R(\eta;z,z')
\end{multline*}
for some smoothing remainder  $R$ in $\sRi{-\rho+1}{K,0,1}{N}$. 
As $\opbw(\Lcal_\rho)$ may be incorporated to such a term (see the last remark after Definition~\ref{514}), we 
obtain that the action
of \eqref{521} on $\int_{-1}^0\opbw(K(\eta;z,z',\cdot))f(z')\,dz'$ gives the first and last terms in \eqref{526}. This
concludes the proof.
\end{proof1}

\section{Solving the Dirichlet-Neumann problem}\label{sec:53}

The goal of this section is to deduce from the parametrix constructed in section~\ref{sec:52} 
an expression for the solution of the Dirichlet-Neumann boundary value problem \eqref{531} below.
\begin{proposition}
  \label{531} {\bf (Dirichlet-Neumann problem)}
Let $\rho, N, K$ be given integers. There is $\sigma>0$ such that, for any $s\in \R, s'\in \N$ with $s-s'\geq\sigma$, there
is $r>0$ and for any function $\eta$ in $\CKH{\sigma}{\R}$ in the ball $\Br{K}{}$ defined in \eqref{216}, for any
function $f$ in the space $\bigcap_{0}^K C^k(I,E^{s-2-\frac{3}{2}k}_{s'})$ (with $E_{s'}^s$ defined in \eqref{5111}), any 
$$
g_+ \in C^K_*(I,\Hds{s}(\Tu,\R)) \, , \quad g_- \in C^K_*(I,\Hds{s-1}(\Tu,\R)) \, , 
$$
the boundary value problem 
\begin{equation}
  \label{eq:531}
\begin{split}
  \bigl[(1+\opbw(\eta'\otimes\eta'))\partial_z^2 - 2i\opbw(\eta'\xi)\partial_z -\opbw(\xi^2)\bigr]u &= f\\
u\vert_{z=0} &= g_+\\
\partial_z u\vert_{z=-1} &= g_-
\end{split}\end{equation}
has a unique solution $u$ in $\bigcap_{0}^K C^k(I,E^{s-\frac{3}{2}k}_{s'})$. Moreover, there are 

$\bullet$ Poisson symbols $e_+ (\eta; \cdot )$ in $\sP{0,+}{K,0,0}{N}$, $e_- (\eta; \cdot )$ in $\sP{-1,-}{K,0,0}{N}$, $K(\eta; \cdot )$ in  $\sPi{-1}{K,0,0}{N}$ (which are those defined in \eqref{524}), 

$\bullet$ smoothing operators $R_+$ in  $\sR{-\rho,+}{K,0,1}{N}$, $R_-$ in $\sR{-\rho-1,-}{K,0,1}{N}$, $R_{\textrm{int}}$ in $\sRi{-\rho-1}{K,0,1}{N} $,  

such that for any
$f, g_+, g_-$ as above, the solution $u$ to \eqref{531} may be written as
\begin{multline}
  \label{eq:532}
u = \opbw(e_+(\eta;z,\cdot))g_+ + \opbw(e_-(\eta;z,\cdot))g_-\\
+ \int_{-1}^0\opbw(K(\eta;z,z',\cdot))f(z',\cdot)\,dz'\\
+ R_+(\eta;z)g_+ + R_-(\eta;z)g_- + \int_{-1}^0 R_{\textrm{int}}(\eta;z,z')f(z',\cdot)\,dz'.
\end{multline}
Moreover
\begin{equation}
  \label{eq:534}
  \partial_z e_+ (\eta; z,x, \xi )\vert_{z=0} = (\tanh\xi)\xi + i\frac{\eta'}{1+\eta'{}^2}\xi -  (\tanh\xi)\xi  \frac{\eta'{}^2}{1+\eta'{}^2}
\end{equation}
modulo $\sGa{0}{K,0,1}{N}$.
\end{proposition}
\begin{proof}
  In the proof, we shall ignore the time dependence i.e.\ argue like if $K=0$.
We look for the solution $u$ of \eqref{531} as 
$$
u = E(t,z)[g_+,g_-,f] + W
$$ 
where the first term is given by \eqref{525} (and the first two lines in \eqref{532}), 
and the remainder $W$ satisfies, according to \eqref{526}-\eqref{527},  
\begin{equation}
  \label{eq:535}
  \begin{split}
    \bigl[(1+&\opbw(\eta'\otimes\eta'))\partial_z^2 -2i\opbw(\eta'\xi)\partial_z -\opbw(\xi^2)\bigr]W\\
&= -R_+^2(\eta;z)g_+ - R_-^1(\eta;z)g_- -\int_{-1}^0 R'_\mathrm{int}(\eta;z,z')f(z')\,dz'\\
&W\vert_{z=0} = 0\\
&\partial_z W\vert_{z=-1} = 0
  \end{split}
\end{equation}
with $R^2_+$ (resp.\ $R^1_-$, resp.\ $R'_\mathrm{int}$) in $\sR{-\rho+2,+}{K,0,1}{N}$ (resp.\ $\sR{-\rho+1,-}{K,0,1}{N}$, resp.\
$\sRi{-\rho+1}{K,0,1}{N}$). We rewrite the equation \eqref{535} as
\begin{equation}
  \label{eq:536}
  \begin{split}
    (\partial_x^2 +\partial_z^2)W &= LW+G\\
W\vert_{z=0} &= 0\\
\partial_z W\vert_{z=-1} &= 0
  \end{split}
\end{equation}
with
\begin{equation}
  \label{eq:537}
\begin{split}
  LW &= -\opbw(\eta'\otimes\eta')\partial_z^2W + 2i\opbw(\eta'\xi)\partial_zW\\
G &= -R^2_+(\eta;z)g_+ - R^1_-(\eta;z)g_- - \int_{-1}^0 R'_\mathrm{int}(\eta;z,z')f(z')\,dz'.
\end{split}\end{equation}
By lemma~\ref{522}, applied with $\eta=0$, $g_+ = g_- = 0$, we may rewrite \eqref{536} as 
\begin{equation} \label{eq:538}
(\mathrm{Id}-M)W(z,x) = \int_{-1}^0K_0(z,z',D_x)G(z',\cdot)\,dz'
\end{equation}
where $ K_0 (z, z', \xi ) \in  \Pti{-1}{0}  $  is the  Poisson symbol  with constant coefficients (in $  x $) 
 defined in \eqref{523}, and $M = M(\eta)$ is the operator
\begin{equation}
  \label{eq:539}
  MW(z,x) \stackrel{\mathrm{def}}{=}   \int_{-1}^0 K_0(z,z',D_x)LW(z')\,dz'.
\end{equation}
In order to  prove that 
$ W $ has the form described in the last line in
\eqref{532},  we have  to invert the operator $ \mathrm{Id}-M $. 
This will be a consequence of the next lemma. 
 We denote  by $H$ the right hand side of \eqref{538}. 

\begin{lemma}
  \label{532}
(i) The operator $M$ is bounded from the subspace 
$$
\big\{ w \in E^s_{s'} \, : \, w \vert_{z=0}=0 \, ,  \partial_z w\vert_{z=-1} = 0 \big\}
$$ 
to itself, and its operator norm is $O_{s,s'}(\norm{\eta}_{\Hds{\sigma}})$ if $ \sigma>\frac{3}{2} $.

(ii) The right hand side $ H $ of \eqref{538} has the same structure as $ G $ 
in \eqref{537} except that the smoothing operators
$R_+^2, R_-^1, R'_{\mathrm{int}}$ are respectively in $\sR{-\rho,+}{K,0,1}{N}$, $\sR{-\rho-1,-}{K,0,1}{N}$,
$\sRi{-\rho-1}{K,0,1}{N}$. Moreover $ H \vert_{z=0}=0 $ and $  \partial_z H \vert_{z=-1} = 0 $. 

(iii) The operator $ M $ sends an expression of the form $ G $  in \eqref{537}, 
satisfying $ G \vert_{z=0} = 0 $ and $  \partial_z G \vert_{z=-1} = 0 $, 
to another expression of the same type. Moreover
$\sum_{k = 1}^{+\infty}M^k$ sends an expression of the form $G$ satisfying 
$ G \vert_{z=0}=0 $ and $  \partial_z G \vert_{z=-1} = 0 $, 
to a similar expression, with $R^2_+$ in 
$ \sR{-\rho,+}{K,0,1}{N}$,
$R^1_-$ in $\sR{-\rho-1,-}{K,0,1}{N}$,  $R'_{\mathrm{int}}$ in $\sRi{-\rho-1}{K,0,1}{N}$.
\end{lemma}
\begin{proof}
  (i) 
  By \eqref{539} and the definition \eqref{537} of $L$, 
 we may write
\begin{equation}
  \label{eq:5311}
  Mw(z,x) = \int_{-1}^0 K_0(z,z',D_x)[\partial_{z'}^2 w_1 +\partial_{z'} w_2](z',\cdot)\,dz'
\end{equation}
where 
\be\label{eq:defw1w2}
w_1 = -\opbw(\eta'\otimes\eta')w \, , \quad  w_2 = 2i\opbw(\eta'\xi)w \, . 
\ee
By Proposition~\ref{215}, 
if $w$ is a function of $E^s_{s'}$ 
then $ w_1 $ is in $E_{s'}^s$ and $ w_2 $ in $  E^{s-1}_{s'}$, with 
$  \norm{w_1}_{E_{s'}^s} + \norm{w_2}_{E_{s'}^{s-1}} \leq C\norm{\eta}_{\Hds{\sigma}}\norm{w}_{E^s_{s'}}$ (if $\sigma>\frac{3}{2}$). 
Performing integrations by parts in \eqref{5311},  
using the  boundary conditions $w\vert_{z=0} = \partial_zw\vert_{z=-1} =
0$, and that $ K_0 $ defined in \eqref{523} 
 satisfies 
$ K_0 (z, z', D_x)\vert_{z'=-1} = \Scal(z,D_x) $  and $ \partial_{z'} K_0 (z, z', D_x)\vert_{z'=0} = 0 $,  
we rewrite \eqref{5311} as
\begin{multline}
  \label{eq:5312}
w_1(z) - \Scal(z,D_x)w_2(-1) + \int_{-1}^0K_0(z,z',D_x)D_x^2 w_1(z')\,dz'\\
- \int_{-1}^0(\partial_{z'}K_0)(z,z',D_x) w_2(z')\,dz'.
\end{multline}
We apply lemma~\ref{516} to $K_0\in \Pti{-1}{0}$ and to $\partial_{z'}K_0\in \Pti{0}{0}$, 
and lemma \ref{Poisson:boundary-interior} to $\Scal \in \Pt{-1,-}{0}$. 
It follows from the previous bound 
on the functions $w_1, w_2$ defined 
in \eqref{defw1w2}, that \eqref{5312} belongs to $E^s_{s'}$, with a norm controlled from above by
$C\norm{\eta}_{\Hds{\sigma}}\norm{w}_{E^s_{s'}}$. Finally notice that  the function $ M w $ defined in \eqref{539} satisfies 
$ (Mw) (z, x) \vert_{z=0} = 0 $, 
$ \partial_z  (Mw)(z, x) \vert_{z=-1} = 0 $ 
because  $ K_0 (z, z', \xi) \vert_{z=0} = 0 $, 
$ \partial_z   K_0 (z, z', \xi) \vert_{z=-1}  = 0 $ . 
This proves (i) of the lemma.

(ii) follows from Proposition~\ref{518}-(ii)-(iii) as $K_0$ is in $\Pti{-1}{0}$.

(iii) Since $ G \vert_{z=0}=0 $ and $  \partial_z G \vert_{z=-1} = 0 $
the action of $M$ on $G$ is given by \eqref{5312}, where we replace $w_1$ (resp.\ $w_2$) by
$-\opbw(\eta'\otimes\eta')G$ (resp.\ $2i\opbw(\eta'\xi)G$). 
 If $ G $ has the form \eqref{537},  Proposition~\ref{518}-(i) implies that the function 
$ w_1 $ (resp.\ $w_2$) is of the
same form \eqref{537} as well, 
with $\rho$ replaced by $\rho-1$ in the case $w_2$ (and actually with a cubic degree of homogeneity). 
Using (ii) and (iii) of Proposition~\ref{518}, we
conclude that $ MG $ has an expression of the same form as $G$. 
Then we decompose 
$$
\sum_{k = 1}^{+\infty}M(\eta)^k G = \sum_{k = 1}^{N-2}M(\eta)^k G +  \sum_{k = N - 1}^\infty M(\eta)^k G 
$$
where, by what proved above, the first finite sum has  the form of $ G $ in \eqref{537}. 
To estimate  the series, use that 
the operator norm of $M(\eta)$ acting on $E^s_{s'} $ is
$O_s(\norm{\eta}_{\Hds{\sigma}})$ (for some $\sigma$ independent of $s$), 
as we proved in  item (i) of the present lemma. 
Consequently $\sum_{k = N-1}^{+\infty}M(\eta)^k$ sends an element of $\Rr{-\rho,\pm}{K,0,1}$, $\Rri{-\rho}{K,0,1}$ to an
element of $\Rr{-\rho,\pm}{K,0,N}$, $\Rri{-\rho}{K,0,N}$, using estimates \eqref{5115b}, \eqref{5115c}, the above continuity
property of $M(\eta)$ and the fact that we may choose the numbers $r(s)$ in (ii) of Definition~\ref{514} small enough.
\end{proof}

For $ \norm{\eta}_{\Hds{\sigma}} $ small 
enough, lemma \ref{532}-(i) implies that 
the operator $ \mathrm{Id}-M $ is invertible, 
and, denoting by $H$ the right hand side of \eqref{538}, 
we may write the solution $ W$ of \eqref{538} as the Neumann series
\begin{equation}
  \label{eq:5310}
  W(z,x) = H + \sum_{k=1}^{+\infty}M^kH \, .
\end{equation}

{\vspace{2ex}\noindent{\sl End of the proof of Proposition~\ref{531}:}}
By lemma~\ref{532}-(ii)-(iii) 
the 
function $ W(z,x) $ in \eqref{5310} may be written as 
$$
R_+(\eta;z)g_++R_-(\eta;z)g_- + \int_{-1}^0R_{\mathrm{int}}(\eta;z,z')f(z')\,dz'
$$ 
in \eqref{532}
for suitable smoothing operators $R_+$ in  $\sR{-\rho,+}{K,0,1}{N}$, $R_-$ in $\sR{-\rho-1,-}{K,0,1}{N}$, $R_{\textrm{int}}$ in $\sRi{-\rho-1}{K,0,1}{N} $.

Finally the expansion \eqref{534} for $ e_+ = \Ccal + e_{+,1} $ defined in \eqref{524}, 
follows by 
\eqref{523} and \eqref{paze+}. 
This concludes the proof of the proposition.
\end{proof}

We shall need in the next chapter a variant of Proposition~\ref{531}, when one solves, instead of \eqref{531}, a differential
boundary value problem. Let us introduce first a class of operators playing the same role, for boundary value problems, as
the operators of Definition~\ref{216}. In the sequel, we shall consider functions $\eta$ with values in $\R$ as functions
with values in $\C^2$, identifying $\eta$ to the vector $\vect{\eta}{\eta}$ of $\C^2$.
\begin{definition}
  \label{533}
Let $m$ be in $\R_+$, $K, p$ in $\N$ with $p\leq N$, $r>0$, $j\in \N$.

(i) We denote by \index{Na@$\Nt{m}{p}$ (Space of homogeneous maps)} $\Nt{m}{p}$ the space of $(p+1)$-linear 
maps of the form
$$
(\eta_1,\dots,\eta_{p},\Phi)\to M(\eta_1,\dots,\eta_{p})\Phi \, , 
$$
symmetric in $(\eta_1,\dots,\eta_p)$, defined for $\eta_1,\dots,\eta_p$ in
$\Hds{\infty}(\Tu,\C^2)$, with values in the space of linear maps from $E^\infty_\infty$ to itself, such that, for some
$\mu\geq 0$, for any $(\eta_1,\dots,\eta_p)$ in $\Hds{\infty}(\Tu,\C^2)^p$, any $\Phi\in
C_*^K(I,E^\infty_\infty)$, any $n_0, n_{p+1}$ in $\N^*$, any $n' = (n_1,\dots,n_p)$ in $(\N^*)^p$, any $j$ in $\N$,
\begin{multline}
  \label{eq:5313}
\norm{\Pin{0}M(\Pin{1}\eta_1,\dots,\Pin{p}\eta_p)\Pin{p+1}\Phi}_{E^0_j}\\
\leq C(n_0+\cdots+n_{p+1})^m\abs{n'}^\mu\Gcalsm{0}{0,p}{\eta_1,\dots\eta_p} \Gcalstm{0,j}{0,1}{\Phi}
\end{multline}
(where we used notations \eqref{213} and \eqref{5115a}). Moreover, we  assume that
condition \eqref{2126} holds.

(When $p=0$, the
above conditions just mean that $M$ is a linear map from $ E^\infty_\infty $ to itself, satisfying 
estimate \eqref{5313} and \eqref{2126}).

(ii) We denote by \index{Nb@$\Nr{m}{K,j,p}$ (Space of non-homogeneous maps)} $\Nr{m}{K,j,p}$ the space of maps $(\eta,\Phi)\to M(\eta)\Phi$, linear in $\Phi$ such that there are
$\sigma_0, \mu$ in $\R_+$ so that, for any $\sigma>\sigma_0$, there is $r(\sigma)\in ]0,r[$, and for any $\eta\in
\Brs{K}{0}{\sigma}\cap \CKH{\sigma+\mu}{\C^2}$, any $\Phi$ in $C_*^K(I,E^\sigma_j)$, one has for any $0 \leq k\leq K$, any $0\leq
j'\leq j$, bounds
\begin{multline}
  \label{eq:5314}
\norm{\partial_t^k(M(\eta)\Phi)}_{E^{\sigma-m-\frac{3}{2}k}_{j'}} \leq C_\sigma\sum_{k'+k''=k}\bigl(\Gcals{\sigma_0}{k',p-1}{\eta}
\Gcalst{\sigma_0,j'}{k'',1}{\Phi}  \Gcals{\sigma+\mu}{k',1}{\eta}\\
+ \Gcals{\sigma_0}{k',p}{\eta}\Gcalst{\sigma,j'}{k'',1}{\Phi}\bigr) \, .
\end{multline}

(iii) We denote by \index{Nc@$\sN{m}{K,j,p}{N}$ (Space of maps)} $\sN{m}{K,j,p}{N}$ the space of sums
\begin{equation}
  \label{eq:5315}
M(\eta) =  \sum_{q=p}^{N-1}M_q(\eta,\dots,\eta) + M_N(\eta)
\end{equation}
where $M_q$ is in $\Nt{m}{q}$ for $q=p,\dots,N-1$ and $M_N$ is in $\Nr{m}{K,j,N}$. 

We denote by $\sN{}{K,j,p}{N}$ the union over $m\geq 0$ of the preceding spaces.
\end{definition}
\textbf{Remarks}: $\bullet$ 
If $M$ is in $\Nt{m}{p}$ then $M(\eta,\dots,\eta)$ is in $\Nr{m}{K,j,p}$ for any $r$, any $j$.
Indeed, the fact that condition \eqref{2126} holds for elements of $\Nt{m}{p}$ implies that,  if
\eqref{5313} does not vanish identically, then $\abs{n_0-n_{p+1}}\leq C\abs{n'}$. Using this inequality, one checks that, 
if $ M $ is in $\Nt{m}{p}$, then $M(\eta,\dots,\eta)$ satisfies bounds \eqref{5314} for any $j$ if $\sigma_0$ is large
enough. Actually, if the largest frequency among $n_0,\dots,n_{p+1}$ is $n_0$ or $n_{p+1}$, and if it is much larger than
$n_1,\dots,n_{p}$, then both of them are of the same
magnitude, and one will get a contribution to the left hand side of \eqref{5314} bounded from above by the last term in the
right hand side. On the other hand, if one among $n_1,\dots,n_{p}$ is  larger than all other integers among
$n_0,\dots,n_{p+1}$, one will get a bound by the first term in the right hand side of \eqref{5314}, the last factor in that
expression $\Gcalsm{\sigma+\mu}{0,1}{\eta_1,\dots,\eta_p}$ (with a $\mu$  possibly different from the one in \eqref{5313}) coming from that
special index. Consequently $M(\eta,\dots,\eta)$ is in $\Nr{m}{K,j,p}$ for any $r$ and any $j$.

In particular, an operator $ M(\eta) $ of $\sN{m}{K,j,p}{N}$ sends a
couple $(\eta, \Phi) $ where 
$\eta $ is in $ \Brs{K}{0}{\sigma}\cap \CKH{\sigma+\mu}{\C^2}$ and  $\Phi$ in $C_*^K(I,E^\sigma_j)$ 
to a function $ M(\eta) \Phi  $ in $  C_*^{K}(I, E^{\sigma-m}_j  ) $.

$\bullet$  
If $M(\eta)$ belongs to $\Nr{0}{K,j,1}$ then  the series $\sum_{\ell\geq p}M(\eta)^\ell$ defines an operator 
of $\Nr{0}{K,j,p}$.
Indeed, given $M(\eta)$ of $\Nr{0}{K,j,1}$, iterating estimate \eqref{5314} with $p=1$, $m=0$, one gets by
induction, for $k\leq K$,
\begin{multline*}
  \norm{\partial_t^k[M(\eta)^\ell\Phi]}_{E^{\sigma-\frac{3}{2}k}_j}\leq
  (3A_KC_\sigma)^\ell\\\times\sum_{k'+k''=k}\bigl(\Gcals{\sigma'_0}{k',1}{\eta}^{\ell-1}\Gcals{\sigma+\mu}{k',1}{\eta}\Gcalst{\sigma_0,j}{k'',1}{\Phi}
\\+ \Gcals{\sigma'_0}{k',1}{\eta}^\ell \Gcalst{\sigma,j}{k'',1}{\Phi}\bigr)
\end{multline*}
for some constant $A_K$ depending only on $K, j$, and some $\sigma'_0$ depending on $\sigma_0, \mu$. If we assume that
$\Gcalsm{\sigma'_0}{K,1}{\eta} = \nnorm{\eta}_{K,\sigma'_0} < r'(\sigma)$ for a small enough $r'(\sigma)$, we 
conclude that the series $\sum_{\ell\geq p}M(\eta)^\ell$ satisfies estimates of the form \eqref{5314} (with $m=0$)
and thus defines an element of $\Nr{0}{K,j,p}$.

\medskip

Let us consider a variant of the above spaces, obtained replacing in \eqref{5313}, \eqref{5314} the function $\Phi$ by
$\frac{\cosh((z+1)D)}{\cosh D}\psi$, for some function $\psi$ depending only on $x$. 
Since $ \frac{\cosh((z+1)\xi)}{\cosh \xi} $ is in $ \Pt{0,+}{0} $ (see \eqref{def:CSymbols}), 
 lemma \ref{para-BI} implies that, for any $ j $, 
\be\label{eq:Psi-psi}
\Norm{\frac{\cosh((z+1)D)}{\cosh D}\psi}_{E^s_j} \leq C\norm{\psi}_{\Hds{s}} \, .
\ee
The following definition of families of operators sending functions of
$x$ to functions of $(z,x)$ is suggested by the estimates obtained plugging \eqref{Psi-psi} inside  \eqref{5313}, \eqref{5314}.
\begin{definition}
  \label{533a}
Let $m$ be in $\R_+$, $K, p$ in $\N$ with $p\leq N$, $r>0$, $j\in \N$.

(i) We denote by \index{Nd@$\Ntd{m}{p}$ (Space of maps)} $\Ntd{m}{p}$ the space of symmetric $p$-linear maps of 
the form $(\eta_1,\dots,\eta_{p}) \to M(\eta_1,\dots,\eta_{p})$, 
 defined  on  $\Hds{\infty}(\Tu,\C^2)^p$, with values in the space of linear maps from 
 $\Hds{\infty}(\Tu,\C)$ to $E^\infty_\infty$, such that 
there is $\mu>0$ and for any  $(\eta_1,\dots,\eta_p)$ in $\Hds{\infty}(\Tu,\C^2)$, any $\psi$ in  $\Hds{\infty}(\Tu,\C)$,
 any $n_0, n_{p+1}$ in $\N^*$, any $n' = (n_1,\dots,n_p)$ in $(\N^*)^p$, any $j$ in $\N$,
\begin{multline}
  \label{eq:5313a}
\norm{\Pin{0}M(\Pin{1}\eta_1,\dots,\Pin{p}\eta_p)\Pin{p+1}\psi}_{E^0_j}\\
\leq C(n_0+\abs{n'}+n_{p+1})^m\abs{n'}^\mu \Gcalsm{0}{0,p}{\eta_1,\dots,\eta_p}\Gcalsm{0}{0,1}{\psi}
\end{multline}
and such that condition \eqref{2126} holds.

(When $p=0$, the
above conditions just mean that $M$ is a linear map from  $\Hds{\infty}(\Tu,\C)$ to $E^\infty_\infty $, satisfying 
estimate \eqref{5313a} and \eqref{2126}).

(ii) We denote by \index{Ne@$\Nrd{m}{K,j,N}$  (Space of maps)} $\Nrd{m}{K,j,N}$ the space of maps $(\eta,\psi)\to
M(\eta)\psi$, linear in $\psi$, such that there are $\sigma_0$ and $\mu$ in $\R_+$ so that, for any $\sigma>\sigma_0$, there
is $r(\sigma)<r$ and, for any $\eta$ in $\Brs{K}{0}{\sigma}\cap\CKH{\sigma+\mu}{\C^2}$, any $\psi$ in $\CKH{\sigma}{\C}$, one
has for any $0 \leq k\leq K$, any $0 \leq j'\leq j$, the bounds
\begin{multline}
  \label{eq:5314a}
\norm{\partial_t^k(M(\eta)\psi)}_{E^{\sigma-m-\frac{3}{2}k}_{j'}}\leq C\sum_{k'+k''=k}\bigl(\Gcals{\sigma_0}{k',N-1}{\eta} \Gcals{\sigma+\mu}{k',1}{\eta}
\Gcals{\sigma_0}{k'',1}{\psi}\\
+ \Gcals{\sigma_0}{k',N}{\eta}  \Gcals{\sigma}{k'',1}{\psi}\bigr).
\end{multline}

(iii) One denotes by \index{Nf@$\sNd{m}{K,j,p}{N}$ (Space of maps)} $\sNd{m}{K,j,p}{N}$ the space of sums \eqref{5315} with
$M_q$ in $\Ntd{m}{q}$ for $q = p,\dots,N-1$, and $M_N$ in $\Nrd{m}{K,j,N}$.
\end{definition}
\textbf{Remarks}: $\bullet$ 
By \eqref{Psi-psi}, the operator $ \frac{\cosh((1+z)D)}{\cosh D} $ is in $\Ntd{0}{0}$, and, 
 if $ M(\eta) $ is in $\sN{m}{K,j,p}{N}$ then 
the operator 
$$
\psi\to M(\eta)\frac{\cosh((1+z)D)}{\cosh D}\psi
$$
is an element of $\sNd{m}{K,j,p}{N}$. 

$ \bullet $ Composition at the left of an operator of
$\sNd{m}{K,j,p}{N}$ by an operator of $\sN{0}{K,j,0}{N}$ gives an element of $\sNd{m}{K,j,p}{N}$.

$\bullet$ It follows from estimates \eqref{5313a}, \eqref{5314a} that if $M(\eta)$ is in  the space $\sNd{m}{K,j,p}{N}$, then,
for any $ 0 \leq j_0\leq j$, the operator $\partial_z^{j_0}M(\eta)$ is in $\sNd{m+j_0}{K,j-j_0,p}{N}$ .

$\bullet$ If one multiplies at the left an element of the space $\sN{m}{K,j,p}{N}$ (resp.\ $\sNd{m}{K,j,p}{N}$) by a smooth function of
$(\eta, \eta',\dots,\eta^{(\ell)})$, defined on a neighborhood of zero, that vanishes with order $ p' $ at $\eta=0$, 
one gets an element of the space $\sN{m}{K,j,p+p'}{N}$ (resp.\
$\sNd{m}{K,j,p+p'}{N}$). 

$\bullet$ Let $(\Phi,\Phi')\to a(\Phi,\Phi';z,x,\xi)$ be an element of $\sGb{m}{K,0,j,p}{N}$
defined in Definition~\ref{511},
linear in $\Phi'$ (see the remarks after Definition~\ref{511}), and $M(\eta)$ be an operator in 
$\sNd{m'}{K,j,1}{N}$ for some $m'$.  Then, 
for any function $\psi (x)$, the map 
\[(\Phi,\psi)\to \tilde{a}(\Phi,\psi;z,x,\xi) = a(\Phi,M(\eta)\psi;z,x,\xi)\]
defines an element of $\sGb{m}{K,0,j,p+q}{N}$, with linear dependence in $\psi$.

$ \bullet $
If $M$ is in $\Ntd{m}{p} $ then $M(\eta,\dots,\eta)$ is in $ \Nrd{m}{K,j,p} $ for any $r > 0 $, any $j$. In particular, 
an operator $ M(\eta) $ of $\sN{m}{K,j,p}{N}$ sends a
couple $(\eta, \psi) $ where 
$\eta $ is in $ \Brs{K}{0}{\sigma}\cap C^K_* (I, {\dot H}^{\sigma+ \mu} (\Tu, \C^2) )$ and  $\psi$ in $ \CKH{\sigma}{\C} $
to a function $ M(\eta) \psi  $ in $  C_*^{K}(I, E^{\sigma-m}_j  ) $.

$ \bullet $ 
If $ M $ is an operator in $ \sNd{m}{K,0,q}{N} $ then it has the form \eqref{MUVB} 
and \eqref{5126} holds (with $ m $ replaced by $ m + \mu $) as well as \eqref{5127} (with some $ \sigma_0' $ large
depending on $ K $ and $m$). 

$\bullet$ Consider a smoothing operator $R(\eta)$ in $\sRb{-\rho}{K,0,j,p}{N}$ (see Definition~\ref{512})
and let $M(\eta)$ be in
    $\sNd{m}{K,j,q}{N}$. Then the composed operator 
    $R(\eta)\circ M(\eta)$ defines a smoothing operator of $\sRb{-\rho+m'}{K,0,j,p+q}{N}$  for some $m'\geq    m$, acting on
    a function    $ \psi (x) $. The 
   argument is the same as in  Proposition \ref{233}-(ii)-(iii). 
    Here we adopt the abuse of notation introduced in the last remark after Definition~\ref{512} 
    to denote these classes with the same notations. 
    
$\bullet$ If  $R(\Phi,\Phi')$ is a smoothing operator of $\sRb{-\rho}{K,0,j,p}{N}$ which is linear in $\Phi'$, and
$M(\eta)$ in $\sNd{m}{K,j,q}{N}$, then $R(\Phi,M(\eta)\psi)$ is a smoothing 
operator of $\sRb{-\rho+m'}{K,0,j,p+q}{N}$ for some
$m'\geq  m$, and using again the
abuse of notation  introduced in the last remark after Definition~\ref{512}. 
The proof is the same  as for Proposition \ref{233}-(iv), 
composing as above \eqref{518} and \eqref{5313a}, or \eqref{519} and \eqref{5314a}.

$\bullet$ 
Let $\tilde{M}(\eta)$ be an operator in $\sNd{m}{K,0,1}{N-1}$, for some $m$. 
Then the para-product operator 
$ \opbw(\tilde{M}(\eta)\psi ) $ 
defines an element of $\sNd{\frac{1}{2}}{K,0,1}{N}$, 
 linear  in $ \psi $, i.e.  the nonhomogeneous part of $ \opbw(\tilde{M}(\eta)\psi ) \tilde \psi  $
  satisfies estimates of the form \eqref{5314a} (with $ m = 1 / 2 $ and $ j'=0 $), where the right hand side is replaced by
\begin{multline}\label{eq:5314aab}\sum_{k'+k''=k}
\bigl(\Gcals{\sigma_0}{k',N-2}{\eta}\Gcals{\sigma_0}{k',1}{\psi} \Gcals{\sigma+\mu}{k',1}{\eta}
\Gcals{\sigma_0}{k'',1}{\tilde{\psi}}\\+ 
\Gcals{\sigma_0}{k',N-1}{\eta} \Gcals{\sigma+\mu}{k',1}{\psi}
\Gcals{\sigma_0}{k'',1}{\tilde{\psi}}\\
+ \Gcals{\sigma_0}{k',N-1}{\eta} \Gcals{\sigma_0}{k',1}{\psi} \Gcals{\sigma}{k'',1}{\tilde{\psi}}\bigr) \, .
\end{multline}
We write the reasoning just for the nonhomogeneous
 term  $\tilde{M}$ in $\Nrd{m}{K,0,N-1}$. 
  At each fixed $z \in [-1,0]$,  $\opbw\bigl((\tilde{M}(\eta)\psi)(z;\cdot)\bigr)\tilde{\psi}$ is the para-product  of the function
$( \tilde{M}(\eta)\psi)(z;\cdot)$ by $\tilde{\psi}$. Consequently, 
for any $\sigma$,  
\begin{multline}\label{eq:para-pr1}
\norm{\partial_t^k\opbw\bigl(
(\tilde{M}(\eta)\psi)(z;\cdot)\bigr)\tilde{\psi}}_{\Hds{\sigma-\frac{3}{2}k}} \leq \\
C \sum_{k_1+k_2 =
    k}\norm{\partial_t^{k_1} (\tilde{M}(\eta)\psi)(z;\cdot)}_{\Hds{1}}\norm{\partial_t^{k_2}\tilde{\psi}}_{\Hds{\sigma-\frac{3}{2}k_2}} \, . 
    \end{multline}
  By \eqref{5314a} (with $ N- 1$ replaced by $ N - 2 $ and $ N $ by $ N - 1 $), we may find some
  $ \sigma'_0 $ large enough (depending on $ \sigma_0 $, $ m $ and $ K$) 
  such that,  for any $z$ in $[-1,0]$, any $ 0\leq k_1 \leq K$,
$$
\norm{\partial_t^{k_1} (\tilde{M}(\eta)\psi)(z;\cdot)}_{\Hds{1}}\leq C\sum_{k'+k'' =
    k_1}\Gcals{\sigma_0'}{k',N-1}{\eta}\Gcals{\sigma_0'}{k'',1}{\psi} \, .
$$
As a consequence, \eqref{para-pr1} is bounded by
$$
C \sum_{k_1+k_2 =
  k}\Gcals{\sigma_0'}{k_1,N-1}{\eta}\Gcals{\sigma_0'}{k_1,1}{\psi}\Gcals{\sigma}{k_2,1}{\tilde{\psi}} \, ,
  $$
uniformly in $ z $. Therefore
$ \| \partial_t^k \opbw (\tilde{M}(\eta)\psi ) \tilde{\psi} \|_{F_0^{\sigma-\frac{3}{2}k}} $  
is bounded by the right
hand side of \eqref{5314aab} (with $ \sigma_0'$ instead of $ \sigma_0 $). 
Recalling  the continuous embedding of $ F_0^{\sigma  -\frac{3}{2}k}$
into $E_0^{\sigma - \frac12 -\frac{3}{2}k}$ (see the first remark after Definition \ref{515})
we have proved that $ \opbw(\tilde{M}(\eta)\psi ) \tilde \psi  $ satisfies 
an estimate of the form \eqref{5314a} with $ m = 1 / 2 $ and $ j' = 0 $, as claimed.

\begin{proposition} \label{534}
Let $K$ be in $\N$, $N$ in $\N^*$, $m\geq 2$, $r>0$, $j$ in $\N$. Consider 
an operator $\tilde{M}(\eta)$ of $\sNd{m}{K,j,1}{N}$,
 and operators $G_\ell(\eta)$ of $\sN{2-\ell}{K,j,1}{N}$, $\ell = 0, 1, 2$. Assume also that we are given a function
  $\tilde{S}_N(\eta,\psi)$ of $(z, x) \in \Bcal $, supported for $z\in [-1,-\frac{1}{4}]$ and such that, for any large enough $\sigma_0$, for any $j$, there is some
  $\sigma'_0>0$ with, for all $k=0,\dots,K$,
  \begin{equation}
    \label{eq:5315a}
    \sum_{k'=0}^k\norm{\partial_t^{k'}\tilde{S}_N(\eta,\psi)}_{E^{j+\sigma_0-\frac{3}{2}k'}_j} \leq
    C\sum_{k'+k''=k}\Gcals{\sigma'_0}{k',N}{\eta}\Gcals{\sigma'_0}{k'',1}{\psi}.
  \end{equation}
Then, eventually increasing $ \sigma_0 $, for any $\sigma>\sigma_0$, there is $r(\sigma)<r$, such that for 
$\eta$ in $\Brs{K}{0}{\sigma}\cap \CKH{\sigma+\mu}{\C^2}$, for $\psi$ in $\CKH{\sigma}{\C}$, the elliptic problem 
\begin{equation}
  \label{eq:5316}
  \begin{split}
    (\partial_x^2+\partial_z^2)\tilde{\Phi} &= G_0(\eta)\Phit + \partial_z[G_1(\eta)\Phit] + \partial_z^2[G_2(\eta)\Phit] +
    \tilde{M}(\eta)\psi + \tilde{S}_N\\
\Phit\vert_{z=0} &= \psi\\
\partial_z\Phit\vert_{z=-1} &= 0
  \end{split}
\end{equation}
has a unique solution $\Phit$ in $C^K_*(I,E^{\sigma-m+2}_j)$. 
If $ \tilde{M} \equiv 0 $ and $ \tilde{S}_N \equiv 0$, then $\Phit$ is in $C^K_*(I,E^{\sigma}_j)$. 
Moreover $ \Phit $  may be written as
\begin{equation}
  \label{eq:5317}
  \Phit = \frac{\cosh((1+z)D)}{\cosh D}\psi + M(\eta)\psi + S_N(\eta,\psi)
\end{equation}
where $M(\eta)$ is some operator in $\sNd{m-2}{K,j,1}{N}$ and the function $S_N$ is supported for $-1\leq z\leq -\frac{1}{8}$ and
satisfies estimates \eqref{5315a} with $E_j^{j+\sigma_0-\frac{3}{2}k}$ replaced by $E^{j+2+\sigma_0-\frac{3}{2}k}_j$ in the left hand side. Moreover, if $\tilde{S}_N\equiv
0$ in \eqref{5316}, then $S_N\equiv 0$ in \eqref{5317} and if $\tilde{M} \equiv 0$, then  $M(\eta)$ is in $\sNd{0}{K,j,1}{N}$.
\end{proposition}

In order to prove the proposition we need the following lemma. 

\begin{lemma} \label{535}
Let  $G_\ell(\eta)$, $\ell = 0, 1, 2 $,  be operators in $\sN{2-\ell}{K,j,1}{N}$, 
and  $\eta$ a function as in the proposition.
Consider the operator 
\begin{equation} \label{eq:5318}
  M_0(\eta)\Phi = \int_{-1}^0 K_0(z,z',D_x)\bigl[\partial_{z'}^2[G_2(\eta)\Phi] + \partial_{z'}[G_1(\eta)\Phi] + G_0(\eta)\Phi\bigr]\,dz' 
\end{equation}
where $K_0(z,z',D_x)$ is  defined in \eqref{523}, acting on functions $\Phi$ satisfying 
\be\label{eq:PhiBC}
\Phi\vert_{z=0} = 0 \, , \quad \partial_z\Phi\vert_{z=-1}=0 \, . 
\ee
Then $M_0(\eta)$ is an operator of $\sN{0}{K,j,1}{N}$. 
Moreover we may express $M_0(\eta)$  
from a distributional kernel
$\tilde{K}_1(\eta;z,z')$ (with values in operators) by
\begin{equation}
  \label{eq:5319}
  M_0(\eta)\Phi =  \int_{-1}^0\tilde{K}_1(\eta;z,z')\Phi(z',\cdot)\,dz'
\end{equation}
and, more generally, for any $ n \in \N $
\be\label{eq:M0n}
M_0(\eta)^n\Phi = \int_{-1}^0\tilde{K}_n(\eta;z,z')\Phi(z',\cdot)\,dz' \, , 
\ee 
where 
$\tilde{K}_n(\eta;z,z')$, $ n \geq 2 $,  are defined iteratively by 
\begin{equation}
  \label{eq:5320}
  \tilde{K}_{n+1}(\eta;z,z') =  \int_{-1}^0 \tilde{K}_n(\eta;z,z'')\tilde{K}_1(\eta;z'',z') 
  \,dz'' 
\end{equation}
and the following estimates hold: 
There are $\sigma_0>0$, $\mu>0$, $\sigma_1>0$ with $\sigma_0\geq \sigma_1+\mu$, and for every $\ell$ in $\N$, any $\sigma\geq
\sigma_0$, a constant $C_{\ell,\sigma}>0$ such that for any $n$ in $\N^*$, for any $0\leq\ell'\leq\ell$, for $0\leq k\leq K$,
  \begin{multline}
    \label{eq:5321}
\sum_{k'=0}^k \Norm{\partial_t^{k'}\int_{-1}^0 \tilde{K}_n(\eta;z,z')(z-z')^\ell\Phi(z',\cdot)\,dz'}_{E^{\sigma+\ell'-\frac{3}{2}k'}_j}\\
\leq C_{\ell,\sigma}^n \Gcals{\sigma_0}{k,n-1}{\eta}\sum_{k'+k''=k}\bigl[\Gcals{\sigma+\ell'+\mu}{k',1}{\eta} \Gcalst{\sigma_1,j}{k'',1}{\Phi}\\
+  \Gcals{\sigma_0}{k',1}{\eta}\Gcalst{\sigma,j}{k'',1}{\Phi}\bigr].
  \end{multline}

\end{lemma}
  \begin{proof}
Since $ \Phi $ satisfies the boundary conditions \eqref{PhiBC}, 
 by \eqref{5311} and  \eqref{5312}  we may write
    \begin{multline}
      \label{eq:5322}
M_0(\eta)\Phi = G_2(\eta)\Phi(z,\cdot) - \Scal(z,D_x)[G_1(\eta)\Phi(-1,\cdot)]\\
+ \int_{-1}^0 K_0(z,z',D_x)[D_x^2 G_2(\eta)\Phi(z',\cdot) +G_0(\eta)\Phi(z',\cdot)]\,dz'\\
-  \int_{-1}^0 \partial_{z'}K_0(z,z',D_x)[G_1(\eta)\Phi(z',\cdot)]\,dz'.
    \end{multline}
Since $ K_0  $ is a Poisson symbol in $ \Pti{-1}{0} $, 
lemma \ref{516} implies  that the operator 
\be\label{eq:gain2}
U\to \int_{-1}^0 K_0(z,z',D_x)U(z')\,dz'  \quad {\rm maps} \quad E^s_j \to E^{s+2}_j \, , 
\ee
and, similarly, since $ \pa_{z'} K_0  $ is in $ \Pti{0}{0} $, 
 $U\to
\int_{-1}^0 \partial_{z'}K_0(z,z',D_x)U(z')\,dz'$  acts from $E^s_j$ to $E^{s+1}_j$. It follows that \eqref{5322} defines 
an operator of $\sN{0}{K,j,1}{N}$.

Remark that \eqref{5322} provides an integral expression of the form \eqref{5319}, with the operatorial valued kernel
\begin{multline}\label{eq:kernels}
\tilde{K}_1 (\eta; z, z' ) = \delta(z-z')G_2(\eta) + \delta(z'+1)\Scal(z,D_x) \circ G_1(\eta) \\
+ K_0(z,z',D_x) \circ [D_x^2 G_2(\eta) +G_0(\eta)] -  \partial_{z'}K_0(z,z',D_x) \circ G_1(\eta)  \, .
\end{multline}
We claim that $\tilde{K}_1 (\eta; z, z' ) $ 
satisfies for any 
$0\leq k\leq K$, $0\leq\ell'\leq \ell$, some $\mu>0$, some $\sigma_1>0$ and any $\sigma>\sigma_1$, the
 estimate
\begin{multline}
  \label{eq:5322a}
\sum_{k'=0}^k\Norm{\partial_t^{k'}\int_{-1}^0 \tilde{K}_1(\eta;z,z')(z-z')^\ell\Phi(z',\cdot)\,dz'}_{E^{\sigma+\ell'-\frac{3}{2}k'}_j}\\
\leq B_{\ell+\sigma}\sum_{k'+k''=k}[\Gcals{\sigma+\ell'+\mu}{k',1}{\eta} \Gcalst{\sigma_1,j}{k'',1}{\Phi} + \Gcals{\sigma_1}{k',1}{\eta} \Gcalst{\sigma,j}{k'',1}{\Phi}]
\end{multline}
for some increasing sequence of constants $B_{\sigma+\ell}$. 
Actually, since $ G_2 (\eta ) $ is in $\sN{0}{K,j,1}{N}$, 
the kernel $\delta(z-z')G_2(\eta)$
satisfies \eqref{5322a} when $\ell=0$, as follows by \eqref{5314} where we replace
$\sigma_0$ by $\sigma_1$. The  estimate \eqref{5322a} when $\ell>0 $ is  trivial. 
Also $\delta(z'+1)\Scal(z,D_x) G_1(\eta) $ satisfies \eqref{5322a}:
Actually since $ \Scal(z,\xi) $ is in $ \Pt{-1,-}{0} $,
the symbol $ (z+1)^\ell\Scal(z,\xi)$ is in $ \Pt{-1-\ell,-}{0} $ (see remarks after Definition \ref{513}), 
lemma \ref{para-BI} implies that the operator 
$(z+1)^\ell\Scal(z,D_x)$
sends $\Hds{\sigma-1}$ to $E^{\sigma+\ell}_\infty$. 
Using that  $G_1(\eta)$ satisfies \eqref{5314}  with $m=1$, we deduce that 
$\delta(z'+1)\Scal(z,D_x) G_1(\eta) $ satisfies  \eqref{5322a} (with $\ell'=0$). 

Moreover, since $ K_0 $ is a Poisson symbol in $ \Pti{-1}{0} $, 
then $ (z- z')^\ell K_0 $ is in $ \Pti{-1-\ell}{0} $ (see remarks after Definition \ref{513}), and lemma \ref{516} implies that 
the operator
$$
U\to \int_{-1}^0(z-z')^\ell K_0(z,z',D_x) U(z' , \cdot )\, dz' $$ 
acts from $E^\sigma_j$ to $E^{\sigma+\ell+2}_j$. Similarly, since $ \pa_{z'} K_0  $ is in $ \Pti{0}{0} $,  
the operator $U\to
\int_{-1}^0(z-z')^\ell\partial_{z'}K_0(z,z',D_x) U(z' , \cdot ) \,dz'$ acts  from $E^\sigma_j$ to $E^{\sigma+\ell+1}_j$.
Using that  $ G_2 (\eta ) $ is in $\sN{0}{K,j,1}{N}$ and  $ G_1 (\eta ) $ is in $\sN{1}{K,j,1}{N}$, 
we deduce from
\eqref{5314} that the two integral terms in \eqref{5322} may be written under the form \eqref{5319}, with a kernel
given by the last line in \eqref{kernels}, 
satisfying the estimate \eqref{5322a}.

The estimate \eqref{5321} for $ n = 1 $ follows by \eqref{5322a},
setting 
\be\label{eq:Csell}
C_{\ell, \sigma} = 3^{\ell+1}(K+1)B_{\sigma+\ell}
\ee 
and $\sigma_0 = \sigma_1+\mu$ where $\sigma_1$ is the index in \eqref{5322a}.

The expression \eqref{M0n} of $ M_0(\eta)^n\Phi $ for $ n \geq 2 $ follows 
by the formula \eqref{5319} of $M_0(\eta)$ and the iterative definition \eqref{5320} of
$\tilde{K}_n$. 

Let us prove by induction \eqref{5321} for any $n$. 
Write
\begin{multline*}
 A^\ell_{n+1} \stackrel{\mathrm{def}}{=} \int_{-1}^0\tilde{K}_{n+1}(\eta;z,z')(z-z')^\ell\Phi(z',\cdot)\,dz'\\
= \sum_{0 \leq \ell_1\leq\ell}\bin{\ell}{\ell_1} \int_{-1}^0 \tilde{K}_{n}(\eta;z,z'')(z-z'')^{\ell-\ell_1}\\
\times \Big( \int_{-1}^0 \tilde{K}_{1}(\eta;z'',z')(z''-z')^{\ell_1}\Phi(z',\cdot)\,dz' \Big) dz''.
\end{multline*}
We now 
apply the inductive estimate \eqref{5321} to the first integral kernel. For any $\ell'\leq\ell$,
we  decompose $\ell' = (\ell'-\ell'_1)+\ell'_1$ with $\ell'_1\leq\ell_1$, $\ell'-\ell'_1\leq\ell-\ell_1$ and
we have apply \eqref{5321} with $\sigma$ (resp.\ $\ell'$) replaced by $\sigma+\ell'_1$ (resp.\ $\ell'-\ell'_1$), 
yielding
\begin{multline}
  \label{eq:5322b}
\sum_{k'=0}^k \norm{\partial_t^{k'}A_{n+1}^{\ell}}_{E^{\sigma+\ell'-\frac{3}{2}k'}_j} \leq
\sum_{\ell_1\leq\ell}\bin{\ell}{\ell_1}\sum_{k'+k''=k}\Gcals{\sigma_0}{k',1}{\eta}^{n-1}\\\times
C^n_{\ell'-\ell'_1,\sigma+\ell'_1}  \Bigl[\Gcals{\sigma+\ell'+\mu}{k',1}{\eta}
\Gcalst{\sigma_1,j}{k'',1}{A_1^{\ell_1}} + \Gcals{\sigma_0}{k',1}{\eta} \Gcalst{\sigma + \ell'_1,j}{k'',1}{A_1^{\ell_1}}\Bigr] \, .
\end{multline}
Next, according to \eqref{5322a} with $ \sigma $ replaced by $ \sigma_1 $, 
$ \ell' = 0 $, $ \ell = \ell_1 $, and since $ \sigma_0 = \sigma_1+\mu $,  we bound 
\begin{multline*}\Gcalst{\sigma_1,j}{k'',1}{A_1^{\ell_1}} = \sum_{k'''=0}^{k''} \norm{\partial_t^{k'''}A^{\ell_{1}}_1}_{E^{\sigma_1-\frac{3}{2}k'''}_j} \\\leq
  2B_{\sigma_1+\ell_1}\sum_{k''_1+k''_2=k''}\Gcals{\sigma_0}{k''_1,1}{\eta} \Gcalst{\sigma_1,j}{k''_2,1}{\Phi}
 \end{multline*}
and according to \eqref{5322a} with $ \ell'  $ replaced by $ \ell'_1 $ and  $ \ell = \ell_1 $, we get 
\begin{multline*}
\Gcalst{\sigma + \ell'_1,j}{k'',1}{A_1^{\ell_1}} = \sum_{k'''=0}^{k''} \norm{\partial_t^{k'''}A^{\ell_{1}}_1}_{E^{\sigma+ \ell'_1-\frac{3}{2}k'''}_j}
\\\leq B_{\sigma+\ell_1}\sum_{k''_1+k''_2=k''}\Bigl[\Gcals{\sigma+\ell'_1+\mu}{k''_1,1}{\eta} \Gcalst{\sigma_1,j}{k''_2,1}{\Phi}\\
+\Gcals{\sigma_1}{k''_1,1}{\eta} \Gcalst{\sigma,j}{k''_2,1}{\Phi}\Bigr].
\end{multline*}
Plugging these estimates inside \eqref{5322b} and using the definition of $C_{\ell, \sigma}$ in \eqref{Csell}, we obtain 
that, for any $ 0 \leq k \leq K $, the left hand side of \eqref{5322b} is bounded by
\begin{multline*}
  2^\ell(3^{\ell+1}(K+1)B_{\sigma+\ell})^n (K+1)B_{\sigma+\ell}\sum_{k'+k''=k}\bigl[3 \Gcals{\sigma+\ell'+\mu}{k',1}{\eta}
  \Gcalst{\sigma_1,j}{k'',1}{\Phi}\\
+\Gcals{\sigma_0}{k',1}{\eta} \Gcalst{\sigma,j}{k'',1}{\Phi}\bigr]\Gcals{\sigma_0}{k,1}{\eta}^n
\end{multline*}
which implies \eqref{5321} at rank $n+1$. This concludes the proof of the lemma.
  \end{proof}
 
 \begin{proof1}{Proof of Proposition~\ref{534}}
 Since  $ \frac{\cosh((1+z)D)}{\cosh D}\psi  $ is the solution of 
$$
    (\partial_x^2+\partial_z^2) \Phi  = 0 \, , \quad 
\Phi \vert_{z=0} = \psi \, , \quad 
\partial_z \Phi \vert_{z=-1} = 0 \, , 
$$
the solution $ \tilde{\Phi} $ of the problem \eqref{5316} may be written as
\[ 
\tilde{\Phi} = \frac{\cosh((1+z)D)}{\cosh D}\psi + \phitu  
\]
where $\phitu$ satisfies
\begin{equation}
  \label{eq:5322bb}
  \begin{split}
    (\partial_x^2+\partial_z^2)\phitu &= G_0(\eta)\phitu + \partial_z[G_1(\eta)\phitu] + \partial_z^2[G_2(\eta)\phitu] +
    \tilde{\underline{M}}(\eta)\psi + \tilde{S}_N\\
\phitu\vert_{z=0} &= 0\\
\partial_z\phitu\vert_{z=-1} &= 0
  \end{split}
\end{equation}
where
\[
\begin{split}
  \tilde{\underline{M}}(\eta) \psi \stackrel{\mathrm{def}}{=}  G_0(\eta)\frac{\cosh((1+z)D)}{\cosh D}\psi
  + \partial_z\biggl[G_1(\eta)\frac{\cosh((1+z)D)}{\cosh D}\psi\biggr] \\
+ \partial^2_z\biggl[G_2(\eta)\frac{\cosh((1+z)D)}{\cosh D}\psi\biggr] + {\tilde M}(\eta)\psi.
\end{split}
\]
By assumption ${\tilde M}(\eta) $ is an operator in  $\sNd{m}{K,j,1}{N}$ for some $ m \geq 2 $. 
Since $G_\ell(\eta)$ are operators in  $\sN{2-\ell}{K,j,1}{N}$, $\ell = 0, 1, 2 $, 
the first and the third remarks following Definition~\ref{533a}, 
imply that $\tilde{\underline{M}}(\eta)$ is in $\sNd{m}{K,j-2,1}{N}$.
Notice that, if $ {\tilde M}(\eta) = 0  $, then $\tilde{\underline{M}}(\eta)$ is in $\sNd{2}{K,j-2,1}{N}$.

 Notice also that  since $ \frac{\cosh((1+z)\xi )}{\cosh \xi} $ is in $ \Pt{0,+}{0} $ (see \eqref{def:CSymbols}) and the function 
$ \psi \in  C^K_*(I, \Hds{\sigma}(\Tu,\C)) $, then,  
lemma \ref{para-BI} implies that $ \frac{\cosh((1+z)D)}{\cosh D}\psi $ is in $ C^K_*(I,E^{\sigma}_j) $,  for any $ j $.

By lemma~\ref{522} applied  with $ \eta = 0 $, 
$ g_+ =  g_- = 0 $, and 
$$ 
f = G_0(\eta)\phitu + \partial_z[G_1(\eta)\phitu] + \partial_z^2[G_2(\eta)\phitu] +
    \tilde{\underline{M}}(\eta)\psi + \tilde{S}_N \, ,  
 $$ 
the function $\phitu$ solves \eqref{5322bb} if and only if
\[
  (\mathrm{Id}-M_0(\eta))\phitu = F \stackrel{\mathrm{def}}{=} 
 \int_{-1}^0 K_0(z,z',D_x)[\tilde{\underline{M}}(\eta)\psi + \tilde{S}_N(\eta,\psi)]\,dz' 
\]
where $ M_0(\eta) \phitu $  
is defined in \eqref{5318} and $ K_0(z,z',D_x) $ in \eqref{523}. 
Notice that $F$ satisfies the boundary conditions \eqref{PhiBC}.
Write
\begin{equation}
\label{eq:5323}
F = \tilde{M}_1(\eta)\psi + \int_{-1}^0 K_0(z,z',D_x)\tilde{S}_N(\eta,\psi)(z')\,dz'
\end{equation}
where 
\be\label{eq:M1eta}
\tilde{M}_1(\eta)\psi = \int_{-1}^0 K_0(z,z',D_x)\tilde{\underline{M}}(\eta)\psi\,dz'.
\ee
We claim that, even if 
$\tilde{\underline{M}}(\eta)$ is in $\sNd{m}{K,j-2,1}{N}$, 
the operator $\tilde{M}_1(\eta)$ belongs to
$\sNd{m-2}{K,j,1}{N}$. Actually, this follows from the fact that, by a direct computation, recalling \eqref{523}, 
\[
  \partial_z\int_{-1}^0K_0(z,z',D_x)U(z')\,dz' =  \int_{-1}^0K'_0(z,z',D_x)D_xU(z')\,dz'
\]
with  a symbol $K'_0$ in $\Pti{-1}{0}$, and, according to \eqref{bpforkzero}, 
\[  
\partial^2_z\int_{-1}^0K_0(z,z',D_x)U(z')\,dz'=  \int_{-1}^0K_0(z,z',D_x)D^2_xU(z')\,dz' + U(z) \, .
\]
As a consequence,  to estimate 
$\partial_z^j $-derivatives of  
$\tilde{M}_1(\eta)\psi$ in \eqref{M1eta}, we need only 
$\partial_z^{j-2}$-derivatives of $\tilde{\underline{M}}(\eta)\psi$. 
Recalling \eqref{gain2} and Definition \ref{533a} we deduce that $\tilde{M}_1(\eta)$ belongs to 
$\sNd{m-2}{K,j,1}{N}$.
If $ {\tilde M}(\eta) = 0  $, then  $ \tilde{\underline{M}}(\eta) $ is in $ \sNd{2}{K,j-2,1}{N}$, and 
$ \tilde{M}_1(\eta) $ is in $\sNd{0}{K,j,1}{N} $.

We shall solve the equation $(\mathrm{Id}-M_0(\eta))\phitu = F$ by the 
Neumann series
\begin{equation}
  \label{eq:5324}
  \phitu = \sum_{n=0}^{\infty}M_0(\eta)^nF \, .
\end{equation}
Since the operator $ M_0(\eta) $ is in $ \sN{0}{K,j,1}{N} $ 
as stated in  lemma \ref{535}, 
the second remark after Definition~\ref{533} implies that the Neumann series $ \sum_{n=0}^{\infty}M_0(\eta)^n $ defines an operator
of  $\sN{0}{K,j,0}{N}$. By the second remark after Definition~\ref{533a}, it follows that the composition 
$$
\Big(\sum_{n=0}^{\infty}M_0(\eta)^n\Big) \circ \tilde{M}_1(\eta) 
$$
with the operator $ \tilde{M}_1(\eta) \in \sNd{m-2}{K,j,1}{N} $ 
in \eqref{5323}-\eqref{M1eta},  
gives an operator $M(\eta)$ in $\sNd{m-2}{K,j,1}{N}$
as in \eqref{5317}. By the sixth remark after Definition \ref{533a}
the function 
$ M(\eta) \psi  $  is in $  C_*^{K}(I, E^{\sigma-m+2}_j  ) $.
If $\tilde{M} \equiv 0$, then $ \tilde{M}_1(\eta) $ is in $\sNd{0}{K,j,1}{N} $, 
and the operator $M(\eta)$ is in $\sN{0}{K,j,1}{N}$, thus $ M(\eta) \psi  $  is in $  C_*^{K}(I, E^{\sigma}_j  ) $.
Summarizing,  if $\tilde{S}_N\equiv 0$, we have proved  \eqref{5317} with $S_N \equiv 0$. 

We are left with \eqref{5324} where we replace $F$ by the
integral term in \eqref{5323}, that is
\be\label{eq:term2}
\sum_{n=0}^{\infty}M_0(\eta)^n \Big( \int_{-1}^0 K_0(z,z',D_x)\tilde{S}_N(\eta,\psi)(z')\,dz' \Big) \, . 
\ee
 By \eqref{M0n}
we may write
\[
M_0(\eta)^n\Phi = \int_{-1}^0\tilde{K}_n(\eta;z,z')\Phi(z',\cdot)\,dz'
\]
so that, if we define $\tilde{K}(\eta;z,z') = \sum_{n \geq 0}\tilde{K}_n(\eta;z,z')$, we have
\begin{equation}
  \label{eq:5325}
  (\mathrm{Id}-M_0(\eta))^{-1}\Phi = \sum_{n =0}^{+\infty}M_0(\eta)^n\Phi = \int_{-1}^0\tilde{K}(\eta;z,z')\Phi(z',\cdot)\,dz',
\end{equation}
the convergence of $\sum_{n \geq 0}\tilde{K}_n$ following from \eqref{5321},  if we assume that $\eta$ is in
$\Brs{K}{0}{\sigma,\ell}$, for a small enough $r(\sigma,\ell)$. For such $\eta$, we  get therefore
\begin{multline}
  \label{eq:5326}
\sum_{k'=0}^k\Norm{\partial_t^{k'}\int_{-1}^0(z-z')^\ell\tilde{K}(\eta;z,z')\Phi(z')\,dz'}_{E^{\sigma+\ell-\frac{3}{2}k'}_j}\\
\leq C'_{\ell,\sigma}\sum_{k'+k''=k}\bigl[\Gcals{\sigma+\ell+\mu}{k',1}{\eta} \Gcalst{\sigma_0,j}{k'',1}{\Phi}\\
+ \Gcals{\sigma_0}{k',1}{\eta} \Gcalst{\sigma,j}{k'',1}{\Phi}\bigr].
\end{multline}
By 
 \eqref{5325}, we decompose the term  \eqref{term2} 
 as $I + II$ where
\begin{equation}
  \label{eq:5327}
  \begin{split}
    I &= \theta(z)\int_{-1}^0 \int_{-1}^0\tilde{K}(\eta;z,z'')K_0(z'',z',D) \tilde{S}_N(\eta,\psi)(z',\cdot)\,dz'dz''\\
 II &= (1-\theta(z))\int_{-1}^0 \int_{-1}^0\tilde{K}(\eta;z,z'')K_0(z'',z',D) \tilde{S}_N(\eta,\psi) (z',\cdot)\,dz'dz''
  \end{split}
\end{equation}
for $\theta \in C^\infty_0(\R)$, equal to one on $[-\frac{1}{8},0]$ and vanishing close to $[-1,-\frac{1}{4}]$. Expression
$II$
is supported for $-1 \leq z \leq -1/8$ and may be written as 
\[
(1-\theta)(z)  \int_{-1}^0\tilde{K}'(\eta;z,z')\tilde{S}_N(\eta,\psi) (z',\cdot)\,dz' \, , 
\] 
with the  new kernel
$$
\tilde{K}'(\eta;z,z') =  \int_{-1}^0\tilde{K}(\eta;z,z'')K_0(z'',z',D) dz''
$$
 that satisfies \eqref{5326} with $E_j^{\sigma+\ell-\frac{3}{2}k'}$ replaced by
$E_j^{\sigma+2+\ell-\frac{3}{2}k'}$ in the left hand side (as integration against $K_0$ gains two derivatives). 
Applying  \eqref{5326} with $\ell=0$ and 
$\sigma=j+\sigma_0$, we obtain that
\begin{multline*}
\sum_{k'=0}^k\norm{\partial_t^{k'}II(t,\cdot)}_{E^{j+\sigma_0+2-\frac{3}{2}k'}_j} \leq
C\sum_{k'+k''=k}\bigl(\Gcals{j+\sigma_0+\mu}{k',1}{\eta}\Gcalst{\sigma_0,j}{k'',1}{\tilde{S}_N}\\ +
\Gcals{\sigma_0}{k',1}{\eta}\Gcalst{j+\sigma_0,j}{k'',1}{\tilde{S}_N}\bigr).
\end{multline*}
If we bound $\Gcalst{\sigma_0,j}{k,1}{\tilde{S}_N}$, $\Gcalst{j+\sigma_0,j}{k,1}{\tilde{S}_N}$ by the right hand side of
\eqref{5315a} (recall \eqref{5115a}), we obtain that $ II $ satisfies a 
similar estimate for a convenient  
$\sigma'_0$ (depending on $j, \mu$). Thus $ II $ contributes to the term $ S_N $ in the right hand side of \eqref{5317}.

Finally we consider 
the term $ I $ in \eqref{5327} 
$$
I = \int_{-1}^0 \theta(z) \tilde{K}'(\eta;z,z')  \tilde{S}_N(\eta,\psi)(z',\cdot)\,dz' \, . 
$$
Using that the support of $ \theta $ and the support of $\tilde{S}_N(z';\cdot)$ stay at positive distance
(recall that $ \supp (\theta ) \subset ] -1/ 4,0 ]$ and $ \supp (\tilde{S}_N) \subset [ -1, -1/ 4 ]$), 
the integral $ I $ may be written 
as
\begin{equation}
  \label{eq:5328}
  \int_{-1}^0(z-z')^\ell\tilde{K}'(\eta;,z,z')\omega(z,z')\tilde{S}_N(z',\cdot)\,dz'
\end{equation}
for a suitable  function $ \omega (z, z' )$ in $C^\infty_0 ([-1,0] \times [-1,0]) $ 
that vanishes close to $z=z'$. 
By the  density of $C^\infty_0(X)\otimes C^\infty_0(Y)$ in $C^\infty_0(X\times Y)$
we write $ \omega(z,z')$ as a rapidly convergent series 
$\sum_{q}\omega_q^1(z)\omega_q^2(z')$, with
$\omega_q^j$ in $C^\infty_0(\R)$. 
This allows us to apply  
the estimates \eqref{5326} for $(z-z')^\ell\tilde{K}'$ (with $\sigma=\sigma_0$) to $(z-z')^\ell\tilde{K}'\omega(z,z')$ 
in  \eqref{5328}. 
In conclusion we obtain
\[\sum_{k' = 0}^k\norm{\partial_t^{k'}I(t,\cdot)}_{E^{\sigma_0+\ell-\frac{3}{2}k'}_j} \leq C\sum_{k'+k''=k}\Gcals{\sigma_0+\ell+\mu}{k',1}{\eta}
\Gcalst{\sigma_0,j}{k'',1}{\tilde{S}_N}.\]
Plugging \eqref{5315a} in the right hand side of the above inequality, and choosing $\ell$ so that
$\sigma-m+2\leq\sigma_0+\ell\leq \sigma-m+1$, we see that we obtain an 
estimate of the form \eqref{5314a}, with  $ \sigma_0 $ replaced
by  some convenient $\sigma'_0 $ and $ m-2 $ instead of $ m $. 
Consequently, we have shown that $I$ is in $ C^K_* (I, E_j^{\sigma-m+2})$ and it 
contributes to the term $ M(\eta)\psi $ in the right hand side of
\eqref{5317} (actually to its non-homogeneous term in $\Nrd{m-2}{K,j,N}$). This concludes the proof.
  \end{proof1}

%% file: chap6max.tex
\chapter[Dirichlet-Neumann and good unknown]{Dirichlet-Neumann operator and the good unknown}\label{cha:6}

\section{The good unknown}\label{sec:61}

We introduce here the good unknown $ \omega $ in terms of which we  shall express the water waves equations. We follow
mainly the approach of Alazard-Métivier~\cite{AM}, Alazard-Burq-Zuily~\cite{ABZ1, ABZ2}. 

Recall that,  	given an integer $ K $, a  real number $ \sigma $ large enough, a function $ \eta $ of 
$ C_*^K(I,\Hds{\sigma}(\Tu,\R))$  with
$\norm{\eta(t,\cdot)}_{L^\infty} < 1 $, we have defined in \eqref{111} the open set 
$$
    \Omega_t = \big\{ (x,y)\in \Tu\times\R \, ;  \ - 1 <y<\eta(t,x) \big\} 
$$
where we take the depth $h = 1 $.
In the sequel we identify $ C_*^K(I,\Hds{\sigma}(\Tu,\R))$  with
the subspace of $\CKHR{\sigma}{\C^2}$ of vector valued functions of the form $\vect{\eta}{\eta}$ with $\eta$ real.
We shall always assume below that 
$\eta$ belongs  to $\Brs{K}{0}{\sigma_0} $ for some given $ \sigma_0 $ large and $ r( \sigma_0) > 0 $ 
small.  

Let $\psi$ be in $C_*^K(I,\Hds{\sigma}(\Tu,\R))$ and consider the elliptic system (at any fixed time $t\in I$)
\begin{equation}
  \label{eq:611}
\begin{split}
  \Delta\phiu(x,y) &= 0 \  \textrm{ in }  \Omega_t\\
\phiu\vert_{y=\eta(t,x)} &= \psi(t,x)\\
\frac{\partial\phiu}{\partial y}\vert_{y=-1} &= 0 \, .
\end{split}
\end{equation}
Define new coordinates $(x,\zp)$ by $y = \eta(t,x) + \zp(1+\eta(t,x))$ so that, in the new coordinates system, the closure of
$\Omega_t$
becomes the strip 
$$
\Bcal = \big\{(x,\zp) \, ; \, x\in \Tu, -1\leq \zp\leq 0 \big\} \, . 
$$ 
We consider the function  $  \varphiu (x,\zp) $ defined for $ (x,\zp) \in \Bcal $ by 
\begin{equation}
  \label{eq:612}
  \varphiu(x,\zp) \stackrel{\textrm{def}} = \phiu(x,\eta(t,x) + \zp(1+\eta(t,x))) \, .
\end{equation}
A direct calculus shows that  $\varphiu$ 
solves the  elliptic problem
\begin{equation}
  \label{eq:613}
  \begin{split}
    (\partial_x^2+ \partial_{\zp}^2)\varphiu &= \partial_{\zp}^2[G_2 (\eta) \varphiu] + \partial_{\zp}[G_1 (\eta)  \varphiu]+ G_0 (\eta) \varphiu\\
\varphiu\vert_{\zp=0} &= \psi(t,x)\\
\partial_{\zp}\varphiu\vert_{\zp=-1} &= 0
  \end{split}
\end{equation}
with
\begin{equation}
  \label{eq:614}
  \begin{split}
    G_2 (\eta)  \varphiu&= -(1+\zp)^2\eta'{}^2\varphiu\\
G_1(\eta) \varphiu &= 2(1+\zp)\eta'(1+\eta)\partial_{x}\varphiu + (1+\zp)[2\eta'{}^2 + \eta''(1+\eta)]\varphiu\\
G_0 (\eta) \varphiu &= -(2\eta+\eta^2)\partial_x^2\varphiu -2\eta'(1+\eta)\partial_x\varphiu - \eta''(1+\eta)\varphiu \, .
  \end{split}
\end{equation}
According to Definition~\ref{533},  the operators $ G_\ell $, $\ell = 0, 1, 2$, are sums of 
elements of $\Nt{2-\ell}{1}$.  By the first remark after Definition~\ref{533} we may regard  $ G_\ell $
as operators of $ \Nr{2- \ell}{K,j,1} $ for any $  j $.
 Then we apply
Proposition~\ref{534}, with an index $ j $ that will be chosen later on 
large enough depending on $ N $ (see lemma \ref{611}) and with $\tilde{S}_N\equiv 0$,
$\tilde{M}=0$. We deduce that 
there exists $\sigma_0>0$ and, for $\sigma>\sigma_0$, some $r(\sigma)<r $, such that, if $\eta$ belongs 
to $\Brs{K}{0}{\sigma}\cap \CKH{\sigma}{\R}$ and  $\psi$ in $\CKH{\sigma}{\R}$, equation \eqref{613} has a
unique solution $  \varphiu $ in $C_*^K(I,E^{\sigma}_j)$,  that may be written, by \eqref{5317}, as 
\begin{equation}
  \label{eq:615}
  \varphiu = \frac{\cosh((\zp+1)D)}{\cosh D}\psi + M(\eta)\psi
\end{equation}
for some $M(\eta) $ in $\sNd{0}{K,j,1}{N}$.

\smallskip

\noindent 
{\bf Remark:} 
 The equation
\eqref{613} satisfied by the function $\varphiu$ is a linear ODE in $\partial_{\tilde{z}}$ with \emph{variable}
coefficients in $ \tilde z $. It is technically convenient, for the constructions to follow, to have instead a solution of  a linear
\emph{constant coefficients} ODE. 
For this reason we introduce in \eqref{619} below the function $\tilde{\Phi} $ that will satisfy the
system \eqref{6110}.

\smallskip

We consider first the function $\varphi(x,z)$ defined from $\phiu$ through a change of variables that flattens the free
boundary of $\Omega_t$, but that modifies the bottom, namely
\begin{equation}
  \label{eq:617}
  \varphi(x,z) \stackrel{\textrm{def}} = \phiu(x,z+\eta(t,x)) = \varphiu\bigl(x,\frac{z}{1+\eta(t,x)}\bigr)
\end{equation}
which is well defined for $-(1+\eta(t,x))\leq z\leq 0$. 
Notice that   $ \varphi(x,z) $ solves the elliptic equation
\begin{equation}\label{eq:equation-for-vphi}
  \biggl(\frac{\partial^2}{\partial x^2} + (1+\eta'^2)\frac{\partial^2}{\partial z^2}
  -2\eta'\frac{\partial^2}{\partial x\partial z} -\eta''\frac{\partial}{\partial z}\biggr) \varphi = 0  
\end{equation}
and satisfies the Dirichlet boundary condition at the top $ \varphi \vert_{z=0} = \psi(t,x) $. Then for $ - 1 \leq z \leq 0 $ we also consider the function 
\begin{equation}
  \label{eq:618}
  \varphi_0(x,z) = \frac{\cosh((z+1)D)}{\cosh D}\psi 
\end{equation}
which solves 
\begin{equation}
\label{eq:new-problem}
(\partial_x^2 + \partial_z^2)\varphi_0 = 0 \quad \textrm{ in } \ \Bcal 
\end{equation}
and  satisfies the Neumann boundary condition at the bottom $ \frac{\partial \varphi_0}{\partial z} \vert_{z=-1} = 0 $.  
Finally, taking a cut off function $ \chi $ in $C^\infty_0(\R)$, 
supported in $[-\frac{1}{2},\frac{1}{2}]$, equal to one on $[-\frac{1}{4},\frac{1}{4}]$,
with values in $[0,1]$, we define the function
\begin{equation} \label{eq:619}
 \tilde{\Phi}(x,z) \stackrel{\textrm{def}} = \chi(z)\varphi(x,z) + (1-\chi)(z)\varphi_0(x,z) \, .
\end{equation}
Notice that, if $\norm{\eta(t,\cdot)}_{L^\infty}<\frac{1}{2}$, the function $ \varphi(x,z) $ in \eqref{617}
is well defined for $ - 1/ 2 \leq z \leq 0 $ and therefore $  \tilde{\Phi}(x,z) $ 
is well defined for $z\in [-1,0]$. 

\begin{lemma} \label{611}
Let the index $ j $ chosen before \eqref{615} satisfy $ j \geq N + 5 $.
Then the function $\tilde{\Phi}$ satisfies the system
\begin{equation} \label{eq:6110}
\begin{split}
  \biggl(\frac{\partial^2}{\partial x^2} + (1+\eta'^2)\frac{\partial^2}{\partial z^2}
  -2\eta'\frac{\partial^2}{\partial x\partial z} -\eta''\frac{\partial}{\partial z}\biggr)\tilde{\Phi} &= \tilde{M}(\eta)\psi +
  \tilde{S}_N(\eta,\psi)\\
\tilde{\Phi}\vert_{z=0} &= \psi\\
\partial_z\tilde{\Phi}\vert_{z=-1} &= 0
\end{split}\end{equation}
where $\tilde{M}(\eta)$ is some operator of $\sNd{N}{K,4,1}{N}$ (for $N$ an arbitrary given integer) and where
$\tilde{S}_N(\eta,\psi)$ is linear in $\psi$ and, for any large
enough  $ \sigma_0$, there is $ \sigma'_0 $ such that, for any $ 0 \leq k \leq K $,  
\begin{equation}
  \label{eq:6111}
  \norm{\partial_t^k \tilde{S}_N(\eta,\psi)}_{E^{\sigma_0+2-\frac{3}{2}k}_4}  \leq C\sum_{k'+k''=k}\Gcals{\sigma'_0}{k',N}{\eta}\Gcals{\sigma'_0}{k'',1}{\psi} \, . 
\end{equation}
Moreover $\tilde{M}(\eta) \psi $ and $\tilde{S}_N(\eta,\psi)$ are supported for $-1\leq z\leq -\frac{1}{4}$.
\end{lemma}

\begin{proof}
By the definition of $ \tilde{\Phi} $ in \eqref{619} and 
\eqref{equation-for-vphi}, \eqref{new-problem}, it follows  that \eqref{6110} is equal to
  \begin{multline}
    \label{eq:6112} 
f = \big[ (1+\eta'{}^2)\partial_z^2 - 2\eta'\partial_x\partial_z - \eta''\partial_z \, , \, \chi(z) \big] (\varphi-\varphi_0)\\
+ (1-\chi(z))\bigl(\eta'^2\partial_z^2\varphi_0 -2\eta'\partial_x\partial_z\varphi_0 - \eta''\partial_z\varphi_0\bigr).
  \end{multline} 
Moreover $ f $ is supported for $-1\leq z\leq -\frac{1}{4}$ because, for $ -1/4\leq z\leq 0 $, the function
$ \tilde{\Phi} $ is equal to $ \varphi $ which is a solution of \eqref{equation-for-vphi}. 
Let us check that $  f $ may be written as the right hand side
of the first line in \eqref{6110}. Applying 
Taylor formula to the last expression for $\varphi$ given by \eqref{617}, we
have
\begin{multline}
  \label{eq:6113}
\varphi(x,z) = \sum_{p=0}^{N-1} \frac{1}{p!}\partial_z^p\varphiu(x,z)\bigl(-\frac{z\eta(x)}{1+\eta(x)}\bigr)^p\\
+\frac{1}{(N-1)!}\int_0^1(1-\lambda)^{N-1}(\partial_z^N\varphiu)\bigl(x,z-\frac{\lambda z\eta}{1+\eta}\bigr)\,d\lambda \bigl(-\frac{z\eta(x)}{1+\eta(x)}\bigr)^N.
\end{multline}
We express in the above formula $ \varphiu $ using \eqref{615}. 
The first term in \eqref{6113}, corresponding to $p=0$ in the sum, is equal to 
$$
\varphiu =  \frac{\cosh(( z+1)D)}{\cosh D}\psi + M(\eta)\psi = \varphi_0 + M(\eta)\psi 
$$
where $ M(\eta) $ is in $ \sNd{0}{K,j,1}{N} $.
By the first, third and fourth remarks after Definition~\ref{533a}, each term 
$$
\frac{1}{p!} \bigl(-\frac{z\eta(x)}{1+\eta(x)}\bigr)^p \partial_z^p \Big(\frac{\cosh((z+1)D)}{\cosh D} + M(\eta) \Big) \psi 
$$
of the sum in \eqref{6113}, for $ 1 \leq p \leq N - 1 $, may be written as 
the action of  an operator of $\sNd{N-1}{K,5,1}{N}$ on $ \psi $,  if the index $j  $ chosen before \eqref{615} is taken large enough relatively to $N$, e.g. $ j \geq N + 4 $. 

In the same way, in the Taylor remainder in \eqref{6113} 
the argument of the integral will be a path of operators in 
$\sNd{N}{K,5,1}{N} $ if $j\geq N+5$. Thus 
using \eqref{5314a} (with $ m $ replaced by $ N $ and $ \sigma $ large depending on $N$) we see that  the Taylor remainder
in \eqref{6113}, and its  $ \partial_z $ derivative,  
satisfy \eqref{6111}, for some $\sigma'_0$ depending on $N$ and $\sigma_0$. 

Collecting the previous arguments we may write 
$$
\varphi - \varphi_0 =  M(\eta)\psi +   \tilde{S}_N(\eta,\psi) 
$$
where the operator $ M(\eta) $ is in $\sNd{N-1}{K,5,1}{N} $ and the functions 
$ \tilde{S}_N(\eta,\psi) $, $ \pa_z \tilde{S}_N(\eta,\psi) $ satisfy \eqref{6111}. Clearly 
$ \tilde{S}_N(\eta,\psi) $ is linear in $ \psi $.
Consequently
the first term in the right hand side of \eqref{6112}
may be written as
$$ 
\big[ (1+\eta'{}^2)\partial_z^2 - 2\eta'\partial_x\partial_z - \eta''\partial_z \, , \, \chi(z) \big] 
(\varphi - \varphi_0) = \tilde{M}(\eta)\psi +   \tilde{S}_N(\eta,\psi) 
$$
where  $\tilde{M}(\eta)$ is in $\sNd{N}{K,4,1}{N}$  (use the third and fourth remarks after Definition~\ref{533a}
and that the commutator $[\pa_{zz}, \chi (z) ] $ is the first order operator $ (\pa_{zz} \chi) +  2 (\pa_z \chi) \pa_z  $) and 
$ \tilde{S}_N(\eta,\psi) $ satisfies \eqref{6111}, 
as claimed in the statement of the lemma.

 Finally, since the operator 
$ \frac{\cosh((z+1)D)}{\cosh D}$ is in $  \Ntd{0}{0} $ (recall Definition \ref{533a}-(i)) 
we conclude that 
the second term in the right hand side of \eqref{6112} may be written as 
$$
(1-\chi(z))\bigl(\eta'^2\partial_z^2  -2\eta'\partial_x\partial_z  - \eta''\partial_z \bigr) 
 \frac{\cosh(( z+1)D)}{\cosh D}\psi = {\tilde M}(\eta) \psi 
$$
where $ {\tilde M}(\eta) $  is the sum of homogeneous operators in $  \Ntd{2}{1} $. 
\end{proof}
The good unknown that allows to express the capillarity-gravity water waves equations without artificial loss of 
derivatives  is defined
as
\begin{equation}
  \label{eq:6114}
  \omega \stackrel{\textrm{def}} = \Phi\vert_{z=0} = \psi -\opbw(B)\eta
\end{equation}
where
\begin{equation}
  \label{eq:6115}
  \Phi \stackrel{\textrm{def}} = \tilde{\Phi} -\opbw(\partial_z\tilde{\Phi})\eta \, , \quad B = \partial_z\tilde{\Phi}\vert_{z=0} \, ,
\end{equation}
$\tilde{\Phi}$ being the function \eqref{619}. To obtain an expression for $\omega$ in terms of $(\eta,\psi)$ alone, we need
to deduce from \eqref{6110} a paradifferential equation satisfied by $ \Phi $ (see Lemma \ref{613}). 

We shall use the following paralinearization lemma for the product of $  p $ functions (in the case $ p = 2 $ it gives  the usual para-product decomposition formula). 
\begin{lemma}
  \label{612} {\bf (Para-product)}
Let $u_1,\dots,u_p$ be $p$ elements in $\Hsz{\sigma}(\Tu,\C)$, with $\sigma>\frac{1}{2}$. Then we may write
\begin{equation}
  \label{eq:6116}
  u_1\cdots u_p = \sum_{j=1}^p\bigl(\opbw(u_1\cdots\hat{u}_j\cdots u_p)u_j+ R_j(u_1\cdots\hat{u}_j\cdots u_p)u_j\bigr)
\end{equation}
where $R_j$ is in $\Rt{-\rho}{p-1}$ for any $\rho$, and where $\hat{u}_j$ means that the $j$-th term is excluded.
\end{lemma}

\begin{proof}
  Let $\chi$ be an even function in $C^\infty_0(\R)$, $\chi\equiv 1$ close to zero, with small enough support. 
  We decompose the product of $  p $ functions as 
   \begin{multline}
    \label{eq:6117}
 u_1\cdots u_p = \sum_{j=1}^p \sum_{n_1,\dots,n_p\in \N^*}
 \prod_{i\neq  j}\chi\Bigl(\frac{n_i}{n_j}\Bigr)\prod_{\ell=1}^p\Pin{\ell}u_\ell\\
+ \sum_{n_1,\dots,n_p}\theta(n_1,\dots,n_p) \prod_{\ell=1}^p\Pin{\ell}u_\ell
  \end{multline}
where 
$$
\theta (n_1,\dots,n_p) = 1 - \sum_{j=1}^p \prod_{i\neq  j}\chi\Bigl(\frac{n_i}{n_j}\Bigr)
$$ 
is supported for $\maxn{1}{p}\sim\maxdn{1}{p}$. 
As a consequence, recalling Definition \ref{214}-(i), 
the last term in \eqref{6117} contributes to the smoothing operators 
in $ \Rt{-\rho}{p-1} $ in the
right hand side of \eqref{6116}. 

All  the terms in the sum \eqref{6117} are the same, up to permutation of the indices. 
The  one corresponding  to $ j = p $,  
$$
 \prod_{i\neq  p}\chi\Bigl(\frac{n_i}{n_p}\Bigr)\prod_{\ell=1}^p\Pin{\ell}u_\ell \, , 
$$
may be written,  using \eqref{usual-Q} and 
\eqref{quantiz-t} for $ \tau = 1 $, as
\begin{equation}
  \label{eq:6118}
  \frac{1}{2\pi}\int_{\R\times\R}e^{i(x-y)\xi} a_{\chi_{p-1}}(u_1, \ldots, u_{p-1}; x,\xi)u_p(y)\,dyd\xi
\end{equation}
where 
\be\label{eq:an'}
a_{\chi_{p-1}} (u_1, \ldots, u_{p-1}; x,\xi) = 
\sum_{n_1,\dots,n_{p-1}} \prod_{\ell=1}^{p-1}\chi\bigl(\frac{n_\ell}{\xi}\bigr)\Pin{\ell}u_\ell(x)  
\ee
has the form \eqref{2112}  (with $ p - 1 $ instead of $ p$), admissible cut-off function 
 $ \chi_{p-1}(n_1, \ldots, n_{p-1}; \xi) = \prod_{\ell=1}^{p-1}\chi\bigl(\frac{n_\ell}{\xi}\bigr)  $ 
 and $ a  = \prod_{\ell=1}^{p-1} u_\ell $. 
Notice that $ a_{\chi_{p-1}} $
is  an homogeneous symbol in $\Gt{0}{p-1}$  supported for $ | n' | \absj{\xi}^{-1}  $ small
where  we denote $ n'  = (n_1,\dots,n_{p-1}) $. 

We apply Lemma \ref{248} to the symbol $ a_{\chi_{p-1}} \in \Gt{0}{p-1} $. 
By \eqref{2437}, \eqref{2438}, \eqref{def-e}, we may write
\eqref{6118}, in terms of the Weyl quantization, as
\begin{equation}
  \label{eq:6119}
  \opw(b)u_p =   \sum_{\alpha=0}^{A-1}\opw(b_\alpha)u_p + \opw(\tilde{b})u_p
\end{equation}
where
\begin{equation}
  \label{eq:6120}
  b_\alpha (u_1, \ldots, u_{p-1}; x, \xi)  =  \frac{(-1)^\alpha}{\alpha! 2^\alpha} 
  \partial_x^\alpha D_\xi^\alpha a_{\chi_{p-1}}(u_1, \ldots, u_{p-1}; x,\xi)
\end{equation}
and where $\tilde{b}$ satisfies estimates of the form \eqref{2434} with $ m = 0 $ 
and the prefactor $ Q $  is 
$O\bigl(\bigl(\frac{\abs{n'}}{\absj{\xi}}\bigr)^A\bigr)$ when $\frac{\abs{n'}}{\absj{\xi}}$ is small. 

By \eqref{6120} and \eqref{an'}, and recalling the Definition \ref{Def:BW} 
of a Bony-Weyl
operator,  the first term in the sum \eqref{6119} is 
 $$
 \opw(b_0)u_p =  \opw(a_{\chi_{p-1}})u_p = \opbw(u_1\cdots u_{p-1})u_p
 $$ 
see \eqref{2113}.

 When $ 1 \leq \alpha \leq A - 1$, it follows from \eqref{6120} and the
  definition of   $ a_{\chi_{p-1}} $ in \eqref{an'} that $ b_\alpha $ is supported for $\abs{n'}\sim \absj{\xi} $. 
  As a consequence   each operator $ \opw(b_\alpha) $ 
  satisfies the condition \eqref{2115} that defines an homogeneous 
  smoothing operator of $  \Rt{-\rho}{p-1} $ (recall formula \eqref{restriz}). 
In conclusion each $ \opw(b_\alpha) u_p $, $ 1 \leq \alpha \leq A - 1$,  
  in \eqref{6119} has the form $R_p(u_1,\dots,u_{p-1})u_{p}$ 
  in the right hand side of \eqref{6116}. 

Finally we consider $ \opw(\tilde{b})u_p $. 
Since $ Q $  is 
$O\bigl(\bigl(\frac{\abs{n'}}{\absj{\xi}}\bigr)^A\bigr)$ when $\frac{\abs{n'}}{\absj{\xi}}$ is small, 
the function $\tilde{b}$ satisfies by \eqref{2434} bounds with an $O(\absj{\xi}^{-A})$ gain, as soon as we assume enough smoothness on   $u_1,\dots,u_{p-1}$ to 
  compensate the loss $\abs{n'}^A$. 
  Moreover 
    $ \tilde{b} $ is supported for $ | n' | \absj{\xi}^{-1} $ small as 
  the symbol $ a_{\chi_{p-1}} $ in \eqref{an'}. This follows from \eqref{clwe} for 
  $b$, and then  
  for $\tilde{b} = b - \sum_{\alpha=1}^{A-1} b_\alpha $, 
  since each $ b_\alpha $ is supported for $ c_1 \absj{\xi}\leq \abs{n'} \leq
    c_2\absj{\xi}$ for some constants $ 0 < c_1 < c_2 \ll 1$.
   Taking $A \sim \rho $, we deduce that 
  also $\opw(\tilde{b}) $ is a smoothing operator in  $ \Rt{-\rho}{p-1} $ and thus it contributes to the
last term in \eqref{6116} (with $j=p$). Actually, 
by the first remark after  Proposition~\ref{215} 
and \eqref{2434} with $ \alpha = \beta = 0 $, and $ Q =  O\bigl(\bigl(\frac{\abs{n'}}{\absj{\xi}}\bigr)^{\rho}\bigr)$, 
 we may apply  the statement of Proposition~\ref{215} with $ m = - \rho $
 (even tough the estimates \eqref{2434} do not imply that $ \tilde{b} $ is a symbol, since  
each  $\partial_\xi^\beta$ derivative of $\tilde{b}$  involves a loss in $\abs{n}^\beta$).
Then the last remark after Proposition~\ref{215} implies that $\opw(\tilde{b}) $ is in $ \Rt{-\rho}{p-1} $. 
This concludes the proof.
\end{proof}

\textbf{Remark}: When we shall apply below  lemma~\ref{612}, we shall replace some arguments by $\eta$ (or one of its derivatives), which belongs to a
space of the form $\Hsz{s}$. The
remaining arguments will be replaced by either a space derivative of $\psi\in \Hds{s}$, or by $\partial_z\Phit$,
$\partial^2_z\Phit$, $B$,  that may be  considered as 
elements of $\Hsz{s}$ (at fixed $z$), even though the function $\Phit$ is defined only modulo constants like $\psi$, since one makes always
act  at least one derivative on it. This will be used without further notice in the proof of lemma~\ref{613} below, as
well as in formulas \eqref{6145}, \eqref{627}, \eqref{628a}, \eqref{629a}.

\begin{lemma}
  \label{613}
Let $\Phit$ be a smooth enough solution of \eqref{6110}. Then the  function $\Phi$ defined by \eqref{6115} 
satisfies the paradifferential equation 
\begin{multline}
\label{eq:6121}
  \bigl[(1+\opbw(\eta'\otimes\eta'))\partial_z^2 -2i\opbw(\eta'\xi)\partial_z -\opbw(\xi^2)\bigr]\Phi\\
=-\opbw[\partial_z\tilde{M}(\eta)\psi + \partial_z\tilde{S}_N(\eta,\psi)]\eta\\
+\sum_{j=1}^2(R'_j(\eta)\partial_z^j\Phit + R''_j(\eta,\partial_z^j\Phit)\eta) + \tilde{M}(\eta)\psi + \tilde{S}_N(\eta,\psi)
\end{multline}
where $\tilde{M}(\eta)$ 
is in $\sNd{N}{K,4,1}{N}$, $\tilde{S}_N$ is linear in $\psi $,  and such that,  for any large
enough  $ \sigma_0 $, there is $ \sigma'_0 $ so that, for any $ 0\leq k\leq K $,   
\begin{equation} \label{eq:6111bis}
  \norm{\partial_t^k \tilde{S}_N(\eta,\psi)}_{E^{\sigma_0+2-\frac{3}{2}k}_4}  \leq C\sum_{k'+k''=k}\Gcals{\sigma'_0}{k',N}{\eta}\Gcals{\sigma'_0}{k'',1}{\psi}, 
\end{equation}
and where $R'_j(\eta)$ (resp.\
$R''_j(\eta,\partial_z^j\Phit)$) is a sum of smoothing operators of $\Rtb{-\rho}{p}$, $p=1, 2$, $\rho$ being an 
arbitrary positive number. Moreover $R''_j (\eta,\partial_z^j\Phit) $ is linear in $\partial_z^j\Phit$
and 
$R'_j(\eta)$, $R''_j(\eta,\partial_z^j\Phit)$ are 
local relatively to the $z$ variable. Finally, $\tilde{M}(\eta) \psi $ and $\tilde{S}_N(\eta,\psi)$ are supported for $-1\leq z\leq -\frac{1}{4}$.
\end{lemma}

\textbf{Remarks}: 
$ \bullet $
 We recall that a local operator is an operator that does not increase the support of functions. Here the expression 
  ``local relatively to $ z$'' means that the $z$-support of functions does not increase. 

$ \bullet $
As stated in lemma \ref{611}
 the terms $\tilde{M}(\eta)\psi$, $\tilde{S}_N(\eta,\psi)$, that come from the
cut-off in \eqref{619},
are supported for $ - 1 \leq z \leq -\frac{1}{4} $. 
Because of that, they will induce only smoothing contributions to the restriction of
$\Phi$ at $z=0$, see Proposition \ref{614}.

\begin{proof}
We  apply the paralinearization lemma~\ref{612} to each term
in the left hand side of equation \eqref{6110} at fixed $z$. We may write
  \begin{multline}
    \label{eq:6123}
\eta'{}^2\partial_z^2\Phit = \opbw(\eta'\otimes\eta')\partial_z^2\Phit + 2\opbw(\eta'\otimes\partial_z^2\Phit)\eta'\\
+ R(\eta',\eta')\partial_z^2\Phit + R(\eta',\partial_z^2\Phit)\eta'
  \end{multline}
and
  \begin{multline}
\label{eq:6124}
    2\eta'\partial_x\partial_z\Phit + \eta''\partial_z\Phit = \opbw(2\eta')\partial_x\partial_z\Phit +
    \opbw(\eta'')\partial_z\Phit\\
+2\opbw(\partial_x\partial_z\Phit)\eta' + \opbw(\partial_z\Phit)\eta''\\
+ R(\eta') \partial_x\partial_z\Phit + R(\partial_x\partial_z\Phit)\eta'\\
+R(\eta'') \partial_z\Phit + R(\partial_z\Phit)\eta''
  \end{multline}
where, using the second remark after Definition \ref{512},  $R$ stands for a generic smoothing operator of $\Rtb{-\rho}{p}$ as defined in (i) of Definition~\ref{512}, with $p=1$ when $R$
has only one argument, and $p=2$ if $R$ has two arguments.  By the composition formula \eqref{224} of
Proposition~\ref{222}, the sum of the first two terms in the right hand side of \eqref{6124} is
\be\label{eq:sum-first2}
\opbw(2\eta')\partial_x\partial_z\Phit +
    \opbw(\eta'')\partial_z\Phit = \opbw(2i\eta'\xi)\partial_z\Phit
\ee
up to a smoothing remainder $ R( \eta ) \pa_z \Phit $ with $ R(\eta) $ in $ \Rt{-\rho}{1} $, 
thus in $\Rtb{-\rho}{1}$  by the second remark after Definition \ref{512}.
By \eqref{6123}, \eqref{6124}, \eqref{sum-first2} we may then write \eqref{6110} as
\begin{multline}
  \label{eq:6125}
 \bigl[(1+\opbw(\eta'\otimes\eta'))\partial_z^2 -2i\opbw(\eta'\xi)\partial_z -\opbw(\xi^2)\bigr]\Phit\\
=-2\opbw(\eta'\otimes\partial_z^2\Phit)\eta' + 2\opbw(\partial_x\partial_z\Phit)\eta'\\
+\opbw(\partial_z\Phit)\eta'' + \sum_{j=1}^2R'_j(\eta)\partial_z^j\Phit + \sum_{j=1}^2R''_j(\eta,\partial_z^j\Phit)\eta\\
+\tilde{M}(\eta)\psi + \tilde{S}_N(\eta,\psi) 
\end{multline}
where $R'_j(\eta)$ (resp.\ $R''_j(\eta,\partial_z^j\Phit)$) is a sum of elements of $\Rtb{-\rho}{p}$, $p=1, 2$ (resp.\
which is moreover linear in $\partial_z^j\Phit$). These operators are local in $ z $ as we use only symbolic calculus
relatively to the $x$ variable, $z$ being just a parameter.
Moreover,  as stated in lemma \ref{611}, the operator 
$ \tilde{M}(\eta) $ 
is in $\sNd{N}{K,4,1}{N}$  and $ \tilde{S}_N(\eta,\psi)  $
satisfy \eqref{6111}, i.e. \eqref{6111bis}. 

On the other hand, we compute
\[\bigl[(1+\opbw(\eta'\otimes\eta'))\partial_z^2 -2i\opbw(\eta'\xi)\partial_z
-\opbw(\xi^2)\bigr]\opbw(\partial_z\Phit)\eta \, .
\]
Using again the composition formula \eqref{224}, we get
\begin{multline}
  \label{eq:6126}
\opbw((1+(\eta'\otimes\eta'))\partial_z^3\Phit)\eta\\
-2\opbw(\eta'\partial_z^2\Phit)\eta' -2\opbw(\eta'\partial_z^2\partial_x\Phit)\eta - \opbw(\eta''\partial_z^2\Phit)\eta\\
+\opbw(\partial_z\Phit)\eta'' + 2\opbw(\partial_z\partial_x\Phit)\eta' + \opbw(\partial_z\partial_x^2\Phit)\eta
\end{multline}
modulo again smoothing operators as in \eqref{6125} and \eqref{6126}. 

Making the difference between \eqref{6125} and  \eqref{6126}, and using the definition \eqref{6115} of $\Phi$, we get
\begin{multline*}
  \bigl[(1+\opbw(\eta'\otimes\eta'))\partial_z^2 -2i\opbw(\eta'\xi)\partial_z
-\opbw(\xi^2)\bigr]\Phi\\
= -\opbw\bigl[\bigl((1+(\eta'\otimes\eta'))\partial_z^2 -2\eta'\partial_z\partial_x + \partial_x^2
-\eta''\partial_z\bigr)\partial_z\Phit\bigr]\eta\\
+ \sum_{j=1}^2R'_j(\eta)\partial_z^j\Phit + \sum_{j=1}^2R''_j(\eta,\partial_z^j\Phit)\eta\\
+\tilde{M}(\eta)\psi + \tilde{S}_N(\eta,\psi)\\
= -\opbw\bigl[\partial_z \tilde{M}(\eta)\psi + \partial_z \tilde{S}_N(\eta,\psi)\bigr]\eta\\
+ \sum_{j=1}^2R'_j(\eta)\partial_z^j\Phit + \sum_{j=1}^2R''_j(\eta,\partial_z^j\Phit)\eta\\
+\tilde{M}(\eta)\psi + \tilde{S}_N(\eta,\psi)
\end{multline*}
where the last equality follows from substitution of \eqref{6110} in the argument of $\opbw$. This concludes
the proof of \eqref{6121}.
\end{proof}
We may deduce from the preceding results the following expression of the 
good unknown $\Phi\vert_{z=0} $  and its $\partial_z \Phi\vert_{z=0} $ derivative. 
\begin{proposition}
  \label{614} {\bf (Good unknown)} 
Let $N$ be in $\N^*$, $K\in \N$, $\rho>0$ be given. 
There is $ r > 0 $ and a function $a_0(\eta,\psi)$ in $\sFRa{K,0,1}{N}$, linear in
$\psi$, such that $B = \partial_z\tilde{\Phi}\vert_{z=0} $ defined in \eqref{6115} is actually given 
by 
\be\label{eq:Ba0}
B (\eta, \psi ) = a_0 (\eta,\psi) \, . 
\ee
Moreover there is a symbol $a_1$ in $\sGa{1}{K,0,1}{N}$ such that, setting 
\begin{equation}
  \label{eq:6127}
  \omega = \psi - \opbw(a_0(\eta,\psi))\eta \, , 
\end{equation}
one has
\begin{equation}
  \label{eq:6128}
  \partial_z\Phi\vert_{z=0} = (D\tanh D)\omega + \opbw(a_1(\eta;\cdot))\omega + R_1(\eta)\omega + R_2(\eta,\omega)\eta
\end{equation}
where $R_1(\eta), R_2(\eta,\omega)$ are in $\sRa{-\rho}{K,0,1}{N}$, $R_2$ being linear in $\omega$.
The symbol $a_1$ is given by 
\begin{equation}
  \label{eq:6130}
  a_1(\eta;x,\xi) = \frac{i\eta'}{1+\eta'{}^2}\xi - (\xi\tanh\xi)\frac{\eta'{}^2}{1+\eta'{}^2} 
\end{equation}
modulo a symbol in $\sGa{0}{K,0,1}{N}$. 

Finally, there is $ \sigma > 0 $ such that, 
for $ \nnorm{\eta}_{K,\sigma} $ small enough,  there is a function $ a(\eta,\omega)$   
in  $\sFRa{K,0,1}{N}$, linear in $\omega$, even in $ x $, such
that 
\begin{equation}
  \label{eq:6131}
  \psi = \omega + \opbw(a(\eta,\omega))\eta \, .
\end{equation}
In the new variables $ (\eta, \omega ) $ the involution $ S = \sm{1}{0}{0}{-1}  $ is unchanged.
\end{proposition}

\begin{proof}
  We have introduced in \eqref{619} the function $\Phit$ and we have proved in Lemma \ref{611}
 that $\Phit$ satisfies equation \eqref{6110}. Then we defined in \eqref{6115}
  $\Phi = \Phit -\opbw(\partial_z\Phit)\eta$ and in \eqref{6114} the good unknown $\omega = \psi-\opbw(B)\eta$ with
  $B= \partial_z\Phit\vert_{z=0}$. Let us show first  that $B$ may be written as a function 
  $a_0(\eta,\psi)$ 
  of $\sFRa{K,0,1}{N}$ as in the statement. To do so, 
  we notice that equation \eqref{6110} solved by $\tilde{\Phi}$ is of the form
  \eqref{5316} with 
 $$
  G_2(\eta) = -\eta'{}^2 \, , \quad G_1(\eta) = 2\eta'\partial_x +\eta'' \, , \quad   G_0 (\eta) = 0 \, , 
  $$ 
  and $\tilde{M}(\eta), \tilde{S}_N $  given by lemma \ref{611}. Since 
 $ \tilde{M}(\eta)$  is in $\sNd{N}{K,4,1}{N}$ and $ \tilde{S}_N $ 
 satisfies \eqref{6111}, we apply Proposition~\ref{534} with $m=N$, $j=4$. By \eqref{5317}, we may write
  \begin{equation}
    \label{eq:6132}
    \Phit = \frac{\cosh((1+z)D)}{\cosh D}\psi + M(\eta)\psi + S_N(\eta,\psi)
  \end{equation}
where $M(\eta)$ is in $\sNd{N-2}{K,4,1}{N}$ and $S_N$ is supported for $ -1 \leq z \leq -\frac{1}{8} $ and 
satisfies, for $ 0 \leq k \leq K $ and some $\sigma'_0$, 
\begin{equation}
  \label{eq:6132a}
  \norm{\partial_t^kS_N(\eta,\psi)(t)}_{E^{\sigma_0+2-\frac{3}{2}k}_4} \leq C\sum_{k'+k''=k}\Gcals{\sigma'_0}{k',N}{\eta}\Gcals{\sigma'_0}{k'',1}{\psi} \, . 
\end{equation}
In particular, close to $z=0$, 
\begin{equation}
  \label{eq:6133}
  \partial_z\Phit = \frac{\sinh((1+z)D)}{\cosh D}(D\psi) + \partial_zM(\eta)\psi 
\end{equation}
and $\partial_zM(\eta)$ is in $\sNd{N-1}{K,3,1}{N}$, by
the third remark after Definition~\ref{533a}. Inserting 
\eqref{6133} in \eqref{6115} and using the fifth remark after
Definition~\ref{533a}, we see that, for $z$ close to zero, we may write
\[
\Phi(z,\cdot) = \Phit(z,\cdot) - \opbw(a(\eta,\psi;z, x))\eta
\]
for some symbol $a$ in the class $\sGb{0}{K,0,3,1}{N}$, 
independent of $\xi$. Restricting to $z=0$, we get,
by the definition of $\omega $ in  \eqref{6114},
\[
\omega = \psi - \opbw(a_0(\eta,\psi;x))\eta
\]
where $a_0$ is in $\sGa{0}{K,0,1}{N}$, and it is independent of $\xi$, i.e. it is  a function 
of $\sFa{K,0,1}{N}$.  Moreover, by construction, 
$a_0 (\eta, \psi ) $  is real valued when $\eta, \psi$
are real valued. We have thus
obtained the representation \eqref{6127} of $\omega$.

Let us now establish \eqref{6128}. 
  The strategy will be to prove first 
 that $ \Phi $ satisfies a paradifferential  equation of the form
\eqref{531}, see \eqref{6136a},  then use Proposition~\ref{531} to express $\Phi$ from the data, see \eqref{6137}, and finally compute
$\partial_z\Phi\vert_{z=0}$ from \eqref{532}, see \eqref{6139a}. 

We recall that $\Phi$ solves the equation \eqref{6121}. In its right hand
side we replace in the terms $R'_j(\eta)\partial_z^j\Phit$, $R''_j(\eta,\partial_z^j\Phit)\eta$ the function $\Phit$ by its
expression \eqref{6132}. Recalling that $ R''_j(\eta,\partial_z^j\Phit)  $ is linear in $ \partial_z^j\Phit $,
we get the following  terms, for $ j = 1, 2 $, 
\begin{equation} \label{eq:6134}
\begin{split}
R'_j(\eta)\partial_z^j\Bigl(\frac{\cosh((1+z)D)}{\cosh D}\psi + M(\eta)\psi\Bigr)\\
R''_j\Bigl(\eta,\partial_z^j\Bigl(\frac{\cosh((1+z)D)}{\cosh D}\psi + M(\eta)\psi\Bigr)\Bigr)\eta\\
R'_j(\eta)\partial_z^j(S_N(\eta,\psi))\\
R''_j(\eta,\partial_z^j(S_N(\eta,\psi)))\eta \, .
\end{split}
\end{equation}
By lemma \ref{613} the smoothing terms $ R'_j, R''_j$ 
belong to $\sRb{-\rho}{K,0,\ell,1}{N}$ for any $\rho$ and any $\ell$,
according to (iii) of Definition~\ref{512}. Moreover, 
it has been seen after \eqref{6132} that 
$ M(\eta) $ is in $\sNd{N-2}{K,4,1}{N} $, and, by the first remark after 
Definition~\ref{533a}, that $ \frac{\cosh((z+1)D)}{\cosh D}$ is in $  \Ntd{0}{0} $. Thus, 
the third remark following Definition~\ref{533a} implies that the operators 
\[
\partial^j_z\Bigl[\frac{\cosh((1+z)D)}{\cosh D} + M(\eta) \Bigr] \, , \quad j = 1, 2 \, , 
\]
are in $\sNd{N}{K,2,0}{N}$. It follows from the third-last and second  last
 remarks following
Definition~\ref{533a} that the first two lines in \eqref{6134} may be written as
\begin{equation}\label{eq:6135}
  R'(\eta)\psi + R''(\eta,\psi)\eta
\end{equation}
where $ R' $ and $ R'' $ are smoothing operators in $\sRb{-\rho}{K,0,2,1}{N}$, for any  
$ \rho $ (up to a renaming of $ \rho $, 
and using the abuse of notation that has been introduced in the last remark after  Definition~\ref{512}). 

The functions in the last two lines of \eqref{6134} are supported
for $-1\leq z\leq -\frac{1}{8}$ as this is true for $ S_N $, and 
the operators $ R'_j $, $ R''_j $ are local in $ z $. Moreover, 
by \eqref{519}, \eqref{6132a} and the continuous embeddings
$ E^s_j \subset F^s_j \subset E^{s- \frac12}_j $,  these terms may be written as 
a function $\tilde{S}^1_N(\eta,\psi)$ that  satisfies, for any large
enough $\sigma_0$, and a convenient $\sigma'_0>\sigma_0$, for 
any $ 0\leq k\leq K $, 
\begin{equation} \label{eq:6136}
\norm{\partial^k_t \tilde{S}^1_N(\eta,\psi)}_{E^{\sigma_0+2-\frac{3}{2}k}_2} \leq C\sum_{k'+k''=k}\Gcals{\sigma'_0}{k',N}{\eta}\Gcals{\sigma'_0}{k'',1}{\psi} \, . 
\end{equation}
We now consider 
the first term $ \opbw{[\partial_z\tilde{M}(\eta)\psi]} \eta $ in the right hand side of \eqref{6121}.  Since 
$\tilde{M}(\eta)$ is in $\sNd{N}{K,4,1}{N}$ then $ \pa_z \tilde{M}(\eta)$ is in $\sNd{N+1}{K,3,1}{N}$
by the third remark following Definition~\ref{533a}, and the last remark after Definition~\ref{533a} implies that 
 $ \opbw{(\partial_z\tilde{M}(\eta) \psi )} $  is in $\sNd{1/2}{K,0,1}{N}$.
In conclusion $ \opbw{[\partial_z\tilde{M}(\eta)\psi]} $
is an operator of the form $ M''(\eta,\psi) $  of $\sNd{1/2}{K,0,1}{N}$, 
linear in $\psi$, supported for $-1\leq z\leq -\frac{1}{8}$, 
as  $\partial_z \tilde M(\eta)\psi$
satisfies such a property.

By  \eqref{6111bis}, 
the  function 
$ \tilde{S}^{(2)}_N(\eta,\psi) = - \opbw[ \partial_z\tilde{S}_N(\eta,\psi)]\eta $ 
in the right hand side of \eqref{6121} satisfies estimates of the form
\begin{equation}
  \label{eq:6111ter}
  \norm{\partial_t^k \tilde{S}^{(2)}_N(\eta,\psi)}_{E^{\sigma_0+2-\frac{3}{2}k}_3}  \leq 
  C \sum_{k'+k''=k}\Gcals{\sigma'_0}{k',N}{\eta}\Gcals{\sigma'_0}{k'',1}{\psi}, 
\end{equation}
(for some $ \sigma_0 $ large depending on $ m $, $ K $ and $ N $) 
and it is supported  for $-1\leq z\leq -\frac{1}{8}$. 

In conclusion, we have written the right hand side of \eqref{6121} as
\begin{equation}\label{eq:6136a}
f = M'(\eta)\psi + M''(\eta,\psi)\eta + \tilde{S}_N(\eta,\psi) + R'(\eta)\psi + R''(\eta,\psi)\eta
\end{equation}
where $ M', M'' $ are in $\sNd{m}{K,0,1}{N}$ for some $ m $, $ M' (\eta ) \psi $ and $ M''(\eta,\psi) $
are supported in $ -1 \leq z \leq - \frac{1}{8}$, where $ \tilde{S}_N $  is  a new function, 
which is the sum of $ \tilde{S}_N $ in the right hand side of \eqref{6121} (that satisfies \eqref{6111bis}) and 
$ \tilde{S}^1_N, \tilde{S}^{(2)}_N $ in \eqref{6136}, \eqref{6111ter}. This  $ \tilde{S}_N $ satisfies \eqref{6136}
 and is supported for $ -1 \leq z \leq -
\frac{1}{8}$. Finally,
the smoothing operators $ R' $, $ R'' $ in $\sRb{-\rho}{K,0,2,1}{N}$ are 
given by \eqref{6135}. 

Consequently, by \eqref{6121} and \eqref{6136a}, the function $ \Phi $
solves an equation of the form \eqref{531}, with $f$ given by \eqref{6136a}, and 
$$
g_- = \opbw\bigl(\frac{D^2}{\cosh D}\psi\bigr)\eta \, , \quad 
g_+ = \Phi\vert_{z=0} = \omega \, . 
$$
By \eqref{532}, we may represent the solution  $ \Phi $ as
\begin{multline}
  \label{eq:6137}
\Phi = \opbw(e_+(\eta;z,\cdot))\omega +  \opbw(e_-(\eta;z,\cdot)) \opbw\bigl(\frac{D^2}{\cosh D}\psi\bigr)\eta  \\ 
+ \int_{-1}^0\bigl[\opbw(K(\eta;z,\zp,\cdot)) +
R_{\mathrm{int}}(\eta;z,\zp)\bigr]f(\zp)\,d\zp \\
+ R_+(\eta;z)\omega + R_-(\eta;z) \opbw\bigl(\frac{D^2}{\cosh D}\psi\bigr)\eta
\end{multline}
with Poisson symbols $e_\pm (\eta; \cdot )$ in $\sP{0,\pm}{K,0,0}{N}$,  $K (\eta; \cdot) $ in $\sPi{-1}{K,0,0}{N}$ and 
smoothing operators $R_{\mathrm{int}}$ in  $\sRi{-\rho-1}{K,0,1}{N}$ and $R_\pm$ in $\sR{-\rho,\pm}{K,0,1}{N}$.

We now substitute the expression of $ f $ in \eqref{6136a} into \eqref{6137}. 
We consider first the contribution of the last two terms $ R'(\eta)\psi + R''(\eta,\psi)\eta $
 to \eqref{6137}. By lemma~\ref{519}, these two
contributions may be written as
\begin{equation}
  \label{eq:6138}
  R'_1(\eta)\psi + R''_1(\eta,\psi)\eta
\end{equation}
where   $\partial^j_z R'_1 $, $\partial^j_z R_1'' $, $j=1,2$, 
are smoothing operators in $\sRa{-\rho}{K,0,1}{N} $ for any given $\rho$, uniformly in $ z \in [-1,0] $  
(thus in particular $\partial_z R'_1 $ and  $\partial_z R_1'' $
are continuous in $ z $ and 
their restriction at $ z = 0 $ is well defined). 
Moreover $R''_1$ is linear in $\psi$. 
Consider next
\begin{multline}
  \label{eq:6139}
\int_{-1}^0\bigl[\opbw(K(\eta;z,\zp,\cdot)) + R_{\mathrm{int}}(\eta;z,\zp)\bigr]\bigl[M'(\eta)\psi 
+ M''(\eta,\psi)\eta \bigr]\,d\zp \, .
\end{multline}
The operators $ M' (\eta), M'' (\eta, \psi) \in \sNd{m}{K,0,1}{N}$ 
satisfy the assumptions of lemma~\ref{5110}, in particular the support condition \eqref{5128}. 
Consequently,  \eqref{6139} can be written as $ R'_1(\eta)\psi + R''_1(\eta,\psi)\eta $
where, for $ z $ close to zero, the derivatives $\partial_z^j R_1' $, $\partial_z^j R_1'' $  
are smoothing operators
of $\sRa{-\rho}{K,0,1}{N}$, uniformly depending on $z$ close enough to zero. 
Finally
\begin{equation}\label{eq:6139aa}
\int_{-1}^0\bigl[\opbw(K(\eta;z,\zp,\cdot)) + R_{\mathrm{int}}(\eta;z,\zp)\bigr] \tilde{S}_N(\eta,\psi) \,d\zp 
\end{equation}
where $ \tilde{S}_N(\eta,\psi)$ is supported for $ - 1 \leq z \leq - 1/ 8  $. 
This property, together with  estimate \eqref{6136} (with $ \sigma_0 $ large depending on $ K $ and $ m $),
shows that the assumptions \eqref{5127}, \eqref{5128} of 
lemma~\ref{5110} hold for  
$ \tilde{S}_N (\eta, \psi ) $ (notice that $ \tilde{S}_N  $ is linear in $ \psi $ and vanishes of order $ N $ in $ \eta $), so that,
applying again lemma \ref{5110}, we conclude that \eqref{6139aa} has the same structure as \eqref{6139}
(actually just $ R'_1(\eta)\psi $).

Taking a $\partial_z$ derivative of \eqref{6137},
and restricting to $z=0$, we conclude, using \eqref{534}, that  $ \partial_z\Phi\vert_{z=0} $ is the sum of 
\begin{multline}
  \label{eq:6139a}
\partial_z\Phi\vert_{z=0} = (D\tanh D)\omega + \opbw(a_1(\eta;\cdot))\omega\\
+ \tilde{R}_1(\eta)\psi + \tilde{R}_2(\eta,\psi)\eta
\end{multline}
where $a_1$ is given by \eqref{6130}, modulo a symbol in $\sGa{0}{K,0,1}{N}$,  and $ \tilde{R}_1, \tilde{R}_2$ are smoothing operators in 
$\sRa{-\rho}{K,0,1}{N}$, $\tilde{R}_2$ being linear in $\psi$, 
plus
\begin{multline}\label{eq:altri-resti}
\opbw( \partial_z e_-(\eta;z,\cdot)_{\vert {z = 0}} ) \opbw\bigl(\frac{D^2}{\cosh D}\psi\bigr) \eta +
\big( \pa_z R_+(\eta;z) \big)_{\vert {z = 0}} \omega  \\
+ \big( \pa_z R_-(\eta;z) \big)_{\vert {z = 0}} \opbw\bigl(\frac{D^2}{\cosh D}\psi\bigr) \eta 
\, . 
\end{multline}
Since $R_+ $ is in $\sR{-\rho,+}{K,0,1}{N}$, then 
$ (\pa_z R_+)_{\vert z = 0 } $ is in $\sR{-\rho+1}{K,0,1}{N}$  
 for any $ \rho $ (see the first two remarks after Definition \ref{514})  
and so it contributes to the smoothing operator $ R_1 (\eta) $ in \eqref{6128}. 
Moreover, since $ e_- $  is  a Poisson symbol in $ \sP{0,- }{K,0,0}{N}$, the derivative
$ \partial_z e_-(\eta;z,\cdot)_{\vert {z = 0}} $ is in $ \sGa{-\rho}{K,0,0}{N} $ for any $ \rho $, and 
the third remark after Proposition \ref{215} implies  
that the operator  $ \opbw{(  (\partial_z e_-)_{\vert z = 0 })}$ is in $\sRa{-\rho}{K,0,0}{N} $. 
Therefore by Proposition \ref{232}, 
the first term in \eqref{altri-resti} contributes to $  \tilde{R}_2(\eta,\psi)\eta $
in  \eqref{6139a}. The same conclusion holds for the last  term in \eqref{altri-resti}.  

To finish the proof of \eqref{6128}, it is enough to
prove \eqref{6131}. Actually  \eqref{6128} is a consequence of  \eqref{6139a}, \eqref{6131}  and 
the composition results of (ii), (iii) and (iv) of Proposition~\ref{233}. 

In order to invert \eqref{6127}, we  consider 
the affine map  
\be\label{eq:op-contr}
\psi\to \omega + \opbw(a_0(\eta,\psi))\eta
\ee
(recall that $ a_0(\eta,\psi) $ is linear in $\psi$). 
Since $ a_0(\eta,\psi) $ is in $ \sFRa{K,0,1}{N} $, we deduce, by Proposition \ref{215}, that  there is $ \sigma $ such that
\[ 
\nnorm{\omega+\opbw(a_0(\eta,\psi))\eta}_{K, \sigma} \leq \nnorm{\omega}_{K, \sigma} +
C \nnorm{\psi}_{K, \sigma} \nnorm{\eta}_{K, s} \, . 
\]
Therefore, for $ \nnorm{\eta}_{K, \sigma} < \frac{1}{2C}$, the map in \eqref{op-contr} is a contraction 
of the ball of $ \CKH{\sigma}{\C} $ 
with radius $ R > 2 \nnorm{\omega}_{K, \sigma} $. 
Let  $ \psi = \Theta(\eta,\omega)$ denote the unique  fixed point in such a ball.
The map $\Theta$ is linear in $\omega$, and,  if $\eta, \omega $ are in $  \CKH{s}{\C} $ with $s\geq \sigma$, then
$\Theta(\eta,\omega)$ is in $  \CKH{s}{\C} $, with an estimate
\begin{equation}
  \label{eq:6140}
  \nnorm{\Theta(\eta,\omega)}_{K,s} \leq \nnorm{\omega}_{K,s} +
C_s \nnorm{\omega}_{K, \sigma} \nnorm{\eta}_{K, s} \, .
\end{equation}
If we iterate the formula \eqref{6127} giving $\psi$ from $\omega$, we get
\[
\psi = \omega + \opbw(a_0(\eta,\omega))\eta + \opbw(\tilde{a}_0(\eta,\psi))\eta
\]
where $\tilde{a}_0$ is obtained replacing in $a_0(\eta,\psi)$, 
$\psi$ by $\opbw(a_0(\eta,\psi))\eta$. 
Iterating the process and using (i) of Proposition~\ref{233} and the fact that 
$\opbw(a_0(\eta,\psi))$ is in $\sMa{0}{K,0,1}{N}$ 
by the third remark following Definition~\ref{216},   we may 
write
\begin{equation}
  \label{eq:6141}
  \psi = \omega + \sum_{p=1}^{N-1}\opbw(a_{0,p}(\eta,\dots,\eta,\omega))\eta + \opbw(a_{0,N}(\eta,\psi))\eta,
\end{equation}
where $a_{0,p}$ are functions in  $\Ft{p} $ and $a_{0,N}$ is a function in $\Fra{K,0,N}$, linear in $\psi$. If we replace in 
the symbol $a_{0,N}(\eta, \psi)$ the variable $\psi$ by  $\Theta(\eta,\omega)$, we deduce from \eqref{6140} 
that estimates of the form \eqref{218} hold for the composition, so that the last term in
\eqref{6141} may be written as $\opbw(a'_{0,N}(\eta,\omega))\eta$ for some function $a'_{0,N}$ in 
$ \Fra{K,0,N}$.

Finally, by \eqref{6127} and since $ a_0 (\eta, \psi ) $ 
is linear in $ \psi $, 
the involution $ S = \sm{1}{0}{0}{-1}  $ translates 
in the variables $(\eta, \omega )$ as the same map $(\eta, \omega )\to (\eta, -\omega)$.
This concludes the proof.
\end{proof}

We may now provide a paradifferential expression of  the Dirichlet-Neumann operator.

\begin{proposition}
  \label{615} {\bf (Dirichlet-Neumann)}
  Let us define
\begin{equation}
 \label{eq:6141a}
  V(\eta,\psi) = \partial_x\tilde{\Phi}\vert_{z=0} -\eta'\partial_z\Phit\vert_{z=0} \, .
\end{equation}
If $\eta, \psi$ are even real valued functions, then $V(\eta,\psi)$ is a function of the 
space  $\sFRa{K,0,1}{N}$ of
Definition~\ref{241} (i.e.\ a symbol of $\sGa{0}{K,0,1}{N}$ independent of $\xi$), for some $ r > 0 $ 
and any $ N $ in  $ \N $, which is
linear in $\psi$, real valued, and odd as a function of $ x $. 
Moreover, there are 

$ \bullet $ symbols  $b^0(\eta;\cdot)$ in
$\sGa{-1}{K,0,1}{N}$, $c^0(\eta,\omega;\cdot)$ in $\sGa{0}{K,0,1}{N}$, $c^0 (\eta,\omega;\cdot) $ being 
linear in $\omega$, even in $(x,\xi)$,  satisfying
\begin{equation}
  \label{eq:6142}
  (\bar{b}^0)^\vee = b^0,\quad (\bar{c}^0)^\vee = c^0 \, , 
\end{equation}

$ \bullet $ 
smoothing operators $R_1(\eta)$, $R_2(\eta,\omega)$ in $\sRa{-\rho}{K,0,1}{N}$, for $\rho$
an arbitrary given number,  $ R_2(\eta,\omega) $ being linear in $\omega$, 
satisfying 
\be\label{eq:318319}
{\overline {R_j V}} = R_j {\overline V} \, , \qquad  R_j \circ \tau  = \tau \circ R_j \, , \quad j = 1,2 \, , 
\ee
where $ \tau $ is defined  above \eqref{318}, 
i.e.\ these operators send real valued
(resp. even) functions
to real valued (resp. even) functions,

 such that
\begin{multline}
  \label{eq:6143}
G(\eta)\psi = (D\tanh D)\omega -i\opbw\bigl(V(\eta,\psi)\xi\bigr)\eta\\
+\opbw(b^0(\eta;\cdot))\omega + \opbw(c^0(\eta,\omega;\cdot))\eta\\
+ R_1(\eta)\omega + R_2(\eta,\omega)\eta \, .
\end{multline}
\end{proposition}

\begin{proof}
  We shall prove here that \eqref{6143} holds with $b^0$ in $\sGa{0}{K,0,1}{N}$. The fact that $b^0$ is actually of order $-1$
  will be shown in section~\ref{sec:72}, computing explicitly its principal symbol. 
  
  According to \eqref{112a}, the   Dirichlet-Neumann operator is given by 
  $$ 
  G(\eta) \psi = (\partial_y\phiu -\eta'\partial_x\phiu)\vert_{y=\eta(t,x)} 
  $$
  where   $\phiu$ is the  velocity potential defined as the solution of \eqref{611}. Using \eqref{617} and  \eqref{619}  ($ \Phit = \varphi $ for $ z $ close to $ 0 $), 
  we may then express
  \begin{equation}
    \label{eq:6144}
    G(\eta)\psi = \bigl((1+\eta'{}^2)\partial_z\Phit -\eta'\partial_x\Phit\bigr)\vert_{z=0} \, .
  \end{equation}
  We now provide a paradifferential expression of $  G(\eta)\psi  $ in \eqref{6144}.
By lemma~\ref{612} we write
\begin{multline}
  \label{eq:6145}
G(\eta)\psi = \partial_z\Phit \vert_{z=0} + \opbw(\eta'\otimes\eta') \partial_z\Phit \vert_{z=0} +
2\opbw(\eta'\otimes \partial_z\Phit \vert_{z=0})\eta'\\
-\opbw(\eta')\partial_x\Phit \vert_{z=0}  - \opbw(\partial_x\Phit \vert_{z=0})\eta'\\
+R'(\eta')(\partial \Phit \vert_{z=0}) + R''(\eta',\partial \Phit \vert_{z=0})\eta'\\
= I+\cdots+VII
\end{multline}
where $R'(\eta')$, $R''(\eta', \cdot )$ are smoothing operators in 
$\sRa{-\rho}{K,0,1}{N} $ (as in item (iv) of Proposition \ref{233}) for an arbitrary $\rho$,
$R''$ being linear in the argument 
$\partial \Phit \vert_{z=0}$, and $\partial \Phit$ standing for either a $\partial_x$ or $\partial_z$
derivative of $\Phit$.

Let us consider $ \partial_z\Phit \vert_{z=0} $.
By differentiating \eqref{6115} with respect to $ z $ we get 
\begin{equation}
  \label{eq:6146}
  \partial_z\Phit = \partial_z\Phi + \opbw(\partial_z^2\Phit)\eta \, .
\end{equation}
On the other hand, \eqref{6133} shows that $\partial_z\Phit\vert_{z=0}$ and $\partial^2_z\Phit\vert_{z=0}$ may be written as
$M(\eta)\psi$ for some $M$ in $\sMa{m}{K,0,0}{N}$ and some $m$ (recalling the definition of the class $\sNd{m}{\cdot}{N}$ of
Definition~\ref{533a}). Then, expressing $ \psi $  as a function of $ (\eta, \omega) $ as in \eqref{6131}, we may write 
\begin{equation}\label{eq:6146a}
\partial_z^2\Phit\vert_{z=0} = M'(\eta)\omega + M''(\eta,\omega)\eta
\end{equation}
for operators $M', M''$ in $\sMa{m}{K,0,0}{N}$, sending real valued functions to real valued functions, with $M''$
linear in $\omega$. Estimates \eqref{2125}, \eqref{2127} imply that such functions define symbols of order zero, independent of
$\xi$, i.e.\ that bounds \eqref{214}, \eqref{218} with $m=0$ are satisfied. Consequently, the last term in \eqref{6146}
restricted at $ z = 0 $ 
may be written as 
\be\label{eq:Der2}
\opbw(\partial_z^2\Phit_{| z = 0})\eta =  \opbw(c^0 (\eta, \omega))\eta 
\ee
where $ c^0 (\eta, \omega) $ is in  
$  \sFa{K,0,1}{N}  $, linear in $\omega$, and so it 
gives a contribution to $ \opbw(c^0)\eta$ in \eqref{6143}. (We shall verify condition \eqref{6142} below). 
Moreover using  \eqref{6146},
the expression of  $\partial_z\Phi\vert_{z=0}$  given by \eqref{6128},  \eqref{Der2}, 
we may write, using Propositions \ref{231} and \ref{232}, 
\[
I+II = \bigl(\mathrm{Id}+\opbw(\eta'\otimes\eta')\bigr)\bigl[(D\tanh D)\omega + \opbw(a_1(\eta,\cdot))\omega\bigr]
\]
modulo contributions of the form $\opbw(c^0)\eta$ (for a different $ c^0 $) and of the form of the smoothing terms in \eqref{6143}. By symbolic
calculus, we may rewrite this as
\be\label{eq:pass2}
\opbw\bigl((1+\eta'{}^2)(\xi\tanh\xi) + (1+\eta'{}^2)a_1(\eta; x, \xi)\bigr)\omega
\ee
modulo a contribution to the $\opbw(b^0(\eta,\cdot))\omega$ term in \eqref{6143} 
(with $b^0$ in $\sGa{0}{K,0,1}{N}$). Taking  into account the expression of the symbol $ a_1 $ in \eqref{6130}, we may
finally write \eqref{pass2} as 
\begin{equation}
  \label{eq:6147}
  (D\tanh D)\omega + i\opbw(\eta'\xi)\omega.
\end{equation}
Consider next $III + IV$. Since $ \Phit (x, 0) = \psi (x ) $ by 
\eqref{619}, \eqref{617}, \eqref{611}, we get, 
according to \eqref{6131}, 
\begin{equation}
  \label{eq:6148}
  \partial_x\Phit\vert_{z=0} = \partial_x\psi = \partial_x \big[ \omega + \opbw(a(\eta,\omega;\cdot))\eta \big]
\end{equation}
so that, again by symbolic calculus
\[III + IV = 2i\opbw(\eta'\partial_z\Phit\vert_{z=0}\xi)\eta - i\opbw(\eta'\xi)[\omega+\opbw(a(\eta,\omega;\cdot))\eta]\]
modulo again contributions to $\opbw(b^0)\omega+\opbw(c^0)\eta$ and to the smoothing terms in \eqref{6143}. Since the symbol $a(\eta, \omega) $ in \eqref{6131} is nothing but 
$ a_0 (\eta, \psi) = B = \partial_z\Phit\vert_{z=0} $ (see \eqref{Ba0}), we obtain
\[
III+ IV = i\opbw(\eta'\partial_z\Phit\vert_{z=0}\xi)\eta - i\opbw(\eta'\xi)\omega
\]
still modulo the same contributions as above. Using again symbolic calculus, we get, up to such contributions,
\be\label{eq:tre45}
III + IV + V = i\opbw\bigl((\eta'\partial_z\Phit-\partial_x\Phit)\vert_{z=0}\xi\bigr)\eta - i\opbw(\eta'\xi)\omega \, .
\ee
Finally we consider the terms $VI+VII$ of \eqref{6145}.
By \eqref{6148} and  \eqref{6133} where we express $\psi$ by \eqref{6131},  
we may express 
$\partial_x\Phit\vert_{z=0}$  and  $\partial_z\Phit\vert_{z=0}$ 
as the right hand side of \eqref{6146a}, so
that, using (ii) and 
(iv) of Proposition~\ref{233} (actually  just for the homogeneous components of the smoothing operator),
we deduce that 
\be\label{eq:VI7}
VI + VII =  R_1(\eta)\omega + R_2(\eta,\omega)\eta 
\ee
contribute to the last two terms in \eqref{6143}.

In conclusion,  by \eqref{6145}, \eqref{6147}, \eqref{tre45}, \eqref{VI7}, \eqref{6141a} we  get 
\[
G(\eta)\psi = (D\tanh D)\omega + i \opbw\bigl(V(\eta, \psi) \xi\bigr)\eta
\]
modulo contributions to the last four terms in \eqref{6143}.

The fact that $V$ is real valued follows from its expression \eqref{6141a}. Moreover, if $\eta$ and $\psi$ are even, the
solution $\phiu$ to \eqref{611} is even in $x$, so that $\varphi$ and $\varphi_0$ given by \eqref{617} and \eqref{618} are
even in $x$, as well as $\Phit$ given by \eqref{619}. Then $\partial_z\Phit$ is even in $x$, $\partial_x\Phit$ is odd in $x$,
so that \eqref{6141a} shows that $V(\eta,\psi)$ is odd in $x$ as claimed. 

Moreover, $G(\eta)\psi$ is even and real valued by \eqref{6144} as well as 
the function $ (D\tanh D)\omega -i\opbw\bigl(V(\eta,\psi)\xi\bigr)\eta $. 
It follows that  the sum of the last four
terms in  \eqref{6143}  is even and real valued. Consequently, we may assume that $b^0, c^0$ satisfy \eqref{6142}
and are even functions of $(x,\xi)$, and that the smoothing operators send real (resp.\ even) functions to real (resp.\ even)
functions.
As a consequence, we may always replace for instance $R_1(\eta)$ by $\frac{1}{2}[\tau\circ
R_1(\eta)+R_1(\eta)\circ\tau]$, that ensures 
that the second condition in \eqref{318319}  is satisfied  
by $R_1(\eta)$, and argue similarly for the other
smoothing terms,  and for the first property in \eqref{318319}. 
This concludes the proof.

As mentioned at the beginning of the proof, the fact that $b^0$ is actually of order $-1$ will be a consequence of some
explicit computations that we postpone to section~\ref{sec:72}.
\end{proof}

\section[Paralinearization of the system]{Paralinearization of the water waves system}\label{sec:62}

This section is devoted to write  system \eqref{113} (with $g=1$) as a paradifferential system in the variables
$(\eta, \omega)$. This has been done already for the first equation $\partial_t\eta = G(\eta)\psi$ in Proposition~\ref{615}, 
see \eqref{6143}.
In this section we deduce from \eqref{113} a paradifferential equation for $\partial_t\omega$. 

We first recall  that by \eqref{6115}, \eqref{Ba0}
\begin{equation} \label{eq:621}
 B(\eta,\psi) = \partial_z\Phit\vert_{z=0} = a_0 (\eta, \psi)\, ,
\end{equation}
and, by \eqref{6141a} and  $ {\tilde \Phi} (x,0) = \psi (x)$ (see \eqref{619}, \eqref{617}, \eqref{613}) 
\be\label{eq:expV}
V(\eta,\psi) =  \partial_x\tilde{\Phi}\vert_{z=0} -\eta'\partial_z\Phit\vert_{z=0} = \partial_x \psi - \eta'  B \, . 
\ee 
We also recall the linear involution  
$$
S : \R^2 \to \R^2 \, , \quad {\rm with \ matrix \ }  
 \bigl[\begin{smallmatrix}1&0\\0&-1\end{smallmatrix}\bigr] \, , 
$$
and we consider  the $\Z/2\Z$ action of the group $ \{ {\rm Id}, S \} $ on functions 
$ \vect{\eta (t, \cdot)}{ \psi (t, \cdot )} $ defined by
\begin{equation}\label{eq:310-real} 
  \vect{\eta}{\psi}_S (t) \stackrel{\mathrm{def}}{=}  S  \vect{\eta}{\psi} (-t) =  \vect{\eta (-t)}{-\psi(-t)} 
\end{equation}
(which is is the analogue in the real variables of the action introduced in \eqref{action}-\eqref{310}). 

The main result of this section is
\begin{proposition}
  \label{621} {\bf (Equation for $ \pa_t \omega $)}
Let $ N $ in $ \N^* $. 
Let  $ (\eta, \psi ) $ be a solution of  \eqref{113}.  
For $ r > 0 $ small enough, there exist \\
$\bullet$ Symbols $a_2 (\eta, \cdot )$ in $\sGa{2}{K,0,1}{N}$, $a_1 (\eta, \omega; \cdot )$ 
in $\sGa{1}{K,0,1}{N}$, and $ a_0 (\eta, \omega; t, \cdot )$ in $\sG{0}{K,1,2}{N}$ 
satisfying
\begin{equation}
  \label{eq:622}
  \Im a_2 \equiv 0,\quad \Re a_1 \in \sGa{0}{K,0,1}{N},\quad \Im a_1 = -V(\eta,\psi)\xi
\end{equation}
such that $a_1$ is linear in $\omega $ (or in $\psi$ according if we express 
$\psi$ from $\omega$ by \eqref{6131}),  
\be\label{eq:a0:rev} 
a_0 (\eta, \omega; - t) = a_0 ( [\eta, \omega]_S ; t) \, , 
\ee 
the homogeneous terms of 
$ a_0 $ are quadratic in $ \omega $, and
\begin{equation}
  \label{eq:623}
  \bar{a}_j^\vee = a_j,\quad a_j(\cdot;x,\xi) = a_j(\cdot;-x,-\xi), \ \   j = 0, 1, 2 \, . 
\end{equation}
Moreover
\begin{equation}
  \label{eq:624a}
  a_2(\eta;x,\xi) = -\kappa\xi^2\bigl[(1+\eta'(t,x)^2)^{-\frac{3}{2}} -1\bigr] 
\end{equation}
modulo a symbol
in $\sGa{0}{K,0,1}{N}$, 

$\bullet$ A smoothing operator $R_1(\eta,\omega)$, resp.\ $R_2(\eta,\omega,\omega)$, resp.\ $R_0(\eta)$, which is linear in
$ \omega $ and belongs to $ \sRa{-\rho}{K,0,1}{N} $, resp.\ which is quadratic in $ \omega $ 
and belongs to $\sRa{-\rho}{K,0,2}{N}$,
resp.\ which belongs to $\sRa{-\rho}{K,0,1}{N}$, satisfying condition \eqref{318319},

such that
\begin{multline}
  \label{eq:624}
\partial_t\omega = -(1+\kappa D_x^2)\eta + \opbw(a_2(\eta;\cdot))\eta + \opbw(a_1(\eta,\omega;\cdot))\omega\\
+ \opbw(a_0(\eta,\omega; t, \cdot))\eta + R_1(\eta,\omega)\omega + R_2(\eta,\omega,\omega)\eta + R_0(\eta)\eta \, .
\end{multline}
\end{proposition}

\noindent
{\bf Remark:} The symbol $ a_0 $ is the only  non-autonomous 
time dependent term  in the right hand side of \eqref{624}. All the other terms are autonomous, namely the time 
dependence enters only through $ \eta (t) $ and $ \omega ( t)  $. 

\medskip

\begin{proof}
  As a consequence of \eqref{6143} and \eqref{6127}, and since $a_0$ is linear in $\psi$ and $c_0$ is linear in $\omega$, we may write 
\be\label{eq:deta}
\partial_t \eta = G(\eta)\psi = (D\tanh D)\psi + M(\eta)\psi
\ee
for some operator $M $ in $\sMa{1}{K,0,1}{N}$. We substitute this expression inside the right hand side of the second equation
\eqref{113}. Using the last remark after Definition~\ref{216}, we may write 
\be\label{eq:dpsi}
\partial_t\psi = M(\eta)\eta + M'(\eta,\psi)\psi + M''(\eta,\psi,\psi)\eta
\ee
for some operators
$$
M \in \sMa{}{K,0,0}{N} \, , \quad 
M', M'' \in \sMa{}{K,0,1}{N}
$$ 
with $M'$ (resp.\ $M''$) linear (resp.\ quadratic) in
$\psi$. By \eqref{621} and \eqref{6133}, we may write also
\begin{equation}
  \label{eq:625}
  B(\eta,\psi) = \partial_z\Phit\vert_{z=0} = (D\tanh D)\psi + M'''(\eta)\psi
\end{equation}
for some $M'''$ in $\sMa{}{K,0,1}{N}$. Differentiating with respect to time we get 
$$
\partial_t[B(\eta,\psi)] = [(D\tanh D) + M'''(\eta)] \partial_t\psi  + \partial_t[M'''(\eta)] \, \psi  
$$
and using the  expression of $\partial_t\psi$ in \eqref{dpsi}, this implies that
\begin{equation}\label{eq:626a1}
\partial_t[B(\eta,\psi)] = M(\eta)\eta + M'(\eta,\psi)\psi + M''(\eta,\psi,\psi)\eta + \partial_t[M'''(\eta)] \, \psi 
\end{equation}
for new $M$ in $\sMa{}{K,0,0}{N}$ and $M', M''$ in $\sMa{}{K,0,1}{N}$ with $M'$ (resp.\ $M''$) linear (resp.\ quadratic) in $\psi $. Consider then the term $ \partial_t[M'''(\eta)] \, \psi $
where we expand the operator $ M''' (\eta)  $  in $ \sMa{}{K,0,1}{N} $ as 
$$ 
M''' (\eta) = M_1''' (\eta) + \ldots + M_{N-1}''' (\eta) + M_N''' (\eta)  \, . 
$$ 
 The homogeneous terms  $ M_j''' (\eta) $, $ j = 1, \ldots, N- 1 $,  
are differentiable with respect to $ \eta $. But, since we have not assumed differentiability 
with respect to $ \eta $ of the non-homogeneous autonomous term (just the estimates \eqref{2127}), we 
deal $ M_N''' (\eta) $  as a time dependent operator. 
We use the same argument of lemma \ref{217} and the fact that $ \partial_t \eta $ solves \eqref{deta}
to conclude that
 $  \partial_t[M'''(\eta)] $ is in $  \sM{}{K,1,1}{N} $. We may insert the terms coming from the
 homogeneous contributions into the autonomous expressions \eqref{626a1} 
 and 
 we get 
 \be\label{eq:626a}
 \partial_t[B(\eta,\psi)] = M(\eta)\eta + M'(\eta,\psi)\psi + M''(\eta,\psi,\psi)\eta + [\partial_t M_N'''(\eta)] \, \psi \, .
 \ee
We remark that the only term that we deal as a time dependent one is the 
function $ [\partial_t M_N'''(\eta)] \, \psi $. All the other terms are autonomous.  
The condition of reversibility of this non-homogeneous time dependent term is that 
the function $ [\partial_t M_N'''(\eta)] \, \psi $ satisfies \eqref{a0:rev} (with $ \psi $ instead of $ \omega $)
as it follows recalling \eqref{310-real}
(remark that we do not say that $ [\partial_t M_N'''(\eta)] \, \psi $  is  quadratic in $ \psi $ because, in 
this time dependent  term we do not substitute the equation for $ \partial_t \eta $). 
 
By \eqref{6144}, \eqref{621}, and since $ {\tilde \Phi} (x,0) = \psi (x)$ (see \eqref{619}, \eqref{617}, \eqref{613}) 
we have 
\be\label{eq:Geta}
G(\eta)\psi = (1+\eta'{}^2)B(\eta,\psi) -\eta'\partial_x\psi \, .
\ee
Thus differentiating the good unknown $ \omega $ defined in \eqref{6114}, since $ \eta, \psi $ solve 
system \eqref{113} with $ g = 1 $, and using \eqref{Geta}, we get 
\begin{multline}
  \label{eq:626}
\partial_t\omega = \partial_t\psi -\opbw(B)\partial_t\eta -\opbw(\partial_t B)\eta\\
= -\eta + \kappa\biggl(\frac{\eta'}{\sqrt{1+\eta'{}^2}}\biggr)' -\frac{1}{2}(\partial_x\psi)^2 +
\frac{1}{2}(1+\eta'{}^2)B(\eta,\psi)^2\\
-\opbw(B)[(1+\eta'{}^2)B(\eta,\psi) -\eta'\partial_x\psi] - \opbw(\partial_tB)\eta \, .
\end{multline}
We now paralinearize the different terms in the right hand side of \eqref{626}. 
Applying lemma~\ref{612}, and using the symbolic calculus formulas \eqref{222}, \eqref{223} 
(which are exact and not asymptotic ones for the symbols considered here), 
we get
\begin{multline}
\label{eq:627}
\eta'{}^2B = \opbw(\eta'{}^2)B + 2\opbw(\eta'B)\eta' + R'_1(\eta)B + R''_1(\eta,B)\eta\\
= \opbw(\eta'{}^2)B + 2i\opbw(\eta'B\xi)\eta -\opbw(\partial_x(\eta'B))\eta\\
 + R'(\eta)\omega + R''(\eta,\omega)\eta
\end{multline}
where 
 $R'$, $R''$ (resp.\ $R'_1$, $R''_1$) are elements of $\sRa{-\rho}{K,0,2}{N}$, with $R''$ (resp.\
$R''_1$) linear in $\omega$ (resp. $B$) that are obtained from the smoothing operators of lemma~\ref{612} expressing $B$ in
terms of $\omega$ by \eqref{625} and \eqref{6131}, and using that by Proposition~\ref{233} the composition 
$R(M(U)U)$ or
$R(U)M(U)$ of  a smoothing operator $R$ and of an element $M$ of a class $\sMa{}{K,0,p}{N}$, gives again a smoothing operator
(up to a change of the smoothing index $\rho$, which is in any case as large as we want). Making act $\opbw(B)$ on
\eqref{627} and using symbolic calculus, we get
\begin{multline}
  \label{eq:628}
\opbw(B)[(1+\eta'{}^2)B] = \opbw(B(1+\eta'{}^2))B + 2i\opbw(\eta' B^2\xi)\eta\\
-\opbw(\partial_x(B^2\eta'))\eta + R'(\eta,\omega)\omega + R''(\eta,\omega,\omega)\eta
\end{multline}
where $R'$ (resp.\ $R''$) is in $\sRa{-\rho}{K,0,3}{N}$ and it is linear in $\omega$ (resp.\ in $\sRa{-\rho}{K,0,3}{N}$ and quadratic in
$\omega$).

In the same way, lemma~\ref{612} allows us to write
\begin{equation}\label{eq:628a}
\eta'\partial_x\psi = \opbw(\eta')\partial_x\psi + \opbw(\partial_x\psi)\eta' + R'(\eta')\partial_x\psi +
R''(\partial_x\psi)\eta'
\end{equation}
with $R', R''$ in $\Rt{-\rho}{1}$. Applying $\opbw(B)$ on \eqref{628a}, using symbolic calculus and expressing in the smoothing operators
$\psi$ from $ (\eta, \omega) $ by \eqref{6131}, we obtain
\begin{multline}
  \label{eq:629}
\opbw(B)(\eta'\partial_x\psi) = i\opbw(\eta'B\xi)\psi + i\opbw(B(\partial_x\psi)\xi)\eta\\
-\frac{1}{2}\opbw(\partial_x(B\eta'))\psi -\frac{1}{2}\opbw\bigl(\partial_x[B(\partial_x\psi)]\bigr)\eta\\
+ R'(\eta,\omega)\omega + R''(\eta,\omega,\omega)\eta
\end{multline}
with $R', R''$ in $\sRa{-\rho}{K,0,2}{N}$, respectively linear and quadratic in $\omega$. 

Applying again lemma~\ref{612}, we get also
\begin{multline}\label{eq:629a}
  \frac{1}{2}(1+\eta'{}^2)B^2 = \opbw((1+\eta'{}^2)B)B +\opbw(B^2\eta')\eta' \\
+ R'_1(\eta,B)B +R''_1(\eta,B,B)\eta
\end{multline}
with $R'_1$ in $\sRa{-\rho}{K,0,1}{N}$ linear in $B$, $R''_1$ in $\sRa{-\rho}{K,0,3}{N}$ quadratic in $B$. Using symbolic
calculus, and expressing $B$ in the smoothing operators in terms of $\eta $, $\omega$,  by \eqref{625} and \eqref{6131}, 
we get
\begin{multline}
\label{eq:6210}
  \frac{1}{2}(1+\eta'{}^2)B^2 = \opbw((1+\eta'{}^2)B)B +i\opbw(B^2\eta'\xi)\eta \\ -\frac{1}{2}\opbw(\partial_x(B^2\eta'))\eta 
  + R'(\eta,\omega)\omega
+R''(\eta,\omega,\omega)\eta
\end{multline}
with $R', R''$ in $\sRa{-\rho}{K,0,1}{N}$  and $\sRa{-\rho}{K,0,2}{N}$ respectively, with $R'$ (resp.\ $R''$) linear (resp.\
quadratic) in $\omega$. In the same way, we write
\begin{multline}
  \label{eq:6211}
\frac{1}{2}(\partial_x\psi)^2 = i\opbw((\partial_x\psi)\xi)\psi - \frac{1}{2}\opbw(\partial_x^2\psi)\psi\\
+ R'(\eta,\omega)\omega +R''(\eta,\omega,\omega)\eta \, .
\end{multline}
Moreover, by \eqref{2313} applied to the remainder of Taylor formula for the function 
$\eta'\to \frac{\eta'}{\sqrt{1+\eta'{}^2}}$ (expanded at the order $ N $), and
using lemma~\ref{612} for the polynomial terms of the expansion, we may write
\[\frac{\eta'}{\sqrt{1+\eta'{}^2}} = \opbw(a(\eta'))\eta' + R(\eta)\eta\]
where $ R(\eta) $ is in $\sRa{-\rho-1}{K,0,2}{N}$ and 
\be\label{eq:aeta}
a(\eta') = (1+\eta'{}^2)^{-3/2} \, . 
\ee 
Consequently, 
using again symbolic calculus, 
\begin{equation}
  \label{eq:6212}
  \partial_x\Bigl(\frac{\eta'}{\sqrt{1+\eta'{}^2}}\Bigr) = -\opbw\bigl(\xi^2 a(\eta')+
  \frac{1}{4}\partial_x^2(a(\eta'))\bigr)\eta + R(\eta)\eta 
\end{equation}
for some other smoothing operator  $ R(\eta) $  in $\sRa{-\rho}{K,0,2}{N} $.  

We now compute $\partial_t\omega$ from \eqref{626}. Plugging in the right hand side of \eqref{626} the formulas
\eqref{628}, \eqref{629}, \eqref{6210}, \eqref{6211}, \eqref{aeta}, \eqref{6212} 
and using \eqref{Ba0}, \eqref{6131}, we get, 
after simplification,  
\begin{multline}
\label{eq:6213}
\partial_t\omega = -\eta -\kappa\opbw\bigl((1+\eta'{}^2)^{-3/2}\xi^2\bigl)\eta
-i\opbw\bigl((\partial_x\psi-\eta'B)\xi\bigr)\psi\\
-i\opbw\bigl(B(\eta' B-(\partial_x\psi))\xi\bigr)\eta + \opbw(a_0^0(\eta))\eta\\
+\opbw(a_0^1(\eta,\omega))\psi + \opbw(a_0^2(\eta,\omega,\omega))\eta - \opbw(\partial_t B)\eta\\
+R(\eta)\eta + R'(\eta,\omega)\omega + R''(\eta,\omega,\omega)\eta
\end{multline}
where $R$ and $R''$ are in $\sRa{-\rho}{K,0,2}{N}$, $R''$ being quadratic in $\omega$, 
$R'$ is in $\sRa{-\rho}{K,0,1}{N}$,
linear in $\omega$, the functions $a_0^0, a_0^1, a_0^2$ belong  respectively to 
$\sFa{K,0,2}{N}$, $\sFa{K,0,1}{N}$, $\sFa{K,0,2}{N}$, $a_0^1$
(resp.\ $a_0^2$) being linear (resp.\ quadratic) in $\omega$, these functions being real valued and even in $x$ (as follows
from the evenness of $\eta$ and $\psi$). Moreover 
$$
a_0^0 (\eta ) = -\frac{\kappa}{4}\partial_x^2\bigl[(1+\eta'{}^2)^{-3/2}\bigr] \, .
$$ 
Notice also that by \eqref{625}, $\partial_tB$ is  real
valued and even in $x$, as $\Phit$ enjoys the same properties. 

The sum of the third and fourth terms in the right hand side
of \eqref{6213} may be written, using that $ V = \partial_x\psi-\eta'B $ by \eqref{expV}, as 
\[
-i \opbw(V\xi)\psi + i \opbw(BV\xi)\eta \, .
\]
Replacing $ \psi $ by its value coming from \eqref{6127} and \eqref{Ba0}, 
namely $ \psi = \omega +\opbw(B) \eta $, and using symbolic calculus, we get
\begin{equation}
  \label{eq:6214}
  - i \opbw(V\xi)\omega - \frac{1}{2} \opbw(V\partial_xB) \eta
\end{equation}
modulo smoothing terms as in \eqref{6213}. By \eqref{6213}, \eqref{6214}, 
expressing in $V(\eta,\psi)$, $B(\eta,\psi)$ the variable $\psi$ in terms of $\eta, \omega $ by  \eqref{6131}, 
and substituting the expression of $ \pa_t B $ in \eqref{626a}, we finally get 
\begin{multline}
  \label{eq:6215}
\partial_t\omega = -(1+\kappa D^2)\eta + \opbw(a_2(\eta;\cdot))\eta + \opbw(a_1(\eta,\omega;\cdot))\omega\\
+ \opbw(a_0(\eta,\omega; t, \cdot))\eta + R(\eta)\eta + R'(\eta,\omega)\omega + R''(\eta,\omega,\omega)\eta
\end{multline}
where\\
$\bullet$ $a_2$ is a symbol in $\sGa{2}{K,0,1}{N}$,  real valued, even in $(x,\xi)$ and satisfies $\bar{a}_2^\vee = a_2$.\\
$\bullet$ $a_1$ is a symbol in $\sGa{1}{K,0,1}{N} $, 
even in $(x,\xi)$ and satisfies $\bar{a}_1^\vee = a_1$. 
Moreover $a_1$ is linear in
$\omega$, $\Re a_1 \in \sG{0}{K,0,1}{N} $ 
and $\Im a_1 = -V(\eta,\psi)\xi$.
\\
$\bullet$ $a_0$ is in $\sG{0}{K,1,2}{N}$,  even in $(x,\xi)$,  satisfies $\bar{a}_0^\vee = a_0 $. 
Moreover $ a_0 $ is even in $ t $ and its homogeneous part is 
autonomous and quadratic in $ \omega $.\\
$\bullet$ $R(\eta)$ (resp.\ $R'(\eta,\omega)$, resp.\ $R''(\eta,\omega,\omega)$) belongs to  $\sRa{-\rho}{K,0,1}{N}$ (resp.\
$\sRa{-\rho}{K,0,1}{N}$, resp.\ $\sRa{-\rho}{K,0,2}{N}$), 
the operator $R'$ (resp.\ $R''$) being linear (resp.\ quadratic) in $\omega$.

Actually, $ a_2 $ is given by $\kappa[1-(1+\eta'{}^2)^{-3/2}]\xi^2 + a_0^0(\eta)$ plus 
the function $M(\eta)\eta$ that comes from the
first term in the expression of $\partial_tB$ given by \eqref{626a}. It satisfies the above requirements since
$\eta$ is even. Moreover, it is given by \eqref{624a} modulo a symbol of order zero.  

The symbol $a_1$ is formed by an order one contribution given by $-iV\xi$ in \eqref{6214}, 
plus the symbol  $ a^1_0(\eta,\omega) $ that arises by the term
$\opbw(a^1_0(\eta,\omega))\psi$ in \eqref{6213}, where we express $\psi$ by \eqref{6131}, and keep only 
$\opbw(a^1_0(\eta,\omega)) \omega $. 
Since $ V $ and $ a_0^1 $ are real valued, and since $ V $ is an odd function of $x$ and
$a_0^1$ an even one, we get that all conditions on $a_1$ are satisfied.

The symbol $a_0$ is computed  from the contribution $-\frac{1}{2}\opbw(V\partial_xB)\eta$ in \eqref{6214}, from
$\opbw(a_0^1(\eta,\omega))\opbw(a(\eta,\omega;\cdot))\eta$ that comes from the sixth term in the right hand side of
\eqref{6213}, where we replace $\psi$ by 
\eqref{6131}, 
from $\opbw(a_0^2(\eta,\omega,\omega))\eta$ in
\eqref{6213}, and from the contribution to $ \opbw(\partial_tB)\eta$ coming from 
\eqref{626a}, that is $ \opbw([\partial_t M_N''' (\eta )] \psi ) \eta $. 
Notice that this non-homogeneous term is the only non-autonomous one. As 
we proved that  
$ [\partial_t M_N''' (\eta )] \psi $ satisfies \eqref{a0:rev} (with $ \psi $ instead of $ \omega $), we get that
the function $ [\partial_t M_N''' (\eta )] (\omega + \opbw(a(\eta,\omega))\eta) $ obtained substituting 
$ \psi $ by \eqref{6131}, satisfies 
\eqref{a0:rev} as well.
 Moreover all the functions appearing as symbols in the above operators are real 
valued and even in $x$, so that the requirements of the statement are satisfied. 

The fact that we may assume that $R(\eta)$,
$R'(\eta,\omega)$, $R''(\eta,\omega,\omega)$ satisfy \eqref{318319} 
follows from the fact that $\partial_t\omega$ in
the left hand side of \eqref{6215} as well as the first four terms in the right hand side of \eqref{6215} are even in $x$ and
real valued. This follows from the fact that $\eta, \omega$ are real valued even functions and 
$\opbw(a_\ell)$ preserves these properties because
 $a_\ell$ is even in $(x,\xi)$ and satisfies
$\bar{a}_\ell^\vee = a_\ell$. As a consequence, we may always replace for instance $R(\eta)$ by $\frac{1}{2}[\tau\circ
R(\eta)+R(\eta)\circ\tau]$, that ensures 
that the second condition in \eqref{318319}  is satisfied  
by $R(\eta)$, and argue similarly for the other
smoothing terms and for the first property in \eqref{318319}. This concludes the proof.
\end{proof}

We summarize in a corollary  the new form of the water waves system that we have obtained
by Propositions \ref{615} and \ref{621}. 

\begin{corollary}
  \label{622} {\bf (Water-waves equations in $(\eta, \omega)$ variables)}
Let $\rho$ in $\R$, $N$ in $\N^*$, $K$ in $\N$ be given (large) numbers. There is $ r > 0 $ 
and there are \\

$ \bullet $
symbols $ a_0, a_1, a_2$
satisfying the properties of Proposition~\ref{621}, 

$ \bullet $
 symbols  $b_1(\eta,\omega;\cdot)$ in $\sGa{1}{K,0,1}{N}$, linear in 
 $\omega$, $b_0(\eta;\cdot)$ belonging to $\sGa{-1}{K,0,1}{N}$, even in $(x,\xi)$, satisfying  
\begin{equation}
  \label{eq:6216}
  \Im b_1 = -V(\eta,\psi) \xi \, , \ 
  \Re b_1\in \sG{0}{K,0,1}{N},\quad \bar{b}^\vee_0 = b_0, \quad \bar{b}^\vee_1 = b_1 \, , 
\end{equation}

$ \bullet $
smoothing operators $R_1(\eta)$, $R'_1(\eta,\omega)$, $R_2(\eta)$ and  $R'_2(\eta,\omega)$,
$R''_2(\eta,\omega,\omega)$ in $\sRa{-\rho}{K,0,1}{N}$, with $R'_1, R'_2$ linear in $\omega$,  $R''_2$ quadratic in
$\omega$, satisfying \eqref{318319}, 
  
 such that, for any large enough $s$, if a 
  small enough function $(\eta,\psi)$
of the space $C_*^K(I,\Hsze{s+\frac{1}{4}}\times \Hdse{s-\frac{1}{4}})$ solves 
system \eqref{113}, 
then $(\eta,\omega)$, with $ \omega $ given
by \eqref{6127}, solves on the same time interval $I$ the system
\begin{multline}
  \label{eq:6217}
\partial_t\eta = (D\tanh D)\omega + \opbw(b_1(\eta,\omega;\cdot))\eta + \opbw(b_0(\eta;\cdot))\omega\\
+ R_1(\eta)\omega + R'_1(\eta,\omega)\eta\\
\partial_t\omega = -(1+\kappa D^2)\eta + \opbw(a_2(\eta;\cdot))\eta + \opbw(a_1(\eta,\omega;\cdot))\omega\\
+  \opbw(a_0(\eta,\omega; t, \cdot))\eta 
 + R_2(\eta)\eta + R'_2(\eta,\omega)\omega + R''_2(\eta,\omega,\omega)\eta \, . 
\end{multline}
System \eqref{6217} 
is reversible with respect to the involution $ S = \sm{1}{0}{0}{-1} $, 
namely writing  
\eqref{6217} as
\be\label{eq:WWrev}
\partial_t \vect{\eta}{\omega} = M(\eta, \omega; t) \vect{\eta}{\omega}
\ee
then $ M(\eta, \omega; - t) S = - SM( [\eta, \omega]_S; t)  $ where $ [\eta, \omega]_S $  is defined in \eqref{310-real} 
(this condition is like \eqref{319}). 
For the autonomous terms of the vector field this condition amounts to  
the reversibility property \eqref{Frev}. 
As a consequence any solution of \eqref{6217} satisfies \eqref{125}. 
Moreover the operator $ M(\eta, \omega; t) $ sends real valued functions into real valued functions and 
$ M(\eta, \omega; t) \circ \tau = \tau \circ M(\eta, \omega; t) $ commutes with the map 
$ \tau $ defined in \eqref{def-tau}. 
\end{corollary}
\begin{proof}
  If $(\eta,\psi)$ is in a small enough ball of center zero in $C_*^K(I,\Hsze{s+\frac{1}{4}}\times \Hdse{s-\frac{1}{4}})$ and
  $\omega$ is defined by \eqref{6127}, we may apply Proposition~\ref{614}, that allows us to recover $\psi$ from 
  $\eta, \omega $ by \eqref{6131}. By \eqref{6143}, the first equation $\partial_t\eta = G(\eta)\psi$ of the water waves system may be
  written as the first equation in \eqref{6217} with 
  $$
  b_1(\eta,\omega;\cdot) = c^0(\eta,\omega;\cdot) -iV(\eta,\psi)\xi \, \quad 
  b_0 = b^0 \, .
  $$ 
   Then \eqref{6142} and the fact that $V$ is real valued and odd 
  imply that \eqref{6216} holds. Moreover  $ b_1 $ is even in $(x, \xi )$ because 
  $ c^0 $ and $ V(\eta,\psi)\xi $   are  even in $(x, \xi )$. 
  The second equation in \eqref{6217} is \eqref{624}.
  
 Finally,  writing \eqref{6217} as the system \eqref{WWrev}, 
 since 
$ b_1 (\eta,\omega)$,  $R'_1 (\eta,\omega) $,  $ a_1(\eta,\omega) $, $ R'_2 (\eta,\omega) $ are linear in $ \omega $,
$ R''_2(\eta,\omega,\omega) $ is quadratic in $ \omega $, 
and the symbol $ a_0(\eta,\omega; t, \cdot ) $ satisfies \eqref{a0:rev},  
  we deduce that the operator $ M(\eta, \omega; t) $ satisfies 
  $ M(\eta, \omega; - t) S = - SM( [\eta, \omega]_S; t)  $.
  Moreover the operator  $ M(\eta, \omega; t) $ sends real valued functions into real valued functions, and 
  commutes with $ \tau $,  because
   the symbols $ a_0, a_1, a_2 $ and  $ b_0, b_1 $ satisfy $ \bar{a}^\vee_j = a_j $, $ \bar{b}^\vee_j = b_j $, 
   are even in $ (x, \xi) $, and 
 the smoothing operators  $R_1 $, $R'_1 $, $R_2 $, $R'_2 $,
$R''_2 $  satisfy \eqref{318319}.   This concludes the proof.
\end{proof}

\section[Equation in complex coordinates]{The capillarity-gravity water waves equations in complex coordinates}\label{sec:63}

We denote by $\lk$ the Fourier multiplier with symbol
\begin{equation}
  \label{eq:631}
  \Lambda_\kappa(\xi) = \Bigl(\frac{\xi\tanh \xi}{1+\kappa \xi^2}\Bigr)^{1/4}(1-\chi(\xi))
\end{equation}
where $\chi$ is an even smooth function, equal to one close to zero and supported for $\abs{\xi}< \frac{1}{2}$.
Notice that, 
on the space of periodic functions with zero mean, or on periodic
functions modulo constants, the operator $\lk$ is just given by
\be\label{eq:Lambda-p}
\lk = \Bigl(\frac{D\tanh D}{1+\kappa D^2}\Bigr)^{1/4} \, .
\ee
We denote by $\lk^{-1}$ the inverse of $\lk$ acting from the space of zero mean functions to
itself. For $\eta$ in $\Hsze{s+\frac{1}{4}}(\Tu,\R)$ and  $\omega$ in $\Hdse{s-\frac{1}{4}}(\Tu,\R)$ we consider the complex function
\begin{equation}
  \label{eq:632}
  u = \lk\omega + i\lk^{-1}\eta 
\end{equation}
which belongs to the space 
\begin{equation}
  \label{eq:633}
  \Hdsz{s}(\Tu,\C) \stackrel{\mathrm{def}}{=} \Big\{ u\in H_{\mathrm{ev}}^s(\Tu,\C); \int_\Tu\Im u\,dx = 0 \Big\} \Big/\R
\end{equation}
where $H_{\mathrm{ev}}^s(\Tu,\C)$ denotes the subspace of even functions in $H^s(\Tu,\C)$.
This space is endowed with the norm $\norm{u}_{\Hds{s}} = \bigl(\sum_1^{+\infty}
n^{2s}\norm{\Pin{}u}_{L^2}^2\bigr)^{\frac{1}{2}}$. Inverting \eqref{632} we may express
\begin{equation}
  \label{eq:634}
  \eta  = \lk\Bigl[\frac{u-\bar{u}}{2i}\Bigr] \, , \quad 
  \omega = \lk^{-1}\Bigl[\frac{u+\bar{u}}{2}\Bigr], 
\end{equation}
where we make here a small abuse of notation, as, in contrast with \eqref{632}, we consider here $\lk^{-1}$ as an operator
sending a space of functions modulo constants to itself. 

Notice  that, since the symbol $   \Lambda_\kappa(\xi) $ in \eqref{631} is real and even, 
the corresponding Fourier multipliers $ \lk,  \lk^{-1} $ map real valued, resp. even, functions into real valued, resp. even, functions.

We now prove that  
system \eqref{6217} in  the real variables $ (\eta, \omega) $, may be written,  
in the complex variable  $U = \vect{u}{\bar{u}}$, as system \eqref{3112}.
Notice first that in these complex coordinates 
the real involution 
$ \vect{\eta}{ \omega} \mapsto \vect{\eta}{ - \omega } $ considered 
for system \eqref{6217} (or \eqref{113})  translates into the complex involution 
$$
S : \vect{u}{\bar u} \mapsto - \vect{\bar u}{u} \, , \quad {\rm with \ matrix \ } \quad  S = - \sm{0}{1}{1}{0} 
$$ 
(that for simplicity of notation we denote with the same letter $ S $).  The real 
system \eqref{WWrev} may be written  in the new complex coordinates as
\be\label{eq:WW-complex}
D_t U = X(U; t ) U 
\ee
where $ D_t = \frac{1}{i} \partial_t $,   the operator 
$  X(U; t ) $ satisfies the  reality condition  \eqref{316}, 
it is parity preserving (according to definition \eqref{318}) and   
 reversible (according to \eqref{319}), namely 
 $ X(U; - t ) = - S X(U_S; t ) S $
 with the new involution $ S = - \sm{0}{1}{1}{0} $. 
 
 We provide 
 below the para-differential structure  of system \eqref{WW-complex} which is suitable for proving energy estimates. 

\begin{proposition} \label{631}
{\bf (Water waves equations in complex coordinates)}
  Let $(\eta,\omega)$ be a 
  solution of system \eqref{6217} belonging to the space 
\[ C_*^K(I,\Hsze{s+\frac{1}{4}}(\Tu,\R))\times C_*^K(I,\Hdse{s-\frac{1}{4}}(\Tu,\R))\]
for some large enough $s, K$, some time interval $ I $, symmetric with respect to $ t = 0 $, $(\eta,\omega)$ being small enough in that space. Define
\begin{equation}
  \label{eq:635}
\begin{split}
  \zeta(U;t,x) &= \frac{1}{2}\bigl[(1+\eta'(t,x)^2)^{-3/2} -1\bigr]\\
&= \frac{1}{2}\Bigl[\bigl(1+\bigl(\partial_x\lk\bigl(\frac{u-\bar{u}}{2i}\bigr)\bigr)^2\bigr)^{-3/2}-1\Bigr].
\end{split}\end{equation}
Recall that we introduced the matrices 
\begin{equation}
  \label{eq:636}
  \Ical_2 = \bigl[\begin{smallmatrix}1&0\\0&1\end{smallmatrix}\bigr], \quad \Kcal =
  \bigl[\begin{smallmatrix}1&0\\0&-1\end{smallmatrix}\bigr], \quad \Jcal =
  \bigl[\begin{smallmatrix}0&-1\\1&0\end{smallmatrix}\bigr], \quad \Lcal = \bigl[\begin{smallmatrix}0&1\\1&0\end{smallmatrix}\bigr]
\end{equation}
and set 
\begin{equation}
  \label{eq:637}
  \mk(\xi) = (\xi\tanh\xi)^{1/2}(1+\kappa\xi^2)^{1/2}(1-\chi(\xi))
\end{equation}
where $\chi$ is as in \eqref{631}. 
Let $N, \rho$ be given positive integers. Then, for $r>0$ small enough, there are symbols $\lambda_j$ in $\sG{j}{K,1,1}{N}$,
$j= -\frac{1}{2}, 0, \frac{1}{2}, 1$,  satisfying
\begin{equation}
  \label{eq:638}
 \Im\lambda_1 \in \sG{0}{K,1,1}{N},\quad \Im\lambda_{\frac{1}{2}}\in \sG{-\frac{1}{2}}{K,1,1}{N}
\end{equation}
and
\begin{equation}
  \label{eq:639}
  \begin{split}
    (\bar{\lambda}_j)^\vee &= \lambda_j,\ j = -\frac{1}{2}, \frac{1}{2},\  (\bar{\lambda}_j)^\vee = -\lambda_j,\ j = 0, 1\\
\lambda_j(U;t,x,\xi) &= \lambda_j(U;t,-x,-\xi),\ j= -\frac{1}{2}, 0, \frac{1}{2}, 1\\
\lambda_j(U;-t,x,\xi) &= \lambda_j(U_S;t,x,\xi),\ j = -\frac{1}{2}, \frac{1}{2}\\
\lambda_j(U;-t,x,\xi) &= -\lambda_j(U_S;t,x,\xi),\ j = 0, 1
  \end{split}
\end{equation}
such that $U = \vect{u}{\bar{u}}$ with $u$ given by \eqref{632}  solves the system
\begin{equation}
  \label{eq:6310}
  D_t U = \opbw(A(U;t,x,\xi))U + R(U;t)U
\end{equation}
where $A(U;t,x,\xi)$ is the matrix  of symbols
\begin{multline}
  \label{eq:6311}
A(U;t,x,\xi) = \bigl(\mk(\xi)(1+\zeta(U;t,x))+ \lambda_{\frac{1}{2}}\bigr)\Kcal\\
+ \bigl(\mk(\xi) \zeta(U;t,x) + \lambda_{-\frac{1}{2}}\bigr)\Jcal + \lambda_1\Ical_2 + \lambda_0\Lcal \, , 
\end{multline}
that satisfies the  reality, parity preserving, reversibility properties
\begin{equation}
  \label{eq:6312}
  \begin{split}
    \bar{A}^\vee(U;t,x,\xi) &= -SA(U;t,x,\xi)S\\
A(U;t,-x,-\xi) &= A(U;t,x,\xi)\\
A(U;-t,x,\xi) &= -SA(U_S;t,x,\xi)S
  \end{split}
\end{equation}
where $ S $ is the involution defined in \eqref{def-S}, 
and  $R(U;t)$ is a smoothing operator in $\sRM{-\rho}{K,1,1}{N}$ satisfying as well the
 reality, parity preserving, reversibility properties
\begin{equation}
  \label{eq:6313}
  \begin{split}
    \overline{R(U;t)V} &= -SR(U;t)S\overline{V}\\
R(U;t)\circ\tau &= \tau\circ R(U;t)\\
R(U;-t) &= -SR(U_S;t)S\, .
  \end{split}
\end{equation}
\end{proposition}
\noindent\textbf{Remark}:  Some of the above symbols or smoothing operators 
are indeed autonomous and belong to the corresponding smaller class with the index $ K' = 0 $ instead of $ K' = 1 $.  
But since this fact will not be used later we shall systematically use this weaker statement. 
The paradifferential system \eqref{6310} is the one presented 
in section \ref{sec:31}  to start the proof of Theorem \ref{311}. 

\begin{proof}
Recalling \eqref{632}, to compute the equation satisfied by $ u $, we insert 
\eqref{6217} in
\be\label{eq:patu}
\partial_t u = \lk\partial_t\omega + i\lk^{-1}\partial_t\eta \, .
\ee
The linear contribution is
\begin{equation}
  \label{eq:6314}
  i(D\tanh D)^{\frac{1}{2}}(1+\kappa D^2)^{\frac{1}{2}}u = i\opbw(\mk(\xi))u
\end{equation}
as follows from \eqref{Lambda-p} and 
the fact that $ \opbw(\mk(\xi)) $,  where $ \mk (\xi)  $ is the symbol defined in \eqref{637}, 
coincides with the operator $(D\tanh D)^{\frac{1}{2}}(1+\kappa D^2)^{\frac{1}{2}}$ when acting on
periodic functions with zero mean or on periodic functions modulo constants. 

The only other term of order $\frac{3}{2}$ comes from  $\opbw(a_2(\eta;\cdot))\eta$
in the formula of
$\partial_t\omega$ in \eqref{6217}. 
Expressing $ \eta $ as a function of $ (u, \bar u) $  as in \eqref{634}, it is given by
\begin{equation}
  \label{eq:6315}
  \lk\opbw(a_2(\eta;\cdot))\lk\Bigl[\frac{u-\bar{u}}{2i}\Bigr] \, .
\end{equation}
We now apply the composition results in Propositions \ref{231} and \ref{232}. 
By the properties of $\#$ in Definition \ref{221}, 
the symbol $ (\Lambda_\kappa(\xi)\# a_2\#\Lambda_\kappa(\xi))_\rho $ is equal to
$ a_2(\eta;x,\xi)\Lambda_\kappa(\xi)^2 $ modulo a symbol of order $ -\frac{1}{2}$ (see in particular \eqref{trilinear}
applied with $ a = c $),  
and  \eqref{6315} may be written as
\begin{multline}
  \label{eq:6316}
i\opbw(\tilde{a}_{\frac{3}{2}}(U;\cdot))(u-\bar{u}) + i\opbw(\tilde{a}_{-\frac{1}{2}}(U;\cdot))(u-\bar{u})\\ 
+R(U)(u-\bar{u})
\end{multline}
where 
$ \tilde{a}_{\frac{3}{2}} $ is the symbol in  $ \sGa{\frac{3}{2}}{K,0,1}{N} \subset \sG{\frac{3}{2}}{K,1,1}{N}$ given by
\be\label{eq:a32}
\tilde{a}_{\frac{3}{2}} = -\frac{1}{2} a_{2}(\eta;x,\xi)\Lambda_\kappa(\xi)^2 \, , 
\ee
$ \tilde{a}_{-\frac{1}{2}}$ is a symbol in $\sGa{-\frac{1}{2}}{K,0,1}{N} \subset \sG{-\frac{1}{2}}{K,1,1}{N}$ and
$R(U)$ is a smoothing operator
in $ \sGa{-\rho}{K,0,1}{N}  \subset \sG{-\rho}{K,1,1}{N} $.
The symbol $ a_2 $ is given by the right hand side of 
\eqref{624a} plus a symbol $ {\tilde a}_2^{(0)} $ of order zero. Since $\Lambda_\kappa(\xi) $ 
defined in \eqref{631} has order $ - 1 / 4 $, the symbol 
$  {\tilde a}_2^{(0)} (\eta;x,\xi)\Lambda_\kappa(\xi)^2 $  
may be added to $\tilde{a}_{-\frac{1}{2}}$ in \eqref{6316}.
By \eqref{624a} the symbol  $ \tilde{a}_{\frac{3}{2}} $ in \eqref{a32} may be written as 
\[
  \frac{\kappa}{2}(1-\chi(\xi))^2\xi^2(\xi\tanh\xi)^{\frac{1}{2}}(1+\kappa\xi^2)^{-\frac{1}{2}}[(1+\eta'{}^2)^{-\frac{3}{2}}-1]
= \mk(\xi)\zeta(U;t,x) + \tilde{r}_{-1/2}
\]
where $\zeta$ is defined in \eqref{635}, $ \mk(\xi) $ in \eqref{637}, 
and $\tilde{r}_{-1/2}$ is a symbol of order $-\frac{1}{2}$ that may be
incorporated to the $\tilde{a}_{-1/2}$ term in \eqref{6316}, up to a change of notation. 
In conclusion we have 
 written \eqref{6316}, thus \eqref{6315}, as
\begin{multline}
  \label{eq:6317}
i\opbw\bigl(\mk(\xi)\zeta(U;t,x)\bigr)(u-\bar{u}) + i\opbw\bigl(a_{-1/2}(U; t, \cdot)\bigr)(u-\bar{u})\\
+R(U;t)(u-\bar{u})
\end{multline}
with $a_{-1/2}$ in $ \sG{-\frac{1}{2}}{K,1,1}{N}$.

The symbol $a_2$ satisfies \eqref{623} as well as the even real symbol 
$   \Lambda_\kappa (\xi) $ defined in \eqref{631}. 
Since these properties are preserved by the composition operator $\#$, we conclude
that also the symbols $\tilde{a}_{\frac{3}{2}}, \tilde{a}_{-\frac{1}{2}}$ in \eqref{6316}  satisfy them.
In addition, by \eqref{635} and \eqref{637}, also $\mk(\xi)\zeta(U;t,x)$ satisfies \eqref{623}.
It follows that also the symbol 
$a_{-\frac{1}{2}}$ in \eqref{6317} satisfy  these conditions. 
Moreover, these symbols, in particular $a_{-\frac{1}{2}}$, 
as functions of $(\eta,\psi)$ do not depend on $ \psi $ (as $ a_2 $). 

The contributions of order one to $\partial_tu $ in \eqref{patu} come from the term 
$\opbw(a_1)\omega$ in the right hand side of the second
equation \eqref{6217} and from the term $\opbw(b_1)\eta$  of the first equation \eqref{6217}, and are given by
\be\label{eq:contr1}
\lk \opbw (a_1) \omega + i \lk^{-1} \opbw (b_1) \eta \, .
\ee
By \eqref{622} and Corollary~\ref{622} we have that 
\be\label{eq:a1b1}
a_1 = -iV(\eta,\psi)\xi \, , \quad b_1 = - i V(\eta,\psi)\xi 
\ee
modulo a symbol of order zero. Moreover 
$ a_1, b_1 $ are linear in  $\psi $,  even in $ (x, \xi) $, and 
satisfy $ \bar{a}_1^\vee = a_1 $, $ \bar{b}_1^\vee = b_1 $, 
as stated  in Proposition
\ref{621} and Corollary \ref{622}. 
Substituting  \eqref{a1b1} and the expressions \eqref{634} of 
$ \eta $ and $ \omega $, \eqref{contr1} may be written as  
\begin{multline*}
  \lk\opbw(-iV(\eta,\psi)\xi)\lk^{-1}\Bigl(\frac{u+\bar{u}}{2}\Bigr)\\
+i\lk^{-1}\opbw(-iV(\eta,\psi)\xi)\lk \Bigl(\frac{u-\bar{u}}{2i}\Bigr)\\
+\opbw(e_{0,+})u + \opbw(e_{0,-}) {\bar u}
\end{multline*}
where $e_{0,\pm}$ are symbols of order zero, even in $(x,\xi)$, satisfy $\bar{e}_{0,\pm}^\vee = e_{0,\pm}$, 
and they are linear in $ \psi $.
By symbolic calculus (i.e. Propositions \ref{231} and \ref{232}), we may write this expression as
\begin{equation}
  \label{eq:6318}
  i\bigl[\opbw(\lambda_1(\eta,\psi))u+ \opbw(\lambda_0(\eta,\psi))\bar{u}\bigr]
\end{equation}
up to contributions 
$ R_1(\eta,\psi)u+ R_2(\eta,\psi)\bar{u}$ with smoothing operators $R_1, R_2$  in 
$ \sRa{-\rho}{K,0,1}{N} \subset \sR{-\rho}{K,1,1}{N} $, 
where  

\noindent
$ \bullet $ $\lambda_1$ is in $ \sGa{1}{K,0,1}{N} \subset \sG{1}{K,1,1}{N}$, with 
$\Im\lambda_1$ in $\sG{0}{K,1,1}{N}$, \\
$ \bullet $ $\lambda_0$ 
is in  $\sGa{0}{K,0,1}{N} \subset \sG{0}{K,1,1}{N}$.

Moreover the symbols $  \lambda_0, \lambda_1 $  satisfy
\be\label{eq:l0l1}
\bar{\lambda}_\ell^\vee = -\lambda_\ell,\quad \lambda_\ell(\eta,\psi;t,-x,-\xi) = \lambda_\ell(\eta,\psi;t,x,\xi),
\ee
and $ \lambda_\ell(\eta,\psi;t,x,\xi) $, 
$ \ell=0,1$, are linear in $ \psi $.

The remaining terms involving paradifferential operators in \eqref{6217} 
bring to $\partial_tu$ in
\eqref{patu} the following contributions
\begin{multline*}
  \lk\opbw(a_0(\eta,\omega; t, \cdot))\lk\Bigl[\frac{u-\bar{u}}{2i}\Bigr]\\
+ i \lk^{-1}\opbw(b_0(\eta;\cdot))\lk^{-1}\Bigl[\frac{u+\bar{u}}{2}\Bigr].
\end{multline*}
Since $ a_0 $ is a symbol of order zero, 
$ b_0 $ of order $ - 1 $, 
and $ \Lambda_\kappa (\xi)  $ of order $ - 1 / 4 $, 
by symbolic calculus we may write it as
\begin{equation}
  \label{eq:6320}
  i\opbw(\mu_{-1/2})u-i\opbw(\lambda_{-1/2})\bar{u} +R_1(\eta,\omega; t)u+R_2(\eta,\omega; t )\bar{u}
\end{equation}
where $ \mu_{-1/2} $ and $\lambda_{-1/2}$ are in $\sG{-\frac{1}{2}}{K,1,1}{N}$. 
In addition, according to the properties of $a_0, b_0$ listed in Proposition~\ref{621} and Corollary~\ref{622},
\begin{equation*} 
\begin{split}
 (\bar{\mu}_{- \frac12})^\vee = \mu_{- \frac12},\quad 
  \mu_{- \frac12}(\eta,\psi;t,-x,-\xi) = \mu_{- \frac12}(\eta,\psi;t,x,\xi), \, \\
  (\bar{\lambda}_{- \frac12})^\vee = \lambda_{- \frac12},\quad 
  \lambda_{- \frac12}(\eta,\psi;t,-x,-\xi) = \lambda_{- \frac12}(\eta,\psi;t,x,\xi) \, ,
  \end{split}
\end{equation*}
and, since 
$ a_0 $ satisfies \eqref{a0:rev},  as well as $ b_0  $ which depends only on $ \eta $, 
$$
\mu_{- \frac12}(\eta,\psi;-t) = \mu_{- \frac12}([\eta,\psi]_S;t) \, , \quad 
\lambda_{- \frac12}(\eta,\psi;-t) = \lambda_{- \frac12}([\eta,\psi]_S;t)  \, , 
$$ 
where 
$ [ \eta, \psi ]_S $ is defined in \eqref{310-real}.

In conclusion \eqref{patu}, \eqref{6314}, \eqref{6317}, \eqref{6318}, \eqref{6320} imply that
\begin{multline}
  \label{eq:6322}
D_t u = \opbw[\mk(\xi)(1+\zeta)+\lambda_{1/2}]u - \opbw[\mk(\xi)\zeta + \lambda_{-1/2}]\bar{u}\\
+\opbw(\lambda_1)u + \opbw(\lambda_0)\bar{u} + R_1(U;t)u + R_2(U;t)\bar{u}
\end{multline}
where $\lambda_j$ are new symbols in $\sG{j}{K,1,1}{N}$ and satisfy
\begin{equation}
  \label{eq:6323}
  \begin{split}
    \Im \lambda_{1/2}\textrm{ is of order } -\frac{1}{2} \, , \ \Im\lambda_1 \textrm{ is of order }  0\\
(\bar{\lambda}_{1/2})^\vee = \lambda_{1/2} \, ,\ (\bar{\lambda}_{-1/2})^\vee = \lambda_{-1/2}\\
(\bar{\lambda}_{1})^\vee = -\lambda_{1} \, ,\ (\bar{\lambda}_{0})^\vee = -\lambda_{0}\\
\lambda_j \textrm{ are even functions of } (x,\xi).
  \end{split}
\end{equation}
Notice that we have actually proved that the symbol $\lambda_{1/2}$ is of order $-\frac{1}{2}$, 
and not just its imaginary part, but we
conserve only those properties of the symbols that are useful for us, and that are preserved under the reductions of
Chapter~\ref{cha:3} and Chapter~\ref{cha:4}. 
Moreover, as we have seen above 
\begin{equation}
  \label{eq:6325}
\begin{split}
  &\lambda_j(U;-t,x,\xi) = -\lambda_j(U_S;t,x,\xi),\quad j=0, 1 \, , \\
&\lambda_j(U;-t,x,\xi) = \lambda_j(U_S;t,x,\xi),\quad j= \frac{1}{2}, -\frac{1}{2} 
\end{split}\end{equation}
(where $ S $ is now the complex involution defined in \eqref{def-S}).

Taking the complex conjugate of  \eqref{6322}, we get, because of
\eqref{opbavv}, the fact that $ \mk(\xi) $ is even in $ \xi $, and  \eqref{6323}, 
\begin{multline}\label{eq:6322-bar}
  D_t\bar{u} = -\opbw[\mk(\xi)(1+\zeta)+\lambda_{1/2}]\bar{u} + \opbw[\mk(\xi)\zeta + \lambda_{-1/2}]u\\
+\opbw(\lambda_1)\bar{u} +\opbw(\lambda_0)u -\overline{R_1(U;t)u} -\overline{R_2(U;t)\bar{u}} \, .
\end{multline}
Finally, we write the equations \eqref{6322} and \eqref{6322-bar} as a  system in the variable
 $U = \vect{u}{\bar{u}}$. We get, using notation \eqref{636},
  \begin{multline}
    \label{eq:6326}
D_tU = \opbw\bigl([\mk(\xi)(1+\zeta)+\lambda_{1/2}]\Kcal + [\mk(\xi)\zeta + \lambda_{-1/2}]\Jcal\\
+\lambda_1\Ical_2 +\lambda_0\Lcal\bigr)U + R(U;t)U
  \end{multline}
for some matrix $R$ of smoothing operators in $\sR{-\rho}{K,1,1}{N}$.
We have finally proved the paradifferential form \eqref{6310}-\eqref{6311} 
of system \eqref{WW-complex}. 

Since $S\Kcal S = -\Kcal, S\Jcal S = -\Jcal, S\Lcal S = \Lcal, S\Ical_2 S = \Ical_2$, it follows from \eqref{6323} that the
matrix $A(U; t, \cdot)$ giving the symbol in the right hand side 
of \eqref{6326} satisfies the 
reality and parity preserving conditions \eqref{311}, \eqref{313}.
Moreover conditions \eqref{6325} are 
equivalent to the reversibility condition \eqref{314}.
We 
have thus checked that all properties \eqref{638} to \eqref{6312} hold.

Finally, by comparing \eqref{WW-complex} and  \eqref{6310} we have that 
$$
R(U; t) = X(U; t ) - \opbw (A(U; t, \cdot )) \, . 
$$
As stated after \eqref{WW-complex} the operator $ X(U; t) $ satisfies the reality, parity preserving, 
 reversibility properties \eqref{316}, \eqref{318}, \eqref{319}.
Since  the operator $ \opbw (A(U; t, \cdot ))  $ satisfies   
 as well  \eqref{316}, \eqref{318}, \eqref{319}  (by \eqref{6312}) 
 we deduce that  $ R(U; t) $
 satisfies \eqref{6313}. This concludes
the proof.
\end{proof}
\textbf{Remark}: In the above system \eqref{6310}, \eqref{6311}, the symbols $\lambda_j$ are in the class $\sG{j}{K,1,1}{N}$ and the
smoothing operator $R(U; t)$ belongs to the space $\sRM{-\rho}{K,1,1}{N}$. This implies, 
recalling \eqref{2123}
and \eqref{2117} for $ k = 0 $ and $ K ' = 1 $, 
that the Sobolev norms of the non-homogeneous terms in the 
right hand side of \eqref{6310} may be
estimated from the Sobolev norms of $U$ and of $\partial_t U$. This dependence of the bounds on first order time derivatives
is irrelevant to derive energy inequalities for the water waves system \eqref{6310}. On the other hand, we
shall also need  to estimate time derivatives of $ U $ by  the space
derivatives, as claimed in \eqref{6330} of the next proposition. For that we use the water waves equation in its
initial formulation \eqref{113}.
\begin{proposition}
  \label{632}
Let $(\eta,\psi)$ be a solution of system \eqref{113} defined on some interval $I$ and belonging to the space
\[C^0(I,\Hsze{s+\frac{1}{4}}(\Tu,\R)\times \Hdse{s-\frac{1}{4}}(\Tu,\R)) \, . \]
Let $ U  = \vect{u}{\bar{u}} $ where $ u $ is defined in \eqref{632}. 
For any $ 0 \leq k\leq K$ there is a constant $ C_k $ such that, as long as 
$ U(t,\cdot)$ stays
in the unit ball of $\Hds{s}(\Tu,\C^2)$ with $ s \gg K $, one has the estimate
\begin{equation}
  \label{eq:6330}
  \norm{\partial_t^kU(t,\cdot)}_{\Hds{s-\frac{3}{2}k}}\leq C_k\norm{U(t,\cdot)}_{\Hds{s}}.
\end{equation}
\end{proposition}

\begin{proof}
We proceed by induction. 
  Assume that \eqref{6330} has been proved for $k=0,\dots,k'\leq K-1$. 
 The assumption $ \norm{U(t,\cdot)}_{\Hds{s}} \leq 1 $ implies,  in particular, 
 that
  \begin{equation}
    \label{eq:6331}
    \sum_{k=0}^{k'}\norm{\partial_t^{k}U(t,\cdot) }_{\Hds{s-\frac{3}{2}k'}}\leq \tilde{C}_{k'}
  \end{equation}
for some constant $\tilde{C}_{k'}$, uniformly for $t$ in $I$. To prove \eqref{6330} with $k=k'+1$, it is enough, 
according to \eqref{632} and \eqref{Lambda-p}, to show that
\[\norm{\partial_t^{k'+1}\omega}_{\Hds{s-\frac{1}{4}-\frac{3}{2}(k'+1)}} +
\norm{\partial_t^{k'+1}\eta}_{\Hsz{s+\frac{1}{4}-\frac{3}{2}(k'+1)}} \leq C\norm{U(t,\cdot)}_{\Hds{s}}.\]
By \eqref{6127}, the fact that $ a_0 (\eta, \psi)  $ is in $ \sF{K,0,1}{N} $
and the bounds for paradifferential operators of Proposition~\ref{215}, it is enough to show that
\begin{equation}
  \label{eq:6332}
  \norm{\partial_t^{k'+1}\psi}_{\Hds{s-\frac{1}{4}-\frac{3}{2}(k'+1)}} +
\norm{\partial_t^{k'+1}\eta}_{\Hsz{s+\frac{1}{4}-\frac{3}{2}(k'+1)}} \leq C\norm{U(t,\cdot)}_{\Hds{s}}
\end{equation}
or, equivalently, that
\begin{equation}
  \label{eq:6333}
\norm{\partial_t^{k'+1}\psi}_{\Hds{s-\frac{1}{4}-\frac{3}{2}(k'+1)}} +
\norm{\partial_t^{k'+1}\eta}_{\Hsz{s+\frac{1}{4}-\frac{3}{2}(k'+1)}} \leq C
\big( \norm{\psi}_{\Hds{s-\frac{1}{4}}} + \norm{\eta}_{\Hsz{s+\frac{1}{4}}}\big) 
\end{equation}
since $\norm{U}_{\Hds{s}}$ is equivalent to 
$ \norm{\psi}_{\Hds{s-\frac{1}{4}}} + \norm{\eta}_{\Hsz{s+\frac{1}{4}}} $
 by \eqref{632}, \eqref{Lambda-p} and \eqref{6127}, \eqref{6131}. One may rewrite
\eqref{113} as
\[
\begin{split}
  \partial_t\eta &= G(\eta)\psi\\
\partial_t\psi &= \Fcal(\eta,\eta',\eta'',\partial_x\psi,G(\eta)\psi)
\end{split}\]
where $\Fcal$ is some analytic function vanishing at the origin. 
Thus we write
$$
\partial_t^{k'+1}\psi = \partial_t^{k'}  \big( \Fcal(\eta,\eta',\eta'',\partial_x\psi,G(\eta)\psi) \big) \, , 
\quad 
\partial_t^{k'+1} \eta = \partial_t^{k'} (G(\eta)\psi) \, , 
$$
and, since for $s\gg K$, $H^{s-\frac{3}{2}K}$ is an algebra, we
may estimate the left hand side of \eqref{6332} in terms of 
\[ 
\norm{\partial_t^k (G(\eta)\psi)}_{\Hsz{s-\frac{3}{2}+\frac{1}{4}-\frac{3}{2}k}}, \quad 
\norm{\partial_t^k\eta}_{\Hsz{s+\frac{1}{4}-\frac{3}{2}k}}, \quad 
 \norm{\partial_t^k\psi}_{\Hds{s+\frac{3}{4}-\frac{3}{2}(k+1)}}
\]
for $ k \leq k' $. 
Let us first prove that, for  any $k\leq k'$, 
\be\label{eq:ult1}
\norm{\partial_t^k
  (G(\eta)\psi)}_{\Hsz{s-\frac{3}{2}+\frac{1}{4}-\frac{3}{2}k}} \leq C \norm{U(t,\cdot)}_{\Hds{s}} \, . 
\ee
We use that
$ G(\eta)\psi$ is expressed from \eqref{6143} where the function 
$ V (\eta, \psi )  $ is in $ \sF {K,0,1}{N} $, the symbol $ b^0 $ in
$\sG{-1}{K,0,1}{N}$, $ c^0 $ in $\sG{0}{K,0,1}{N}$,
and the smoothing remainders
$ R_1, R_2 $  in  $\sR{-\rho}{K,0,1}{N}$. 
According to Proposition~\ref{215} applied
to symbols of order one,  \eqref{2117} with $ k = 0 $, $ K' = 0 $, and  the first remark after
Definition \ref{214},  we get for any $k\leq k'$
\[
  \norm{\partial_t^k   (G(\eta)\psi)}_{\Hsz{s-\frac{3}{2}+\frac{1}{4}-\frac{3}{2}k}} \leq C\sum_{k''\leq
  k}\bigl(\norm{\partial_t^{k''}\psi}_{\Hds{s-\frac{1}{4}-\frac{3}{2}k''}} + \norm{\partial_t^{k''}\eta}_{\Hsz{s-\frac{1}{4}-\frac{3}{2}k''}}\bigr)
\]
as long as \eqref{6331} holds. Thus \eqref{6131}, the fact that 
$ a (\eta, \psi) $  is in $ \sF {K,0,1}{N} $,  \eqref{634} 
and the inductive assumption \eqref{6330} for $k''\leq k\leq k'$, imply \eqref{ult1}.
Similarly 
we deduce that, for any $ 0 \leq k \leq k ' $, 
$$
\norm{\partial_t^k\eta}_{\Hsz{s+\frac{1}{4}-\frac{3}{2}k}} +  
\norm{\partial_t^k\psi}_{\Hds{s+\frac{3}{4}-\frac{3}{2}(k+1)}} \leq C \norm{U(t,\cdot)}_{\Hds{s}} \, ,
$$
and \eqref{6332} is proved.
\end{proof}

%% file: chap7.tex
\chapter{Proof of some auxiliary results}\label{cha:7}

\section{Non resonance condition}\label{sec:71}

Recall that we defined in \eqref{3115}, for $\kappa$ a positive number, and $n$ a positive integer
\begin{equation}
  \label{eq:711}
  \mk(n) = (n\tanh n)^{\frac{1}{2}}(1+\kappa n^2)^{\frac{1}{2}}.
\end{equation}
Moreover, if $p$ is in $\N$, $-1\leq \ell\leq p+1$, $n_0,\dots,n_{p+1}$ are in $\N^*$, we have defined 
the  ``small divisors''
\begin{equation}
  \label{eq:712}
  \Dcal_\ell(n_0,\dots,n_{p+1}) = \sum_{j=0}^\ell \mk(n_j) - \sum_{j=\ell+1}^{p+1}\mk(n_j).
\end{equation}
The goal of this section is to prove the following:
\begin{proposition}
  \label{711} {\bf (Non resonance condition)}
There is a subset $\Ncal$ of $]0,+\infty[$ with zero measure  such that, for any compact interval $[a,b]\subset
]0,+\infty[$, there is an integer $N_0\in \N$, and for any $ \kappa $ in $[a,b]-\Ncal$, there is a positive constant $c$ such that
the inequality  
\begin{equation}
\label{eq:713}
\abs{\Dcal_\ell(n_0,\dots,n_{p+1})}\geq c\max(n_0,\dots,n_{p+1})^{-N_0}
\end{equation}
holds for any $(n_0,\dots,n_{p+1})$ in $(\N^*)^{p+2}$ if $p$ is odd or $p$ is even and $\ell\neq\frac{p}{2}$, and for any
$(n_0,\dots,n_{p+1})$ in $(\N^*)^{p+2}$ such that  
\begin{equation}
  \label{eq:714}
  \{n_0,\dots,n_\ell\} \neq \{n_{\ell+1},\dots,n_{p+1}\}
\end{equation}
when $p$ is even and $\ell = \frac{p}{2}$.
\end{proposition}
To prove the proposition, we may fix a compact interval $[a,b]$.
The assumption made in the
statement about $n_0,\dots,n_{p+1}$ means that we have to show that, given $q$ in $\N$, $q\geq 2$, and $(c_1,\dots,c_q)$
in $(\Z^*)^q$, there exist $N_0$ in $\N$, $\delta>0$, $c>0$ such that for all $\gamma\in]0,1[$, there is a subset $\Ncal_\gamma$
of $[a,b]$ of measure $O(\gamma^\delta)$ such that, for any $\kappa$ in $[a,b]-\Ncal_\gamma$, any integers $1\leq
j_1<\cdots<j_q$, we have
\begin{equation}
  \label{eq:715}
  \Abs{\sum_{i=1}^q c_i\mk(j_i)}\geq c\gamma \Big(\sum_{i=1}^q\abs{j_i} \Big)^{-N_0}.
\end{equation}
We define $\tau(\xi) = (\xi\tanh\xi)^{1/2}$ and associate to $(j_1,\dots,j_q)$ the points of the interval $[0,1]$
defined by 
\begin{equation}
  \label{eq:716}
  x_0 = \Bigl(\sum_{\ell=1}^q(\abs{j_\ell}+\tau(j_\ell))\Bigr)^{-1},\ x_i = x_0j_i,\ x_{i+q} = \zeta_i = x_0\tau(j_i),
\end{equation}
$i=1,\dots,q$. We denote $X = [0,1]^{2q+1}$ so that $(x_0,\dots,x_q,\zeta_1,\dots,\zeta_q) =
(x_0,\dots,x_q,x_{1+q},\dots,x_{2q}) \in X$. We shall deduce \eqref{715}
from Theorem~5.1 in \cite{DSz}, applied with the preceding space $X$ and $Y = [a,b]$ to the function 
\begin{equation}
  \label{eq:717}
  f(x,\kappa) = \sum_{i=1}^q c_i\zeta_i\sqrt{x_0^2+\kappa x_i^2}
\end{equation}
and taking
\begin{equation}
  \label{eq:718}
  \rho(x) = x_0\prod_{1\leq i_1<i_2\leq q}(x_{i_1} - x_{i_2})\Bigl(\sum_{i=1}^q \zeta_i\Bigr).
\end{equation}
Let us check that the assumptions of Theorem~5.1 of \cite{DSz} hold. The function $f$ is continuous and subanalytic. We have
to check moreover that $f$ is analytic on the set $\{x\in X; \rho(x)\neq 0\}\times Y$. Since
\begin{equation}
  \label{eq:719}
  \{x\in X; \rho(x)\neq 0\} = \big\{ x \in X; x_0\neq 0,x_{i_1}\neq x_{i_2}, 1\leq i_1<i_2\leq q, \sum_1^q\zeta_i\neq 0 \big\}
\end{equation}
is contained in $\{x_0\neq 0\}$, this is clear. Moreover, we have to verify that for all $x$ in $X$ such that $\rho(x)\neq
0$, the function $\kappa\to f(x,\kappa)$ has only finitely many zeros in $[a,b]$. This is a consequence of the following
lemma.
\begin{lemma}
  \label{712}
Let $\rho(x)\neq 0$. Then the analytic function $\kappa\to f(x,\kappa)$ is not identically zero on $[a,b]$.
\end{lemma}
\begin{proof}
  By \eqref{719}, if $\rho(x)\neq 0$, we have that $x_0\neq 0$, $x_i\neq x_j$ if $1\leq i<j\leq q$ and $\sum_1^q\zeta_i\neq
  0$. In particular, not all $\zeta_i$ vanish and we may assume $\zeta_{i_1}>0$,\dots,$\zeta_{i_m}>0$ for some $1\leq m\leq
  q$. We have to prove that the function
  \begin{equation}
    \label{eq:7110}
    \kappa\to f(x,\kappa) = \sum_{\ell=1}^m c_{i_\ell}\zeta_{i_\ell}\sqrt{x_0^2+\kappa x_{i_\ell}^2}
  \end{equation}
is not identically zero on $[a,b]$ or equivalently on $[0,+\infty[$. Since $x_0\neq 0$, up to a change of notation, we may
assume that $x_0^2=1$ in \eqref{7110}. Argue by contradiction assuming that \eqref{7110} vanishes identically on
$[0,+\infty[$.

\textbf{Case 1}: $x_{i_\ell}\neq 0$ for $\ell = 1,\dots,m$. We expand in Taylor series on a neighborhood of $\kappa=0$
\begin{equation}
  \label{eq:7111}
  \sqrt{1+\kappa x_{i_\ell}^2} = \sum_{n=0}^{+\infty}a_n\kappa^n x_{i_\ell}^{2n}
\end{equation}
with Taylor coefficients $a_n$ that are all non zero. If we assume that \eqref{7110} with $x_0^2=1$ vanishes identically for
$\kappa\in [0,+\infty[$, plugging \eqref{7111} in \eqref{7110} and writing that the coefficients of $\kappa^n$ in the
resulting expression all vanish, we get
\begin{equation}
  \label{eq:7112}
  \sum_{\ell=1}^m c_{i_\ell}\zeta_{i_\ell}x_{i_\ell}^{2n} = 0
\end{equation}
for any $n\in \N^*$. As all $c_{i_\ell}$ are non zero, as well as the $\zeta_{i_\ell}$, this implies that the Van der Monde
determinant
\begin{equation}
  \label{eq:7113}
  \begin{vmatrix}
    x_{i_1}^2&\dots&x_{i_m}^2\\
 x_{i_1}^4&\dots&x_{i_m}^4\\
\vdots&\vdots&\vdots\\
 x_{i_1}^{2m}&\dots&x_{i_m}^{2m}
  \end{vmatrix}
\end{equation}
has to vanish, which is a contradiction as we assumed $x_{i_\ell}\neq 0$, for any $\ell$ and $x_i\neq x_j$ if $i\neq j$.

\textbf{Case 2}: One of the $x_{i_\ell}$ vanishes, for instance $x_{i_1}=0$. Since the $x_i$ are two by two distincts, the
other $x_{i_\ell}$ are nonzero and \eqref{7110} with $x_0^2=1$ may be written as
\[c_{i_1}\zeta_{i_1}+\sum_{\ell=2}^m c_{i_\ell}\zeta_{i_\ell}\sqrt{1+\kappa x_{i_\ell}^2}.\]
If this quantity vanishes, expanding as above in Taylor series at $\kappa =0$, we conclude that again a Van der Monde
determinant in the variables $(x_{i_2},\dots,x_{i_m})$ has to vanish, that contradicts the fact that $x_i\neq x_j$ for $i\neq
j$. This concludes the proof.
\end{proof}
\begin{proof1}{End of proof of Proposition~\ref{711}}
As we have already seen, we have to construct $\Ncal_\gamma$ such that \eqref{715} holds if $\kappa$ is in
$[a,b]-\Ncal_\gamma$. We apply Theorem~5.1 of \cite{DSz}, whose assumptions are satisfied because of lemma~\ref{712}. There
exists $N_1$ in $\N$, $\delta>0$, $C>0$ such that, for all $N\geq N_1$, all $x\in X = [0,1]^{2q+1}$ such that $\rho(x)\neq
0$, 
\begin{equation}
  \label{eq:7114}
  \mathrm{meas}\, \big\{\kappa\in [a,b]; \abs{f(x,\kappa)}\leq \gamma\abs{\rho(x)}^N \big\}\leq C\gamma^\delta\abs{\rho(x)}^{N\delta}.
\end{equation}
Define $\Ncal_\gamma$ to be the set of $\kappa$ in $[a,b]$ such that $\abs{f(x,\kappa)}\leq \gamma\abs{\rho(x)}^N$ for some
$x$ defined by \eqref{716} from some $1\leq j_1<\cdots<j_q$. Notice that
\begin{multline*}
  \abs{\rho(x)} = \abs{x_0\prod_{1\leq i_1<i_2\leq q}(x_{i_1}-x_{i_2})\sum_1^q\zeta_i}\\
\leq \bigl(\sum(\abs{j_i}+\tau(j_i))\bigr)^{-1}\prod_{1\leq i_1<i_2\leq
  q}\frac{\abs{j_{i_1}-j_{i_2}}}{\sum(\abs{j_i}+\tau(j_i))} \frac{\sum\tau(j_i)}{\sum(\abs{j_i}+\tau(j_i))}
\end{multline*}
satisfies
\begin{equation}
  \label{eq:7115}
  C^{-1}\Bigl(\sum_{i=1}^q\abs{j_i}\Bigr)^{-N_2}\leq \abs{\rho(x)}\leq C \Bigl(\sum_{i=1}^q\abs{j_i}\Bigr)^{-1}
\end{equation}
with $N_2 = \frac{q^2-q+4}{2}$ since $\abs{j_{i_1}-j_{i_2}}\geq 1$ and $\tau(j_i)\geq (\tanh 1)^{1/2}$. Consequently, the
measure of the set $\Ncal_\gamma$ is bounded from above using  \eqref{7114} and \eqref{7115} by
\[
\sum_{1\leq j_1<\cdots<j_q}C\gamma^\delta\Bigl(\sum_{i=1}^q\abs{j_i}\Bigr)^{-N\delta} = O(\gamma^\delta)
\]
if $N$ is chosen large enough so that $N\delta>q$. Moreover, when $\kappa$ is outside $\Ncal_\gamma$, we have
\[
\abs{f(x,\kappa)}\geq \gamma\abs{\rho(x)}^N \geq c\gamma\Bigl(\sum_{i=1}^q\abs{j_i}\Bigr)^{-NN_2}
\]
i.e.\ taking \eqref{717}, \eqref{716} into account
\[
\Abs{\sum_{i=1}^qc_i\mk(j_i)} \Bigl(\sum_{\ell=1}^q(\abs{j_\ell}+\tau(j_\ell))\Bigr)^{-2}
\geq c\gamma\Bigl(\sum_{i=1}^q\abs{j_i}\Bigr)^{-NN_2}
\]
which implies \eqref{715} with $N_0 = NN_2-2$. In conclusion the set 
$\Ncal = \bigcap_{0<\gamma<1} \Ncal_\gamma$
has zero measure, and for all $\kappa \in [a,b] \setminus \Ncal$ the bound \eqref{713} holds for some $c > 0$.
\end{proof1}

\section[Structure of Dirichlet-Neumann operator]{Precise structure of the Dirichlet-Neumann operator}\label{sec:72}

The Dirichlet-Neumann operator, in one dimension, is well known to be expressed as
 $G(\eta)\psi$ in \eqref{6143} with a symbol $b^0$
of order as negative as desired, if $\eta$ is smooth enough. In our framework, it is sufficient to know that $b^0$ is of
order $-1$, and this is the assertion that is made in the statement of Proposition~\ref{615}. 
We prove such property in this section. 

\begin{proposition}
  \label{721}
With the notation of Proposition~\ref{615}, the symbol $b^0 (\eta; \cdot) $ belongs to   $\sGa{-1}{K,0,1}{N}$.
\end{proposition}

We shall use the following lemma.
Recall that we have defined in \eqref{528} the symbol 
\[
e_+^0(\eta;z,x,\xi) = \Ccal(z,\xi) + e_{+,1}^0(\eta;z,x,\xi),
\] 
whose explicit value is given by $w_+$ in \eqref{5215}, with $ \Ccal(z,\xi) $  defined in \eqref{523} and 
$ e_{+,1}^0 $   in $\sP{0,+}{K,0,1}{N} $. Thus the symbol $e_+$ in \eqref{5221} may be written as
\begin{equation}
  \label{eq:721}
  e_+(\eta;z,x,\xi) = e^0_+(\eta;z,x,\xi) + e^1_{+,1}(\eta;z,x,\xi) +\cdots + e^{\rho-1}_{+,1}(\eta;z,x,\xi) 
\end{equation}
where $e_{+,1}^j$ is in $\sP{-j,+}{K,0,1}{N} $. 

\begin{lemma}
  \label{722}
The component $e_{+,1}^1$ of the symbol $ e_+ $ in \eqref{721} solves for say $\abs{\xi}\geq 1$ the equation
\begin{equation}
  \label{eq:722}
  \begin{split}
    &Pe_{+,1}^1 = -\abs{\xi}(\sign\xi -i\eta')^{-2}\eta''(x)e_+^0(\eta;z,x,\xi)\\
&e_{+,1}^1\vert_{z=0} = 0\\
&\partial_z e_{+,1}^1\vert_{z=-1} = 0
  \end{split}
\end{equation}
where $ P = (1+\eta'{}^2)\partial_z^2 -2i\eta'\xi\partial_z -\xi^2$  is the operator introduced in \eqref{529}.
\end{lemma}
\begin{proof}
  Recall that by definition $e_{+,1}^1$ solves equation \eqref{5224} with $j=1$, whose right hand side $\Lcal_{1,+}(e^0_+)$
  is given by the opposite of the symbols of order one in those terms in the right hand side of \eqref{5223} that follow the
  $\opbw(Pe_+)$ contribution. In other words, we obtain  $\Lcal_{1,+}(e^0_+)$ taking the opposite of the principal term in
\begin{multline*}
\bigl((\eta'{}^2\#\partial_z^2e_+^0)_{\rho,N} -\eta'{}^2\partial_z^2 e_+^0\bigr) -2i\bigl(((\eta'\xi)\#\partial_z e_+^0) _{\rho,N} -
(\eta'\xi)\partial_z e_+^0\bigr) \\- \bigl((\xi^2\# e_+^0) _{\rho,N} - \xi^2e_+^0\bigr)
\end{multline*}
so that
\begin{equation}
  \label{eq:723}
  Pe_{+,1}^1 = -\biggl[\frac{1}{2i}\absp{\eta'{}^2,\partial_z^2 e_+^0} - \absp{\eta'\xi,\partial_z e_+^0}  - \frac{1}{2i}\absp{\xi^2,e_+^0}\biggr]
\end{equation}
for $\abs{\xi}\geq 1$. The condition $Pe_+^0 = 0$ brings an expression for $\partial_z^2 e_+^0$ as
\begin{equation}
  \label{eq:724}
  \partial_z^2 e_+^0 = 2i\frac{\eta'}{1+\eta'{}^2}\xi\partial_z e_+^0 + \frac{\xi^2}{1+\eta'{}^2}e_+^0.
\end{equation}
Since  
$ e_+^0$ is given by $w_+$ in \eqref{5215}, taking into account \eqref{5213}, we may
write
\begin{equation}
  \label{eq:725}
  \partial_z e_+^0 = \frac{\xi}{1+\eta'{}^2}\biggl[i\eta' + \tanh\Bigl[\frac{(z+1)\xi}{1+\eta'{}^2}\Bigr]\biggr]e_+^0(\eta;z,x,\xi)
\end{equation}
modulo a symbol in $\sP{-\infty,+}{K,0,1}{N}$ that may be discarded as it 
will play no role in the computation of the symbol $e_{+,1}^1$ of
order $-1$. 
  Plugging in \eqref{724}, we get
\begin{equation}
  \label{eq:726}
  \partial^2_z e_+^0 = \frac{\xi^2}{(1+\eta'{}^2)^2}\Bigl[(1-\eta'{}^2) + 2i\xi\eta'\tanh\bigl[\frac{(z+1)\xi}{1+\eta'{}^2}\bigr]\Bigr]e_+^0(\eta;z,x,\xi)
\end{equation}
modulo a symbol in of order $-\infty$. Note that if $z$ stays in $[-\frac{1}{2},0]$ and $\abs{\xi}\geq 1$,
$\tanh\bigl[\frac{(z+1)\xi}{1+\eta'{}^2}\bigr] -\sign\xi$ is of order $-\infty$. Moreover, 
for $z$ in $[-1,-\frac{1}{2}]$, the symbol $e_+^0$ is of order $-\infty$ by construction. We may thus rewrite \eqref{725}, \eqref{726} as
\begin{equation}
  \label{eq:727}
  \begin{split}
    \partial_ze_+^0 &= \xi(\sign\xi-i\eta')^{-1}e_+^0\\
\partial^2_ze_+^0 &= \xi^2(\sign\xi-i\eta')^{-2}e_+^0
  \end{split}
\end{equation}
modulo again symbols of $\sP{-\infty,+}{K,0,0}{N}$. Moreover, from \eqref{5215}, we deduce that $e_+^0$ may be written in the
region $\abs{\xi}>1$ as $e^{iza\xi+z\abs{\xi}(1+c)}$ modulo symbols of order $-\infty$, so that
\begin{equation}
  \label{eq:728}
  \begin{split}
    \absp{\xi,e_+^0} &= iz\xi\absp{\xi,\eta'}(\sign\xi-i\eta')^{-2}e_+^0\\
\absp{\eta',e_+^0} &= z\absp{\eta',\xi}(\sign\xi-i\eta')^{-1}e_+^0
  \end{split}
\end{equation}
for $\abs{\xi}\geq 1$, modulo symbols of $\sP{-\infty,+}{K,0,1}{N}$. If one plugs \eqref{727} inside \eqref{723} and uses
\eqref{728}, we get by a direct computation that \eqref{723} is given by
\[-\abs{\xi}(\sign\xi-i\eta')^{-2}\absp{\xi,\eta'}e_+^0\]
for $\abs{\xi}\geq 1$. This gives \eqref{722}.
\end{proof}
We compute now the $\partial_z$ derivative of $e_{+,1}^1$.
\begin{lemma}
  \label{723}
The solution $e_{+,1}^1$ of \eqref{722} satisfies for $\abs{\xi}\geq 1$
\begin{equation}
  \label{eq:729}
  (1+\eta'{}^2)\partial_ze_{+,1}^1\vert_{z=0} = -\frac{1}{2}\eta''\frac{(\sign\xi+i\eta')^2}{1+\eta'{}^2}
\end{equation}
modulo a symbol in $\sGa{-\infty}{K,0,1}{N}$.
\end{lemma}
\begin{proof}
  According to lemma~\ref{522}, the solution $e_{+,1}^1$ of \eqref{722} is given by
  \begin{equation}
    \label{eq:7210}
    e_{+,1}^1(\eta;,z,x,\xi) = -\abs{\xi}(\sign\xi-i\eta')^{-2}\eta''(x)\int_{-1}^0 K^0(\eta;z,z',x,\xi)e_+^0(z';x,\xi)\,dz'
  \end{equation}
where $K^0$ is the kernel in \eqref{528}. This kernel has been defined in the proof of lemma~\ref{522} as $K^0 =
(1+\eta'{}^2)^{-1}\tilde{K}$, with $\tilde{K}$ given by \eqref{5220}. It follows that \eqref{729} is given by the product of
\begin{equation}
  \label{eq:7211}
  -\abs{\xi}(\sign\xi-i\eta')^{-2}\eta''(x)
\end{equation}
and of
\begin{equation}
  \label{eq:7212}
  \int_{-1}^0 \partial_zw_-(0,\xi,a,b)w_+(z',\xi,a,b)W(z',\xi,a,b)^{-1}e_+^0(z';x,\xi)\,dz'.
\end{equation}
According to \eqref{5215} and \eqref{5219}
\begin{multline*}
w_+(z')W(z')^{-1} = e^{-i(z'+1)a\xi}\cosh((z'+1)\xi(1+c))\\\times\Bigl(1-i\frac{a}{1+c}\tanh(\xi(z'+1)(1+c))\Bigr)
\end{multline*}
and
\[\partial_zw_-(0) = e^{ia\xi}\cosh(\xi(1+c))^{-1}\Bigl(1-i\frac{a}{1+c}\tanh(\xi(1+c))\Bigr)^{-1}\]
so that the integral \eqref{7212} is given by
\[\int_{-1}^0 e^{-iaz'\xi}\frac{\cosh((z'+1)\xi(1+c))}{\cosh(\xi(1+c))}e_+^0(z',\xi)\,dz'\]
modulo a symbol of order $-\infty$. Replacing $e_+^0$ by its value $w_+$ given by \eqref{5215}, we reduce the above integral
to
\[\int_{-1}^0 e^{2z'\abs{\xi}(1+c)}\,dz' = \frac{1+\eta'{}^2}{2\abs{\xi}}\]
modulo a symbol of order $-\infty$. Multiplying this value of \eqref{7212} by \eqref{7211}, we obtain for \eqref{729} an
expression that may be written as the right hand side of this equality.
\end{proof}
\begin{proof1}{Proof of Proposition~\ref{721}}
We want to show that, in the expression of $G(\eta)\psi$ given by \eqref{6143}, the symbol 
$b^0$ is of order $-1$ (and not just of order
zero). We have thus to examine those contributions in \eqref{6145} given by paradifferential operators acting on 
$\psi$ (or on $\omega$). They
come from 
$$
I+II+IV =  \partial_z\Phit\vert_{z=0} + \opbw(\eta'\otimes\eta')\partial_z\Phit\vert_{z=0} -\opbw(\eta')\partial_x\Phit\vert_{z=0} \, .
$$ 
The smoothing term $VI$ gives the last but one term in \eqref{6143}. Taking \eqref{6146} into
account, the paradifferential operators acting on 
$\psi$ (or on $\omega$) come from 
\begin{equation}
  \label{eq:7213}
  \partial_z\Phi\vert_{z=0} + \opbw(\eta'\otimes\eta')\partial_z\Phi\vert_{z=0} -\opbw(\eta')\partial_x\Phi\vert_{z=0} \, .
\end{equation}
Recall that $\Phi$ is given by \eqref{6137} where 
the symbol  $e_+$ in $\sP{0,+}{K,0,0}{N}$ and the integral expression and the last term in the right hand side provide
smoothing contributions. 
Note  in particular  that the integral term 
$ \int_{-1}^0\bigl[\opbw(K(\eta;z,\zp,\cdot))_{|z=0} f(\zp)\,d\zp $ 
 in \eqref{6137} is smoothing because the function $ f(\zp) $ given by \eqref{6136a}
 is supported for $\tilde{z}\leq -1/8 $  and then we use lemma \ref{5110} (condition \eqref{5127} holds by \eqref{6136}, changing the value of $\sigma_0$).
Modulo such smoothing terms, we may replace in \eqref{7213} $\Phi$ by
$\opbw(e_+(\eta;\cdot))\omega $. Using symbolic calculus we are thus left with showing that
the operator
\begin{multline}
  \label{eq:7214}
\opbw(1+\eta'{}^2)\opbw\bigl(\partial_ze_+ \vert_{z=0}\bigr) -i\opbw(\eta'\xi)\opbw(e_+\vert_{z=0})\\
+ \frac{1}{2}\opbw(\eta'')\opbw(e_+\vert_{z=0})
\end{multline}
may be written as
\begin{equation}
  \label{eq:7215}
  \opbw\bigl(\xi\tanh\xi + b_{-1}(\eta';\cdot)\bigr)
\end{equation}
where $b_{-1}$ is in $\sGa{-1}{K,0,1}{N}$. 

Recall that the symbol $e_+$ in \eqref{7214} has the 
 form \eqref{721},  namely may be written as
\[
  \begin{split}
    e_+ &= e_+^0 + e_{+,1}^1 + r\\
e_+^0 &= \Ccal(z,\xi) + e_{+,1}^0(\eta;z,x,\xi)
  \end{split}
\]
where $e_{+,1}^1 \in \sP{- 1,+}{K,0,1}{N} $ is the solution of \eqref{722}, 
$ r = e_{+,1}^2+\cdots +e_{+,1}^{\rho-1} $ is in $ \sP{-2,+}{K,0,1}{N} $
and $e_+^0  \in \sP{0,+}{K,0,0}{N} $ is given by \eqref{528}, actually is equal to $w_+$ in \eqref{5215}. The contribution coming from $r$ to
\eqref{7214} may be written as $\opbw(b_{-1}(\eta;\cdot))\omega$ for a symbol of order $-1$, since $r$ is of order $-2$. If
we consider now the contributions of $e_+^0+ e_{+,1}^1$ to \eqref{7214}, we see using symbolic calculus that, up to
expressions of the form $\opbw(b_{-1})$ with $b_{-1}$ in $\sGa{-1}{K,0,1}{N}$, the operator \eqref{7214} 
is given by a paradifferential operator with symbol
\begin{multline}
  \label{eq:7217}
(1+\eta'{}^2)\partial_z(e_+^0 + e_{+,1}^1)\vert_{z=0} + \frac{1}{2i}\Absp{\eta'{}^2,\partial_z e_+^0\vert_{z=0}}\\
-i\eta'\xi(e_+^0 + e_{+,1}^1)\vert_{z=0} -\frac{1}{2}\Absp{\eta'\xi,e_+^0\vert_{z=0}} + \frac{1}{2}\eta'' e_+^0\vert_{z=0} \, .
\end{multline}
By \eqref{528},  \eqref{523}, \eqref{5211},  \eqref{722}  we have  $e_+^0\vert_{z=0} = 1, e^{1}_{+,1}\vert_{z=0} = 0$  and by \eqref{725}, 
$$
\partial_z e_+^0\vert_{z=0} = \frac{\xi\tanh\xi +i\eta'\xi}{1+\eta'{}^2}
$$
modulo symbols in $\sGa{-1}{K,0,1}{N}$. Inserting in \eqref{7217}  these formulas 
and \eqref{729} we  get
\begin{multline*}
  \xi\tanh\xi + i\eta'\xi -\frac{1}{2}\eta''\frac{(\sign\xi +i\eta')^2}{1+\eta'{}^2} +
  \frac{1}{2i}\frac{1}{1+\eta'{}^2}\absp{\eta'{}^2,\abs{\xi}+i\eta'\xi}\\
-i\eta'\xi + \frac{1}{2}\eta'' = \xi\tanh\xi
\end{multline*}
  modulo symbols in $\sGa{-1}{K,0,1}{N}$ (that depend on $ \eta' $). 
  We have thus obtained that \eqref{7214} has  the form \eqref{7215},
  which concludes the proof.
\end{proof1}